\documentclass[reqno]{amsart}

\address{Simons Center for Geometry and Physics,
State University of New York, Stony Brook, NY 11794-3636 U.S.A.\\
Center for Geometry and Physics, Institute for Basic Sciences (IBS),
77 Cheongam-ro, Nam-gu, Pohang, Korea} \email{kfukaya@scgp.stonybrook.edu}
\address{Center for Geometry and Physics, Institute for Basic Sciences (IBS),
77 Cheongam-ro, Nam-gu, Pohang, Korea
\\
Department of Mathematics,
POSTECH, Pohang, Korea}
\email{yongoh1@postech.ac.kr}
\address{Graduate School of Mathematics,
Nagoya University, Nagoya, Japan} \email{ohta@math.nagoya-u.ac.jp}
\address{Research Institute for Mathematical Sciences, Kyoto University, Kyoto, Japan}
\email{ono@kurims.kyoto-u.ac.jp}

\usepackage{xypic}
\usepackage{enumerate}
\usepackage{amsmath}
\usepackage[all]{xy}
\usepackage{color}
\usepackage{yhmath}
\usepackage{mathrsfs}
\usepackage{bbding}
\usepackage{graphicx}
\usepackage{stmaryrd}

\newcommand{\leftineqineq}{\langle \!\langle}
\newcommand{\rightineqineq}{\rangle \!\rangle}

\def\E{\ifmmode{\mathbb E}\else{$\mathbb E$}\fi} 
\def\N{\ifmmode{\mathbb N}\else{$\mathbb N$}\fi} 
\def\R{\ifmmode{\mathbb R}\else{$\mathbb R$}\fi} 
\def\Q{\ifmmode{\mathbb Q}\else{$\mathbb Q$}\fi} 
\def\C{\ifmmode{\mathbb C}\else{$\mathbb C$}\fi} 
\def\H{\ifmmode{\mathbb H}\else{$\mathbb H$}\fi} 
\def\Z{\ifmmode{\mathbb Z}\else{$\mathbb Z$}\fi} 
\def\P{\ifmmode{\mathbb P}\else{$\mathbb P$}\fi} 
\def\T{\ifmmode{\mathbb T}\else{$\mathbb T$}\fi} 
\def\SS{\ifmmode{\mathbb S}\else{$\mathbb S$}\fi} 
\def\DD{\ifmmode{\mathbb D}\else{$\mathbb D$}\fi} 
\def\K{\ifmmode{\mathbb K}\else{$\mathbb K$}\fi}

\newcommand{\del}{\partial}

\newcommand{\be}{\begin{equation}}
\newcommand{\ee}{\end{equation}}
\newcommand{\bea}{\begin{eqnarray}}
\newcommand{\eea}{\end{eqnarray}}
\newcommand{\beastar}{\begin{eqnarray*}}
\newcommand{\eeastar}{\end{eqnarray*}}
\newcommand{\bc}{\begin{center}}
\newcommand{\ec}{\end{center}}
\newcommand{\llb}{\llbracket}
\newcommand{\lb}{[}
\newcommand{\rrb}{\rrbracket}

\def\L{\mathbb{L}}

\theoremstyle{theorem}
\newtheorem{thm}{Theorem}[section]
\newtheorem{cor}[thm]{Corollary}
\newtheorem{lem}[thm]{Lemma}
\newtheorem{sublem}[thm]{Sublemma}
\newtheorem{prop}[thm]{Proposition}

\theoremstyle{definition}
\newtheorem{defn}[thm]{Definition}
\newtheorem{rem}[thm]{Remark}

\newtheorem{exm}[thm]{Example}
\newtheorem{conds}[thm]{Condition}
\newtheorem{prob}[thm]{Problem}

\newtheorem{asmp}[thm]{Assumption}
\newtheorem{proper}[thm]{Properties}

\newtheorem*{thm*}{Theorem}

\numberwithin{equation}{section}

\def\dudtau{{\frac{\del u}{\del \tau}}}
\def\dudt{{\frac{\del u}{\del t}}}

\def\R{{\mathbb R}}

\def\Crit{{\hbox{Crit}}}
\def\E{{\mathbb E}}
\def\Z{{\mathbb Z}}
\def\C{{\mathbb C}}
\def\R{{\mathbb R}}

\def\N{{\mathbb N}}

\def\e{\varepsilon} 
\def\CA{{\mathcal A}}

\def\CG{{\mathcal G}}

\def\CJ{{\mathcal J}}
\def\CK{{\mathcal K}}
\def\CL{{\mathcal L}}
\def\CM{{\mathcal M}}

\def\CP{{\mathcal P}}
\def\CQ{{\mathcal Q}}

\def\CP{{\mathcal P}}

\def\opname#1{\mathop{\kern0pt{\rm #1}}\nolimits}

\def\dim{\opname{dim}}
\def\vol{\opname{vol}}

\def\Ham{\opname{Ham}}
\def\Cal{\opname{Cal}}
\def\Per{\opname{Per}}

\def\supp{\operatorname{supp}}

\def\Per{\operatorname{Per}}

\def\Crit{\operatorname{Crit}}
\def\Fix{\operatorname{Fix}}

\def\mq{\mathfrak{q}}
\makeindex
\begin{document}
\quad \vskip1.375truein

\def\mq{\mathfrak{q}}
\def\mp{\mathfrak{p}}
\def\mH{\mathfrak{H}}
\def\mh{\mathfrak{h}}
\def\ma{\mathfrak{a}}
\def\ms{\mathfrak{s}}
\def\mm{\mathfrak{m}}
\def\mn{\mathfrak{n}}
\def\mz{\mathfrak{z}}
\def\mw{\mathfrak{w}}
\def\Hoch{{\tt Hoch}}
\def\mt{\mathfrak{t}}
\def\ml{\mathfrak{l}}
\def\mT{\mathfrak{T}}
\def\mL{\mathfrak{L}}
\def\mg{\mathfrak{g}}
\def\md{\mathfrak{d}}
\def\mr{\mathfrak{r}}

\title[Spectral invariants with bulk]
{Spectral invariants with bulk, quasi-morphisms and Lagrangian Floer theory}

\author[K. Fukaya, Y.-G. Oh, H. Ohta, K.
Ono]{Kenji Fukaya, Yong-Geun Oh, Hiroshi Ohta, Kaoru Ono}
\thanks{Kenji Fukaya is supported partially by JSPS Grant-in-Aid for Scientific Research
No. 23224002, NSF Grant No. 1406423, and Simons Collaboration on Homological Mirror Symmetry.
Yong-Geun Oh is supported by the IBS project IBS-R003-D1.
Hiroshi Ohta is supported by JSPS Grant-in-Aid for Scientific Research Nos. 23340015, 15H02054.
Kaoru Ono is supported by JSPS Grant-in-Aid for
Scientific Research, Nos. 18340014, 21244002, 26247006, 23224001.}

\begin{abstract}
In this paper we first develop various enhancements of the theory of
spectral invariants of Hamiltonian Floer homology and of
Entov-Polterovich theory of spectral symplectic quasi-states and
quasi-morphisms by incorporating \emph{bulk deformations}, i.e.,
deformations by ambient cycles of symplectic manifolds, of the Floer
homology and quantum cohomology. Essentially the same kind of
construction is independently carried out by Usher \cite{usher:talk}
in a slightly less general context. Then we explore various applications of
these enhancements to the symplectic topology, especially new
construction of symplectic quasi-states, quasi-morphisms and new
Lagrangian intersection results on toric
and non-toric manifolds
\par
The most novel part of this paper is to use open-closed
Gromov-Witten-Floer theory (operator $\frak q$ in \cite{fooo:book1}
and its variant involving closed orbits of periodic Hamiltonian
system) to connect spectral invariants (with bulk deformation),
symplectic quasi-states, quasi-morphism to the Lagrangian Floer
theory (with bulk deformation).
\par
We use this open-closed Gromov-Witten-Floer theory to produce new
examples. Especially using the calculation of Lagrangian Floer
cohomology with bulk deformation in \cite{fooo:toric1,fooo:bulk, fooo:toricmir}, we
produce examples of compact symplectic manifolds $(M,\omega)$ which
admits uncountably many independent quasi-morphisms
$\widetilde{\text{\rm Ham}}(M,\omega) \to \R$.
We also obtain a new intersection result for the Lagrangian
submanifold in $S^2 \times S^2$ discovered in \cite{fooo:S2S2}.
\par
Many of these applications were announced in \cite{fooo:toric1,fooo:bulk,fooo:S2S2}.
\end{abstract}

\date{October 19, 2016}

%

\maketitle
\newpage
\tableofcontents
\newpage



\centerline{\bf\LARGE Preface}
\par\medskip
Floer theory is an important tool in the study of symplectic
topology and in the study of mirror symmetry phenomenon of string
theory in physics (and also in low dimensional topology, which is
not touched upon in this paper). Since the theory was invented by
Floer in the late 80's, it has gone through the period of technical
underpinning and theoretical enhancement, especially manifested in the
Lagrangian Floer theory. Now the period of theoretical and
technical enhancement of the Gromov-Witten-Floer theory has passed and
the time for applying this enhanced theory as a major tool comes for
studying symplectic topology and its related areas.

Before these theoretical enhancements took place, the so
called, spectral invariants, were introduced by the second named
author in the context of Lagrangian submanifolds of the cotangent
bundle. This construction adapts the classical
Rabinowitz-Hofer-Zehnder's mini-max theory of the action functional
and Viterbo's stable Morse theory of generating functions to the
mini-max theory for the chain level Floer homology in the construction
of symplectic invariants. This construction was subsequently applied
to the Hamiltonian Floer theory by Schwarz in the aspherical case
and then by the second named author in general. The construction in
the non-exact case involves the usage of a Novikov ring and
non-Archimedean analysis in a significant manner. These spectral
invariants have been used by Entov-Polterovich in their
remarkable construction of symplectic quasi-states and
quasi-morphisms, which starts to unveil the true symplectic nature of
the analytical construction of Floer homology and spectral
invariants.

In this memoir, we develop various enhancements of the theory of
spectral invariants of Hamiltonian Floer homology and of
Entov-Polterovich theory of spectral symplectic quasi-states and
quasi-morphisms by incorporating the above mentioned theoretical
amplification and also non-Archimedean analysis into the theory.
More specifically, we incorporate \emph{bulk deformations}
introduced in \cite{fooo:book1}, i.e., the deformations by ambient
cycles of symplectic manifolds, of Floer homology and of quantum
cohomology into the construction of spectral invariants. Essentially
the same kind of construction is independently carried out by Usher
\cite{usher:talk} in a slightly less general context. Then we
explore various applications of these enhancements to the symplectic
topology, especially a new construction of symplectic quasi-states,
quasi-morphisms and new Lagrangian intersection results for toric
and non-toric manifolds by exploiting various explicit calculations
involving the potential function that the
authors made in a series of
recent papers on the Lagrangian Floer theory of toric manifolds
\cite{fooo:toric1,fooo:bulk, fooo:toricmir}.
\par
The most novel part of this memoir is our usage of the open-closed
Gromov-Witten-Floer theory (more specifically, the usage of operator
$\frak q$, which is a morphism between the quantum cohomology of
$(M,\omega)$ and the Hochschild cohomology of Lagrangian Floer
cohomology in \cite{fooo:book1} and its variant involving closed
orbits of periodic Hamiltonian system), which connect spectral
invariants (with bulk deformation), symplectic quasi-states,
quasi-morphism to the Lagrangian Floer theory (with bulk deformation)
and open-closed Gromov-Witten theory.
\par
We use this open-closed Gromov-Witten-Floer theory to produce new
examples. Especially, using the calculation of Lagrangian Floer
cohomology and the critical point theory of the potential function with
a bulk deformation, we produce examples of compact symplectic manifolds
$(M,\omega)$ that admit a continuum of linearly independent quasi-morphisms
$\widetilde{\text{\rm Ham}}(M,\omega) \to \R$.
\par
We also prove a new intersection result for the Lagrangian submanifold in
$S^2 \times S^2$ discovered in \cite{fooo:S2S2}.
\par
Many of these applications were announced
in \cite{fooo:toric1,fooo:bulk,fooo:S2S2}.
\par
On the technical side, we exclusively use the de Rham version of
various constructions and calculations performed in the relevant
Gromov-Witten-Floer theory. We adopt the de Rham version in this
memoir partly to make the exposition consistent with the series of
our papers studying toric manifolds  \cite{fooo:toric1,fooo:bulk, fooo:toricmir}.
Another reason is that many proofs become simpler when we use the de
Rham version of the constructions. We have no doubt that all the
results can be proved by using singular homology. Then we can use
$\Q$ instead of $\R$ as the underlying
ground field.

\par
\bigskip
November 2011

\begin{flushright}
Fukaya, Oh, Ohta, Ono
\end{flushright}
\par\newpage

\section{Introduction}
\label{sec:introduction}

\subsection{Introduction}
\label{subsec:introduction}

Let $(M,\omega)$ be a compact symplectic manifold.
We consider one-periodic nondegenerate Hamiltonians $H: S^1 \times M
\to \R$, not necessarily normalized, and one-periodic family
$J = \{J_t\}_{t \in S^1}$ of almost complex structures compatible
with $\omega$. To each given such pair $(H,J)$, we can associate the
Floer homology $HF(H,J)$ by considering the perturbed
Cauchy-Riemann equation
\be\label{eq:CRHJ} \dudtau + J_t \left(\dudt
- X_{H_t}(u)\right) = 0,
\index{perturbed Cauchy-Riemann equation}
\ee
where $H_t(x) = H(t,x)$ and
$X_{H_t}$ is the Hamiltonian vector field associated to $H_t \in
C^\infty(M)$.
The associated chain complex
$(CF(M,H), \del_{(H,J)})$ is generated by
certain equivalence classes of
the pairs $[\gamma,w]$ where $\gamma$ is a loop satisfying
$\dot \gamma(t) = X_{H_t}(\gamma)$, $w: D^2 \to M$ is a disc with $ w\vert_{\del D^2} = \gamma$
and
$[\gamma,w]$ is the homotopy class relative to the boundary $\gamma$.
See Definitions \ref{Lambdahatonashi} and \ref{Lambda(G)}.
The boundary operator $\del_{(H,J)}$ is defined by an appropriate weighted
count of the solution set of the equation (\ref{eq:CRHJ}).
This chain complex carries a natural
downward filtration provided by the action functional
\begin{equation}\label{mathcal A_H}
\CA_H([\gamma,w]) = - \int w^*\omega - \int_0^1 H(t,\gamma(t))\, dt,
\index{action functional}\index{$\mathcal A_H$}
\end{equation}
since \eqref{eq:CRHJ} is the negative $L^2$-gradient flow of $\CA_H$ with respect to
the $L^2$-metric on $\CL(M)$.
\par

The homology group of $(CF(M,H), \del_{(H,J)})$
is the Floer homology $HF_*(M;H,J)$
associated to the one-periodic Hamiltonian $H$. $HF_*(M;H,J)$ is known to
be isomorphic to the ordinary homology $H_*(M)$ of $M$ with appropriate Novikov
field as its coefficient ring.
(\cite{floer:cmp}).
The isomorphism is given by the Piunikhin isomorphism \cite{Piu94}
\be\label{eq:flat}
{\mathcal P}_H: H_*(M) \to HF_*(M;H,J).
\ee
We use it to transfer a quantum cohomology class $a \in QH^*(M)\cong H^*(M)$
to a Floer homology class via the map $a \mapsto  a^\flat_H: = {\mathcal P}_H \circ \flat(a)$
where $\flat$ is the Poincar\'e duality: $QH^*(M) \to H_*(M)$.

The spectral invariants constructed by the second named author in \cite{oh:alan}
for the general non-exact case are defined as follows. (See \cite{viterbo,oh:cag1,schwarz}
for the earlier related works for the exact case.) For a given quantum cohomology class
$a$, we consider Floer cycles $\alpha$ representing the associated Floer homology class $a^\flat_H$.
Then we take the mini-max value
\bea
\lambda_H(\alpha) & = & \max\{\CA_H([\gamma_i,w_i]) \mid
\alpha = \sum a_i [\gamma_i,w_i], a_i \in \C \setminus \{0\}\},\label{eq:rhoHa}\\
\rho(H;a) & = & \inf \{\lambda_H(\alpha) \mid
\del_{(H,J)}(\alpha) = 0, [\alpha] = a^\flat_H \}.
\label{eq:rhoHa2}
\eea
\index{spectral invariant}
\par
We can show that the right hand side of (\ref{eq:rhoHa2}) is independent on the choice of $J$.
So $J$ is not included in the notation $\rho(H;a)$. Via the $C^0$-continuity of the function
$H \mapsto \rho(H;a)$, the function extends continuously to arbitrary
continuous function $H$.

It induces an invariant of an element of the universal cover of the
group of Hamiltonian diffeomorphisms as follows.
We denote by ${\rm Ham}(M,\omega)$ the group of Hamiltonian diffeomorphisms
of $M$ and by $\widetilde{\rm Ham}(M,\omega)$ its universal cover \index{$\widetilde{\rm Ham}(M,\omega)$} and
by $\phi_H: t \mapsto \phi_H^t$ the Hamiltonian path (based at the identity)
generated by the (time-dependent) Hamiltonian $H$ and its time one map by
$\psi_H = \phi_H^1 \in {\rm Ham}(M,\omega)$.
Each Hamiltonian $H$ generates the Hamiltonian path $\phi_H$ which in turn determines
an element $\widetilde \psi_H=[\phi_H] \in \widetilde{\rm Ham}(M,\omega)$. Conversely,
each smooth Hamiltonian path $[0,1] \to {\rm Ham}(M,\omega)$ based at
the identity is generated by a unique normalized Hamiltonian $H$, i.e.,
$H$ satisfying
\be\label{eq:normalized}
\int_M H_t \, \omega^n = 0.
\ee
\index{Hamiltonian!normalized Hamiltonian}
It is proved in \cite{oh:alan, oh:minimax} that $\rho(H;a)$
for normalized Hamiltonians $H$ depends only on
the homotopy class $\widetilde\psi = \widetilde\psi_H$ of the path $\phi_H$
and $a$, which we denote by
\begin{equation}\label{rhofordiff}
\rho(\widetilde\psi;a) := \rho(H;a).
\end{equation}
This homotopy invariance is proved for the rational $(M,\omega)$
in \cite{oh:alan} and for the irrational case in \cite{oh:minimax,usher:specnumber} respectively.

In a series of papers \cite{EP:morphism,EP:states,EP:rigid},
Entov and Polterovich discovered remarkable applications of these spectral invariants
to the theory of symplectic intersections and to the
study of $\mathrm{Ham}(M,\omega)$ by combining ideas from dynamical systems,
function theory and quantum cohomology. We briefly summarize their construction now.
(As we explain in Subsection \ref{subsec:conv} we adopt conventions slightly different  from
Entov and Polterovich's below. However there is no essential mathematical difference.)

Let $QH^*(M;\Lambda)$ be the quantum cohomology ring of $M$.
(Here $\Lambda$ is the (universal) Novikov field. See Notation and Convention (\ref{item16}).)
They considered idempotent elements $e$ of $QH^*(M;\Lambda)$, i.e.,
those satisfying $e^2 = e$
and defined the function
$\zeta_e: C^0(M) \to \R$ by
\begin{equation}\label{asymptoticrho0}
\zeta_e(H) :=  \lim_{n \to \infty} \frac{\rho(n H;e)}{n}
\end{equation}
\index{partial symplectic quasi-state}
for autonomous $C^\infty$ functions $H$ and then extended
its definition to $C^0(M)$ by
continuity.

It is proved in \cite{EP:morphism,EP:states, EP:rigid} that
$\zeta_e$ satisfies most of the properties of
\emph{quasi-states} introduced by Aarnes \cite{aarnes:adv} and
introduced the notion of a \emph{partial symplectic quasi-state}.
They also formulated the notions of
\emph{heavy} and \emph{super-heavy} subsets of symplectic manifolds.
Entov and Polterovich also considered
the map $\mu_e: \widetilde{\Ham}(M,\omega) \to \R$  by
\begin{equation}\label{eq:mue}
\mu_e(\widetilde \psi) = - \vol_\omega(M) \lim_{n \to \infty} \frac{\rho(\widetilde \psi^n;e)}{n}.
\end{equation}
Whenever $e$ is the unit of a direct factor of $QH^*(M;\Lambda)$, which is a field,
$\mu_e$ becomes a homogeneous quasi-morphism. Namely it satisfies
\begin{equation} \label{eq:quasi-morphism}
|\mu_e(\widetilde\psi_1) + \mu_e(\widetilde\psi_2) -
\mu_e(\widetilde\psi_1\widetilde\psi_2)| <  C,
\end{equation}
for some constant $C$ independent of $\widetilde\psi_1$, $\widetilde\psi_2$ and
\be\label{eq:homo-mue}
\mu_e(\widetilde\psi^n) = n\mu_e(\widetilde\psi), \qquad \text{for $n \in \Z$.}
\ee
These facts were proved by Entov-Polterovich \cite{EP:morphism} in case  $QH^*(M;\Lambda)$
is semi-simple and $M$ is monotone. The monotonicity assumption was improved
by Ostrover \cite{ostrober2}. It was observed by McDuff that instead of the semi-simplicity
assumption one has only to assume that $e$ is the unit of a field factor of $QH^*(M;\Lambda)$.
In fact, Entov and Polterovich prove several other
symplectic properties of $\mu_e$, and call them \emph{Calabi quasi-morphisms}.
\par
One can generalize the constructions mentioned above by involving
bulk deformations.
 Here bulk deformation means a deformation of
various Floer theories and others by using a cohomology class ${\frak b}$
of the ambient symplectic manifold.
We call ${\frak b}$ a bulk cycle and denote the corresponding deformed Floer homology by $HF^{\frak b}(M;H,J)$
which we call a Floer homology with bulk.
(We will define it in Part \ref{part:bulk-hamFloer}.)
Actually as in the case of $HF(M;H,J)$, the Floer homology with bulk $HF^{\frak b}(M;H,J)$ as a $\Lambda$ module is also isomorphic
to the ordinary homology group $H(M;\Lambda)$ of $M$.
The generalization is straightforward way except the following point:
\par
We note that for the construction of partial quasi-states and quasi-morphism the
following triangle inequality of spectral invariant plays an important role.
\begin{equation}\label{productprop}
\rho(\widetilde\psi_1\circ\widetilde\psi_2,a\cup_Q b) \le \rho(\widetilde\psi_1,a) + \rho(\widetilde\psi_2,b),
\end{equation}
here $\cup_Q$ is the product of the small quantum cohomology ring $QH^*(M;\Lambda)$
and $\rho$ is the spectral invariant (\ref{rhofordiff}) (without bulk deformation).
\par
We generalize the definitions of (\ref{eq:flat}),
(\ref{eq:rhoHa2}),
(\ref{rhofordiff}) to obtain corresponding ones ${\mathcal P}_H^{\frak b}$,
$\rho^{\frak b}(H;a)$,
$\rho^{\frak b}(\widetilde\psi;a)$, respectively.
Then (\ref{productprop})  becomes
\begin{equation}\label{productprop2}
\rho^{\frak b}(\widetilde\psi_1\circ\widetilde\psi_2,a\cup^{\frak b} b) \le
\rho^{\frak b}(\widetilde\psi_1,a) + \rho^{\frak b}(\widetilde\psi_2,b),
\end{equation}
where $\cup^{\frak b}$ is the deformed cup product by $\frak b$. (See
Definition \ref{defn:prod} for its definition.)
Thus in place of the small quantum cohomology ring $QH^*(M;\Lambda)$
the $\frak b$-deformed quantum cohomology ring
(which we denote by $QH_{\frak b}^*(M;\Lambda)$) plays an important role here.
\par
Whenever $e \in QH^*_{\frak b}(M;\Lambda)$ is an idempotent, we define
\begin{equation}\label{partialsympstate}
\zeta_e^{\frak b}(H) = - \lim_{n\to \infty}\frac{\rho^{\frak b}(nH;e)}{n}
\end{equation}
for the autonomous function $H = H(x) \in C^\infty(M)$
which in turn defines a partial symplectic quasi-state on $C^0(M)$.
See Definition \ref{defpscquasi-state} and Theorem \ref{thm:state}.
Similarly we can define $\mu_e^{\frak b}: \widetilde{\Ham}(M,\omega) \to \R$. \index{$\mu_e^{\frak b}$}
We will call any such partial quasi-state or quasi-morphism obtained from
the spectral invariants a \emph{spectral partial quasi-state} or
a \emph{spectral quasi-morphism} respectively.
\index{spectral partial quasi-state}\index{quasi-morphism!spectral quasi-morphism}

\begin{thm}\label{existquasihomo}
Let $\Lambda e \cong \Lambda$ be a direct factor of $QH^*_\frak b(M;\Lambda)$
and $e$ its unit. Then
$$
\mu_e^\frak b : \widetilde{\rm Ham}(M,\omega) \to \R
$$
is a homogeneous Calabi quasi-morphism.
\end{thm}
Theorem \ref{existquasihomo} is proved in Section \ref{sec:const-morphism}.
In particular, combined with the study of big quantum cohomology of
toric manifolds \cite{fooo:toricmir}, this implies the following:
(The proof is completed in Subsection \ref{toricexistqh}.)

\begin{cor}\label{existtoric}
For any compact toric manifold $(M,\omega)$, there exists a
nontrivial homogeneous Calabi quasi-morphism
$$
\mu_e^\frak b : \widetilde{\rm Ham}(M,\omega) \to \R.
$$
\end{cor}
We say that a quasi-morphism is {\it nontrivial} if it is not bounded.
Corollary \ref{existtoric} is  also proved independently by Usher \cite{usher:talk}.

It is in general very hard to calculate spectral invariants and
partial quasi-states or quasi-morphisms obtained therefrom.
In Part 4 of this paper we provide a means of estimating them
in certain cases. We recall the following definition:

\begin{defn}\label{heavy1}(Entov-Polterovich \cite{EP:rigid})
Let $\zeta : C^0(M) \to \R$ be any partial quasi-state.
A closed subset $Y \subset X$ is called \emph{$\zeta$-heavy} \index{heavy subset!$\zeta$-heavy} if
\be
\zeta(H) \leq \sup \{H(p) \mid p \in Y\}
\ee
for any $H \in C^0(X)$.
$Y \subset X$ is called {\it $\zeta$-superheavy} if
\be
\zeta(H) \geq \inf \{H(p) \mid  p \in Y\}
\ee
for any  $H \in C^0(X)$. \index{superheavy subset!$\zeta$-superheavy}
\end{defn}
\begin{rem}
\begin{enumerate}
\item
Entov-Polterovich  proved  in \cite[Theorem 1.4 (i)]{EP:rigid}  that superheaviness implies
heaviness for $\zeta_e$. The same can be proved for
$\zeta_e^{\frak b}$ by the same way.
\item
Actually in Definition \ref{heavy1}
we use a different sign convention from that of Entov-Polterovich.
See Remark \ref{rem114444}.
\end{enumerate}
\end{rem}

We can also define similar notions $\mu$-heaviness and $\mu$-superheaviness, for
$\mu : \widetilde{\rm Ham}(M,\omega) \to \R$, instead of
$\zeta : C^0(M) \to \R$.
See Definition \ref{tdheavy}.
In Part 4 of this paper we provide a way to use Lagrangian Floer theory
to show certain Lagrangian submanifold is $\mu$-heavy or $\mu$-superheavy
for $\mu = \mu_e$ or $\mu_e^{\frak b}$.

Let $L$ be a relatively spin Lagrangian submanifold of $M$.
In \cite{fooo:book1} we associated to $L$ a set
$\mathcal M_{\rm weak,def}(L;\Lambda_+)$, which we call the {\it Maurer-Cartan moduli space}.
(See also Definition \ref{bulkMCelement}.)

\begin{rem}
The Maurer-Cartan moduli space that appears in \cite{fooo:book1} uses
the Novikov ring $\Lambda_+$. A technical enhancement to its $\Lambda_0$-version was performed in
\cite{fooo:bulk,fukaya:cyc} using the idea of Cho \cite{cho:Bfield},
which is used in this paper. In this introduction, however, we state
only the $\Lambda_+$-version for the simplicity of exposition.
\end{rem}

The Maurer-Cartan moduli space comes with a map
$$
\pi_{\rm bulk} : \mathcal M_{\rm weak,def}(L;\Lambda_+)
\to \bigoplus_{k}H^{2k}(M;\Lambda_+).
$$
For each  $\text{\bf b} \in \mathcal M_{\rm weak,def}(L;\Lambda_+)$,
the Floer cohomology
$
HF^*((L,\text{\bf b}),(L,\text{\bf b});\Lambda_0)
$
deformed by $\text{\bf b}$ is defined in \cite[Definition 3.8.61]{fooo:book1}.
(See Definition \ref{FLoercohbulk}.)
Moreover the open-closed map
\begin{equation}\label{openclosedhomo}
i_{{\rm qm},\text{\bf b}}^{\ast} :
H^*(M;\Lambda_0) \to HF^*((L,\text{\bf b}),(L,\text{\bf b});\Lambda_0)
\end{equation}
is constructed in \cite[Theorem 3.8.62]{fooo:book1}.
(See (\ref{iqmdefformula}).)
Utilizing this map $i_{{\rm qm},\text{\bf b}}^{\ast}$, we can locate
$\mu_e^\frak b$-superheavy Lagrangian submanifolds in several circumstances.
\index{superheavy subset!$\mu$-superheavy}
\index{heavy subset!$\mu$-heavy}
\begin{thm}\label{supportmain}
Consider a pair $(\text{\bf b},\frak b)$ with
$\text{\bf b} \in \mathcal M_{\rm weak,def}(L;\Lambda_+)$ and
$\pi_{\rm bulk}(\text{\bf b}) = \frak b$.
Let $e$ be an idempotent of $QH^*_{\frak b}(M;\Lambda)$ such that
$$
i_{{\rm qm},\text{\bf b}}^{\ast}(e) \ne 0 \in HF^*((L,\text{\bf b}),(L,\text{\bf b});\Lambda).
$$
Then $L$ is $\zeta_e^\frak b$-heavy and $\mu_e^\frak b$-heavy.
\par
If $e$ is a unit of a field factor of $QH^*_{\frak b}(M;\Lambda)$ in addition,
then $L$ is $\zeta_e^\frak b$-superheavy and
$\mu_e^\frak b$-superheavy.
\end{thm}

See Definition \ref{tdheavy} for the definitions of
$\mu_e^\frak b$-heavy and $\mu_e^\frak b$-superheavy sets.
Theorem \ref{supportmain} (Theorem \ref{thm:heavy}) is proved in Section \ref{sec:heavy}.
\begin{rem}
\begin{enumerate}
\item
Theorem \ref{supportmain} gives rise to a proof of a conjecture made in \cite[Remark 1.7]{fooo:toric1}.
\item
Theorem \ref{supportmain} is closely related to  \cite[Theorem 1.20]{EP:rigid}.
\end{enumerate}
\end{rem}

Theorem \ref{supportmain} also proves linear independence of some spectral
Calabi quasi-morphisms in the following sense.

\begin{defn}
Let
$$
\mu_j : \widetilde{\rm Ham}(M,\omega) \to \R
$$
be homogeneous Calabi-quasi-morphisms for $j=1,\dots,N$.
We say that they are {\it linearly independent} if there exists
a subgroup $\cong \Z^N$ of $\widetilde{\rm Ham}(M,\omega)$ such that
the restriction of
$(\mu_1,\dots,\mu_N):  \widetilde{\rm Ham}(M,\omega) \to \R^N$ to this subgroup
is an isomorphism to a lattice in $\R^N$.
A (possibly infinite) set of quasi-morphisms of $\widetilde{\rm Ham}(M,\omega)$ is
said to be {\it linearly independent} if any of its finite subset is
linearly independent in the above sense.
The case of ${\rm Ham}(M,\omega)$ can be defined in the same way.
\end{defn}

\begin{cor}\label{lieind}
Let $L_j$ be mutually disjoint relatively spin Lagrangian submanifolds.
($j = 1,\dots,N$.)
Let $\frak b_j \in  H^{{\rm even}}(M;\Lambda_+)$ and
$\text{\bf b}_j \in \mathcal M_{\rm weak,def}(L_j;\Lambda_+)$
with $\pi_{\rm bulk}(\text{\bf b}_j) = \frak b_j$.
Let $e_j$ be the unit of a field factor of $QH_{\frak b_j}^*(M,\Lambda)$ such that
$$
i^*_{{\rm qm},\text{\bf b}_j}(e_j) \ne 0 \in HF^*((L_j,\text{\bf b}_j),(L_j,\text{\bf b}_j);\Lambda),
\quad j=1,\ldots,N.
$$
Then $\mu_{e_j}^{\frak b_j}$ ($j=1,\ldots,N$) are linearly independent.
\end{cor}

This corollary follows from Theorem \ref{supportmain} mentioned above and \cite[Theorem 8.2]{EP:rigid}.
(See also Section \ref{sec:independence} of this paper.)

The study of toric manifolds \cite{fooo:bulk} and deformations of some toric orbifolds
\cite{fooo:S2S2} provides examples for which the hypothesis of Corollary \ref{lieind}
is satisfied. This study gives rise to the following theorem

\begin{thm}\label{uncount}
Let $M$ be one of the following three kinds of symplectic manifolds:
\begin{enumerate}
\item $S^2\times S^2$ with monotone toric symplectic structure,
\item
Cubic surface,
\item
$k$ points blow up of $\C P^2$ with a certain toric symplectic structure,
where $k\ge 2$.
\end{enumerate}
Then $(M,\omega)$ carries a collection $\{\mu_a\}_{a \in \frak A}$
of linearly independent Calabi quasi-morphisms
$$
\mu_a : \widetilde{\rm Ham}(M,\omega) \to \R
$$
for an uncountable indexing set $\frak A$.
\end{thm}
\begin{rem}
\begin{enumerate}
\item In the case of $(M,\omega) = S^2\times S^2$, we have
quasi-morphisms
$
\mu_a : {\rm Ham}(M,\omega) \to \R
$
in place of
$
\mu_a : \widetilde{\rm Ham}(M,\omega) \to \R
$. See Corollary \ref{S2S2qm}.
\item
We can explicitly specify the symplectic structure used in Theorem \ref{uncount}
(3). See Section \ref{sec:exotic}.
\item We can also construct an example of the
symplectic manifold that admits an uncountable set of
linearly independent quasi-morphisms in higher dimension by
the similar way. (For example we can take direct product with
$S^2$ of the symplectic manifolds in  Theorem \ref{uncount}.)
\item
Theorem \ref{uncount} for $S^2 \times S^2$ was announced in \cite[Remark 7.1]{fooo:S2S2},
and for the case of $k$-points ($k\ge 2$) blow up of $\C P^2$ in \cite[Remark 1.2 (3)]{fooo:bulk},
respectively.
\item
Biran-Entov-Polterovich constructed an uncountable family of linearly independent
Calabi quasi-morphisms for the case of the group ${\rm Ham}(B^{2n}(1);\omega)$
of compactly supported Hamiltonian diffeomorphisms of balls with $n \ge 2$ in
\cite{biranentovpol}. Theorem \ref{uncount} provides the first example of closed
$M$ with such property.
\item Existence of infinitely many linearly independent homogeneous Calabi quasi-morphisms
on ${\rm Ham}(M,\omega)$ is still an open problem for the case of $M = \C P^2$ as well as $\C P^n$
for general $n \geq 2$.
\item
Theorem \ref{uncount} implies that the second bounded cohomology of
$\widetilde{\rm Ham}(M;\omega)$ is of infinite rank for  $(M,\omega)$ appearing in
Theorem  \ref{uncount}.
More precisely, the \emph{defect} $\operatorname{Def}_a$ defined by
$$
\operatorname{Def}_a(\phi,\psi): = \mu_a(\phi) + \mu_a(\psi) -
\mu_a(\phi\psi)
$$
defines a bounded two-cocycle. \index{defect} It follows from the perfectness of the group
${\rm Ham}(M;\omega)$ \cite[Corollary 4.3.2 (1)]{banyagabook} that the set of cohomology classes of
$\{\operatorname{Def}_a\}$ is linearly independent in the 2nd bounded cohomology group of
$\widetilde{\rm Ham}(M,\omega)$. In fact, if otherwise, we can find a nonzero
linear combination of elements from $\{\operatorname{Def}_a\}$ which is a coboundary.
Namely we have
$$
\sum c_i \operatorname{Def}_{a_i}(\phi,\psi) = \rho(\phi) + \rho(\psi) -
\rho(\phi\psi)
$$
where $\rho : \widetilde{\rm Ham}(M;\omega) \to \R$ is bounded and $c_i \in \R$.
Therefore
\begin{equation}\label{form1.25}
\sum c_i \mu_{a_i} - \rho : \widetilde{\rm Ham}(M;\omega) \to \R
\end{equation}
is a homomorphism. The linear independence of $\mu_{a_i}$
implies that $\sum c_i \mu_{a_i}$ is unbounded. Therefore (\ref{form1.25})
defines a nontrivial homomorphism, which contradicts the
perfectness of $\widetilde{\rm Ham}(M,\omega)$.
\item
At the final stage of completing this paper,
a paper \cite{Borman}
appears in the arXiv which discusses a result related to Theorem \ref{uncount} (3)
using \cite{abreu}.
\end{enumerate}
\end{rem}

Another corollary of Theorem \ref{supportmain} combined with Theorem 1.4 (iii),
 \cite[Theorem 1.8]{EP:rigid} is the following intersection result of
the exotic Lagrangian tori discovered in \cite{fooo:S2S2}.

\begin{thm}\label{dips2s2}
Let $T(u) \subset S^2(1) \times S^2(1)$ for $0 <u \leq 1/2$ be
the tori from \cite{fooo:S2S2}. Then we have
$$
\psi(T(u)) \cap (S^1_{\text{\rm eq}} \times S^1_{\text{\rm eq}}) \neq \emptyset
$$
for any symplectic diffeomorphism $\psi$ of $S^2(1) \times S^2(1)$.
\end{thm}
\begin{rem}
Theorem \ref{dips2s2} was announced in the introduction of \cite{fooo:S2S2}.
The proof is given in
Subsection \ref{subsec:POb-T(u)}.
\end{rem}

A brief outline of the content of the paper is now in order. The present paper
consist of 7 parts.
Part 1 is  a review. In Part 2, we first enhance the Hamiltonian Floer theory by
involving its deformations by ambient cohomology classes, which we call
bulk deformations.
In this paper, we
use de Rham (co)cycles instead of singular cycles as in \cite{fooo:bulk,fooo:toricmir}.
After this enhancement, we generalize construction of spectral invariants
in \cite{oh:alan} involving bulk deformations and define \emph{spectral invariants with bulk}.
Part 3 then generalizes construction \cite{EP:states,EP:morphism} of symplectic partial quasi-states and
Calabi quasi-morphisms by replacing the spectral invariants defined in \cite{oh:alan}
by these spectral invariants with bulk.
\par
In the course of carrying out these
enhancements, we also unify, clarify and enhance many known constructions in
Hamiltonian Floer theory in the framework of virtual fundamental chain technique.
(We use its version based on Kuranishi structures and
accompanied abstract perturbation theory originally established in
\cite{fukaya-ono} and further enhanced in \cite[Appendix A.2]{fooo:book2}.)
Variants of virtual fundamental chain technique, some of which we use in this paper,
are given in
\cite{fooo:bulk,fooo:toricmir,fukaya:cyc}.
Explanation of further detail of virtual fundamental chain technique is provided in \cite{fooo:techI,fooo:tech2}.
Systematic application of  virtual technique is needed particularly
because many constructions related to the study of spectral invariants (with bulk)
have to be done in the chain level, not just in homology. Examples of such enhancement
include construction of pants product \cite{schwarz1} and Piunikhin isomorphism
whose construction was outlined in \cite{Piu94,rt, pss}. We give a complete construction of
both of these in general compact symplectic manifolds without assuming any conditions
on $(M,\omega)$ such as semi-positivity or rationality.
\par
In Part 4, we connect the study of spectral invariants to the Lagrangian Floer
theory developed in \cite{fooo:book1,fooo:book2}. The main construction in the study
is based on open-closed Gromov-Witten theory developed in \cite[Section 3.8]{fooo:book1},
which induces a map from the quantum cohomology of the ambient
symplectic manifolds to the Hochschild cohomology of $A_\infty$ algebra (or more
generally that of Fukaya category of $(M,\omega)$). This map was defined in
\cite{fooo:book1} and further studied in \cite{fooo:bulk}, \cite[Section 2.6]{fooo:toricmir}
and etc..
This part borrows
much from \cite{fooo:book1,fooo:book2,fooo:toricmir} in its exposition.
The main new ingredient is a construction of a map from Floer homology of periodic Hamiltonians
to Floer cohomology of Lagrangian submanifold, through which the map from quantum cohomology
to Floer cohomology of Lagrangian submanifold factors
(Subsection \ref{subsec:PIirelation}).
We also study its properties especially those related to the
filtration. A similar construction was used by Albers \cite{albers} and also by Biran-Cornea
\cite{biran-cor} in the monotone context, which has been exploited by
Entov-Polterovich \cite{EP:rigid} in their study of symplectic intersections.
\par
In Part 5, we combine the results obtained in the previous parts together with
the results on the Lagrangian Floer theory of toric manifolds obtained in
the series of our previous papers \cite{fooo:toric1,fooo:bulk,fooo:S2S2,fooo:toricmir},
give various new constructions of Calabi quasi-morphisms and new
Lagrangian intersections results on toric manifolds and other
K\"ahler surfaces. These results are obtained by detecting the heaviness
of Lagrangian submanifolds in the sense of Entov-Polterovich \cite{EP:rigid}
in terms of spectral invariants, critical point theory of potential functions
and also the closed-open map from quantum cohomology to Hochschild cohomology
of $A_\infty$-algebra of Lagrangian submanifolds.
\par
Finally in Part 6, we prove various technical results necessary to
complete the constructions carried out in the previous parts. For
example, we establish the isomorphism property
of the Piunikhin map with bulk. We give the construction of Seidel homomorphism
with bulk extending the results of \cite{seidel:auto} and generalize the
McDuff-Tolman's representation of quantum cohomology ring of toric manifolds
in terms of Seidel elements \cite{mc-tol} to that of big quantum cohomology ring.

Some portion of the present paper is devoted to proving various results in
Hamiltonian Floer theory, spectral invariants,
Entov-Polterovich theory of symplectic quasi-states and others
for arbitrary compact symplectic manifolds.
These proofs are often similar to those in the literature except we apply virtual
fundamental chain technique (via the Kuranishi structure) systematically.
In fact, most of the literature assume semi-positivity since they do not use  virtual
fundamental chain technique. Even when virtual fundamental chain technique
is used, not enough details on the way how to apply the technique are provided.
Because of these reasons, for readers's convenience and for the completeness' sake,
we provide a fair amount of these details on the proofs in the literature
in a unified and coherent fashion in the most general context using the
framework of Kuranishi structure and virtual fundamental chain technique, without imposing any restrictions on the
ambient symplectic manifold $(M,\omega)$.
We include a brief summary of the theory of Kuranishi structure
and its perturbation as Part \ref{part7} for reader's convenience.
\par
On the other hand, the most novel part of this paper
is to combine the story of spectral invariants and Entov-Polterovich theory
with that of bulk deformations, Lagrangian Floer theory
and the calculation of them in the toric case to obtain new examples,
especially those appearing in Theorem \ref{uncount}.

\par\bigskip

\subsection{Notations and Conventions}
\par\smallskip
We follow the conventions
of \cite{oh:alan,oh:minimax,oh:hameo2} for the definition of
Hamiltonian vector fields and action functional and others appearing
in the Hamiltonian Floer theory and in the construction of spectral
invariants and Entov-Polterovich's Calabi quasi-morphisms. There are
differences from e.g., those used in
\cite{EP:morphism,EP:states,EP:rigid} one way or the other. (See
Subsection \ref{subsec:conv} for the explanation of the differences.)

\begin{enumerate}
\item The Hamiltonian vector field $X_H$ is defined by
$dH = \omega(X_H,\cdot)$.
\item The flow of $X_H$ is denoted by $\phi_H: t \mapsto \phi_H^t$ and
its time-one map by $\psi_H = \phi_H^1 \in \Ham(M,\omega)$.
We say that $H$ or its associated map $\psi_H$ is {\it nondegenerate}\index{nondegenerate!Hamiltonian}
if at $p \in \operatorname{Fix}\psi_H$, the differential
$d_p\psi_H : T_p M \to T_pM$ does not have eigenvalue $1$.
We denote by $\Ham _{{\rm nd}}(M,\omega)$\index{$\text{\rm Ham}_{{\rm nd}}(M,\omega)$}
the subset of $\Ham(M,\omega)$ consisting of
nondegenerate $\psi_H$'s.
\item We denote by $[\phi_H]$ the path homotopy class of $\phi_H:[0,1] \to \Ham(M,\omega)$
relative to the ends which we generally denote $\widetilde \psi_H =
[\phi_H]$. We denote by $z^p_H(t) = \phi_H^t(p)$ the solution
associated to a fixed point $p$ of $\psi_H = \phi_H^1$.
\item $\widetilde H(t,x) = -H(1-t,x)$ is the
\emph{time-reversal} \index{time-reversal!Hamiltonian}
Hamiltonian generating
$\phi_H^{1-t}\phi_H^{-1}$.
\item We denote by $H_1 * H_2$ the Hamiltonian generating
the \emph{concatenation} \index{concatenation!Hamiltonian} of the two Hamiltonian paths $\phi_{H_1}$
followed by $\phi_{H_2}^t$. More explicitly, it is defined by
$$
(H_1 * H_2)(t,x) = \begin{cases}2H_1(2t,x) \quad & 0 \leq t\leq 1/2 \\
2H_2(2t-1,x) \quad & 1/2 \leq t \leq 1.
\end{cases}
$$
\item The action functional $\CA_H: \widetilde \CL_0(M) \to \R$ is defined by
$$
\CA_H([\gamma,w]) = -\int w^*\omega - \int_0^1 H(t,\gamma(t))\,d t.
$$
\item
We denote by ${\text{Per}(H)}$ the set of
periodic orbits of periodic Hamiltonian system associated to the
Hamiltonian $H : [0,1] \times M \to \R$.
$\text{\rm Crit}(\mathcal A_H)$ is the set of
critical points of $\CA_H$.  Its element is denoted by
$[\gamma,w]$ where $\gamma \in {\text{Per}(H)}$
and $w$ is a disk which bounds $\gamma$.
We identify $[\gamma,w]$ and $[\gamma,w']$ if $w$ is
homotopic to $w'$.
We define $[\gamma,w] \sim [\gamma,w']$ if
$\int w^*\omega = \int (w')^*\omega$ and
denote by $\widehat{\text{Per}}(H)$
the set of the equivalence classes.
An element of $\widehat{\text{Per}}(H)$ is denoted by
$\llb \gamma,w \rrb$.
There exists a map
$$
\text{\rm Crit}(\CA_H)
\overset{\pi}\longrightarrow
\widehat{\text{\rm Per}}(H)
\longrightarrow
{\text{\rm Per}}(H)
$$
defined by $[\gamma,w] \mapsto \llb \gamma,w \rrb$,
$\llb \gamma,w \rrb \mapsto  \gamma$.
\item $\CJ_\omega$ = the set of $\omega$-compatible almost complex structures.
$j_\omega = \CL(\CJ_\omega)$ = the set of $S^1$-family $J$ of compatible
almost complex structures;
$J =\{J_t\}_{t \in S^1}$.
\item $\CP(j_\omega)$ = $\operatorname{Map}([0,1] \times S^1, \CJ_\omega)$; \quad
$(s,t)  \in [0,1] \times S^1 \mapsto J^s_t \in \CJ_\omega$.
\item $\CK = \{\chi: \R \to [0,1]\}$ where $\chi$ is a smooth function with
$\chi'(\tau) \geq 0$, $\chi(\tau) \equiv 0$ for $\tau \leq 0$ and $\chi(\tau) \equiv 1$ for
$\tau \geq 1$. We define $\tilde \chi$ by $\tilde \chi = 1 -\chi$.
\item For given $H \in C^\infty(S^1 \times M,\R)$,
we define the $\R$-family $H_\chi$ by
\begin{equation}\label{notationFchi}
H_\chi(\tau,t,x) = \chi(\tau)H(t,x).
\end{equation}
\item
For $J\in \CP(j_\omega)$ we take $J_s =\{J_{s,t} ; t\in S^1\}$
such that
$$
J_{1,t} = J_t, \quad J_{0,t} = J_0,\quad J_{s,0} = J_0,
$$
and put
$$
J_{\chi}(\tau,t) = J_{\chi(\tau),t}.
$$
\item
If $H \in C^\infty([0,1] \times S^1 \times M,\R)$
and
$J\in \CP(j_\omega)$, we put
\begin{equation}\label{notationHchi}
H^\chi(\tau,t,x) = H(\chi(\tau),t,x), \quad \, J^\chi(\tau,t,x) = J(\chi(\tau),t,x).
\end{equation}
\item Let $\Omega_*(M) \widehat \otimes \Lambda^{\downarrow}$ be the completion of
the algebraic tensor product $\Omega_*(M) \otimes \Lambda^{\downarrow}$ with respect to
the metric induced by the valuation $\frak v_q$ on $\Lambda^\downarrow$ defined in
item (18) below.
The {\it Piunikhin chain map}
$$
\CP^{\frak b}_{(H_\chi,J_\chi)}: \Omega_*(M) \widehat \otimes \Lambda^{\downarrow} \to
CF_*(M,H;\Lambda^\downarrow)
$$
is associated to $(H_\chi,J_\chi)$ in (11),(12).
(See Section \ref{sec:deform-bdy}).
$\Lambda^\downarrow$ is defined in item (17).
The map
$$
\CQ^{\frak b}_{(H_{\tilde \chi},J_{\tilde \chi})}:
CF_*(M,H;\Lambda^\downarrow) \to \Omega_*(M) \widehat \otimes \Lambda^{\downarrow}
$$
is associated to $\tilde\chi(\tau) = \chi(1-\tau)$.
(See Section \ref{sec:appendix1}.)
\item \index{shuffle}
We denote the set of shuffles of $\ell$ elements by
\begin{equation}\label{shuff0}
\text{\rm Shuff}(\ell) = \{ (\mathbb L_1,\mathbb L_2) \mid
\mathbb L_1 \cup \mathbb L_2 = \{1,\ldots,\ell\}, \,\,\mathbb L_1 \cap
\mathbb L_2 = \emptyset \}.
\end{equation}
For $(\mathbb L_1,\mathbb L_2) \in \text{\rm Shuff}(\ell)$
let
$\#\mathbb L_i$ be the order of this subset.
Then $\#\mathbb L_1 + \#\mathbb L_2 = \ell$.
\par
The set of {\it triple shuffles} is the set of
$(\mathbb L_1,\mathbb L_2,\mathbb L_3)$ such that
$\mathbb L_1 \cup \mathbb L_2 \cup \mathbb L_3 = \{1,\ldots,\ell\}$
and that $\mathbb L_1,\mathbb L_2,\mathbb L_3$ are mutually disjoint.
\item\label{item16}
 \index{universal Novikov ring}\index{universal Novikov field}
The universal Novikov ring $\Lambda_0$ and its filed $\Lambda$  of fractions are defined by
$$
\aligned
\Lambda_0
& = \left.\left\{ \sum_{i=1}^{\infty} a_i T^{\lambda_i} ~\right\vert~ a_i \in \C, \lambda_i \in \R_{\ge 0},
\lim_{i\to\infty} \lambda_i = +\infty \right\}, \\
\Lambda
& = \left.\left\{ \sum_{i=1}^{\infty} a_i T^{\lambda_i} ~\right\vert~ a_i \in \C, \lambda_i \in \R,
\lim_{i\to\infty} \lambda_i = +\infty \right\}
\cong \Lambda_0[T^{-1}].
\endaligned
$$
The maximal ideal of $\Lambda_0$ is denoted by
$$
\Lambda_+
= \left.\left\{ \sum_{i=1}^{\infty} a_i T^{\lambda_i} ~\right\vert~ a_i \in \C, \lambda_i \in \R_{> 0},
\lim_{i\to\infty} \lambda_i = +\infty \right\}.
$$
We define the valuation $\frak v_T$\index{valuation}\index{valuation!$\frak v_T$} on $\Lambda$ by\index{$\frak v_T$}
$$ \frak v_T \left(\sum_{i=1}^{\infty} a_i T^{\lambda_i}\right)
  = \inf \{ \lambda_i \mid a_i \ne 0 \}, \quad
  \frak v_T (0) = +\infty.
$$
\item\label{item17}
We also use the following (downward) Novikov ring $\Lambda^\downarrow_0$ and field $\Lambda^\downarrow$:
$$
\aligned \Lambda^\downarrow_0
& = \left.\left\{ \sum_{i=1}^{\infty} a_i q^{\lambda_i} ~\right\vert~ a_i \in \C, \lambda_i \in \R_{\le 0},
\lim_{i\to\infty} \lambda_i = -\infty \right\}, \\
\Lambda^\downarrow
& = \left.\left\{ \sum_{i=1}^{\infty} a_i q^{\lambda_i} ~\right\vert~ a_i \in \C, \lambda_i \in \R,
\lim_{i\to\infty} \lambda_i = -\infty \right\}
\cong \Lambda^\downarrow_0[q].
\endaligned
$$
The maximal ideal of $\Lambda^\downarrow_0$ is denoted by
$$
\Lambda^\downarrow_-
= \left.\left\{ \sum_{i=1}^{\infty} a_i q^{\lambda_i} \in \Lambda^\downarrow ~\right\vert~ \lambda_i < 0 \right\}.
$$
\item
We define the valuation $\frak v_q$\index{valuation!$\frak v_q$}  on $\Lambda^\downarrow$ by\index{$\frak v_q$}
$$
 \frak v_q \left(\sum_{i=1}^{\infty} a_i q^{\lambda_i}\right)
  = \sup \{ \lambda_i \mid a_i \ne 0 \}, \quad
  \frak v_q (0) = -\infty.
$$
Of course, $\Lambda^\downarrow_0$ and $\Lambda^\downarrow$ are isomorphic to $\Lambda_0$ and
$\Lambda$ respectively by the isomorphism $q \mapsto T^{-1}$.
Under the isomorphism we have $\frak v_q = -\frak v_T$.
The downward universal Novikov rings seem to be more commonly used in the study
of spectral invariant (e.g., \cite{oh:alan}), while the upward versions $\Lambda$ and $\Lambda_0$ are used
in Lagrangian Floer theory (e.g., \cite{fooo:book1,fooo:book2}).
\item
Let $V$ be a $\Z$ graded vector space over $\C$.
We put $
B_kV = \underbrace{V\otimes \cdots \otimes V}_{\text{$k$ times}}
$ and $BV=\bigoplus_{k=0}^{\infty}B_k V$ where
$B_0 V=\C$.
Then $BV$\index{$BV$} has a structure of coassociative
coalgebra with coproduct. We note that
we have two kinds of coproduct structures on
$BV$.
One is the {\it deconcatenation coproduct}\index{coproduct!deconcatenation} defined by
\be\label{deconcoproduct}
\Delta_{\rm decon}(x_1 \otimes \cdots \otimes x_k)
= \sum_{i=0}^k (x_1 \otimes \cdots \otimes x_i)
\otimes (x_{i+1} \otimes \cdots\otimes x_k).
\ee
The other is the {\it shuffle coproduct}\index{coproduct!shuffle} defined by
\be\label{shuffcoproduct}
\aligned
& \Delta_{\rm shuff}(x_1 \otimes \cdots \otimes x_k) \\
& =
\sum_{(\mathbb L_1,\mathbb L_2) \in \text{\rm Shuff}(k)}
(-1)^{\ast} (x_{\ell_1(1)} \otimes \cdots \otimes x_{\ell_1(k_1)})
\otimes (x_{\ell_2(1)} \otimes \cdots\otimes x_{\ell_2(k_2)}),
\endaligned
\ee
where $\mathbb L_j=\{\ell_j(1), \dots , \ell_j(k_j)\}$ with $\ell_j(1) < \dots < \ell_j(k_j)$ for $j=1,2$ and
\be\label{shufflesignbeforshift}
\ast =
\sum _{\ell_1(i)>\ell_2(j)} \deg x_{\ell_1(i)}
\deg x_{\ell_2(j)}.
\ee
It is easy to see that
\index{$\Delta_{\rm decon}$}\index{$\Delta_{\rm shuff}$}
\be\label{eqs:exponentials}
\aligned
\Delta_{\rm decon}\left(\sum_{k=0}^{\infty} x^{\otimes k}\right)
& =\left(\sum_{k=0}^{\infty} x^{\otimes k}\right)\otimes
\left(\sum_{k=0}^{\infty} x^{\otimes k}\right) \\
\Delta_{\rm shuff}\left(\sum_{k=0}^{\infty} \frac{x^{\otimes k}}{k!}\right)
& =\left(\sum_{k=0}^{\infty} \frac{x^{\otimes k}}{k!}\right)\otimes
\left(\sum_{k=0}^{\infty} \frac{x^{\otimes k}}{k!}\right)
\endaligned
\ee
if $\deg x$ is even.
We write $e^x=\sum_{k=0}^{\infty} x^{\otimes k}$ or
$e^x=\sum_{k=0}^{\infty} \frac{x^{\otimes k}}{k!}$
according as we use $\Delta_{\rm decon}$
or $\Delta_{\rm shuff}$ as coproduct structures.
\par
Actually we only use $\Delta_{\rm decon}$ for $BV$ in this article.
However on $EV$, which we describe in item (20), we use
$\Delta_{\rm shuff}$, which is induced by $\Delta_{\rm shuff}$ on $BV$.
\par\noindent
\item
The symmetric group ${\rm Perm}(k)$ of order $k!$ acts on $B_kV$ by
$$
\sigma \cdot (x_1 \otimes \cdots \otimes x_k)
= (-1)^* x_{\sigma(1)} \otimes \cdots \otimes x_{\sigma(k)},
$$
where
$
* = \sum_{i<j; \sigma(i)>\sigma(j)} \deg x_i \deg x_j.
$
We denote by $E_kV$ the quotient of $B_kV$ by the submodule generated by
$\sigma \cdot {\bf x}-{\bf x}$ for $\sigma \in {\rm Perm}(k)$, ${\bf x} \in B_kV$.
\index{$\text{\rm Perm}(k)$}
We denote by $[{\bf x}]$ an element
of $E_kV$ and
put $EV=\bigoplus_{k=0}^{\infty}E_kV$. \index{$EV$}
The shuffle coproduct structure on $BV$ induces a coproduct structure on $EV$, which we
also denote by $\Delta_{\rm shuff}$. It is given by
\be\label{shuffcoproductEC}
\aligned
& \Delta_{\rm shuff}([x_1 \otimes \cdots \otimes x_k]) \\
& =
\sum_{(\mathbb L_1,\mathbb L_2) \in \text{\rm Shuff}(k)} (-1)^{\ast}([x_{\ell_1(1)} \otimes \cdots \otimes x_{\ell_1(k_1)}])
\otimes ([x_{\ell_2(1)} \otimes \cdots\otimes x_{\ell_2(k_2)}]).
\endaligned
\ee
Here $\ast$ is the same as \eqref{shufflesignbeforshift}.
Then $EV$ becomes a coassociative and graded
cocommutative coalgebra.
\par
In \cite{fooo:book1,fooo:toric1, fooo:bulk}, we denote by $E_kV$
the {\it ${\rm Perm}(k)$-invariant subset} of $B_kV$
and use the deconcatenation coproduct restricted to the subset.
In \cite{fooo:toricmir}, we use $E_kV$
as the {\it quotient space} and the shuffle coproduct on it as we do in this paper. This paper follows the conventions used in  \cite{fooo:toricmir}.
\par\noindent
\item
Let $L$ be a relatively spin closed Lagrangian submanifold of a symplectic manifold $(M,\omega)$.
\begin{enumerate}
\item For the case $V=\Omega (L)[1]$, we always use
the deconcatenation coproduct
$\Delta_{\rm decon}$ on $B(\Omega(L)[1])$.
\item
For the case $V=\Omega(M)[2]$, we always use
the shuffle coproduct $\Delta_{\rm shuff}$ on
$E(\Omega(M)[2])$.
\end{enumerate}
Here $\Omega (L)[1]$ (resp. $\Omega(M)[2]$) is the degree shift by
$+1$ of $\Omega(L)$, i.e., $(\Omega(L)[1])^d=\Omega ^{d+1}(L)$ (resp. $+2$ of $\Omega(M)$, i.e., $(\Omega(M)[2])^d=\Omega ^{d+2}(M)$.)
Therefore, no confusion can occur even if
we use the same notation $\Delta$ for
the coproducts
$\Delta_{\rm decon}$ and $\Delta_{\rm shuff}$.

\item
Sometimes we regard the de Rham complex $(\Omega(M),d)$ as a chain complex and
consider its homology. In that case we put
$$
\Omega_k(M) = \Omega^{\dim M -k}(M), \quad
\partial = (-1)^{\deg+1}d.
$$
See  \cite[Remark 3.5.8]{fooo:book1} for this sign convention.
When a cohomology class $a \in H^{\dim M-k}(M)$ is represented by a
differential form $\alpha$ and we regard $\alpha$ as an element
of the chain complex $(\Omega_*(M),\partial)$, we denote the homology class
by $a^{\flat} \in H_{k}(M)$.
\index{$\flat$} We identify the singular homology $H_{k}(M)$
with $H_k(\Omega_*(M),\delta)$ by identifying $a^\flat \in H_k(\Omega_*(M),\delta)$\index{$\Omega_*(M)$}
with the singular homology class represented by a cycle $C$ satisfying
$\int_C \beta = \int_M a^\flat \wedge \beta$ for all closed form $\beta$ on $M$.
\par
More generally
If $a \in H^*(M) \otimes \Lambda$ is expressed as
$$
a = \sum a_i T^{\lambda_i} \, \quad a_i \in H^*(M), \quad \lambda_i \to \infty
$$
then we define
$$
a^\flat = \sum a_i^\flat q^{- \lambda_i}, \quad -\lambda_i \to -\infty.
$$
\item We always denote \index{$QH^*(M)$} by $QH^*(M) = QH^*(M,\Lambda)$ the standard quantum cohomology ring with
Novikov ring equipped with \emph{upward} valuations. As a module, it is isomorphic to
$QH^*(M) = H^*(M) \otimes \Lambda \cong H^*(M) \otimes \Lambda^\downarrow = QH_*(M)$.
\item
When we say that the boundary orientation of some moduli space is compatible with
the orientation of strata corresponding to bubbling off disks (with boundary marked points),
the compatibility means in the sense of
\cite[Proposition 8.3.3]{fooo:book2}.

\item
In the present paper, we will exclusively use the quantum \emph{cohomology},
the Hamiltonian Floer \emph{homology} and Lagrangian Floer \emph{cohomology}.
Depending on the usage of homology or cohomology, we use different valuations $\frak v_T$
or ${\frak v}_q = -\frak v_T$ on the Novikov ring $\Lambda$. When we emphasize we use
parameter $q$ and valuation ${\frak v}_q$ we write
$\Lambda^\downarrow$.
\end{enumerate}

\subsection{Difference  between Entov-Polterovich's convention and ours}
\label{subsec:conv}

In this subsection we explain the difference between our convention and
the one of Entov-Polterovich.
The discussion of this subsection is not used anywhere else in this paper.
Entov and Polterovich (see  \cite[equation (9)]{EP:rigid}) use the action functional
\be\label{mathcalA_HEP}
\widetilde{\CA}_H([\gamma, w])=- \int w^* \omega + \int_0^1 H(t,\gamma(t))\, dt
\ee
instead of \eqref{mathcal A_H}. We note the relationship
\be\label{eq:actionEP}
\widetilde{\CA}_H = \CA_{-H}.
\ee
They use it to define $c(a^\flat;H)$ for a quantum homology class
$a^\flat \in H_*(M;\Lambda^{\downarrow})$.
(They do it in the same way as $\lambda_H$ given in (\ref{eq:rhoHa}) except the
difference of convention mentioned above.)
Put
$$
\rho^{EP}(H;a) = - c(a^\flat;H).
$$
and
define the function
$\zeta^{EP}_e: C^0(M) \to \R$:
\begin{equation}\label{asymptoticrho}
\zeta^{EP}_e(H) :=  \lim_{n \to \infty} \frac{\rho^{EP}(n H;e)}{n}
\end{equation}
for autonomous $C^\infty$ function $H$ and
\begin{equation}\label{eq:mueEP}
\mu^{EP}_e(\widetilde \psi) = - \vol_\omega(M) \lim_{n \to \infty} \frac{\rho^{EP}(\widetilde \psi^n;e)}{n},
\end{equation}
where $\rho^{EP}(\widetilde \psi^n;e)$ is defined from
$\rho^{EP}(H;e)$ in the same way as in (\ref{rhofordiff}).

These functions exactly coincide with those defined by
\cite{EP:morphism,EP:states,EP:rigid}.

\par
Now in the case $\frak b =0$
(\ref{mathcalA_HEP}) implies the equalities
\begin{equation}\label{signdifference}
\zeta_e(H)=-\zeta_e^{EP}(-H), \qquad \mu_e(\widetilde{\psi}_H)= -\mu_e^{EP}(\widetilde{\psi}_{-H}),
\end{equation}
which relate (\ref{asymptoticrho0}), (\ref{eq:mue})
and  (\ref{asymptoticrho}), (\ref{eq:mueEP}).

For example, when $e = 1$ and $H$ is a $C^2$ small autonomous
Morse function, we have
$$
\rho(H;1) = - \min H,  \qquad \rho^{EP}(H;1) = \max H,
$$
which also shows that the values of Entov-Polterovich's spectral invariant are
different in general from ours (the latter is also the same as that of \cite{oh:alan}).
\par
Moreover we can find an example for which the identity $\zeta_e(H)\ne -\zeta_e^{EP}(H)$ holds, as follows.
Consider an autonomous Hamiltonian $H$ on $T^2=S^1 \times S^1$ of the form $\pi_1^*h$, where
$\pi_1: S^1 \times S^1 \to S^1$ is the first projection.
Note that there are no non-constant contractible periodic orbits of $X_{kH}$
for any real number $k$.
We find that
$\rho(kH;1)=  - \min kH$, $\rho^{EP}(kH;1) = \max kH$.
Therefore we have $\zeta_1(H) = -\min H$, while
$\zeta_1^{EP}(H) = \max H$.
\par
The identities (\ref{signdifference}) imply that, for the case $\zeta_e$ is
actually homogeneous, not just semi-homogeneous,
(namely if $\zeta_e(-H) = -\zeta_e(H)$), then the two (partial) quasi-states
$\zeta_e$ and $\zeta^{EP}_e$ coincide.

\begin{rem}\label{rem114444}
Note the direction of inequality we used in Definition \ref{heavy1}
is opposite to the one in Entov-Polterovich \cite{EP:rigid}.
However, because
of the minus sign in
(\ref{signdifference}), Definition \ref{heavy1}
is indeed consistent with that in \cite{EP:rigid}. We feel that the direction of the inequality
above is easier to memorize and suits better for the general principle in
symplectic topology in that getting a lower bound for the spectral invariants
is much more nontrivial in general than getting an upper bound which is just
a consequence of the general `positivity' principle in symplectic topology.
\end{rem}
We also remark that Entov-Polterovich use quantum {\em homology}
and a Novikov ring similar to $\Lambda^{\downarrow}$.
We use quantum {\em cohomology} and the Novikov ring $\Lambda$.

\part{Review of spectral invariants}

\section{Hamiltonian Floer-Novikov complex}
\label{subsec:pert-action}

We first recall from Notations and Conventions (18) in Section \ref{sec:introduction} that we define
a valuation $\frak v_q$ \index{valuation} on the (downward) universal Novikov field $\Lambda^\downarrow$ by
\begin{equation}\label{vqdef}
 \frak v_q\left(\sum_{i=1}^{\infty} a_i q^{\lambda_i}\right)
 = \sup \{ \lambda_i \mid a_i \ne 0 \}.
\end{equation}
It satisfies the following properties:
\begin{enumerate}
\item $ \frak v_q(xy) =  \frak v_q(x) +  \frak v_q(y)$,
\item $ \frak v_q(x+y) \le \max\{ \frak v_q(x),  \frak v_q(y)\}$,
\item  $\frak v_q(x) = - \infty$ if and only if $x =0$,
\item  $\frak v_q(q) = 1$,
\item  $\frak v_q(ax) = \frak v_q(x)$ if $a \in \C \setminus \{0\}$.
\end{enumerate}
This field will be the coefficient field of the Floer chain module whose
description is in order.

Let $\widetilde{\CL}_0(M)$ be the set of all the pairs $[\gamma,w]$
where $\gamma$ is a loop $\gamma: S^1 \to M$ and $w: D^2 \to M$ a
disc with $ w\vert_{\del D^2} = \gamma$. We denote by $\pi_2(\gamma)$ the set
of homotopy classes of such $w$'s. \index{$\pi_2(\gamma)$}
We identify $[\gamma,w]$
and $[\gamma',w']$ if $\gamma = \gamma'$ and  $w$ is homotopic to
$w'$ relative to the boundary $\gamma$.
\index{$\lb \gamma,w]$} When a one-periodic
Hamiltonian $H:(\R/\Z) \times M \to \R$ is given, we consider the
perturbed functional $\CA_H: \widetilde \CL_0(M) \to \R$ defined by
\be\label{eq:AAH} \CA_H([\gamma,w]) = -\int w^*\omega - \int
H(t,\gamma(t))dt.
\ee
For a Hamiltonian $H:[0,1] \times M \to \R$,
we denote its flow, a Hamiltonian isotopy, by $\phi_H: t\mapsto
\phi_H^t \in \text{\rm Ham}(M,\omega)$. This gives a one-to-one
correspondence between equivalence classes of $H$ modulo the
addition of a function on $[0,1]$ and Hamiltonian isotopies. We
denote the time-one map by $\psi_H: = \phi_H^1$. We put
$$
\Fix \psi_H = \{ p \in M \mid \psi_H(p) = p\}.
$$
Each element $p\in \Fix \psi_H$ induces a map
$
z_p=z_p^H : S^1 \to M
$
by the correspondence
\be\label{eq:Fix-Per}
z_p(t) = \phi_H^t(p),
\ee
where $t \in \R/\Z \cong S^1$. The loop $z_p$ satisfies Hamilton's equation
$$
\dot x = X_H(t,x), \quad z_p(0) = p.
$$
Here $X_H$ is the (time-dependent) Hamiltonian vector field given by
$X_H(t,x) = X_{H_t}(x)$ where $X_{H_t}$ is the Hamiltonian vector field
generated by the function $H_t:C^\infty(M) \to \R$.
We denote by $\text{\rm Per}(H)$ the set of one-periodic solutions of $\dot x = X_H(t,x)$.
Then \eqref{eq:Fix-Per} provides a one-to-one correspondence between $\operatorname{Fix}\psi_H$ and
$\operatorname{Per}(H)$. The next lemma is well-known.
\index{$\text{\rm Per}(H)$}

\begin{lem}
The set of critical points of $\CA_H$ is given by
$$
\text{\rm Crit}(\CA_H) = \{ [\gamma,w] \mid \gamma \in \text{\rm Per}(H), \, w\vert_{\del D^2}  = \gamma \}.
$$
\index{$\text{\rm Crit}(\CA_H)$}
\end{lem}

Hereafter we assume that our Hamiltonian $H$ is normalized in the sense
of (\ref{eq:normalized}) unless otherwise stated explicitly.

We say that $H$ or its associated map $\psi_H$ is {\it nondegenerate}\index{Hamiltonian!nondegenerate}
if at any $p \in \operatorname{Fix}\psi_H$, the differential
$d_p\psi_H : T_p M \to T_p M$ does not have eigenvalue $1$.
The cardinality of $\text{\rm Per}(H)$ is finite if $\psi_H$ is nondegenerate.
For each nondegenerate Hamiltonian, we are given the canonical real and integer grading
\be\label{eq:gradings}
\CA_H: \text{\rm Crit}(\CA_H) \to \R, \quad \mu_H: \text{\rm Crit}(\CA_H) \to \Z
\ee
\index{Conley-Zehnder index}\index{$\mu_H$}
induced by the action functional and the Conley-Zehnder index \cite{conley-zehn} respectively.
\par
We define an equivalence relation $\sim$ on $\text{\rm Crit}(\CA_H)$
such that $[\gamma,w] \sim [\gamma',w']$ if and only if $\gamma = \gamma'$ and
$$
\int w^*\omega = \int (w')^*\omega.
$$
We denote by $\widehat{\text{\rm Per}}(H)$ the set of such equivalence
classes.  \index{$\widehat{\text{\rm Per}}(H)$}
The equivalence class of  $[\gamma,w] \in \widehat{\text{\rm Per}}(H)$ is denoted by
$\llb \gamma,w \rrb$.
\index{$\llb \gamma,w \rrb$}
The map $\mathcal A_H$ in (\ref{eq:gradings}) induces a map
$\widehat{\text{\rm Per}}(H) \to \R$, which we denote by the same symbol.
\par
We define maps
\begin{equation}\label{forgetper}
\text{\rm Crit}(\CA_H)
\overset{\pi}\longrightarrow
\widehat{\text{\rm Per}}(H)
\longrightarrow
{\text{\rm Per}}(H)
\end{equation}
by $[\gamma,w] \mapsto \llb \gamma,w \rrb$
and $ \llb \gamma,w \rrb \mapsto \gamma$.  We denote by $\pi : \text{\rm Crit}(\CA_H)
\to \widehat{\text{\rm Per}}(H)$ the first map in (\ref{forgetper}).

Note that the obvious concatenation defines a map
\begin{equation}\label{pi2action}
\# : \pi_2(M) \times \pi_2(\gamma)\to \pi_2(\gamma).
\end{equation}
We consider the map
\begin{equation}\label{energymapmap}
\pi_2(M) \to \R
\end{equation}
defined by $[v] \mapsto \int v^*\omega$ for $v : S^2 \to M$.
Let $K_2(M)$ be the kernel of this map.
\index{$K_2(M)$}
\begin{lem}\label{determinedifference}
By the action (\ref{pi2action}) we can identify
$$
\widehat{\text{\rm Per}}(H) = \text{\rm Crit}(\CA_H)/K_2(M).
$$
\end{lem}
The proof is easy and is omitted.
See Lemma \ref{lem4545} for the relation between $\widehat{\text{\rm Per}}(H)$
and ${\text{\rm Per}}(H)$.

We consider the
$\Lambda^\downarrow$ vector space
$\widehat{CF}(M,H;\Lambda^\downarrow)$ with a basis given by the
elements of the  set $\widehat{\text{\rm Per}}(H)$.
\begin{defn}\label{Lambdahatonashi}\index{$CF(M,H;\Lambda^\downarrow)$}
We define an equivalence relation $\sim$ on $\widehat{CF}(M,H;\Lambda^\downarrow)$ so that
$\llb\gamma,w\rrb \sim q^c\llb\gamma',w'\rrb$ if and only if
\be\label{eq:equiv-rel}
\gamma= \gamma', \qquad \int_{D^2} w^{\prime *}\omega = \int _{D^2} w^*\omega + c.
\ee
\par
The quotient of $\widehat{CF}(M,H;\Lambda^\downarrow)$ modded out by this equivalence relation $\sim$
is called the Floer complex of the periodic Hamiltonian $H$ and denoted by
${CF}(M,H;\Lambda^\downarrow)$.
\index{$CF(M,H;\Lambda^\downarrow)$}
\end{defn}
Here we ignore the $\Z$-grading, which is given by the Conley-Zehnder indices,
in defining the equivalence relation unlike the case of standard literature
of Hamiltonian Floer homology such as \cite{hofer-sal}, where additional requirement
$
c_1(\overline w \# w') = 0
$
is imposed in the definition of Floer complex, denoted by $CF(H)$. For the purpose of
the current paper, the equivalence relation \eqref{eq:equiv-rel}
is enough and more favorable in that it makes the associated Novikov ring
becomes a field.
To differentiate the current definition from $CF(H)$, we denote the complex
used in the present paper by ${CF}(M,H;\Lambda^\downarrow)$.

\begin{lem}
As a $\Lambda^\downarrow$ vector space, ${CF}(M,H;\Lambda^\downarrow)$ is isomorphic to
the direct sum $(\Lambda^\downarrow)^{\# \text{\rm Per}(H)}$.
\par
Moreover the following holds: We fix a lifting $\llb\gamma,w_\gamma\rrb
\in \widehat{\text{\rm Per}}(H)$ for each
$\gamma \in \text{\rm Per}(H)$. Then any element $x$ of $CF(M,H;\Lambda^\downarrow)$
is uniquely written as a sum
\begin{equation}\label{xexpand}
x = \sum_{\gamma \in \text{\rm Per}(H)} x_\gamma \llb\gamma,w_\gamma\rrb, \quad
\text{with $x_\gamma \in \Lambda^\downarrow$}.
\end{equation}
\end{lem}

The proof is easy and omitted.
\begin{defn}\label{valuationv} \index{filtration}
\begin{enumerate}
\item
Let $x$ be as in $(\ref{xexpand})$. We define the
{\it level function}\index{level function}\index{$\lambda_H$}
$$
\lambda_H(x) = \max \{ \frak v_q(x_\gamma) + \CA_H(\llb\gamma,w_\gamma\rrb) \mid x_\gamma \neq 0 \}.
$$
\item
We define a filtration $F^{\lambda}CF(M,H;\Lambda^\downarrow)$ on $CF(M,H;\Lambda^\downarrow)$ by
$$
F^{\lambda}CF(M,H;\Lambda^\downarrow)
=
\{ x \in CF(M,H;\Lambda^\downarrow) \mid \lambda_H(x) \le \lambda\}.
$$
We have
$$
F^{\lambda_1}CF(M,H;\Lambda^\downarrow)  \subset F^{\lambda_2}CF(M,H;\Lambda^\downarrow)
$$
if $\lambda_1 < \lambda_2$. We also have
$$
\bigcap_{\lambda} F^{\lambda}CF(M,H;\Lambda^\downarrow)
=\{0\},
\quad
\bigcup_{\lambda} F^{\lambda}CF(M,H;\Lambda^\downarrow)
=CF(M,H).
$$
\item
We define a metric $d_q$ on $CF(M,H;\Lambda^\downarrow)$ by
\be\label{eq:metricdq}
d_q(x,x') = e^{\lambda_H(x-x') }.
\ee
\end{enumerate}
\end{defn}

Definition \ref{valuationv}, together with (\ref{vqdef}) and (\ref{eq:equiv-rel}),
implies that
$$
\lambda_H(a\frak x) = \frak v_q(a) + \lambda_H(\frak x)
$$
for $a \in \Lambda^{\downarrow}$,
$\frak x \in CF(M,H;\Lambda^\downarrow)$.
We also have
$$
q^{\lambda_1} F^{\lambda_2}CF(M,H;\Lambda^\downarrow)
\subseteq
F^{\lambda_1+\lambda_2}CF(M,H;\Lambda^\downarrow).
$$

\begin{lem}\label{lem37}
\begin{enumerate}
\item The definition of $\lambda_H$ given in Definition $\ref{valuationv}$
is independent of the choice of the lifting $\gamma \mapsto [\gamma,w_\gamma]$ with
$\gamma \in \Per(H)$.
\item
$CF(M,H;\Lambda^{\downarrow})$ is complete with respect to the metric $d_q$.
\item
The infinite sum
$$
\sum_{\llb \gamma,w \rrb \in \widehat{\Per}(H)} x_{\llb\gamma,w\rrb} \llb\gamma,w\rrb
$$
converges in $CF(M,H;\Lambda^{\downarrow})$ with respect to the metric $d_q$, provided
$$
\{ \llb\gamma,w\rrb \in \Per(H) \mid \frak v_q(x_{\llb\gamma,w\rrb}) +   \mathcal A_H(\llb\gamma,w\rrb) > -C,
\,\, x_{\llb\gamma,w\rrb} \ne 0\}.
$$
is finite for any $C \in \R$.
\end{enumerate}
\end{lem}
The proof is easy and omitted.

\section{Floer boundary map}
\label{subsec:boundary}
\index{Floer boundary map|(}

In this section we define the boundary operator $\partial_{(H,J)}$
on $CF(M,H;\Lambda^{\downarrow})$ so that it becomes a filtered complex.
Let a nondegenerate one-periodic Hamiltonian function $H$ and a
one-periodic family $J = \{J_t\}_{t\in S^1}$ of
compatible almost complex structures be given. The study of the following perturbed Cauchy-Riemann
equation \index{perturbed Cauchy-Riemann equation}
\be\label{eq:HJCR0}
\dudtau + J\Big(\dudt - X_{H_t}(u)\Big) = 0
\ee
is the heart of the Hamiltonian Floer theory.
Here and hereafter $J$ in (\ref{eq:HJCR0}) means $J_t$.
\begin{rem}\label{rem31}
\begin{enumerate}
\item
In this paper, we {\it never} use perturbation of (a family of) compatible almost complex
structures $J$ to achieve
transversality of the moduli space of
the Floer equations (\ref{eq:HJCR0}) but use abstract perturbations (multisections and CF-perturbations
    of the Kuranishi structure
in the sense of \cite{fukaya-ono,fooo:book2,fooo:tech2})
to achieve the necessary transversality.
See Part \ref{part7} for  a summary of Kuranishi structure and CF-perturbation.
\item
In Parts 1-3 (and somewhere in the appendix) we use a $t$-parameterized
family of compatible almost complex structures $\{J_t\}_{t\in S^1}$.
However we emphasize that we do {\it not} need to use an
$S^1$-parameterized
family of compatible almost complex structures but can
use a {\it fixed} compatible almost complex structure $J$,
to prove all of our main results of this paper.
(We need to use a $t$-dependent $J$ for the construction in
Sections  \ref{sec:appendix4} and   \ref{subsec:seibulkspectral}.)
The $t$-dependent $J$ is included only for the sake of consistency with the
reference on spectral invariants.
(Traditionally $t$-dependent $J$ had been used
to achieve transversality.
As we mentioned in (1), we do {\it not} need this
extra freedom in this paper since
we use abstract perturbations.)
\end{enumerate}
\end{rem}

The following definition is useful for the later discussions.

\begin{defn}\label{eq:pi2zz'}\index{$\pi_2(\gamma,\gamma')$} Let $\gamma, \, \gamma' \in \text{Per}(H)$. We
denote by $\pi_2(\gamma,\gamma')$ the set of homotopy classes of smooth maps
$
u: [0,1] \times S^1  \to M
$
relative to the boundary
$u(0,t) = \gamma(t)$,
$u(1,t) = \gamma'(t).
$
We denote by $[u] \in \pi_2(\gamma,\gamma')$ its homotopy class.
\par
As we defined at the beginning of Section \ref{subsec:pert-action},
$\pi_2(\gamma)$ denotes the set of relative homotopy classes
of the maps
$
w: D^2 \to M; w|_{\del D^2} = \gamma.
$
For $C \in \pi_2(\gamma,\gamma')$, there is a natural map of
$
(\cdot) \# C: \pi_2(\gamma) \to \pi_2(\gamma')
$
induced by the gluing map
$
w \mapsto w \# C.
$
There is also the natural gluing map
$$
\pi_2(\gamma_0,\gamma_1) \times \pi_2(\gamma_1,\gamma_2) \to \pi_2(\gamma_0,\gamma_2)
, \qquad
(u_1, u_2) \mapsto u_1\# u_2.
$$
\end{defn}

For each $[\gamma,w], [\gamma',w'] \in \text{Crit}(\mathcal A_H)$, we will define a moduli space
$\CM(H,J;[\gamma,w], [\gamma',w']).$
We begin with the definition of the energy.

\begin{defn}{(\rm Energy)} For a given smooth map $u: \R \times S^1 \to M$,
we define the energy of $u$ by
$$
E_{(H,J)}(u) = \frac{1}{2} \int \Big(\Big|\dudtau\Big|^2_{J} + \Big|
\dudt - X_{H_t}(u)\Big|_{J}^2 \Big)\, dt\, d\tau.
$$
Here and hereafter we denote
$$
\left\vert \frac{\partial u}{\partial \tau} \right\vert_J^2
=
\omega\left(\frac{\partial u}{\partial \tau},J_t\left(\frac{\partial u}{\partial \tau}\right)\right).
$$
\end{defn}

\begin{defn}\label{stripmk0int0}\index{$\mathcal M(H,J;[\gamma,w], [\gamma',w'] )$}
We denote by $\widehat{\overset{\circ}{{\CM}}}(H,J;[\gamma,w], [\gamma',w'] )$ the set of all smooth maps
$u: \R \times S^1 \to M$ which satisfy the following conditions:
\begin{enumerate}
\item The map $u$ satisfies the equation:
\be\label{eq:HJCR}
\dudtau + J\Big(\dudt - X_{H_t}(u)\Big) = 0.
\ee
\item The energy
$E_{(H,J)}(u) $ is finite.
\item The map $u$ satisfies the following asymptotic boundary condition.
$$
\lim_{\tau\to-\infty}u(\tau, t) = \gamma(t), \qquad \lim_{\tau\to +\infty}u(\tau, t) = \gamma'(t).
$$
\item
The homotopy class of the concatenation $w\# u$ of $w$ and $u$  gives
the same element as
$[\gamma',w'] \in \pi_2(\gamma')$. Namely
 $[\gamma',w\# u] = [\gamma',w']$.
\end{enumerate}
It carries an $\R$-action induced by the domain translations in the $\tau$-direction.
We denote the quotient space of this $\R$-action by $\overset{\circ}{\CM}(H,J;[\gamma,w], [\gamma',w'])$.
\par
When $[\gamma,w] =  [\gamma',w']$, we set the space
$\overset{\circ}{{\CM}}(H,J;[\gamma,w], [\gamma,w])$ to be the empty set by definition.
\end{defn}

\begin{rem} The conditions (1) and (2) above make the convergence in (3) one
of an exponential order, which in turn enables the statement (4) to make sense.
\end{rem}
For any $\alpha \in \pi_2(M)$, we have a canonical
homeomorphism
\begin{equation}\label{homeo333}
\overset{\circ}{{\CM}}(H,J;[\gamma,w], [\gamma,w])
\cong
\overset{\circ}{{\CM}}(H,J;[\gamma,\alpha\# w], [\gamma,\alpha\# w]).
\end{equation}
This is because the condition $[\gamma',w\# u] = [\gamma',w']$
in item (4) is equivalent to
$[\gamma',\alpha\# w\# u] = [\gamma',\alpha\# w']$.

\begin{prop}\label{connkura}
\begin{enumerate}
\item The moduli space
$\overset{\circ}{\CM}(H,J;[\gamma,w], [\gamma',w'] )$ has a compactification
${{\CM}}(H,J;[\gamma,w], [\gamma',w'] )$ that is Hausdorff.
\item
The space ${{\CM}}(H,J;[\gamma,w], [\gamma',w'])$ has an orientable Kuranishi structure with corners.
(See Definitions $\ref{kstructuredefn}$ and $\ref{defn345}$ for the definition of  Kuranishi structure with corners.)
\item
The normalized boundary\footnote{See Definition \ref{defnnormalizedboundary}
for the definition of normalized boundary.} of ${{\CM}}(H,J;[\gamma,w], [\gamma',w'])$ in the
sense of Kuranishi structure is described by
\begin{equation}
\aligned
&\partial {\CM}(H,J;[\gamma,w], [\gamma',w']) \\
&= \bigcup {\CM}(H,J;[\gamma,w], [\gamma'',w'']) \times {\CM}(H,J;[\gamma'',w''], [\gamma',w']),
\endaligned
\end{equation}
where the union is taken over all $[\gamma'',w''] \in \mathrm{Crit}(\mathcal A_{H})$.
\item
Then we have the following formula for the (virtual) dimension:
\begin{equation}\label{dimension0}
\dim {\CM}(H,J;[\gamma,w], [\gamma',w']) = \mu_H([\gamma',w']) -  \mu_H([\gamma,w]) - 1.
\end{equation}
\item
We can define orientations of ${\CM}(H,J;[\gamma,w], [\gamma',w'])$ so that
$(3)$ above is compatible with this orientation.
\item
The homeomorphism $(\ref{homeo333})$ extends to the compactifications
and their Kuranishi structures are identified by the homeomorphism.
\end{enumerate}
\end{prop}

This is proved in \cite[Chapter 4]{fukaya-ono}. More precisely, the dimension formula (4) is proved
in \cite{floer:cmp,salamon-zehnder}, and (1) is \cite[Theorem 19.12]{fukaya-ono},
(2), (3) are \cite[Theorem 19.14]{fukaya-ono} and
(5) is \cite[Lemma 21.4]{fukaya-ono}.
Thorough detail of the proof of this theorem is given also in \cite[Part 5]{fooo:techI}.
In fact Proposition \ref{connkura} is \cite[Theorem 29.4]{fooo:techI}  which is proved
in  \cite[Section 31]{fooo:techI}.
See  \cite{floer:cmp,hofer-sal,Ono95} etc. for the
earlier works for the semi-positive cases.
\par
We remark that $CF(M,H;\Lambda^{\downarrow})$ as a $\Lambda^{\downarrow}$-vector space carries a natural $\Z_2$-grading
given by the homological degree (${\rm mod}\, 2$) of {$\llb\gamma,w\rrb$.
This is because
$$
\mu_H([\gamma,w']) - \mu_H([\gamma,w])= 2 c_1(M) \cap [\overline w \# w']
$$
holds for any pair $[w],\, [w'] \in \pi_2(\gamma)$. Here
$\overline w\# w'$ is a 2-sphere obtained by gluing $\overline w$ and $w'$ along $\gamma$
where $\overline w$ is the $w$ with opposite orientation.
(See \cite[page 557]{floer:cmp}  and \cite[Appendix]{oh:lecture}  for its proof.)
In particular, it implies that the parity of $\mu_H([\gamma,w])$
depends only on $\gamma \in \Per(H)$
but not on its lifting $[\gamma,w] \in \text{\rm Crit}(\mathcal A_H)$.
This gives rise to a natural $\mathbb Z_2$-grading on $CF(M,H;\Lambda^\downarrow)$,
that is, the homological degree mod 2 of $[\gamma,w_\gamma]$,
which is $n - \mu_H([\gamma,w_\gamma]) \mod 2$. So we put:
\begin{equation}\label{35formula}\index{$CF_1(M,H;\Lambda^\downarrow)$}\index{$CF_0(M,H;\Lambda^\downarrow)$}
\aligned
CF_1(M,H;\Lambda^\downarrow)
&= \bigoplus_{\gamma; \text{$n-\mu_H([\gamma,w_\gamma])$ is odd.}} \Lambda^\downarrow \llb\gamma,w_\gamma\rrb, \\
CF_0(M,H;\Lambda^\downarrow)
&= \bigoplus_{\gamma ; \text{$n-\mu_H([\gamma,w_\gamma])$ is even.}} \Lambda^\downarrow \llb\gamma,w_\gamma\rrb.
\endaligned\end{equation}
\begin{rem}
We remark that the degree of Floer chain defined above is shifted by $n$ from
the Conley-Zehnder index $\mu_H$.
By this shift, the degree will coincide with the degree of (quantum) cohomology group
of $M$ by the isomorphism in Theorem \ref{Piuiso}.
See also Remark \ref{rem52} (2) and \cite[Section 7.2]{oh:lecture} .
\end{rem}
We use Proposition \ref{connkura} to define
the Floer boundary map
$$
\del_{(H,J)} : CF_{k+1}(M,H;\Lambda^\downarrow) \to CF_k(M,H;\Lambda^\downarrow)
$$
whose explanation is in order.

First for each given pair $[\gamma,w], \, [\gamma',w']$, it defines
a homotopy class $B \in \pi_2(\gamma,\gamma')$ such that $[w \# B] = w'$ in
$\pi_2(\gamma')$. By an abuse of notation, we denote this class $B = [\overline w \# w']$.
It follows from the Gromov-Floer compactness and the energy identity given \eqref{energyincrease}
that for each given $[\gamma,w]$ the set of all $[\gamma',w'] \in \Crit \CA_H$ satisfying
$$
{\CM}(H,J;[\gamma,w],[\gamma',w']) \ne \emptyset, \qquad [\overline w \# w']\cap[\omega] < A
$$
is finite for any fixed $A\in \R$. Using this finiteness, we construct a system of
multisections $\frak s$ on ${\CM}(H,J;[\gamma,w], [\gamma',w'])$ inductively over the symplectic
area
$
(\overline w \# w') \cap \omega \in \R_{\ge 0}
$
so that they are transversal to $0$ and compatible with the identification made in Proposition \ref{connkura} (3).
Such an inductive construction is proven to be possible for the relative version of
the construction of Kuranishi structures in \cite[Theorem 6.12]{fukaya-ono}  (that is, \cite[Lemma A1.20]{fooo:book2}).
Detail of the construction of such compatible system of Kuranishi structures
is given in \cite[Parts 4 and 5]{fooo:techI}. Construction of the multisection then is detailed in
\cite[Section 13]{fooo:tech2}.
Now we define
the operator $\partial_{(H,J)}$ by the formula:
\begin{equation}\label{boundaryop}
\partial_{(H,J)} (\llb \gamma,w \rrb)
= \sum_{[\gamma',w'] \in \text{\rm Crit}(\mathcal A_H)
\atop \mu_H([\gamma',w'])
= \mu_H([\gamma,w])  + 1}
\frak n_{H,J}([\gamma,w ], [\gamma',w']) \,\, \llb\gamma',w'\rrb.
\end{equation}
Here we put
$$\frak n_{H,J}([\gamma,w], [\gamma',w'])
=
\# {\CM}(H,J;[\gamma,w],[\gamma',w'])^{\frak s}.
$$
The rational number
$\# {\CM}(H,J;[\gamma,w],[\gamma',w'])^{\frak s}$ is the
virtual fundamental $0$-chain of ${\CM}(H,J;[\gamma,w],[\gamma',w'])$
with respect to the multisection ${\frak s}$.
Namely it is the order of the zero set of ${\frak s}$ counted with
sign and multiplicity.
(See \cite[Definition 4.6]{fukaya-ono}  or \cite[Definition A1.28]{fooo:book2}  for its precise
definition. For the case of virtual dimension zero, which we use here, the detailed definition of
virtual fundamental chain is given also in \cite[Definitions 14.4 and 14.6]{fooo:tech2}.)
Hereafter we omit $\frak s$ and simply write
${\CM}(H,J;[\gamma,w], [\gamma',w'])$ for the perturbed moduli space.
\par
By the above mentioned finiteness, Lemma \ref{lem37} (3) implies that the right hand side of
(\ref{boundaryop}) converges in $d_q$-metric, i.e., with respect to the downward valuation.
\begin{rem}\label{remark38new}
To define the right hand side of (\ref{boundaryop})
we assume $\llb \gamma,w \rrb \in \widehat{\rm Per}(H)$ is
represented by $[\gamma,w] \in {\rm Crit}(\mathcal A_H)$,
i.e., $\pi([\gamma,w]) = \llb \gamma,w \rrb$ and the sum
is taken over $[\gamma',w'] \in \text{\rm Crit}(\mathcal A_H)$.
\footnote{The map $\pi$ is defined as in (\ref{forgetper}).}
The right hand side of (\ref{boundaryop}) is independent of
the choice of such $[\gamma,w]$
as far as it represents the same element of $\widehat{\rm Per}(H)$. In fact we have
$$
\frak n_{H,J}([\gamma,w], [\gamma',w'])
=
\frak n_{H,J}([\gamma,\alpha\#w], [\gamma',\alpha\#w'])
$$
by Proposition \ref{connkura} (5). (We take a multisection $\frak s$
which is compatible with the isomorphism in  Proposition \ref{connkura} (5).)
The independence of the right hand side of (\ref{boundaryop})
then follows from Lemma \ref{determinedifference}.
\par
Lemma \ref{determinedifference} also implies that this definition is
consistent with the equivalence relation (\ref{eq:equiv-rel}).
(Hereafter in a similar situation we do not mention this kind of remarks.)
\end{rem}

We can prove
\begin{equation}\label{eq:3737}
\partial_{(H,J)} \circ \partial_{(H,J)} = 0
\end{equation}
by applying Proposition \ref{connkura} (3) in the case
when $\mu_H([\gamma',w']) -  \mu_H([\gamma,w]) = 2$.
(See \cite[Lemma 20.2]{fukaya-ono}.)
The boundary property (\ref{eq:3737}) follows also from \cite[Theorem 14.10]{fooo:tech2}.
\begin{rem}
Instead of multisections we can use CF-perturbations
(See Definition \ref{defCFGCSFKUA}.) to define the boundary operator $\partial_{(H,J)}$.
If we use CF-perturbations instead of multisections,
(\ref{eq:3737}) follows from Stokes' formula, Theorem \ref{them48},
and composition formula, Theorem \ref{compform}.
\par
We use only CF-perturbations from now on in this paper, except
in Part 4, where we use multisection sometimes.
(See Remark \ref{remark206666} and other remarks quoted therein.)
Actually we may regard a multisection itself as a CF-perturbation.
(See Remark \ref{multiseccont}.)
\end{rem}

\begin{lem}\label{filtered} For any $\lambda \in \R$,
$$
\partial_{(H,J)} (F^{\lambda}CF(M,H;\Lambda^\downarrow))
\subset F^{\lambda}CF(M,H;\Lambda^\downarrow).
$$
\end{lem}
\begin{proof}
We take $[\gamma,w] \in \text{Crit}(\mathcal A_H)$ such that
$\pi([\gamma,w]) = \llb \gamma,w \rrb$
and suppose $\llb \gamma,w \rrb \in \widehat{\Per}(H)$ satisfies
$\CA_H(\llb \gamma,w \rrb) \leq \lambda$.

If $u \in \widetilde{\mathcal M}(H,J;[\gamma,w],[\gamma',w'])$,
then
\bea\label{energyincrease}
\int u^*\omega
&= &\int_{\tau\in \R}\int_{t \in S^1} \omega
\left(
\frac{\partial u}{\partial \tau}, \frac{\partial u}{\partial t}
\right)
dt\, d\tau \nonumber \\
&= & \int_{\tau\in \R}\int_{t \in S^1} \omega
\left(
\frac{\partial u}{\partial \tau}, J\frac{\partial u}{\partial \tau} + X_{H_t}(u)
\right)
dt\, d\tau \nonumber\\
&= &\int_{\tau\in \R}\int_{t \in S^1} \left \vert\frac{\partial
u}{\partial \tau}\right\vert_J^2 \, dt\,d\tau -  \int_{\tau\in
\R}\int_{t \in S^1} (dH_t(u(\tau,t)))\left(\frac{\del u}{\del
\tau}\right) \, dt\,d\tau
\nonumber\\
&=& \int_{\tau\in \R}\int_{t \in S^1} \left \vert\frac{\partial
u}{\partial \tau}\right\vert_J^2\, dt\,d \tau -  \int_{\tau\in
\R}\left(\int_{t\in S^1} \frac{\partial}{\partial
\tau}(H_t(u(\tau,t))\, dt \right)d\tau
\nonumber \\
&= & E_{(H,J)}(u) -  \int_{t \in S^1} H_t(\gamma'(t)) dt+ \int_{t \in
S^1} H_t(\gamma(t)) dt. \eea
Therefore
$$
\int u^*\omega + \int_{t \in S^1} H_t(\gamma'(t))\, dt - \int_{t \in S^1}
H_t(\gamma(t))\, dt =E_{(H,J)}(u) \ge 0.
$$
Combined with $w\# u \sim w'$, this implies $\mathcal A_H([\gamma,w]) - \mathcal A_H([\gamma',w']) = E_{(H,J)}(u)$
and so
$$
\mathcal A_H([\gamma',w']) = \mathcal A_H([\gamma,w]) - E_{(H,J)}(u) \leq \mathcal A_H([\gamma,w])
$$
and hence Lemma \ref{filtered} holds.
\end{proof}
\index{Floer boundary map|)}
\begin{defn}\index{Floer homology}\index{$HF_*(H,J;\Lambda^\downarrow)$}
The \index{Floer homology! of periodic Hamiltonian system}{\em Floer homology} with $\Lambda^\downarrow$ coefficients is defined by
$$
HF_*(H,J;\Lambda^\downarrow): =
\frac{\operatorname{Ker}\del_{(H,J)}}{\operatorname{Im}
\del_{(H,J)}}.
$$
\index{$HF_*(H,J;\Lambda^\downarrow)$}
\end{defn}
\begin{thm}\label{Piuiso}
We may choose the orientation in Proposition $\ref{connkura}$ $(5)$ so that
$HF_*(H,J;\Lambda^\downarrow)$ is isomorphic to
the singular (co)homology $H(M;\Lambda^\downarrow)$ with $\Lambda^{\downarrow}$
coefficients.
\end{thm}
This is proved in \cite[Theorem 22.1]{fukaya-ono}.
We will describe a construction of isomorphism (which is different from the one in  \cite{fukaya-ono})
below because we need to specify the isomorphism to encode each spectral invariant
by the corresponding quantum cohomology class.
\begin{defn}\label{defn:chi}
Consider a smooth function \index{$\chi$}
$
\chi : \R \to [0,1]
$
with the  properties
\bea
\chi(\tau) & = & \begin{cases} 0 \quad \mbox{ for }\, \tau \le 0\\
1 \quad \mbox{ for }\, \tau \ge 1
\end{cases}\\
\chi'(\tau) & \geq & 0.
\eea
We denote by $\CK$ the set of such elongation functions.
\end{defn}
We note that $\CK$ is convex and so contractible.
We will fix a choice of $\chi$ for our various constructions  throughout
the present paper but the Floer theoretic invariants coming out of such
constructions are independent of the choice thanks to contractibility of $\CK$.

For a given $S^1$-dependent family of almost complex structures $J$, we consider a two-parameter family
$\{J_s\}_{s \in [0,1]}$  such that
\begin{equation}\label{HsJs}
J_{0,t} = J_0, \quad J_{1,t} = J_t
\end{equation}
where $J_0$ is a
time-independent almost complex structure. We also
assume
$$
J_{s,t} \equiv J_0, \quad \mbox{for }\, (s,t) \in \del [0,1]^2 \setminus (\{1\} \times [0,1]).
$$
For each nondegenerate $H : S^1\times M \to \R$, we define $(\R \times S^1)$-family to
$(H_\chi,J_\chi)$ on $\R \times S^1$ by
\be\label{eq:paraHJ}
H_\chi(\tau,t) = \chi(\tau) H_t, \quad J_\chi(\tau,t) = J_{\chi(\tau),t}.
\ee
\begin{rem}\label{JS1inv}
It is very important that the family $J_s$ is $t$-{\it independent} for $s=0$.
See the proof of Proposition \ref{piuBULKkura}.
\end{rem}

\begin{defn}\label{stripmk0int1}
We denote by $\overset{\circ}{{\CM}}(H_\chi,J_\chi;*,[\gamma;w])$ the set of all maps
\index{$\mathcal M(H_\chi,J_\chi;*,[\gamma;w])$}
$u: \R \times S^1 \to M$ satisfying the following conditions:
 \begin{enumerate}
 \item The map $u$ satisfies the equation:
\be\label{eq:HJCR2}
\dudtau + J_\chi\Big(\dudt - \chi(\tau)X_{H_t}(u)\Big) = 0.
\ee
Here and hereafter $J_\chi$ in $(\ref{eq:HJCR2})$ means
$J_{\chi,t}$.
\item The energy
$$
E_{(H_\chi,J_\chi)}(u):= \frac{1}{2} \int \Big(\Big|\dudtau\Big|^2_{J_\chi} + \Big|
\dudt - \chi(\tau)X_{H_t}(u)\Big|_{J_\chi}^2 \Big)\, dt\, d\tau
$$
is finite.
\item The map $u$ satisfies the following asymptotic boundary condition:
$$
\lim_{\tau\to +\infty}u(\tau, t) = \gamma(t).
$$
In particular, $u$ defines a homotopy class $[u] \in \pi_2(\gamma)$.
\item We have $[u] = [w]$ in $\pi_2(\gamma)$.
\end{enumerate}
\end{defn}

\begin{figure}[h]
\centering
\includegraphics[scale=0.5]{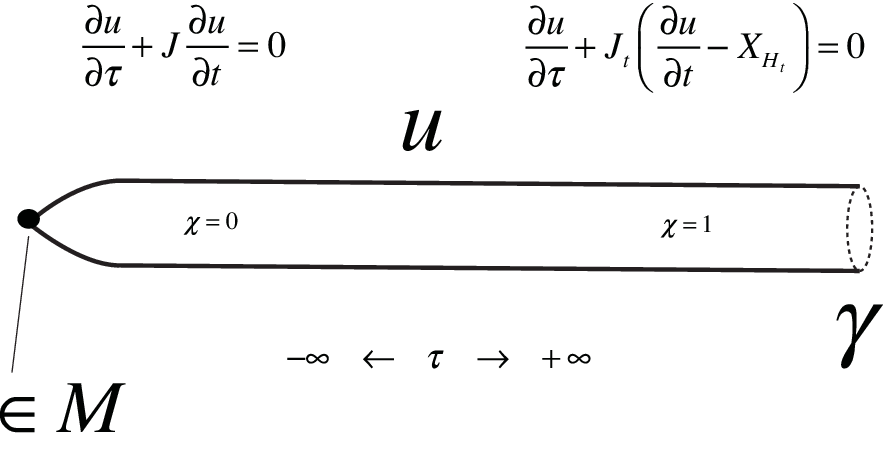}
\caption{An element of $\overset{\circ}{{\CM}}(H_\chi,J_\chi;*,[\gamma;w])$}
\label{Figure1}
\end{figure}
\par
We note that since $\chi(\tau) X_{H_t} \equiv 0$ and $J_{\chi(\tau)} \equiv J_0$
 for $\tau < -1$,
which turns \eqref{eq:HJCR2} into the genuine $J_0$-holomorphic curve equation,
the removable singularity theorem (due to Sacks-Uhlenbeck and Gromov,
see e.g., \cite[Theorem 4.5.1]{sikarav}) gives rise to
a well-defined limit
\begin{equation}\label{ev-inf}
\lim_{\tau\to -\infty}u(\tau, t)
\end{equation}
which does not depend on $t$. Therefore the homotopy class condition required
in (4) above makes sense.
\par
We denote this assignment of the limit by
\be\label{eq:ev-infty}
{\rm ev}_{-\infty} : \overset{\circ}{{\CM}}(H_\chi,J_\chi;*,[\gamma,w]) \to M.
\ee
\begin{rem}\label{rem:*}
We put $*$ on the left of $[\gamma;w]$ in the notation of the moduli
space to highlight the fact that $[\gamma,w]$ appears as \emph{positive} asymptotic
boundary condition of the cylinder $\R \times S^1$. We use similar notations using $*$
in various other Floer theoretic moduli spaces arising throughout the paper in a various
different ways.
Note we do {\it not} require the limit point (\ref{ev-inf}) to be a particular base point of $M$
--we only require the convergence to some point of $M$.
\end{rem}
\begin{prop}\label{isomPiukura}
\begin{enumerate}
\item The moduli space
$\overset{\circ}{{\CM}}(H_\chi,J_\chi;*,[\gamma,w])$ has a compactification
${{\CM}}(H_\chi,J_\chi;*,[\gamma,w])$ that is Hausdorff.
\item
The space ${{\CM}}(H_\chi,J_\chi;*,[\gamma,w])$ has an orientable Kuranishi structure with corners.
\item
The  (normalized) boundary of ${{\CM}}(H_\chi,J_\chi;*,[\gamma,w])$ is described by
\begin{equation}\label{bdryhomomap}
\aligned
&\partial {{\CM}}(H_\chi,J_\chi;*,[\gamma,w]) \\
&= \bigcup {{\CM}}(H_\chi,J_\chi;*,[\gamma',w']) \times {\CM}(H,J;[\gamma',w'],
[\gamma,w]),
\endaligned
\end{equation}
where the union is taken over all $[\gamma',w'] \in \text{\rm Crit}(\mathcal A_H)$.
\item
The (virtual) dimension satisfies
the following equality $(\ref{dimension})$.
\begin{equation}\label{dimension}
\dim  {{\CM}}(H_\chi,J_\chi;*,[\gamma,w]) = \mu_H([\gamma,w]) + n.
\end{equation}
\item
We can define a system of orientations of $ {{\CM}}(H_\chi,J_\chi;*,[\gamma,w]) $ so that
$(3)$ above is compatible with this orientation.
\item
The map ${\rm ev}_{-\infty}$ becomes a weakly submersive map in the sense of
\cite[Definition {\rm  A1.13}]{fooo:book2}.  (See Definition $\ref{mapkura}$.)
\end{enumerate}
\end{prop}
The proof is mostly the same as that of Proposition \ref{connkura} and is omitted.
We can prove the weak submersivity of  ${\rm ev}_{-\infty}$ in item (6) in the same way as
\cite[Section 7]{fooo:book2}. It is proved more explicitly in \cite[Lemma 3.1]{fukaya:cyc}.
(There the case of the evaluation map at the boundary marked point is
treated. The case of ${\rm ev}_{-\infty}$, which is item (6) above, can be proved in
the same way.)
An element of the right hand side of (\ref{bdryhomomap}) is drawn in Figure  \ref{Figure2} below.
\begin{figure}[h]
\centering
\includegraphics[scale=0.5]
{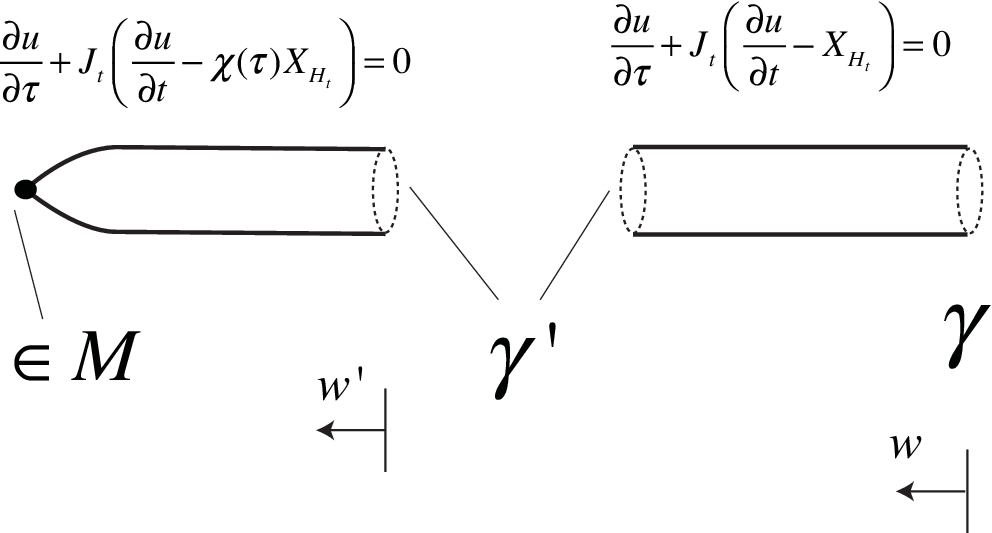}
\caption{An element of (\ref{bdryhomomap})}
\label{Figure2}
\end{figure}
We take a system of CF-perturbations $\{\widehat{\frak S^{\epsilon}}\}$ on $ {{\CM}}(H_\chi,J_\chi;*,[\gamma,w])$
for various $[\gamma,w]$ so that it is compatible at the boundary
described in (\ref{bdryhomomap}) and ${\rm ev}_{-\infty}$ is strongly submersive with respect to $\{\widehat{\frak S^{\epsilon}}\}$, in the
sense of Definition \ref{CFtransv}. Existence of such CF-perturbations follows from Theorem
\ref{existperturbcontkura}.
\par
Let $h$ be a differential $k$ form on $M$. We define
\begin{equation}\label{Piunikin}\index{$\mathcal P_{(H_\chi,J_\chi)}(h)$}
\CP_{(H_\chi,J_\chi)}(h)
= \sum_{[\gamma,w] \in \text{\rm Crit}(\mathcal A_H)} \left(\int_{({\CM}(H_\chi,J_\chi;*,
[\gamma,w]),\widehat{\frak S^{\epsilon}})}
{\rm ev}_{-\infty}^{*}(h)\right)
\llb \gamma,w \rrb.
\end{equation}
(The symbol $\mathcal P$ stands for Piunikhin \cite{Piu94}.)
Here the sum is taken over $[\gamma,w] \in \text{\rm Crit}(\mathcal A_H)$ with $\mu_H([\gamma,w]) = k-n = \deg h-n$
and $\widehat{\frak S}$ is a CF-perturbation.
The integration in (\ref{Piunikin}) is defined as
(\ref{form34333}).
\par
The definition (\ref{Piunikin}) induces a map
$$
\CP_{(H_\chi,J_\chi)} : \Omega(M) \widehat\otimes \Lambda^\downarrow \to CF(M,H;\Lambda^\downarrow).
$$
Here $\Omega(M) \widehat\otimes \Lambda^\downarrow$ is the completion with respect to the norm $\frak v_q$
of the algebraic tensor product $\Omega(M) \otimes \Lambda^\downarrow$
(over $\R$).
\index{$\widehat\otimes$}
\par
Let $(\Omega^*(M),d)$ be the de Rham complex of $M$. We regard it as a
{\it chain} complex $(\Omega_*(M),\partial)$, where
\be\label{eq:deRhamchain}
\Omega_k(M) = \Omega^{\dim M-k}(M), \quad  \partial = (-1)^{\deg+1} d.
\ee
\begin{lem}
$\CP_{(H_\chi,J_\chi)}$ defines a chain map
$$
\CP_{(H_\chi,J_\chi)} : (\Omega(M),\partial) \,\widehat{\otimes}\, \Lambda^\downarrow \to (CF(M,H;\Lambda^\downarrow),
\partial_{(H,J)})
$$
from the de Rham complex to the Floer complex.
\end{lem}
\begin{proof}
We can prove
$
\CP_{(H_\chi,J_\chi)} \circ \partial = \partial_{(H,J)} \circ \CP_{(H_\chi,J_\chi)}
$
by applying Stokes' theorem (Theorem \ref{them48}, \cite[Proposition 9.26]{fooo:tech2}),
composition formula (Theorem \ref{compform}, \cite[Theorem 10.20]{fooo:tech2}),
Proposition \ref{isomPiukura} (3) and the definition.
\end{proof}
\begin{lem}
$\CP_{(H_\chi,J_\chi)}$ induces a chain homotopy equivalence.
\end{lem}
The proof is similar to the argument established in various similar situations.
(One of the closest descriptions we can find in the literature is \cite[Section 8 Proposition 8.24]{fooo:bulk},
where a similar lemma is proved in the case of Lagrangian Floer theory.)
We give a proof in the appendix for completeness' sake.
\begin{rem}
\begin{enumerate}
\item
Note we need to take care of the problem of running-out, which is explained in detail in \cite[Section 7.2.3]{fooo:book2}.
We explain it briefly below.
In our situation we need to perturb infinitely many
moduli spaces such as ${\CM}(H,J;[\gamma,w], [\gamma',w'])$
simultaneously.
It is in general difficult to do so.
We go around this problem as follows.
For each $E>0$ we have only finitely many moduli spaces
for which symplectic area of its element is less than $E$.
We can perturb them appropriately
for any fixed $E$.
(The perturbation depends on $E$, however.)
By doing so, we obtain a chain complex $C_E$ over $\Lambda_0^\downarrow/q^{-E}\Lambda_0^\downarrow$
for any $E$.
We can also prove that, if $E' > E$ then $C_{E'} \mod q^{-E}\Lambda_0^\downarrow$
is chain homotopy equivalent to $C_E$.
We can use this fact and an elementary homological
algebra to lift $C_E$ to a chain complex over $\Lambda^\downarrow_0/q^{-E'}\Lambda_0^\downarrow$
so that it becomes chain homotopy equivalent to $C_{E'}$.
We repeat this process and by taking a projective limit we obtain a
chain complex over $\Lambda^\downarrow_0$ whose reduction is
chain homotopy equivalent to $C_{E}$ for any $E$.
\par
The technical difficulty to perform this construction in our situation is much
simpler than that of \cite[Section 7.2]{fooo:book2}, since
we here need to take a projective limit of chain complex (or DGA)
which is much simpler than that of $A_{\infty}$ algebra in general (which is discussed in
\cite[Section 7.2]{fooo:book2}).
So we omit the detail.
\item
Here we work over  the $\Lambda^\downarrow$ coefficients and so only with the $\Z_2$ grading
given in (\ref{35formula}).
However if the minimal Chern number is $2N$, we can define a
$\Z_{2N}$-grading.
\item
Here we use $\C$ as the ground field.  Up until now, we can
work with $\Q$ in the same way. We prefer to use $\C$ in this memoir since we will use de Rham theory
later on to involve bulk deformations in our constructions.
In addition, the de Rham theory is used
for the Lagrangian Floer theory of toric manifolds in various
calculations and applications developed in \cite{fooo:toric1} etc.
\item
We use de Rham cohomology of $M$ to define $\mathcal P_{(H_\chi,J_\chi)}$.
There are several other ways of constructing this isomorphism.
One uses singular (co)homology as in \cite[Section 7.2]{fooo:book2}
(especially \cite[Proposition 7.2.21]{fooo:book2}) and references therein. This approach allows
one to work with $\Q$ coefficients, which may have some additional applications.
Other uses Morse homology as proposed in \cite{rt, pss}.
The necessary analytic details of the latter approach
has been established recently in \cite{oh-zhu}.
\end{enumerate}
\end{rem}

\section{Spectral invariants}
\label{sec:spectral}

The very motivating example of Floer-Novikov complex and its chain level
theory was applied by the second named author in the Hamiltonian Floer
theory \cite{oh:alan,oh:minimax}.
Namely, a spectral number which we denote by $\rho(H;a)$ is associated
to $a \in QH(M;\Lambda) = H^*(M) \otimes \Lambda$
and a Hamiltonian $H$, and is proved to be independent
of various choices, especially of $J$ in \cite{oh:alan}.
Via the canonical identification $H^*(M) \otimes \Lambda \cong H^*(M) \otimes \Lambda^\downarrow$
induced by the change of the formal variable $T \mapsto T^{-1} =: q$ which reverses the direction of valuation,
we often abuse our notation for $a$ also to denote this formal change.
\par
In this section we give a brief summary of this construction.
Let $H : S^1 \times M \to \R$ be a normalized time-dependent nondegenerate Hamiltonian.
\par
\begin{defn}\label{criticalvalue}
We denote the set of (spherical) symplectic periods of $(M,\omega)$ by
$$
G(M,\omega) = \{ \alpha \cap [\omega] \mid \alpha \in \pi_2(M)\}.
$$
The set $G(M,\omega)$ is a countable subset of $\R$ which is a
subgroup of the additive group of $\R$ as a group. It may or may not be discrete.

For each given $H$, we define the {\it action spectrum} of $H$ by
$$
\mbox{\rm Spec}(H) := \{\CA_H(\gamma,w)\in \R ~|~ [\gamma,w] \in
\mbox{\rm Crit}(\CA_H) \},
$$
i.e., the set of critical values of $\CA_H: \widetilde{\CL_0}(M)
\to \R$.\index{action spectrum}
\end{defn}

Then for any given $H$ the subgroup $G(M,\omega)$
of $(\R,+)$ acts on $\text{\rm Spec}(H)$ by an addition.

\begin{defn}\label{submonoid}
Let $G \subset \R$  be a submonoid. We say that a subset $G' \subset \R$ is a \emph{$G$-set}
\index{$G$-set} if
$g\in G, g'\in G'$ implies $g+g' \in G'$.
\end{defn}\index{$G$-set}
With this definition, $\mbox{Spec}(H)$ is $G(M,\omega)$-set

\begin{lem}\label{Gomegaaction}
If $\lambda \in \mbox{\rm Spec}(H)$ and
$g \in G(M,\omega)$ then
$\lambda \pm g\in  \mbox{\rm Spec}(H)$.
\par
If $H$ is nondegenerate, then
the quotient space
$\mbox{\rm Spec}(H)/G(M,\omega)$ with the above action is a finite set and
$$
\# \left(\mbox{\rm Spec}(H)/G(M,\omega)\right) \le \# \text{\rm Per}(H).
$$
\end{lem}
\begin{proof}
Let $[\gamma,w], [\gamma,w'] \in \mbox{\rm Crit}(\CA_H)$.
We glue $w$ and $w'$ along $\gamma$ and obtain
$\overline w \# w'$. Its homology class in $H_2(M;\Z)$ is well-defined.
We have
$$
\CA_H([\gamma,w]) - \CA_H([\gamma,w'])
= \int_{\overline w \# w'} \omega.
$$
The lemma follows easily from this fact  and the fact that $\mbox{\rm Per}(H)$ is a finite set.
\end{proof}

\begin{rem} One can actually make the equality
$\# \left(\mbox{\rm Spec}(H)/G(M,\omega)\right) = \# \text{\rm Per}(H)$ hold
by further perturbing $H$, but we will not do that because it is not needed.
\end{rem}
\begin{lem}\label{lem4545}
$G(M,\omega)$ acts on $\widehat{\Per}(H)$
such that
$\widehat{\Per}(H)/G(M,\omega) = {\Per}(H)$
and
$
\mathcal A_H(g\llb \gamma,w \rrb) = g + \mathcal A_H(\llb \gamma,w \rrb)$.
\end{lem}
\begin{proof}
It is obvious from definition that $\pi_2(M)$ acts on $\text{\rm Crit}(\mathcal A_H)$
such that $\text{\rm Crit}(\mathcal A_H)/\pi_2(M) \cong \text{\rm Per}(H)$.
The lemma then follows from the fact that
$G(M,\omega)$ is the image of (\ref{energymapmap}) and Lemma \ref{determinedifference}
in the same way as the proof of Lemma \ref{Gomegaaction}.
\end{proof}

The following definition is standard.
\begin{defn}\index{Hamiltonian!homotopic}We say that two one-periodic Hamiltonians $H$ and $H'$
are {\it homotopic} if
\index{homotopic ! Hamiltonian}  $\phi_H^1 = \phi_{H'}^1$ and if
there exists $\{H^s\}_{s\in [0,1]}$ a one-parameter family of
one periodic Hamiltonians such that $H^{0} = H$, $H^1 = H'$ and
$\phi_H^1 = \phi_{H^s}^1$ for all $s \in [0,1]$.
In this case we denote $H \sim H'$ and denote the set of equivalence classes
by $\widetilde{\rm Ham}(M,\omega)$.
\end{defn}

The following lemma was proven in \cite{schwarz, pol:book} in the aspherical case
and in \cite{oh:jkms} for the general case.
We provide its proof in Section \ref{sec:proofhomotopy} for reader's convenience.

\begin{prop} Suppose that $H, \, H'$ are normalized one periodic Hamiltonians.
If $H \sim H'$, we have
$
\mbox{\rm Spec}(H) = \mbox{\rm Spec}(H')
$
as a subset of $\R$.
\end{prop}
This enables one to make the following definition

\begin{defn}
We define the spectrum
of $\widetilde \psi \in \widetilde{\rm Ham}(M,\omega)$ to be
$
\mbox{Spec}(\widetilde \psi) := \mbox{Spec}(\underline H)
$
for a (and so any) Hamiltonian $H$ satisfying $\widetilde \psi=[\phi_H]$.
\end{defn}

Here we denote the {\em normalization} \index{normalization ! Hamiltonian} of $H$ by
$$
\underline H(t,x) = H(t,x) - \frac{1}{\vol_\omega(M)}\int_M H_t\, \omega^n.
$$
\index{$\underline H(t,x)$}

\begin{defn}\label{Lambda(G)}
Let $G$ be a subset of $\R$, which is a monoid.
We denote by $\Lambda^\downarrow(G)$ the set of all elements
$$
\sum a_iq^{\lambda_i} \in \Lambda^\downarrow
$$
such that if $a_i\ne 0$ then $- \lambda_i \in G$.
We note that
$\Lambda^\downarrow(G)$ forms a subring of $ \Lambda^\downarrow$ and
$\Lambda^\downarrow(G)$ is a field if $G$ is a subgroup of
the (additive) group $\R$.
We write
$$
\Lambda^\downarrow(M) =
\Lambda^\downarrow(G(M,\omega)).
$$
\index{$\Lambda^\downarrow(M)$}
\par
Suppose that $H$ is nondegenerate. We denote by $CF(M,H)$
\index{$CF(M,H)$}
the subset of $CF(M,H;\Lambda^\downarrow)$ consisting of the infinite sums
\begin{equation}\label{CFinfinitesum}
\sum_{\llb \gamma,w \rrb \in \widehat{\text{\rm Per}}(H)} a_{\llb \gamma,w \rrb} \llb \gamma,w \rrb
\end{equation}
with $a_{\llb \gamma,w \rrb}\in \C$. We denote by  $F^{\lambda}CF(M,H)$ the subset of $CF(M,H)$
consisting of elements (\ref{CFinfinitesum})
such that
$\mathcal A_H(\llb \gamma,w \rrb) \le \lambda$.
\end{defn}
\begin{lem}
\begin{enumerate}
\item
$CF(M,H)$ is a vector space over $\Lambda^\downarrow(M)$.
\item
$\{\llb \gamma,w_\gamma \rrb \mid \gamma \in \text{\rm Per}(H) \}$ is a basis of $CF(M,H)$
over $\Lambda^\downarrow(M)$.
\item
We have
$$
CF(M,H;\Lambda^\downarrow)
\cong CF(M,H) \otimes_{\Lambda^\downarrow(M)} \Lambda^\downarrow.
$$
\item
The Floer boundary operator $\partial_{(H,J)}$ preserves the submodule
$CF(M,H) \subset CF(M,H;\Lambda^\downarrow)$.
\end{enumerate}
\end{lem}
This is an easy consequence of Lemma \ref{lem4545}.

\begin{lem} The chain map $\CP_{(H_\chi,J_\chi)}$
in $(\ref{Piunikin})$ induces a $\Lambda^\downarrow(M)$-linear map, which we still denote
$$
\CP_{(H_\chi,J_\chi)} : C(M;\Lambda^\downarrow(M)) \to CF(M,H;\Lambda^\downarrow(M))
$$
which are chain-homotopic to one another for different choices of $\chi$.
\end{lem}
This is immediate from the definition of $\CP_{(H_\chi,J_\chi)}$.
Therefore this together with Theorem \ref{Piuiso} gives rise to an isomorphism
\begin{equation}
\CP_{(H_\chi,J_\chi),\ast} : H_*(M;\Lambda^\downarrow(M))
\cong HF_*(M,H)
\end{equation}
where the right hand side is the homology of $(CF(M,H),\partial)$.
This isomorphism does not depend on the choice of $\chi$'s. We also use the
isomorphism, which is the $\Lambda-\Lambda^{\downarrow}$ linear extension of
the Poincar\'e duality isomorphism, denoted by
$$
\flat: QH^*(M;\Lambda) = H^*(M) \otimes \Lambda \to H_*(M) \otimes \Lambda^\downarrow \cong H_*(M,\Lambda^\downarrow)
$$
which is nothing but the map induced by first taking the Poincar\'e duality and then
changing the formal variable $T \to q = T^{-1}$. For the simplicity
of notation, we will just denote $QH^*(M) = QH^*(M;\Lambda)$.
\par
The filtration $F^{\lambda}CF(M,H;\Lambda^\downarrow)$
induces a filtration $F^{\lambda}CF(M,H)$ on $CF(M,H)$ in an obvious way.
\begin{defn}\index{$\rho(\frak x)$}
\begin{enumerate}
\item
Let $\frak x \in HF(H,J)$ be any nonzero Floer homology class.
We define its {\it spectral invariant} \index{spectral invariant!definition}$\rho(\frak x)$ by
$$
\rho(\frak x)
= \inf \{\lambda \mid x \in F^{\lambda}CF(M,H;\Lambda^\downarrow),
\,\, \partial_{(H,J)} x = 0, \,\, [x] = \frak x\}.
$$
\item
If $a \in H^*(M;\Lambda(M))$
and $H$ is a nondegenerate time-dependent Hamiltonian, we define
the {\it spectral invariant} $\rho(H;a)$ by
$$
\rho(H,J;a) = \rho(a^\flat_H), \quad a^\flat_H: = \CP_{(H_\chi,J_\chi),\ast}(a^\flat)
$$
where the right hand side is defined in $(1)$ and $a^\flat$ is the Poincar\'e dual of the cohomology class $a$. See Notations and Conventions (22).
(Here $\Lambda(M) \subset \Lambda$ is a subring
which is identified with $\Lambda^{\downarrow}(M) \subset \Lambda^{\downarrow}$
by the canonical isomorphism $\Lambda \cong \Lambda^{\downarrow}$.)
\index{$\Lambda(M)$}
\end{enumerate}
\end{defn}
It is proved in \cite{oh:alan,oh:minimax} that $\rho(H,J;a)$ is
independent of $J$. The same can be proved in general under other choices involved
in the definition such as the abstract perturbations
in the framework of Kuranishi structure \emph{as long as the Hamiltonian
$H$ is fixed}. So we omit $J$ from notation and just denote it by
$\rho(H;a)$.

We introduce the following standard invariants associated to the
Hamiltonian $H:[0,1] \times M \to \R$ called the {\it positive and  negative parts of Hofer norm}
$E^\pm(H)$ \index{Hofer norm} \index{Hofer norm!negative part}\index{Hofer norm!positive part}
\bea\label{eq:=-Hofernorm}
\index{$E^+(H)$} \index{$E^-(H)$}
E^+(H)
& : = & \int_{t\in S^1} \max_x H_t \, dt \\
E^-(H)
& : = & \int_{t\in S^1} - \min_x H_t \, dt
\eea
for any Hamiltonian $H$. We have the {\em Hofer norm} \index{Hofer norm! Hamiltonian} \index{$\Vert H\Vert$}
$\|H\| = E^+(H) + E^-(H)$. We like to emphasize that $H$ is not necessarily
one-periodic time-dependent family.
\begin{lem}\label{lem:E+rhoE-}
We have
$$
-E^+(H'-H) \leq \rho(H';a) - \rho(H;a) \leq
E^-(H'-H).
$$
\end{lem}
This lemma enables one to
extend, by continuity, the definition of $\rho(H;a)$ to any Hamiltonian $H:S^1 \times M \to \R$
which is not necessarily nondegenerate.
Lemma \ref{lem:E+rhoE-} is proved in a generalized form as Theorem \ref{contspect}.

The following homotopy invariance is also proved
in \cite{oh:alan,oh:minimax,usher:specnumber}.

\begin{thm}[{\rm Homotopy invariance}]
\label{homotopyinvtheorem} Suppose $H, \, H'$ are normalized.
If $H \sim H'$ then
$
\rho(H;a) = \rho(H';a).
$
\end{thm}
We will prove it in Section \ref{sec:proofhomotopy} for completeness. This homotopy
invariance enables one to extend the definition of $\rho(H;a)$ to non-periodic
$H:[0,1] \times M \to \R$.
\par
Consider the set of smooth functions $\zeta:[0,1] \to [0,1]$
satisfying $\zeta(0) = 0$, $\zeta(1) = 1$
and $\zeta$ is constant in a neighborhood of $0$ and $1$. Note that this set is convex and so contractible.
Let $H: [0,1] \times M \to \R$ and $\zeta:[0,1] \to [0,1]$ be such a function. Denote
$$
H^\zeta(t,x) = \zeta'(t)H(\zeta(t),x).
$$
We note that $H^\zeta$ may be regarded as a map for $S^1 \times M$
since $H = 0$ in a neighborhood of $\{0,1\} \times M$.
Moreover the above mentioned convexity implies that
$H^{\zeta_1} \sim H^{\zeta_2}$.
Therefore $\rho(H^{\zeta_1};a) = \rho(H^{\zeta_2};a)$ for any such $\zeta_i$.
We define the common number to be $\rho(H;a)$.
This gives rise to the map
\begin{equation}\label{spH}
\rho: C^\infty([0,1] \times M,\R) \times (QH^*(M) \setminus\{0\}) \to \R.
\end{equation}
Its basic properties are summarized in the next theorem.
For given $a \in QH^*(M)$, we represent it by
$$
a = \sum_{g\in G(M,\omega)} a_g T^{-g}  = \sum_{g\in G(M,\omega)} a_g q^g, \quad a_g \in H^*(M;\C).
$$
We define \index{$\lambda_q$}
\begin{equation}\label{defvq}
\lambda_q(a): =  \max \{g \in \R ~|~  a_g \neq 0 \}.
\end{equation}
\begin{thm}\label{thm:axiomsh-H} \index{spectral invariant!axiom}
{\rm (Oh)} Let $(M,\omega)$ be any closed symplectic manifold.
Then the map $\rho$ in $(\ref{spH})$ satisfies the following properties:
Let $H, \, H' \in C^\infty([0,1] \times M,\R)$ and $0 \neq a \in QH^*(M)$.
Then
\begin{enumerate}
\item {\rm (Nondegenerate spectrality)}
$\rho(H;a) \in \mbox{\rm Spec}(H)$,
if $\widetilde \psi_H$ is nondegenerate.
\item {\rm (Projective invariance)}
$\rho(H;\lambda a) = \rho(H;a)$ for any
$0 \neq \lambda \in \C$.
\item {\rm (Normalization shift)} For any function $c:[0,1] \to \R$,
$\rho(H + c(t);a) = \rho(H;a) - \int_0^1 c(t)dt$.
\item {\rm (Normalization)}
$\rho(\underline 0;a) = \frak v_q(a) = - \frak v_T(a)$,
where $\underline 0$ is the identity in
$\widetilde{\rm Ham}(M,\omega)$.
\index{$\underline 0$}
\item {\rm (Symplectic
invariance)} $\rho(H\circ \eta;\eta^*a) =
\rho(\widetilde\phi;a)$ for any symplectic diffeomorphism $\eta$.
In particular, if $\eta \in {\rm Symp}_0(M,\omega)$,
then we have $\rho(H\circ \eta;a) =
\rho(H;a)$.
\item {\rm (Triangle inequality)}
$\rho(H\# H'; a\cup_Q b) \leq \rho(H;a) + \rho(H';b) $,
where $a\cup_Q b$ is the quantum cup product.
\item {\rm ($C^0$-hamiltonian continuity)} We have
$$
-E^+(H'-H) \leq \rho(H';a) - \rho(H;a) \leq
E^-(H'-H).
$$
\item {\rm (Additive triangle inequality)}
$\rho(H;a+b) \le \max\{\rho(H;a),\rho(H;b)\}$.
\end{enumerate}
\end{thm}

Theorem \ref{thm:axiomsh-H} is stated by the second named author \cite{oh:alan,oh:minimax}
in the general context but without detailed account on the construction of
virtual fundamental classes in the various moduli spaces entering in the proofs.
In the present paper, we provide these details in the framework of Kuranishi
structures \cite{fukaya-ono}. A purely algebraic treatment of the
statement (1) is given by Usher \cite{usher:specnumber}.

By considering the normalization $\underline H(t,x)$ of
$H(t,x)$ and the one-one correspondence between $\underline{H}$ and its
Hamiltonian path $\phi_H$, we can interpret
$\rho(H;a)$ as the invariant of the associated Hamiltonian path $\phi_H$ by setting
$$
\rho(\phi_H;a) := \rho(\underline H;a).
$$
The invariance of $H \mapsto \rho(H;a)$ under the equivalence relation $H\sim H'$
enables one to push this down to $\widetilde{\Ham}(M,\omega)$ which we denote $\rho (\phi_H;a)$.
We denote  the resulting map by
\begin{equation}\label{eq:sptildepsi}
\rho: \widetilde{\rm Ham}(M, \omega) \times (QH^*(M) \setminus\{0\}) \to \R.
\end{equation}
Its basic properties are summarized in the next theorem, which are
immediate translation of those stated in Theorem
\ref{thm:axiomsh-H}.

\begin{thm}\label{axiomsh}
Let $(M,\omega)$ be any closed symplectic manifold. Then the map
$\rho$ in $(\ref{eq:sptildepsi})$ has the following properties: Let
$\widetilde \psi, \widetilde \phi \in \widetilde{\rm Ham}(M,\omega)$
and $0 \neq a \in H^{*}(M; \Lambda(M))$.
\begin{enumerate}
\item {\rm (Nondegenerate spectrality)}
$\rho(\widetilde \psi;a) \in \mbox{\rm Spec}(\widetilde \psi)$,
if $\widetilde \psi$ is nondegenerate.
\item {\rm (Projective invariance)}
$\rho(\widetilde\phi;\lambda a) = \rho(\widetilde\phi;a)$ for any
$0 \neq \lambda \in \C$.
\item {\rm (Normalization)}
We have $\rho(\underline 0;a) =
\frak v_q(a)$.
\item {\rm (Symplectic
invariance)} $\rho(\eta \circ\widetilde\phi\circ \eta^{-1};\eta^*a) =
\rho(\widetilde\phi;a)$ for any symplectic diffeomorphism $\eta$. In particular, if $\eta \in {\rm Symp}_0(M,\omega)$,
then we have $\rho(\eta \circ\widetilde\phi\circ \eta^{-1};a) =
\rho(\widetilde\phi;a)$.
\item {\rm (Triangle inequality)} $\rho(\widetilde\phi \circ
\widetilde\psi; a\cup_Q b) \leq \rho(\widetilde\phi;a) +
\rho(\widetilde\psi;b) $,
where $a\cup_Q b$ is the quantum cup product.
\item {\rm ($C^0$-hamiltonian continuity)} We have
$$
 |\rho(\widetilde\phi;a) - \rho(\widetilde\psi;a)| \leq
 \max\{
 \|\widetilde\phi \circ \widetilde\psi^{-1} \|_+,
\|\widetilde\phi \circ \widetilde\psi^{-1} \|_-
\}
$$
where $\| \cdot\|_\pm$ is the positive and negative parts of
Hofer norm on $\widetilde{\rm Ham}(M,\omega)$.
\item {\rm (Additive triangle inequality)}
$\rho(\widetilde\phi;a+b) \le \max\{\rho(\widetilde\phi;a),\rho(\widetilde\phi;b)\}$.
\end{enumerate}
\end{thm}
Here we explain the meaning of the negative and positive parts of {\em Hofer norm}
\index{Hofer norm! Hamiltonian diffeomorphism}
$\|\widetilde \psi\|_\pm$. For $\widetilde \psi \in \widetilde{\rm Ham}(M,\omega)$, we define
\be\label{eq:tildepsi=-norm}
\|\widetilde \psi\|_\pm = \inf_H \{E^\pm(H) \mid   [\phi_H] = \widetilde \psi\}
\ee
respectively, and the (strong) \emph{Hofer norm} $\|\widetilde\psi\|$ is
defined by
\be\label{eq:strongHofernorm}
\|\widetilde \psi\| = \inf_H \{\|H\| \mid   [\phi_H^1] = \widetilde \psi \}.
\ee
There is another norm, sometimes called the \emph{medium Hofer norm},\index{Hofer norm!medium} which is defined by
\be\label{eq:mediumHofernorm}
\|\widetilde \psi\|_{\text{med}} = \|\widetilde\psi\|_+ + \|\widetilde\psi\|_-.
\ee
Obviously we have
$$
|\rho(\widetilde \psi;a) - \rho(\widetilde{id};a)| \leq \|\widetilde \psi\|_{\text{med}} \leq \|\widetilde \psi \|
$$
for all $a \in QH^*(M;\Lambda(M))$.
Here $\widetilde{id}$ stands for the constant Hamiltonian isotopy at the identity.
If $a \in H^*(M;\C) \subset QH^*(M;\Lambda(M))$, we find that
$$
|\rho(\widetilde \psi;a)| \leq \|\widetilde \psi\|_{\text{med}} \leq \|\widetilde \psi\|.
$$
See the introduction of \cite{oh:dmj} for the related discussion.

\begin{rem}
There is another important property, that is
compatibility with Poincar\'e duality observed by Entov-Polterovich \cite{EP:morphism}
in the case $M$ is semi-positive and $M$ is rational.
Those assumptions are removed by Usher \cite{usher:duality}.
We will discuss some enhancement of this point later in Section \ref{sec:duality}.
\end{rem}
We refer readers to the above references for
the proof of Theorems \ref{axiomsh}. Later we will prove
its enhancement including bulk deformations.
Here are some remarks.
\begin{rem}\label{rem:unnormalized} We like to note that the constructions of $\rho(H;a)$
given in \cite{oh:alan} can be carried over whether or not $H$ is normalized. We need the normalization
only to descend the spectral function $H \mapsto \rho(H;a)$ to
the universal covering space $\widetilde{\rm Ham}(M,\omega)$ as in
Theorem \ref{axiomsh}.
\end{rem}

\begin{rem}\label{rem:EPversusFOOO}
In \cite{EP:morphism,EP:states,EP:rigid},
Entov-Polterovich used different sign conventions from the ones \cite{oh:alan}
and the present paper. If we compare our convention with
the one from \cite{EP:rigid}, the only difference lies in the
definition of Hamiltonian vector field: our definition, which is the
same as that of \cite{oh:alan}, is given by
$
dH = \omega(X_H, \cdot)
$
while \cite{EP:rigid} takes
$
dH = \omega(\cdot, X_H).
$
Therefore by replacing $H$ by $-H$, one has the same set of
closed loops as the periodic solutions of the corresponding
Hamiltonian vector fields.

This also results in the difference in the definition of action functional:
our definition, the same as the one in \cite{oh:alan}, is given by
\be\label{eq:EPAH1}\index{action functional}
\CA_H([\gamma,w]) = - \int w^* \omega - \int_0^1 H(t,\gamma(t))\, dt
\ee
while
\cite{EP:morphism} and \cite{EP:rigid} takes
\be\label{eq:EPAH2}
- \int w^* \omega + \int_0^1 H(t,\gamma(t))\, dt
\ee
as its definition.
We denote the definition \eqref{eq:EPAH2} by $\widetilde{\CA}_H([\gamma,w])$
for the purpose of comparison of the two below.

Therefore \emph{under the change of $H$ by $-H$}, one has the same set of
$\Crit(\CA_H)$ and $\Crit(\widetilde{\CA}_H)$ with the same action integrals.
Since both conventions use the same associated almost K\"ahler
metric $\omega(\cdot, J \cdot)$, the associated perturbed Cauchy-Riemann
equations are exactly the same.

In addition, Entov and Polterovich \cite{EP:morphism,EP:states}
use the notation $c(\frak a,H)$ for the spectral numbers where
$\frak a$ is the quantum \emph{homology} class. Our $\rho(H;a)$ is nothing but
\be\label{eq:cversusrho}
\rho(H;a) = c(a^\flat;\widetilde H) = c(a^\flat;\overline H)
\ee
where $a^\flat$ is the homology class Poincar\'e dual to
the cohomology class $a$  and $\widetilde H$ is the function defined by
$\widetilde H(t,x) = -H(1-t,x)$ and $\overline H$ is the inverse Hamiltonian of $H$ given by
\be\label{eq:inverseH}
\overline H(t,x) = - H(t,\phi_H^t(x)).
\ee
The second identity of \eqref{eq:cversusrho}
follows from the fact that $\widetilde H \sim \overline H$. More precisely,
$\widetilde H$ generates the flow
$
\phi_{\widetilde H}:\phi_H^{1-t}\circ (\phi_H^1)^{-1}
$
which can be deformed to $\phi_{\overline H}: t \mapsto (\phi_H^t)^{-1}$.
In fact the following explicit formula
provides such a deformation
\be\label{eq:phist}
\phi_s^t = \begin{cases} \phi_H^{s-t}\circ (\phi_H^s)^{-1} \quad & \mbox{for } \, 0 \leq t \leq s \\
(\phi_H^t)^{-1} \quad & \mbox{for }\, s \leq t \leq 1
\end{cases}
\ee
for $0 \leq s \leq 1$.
(See the proof of \cite[Lemma 5.2]{oh:dmj}  for this formula.)
\par
With these understood, one can translate all the statements
in \cite{EP:morphism,EP:states} into the ones in terms of our notations.
\end{rem}

\part{Bulk deformations of Hamiltonian Floer homology and spectral invariants}
\label{part:bulk-hamFloer}

In this part, we deform Hamiltonian Floer homology by
the element $\frak b \in H^{{\rm even}}(M,\Lambda_0)$
in a way similar to the case of Lagrangian Floer theory
in \cite[Section 3.8]{fooo:book1}.
We will denote the resulting $\frak b$-deformation by
$
HF_*^{\frak b}(H,J_0;\Lambda).
$
As a $\Lambda $-module, it is isomorphic to the singular
homology $H_*(M;\Lambda)$ for any $\frak b$. Recall that we regard
the de Rham complex as a chain complex \eqref{eq:deRhamchain}.
\par
Using the filtration we obtain a version of spectral invariants,
the spectral invariants with bulk deformation,
which contains various new information as we demonstrate later in Part 5.

\section{Big quantum cohomology ring: review}
\label{sec:bigquantum} \index{big quantum cohomology}

In this section, we exclusively denote by $J_0$ the \emph{time-independent}
almost complex structures.

The theory of spectral invariants explained in  Part 1 is closely related to the
{\it (small)} quantum cohomology.
The spectral invariant with bulk we are going to construct is closely related to the
{\it big} quantum cohomology,
which we review in this section.
\par
Let $(M,\omega)$ be a closed symplectic manifold and $J_0$ a compatible
(time independent) almost complex structure. For $\alpha \in H_2(M;\Z)$ let
$\mathcal M_{\ell}^{\rm cl}(\alpha;J_0)$ \index{$\mathcal M_{\ell}^{\rm cl}(\alpha;J_0)$}
 be the moduli space of
$J_0$-stable maps of genus zero, with $\ell$-marked points and of homology
class $\alpha$.
(See \cite[Definition 7.7]{fukaya-ono}  for example for the
definition of this moduli space.)
There exists an evaluation map
$$
\text{\rm ev}: \mathcal M_{\ell}^{\rm cl}(\alpha;J_0) \to M^{\ell}.
$$
The moduli space $\mathcal M_{\ell}^{\rm cl}(\alpha;J_0)$
\index{$\mathcal M_{\ell}^{\rm cl}(\alpha;J_0)$} has a virtual fundamental cycle
and hence defines a class
$$
\text{\rm ev}_*[\mathcal M_{\ell}^{\rm cl}(\alpha;J_0)] \in H_{*}(M^{\ell};\Q).
$$
(See \cite[Theorem 1.3]{fukaya-ono}.) Here the sub-index $*$ in $H_{*}(M^{\ell};\Q)$
is given by
$
* = 2n + 2c_1(M)(\alpha) + 2\ell - 6
$.
Let $h_1,\dots,h_{\ell}$ be closed differential forms on $M$ such that
\begin{equation}\label{degreecondGW}
\sum \text{\rm deg}\, h_i = 2n + 2c_1(M)(\alpha) + 2\ell - 6.
\end{equation}
We define Gromov-Witten invariant by
\begin{equation}\label{GWinteg}\index{$GW_{\ell}$}
GW_{\ell}(\alpha:h_1,\dots,h_{\ell}) =
\int_{\mathcal M_{\ell}^{\rm cl}(\alpha;J_0)}
\text{\rm ev}^* (h_1 \times \dots \times h_{\ell})
\in \R.
\end{equation} \index{Gromov-Witten invariant!definition}
More precisely, we take a CF perturbation of the Kuranishi structure of
$\mathcal M_{\ell}^{\rm cl}(\alpha;J_0)$ and define the integration in (\ref{GWinteg})
using it.
(See (\ref{form34333}) and \cite[Definition 10.22]{fooo:tech2}.) We can prove that (\ref{GWinteg}) is independent of
the almost complex structure $J_0$. \index{$GW_{\ell}(\alpha:h_1,\dots,h_{\ell})$}
We put $GW_{\ell}(\alpha:h_1,\dots,h_{\ell}) = 0$ unless (\ref{degreecondGW})
is satisfied. We now define
\begin{equation}\label{sumGW}
GW_{\ell}(h_1,\dots,h_{\ell})
= \sum_{\alpha} T^{\alpha \cap \omega} GW(\alpha:h_1,\dots,h_{\ell})
\in \Lambda.
\end{equation}
By Stokes' theorem ((\ref{stokessss}), \cite[Theorem 8.11]{fooo:tech2})
we can prove that $GW_{\ell}(h_1,\dots,h_{\ell})$
depends only on the de Rham cohomology class of $h_i$ and is independent
of the closed forms $h_i$ representing de Rham cohomology class.
\par
By extending the definition (\ref{sumGW}) linearly over to a $\Lambda $-module homomorphism,
we obtain:
$$
GW_{\ell}: H(M;\Lambda)^{\ell\otimes} \to \Lambda.
$$
\begin{defn}\label{defn:prod}\index{bulk deformation!quantum product}
Let $\frak b \in H^{{\rm even}}(M;\Lambda _0)$.
For each given pair $\frak c,\frak d \in H^*(M; \Lambda)$, we define a product $\frak c
\cup^{\frak b} \frak d \in H(M;\Lambda) $ by the following formula
\index{$\cup^{\frak b}$}
\begin{equation}\label{defcup}
\langle \frak c \cup^{\frak b} \frak d, \frak e\rangle_{\text{\rm PD}_M}
= \sum_{\ell=0}^{\infty} \frac{1}{\ell!} GW_{\ell+3}(\frak c,\frak d,\frak e,\frak b,\dots,\frak b).
\end{equation}
Here $\langle \cdot,\cdot\rangle_{\text{\rm PD}_M}$ denotes the
Poincar\'e duality. We call $\cup^{\frak b}$ the {\it deformed
quantum cup product}. \index{deformed quantum product}
\end{defn}
\par
\begin{rem}\label{rem:frakbL0}
We note that the right  hand side of (\ref{defcup}) is an infinite sum in general.
If $\frak b \in H^{{\rm even}}(M;\Lambda_+ )$, it converges in
$q$-adic topology so (\ref{defcup}) makes sense.
For a general element
$\frak b \in H^{{\rm even}}(M;\Lambda _0)$, we split
\begin{equation}\label{decompb}
\frak b = \frak b_0 + \frak b_2 +\frak b_+
\end{equation}
with $\frak b_0 \in H^{0}(M;\Lambda _0)$,
$\frak b_2 \in H^{2}(M;\C)$,  and $\frak b_+ \in H^2(M;\Lambda _+)
\oplus
\bigoplus_{k\ge 2}H^{2k}(M;\Lambda _0)$ and define
\begin{equation}\label{GWseqconvform}
\langle \frak c \cup^{\frak b} \frak d, \frak e\rangle_{\text{\rm PD}_M}
= \sum_{\ell=0}^{\infty}
\sum_{\alpha}\frac{\exp({\frak b_2\cap \alpha})}{\ell!}
T^{\alpha\cap \omega}GW_{\ell+3}(\alpha:\frak c,\frak d,\frak e,\frak b_+,\dots,\frak b_+).
\end{equation}
We can prove that (\ref{GWseqconvform}) converges in $q$-adic topology.
(This can be proved in the same way as in \cite[Section 9]{fooo:bulk}.
See \cite[Lemma 2.29]{fooo:toricmir}.)
We put the factor $\ell!$ since our set of interior marked points is ordered.
\begin{rem}\label{rem53}
Here  the factor $T^{\alpha\cap \omega}$  is the usual weight we put in Gromov-Witten
invariant and in Lagrangian Floer theory. The appearance of the exponential factor
$\exp({\frak b_2\cap \alpha})$ is related to the following formula
$$
GW_{\ell +3}(\alpha;\frak c,\frak d,\frak e,\frak b_2, \dots, \frak b_2)
= \frac{(\frak b_2 \cap \alpha)^{\ell}}{\ell !}GW_3(\alpha : \frak c,\frak d, \frak e).
$$
This formula is called  the divisor axiom (See for example \cite[Theorem 23.1.4]{fukaya-ono}).
In fact if we formally `expand' the formal sum
$$
\sum_{\ell=0}^{\infty}
\sum_{\alpha}\frac{1}{\ell!} T^{\alpha \cap \omega}GW_{\ell+3}(\alpha:\frak c,\frak d,\frak e,\frak b,\dots,\frak b)
$$
by substituting $\frak b = \frak b_2 + \frak b_+$, a formal calculation based on the divisor axiom
gives rise to \eqref{GWseqconvform}. However this sum as it is does not make sense since it
does \emph{not} converge in the $q$-adic topology.
\end{rem}
Geometrically considering the element $\frak b \in H^{2}(M;\Lambda _0)$ corresponds to
twisting the Hamiltonian Floer theory by a $B$-field and is the analog to
Cho's trick of considering nonunitary line bundles \cite{fukaya:family,cho:Bfield}.
(We remark that this $q$-adic convergence of Gromov-Witten invariant had been
known for a long time.)
\end{rem}
It is now well-established that the product $\cup^{\frak b}$
is associative and graded commutative and is independent of $J_0$.
We thus obtain a $\Z_2$-graded commutative ring \index{$QH^*_{\frak b}(M)$}
$$
QH^*_{\frak b}(M)
= (H^*(M;\Lambda),\cup^{\frak b}) \cong (H^*(M) \otimes \Lambda,\cup^{\frak b}).
$$
As we will see later,
for the purpose of construction of spectral invariants and of
partial symplectic quasi-states and quasi-morphisms,
it is important to use a smaller Novikov ring than $\Lambda$.
We discuss this point now.

\begin{defn}\label{def:deformcup} \index{$G$-gapped}
Let $G$ be a discrete submonoid of $\R$. We say an element
$\frak b \in H(M;\Lambda _0)$ to be $G$-{\it gapped} if $\frak b$ can be written as
$$
\frak b = \sum_{g \in G} T^{g} b_g, \qquad b_g \in H(M;\C).
$$
\end{defn}
For each $\frak b \in H(M;\Lambda _0)$ there exists a smallest discrete submonoid $G$ such that
$\frak b$ is $G$-gapped. We write this monoid as $G_0(\frak b)$.
\index{$G_0(\frak b)$}
Let
$
G_0(M,\omega)
$
be the monoid generated by the set
$$
\{ \alpha \cap \omega \mid \mathcal M_{\ell}^{\rm cl}(\alpha;J_0) \ne \emptyset \}.
$$
Then
$
G_0(M,\omega)
$ is discrete by the Gromov compactness.
Let $G_0(M,\omega,\frak b)$ be the discrete monoid generated by
$G_0(M,\omega) \cup G_0(\frak b)$.
\index{$G_0(M,\omega,\frak b)$}
We define
\begin{equation}
\Lambda _0(M,\omega,\frak b)
=
\left.\left\{ \sum a_i T^{\lambda_i}
\in \Lambda _0 ~\right\vert~ \lambda_i \in G_0(M,\omega,\frak b)\right\}.
\end{equation}
The following is easy to check.

\begin{lem} The bilinear map $\cup^{\frak b}$ induces a ring structure on
$H^*(M;\Lambda _0(M,\omega,\frak b))$.
\end{lem}

We have thus obtained the associated quantum cohomology ring
\begin{equation}\label{qcsmallcoef}
QH^*_{\frak b}(M;\Lambda_0 (M,\omega,\frak b))
= (H^*(M;\Lambda_0 (M,\omega,\frak b)),\cup^{\frak b}).
\end{equation}

\begin{rem}\label{rem52}
\begin{enumerate}
\item Via the identification $q  = T^{-1}$, we will use
$$
\Lambda^\downarrow_0(M,\omega,\frak b)
=
\left\{ \sum_i a_i q^{- \lambda_i} \in \Lambda
\mid \lambda_i \in G_0(M,\omega;\frak b), \, \lambda_i \to \infty \right\}
$$
in place of $\Lambda_0(M,\omega,\frak b)$ in (\ref{qcsmallcoef})
in the construction of Floer \emph{homology} as before.
\item
Entov-Polterovich \cite{EP:morphism,EP:states,EP:rigid}
use quantum {\it homology}, where the degree is shifted by $2n$
from the usual degree.
The isomorphism in Theorem \ref{Piuiso} then
preserves the degree when we use the Conley-Zehnder
index as the degree of Floer homology.
\par
Here we use the usual degree convention of quantum {\it cohomology}
but shift the degree of Floer homology by $n$ from
the Conley-Zehnder index.
\par
In this convention, the (quantum) cup product
is ($\Z_2$)-degree preserving.
In `quantum homology', the product of
degree $d_1$ and $d_2$ classes has
degree $d_1+d_2 -2n$.
We prefer to choose our convention so that
the quantum product becomes degree-preserving as in the usual grading
convention in the quantum cohomology ring.
\end{enumerate}
\end{rem}
\index{big quantum cohomology|)}

\section{Hamiltonian Floer homology with bulk deformations}
\label{sec:deform-bdy} \index{bulk deformation!Hamiltonian Floer
homology|(}

In this section we modify the construction of Section \ref{subsec:boundary}
and incorporate bulk deformations.
\par
Let $[\gamma,w],\, [\gamma',w'] \in \Crit(\CA_H)$.
Below we will need to consider the moduli space of marked
Floer trajectories
$
\CM_{\ell}(H,J;[\gamma,w],[\gamma',w'])
$
for each $\ell = 0,1,\dots$.
The moduli space $\CM_{0}(H,J;[\gamma,w],[\gamma',w'])$
coincides with $\CM(H,J;[\gamma,w],[\gamma',w'])$
which is defined in Definition \ref{stripmk0int0} and Proposition \ref{connkura}.
\begin{defn}\index{$\mathcal M_{\ell}(H,J;[\gamma,w], [\gamma',w'])$}
We denote by $\widehat{\overset{\circ}{{\CM}}}_{\ell}(H,J;[\gamma,w], [\gamma',w'])$ the set of all
\linebreak
$(u;z^+_1,\dots,z^+_{\ell})$, where $u$ is a map
 $u: \R \times S^1 \to M$ which satisfies Conditions
(1) - (4) of Definition \ref{stripmk0int0}
and $z^+_i$ $(i=1,\dots,\ell)$ are mutually distinct points of
$\R \times S^1$.
It carries an $\R$-action by translations in $\tau$-direction.
We denote its quotient space by $\overset{\circ}{{\CM}}_{\ell}(H,J;[\gamma,w], [\gamma',w'])$.
We define the evaluation map
$$
{\rm ev} = ({\rm ev}_1\dots,{\rm ev}_{\ell}): \overset{\circ}\CM_\ell(H,J;[\gamma,w],[\gamma',w']) \to M^{\ell}
$$
by
$$
{\rm ev}(u;z^+_1,\dots,z^+_{\ell})
= (u(z^+_1),\dots,u(z^+_{\ell})).
$$
\end{defn}\index{shuffle}
We use the following notation in the next proposition.
Denote the set of shuffles of $\ell$ elements by \index{$\text{\rm Shuff}(\ell)$}
\begin{equation}\label{shuff}
\text{\rm Shuff}(\ell) = \{ (\mathbb L_1,\mathbb L_2) \mid \mathbb
L_1 \cup \mathbb L_2 = \{1,\ldots,\ell\}, \,\,\mathbb L_1 \cap
\mathbb L_2 = \emptyset \}.
\end{equation}
For $(\mathbb L_1,\mathbb L_2) \in \text{\rm Shuff}(\ell)$
let
$\#\mathbb L_i$ be the order of this subset.
Then $\#\mathbb L_1 + \#\mathbb L_2 = \ell$.
\par
For any $\alpha \in \pi_2(M)$ there exists a canonical homeomorphism
\begin{equation}\label{homeo62222}
\overset{\circ}{{\CM}}_{\ell}(H,J;[\gamma,w], [\gamma',w'])
\cong
\overset{\circ}{{\CM}}_{\ell}(H,J;[\gamma,\alpha\#w], [\gamma',\alpha\#w']).
\end{equation}
\begin{prop}\label{connBULKkura}
\begin{enumerate}
\item The moduli space
$\overset{\circ}{\CM}_{\ell}(H,J;[\gamma,w], [\gamma',w'])$ has a compactification
${{\CM}}_{\ell}(H,J;[\gamma,w], [\gamma',w'])$ that is Hausdorff.
\item
The space ${{\CM}}_{\ell}(H,J;[\gamma,w], [\gamma',w'])$ has an orientable Kuranishi structure with corners.
\item
The normalized boundary of ${{\CM}}_{\ell}(H,J;[\gamma,w],[\gamma',w'])$ is described by
\begin{equation}\label{62formu}
\aligned
&\partial
 {\CM}_{\ell}(H,J;[\gamma,w], [\gamma',w']) =\\
& \bigcup {\CM}_{\#\mathbb L_1}(H,J;[\gamma,w],[\gamma'',w''])
\times {\CM}_{\#\mathbb L_2}(H,J;[\gamma'',w''],[\gamma',w']),
\endaligned\end{equation}
where the union is taken over all $[\gamma'',w''] \in \text{\rm Crit}(\mathcal A_H)$,
and $(\mathbb L_1,\mathbb L_2) \in \text{\rm Shuff}(\ell)$.
\item
Let $\mu_H : \mbox{\rm Crit}(\CA_H) \to  \Z$
be the Conley-Zehnder index. Then the (virtual) dimension satisfies
the following equality $(\ref{dimensionboundaryb})$.
\begin{equation}\label{dimensionboundaryb}
\dim {\CM}_{\ell}(H,J;[\gamma,w], [\gamma',w']) = \mu_H([\gamma',w']) -  \mu_H([\gamma,w]) - 1+2\ell.
\end{equation}
\item
We can define orientations of ${\CM}_{\ell}(H,J;[\gamma,w], [\gamma',w'])$ so that
$(3)$ above is compatible with this orientation.
\item The evaluation map ${\rm ev}$ extends to a map
$\CM_\ell(H,J;[\gamma,w],[\gamma',w']) \to M^{\ell}
$, which we denote also by ${\rm ev}$. It is a strongly continuous \index{strongly continuous} and weakly submersive\index{weakly submersive} map
in the sense of \cite[Definition 4.6]{fooo:tech2}, \cite[Definition A.1.13]{fooo:book2}
(see also Definition $\ref{mapkura}$) which is compatible with $(3)$. Namely
if we denote
$$
\L_1 = \{i_1,\dots,i_{\#\L_1}\},\,
\L_2 = \{j_1,\dots,j_{\#\L_2}\}
$$
with
$i_1 <\dots < i_{\#\L_1}$, $j_1 < \dots < j_{\#\L_2}$,
then ${\rm ev}_k$ of the first factor (resp. the second factor) of the right hand side of
$(\ref{62formu})$ coincides with ${\rm ev}_{i_k}$ (resp. ${\rm ev}_{j_k}$)
of the left hand side of $(\ref{62formu})$.
\item
The homeomorphism $(\ref{homeo62222})$ extends to the compactifications
and their Kuranishi structures are identified by the homeomorphism.
\end{enumerate}
\end{prop}
The proof of Proposition \ref{connBULKkura} is the same as that of Propositions \ref{connkura}
and \ref{isomPiukura}.
The detail of the proof of \cite[Theorem 29.4]{fooo:techI}
given in \cite[Section 31]{fooo:techI}   proves Proposition \ref{connBULKkura} also.
An element of (\ref{62formu})
 is drawn in Figure \ref{Figure3} below.
\begin{figure}[h]
\centering
\includegraphics[scale=0.5]
{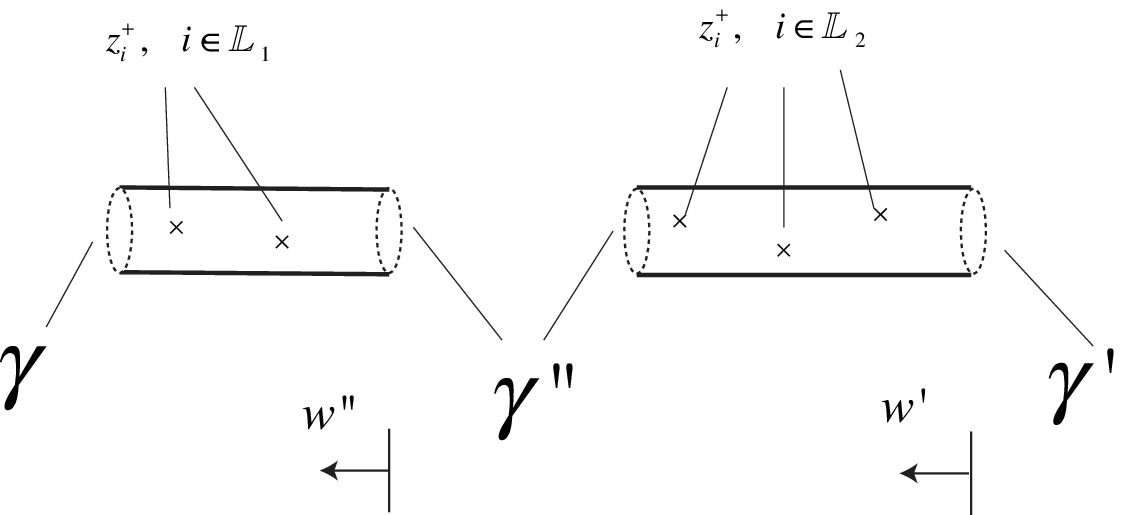}
\caption{An element of (\ref{62formu})}
\label{Figure3}
\end{figure}
We are ready to define the deformed boundary map $\del_{(H,J)}^\frak b$.
We start with defining the following operator:
\begin{defn}
Let $[ \gamma,w],\, [\gamma',w'] \in \text{\rm Crit}(\mathcal A_H)$ and $h_i$ ($i = 1,\ldots,\ell$) be
differential forms on $M$. We define
\index{$\frak n_{(H,J);\ell}([\gamma,w],[\gamma',w']) (h_1,\ldots,h_{\ell})$}
$\frak n_{(H,J);\ell}([\gamma,w],[\gamma',w']) (h_1,\ldots,h_{\ell}) \in \C$ by
\be\label{eq:nww'C}
\frak n_{(H,J);\ell}([\gamma,w],[\gamma',w']) (h_1,\ldots,h_{\ell}) =
\int_{\CM_{\ell}(H,J;[\gamma,w],[\gamma',w'])}
{\rm ev}_1^*h_1 \wedge\dots \wedge {\rm ev}_\ell^*h_\ell.
\ee
By definition (\ref{eq:nww'C}) is zero unless
$$
\sum_{i=1}^{\ell} \text{\rm deg} \, h_i
\ne \dim \CM_{\ell}(H,J;[\gamma,w],[\gamma',w']),
$$
where the right hand side is as in (\ref{dimensionboundaryb}).
\begin{rem}\label{compatimulti}
In order to define the integration in (\ref{eq:nww'C})
using Definition \ref{def320222}, we need to take a CF-perturbation
of $\CM_\ell(H,J;[\gamma,w], [\gamma',w'])$ that is transversal to $0$.
In our situation the integration (\ref{eq:nww'C}) depends on the choice of a CF-perturbation
as well as the small parameter $\epsilon>0$ since $\CM_\ell(H,J;[\gamma,w], [\gamma',w'])$ has
codimension one boundary.
We take a system of CF-perturbations of $\CM_\ell(H,J;[\gamma,w], [\gamma',w'])$ so that
the system is compatible
with the decomposition (\ref{62formu}) of the normalized boundary
$\partial\CM_\ell(H,J;[\gamma,w], [\gamma',w'])$.
We take our CF-perturbations so that
it is consistent with the isomorphism in
Proposition \ref{connBULKkura} (7).
\footnote{From now on we do not repeat this kind of remarks in the similar
situations.}
We also fix a sufficiently small
constant $\epsilon > 0$.
\end{rem}
\par
We linearly extend the definition of $\frak n_{(H,J);\ell}([\gamma,w],
[\gamma',w'])$ to
a $\Lambda $-multilinear map and compose it with a
canonical isomorphism $\Lambda \cong \Lambda^{\downarrow}$.
We denote the composition $(\Omega(M) \widehat\otimes \Lambda)^{\ell} \to \Lambda^{\downarrow}$
by the same symbol $\frak n_{(H,J);\ell}([\gamma,w],[\gamma',w'])$.
\par
Let $\frak b \in H^{{\rm even}}(M;\Lambda _0)$ and split
$\frak b = \frak b_0 + \frak b_2 + \frak b_+$ as in
(\ref{decompb}). We take closed forms that represent the cohomology classes
$\frak b_0$, $\frak b_2$, $\frak b_+$ respectively
and denote them by the same symbols.
We then define $\frak n_{(H,J);\ell}^\frak b([\gamma,w],[\gamma',w']) \in \Lambda^{\downarrow}$
for $[\gamma,w], [\gamma',w'] \in \text{\rm Crit}(\mathcal A_H)$ by
\index{$\frak n_{(H,J)}^\frak b([\gamma,w],[\gamma',w'])$}
\be\label{bdrycoefb}
\aligned
&\frak n_{(H,J)}^\frak b([\gamma,w],[\gamma',w'])
\\
&=
\sum_{\ell=0}^{\infty}\frac{\exp(\int (w')^*\frak b_2 - \int w^*\frak b_2)}{\ell!}
\frak n_{(H,J);\ell}([\gamma,w],[\gamma',w'])(\underbrace{\frak b_+, \dots,\frak b_+}_{\ell}).
\endaligned
\ee
Let $\llb \gamma,w \rrb, \llb \gamma',w' \rrb \in \widehat{\Per}(H)$.
Using (\ref{bdrycoefb}) we define $\frak n_{(H,J);\ell}^\frak b(\llb \gamma,w \rrb,\llb \gamma',w' \rrb) \in \Lambda^{\downarrow}$ as follows.
Suppose $\llb \gamma,w \rrb$ is represented by $[\gamma,w] \in \text{\rm Crit}(\mathcal A_H)$,
i.e., $\pi([\gamma,w]) = \llb \gamma,w \rrb$. We then put:
\index{$\frak n_{(H,J)}^\frak b(\llb \gamma,w \rrb,\llb \gamma',w' \rrb)$}
\begin{equation}\label{nf66666}
\frak n_{(H,J)}^\frak b(\llb \gamma,w \rrb,\llb \gamma',w' \rrb)
=
\sum_{[\gamma',w'] \in \text{\rm Crit}(\mathcal A_H)
\atop \pi([\gamma',w']) = \llb \gamma',w' \rrb}
\frak n_{(H,J)}^\frak b([\gamma,w],[\gamma',w']).
\end{equation}
Using Propositoin \ref{connBULKkura} (7) we can show that the right hand side of (\ref{nf66666})
is independent of $[\gamma,w]$ but depends only on $\llb \gamma,w \rrb$
in the same way as Remark \ref{remark38new}.
\end{defn}
\begin{lem}\label{adiccomv1}
When $\deg {\mathfrak b}_+ >2$, the right hand sides of
\eqref{bdrycoefb} and $(\ref{nf66666})$ are a finite sum. On the other hand,
when $\deg {\mathfrak b}_+ =2$, the right hand side of
\eqref{bdrycoefb} and $(\ref{nf66666})$ converge in the $q$-adic topology.
\end{lem}
\begin{proof}
Consider the case that $\deg {\mathfrak b}_+ >2$.
Suppose ${\CM}_{\ell}(H,J;[\gamma,w],[\gamma',w'])\neq \emptyset$
and so ${\CM}(H,J;[\gamma,w],[\gamma',w']) \neq \emptyset$.
The summand corresponding to $\ell$ in the right hand side of
$(\ref{bdrycoefb})$
vanishes unless
$\ell \deg {\mathfrak b}_+ =  \dim {\CM}_{\ell}(H,J;[\gamma,w], [\gamma',w']
) = \mu_H([\gamma^{\prime},w^{\prime}]) - \mu_H([\gamma,w]) + 2\ell -1$, i.e.,
$$
\ell(\deg {\mathfrak b}_+ -2) = \mu_H([\gamma^{\prime},w^{\prime}]) -
\mu_H([\gamma,w]) -1.
$$
Since the right hand side depends only on
$[\gamma, w], [\gamma^{\prime}, w^{\prime}]$, it implies
boundedness of the number of possible choices of $\ell$. This finishes the
proof of the first statement.

For the second statement, we note that ${\mathfrak v}_q ({\mathfrak b}_+)
= - {\mathfrak v}_T({\mathfrak b}_+)< 0$.
Hence the value of ${\mathfrak v}_q$ of the $\ell$-th term in the right
hand of
\eqref{bdrycoefb} diverges to $-\infty$ as $\ell$ tends to $\infty$.
Hence the proof.
\end{proof}

\begin{defn}
We define a deformed Floer boundary map
$$
\del^{\frak b}_{(H,J)}: CF_*(M,H;\Lambda^{\downarrow}) \to CF_*(M,H;\Lambda^{\downarrow})
$$
by
\index{$\del^{\frak b}_{(H,J)}$}
\begin{equation}\label{defboundary}
\del^{\frak b}_{(H,J)}(\llb \gamma,w \rrb) =
\sum_{\llb \gamma',w' \rrb}  \frak n_{(H,J)}^{\frak b}(\llb \gamma,w \rrb,\llb \gamma',w' \rrb)\llb \gamma',w' \rrb.
\end{equation}
\end{defn}
We point out that the sum in (\ref{defboundary}) may not be a finite sum.

\begin{lem}\label{adiccomv2}
The right hand side of $(\ref{defboundary})$ converges in $CF(M,H;\Lambda^{\downarrow})$
and $\del^{\frak b}_{(J,H)}$ is continuous in $q$-adic topology.
\end{lem}
\begin{proof}
Let $E$ be any real number and $[\gamma',w'] \in  \Crit \CA_H$.
By Gromov-Floer compactness, the number of $[\gamma',w']$ such that
${\CM}_{\ell}(H,J;[\gamma,w],[\gamma',w'])$ is nonempty and
$$
\mathcal A_H([\gamma,w]) - \mathcal A_H([\gamma',w']) < E
$$
is finite. The lemma now follows from the definition of
convergence in $CF(M,H;\Lambda^{\downarrow})$.
\end{proof}
Combining Proposition \ref{connBULKkura} (3) with Stokes' theorem
\cite[Lemma C.9]{fooo:toric1}  and
composition formula \cite[Lemma C.10]{fooo:toric1}, we can
check
$$
\del_{(H,J)}^{\frak b} \circ \del_{(H,J)}^{\frak b} = 0.
$$
\begin{defn}
We define {\em Floer homology of periodic Hamiltonian system with bulk deformation} by
$$
HF_*^\frak b(M,H,J;\Lambda^{\downarrow})
= \frac{\text{\rm Ker}\,\,\del^{\frak b}_{(H,J)}}
{\text{\rm Im}\,\,\del^{\frak b}_{(H,J)}}.
$$
\index{$HF_*^\frak b(M,H,J;\Lambda^{\downarrow})$}
\end{defn}

Now we take two parameter family $\{(H_\chi,J_\chi)\}_{\tau \in \R}$
as in (\ref{eq:paraHJ}) in the proof of Theorem \ref{Piuiso}.

\begin{thm}\label{Pbulkiso}
There exists a $\Lambda^{\downarrow}$-module isomorphism
$$
\mathcal P_{(H_\chi,J_\chi), \ast}^{\frak b} : H_*(M;\Lambda^{\downarrow})
\cong HF_*^\frak b(M,H,J;\Lambda^{\downarrow})
$$
for all $\frak b$. We call it the {\it Piunikhin map with
bulk}.\index{Piunikhin isomorphism!with bulk}
\index{$\mathcal P_{(H_\chi,J_\chi), \ast}^{\frak b} $}

\end{thm}
\begin{proof}
The proof, which we discuss below, is similar to the proof of Theorem \ref{Piuiso}.
We recall that we identify the de Rham complex with a chain complex
{
$$
(\Omega_*(M) \widehat \otimes \Lambda^\downarrow, \partial), \quad \Omega_*(M): = \Omega^{2n-*}(M),
\quad \partial = (-1)^{\text{deg} +1} d.
$$
We denote by the map $\flat: (\Omega^*(M), \Lambda) \to (\Omega_*(M), \Lambda^\downarrow)$
\index{$\flat$}
this isomorphism between the two.}
In this section we only give the definition of the map
$\mathcal P_{(H_\chi,J_\chi), \ast}^{\frak b}$.
In Section \ref{sec:appendix1} we prove that it is indeed an isomorphism.
\begin{defn}\label{buldPmoduli}
We denote by $\overset{\circ}{\CM}_{\ell}(H_\chi,J_\chi;*,[\gamma,w])$ the set of all
\index{$\mathcal M_{\ell}(H_\chi,J_\chi;*,[\gamma,w])$}
$(u;z^+_1,\dots,z^+_{\ell})$ of maps
 $u: \R \times S^1 \to M$ and
$z^+_i$, $i=1,\dots,\ell$ such that
$u$ satisfies (1)-(4) of Definition \ref{stripmk0int1}
and $z^+_i \in \R \times S^1$ are mutually distinct.
 \end{defn}
The assignment $(u;z^+_1,\dots,z^+_{\ell}) \mapsto (u(z^+_1),\dots,u(z^+_{\ell}))$
 defines an evaluation map
$$
{\rm ev} = ({\rm ev}_1\dots,{\rm ev}_{\ell}) = \overset{\circ}{\CM}_{\ell}(H_\chi,J_\chi;*,[\gamma,w]) \to M^{\ell}.
$$
\begin{prop}\label{piuBULKkura}
\begin{enumerate}
\item
$\overset{\circ}{\CM}_{\ell}(H_\chi,J_\chi;*,[\gamma,w])$ has a compactification,
denoted by
${{\CM}}_{\ell}(H_\chi,J_\chi;*,[\gamma,w])$, that is Hausdorff.
\item
The space ${{\CM}}_{\ell}(H_\chi,J_\chi;*,[\gamma,w])$ has an orientable Kuranishi structure with corners.
\item
The normalized boundary of ${{\CM}}_{\ell}(H_\chi,J_\chi;*,[\gamma,w])$ is described by
\begin{equation}\label{bdryPIunikinmodu}
\aligned
&\partial {\CM}_{\ell}(H_\chi,J_\chi;*,[\gamma,w]) \\
&= \bigcup
{{\CM}}_{\#\mathbb L_1}(H_\chi,J_\chi;*,[\gamma',w'])
\times {\CM}_{\#\mathbb L_2}(H,J;[\gamma',w'],[\gamma,w])
\endaligned\end{equation}
where the union is taken over all $[\gamma',w'] \in \text{\rm Crit}(\mathcal A_H)$,
 and $(\mathbb L_1,\mathbb L_2) \in \text{\rm Shuff}(\ell)$.
\item
Let $\mu_H : \mbox{\rm Crit}(\CA_H) \to  \Z$
be the Conley-Zehnder index. Then the (virtual) dimension satisfies
the following equality:
\begin{equation}\label{dimensionboundaryb2}
\dim {\CM}_{\ell}(H_\chi,J_\chi;*,[\gamma,w]) = \mu_H([\gamma,w]) +n + 2\ell.
\end{equation}
\item
We can define orientations of ${\CM}_{\ell}(H_\chi,J_\chi;*,[\gamma,w])$ so that
$(3)$ above is compatible with this orientation.
\item The map
${\rm ev}$ extends to a strongly smooth map
$\CM_\ell(H_\chi,J_\chi;*,[\gamma,w]) \to M^{\ell}
$, which we denote also by ${\rm ev}$.
It is compatible with $(3)$ in the same sense as
Proposition $\ref{connBULKkura}$ $(6)$.
\item
The map $\text{\rm ev}_{-\infty}$ which sends $(u;z^+_1,\dots,z^+_{\ell})$ to
$
\lim_{\tau\to -\infty}u(\tau, t)
$
extends to a strongly continuous smooth map $\CM_\ell(H_\chi,J_\chi;*,[\gamma,w]) \to M$,
which we denote also by ${\rm ev}_{-\infty}$. It is compatible with $(3)$.
\end{enumerate}
\end{prop}
\begin{proof}
The proof of Proposition \ref{piuBULKkura} is mostly the same as that of Proposition \ref{connkura}.
(See also \cite[Parts 4 and 5]{fooo:techI}.)
We only need to see that in (\ref{bdryPIunikinmodu})
the boundary component such as
\begin{equation}\label{67doesnotoccur}
\mathcal M_{\#\L_1}(0,J_0;*,*;C)
\times
\mathcal M_{\#\L_2}(H_{\chi},J_{\chi};*,[\gamma,(-C)\#w])
\end{equation}
does not appear.
\index{$\mathcal M_{\#\L_1}(0,J_0;*,*;C)$}
(Here the first factor of (\ref{67doesnotoccur}) is a
compactified moduli space of the $J_0$-holomorphic maps
$\R \times S^1\to M$ with finite energy and of homotopy class $C \in \pi_2(M)$.
The way of our usage of $*$ for the arguments used in the notation of the moduli spaces
is the same as in Remark \ref{rem:*}. In particular the second $*$ of the first factor of
\eqref{67doesnotoccur} highlights the fact that the limit at $+\infty$ of an element in the moduli
space converges to an unspecified point of $M$.)
\par
In fact, the moduli space  $\mathcal M_{\#\L_1}(0,J_0;*,*;C)$
has an extra $S^1$ symmetry by the $S^1$ action of the domain
$\R \times S^1$. (See Remark \ref{JS1inv}.)
So after taking a quotient by this $S^1$ action, this component is of codimension 2.
(See the proof of Lemma \ref{MBchainhomotopyprop}.)
\end{proof}
Let $[\gamma,w] \in \text{\rm Crit}(\mathcal A_H)$ and $h_i$ ($i = 1,\ldots,\ell$),
$h$ be
differential forms on $M$.
We take a system of CF-perturbations on $\CM_\ell(H_\chi,J_\chi;*,[\gamma,w])$ such that
it is compatible with (3). We use the system to define (\ref{ndRdef69}) below.
See Remark \ref{compatimulti}.
We define
$\frak n_{(H_\chi,J_\chi)}(h;[\gamma,w]) (h_1,\ldots,h_{\ell}) \in \C$ by
\begin{equation}\label{ndRdef69}
\aligned
&\frak n_{(H_\chi,J_\chi)} (h;[\gamma,w])(h_1,\ldots,h_{\ell})\\
&= \int_{{\CM}_{\ell}(H_\chi,J_\chi;*,[\gamma,w])}
{\rm ev}_{-\infty}^*h \wedge {\rm ev}_1^*h_1 \wedge\dots \wedge {\rm ev}_\ell^*h_\ell.
\endaligned
\end{equation}
We note that (\ref{ndRdef69}) is zero by definition unless the dimensional restriction
$$
\text{\rm deg} \, h + \sum_{i=1}^{\ell} \text{\rm deg} \, h_i
= \dim {\CM}_{\ell}(H_\chi,J_\chi;*,[\gamma,w]),
$$
holds.
We extend $\frak n_{(H_\chi,J_\chi)}(h;[\gamma,w])$ to a $\Lambda$-multilinear map
$(\Omega(M) \widehat\otimes \Lambda) \to \Lambda$
and identify $\Lambda = \Lambda^{\downarrow}$ by $q = T^{-1}$.
We denote it by the same symbol.
\par
Let $\frak b \in H^{{\rm even}}(M;\Lambda _0)$. We decompose
$\frak b = \frak b_0 + \frak b_2 + \frak b_+$ as in
(\ref{decompb})
and  regard  $ \frak b_0$, $\frak b_2$, $\frak b_+$ as de Rham (co)homology classes
by representing them by the closed differential forms.
\emph{
We define an element $\frak n^\frak b_{(H_\chi,J_\chi)}(h;[\gamma,w]) \in \Lambda^{\downarrow}$}
by :

\begin{equation}\label{bdrycoefb2}
\frak n^\frak b_{(H_\chi,J_\chi)}(h;[\gamma,w]) :=
\sum_{\ell=0}^{\infty}\frac{\exp({\int w^*\frak b_2})}{\ell!}
\frak n_{(H_\chi,J_\chi),\ell}([\gamma,w])(h;\underbrace{\frak b_+,
\dots,\frak b_+}_{\ell})
\end{equation}
for each given $[\gamma,w] \in \text{\rm Crit}(\mathcal A_H)$ and a differential form $h$ on $M$.
We can prove that the sum in (\ref{bdrycoefb2}) converges in $q$-adic topology,
in the same way as in Lemma \ref{adiccomv1}.
We now define
\begin{equation}\label{Piunikinb}
\mathcal P^{\frak b}_{(H_\chi,J_\chi)}(h)
:= \sum_{[\gamma,w] \in \text{\rm Crit}(\mathcal A_H)} \frak n^\frak b_{(H_\chi,J_\chi)}(h;[\gamma,w])\,\,
\llb \gamma,w \rrb.
\end{equation}
We can prove that the right hand side is an element of
$CF(H,J;\Lambda^{\downarrow})$ in the same way as in Lemma \ref{adiccomv2}.
Thus we have defined
$$
\mathcal P^{\frak b}_{(H_\chi,J_\chi)}
: \Omega_*(M) \widehat{\otimes} \Lambda^{\downarrow}  \to CF_*(M,H;\Lambda^{\downarrow}).
$$
Then the identity
\begin{equation}\label{boundaryprop}
\mathcal P^{\frak b}_{(H_\chi,J_\chi)}  \circ \partial
= \partial_{(H,J)}^{\frak b} \circ \mathcal P^{\frak b}_{(H_\chi,J_\chi)}
\end{equation}
is a consequence of (\ref{bdryPIunikinmodu}),
Stokes' theorem (Theorem \ref{them48}) and
composition formula (Theorem \ref{compform}).
We can easily prove that $\mathcal P^{\frak b}_{(H_\chi,J_\chi)}$
are chain homotopic to one another when $\chi$ is varied in $\CK$.
We denote by
\begin{equation}\label{Piuniquinhomoeq}\CP_{(H_\chi,J_\chi),\ast}^{\frak b} : H_*(M;\Lambda^{\downarrow}) \to HF_*^\frak b(M,H,J;\Lambda^{\downarrow})
\end{equation}
\index{$\mathcal P_{(H_\chi,J_\chi),\ast}^{\frak b}$}
the map induced on homology. We will prove in Section \ref{sec:appendix1} that it is an isomorphism.
\end{proof}
\index{bulk deformation!Hamiltonian Floer homology|)}

\section{Spectral invariants with bulk deformation}
\label{sec:specbulk}

We next modify the argument given in Section \ref{sec:spectral} and define
spectral invariants with bulk. Let $\frak b \in
H^{{\rm even}}(M;\Lambda_0)$. We consider \emph{discrete} submonoids
$G_0(M,\omega)$ and $G_0(M,\omega,\frak b)$ of $\R$ in Definition
\ref{def:deformcup}.

\begin{defn} \index{$G(M,\omega)$}\index{$G(M,\omega,\frak
b)$}\index{$G_0(M,\omega)$}\index{$G_0(M,\omega,\frak b)$} We denote
by $G(M,\omega)$ and $G(M,\omega,\frak b)$ the sub{\it group} of
$(\R,+)$ generated by the monoids $G_0(M,\omega)$ and
$G_0(M,\omega,\frak b)$, respectively.
\end{defn}
We note that $G(M,\omega)$ and $G(M,\omega,\frak b)$ are not
necessarily discrete. We also remark that $G(M,\omega,\frak b)$ may
not even be finitely generated.
\par
Let $H$ be a time-dependent Hamiltonian on $M$.
We defined ${\rm Spec}(H)$ in Definition \ref{criticalvalue}.

\begin{defn}\index{$\text{\rm Spec}(H;\frak b)$}
We define
$$
\aligned
{\rm Spec}(H;\frak b)
&= {\rm Spec}(H) + G(M,\omega,\frak b)\\
&= \{ \lambda + g \mid
\lambda \in {\rm Spec}(H),
g \in G(M,\omega,\frak b) \}.
\endaligned$$
\end{defn}
For a monoid $G \subset \R$, the ring $\Lambda^{\downarrow}(G)$ was defined in Definition \ref{Lambda(G)}.
\begin{defn}
Suppose $H$ is nondegenerate. We put
\index{$\Lambda^{\downarrow}_{\frak b} (M)$}
$$
\Lambda^{\downarrow}_{\frak b} (M)
= \Lambda^{\downarrow} (G(M,\omega,\frak b))
$$
and \index{$CF(M,H;\Lambda^{\downarrow}_{\frak b} (M))$}
$$
CF(M,H;\Lambda^{\downarrow}_{\frak b} (M))
= CF(M,H) \otimes_{\Lambda^{\downarrow} (M)} \Lambda^{\downarrow}_{\frak b} (M).
$$
Here $CF(M,H)$ is defined in Definition \ref{Lambda(G)}.
\end{defn}
By definition, an element $x \in CF(M,H;\Lambda^{\downarrow}_{\frak b} (M))$ can be written as
$$
x = \sum_{\gamma \in \text{\rm Per}(H)} x_\gamma \llb \gamma, w_\gamma \rrb, \quad x_\gamma \in \Lambda^{\downarrow}_{\frak b} (M)
$$
similarly as in \eqref{xexpand}. We define the {\it level function}\index{level function}
$\lambda_H: CF(M,H;\Lambda^{\downarrow}_{\frak b} (M)) \to \R$\index{$\lambda_H$}
by
$$
\lambda_H(x) =
\max\{\frak v_q(x_\lambda) + \CA_H(\llb \gamma,w_\gamma \rrb) \mid x_\gamma \neq 0 \, \text{in the above sum}\}.
$$
\begin{lem}\label{frabCFlem}
Suppose $H$ is nondegenerate.
\begin{enumerate}
\item
The set $\{\llb \gamma,w_\gamma \rrb \mid \gamma \in \text{\rm Per}(H)\}$ forms a basis of
the vector space  $CF(M,H;\Lambda^{\downarrow}_{\frak b} (M))$ over the field $\Lambda^{\downarrow}_{\frak b} (M)$.
\item
If $\frak x \in CF(M,H;\frak b) \setminus \{0\}$ then
$
\lambda_H(\frak x) \in \text{\rm Spec}(H;\frak b).
$
\end{enumerate}
\end{lem}
\begin{proof} Statement (1) follows from the fact that
$\mathcal A_H([\gamma,w]) - \mathcal A_H([\gamma,w'])
\in G(M,\omega,\frak b)$ for $\gamma \in \text{\rm Per}(H)$,
$[\gamma,w], [\gamma,w'] \in \text{\rm Crit}(\mathcal A_H)$.
Then statement (2) follows from statement (1).
\end{proof}
We denote by ${\rm Ham}_{\text{\rm nd}}(M,\omega)$ the set of nondegenerate
Hamiltonian diffeomorphisms. \index{$\text{\rm Ham}_{{\rm nd}}(M,\omega)$}
By now, it is well-established that for any $H \in {\rm Ham}_{\text{nd}}(M,\omega)$ the map $\partial^{\frak b}_{(H,J)}:
CF(M,H;\Lambda^{\downarrow}) \to CF(M,H;\Lambda^{\downarrow})$ is defined and preserves
the subspace
$CF(M,H;\Lambda^{\downarrow}_{\frak b} (M))$ of $CF(M,H;\Lambda^{\downarrow})$. Moreover the filtration of
$CF(M,H;\Lambda^{\downarrow}(M))$ induces one on
$CF(M,H;\Lambda^{\downarrow}_{\frak b} (M))$ given by
$$
F^{\lambda}CF(M,H;\Lambda^{\downarrow}_{\frak b} (M))= F^{\lambda}CF(M,H;\Lambda^{\downarrow}(M))
\cap CF(M,H;\Lambda^{\downarrow}_{\frak b} (M)).
$$
We denote the homology of  $(CF(M,H;\Lambda^{\downarrow}_{\frak b} (M)),\partial^{\frak b}_{(H,J)})$
by $HF^{\frak b}(M,H,J;\Lambda^{\downarrow}_{\frak b} (M))$. Then
Lemma \ref{frabCFlem} implies
\begin{equation}\label{stgptenso}
HF^{\frak b}_*(M,H,J;\Lambda^{\downarrow})
\cong
HF^{\frak b}_*(M,H,J;\Lambda^{\downarrow}_{\frak b} (M))
\otimes_{\Lambda^{\downarrow}_{\frak b} (M)}\Lambda^{\downarrow}.
\end{equation}
Therefore Theorem \ref{Pbulkiso}  implies:
\begin{lem} The map $\CP_{(H_{\chi},J_{\chi}),\ast}^{\frak b}$ in \eqref{Piuniquinhomoeq}
induces an isomorphism
$$
H(M;\Lambda^{\downarrow}_{\frak b} (M)) \cong HF^{\frak b}(M,H,J;\Lambda^{\downarrow}_{\frak b} (M)).
$$
\end{lem}
\begin{defn}
\begin{enumerate}
\item
Let $\frak x \in HF^{\frak b}(M,H,J;\Lambda^{\downarrow})$. We define its {\it spectral invariant} $\rho^{\frak b}(\frak x)$\index{spectral invariant}
by
$$
\rho^{\frak b}(\frak x)
= \inf \{\lambda \mid x \in F^{\lambda}CF(M,H,J;\Lambda^{\downarrow}),
\,\, \partial_{H,J}^{\frak b}(x) = 0, \,\, [x] = \frak x  \in HF^{\frak b}(M,H,J;\Lambda^{\downarrow})\}.
$$
\item
If $a \in H^*(M;\Lambda_{\frak b}  (M))$
\footnote{Here $\Lambda_{\frak b}  (M)$ is a subalgebra of $\Lambda$ which is identified to
\index{$\Lambda_{\frak b}  (M)$}
the subalgebra $\Lambda_{\frak b}^{\downarrow}(M)$ of $\Lambda^{\downarrow}$ by the canonical
isomorphism $\Lambda \cong \Lambda^{\downarrow}$.}
and $H$ is a nondegenerate time-dependent Hamiltonian, we define
the {\it spectral invariant with bulk} $\rho^{\frak b}(H;a)$\index{spectral invariant!with bulk}
\index{$\rho^{\frak b}(H;a)$} by
$$
\rho^{\frak b}(H;a) = \rho^{\frak b}(a^{\frak b;\flat}_H), \quad
a^{\frak b;\flat}_H: = \CP_{(H_{\chi},J_{\chi}),\ast}^{\frak b}(a^\flat) \in HF^{\frak b}_*(M,H,J;\Lambda^\downarrow)
$$
\index{$a^{\frak b;\flat}_H$}
where the right hand side is as in $(1)$, and we regard
$$
\mathcal P^{\frak b}_{(H_{\chi},J_{\chi}),\ast}(a^\flat)
\in HF^{\frak b}(M,H,J;\Lambda^{\downarrow}_{\frak b}(M))
\subset
HF(M,H,J;\Lambda^{\downarrow}).
$$
\end{enumerate}
\end{defn}

By the same procedure exercised for the spectral invariant $\rho(H;a)$,
we can prove that $\rho^{\frak b}(\CP_{(H_{\chi},J_{\chi}),\ast}^{\frak b}(a))$ do not depend on
the choices of $J$ and $\chi$ or of other choices involved in the construction of
virtual fundamental cycles, and hence $\rho^{\frak b}(H;a)$ is
well-defined.

\begin{thm}[{\rm Homotopy invariance}]
\label{homotopyinvbulk}\index{spectral invariant!homotopy
invariance}
\begin{enumerate}
\item
The spectral invariant  $\rho^{\frak b}(H;a)$ is independent of the
almost complex structure and other choices involved in the definition.
\item
The spectral invariant $\rho^{\frak b}(H;a)$ depends only on the homology class of $\frak b$ and is independent of the
choices of differential forms which represent it.
\item
Suppose $\phi_H^1 = \phi_{H'}^1$ and the paths
$\phi_H$ and $\phi_{H'}$ are homotopic relative to the ends. Then
$$
\rho^{\frak b}(H;a) = \rho^{\frak b}(H';a).
$$
\end{enumerate}
\end{thm}
Theorem \ref{homotopyinvbulk} (1) is proved in Section \ref{sec:proofcontinuity}.
Theorem \ref{homotopyinvbulk} (3) is proved in Section
\ref{sec:proofhomotopy}.
Theorem  \ref{homotopyinvbulk} (2) is proved in Section \ref{sec:appendix2}.
\par
Theorem \ref{homotopyinvbulk} implies that the function $H \mapsto \rho^{\frak b}(\underline{H};a)$ descends to $\widetilde{\rm Ham}_{\rm nd}(M,\omega)$,
which is the inverse image of ${\rm Ham}_{\rm nd}(M,\omega)$ in
the universal covering
$\widetilde{\rm Ham}(M,\omega) \to {\rm Ham}(M,\omega)$. We denote
by $\rho^{\frak b}(\widetilde \psi_H;a) = \rho^{\frak b}(\underline H;a)$ if
$\widetilde \psi_H = [\phi_H]
\in \widetilde{\rm Ham}_{\rm nd}(M,\omega)$ associated to $H$ as before.
\par
We have thus defined a map
\begin{equation}\label{spnondegb}
\rho^{\frak b}: \widetilde{\rm Ham}_{\rm nd}(M, \omega) \times (QH_{\frak b}^*(M)\setminus\{0\}) \to \R.
\end{equation}
It still satisfies the conclusions of Theorem \ref{axiomsh}.
Namely we have:
\begin{thm}\label{axiomshbulk} Let $(M,\omega)$ be any closed symplectic
manifold.
Then the map $\rho^{\frak b}$ in $(\ref{spnondegb})$ extends to
\begin{equation}\label{spgen}
\rho^{\frak b}: \widetilde{\rm Ham}(M, \omega) \times (QH_{\frak b}^*(M)\setminus\{0\}) \to \R.
\end{equation}
It has the following properties.
Let $\widetilde \psi, \widetilde \phi \in
\widetilde{\rm Ham}(M,\omega)$ and $0 \neq a \in H^*(M;\Lambda_{\frak b}(M))$.
\begin{enumerate}
\item {\rm (Nondegenerate spectrality)} If $\widetilde \psi$ is non-degenerate, then
$\rho(\widetilde \psi;a) \in \mbox{\rm Spec}(H;\frak b)$.
\index{$\rho(\widetilde \psi;a)$}
\item {\rm (Projective invariance)}
$\rho^{\frak b}(\widetilde\phi;\lambda a) = \rho^{\frak b}(\widetilde\phi;a)$ for any
$0 \neq \lambda \in \C$.
\item {\rm (Normalization)}
We have $\rho^{\frak b}(\underline 0;a) = \lambda_q(a)$ where $\underline 0$ is the identity in
$\widetilde{\rm Ham}(M,\omega)$
and $\frak v_q(a)$ is as in $(\ref{defvq})$.
\item {\rm (Symplectic invariance)} $\rho^{\eta^*\frak b}(\eta^{-1} \circ\widetilde\phi\circ \eta;\eta^*a) =
\rho^{\frak b}(\widetilde\phi;a)$ for any symplectic diffeomorphism
$\eta$. In particular, if $\eta \in {\rm Symp}_0(M,\omega)$,
then we have $\rho^{\frak b}(\eta^{-1} \circ\widetilde\phi\circ \eta;a) =
\rho^{\frak b}(\widetilde\phi;a)$.
\item {\rm (Triangle inequality)} $\rho^{\frak b}(\widetilde\phi \circ
\widetilde\psi; a\cup^{\frak b} b) \leq \rho^{\frak b}(\widetilde\phi;a) +
\rho^{\frak b}(\widetilde\psi;b)$,
where $a\cup^{\frak b} b$ is the $\frak{b}$-deformed quantum cup product.
\item {\rm ($C^0$-Hamiltonian continuity)} We have
$$
\vert \rho^{\frak b}(\widetilde\phi;a) - \rho^{\frak b}(\widetilde\psi;a)
\vert
\leq
\max\{
\|\widetilde\phi \circ \widetilde\psi^{-1} \|_+ ,
\|\widetilde\phi \circ \widetilde\psi^{-1} \|_-
\}
$$ where $\| \cdot\|_\pm$ is the positive and negative parts of
Hofer norm on $\widetilde{\rm Ham}(M,\omega)$. In
particular, the function $\rho_a: \widetilde \psi \mapsto
\rho^{\frak b}(\widetilde \psi;a)$ is continuous
with respect to the quotient topology under the equivalence relation $\sim$
on the space of Hamiltonian paths $\{\widetilde\psi_H \mid H \in C^\infty(S^1 \times M,\R) \}$.
\item {\rm (Additive triangle inequality)}
$\rho^{\frak b}(\widetilde\psi;a+b) \le \max\{\rho^{\frak b}(\widetilde\psi;a),\rho^{\frak b}(\widetilde\psi;b)\}$.
\end{enumerate}
\end{thm}
The proofs of Theorems \ref{homotopyinvbulk} and \ref{axiomshbulk} occupy the rest of this part.
Most of the proofs are minor changes of the proofs of
Theorem \ref{axiomsh} in \cite{oh:alan,oh:minimax} and of \cite{usher:specnumber}.

\section{Proof of the spectrality axiom}
\label{sec:sectrality}

In this section we prove Theorem \ref{axiomshbulk} (1).
To include the case when $(M,\omega)$ is not rational
we use some algebraic results exploited by Usher \cite{usher:specnumber}.
We reprove a similar result in Subsection \ref{subsec:Usher}
using the universal Novikov ring.

\subsection{Usher's spectrality lemma}
\label{subsec:Usher} \index{Usher's spectrality lemma|(}

Let $G$ be a subgroup of $\R$.
(We do {\it not} assume that $G$ is discrete.)
We define
$$
\aligned
\Lambda^{\downarrow}(G)
&=
 \left.\left\{ \sum_{i=1}^{\infty} a_i q^{\lambda_i} ~\right\vert~ a_i \in \C, \lambda_i \in \R,
\, \lambda_i \in G,
\lim_{i\to\infty} \lambda_i = -\infty \right\},
\\
\Lambda^{\downarrow}_0 (G)
&=
 \left.\left\{ \sum_{i=1}^{\infty} a_i q^{\lambda_i} ~\right\vert~ a_i \in \C, \lambda_i \in \R_{\le 0},
\, \lambda_i \in G,
\lim_{i\to\infty} \lambda_i = -\infty \right\},
\\
\Lambda^{\downarrow}_+ (G)
&= \left.\left\{ \sum_{i=1}^{\infty} a_i q^{\lambda_i} ~\right\vert~ a_i \in \C, \lambda_i \in \R_{< 0},
\, \lambda_i \in G,
\lim_{i\to\infty} \lambda_i = -\infty \right\}.
\endaligned
$$
(Note the above definition of $\Lambda^{\downarrow}(G)$ coincides with
Definition \ref{Lambda(G)}.)
\index{$\Lambda^{\downarrow}_0 (G)$}\index{$\Lambda^{\downarrow}_+ (G)$}
\par
It is easy to see that $\Lambda^{\downarrow}(G)$ is a field of fraction of $\Lambda_0^{\downarrow}(G)$.
\par
Let $\overline C$ be a finite dimensional $\C$ vector space. We put
$$
C = \overline C \otimes \Lambda^{\downarrow},
\quad
C(G) = \overline C \otimes \Lambda^{\downarrow}(G) \subset C.
$$
\par
Let $e_i$ ($i=1,\dots,N$) be a $\C$-basis of $\overline C$  and
$\lambda^0_i$ for $i=1,\dots,N$ be real numbers.
We define\index{$\lambda_q$}
$
\lambda_q : C \to \R
$
by
$$
\lambda_q\left(\sum_{i=1}^N x_i e_i\right) = \sup \{\frak v_q(x_i)+\lambda_q(e_i) \mid i=1,\dots,N\},
$$
i.e., $\lambda_q(e_i) = \lambda_i^0$ for $i = 1, \ldots, N$.
It defines a norm with respect to which $C$ and $C(G)$ are complete. Then we define a $G$-set
\be\label{eq:G'}\index{$G'$}
G' = \bigcup_{i=1}^N\{ \lambda_i^0 + g \mid g \in G\}.
\ee
It follows from the definition of $\lambda_q(x)$ that
if $x \in C(G)$ then $\lambda_q(x) \in G'$.
We put
$$
F^{\lambda}C = \{x \in C \mid \lambda_q(x) \le \lambda\},
\quad
F^{\lambda}C(G) = F^{\lambda}C \cap C(G).
$$

Suppose that $\overline C$ is $\Z_2$-graded, i.e., $\overline C = \overline C^0 \oplus \overline C^1$ and
each of the element of our basis $e_i$ lies in either $\overline C^0$ or $\overline C^1$.
Let a $\C$-linear map
$$
\partial_g : \overline C^i \to \overline C^{i-1}
$$
be given for each $g \in G$. Assuming that
$\{g \mid \partial_g \ne 0\} \cap \R_{>E}$ is a finite set for
any $E \in \R$,
we put
$$
\partial = \sum_{g\in G} q^g \partial_g : C \to C.
$$
It induces a linear map
$C(G) \to C(G)$, which we also denote by $\partial$.
If $\del$ satisfies $\partial \partial = 0$,
$(C,\partial)$ and $(C(G),\partial)$ define chain complexes.
Denote by $H(C)$, $H(C(G))$ their homologies respectively, and denote by
$H(C(G)) \to H(C)$ the natural homomorphism induced by $\Lambda^{\downarrow}(G) \hookrightarrow
\Lambda^{\downarrow}$.
\begin{defn}
For $\frak x \in H(C)$, we define the level
$$
\rho(\frak x) = \inf \{\lambda_q(x) \mid x \in C(G), \partial x = 0,
[x] = \frak x\}.
$$
\end{defn}
Now the following theorem is proved by Usher \cite{usher:specnumber}.
Here we give its proof for completeness' sake exploiting the algebraic
material developed in \cite[Subsection 6.3]{fooo:book1}.

\begin{prop}\label{spectrarityalg}{\rm (Usher)}
$\rho(\frak x) \in G'$ for any  $\frak x \in {\rm Im}(H(C(G)) \to H(C))$.
\end{prop}
\begin{proof}
We first need to slightly modify the discussion in \cite[Subsection 6.3]{fooo:book1}
since the energy level of the basis $e_i$ is not zero but is $\lambda_i^0$ here.

We say
$$
e_i \sim e_j \quad\mbox{if and only if } \, \lambda_i^0 - \lambda_j^0 \in G.
$$
By re-choosing the basis $\{e_i\}_{1 \leq i \leq N}$ into
the form $\{q^{\mu_i} e_i\}_{1 \leq i \leq N}$ with $\mu_i \in G$ if necessary, we may assume,
without loss of generality, that
$\lambda_i^0 = \lambda_j^0$ if $e_i \sim e_j$.
We assume this in the rest of this subsection.
\par
For each $\lambda \in G'$, define
$$
I(\lambda) = \{ i \mid \lambda - \lambda_i^0 \in G, \, 1 \leq i \leq N\}.
$$
We denote by $\mu(\lambda)$ the difference $\lambda - \lambda_i^0$
for $i\in I(\lambda)$. By the definition of $\sim$ and the hypothesis
we put above, the value $\mu(\lambda)$ is independent of $i$.
We take the direct sum
$$
\overline C(\lambda) = \bigoplus_{i \in I(\lambda)}\C e_i.
$$

Let $x \in C(G)$ be a nonzero element and denote $\lambda  = \lambda_q(x)$.
Then
there exists a unique $\sigma(x) \in \overline C(\lambda)$ such that
$$
\lambda_q(x - q^{\mu(\lambda)}\sigma(x)) < \lambda_q(x).
$$
We call $\sigma(x)$ the {\it symbol} of $x$.
\index{$\Lambda (G)$ vector subspace}
\begin{defn}(Compare \cite[Section 6.3.1]{fooo:book1})\index{symbol}
Let $V \subset C(G)$ be a $\Lambda^{\downarrow}(G)$ vector subspace.
A basis $\{e'_i \mid i=1,\dots,N'\}$ of $V$ is said to be a {\it
standard basis} \index{standard basis} if the symbols  $\{\sigma(e'_i) \mid i=1,\dots,N'\}$
are linearly independent over $\C$.
\end{defn}\index{$\Lambda^{\downarrow}(G)$}
If $\{e'_i \mid i=1,\dots,N'\}$  is a standard basis, then we have
\begin{equation}
\lambda_q\left(\sum_i a_i e'_i\right) = \max \{\frak v_q(a_i) + \lambda_q(e'_i) \mid i=1,\dots,N' \}.
\end{equation}
\begin{lem}\label{lemstandardbasis0}
Any $V \subset C(G)$ has standard basis. Moreover if $V_1 \subset
V_2 \subset C$ are $\Lambda^{\downarrow}(G)$ vector subspaces,
 then any standard
basis of $V_1$ can be extended to one of $V_2$.
\end{lem}
\begin{proof}
The proof is similar to the proof of \cite[Lemma 6.3.2   and Lemma 6.3.2bis]{fooo:book1}.
We give the detail below since we considered $\Lambda$ in place of
$\Lambda^{\downarrow}(G)$ in \cite{fooo:book1}.
\par
Let $x_1,\dots,x_k$ be a standard basis of $V_1$.
We prove the following by induction on $\ell$.
\begin{sublem}\label{mslemmasandard}
For any given $\ell \le \dim V_2 - \dim V_1$,
there exists $y_1,\dots,y_{\ell}$ such that the set
$\{\sigma(x_1),\dots,\sigma(x_k),\sigma(y_1),\dots,\sigma(y_{\ell})\}$ is
linearly independent over $\C$.
\end{sublem}
\begin{proof}
The proof is by induction on $\ell$.
Suppose we have $y_1,\dots,y_{\ell}$ as in the sublemma and $\dim V_2 - \dim V_1 \geq \ell +1$.
We will find $y_{\ell+1}$.
\par
Pick $z_1, \dots, z_m \in \overline{C}$ such that
$\{\sigma(x_1), \dots, \sigma(x_k), \sigma(y_1),\dots \sigma(y_\ell), \sigma(z_1), \dots, \linebreak \sigma(z_m)\}$
is a basis of $\overline{C}$ as a $\C$-vector space.
In particular, $\{x_1,\dots,x_k,y_1, \dots, y_\ell, \linebreak z_1, \dots, z_m\}$ is a basis of $C$ as a
$\Lambda (G)$-vector space.
We define $A : C(G) \to C(G)$ a $\Lambda (G)$-linear isomorphism by
$$A(x_i) = q^{\mu(\lambda_q(x_i))} \sigma(x_i), \quad
A(y_j) = q^{\mu(\lambda_q(y_j))} \sigma(y_j), \quad
A(z_h)=z_h
$$
for $i=1,\dots,k$, $j=1,\dots,\ell$, $h=1, \dots, m$.
Note that $A$ preserves filtration and $\sigma \circ A = \sigma$.
We take $y' \in V_2$ that is linearly independent of
$\{ x_1,\dots,x_k,y_1,\dots,y_{\ell}\}$ over $\Lambda (G)$.
We write
$$
A(y') = \sum_{n=1}^{\infty} q^{\mu(\lambda_n)}\overline y'_n
$$
where $\overline y'_n \in \overline C(\lambda_n)$.
Note $\lambda_q(q^{\mu(\lambda_n)}\overline y'_n) = \lambda_n$.
Moreover, we may assume that $\lambda_n > \lambda_{n+1}$ and
$\lim_{n\to \infty} \lambda_n = -\infty$.
\par
By the assumption that $\dim V_2 \geq \dim V_1 + \ell + 1 = k + \ell +1$ and that $\overline y'$ is linearly
independent of $\{x_1,\dots,x_k,y_1,\dots,y_{\ell}\}$,
there exists $n$ such that
\begin{equation}\label{condfornnn}
\overline y'_n \notin \bigoplus_{i=1}^k \C \sigma(x_i)
\oplus
\bigoplus_{j=1}^{\ell} \C \sigma(y_j).
\end{equation}
Let $n_0$ be the smallest number satisfying (\ref{condfornnn}).
Put
$$
y'' = \sum_{n=n_0}^{\infty} q^{\mu(\lambda_n)}\overline y'_n.
$$
Clearly, $\sigma(y'')$ is linearly independent to
$\sigma(x_1), \dots, \sigma(x_k), \sigma(y_1),\dots, \sigma(y_\ell)$.
Hence
$y_{\ell +1}=A^{-1}(y'')$ has the required property.
\end{proof}
Lemma \ref{lemstandardbasis0} easily follows from Sublemma \ref{mslemmasandard}.
\end{proof}
\par
We now consider $\partial : C(G) \to C(G)$ and its matrix element
with respect to
a basis of $C(G)$.
Choose a basis $\{e'_i \mid i=1,\dots, b\} \cup
\{e''_i \mid i=1,\dots, h\}
\cup
\{e'''_i \mid i=1,\dots, b\} $
such that
$\{e'_i \mid i=1,\dots, b\}$ is a standard basis of $\text{\rm Im}\,\partial$,
 $\{e'_i \mid i=1,\dots, b\} \cup
\{e''_i \mid i=1,\dots, h\} $ is a standard basis of $\text{\rm Ker}\,\partial$
and
 $\{e'_i \mid i=1,\dots, b\} \cup
\{e''_i \mid i=1,\dots, h\}
\cup
\{e'''_i \mid i=1,\dots, b\} $
is a standard basis of $C$.
(We may also assume that $e'_i, e''_i, e'''_i$ are \emph{homogeneous so that they lie}
either in  $C^0$ or in  $C^1$.)
Such a basis exists by Lemma \ref{lemstandardbasis0}.
\par
\begin{lem}
If $a \in H(C(G),\partial)$, there exists a unique $a_i \in \Lambda^{\downarrow}(G)$ such that
$\sum_{i=1}^h a_i e''_i$ represents $a$.
Moreover
\begin{equation}\label{infattainH00}
\inf\{ \lambda_q(x) \mid x \in \text{\rm Ker}\,\partial, \,\, a =[x]\}
=
\lambda_q\left(\sum_{i=1}^h a_i e''_i\right).
\end{equation}
\end{lem}
The proof is easy and so omitted.
\par
We note that by the definition \eqref{eq:G'} of $G'$,
$
\lambda_q\left(\sum_{i=1}^h a_i e''_i\right) \in G'.
$
Proposition \ref{spectrarityalg} is proved.
\end{proof}

\begin{rem}
From the above discussion we have proved
$$
\inf \{ \lambda_q(x) \mid x \in C(G), \partial x = 0,
[x] = \frak x\}
=
\inf \{ \lambda_q(x) \mid x \in C, \partial x = 0,
[x] = \frak x\}
$$
for $\frak x \in {\rm Im}(H(C(G)) \to H(C))$ at the same time.
\end{rem}
\index{Usher's spectrality lemma|)}

\subsection{Proof of nondegenerate spectrality}
\label{subsec:spectrarity}

In this subsection we apply Proposition \ref{spectrarityalg} to prove the following
theorem.
\begin{thm}\label{bulkspectrality}
If $H$ is nondegenerate, then
$
\rho^{\frak b}(H;a) \in {\rm Spec}(H;\frak b).
$
\end{thm}
\begin{proof}
We put $G = G(M,\omega,\frak b)$.
Let $\overline C$ be the $\C$ vector space
whose basis is given by $\{[\gamma] \mid \gamma \in \text{\rm Per}(H)\}$.
Then we have
$$
C(G) \cong CF(M,H;\Lambda^{\downarrow}_{\frak b} (M)), \quad C \cong CF(M,H;\Lambda^{\downarrow}).
$$
In fact, an isomorphism
$
I : C(G) \cong CF(M,H;\Lambda^{\downarrow}_{\frak b} (M))
$
can be defined by
\begin{equation}\label{defIII}
I([\gamma])  =   \llb \gamma,w_\gamma \rrb,
\end{equation}
where we take and fix a bounding disc $w_\gamma$ for each $\gamma$.
\par
For each member $e_i = [\gamma_i]$ of the basis of $\overline C$,
we put
$
\lambda_i^0 = \mathcal A_{H}([\gamma_i,w_{\gamma_i}]).
$
Then
$
G' = {\rm Spec}(H;\frak b)
$
and the map $I$ preserves filtration.
Theorem \ref{bulkspectrality} now follows from
Proposition \ref{spectrarityalg}.
\end{proof}

\section{Proof of $C^0$-Hamiltonian continuity}
\label{sec:proofcontinuity}
\index{spectral invariant!$C^0$-Hamiltonian continuity}

In this section we prove the following:

\begin{thm}\label{contspect}
Let $H, \, H' : S^1\times M \to \R$ be smooth functions such that
$\psi_H$ and $\psi_{H'}$ are nondegenerate.
Let $a \in QH^*_{\frak b}(M)$ and $\frak b \in H^{{\rm even}}(M;\Lambda_0)$.
Then we have
\begin{equation}
-E^+(H'-H) \leq
\rho^{\frak b}(H';a) - \rho^{\frak b}(H;a)
\leq E^-(H'-H).
\end{equation}
\end{thm}

Theorem \ref{contspect} together with Theorem \ref{homotopyinvbulk}
implies Theorem \ref{axiomshbulk} (6). (See the end of Section \ref{sec:proofhomotopy}.)
We will also prove the following theorem at the same time in this section.

\begin{thm}\label{Jindepedence}
The value $\rho^{\frak b}(H,J;a)$ is independent of
the choices of $J$ and the abstract perturbations of the moduli space
we use during the construction of the number $\rho^{\frak b}(H,J;a)$.
\end{thm}

Theorem \ref{Jindepedence} is Theorem \ref{homotopyinvbulk} (1).

\begin{proof}
The proofs of Theorems \ref{contspect}, \ref{Jindepedence} are mostly the same as
one presented in \cite{oh:alan,oh:dmj,oh:minimax}.
Let $H,\, {H'}$ be in Theorems \ref{contspect} and $J, \, J' \in j_\omega$.
We interpolate them by the family in $\CP(j_\omega) = \operatorname{Map}([0,1],j_\omega)$
$$
(F^s,J^s), \quad 0 \leq s \leq 1
$$
where $\{J^s\}_{0 \leq s \leq 1}$ with $J^0 = J, \, J^1 = J'$ and
\begin{equation}\label{Fsformula}
F^s : = H + s({H'}-H) : S^1 \times M \to \R.
\end{equation}
(Note $J^s \ne J_s$ where $J_s$ is as in (\ref{HsJs}).)
Let $\chi: \R \to [0,1]$ be as in Definition \ref{defn:chi} and elongate
the family to the $(\R \times S^1)$-family $(F^\chi,J^\chi)$ by
$$
F^{\chi}(\tau,t,x) = F^{\chi(\tau)}(t,x),
\quad
J^{\chi}_t = J^{\chi(\tau)}_t.
$$
We put $F^{\chi(\tau)}_t(x) = F^{\chi(\tau)}(t,x)$.
Using this family, we construct a chain map
\begin{equation}\label{HtoGchainmap}
\mathcal P^{\frak b}_{(F^\chi,J^\chi), H, H'} :
(CF(M;H;\Lambda^\downarrow),\partial_{(H,J)}^{\frak b})
\to (CF(M;{H'};\Lambda^\downarrow),\partial_{({H'},J')}^{\frak b})
\end{equation}
by studying the equation
\begin{equation}\label{eq:HJCR242}
\dudtau + J^\chi\Big(\dudt - X_{F^{\chi(\tau)}_t}(u)\Big) = 0
\end{equation}
with finite energy
$$
E_{(F^\chi,J^\chi)}(u) = \frac{1}{2} \int \Big(\Big|\dudtau\Big|^2_{J^\chi} + \Big|
\dudt - X_{F^{\chi(\tau)}_t}(u)\Big|_{J^\chi}^2 \Big)\, dt\, d\tau
$$
To simplify the notation, we denote $\mathcal P^{\frak b}_{(F^\chi,J^\chi), H, H'}$
by $\mathcal P^{\frak b}_{(F^\chi,J^\chi)}$ when no confusion can occur.

Let $[\gamma,w] \in \mathrm{Crit}(\mathcal A_H)$, $[\gamma',w'] \in \mathrm{Crit}(\mathcal A_{H'})$.
First we prove the following bound for the action change.
\begin{lem}\label{connectinghomofilt} If the pair $[\gamma, w], \, [\gamma',w']$
carries a finite energy solution $u$ of \eqref{eq:HJCR242} satisfying
$$
u(-\infty) = \gamma, \, u(\infty) = \gamma', \quad w \# u \sim w'
$$
then
$$
\mathcal A_{H'}([\gamma',w']) - \mathcal A_H([\gamma,w]) \le E^-(H'-H).
$$
\end{lem}
\begin{proof}
Let $u$ be as above. Then by the same computation as in the proof of Lemma \ref{filtered}, we obtain
\beastar
\mathcal A_{H'}([\gamma',w']) - \mathcal A_{H}([\gamma,w])
& = &
- E_{(H,J)}(u) -
\int_{\R} \int_{S^1} \chi'(\tau) (H'-H) \circ u(\tau,t) \,\,dtd\tau \\
&\le & \int_0^1 - \min_x (H'_t(x) - H_t(x))\, dt  = E^-(H'-H)
\eeastar where the inequality follows since $\chi'\ge 0$ and
$\int\chi' d\tau = 1$.
\end{proof}

\begin{defn}\label{stripmk0int24}
\index{$\mathcal M_{\ell}(F^\chi,J^\chi;[\gamma,w],[\gamma',w'])$}
We denote by $\overset{\circ}{\CM}_{\ell}(F^\chi,J^\chi;[\gamma,w],[\gamma',w'])$ the set of all
smooth maps
$u: \R \times S^1 \to M$ which satisfy the following conditions:
 \begin{enumerate}
\item The map $u$ satisfies the equation \eqref{eq:HJCR242}.
\item The energy $E_{(F^\chi,J^\chi)}(u)$ is finite.
\item The map $u$ satisfies the following asymptotic boundary condition:
\footnote{The moduli space ${\CM}_{\ell}(F^\chi,J^\chi;[\gamma,w],[\gamma',w'])$
is {\it different} from  ${\CM}_{\ell}(F_\chi,J_\chi;[\gamma,w],[\gamma',w'])$, which is
defined in Definition \ref{stripmk0int1}.
Compare the definitions (\ref{Fsformula}) and  (\ref{HsJs})
of $F^{\chi}$ and $F_{\chi}$ respectively.}
$$
\lim_{\tau\to -\infty}u(\tau, t) = \gamma(t), \quad \lim_{\tau\to
+\infty}u(\tau, t) = \gamma'(t).
$$
\item The homotopy class of $w\# u$ is $[w']$,
where $\#$ is the obvious concatenation.
\item
$z_i^+$ are mutually distinct points in $\R \times S^1$.
\end{enumerate}
\end{defn}
The assignment $(u;z^+_1,\dots,z^+_{\ell}) \mapsto (u(z^+_1),\dots,u(z^+_{\ell}))$
 defines an evaluation map
$$
{\rm ev} = ({\rm ev}_1\dots,{\rm ev}_{\ell}) = \overset{\circ}{\CM}_{\ell}(F^\chi,J^\chi;[\gamma,w],[\gamma',w']) \to M^{\ell}.
$$
We remark that for any $\alpha \in \pi_2(M)$ there exists a canonical homeomorphism
\begin{equation}\label{homeo955}
\overset{\circ}{\CM}_{\ell}(F^\chi,J^\chi;[\gamma,w],[\gamma',w'])
\cong
\overset{\circ}{\CM}_{\ell}(F^\chi,J^\chi;[\gamma,\alpha\# w],[\gamma',\alpha\# w']).
\end{equation}
\begin{prop}\label{piuBULKkura2}
\begin{enumerate}
\item
The moduli space
$\overset{\circ}{\CM}_{\ell}(F^\chi,J^\chi;[\gamma,w],[\gamma',w'])$ has a compactification
${\CM}_{\ell}(F^\chi,J^\chi;[\gamma,w],[\gamma',w'])$ that is Hausdorff.
\item
The space ${\CM}_{\ell}(F^\chi,J^\chi;[\gamma,w],[\gamma',w'])$ has an orientable Kuranishi structure with corners.
\item
The normalized boundary of ${\CM}_{\ell}(F^\chi,J^\chi;[\gamma,w],[\gamma',w'])$ is described by
\begin{equation}\label{bdryPIunikinmodu2}
\aligned
&\partial {\CM}_{\ell}(F^\chi,J^\chi;[\gamma,w],[\gamma',w']) \\
&= \bigcup
{{\CM}}_{\#\mathbb L_1}(H,J;[\gamma,w],[\gamma'';w''])
\times {\CM}_{\#\mathbb L_2}(F^\chi,J^\chi;[\gamma'',w''],[\gamma',w'])
\\
& \quad \cup
\bigcup
{\CM}_{\#\mathbb L_1}(F^\chi,J^\chi;[\gamma,w],[\gamma''';w'''])
\times {\CM}_{\#\mathbb L_2}({H'};J';[\gamma''';w'''],[\gamma',w'])
\endaligned\end{equation}
where the first union is taken over all $(\gamma'',w'') \in \text{\rm Crit}(\mathcal A_H)$,
 and $(\mathbb L_1,\mathbb L_2) \in \text{\rm Shuff}(\ell)$
 and the second union is taken over all $(\gamma''',w''') \in \text{\rm Crit}(\mathcal A_{H'})$,
 and $(\mathbb L_1,\mathbb L_2) \in \text{\rm Shuff}(\ell)$.

\item
Let $\mu_H : \mbox{\rm Crit}(\CA_H) \to  \Z$, $\mu_{H'} : \mbox{\rm Crit}(\CA_{H'}) \to  \Z$,
be the Conley-Zehnder indices. Then the (virtual) dimension satisfies
the following equality:
\begin{equation}\label{dimensionboundaryb24}
{\CM}_{\ell}(F^\chi,J^\chi;[\gamma,w],[\gamma',w']) =
\mu_{H'}([\gamma',w']) - \mu_H([\gamma,w]) +2\ell.
\end{equation}
\item
We can define orientations of ${\CM}_{\ell}(F^\chi,J^\chi;[\gamma,w],[\gamma',w'])$ so that
$(3)$ above is compatible with this orientation.
\item
${\rm ev}$ extends to a weakly submersive map
$$
{\CM}_{\ell}(F^\chi,J^\chi;[\gamma,w],[\gamma',w']) \to M^{\ell},
$$
which we denote also by ${\rm ev}$. It is compatible with $(3)$.
\item
The homeomorphism $(\ref{homeo955})$ extends to the compactifications
and their Kuranishi structures are identified by the homeomorphism.
\end{enumerate}
\end{prop}
The proof of Proposition \ref{piuBULKkura2} is the same as that of Proposition \ref{connkura}
(See \cite[Section 31 and Part 4]{fooo:techI}.) and
so is omitted.
\begin{defn}
Let $[\gamma,w] \in \text{\rm Crit}(\mathcal A_H)$,
$[\gamma',w'] \in \text{\rm Crit}(\mathcal A_{H'})$   and let $h_i$ ($i = 1,\ldots,\ell$) be
differential forms on $M$. We define
$\frak n_{(F^\chi,J^\chi);[\gamma,w],[\gamma',w']} (h_1,\ldots,h_{\ell}) \in \C$ by
\begin{equation}\label{ndRdef634}
\frak n_{(F^\chi,J^\chi);[\gamma,w],[\gamma',w']} (h_1,\ldots,h_{\ell})
= \int_{{\CM}_{\ell}(F^\chi,J^\chi;[\gamma,w],[\gamma',w'])}
{\rm ev}_1^*h_1 \wedge\dots \wedge {\rm ev}_\ell^*h_\ell.
\end{equation}
(We take and use CF-perturbation to define the right hand side.)
\par
By definition (\ref{ndRdef634}) is zero if
$$
\sum_{i=1}^{\ell} \text{\rm deg} \, h_i
\ne \dim {\CM}_{\ell}(F^\chi,J^\chi;[\gamma,w],[\gamma',w']),
$$
where the right hand side is as in (\ref{dimensionboundaryb24}).
We extend (\ref{ndRdef634}) to
$$
\frak n_{(F^\chi,J^\chi);[\gamma,w],[\gamma',w']} :
B_{\ell}(\Omega(M) \widehat{\otimes} \Lambda)
\to
\Lambda^{\downarrow}
$$
by ($\Lambda,\Lambda^{\downarrow})$ linearity.
\par
Note that we need to make an appropriate choice of a
compatible system of CF-perturbations in order to define the integral
given in (\ref{ndRdef634}). See Remark
\ref{compatimulti}. We sometimes omit this remark from now on.
\par
Let $\frak b \in H^{{\rm even}}(M;\Lambda _0)$.
We split
$\frak b = \frak b_0 + \frak b_2 + \frak b_+$ as in
(\ref{decompb}).
We take closed differential forms representing $\frak b_0$, $\frak b_2$, $\frak b_+$
which we still denote by the same letters respectively with an abuse of notations.
Define $\frak n_{(F^\chi,J^\chi)}^{\frak b}([\gamma,w],[\gamma',w']) \in \Lambda^{\downarrow}$ by
the sum
\begin{equation}\label{bdrycoefb4}
\aligned
\frak n_{(F^\chi,J^\chi)}^{\frak b}([\gamma,w],[\gamma',w'])
=
\sum_{\ell=0}^{\infty}&\frac{\exp({\int (w')^*\frak b_2 -
\int w^* \frak b_2)}}{\ell!}  \\
&\frak n_{(F^\chi,J^\chi);[\gamma,w],[\gamma',w']}(
\underbrace{\frak b_+,
\dots,\frak b_+}_{\ell}).
\endaligned
\end{equation}
\end{defn}
We can prove that the sum in (\ref{bdrycoefb4}) converges in $q$-adic topology,
in the same way as in Lemma \ref{adiccomv1}.
We now set \index{$\mathcal P^{\frak b}_{(F^\chi,J^\chi)}(\llb \gamma,w \rrb)$}
\begin{equation}\label{Piunikinb4}
\mathcal P^{\frak b}_{(F^\chi,J^\chi)}(\llb \gamma,w \rrb)
= \sum_{[\gamma',w'] \in \text{\rm Crit}(\mathcal A_{H'})} \frak n^\frak b_{(F^\chi,J^\chi)}
([\gamma,w],[\gamma',w'])\,\,
\llb \gamma',w' \rrb.
\end{equation}
Here we suppose $\llb \gamma,w \rrb$ is
represented by $[\gamma,w]$, i.e., $\pi([\gamma,w]) = \llb \gamma,w \rrb$.
We then can show by using Proposition \ref{piuBULKkura2} (7)
that the right hand side is independent of such a representative $[\gamma,w]$.
(We take our CF-perturbation so that it is
compatible with this isomorphism.)
\par
We can prove that the right hand side defines an element of
$CF({H'},J';\Lambda^{\downarrow})$ in the same way as in Lemma \ref{adiccomv2}.
Thus we have defined (\ref{HtoGchainmap}). Then
\begin{equation}\label{boundaryprop4}
\mathcal P^{\frak b}_{(F^\chi,J^\chi)}  \circ \partial_{(H,J)}^{\frak b}
= \partial_{({H'},J')}^{\frak b} \circ \mathcal P^{\frak b}_{(F^\chi,J^\chi)}
\end{equation}
is a consequence of (\ref{bdryPIunikinmodu2}), Stokes' theorem (Theorem \ref{them48})
and the composition formula (Theorem \ref{compform}).

Now we would like to study the relationship between
the Piunikhin maps $\CP^{\frak b}_{(H_\chi,J_\chi)}$ as we vary
the pair $(H,J)$ and the elongation function $\chi \in \CK$ introduced in
Definition \ref{defn:chi}. Let $\chi \in \CK$ and consider the three maps
$
\mathcal P^{\frak b}_{(H_{\chi},J_{\chi})}$, $\mathcal P^{\frak b}_{(H'_{\chi},J'_{\chi})}$ {and }
$\mathcal P^{\frak b}_{(F^{\chi},J^{\chi})}
$.

\begin{prop}\label{compaticompositte}
$\mathcal P^{\frak b}_{(F^{\chi},J^{\chi})} \circ \mathcal P^{\frak b}_{(H_{\chi},J_{\chi})}$
is chain homotopic to $\mathcal P^{\frak b}_{(H'_{\chi},J'_{\chi})}$.
\end{prop}
\begin{proof}
Let $J_s, J'_s$ be as in
(\ref{HsJs}) and $(F^s,J^{s})$ as in
(\ref{Fsformula}).
For $S \in [1,\infty)$, $\tau \in \R$, we define
$
G_S(\tau,t,x)
$
by
$$
G_S(\tau,t,x)
=
\begin{cases}
\chi(\tau + 2S)H_t(x)
& \tau \le 0, \quad S\ge 1 \\
F^{\chi(\tau - 2S)}_t(x)
& \tau \ge 0, \quad S\ge 1.
\end{cases}
$$
We also define $J_S(\tau,t,x)$ by
$$
J_S(\tau,t,x)
=
\begin{cases}
J_{\chi(\tau + 2S),t}
& \tau \le 0, \quad S\ge 1 \\
J^{\chi(\tau - 2S)}_t
& \tau \ge 0, \quad S\ge 1.
\end{cases}
$$
We extend the definition of $G_S$ to those $S \in [0,1]$ by the formula,
$$
G_S(\tau,t,x)
=
(1-S) \chi(\tau)H'(t,x)
+ SG_1(\tau,t,x).
$$
Note that $G_S$ may not be smooth at $S=1$ on $S$ and on $\tau \in [-10,10]$.
Here the specific number 10 does not assume any significance which
can be replaced by any sufficiently large constant but fixed.
It is chosen so that $G_S$ is smooth on $\tau$ at least outside $[-10,10]$.
We modify $G_S$ on a small
neighborhood of this set so that $G_S$ becomes
a smooth family. We denote the resulting modification still by the same symbol $G_S$
by an abuse of notation.
\par
We also extend the definition of $J_S$ to those $S\in [0,1]$ so that the following holds:
\begin{enumerate}
\item
At $S=0$, $J_S(\tau,t)$ coincides with $J'_{\chi(\tau),t}$.
\item
$J_S$ is $t$ independent for $\tau < -10$.
(It may be $S$-dependent there.)
\end{enumerate}
\par
We denote the family obtained above by
$$
(\CG,\CJ) = \{(G_S,J_S)\}_{S \in \R_{\ge 0}}.
$$
We put $G_{S;\tau,t}(x) = G_S(\tau,t,x)$.
Now for each $S \in \R_{\ge 0}$, we consider
\begin{equation}\label{CRGJS}
\dudtau + J_S\Big(\dudt - X_{G_{S;\tau,t}}(u)\Big) = 0
\end{equation}
and define its moduli space $\overset{\circ}{\CM}_{\ell}(G_S,J_S;*,[\gamma,w])$
defined in Definition \ref{buldPmoduli}. We put
\index{$\mathcal M_{\ell}(para;*,[\gamma',w'])$}
\begin{equation}\label{FSmodulipara}
\overset{\circ}{\CM}_{\ell}(para;*,[\gamma',w'])
= \bigcup_{S \in \R_{\ge 0}} \{S\} \times\overset{\circ}{\CM}_{\ell}(G_S,J_S;*,[\gamma',w']).
\end{equation}
Here `{\it para}' stands for the parameterized moduli space.
We have the natural evaluation maps
\begin{equation}\label{eq:evpara*}
\text{\rm ev}:\overset{\circ}{\CM}_{\ell}(para;*,[\gamma',w']) \to M^\ell,\quad
\text{\rm ev}_{-\infty}:\overset{\circ}{\CM}_{\ell}(para;*,[\gamma',w']) \to L.
\end{equation}
\begin{lem}\label{parakuralem}
\begin{enumerate}
\item
The moduli space
$\overset{\circ}{\CM}_{\ell}(para;*,[\gamma',w'])$ has a compactification,
denoted by ${\CM}_{\ell}(para;*,[\gamma',w'])$, that is Hausdorff.
\item
The space ${\CM}_{\ell}(para;*,[\gamma',w'])$ has an orientable Kuranishi structure with corners.
\item
The normalized boundary of ${\CM}_{\ell}(para;*,[\gamma',w'])$ is described by
the following three types of components:
\begin{equation}\label{bdryparamoduli1}
\bigcup{\CM}_{\#\L_1}(para;*,[\gamma'',w''])
\times
{\CM}_{\#\L_2}(H',J';[\gamma'',w''],[\gamma',w'])
\end{equation}
where the union is taken over all
$(\L_1,\L_2) \in \text{\rm Shuff}(\ell)$, $[\gamma'',w'']
\in \text{\rm Crit}(\mathcal A_{H'})$.
\begin{equation}\label{bdryparamoduli2}
\bigcup{\CM}_{\#\mathbb L_1}(H_{\chi},J_{\chi};*,[\gamma,w])
\times {\CM}_{\#\mathbb L_2}(F^{\chi},J^{\chi};[\gamma,w],[\gamma',w'])
\end{equation}
where the union is taken over all
$(\L_1,\L_2) \in \text{\rm Shuff}(\ell)$, $[\gamma,w]
\in \text{\rm Crit}(\mathcal A_{H})$.
\begin{equation}\label{bdryparamoduli3}
{\CM}_{\ell}(H'_{\chi},J'_{\chi};*,[\gamma',w']).
\end{equation}
\item
Then the (virtual) dimension is given by
\begin{equation}\label{dimensionboundaryb249}
\dim {\CM}_{\ell}(para;*,[\gamma',w']) =
\mu_{H'}([\gamma',w']) +n+1 +2\ell.
\end{equation}
\item
We can define orientations of ${\CM}_{\ell}(H_{\chi},J_{\chi};*,[\gamma',w'])$ so that
$(3)$ above is compatible with this orientation.
\item The evaluation map $\text{\rm ev}:\overset{\circ}{\CM}_{\ell}(para;*,[\gamma',w']) \to M^\ell$
given in \eqref{eq:evpara*} extends to a weakly submersive map
${\CM}_{\ell}(H_{\chi},J_{\chi};*,[\gamma',w']) \to M^{\ell}$, which we
denote also by ${\rm ev}$. The family of evaluation maps are compatible with
the boundary description $(3)$.
\par Similarly the map
${\rm ev}_{-\infty} : \overset{\circ}{\CM}_{\ell}(para;*,[\gamma',w']) \to L$
can be also extended to ${\CM}_{\ell}(H_{\chi},J_{\chi};*,[\gamma',w'])$.
\end{enumerate}
\end{lem}
\begin{proof}
The proof is mostly the same as that of Proposition \ref{connkura}.
We only mention how the boundary components are given as in (3).
\par
The type (\ref{bdryparamoduli1}) appears when there is a bubble to  $\tau \to \infty$.
The bubble to $\tau \to -\infty$ is of codimension 2 by the $S^1$-equivariance.
(See the proof of Lemma \ref{MBchainhomotopyprop}.)
\par
The types (\ref{bdryparamoduli2}) and (\ref{bdryparamoduli3}) correspond to the cases of
$S \to \infty$ and $S=0$ respectively.
\end{proof}
We use this parameterized moduli space in the same way as we did in the definition of
$\mathcal P^{\frak b}_{(H_\chi,J_\chi)}$ and define a degree one map
\index{$\mathcal H^{\frak b}_{(\CG,\CJ)}$}
$$
\mathcal H^{\frak b}_{(\CG,\CJ)}
: \Omega(M) \widehat\otimes \Lambda^\downarrow
\to CF(M,{H'};\Lambda^\downarrow).
$$
\par
We apply Lemma \ref{parakuralem}
together with Stokes' theorem (Theorem \ref{them48}) and the
composition formula (Theorem \ref{compform}) to derive the equality
\begin{equation}\label{Ghomotopy}
\partial^{\frak b}_{(H',J')} \circ \mathcal H^{\frak b}_{(\CG,\CJ)}
+ \mathcal H^{\frak b}_{(\CG,\CJ)}  \circ \partial
= \mathcal P^{\frak b}_{(F^{\chi},J^{\chi})}\circ \mathcal P^{\frak b}_{(H_{\chi},J_{\chi})} -
\mathcal  \mathcal P^{\frak b}_{(H'_{\chi},J'_{\chi})}.
\end{equation}
Then Proposition \ref{compaticompositte} follows from (\ref{Ghomotopy}).
\end{proof}
Now we are ready to complete the proof of Theorem \ref{contspect}.
By Lemma \ref{connectinghomofilt}, we have
\begin{equation}\label{connectinghomofiltformu}
\mathcal P^{\frak b}_{(F^{\chi};J^{\chi})} \left( F^{\lambda}
CF(M,H;\Lambda^\downarrow) \right) \subset F^{\lambda +E^-(H'-H)}
CF(M,{H'};\Lambda^\downarrow).
\end{equation}
Let $\rho = \rho^{\frak b}(H;a)$ and $\epsilon >0$. We take an element $x \in
F^{\rho + \epsilon}  CF(M,H;\Lambda^\downarrow)$ representing the class
$\mathcal P^{\frak b}_{(H_\chi,J_\chi)}(a^\flat)$.
Then the element
$\mathcal P^{\frak b}_{(F^{\chi},J^{\chi})}(x) \in F^{\rho +
\epsilon - E^-(H'-H)}  CF(M,G;\Lambda^\downarrow)$ represents the
Floer homology class $\mathcal P^{\frak
b}_{(F^{\chi},J^{\chi})}\mathcal P^{\frak
b}_{(H_{\chi},J_{\chi})}(a^\flat) = \mathcal P^{\frak
b}_{(H'_{\chi},J'_{\chi})}(a^\flat)$. (Proposition
\ref{compaticompositte}). Therefore $\rho^{\frak b}({H'};a) \le \rho +
\epsilon + E^-(H'-H) $. Since $\epsilon$ is an arbitrary positive number, we
have
$$
\rho^{\frak b}({H'};a) \le \rho^{\frak b}(H;a) +  E^-(H'- H).
$$
By exchanging the role of $H'$ and $H$ we have
$$
\rho^{\frak b}(H;a) \le \rho^{\frak b}({H'};a) +  E^+({H'}- H).
$$
The proof of Theorem \ref{contspect} is complete.
\end{proof}
We also note that Theorem \ref{Jindepedence} follows from the above
argument applied to the case $H=H'$ but $J\ne J'$. \qed
\index{spectral invariant!$C^0$-Hamiltonian continuity|)}
\par\medskip
\section{Proof of homotopy invariance}
\label{sec:proofhomotopy}\index{spectral invariant!homotopy
invariance|(}

In this section we prove Theorem \ref{homotopyinvbulk} (3) and
Theorem \ref{homotopyinvtheorem}.
Let $H^s$, $s \in [0,1]$ be a {one-parameter
family of normalized periodic Hamiltonians $H^s : S^1 \times M \to \R$
such that:
\be\label{eq:homotopy}
\phi_{H^s}^1 \equiv \psi, \, \phi_{H^s}^0 \equiv id \quad \mbox{for all }\, s \in [0,1].
\ee
We assume without loss of generality that
$H^s(t,x) \equiv 0$ on a neighborhood of $\{[0]\} \times M \subset S^1 \times M$.
\par
We first define an isomorphism
\begin{equation}
I_s : {\rm Crit}(\mathcal A_{H^0}) \to {\rm  Crit}(\mathcal A_{H^s}).
\end{equation}
Let $\gamma \in {\rm Per}(\mathcal A_{H^0})$.
Put $p=\gamma(0)$ and
$\gamma_s = z_p^{H^s}$ defined by
$$
z_p^{H^s}(t) = \phi_{H^s}^t(p).
$$
By \eqref{eq:homotopy},  $z_p^{H^s}(1) = \gamma(1)= \psi(p)$ for all $s \in [0,1]$.
Moreover, we have
$
z_p^{H^s} \in {\rm Per}(\mathcal A_{H^s}).
$
We note that $z_p^{H^0} = \gamma$.
\par
Next let $[\gamma,w] \in {\rm Crit}(\mathcal A_{H^0})$ be a lifting of $\gamma$.
By concatenating $w$ with
$\bigcup_{\sigma \le s} \gamma_{\sigma}$ to obtain
$w_{s} : D^2 \to M$ such that
$w_{s} \vert_{\partial D^2} = \gamma_s$.
We now define
\begin{equation}\label{Isigma}
I_s([\gamma,w]) = [\gamma_s,w_s].
\end{equation}

The following is proved in \cite[Proposition 3.1]{schwarz} for the symplectically
aspherical case and in \cite{oh:jkms} in general.
The following proof is borrowed from \cite{oh:jkms}.
\begin{prop}\label{spectrumconst} Suppose that each $H^s$ is normalized and
satisfies \eqref{eq:homotopy}. Then we have
$$
\mathcal A_{H^s}(I_s([\gamma,w])) = \mathcal A_{H^{0}}([\gamma,w])
$$
for all $s \in [0,1]$.
\end{prop}
\begin{proof}
To prove the equality, it is enough to prove
\be\label{eq:derivative}
\frac{d}{d s} \mathcal A_{H^s}(I_s([\gamma,w])) = 0
\ee
for all $s \in [0,1]$.
\par
Note that $\mathcal A_{H^0}(I_0([\gamma,w])) = \mathcal A_{H^{0}}([\gamma,w])$.
Denote $H = H(s,t,x):= H^s(t,x)$ and denote by
$K=K(s,t,x)$ the normalized Hamiltonian generating the vector field
$$
\frac{\del \phi_{H^s}^t}{\del s} \circ (\phi_{H^s}^t)^{-1} =: X_K
$$
in $s$-direction.  We compute
$$
\frac{d}{d s} \mathcal A_{H^s}(I_s([\gamma,w])) = (d\CA_{H^s}(I_s([\gamma,w])))
\left(\frac{d}{ds}I_s([\gamma,w])\right)
- \int_0^1 \frac{\del H}{\del s}(s,t,\gamma_s(t))\, dt.
$$
Using that $I_s([\gamma,w]) \in \Crit \CA_{H^s}$, this reduces to
\be\label{eq:reduction}
\frac{d}{d s} \mathcal A_{H^s}(I_s([\gamma,w])) = - \int_0^1 \frac{\del H}{\del s}(s,t,\gamma_s(t))\, dt.
\ee
By \eqref{eq:homotopy}, we have
$$
X_K(s,1,x) = 0 = X_K(s,0,x)
$$
which implies $dK_{s,1} \equiv 0$. Therefore
$K_{s,1} \equiv c(s)$ where $c:[0,1] \to \R$ is a function of $s$ alone.
Then by the normalization condition, we obtain
\be\label{eq:Ks10}
K_{s,1} \equiv 0 \equiv K_{s,0}.
\ee
\begin{lem}
\be\label{eq:delsdelt}
\frac{\del H}{\del s}(s,t,\phi_{H^s}^t(p)) = \frac{\del}{\del t}\left(K(s,t,\phi_{H^s}^t(x)(p))\right).
\ee
\end{lem}
\begin{proof} The following is proved
$$
\frac{\del K}{\del t} - \frac{\del H}{\del s} - \{H,K\} = 0
$$
in \cite[Proposition I.1.1]{banyagapaper}  for \emph{normalized}
family $H^s$. By rewriting this into
$$
\frac{\del K}{\del t} + \{K,H\} = \frac{\del H}{\del s}
$$
and recalling the definition
$$
\{K,H\} = \omega(X_K,X_H) = dK(X_H)
$$
of the Poisson bracket (in our convention), \index{Poisson
bracket}it is easy to check that this condition is equivalent to
\eqref{eq:delsdelt}. Here the exterior differential and the Poisson
bracket are taken over $M$ for each fixed $(s,t)$.
\end{proof}
Therefore we obtain
\beastar
\int_0^1 \frac{\del H}{\del s}(s,t,\gamma_s(t))\, dt
& = & \int_0^1 \frac{\del H}{\del s}(s,t,\phi_{H^s}^t(p))\, dt\\
& = & \int_0^1 \frac{\del}{\del t}\left(K(s,t,\phi_{H^s}^t(p))\right)\, dt\\
& = & K(s,1,\phi_{H^s}^1(p)) - K(s,0,\phi_{H^s}^0(p)) \\
& = & K(s,1,\psi(p)) - K(s,0,p) = 0
\eeastar
where the last equality comes from \eqref{eq:Ks10}. Substituting this into
\eqref{eq:reduction}, we have finished the proof.
\end{proof}
The following corollary is immediate.
\begin{cor}\label{cor105}
$\mathrm{Spec}(H^0) = \mathrm{Spec}(H^s)$.
Moreover
$\mathrm{Spec}(H^0;\frak b) = \mathrm{Spec}(H^s;\frak b)$.
\end{cor}
The following lemma is proved
for arbitrary $(M,\omega)$ by the second named author in \cite{oh:ajm1}. (The corresponding
theorem in the aspherical case was proved in \cite{schwarz}
 generalizing
a similar theorem in \cite{hofer-Zeh}.)
\begin{lem}\label{measure0}
The set $\mathrm{Spec}(H)$ has measure zero for any periodic
Hamiltonian $H$.
\end{lem}

This, together with Lemma \ref{measure0} and the fact that the set
$G(M,\omega,\frak b)$ is countable, implies

\begin{cor}\label{measure0b}
$\mathrm{Spec}(H;\frak b)$ has measure zero for any periodic
Hamiltonian $H$ and $\frak b$.
\end{cor}

\begin{proof}[Proof of Theorem \ref{homotopyinvbulk} (3)]
By Theorem \ref{homotopyinvbulk} (1) which is proved in
Section \ref{sec:proofcontinuity}, the number $\rho^{\frak b}(H^s;a)$ is
well-defined independent of the choices of $J$ and
other choices such as abstract perturbations.
By Theorem \ref{contspect} the function
$s \mapsto \rho^{\frak b}(H^s;a)$ is continuous.
Moreover $\rho^{\frak b}(H^{\rho};a)$
is contained in the set $\mathrm{Spec}(H^s;\frak b) \subset \R$
that is independent of $s$ and has Lebesgue measure $0$.
(This independence follows from Corollary \ref{cor105}.)
Therefore $s \mapsto \rho^{\frak b}(H^s;a)$
must be a constant function, as required.
\end{proof}
Theorem \ref{homotopyinvtheorem} is a special case of Theorem
\ref{homotopyinvbulk} for ${\frak b} = 0$.
\index{spectral invariant!homotopy invariance|)}
\section{Proof of the triangle inequality}
\label{sec:prooftriangle}

In this section we prove Theorem \ref{axiomshbulk} (5).
The proof is divided into several steps.

\subsection{Pants products}
\label{subsec:pants}
\index{pants product}

In this subsection, we define a product structure
of Floer homology of periodic Hamiltonian system.
It is called the {\it pants product}.
Let $J_1 = \{J_{1,t}\}$, $J_2 = \{J_{2,t}\}$ be two $S^1$-parameterized
families of compatible almost complex structures on $M$.
We assume that
\begin{equation}
J_{1,t} = J_{2,t} = J_0, \quad \text{if $t$ is in a neighborhood of $[1] \in S^1$.}
\end{equation}
Here $J_0$ is a certain compatible almost complex structure on $M$.
We remark that we have already proved $J$-independence of
the spectral invariant. So we may assume the above condition
without loss of generality.
(Actually we may also choose $J_{1,t} = J_{2,t}  = J_0$ without loss of generality.
See Remark \ref{rem31} (2).)
\par
We next take time-dependent Hamiltonians $H_{1}, H_2$.
After making the associated Hamiltonian
isotopy constant neat $t = 0, 1$, we may
assume
\begin{equation}\label{Hcond11}
H_{1,t} = H_{2,t} = 0, \quad \text{if $t$ is in a neighborhood of $[1] \in S^1$.}
\end{equation}
The {\em pants product} \index{pants product} is defined by a chain map
\begin{equation}\label{mproduct}
\aligned
\frak m^{\rm{cl}}_2 : CF(M,H_1,J_1;\Lambda^{\downarrow})
&\otimes CF(M,H_2,J_2;\Lambda^{\downarrow})
\\
&\to CF(M,H_1 * H_2,J_1 * J_2
;\Lambda^{\downarrow})
\endaligned
\end{equation}
\index{$\frak m^{\rm{cl}}_2$}
where
\begin{equation}\label{eq:defconcat}
(H_1 * H_2) (t,x) =
\begin{cases}
2H_1(2t,x)  &t\le 1/2, \\
2H_2(2t-1,x)  &t\ge 1/2
\end{cases}
\end{equation}
and
\be
(J_1 * J_2) (t,x) = \begin{cases}
J_1(2t,x)  & t\le 1/2, \\
J_2(2t-1,x) & t\ge 1/2.
\end{cases}
\ee
\begin{rem}
Our definition of the product Hamiltonian $H_1 * H_2$ is different from those used in
\cite{schwarz,oh:alan}.
But the same definition is used in \cite{abbsch2}.
\end{rem}
\begin{rem}
Note we include $J_1$,$J_2$ in the notation of
$CF(M,H_i,J_i;\Lambda^{\downarrow})$,
since the boundary operator depends on the almost
complex structure $J_i$.
(Its homology group $HF(M,H_i,J_i;\Lambda^{\downarrow})$
is independent of $J_i$.)
\end{rem}
It is easy to see that
$
\psi_{H_1 * H_2}
=
\psi_{H_2} \circ \psi_{H_1}.
$

In the symplectically aspherical case, the detail of the construction (\ref{mproduct})
is written in \cite{schwarz1}.
Its generalization to arbitrary symplectic manifold is rather immediate
with the virtual fundamental chain technique in the framework of Kuranishi
structure \cite{fukaya-ono}. We treat this construction for the general case here together with its generalization
including bulk deformations.

Let $\Sigma = S^2 \setminus \{\text{3 points}\}$.
We choose a function
$
h : \Sigma \to \R
$
with the following properties:
\begin{conds}\label{cond:h}
\begin{enumerate}
\item It is proper.
\item It is a Morse function with a unique critical point $z_0$ such that
$h(z_0) = 0$.
\item For $\tau < 0$,
the preimage $h^{-1}(\tau)$ is a disjoint union of two $S^1$'s,
and for $\tau > 0$, $h^{-1}(\tau)$ is one $S^1$.
\end{enumerate}
\end{conds}
See Figure 4.
We fix a Riemannian metric on $\Sigma$ such that $\Sigma$ is isometric
to the three copies of $S^1 \times [0,\infty)$ outside a compact set.
Let $\psi_{\nabla h}^t$ be the one parameter subgroup
associated to the gradient vector field of $h$. We put
$$
\frak S = \{z \in \Sigma
\mid  \lim_{t\to\infty}
\psi_{\nabla h}^t(z) = z_0, \,\,\text{or}
\,\,
\lim_{t\to -\infty}
\psi_{\nabla h}^t(z) = z_0 \}
$$
i.e., the union of stable and unstable manifolds of $z_0$. Take a diffeomorphism
$$
\varphi: \R \times ((0,1/2) \sqcup (1/2,1)) \to \Sigma \setminus \frak S
$$
such that
$
h(\varphi(\tau,t)) = \tau
$
and put a complex structure $j_{\Sigma}$ on $\Sigma$ with respect to
which $\varphi$ is conformal. Such a complex structure can be chosen by
first pushing forward the standard one on $\R \times ((0,1/2) \sqcup (1/2,1)) \subset \C$
and extending it to whole $\Sigma$. This choice of $\varphi$ and $j_\Sigma$
also provides the cylindrical ends near each puncture of $\Sigma$.
See Figure 4.
\begin{figure}[h]
\centering
\includegraphics[scale=0.3]
{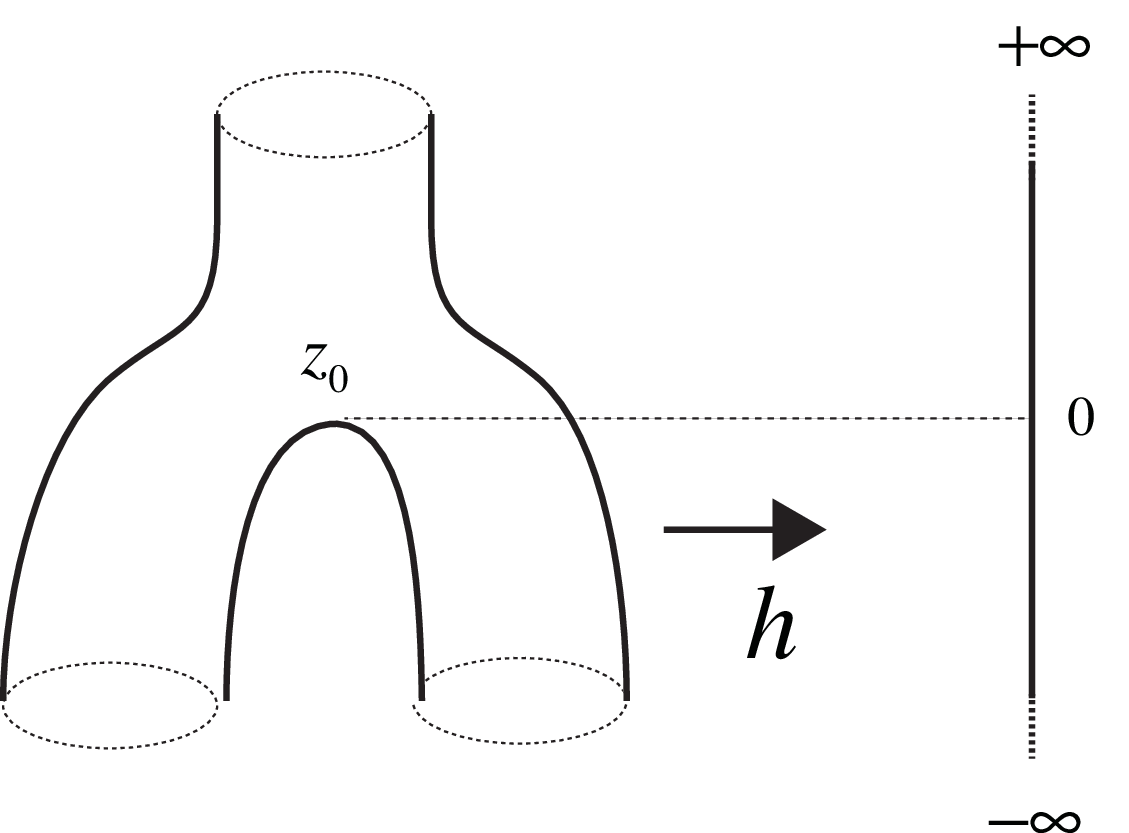}
\caption{function $h$}
\label{Figure4}
\end{figure}
\par
We define a smooth function
$
H^\varphi: \Sigma\times M \to \R
$
by:
\begin{equation}\label{defsharp}
H^\varphi(\varphi(\tau,t),x) = (H_1 * H_2)(t,x)
\end{equation}
on $\Sigma \setminus \frak S$ and extending to $\frak S$
by 0. This is consistent with the assumption (\ref{Hcond11}).
\par
We define a $\Sigma$-parameterized family $J^{\varphi}$ of almost
complex structures by
$$
J^{\varphi}_{\varphi(\tau,t)} = (J_1 * J_2)_t
$$
first on $\Sigma \setminus \frak S$.
Note that the right hand side is $J_0$ in a neighborhood of $\frak S$.
Therefore we can trivially extend its definition
across $\frak S$ to whole $\Sigma$.
\par
For $\tau< 0$, we take the identification
$$
h^{-1}(\tau) \cong ([0,1/2]/\sim)  \sqcup ([1/2,1]/\sim),
$$
where $0\sim 1/2$ and $1/2\sim 1$. Consider the natural diffeomorphisms
\beastar
\varphi_1 &:&  ([0, 1/2]/\sim) \to ([0,1]/\sim); t \mapsto 2t\\
\varphi_2 &:&  ([1/2,1]/\sim) \to  ([0,1]/\sim); t \mapsto 2t - 1.
\eeastar
Then we have the identity
\be\label{eq:H1dtglueH2dt} (H_1 *
H_2) \,dt = \varphi_i^*(H_i\, dt), \quad i =1, \,2.
\ee
This can be easily seen from the definition of $H_1 * H_2$.
\par
Hereafter in this section, we assume that $H_1,\, H_2, \, H_1 * H_2$ are all nondegenerate.
Let $[\gamma_1,w_1] \in {\rm Crit}(\mathcal A_{H_1})$,
$[\gamma_2,w_2] \in {\rm Crit}(\mathcal A_{H_2})$ and
$[\gamma_3,w_3] \in {\rm Crit}(\mathcal A_{H_1\# H_2})$.

\begin{defn}\label{pantsmoduli}
We denote by $\overset{\circ}{{\CM}}_{\ell}(H^\varphi,J^\varphi;
[\gamma_1;w_1],[\gamma_2;w_2],[\gamma_3,w_3])$
\index{$\mathcal M_{\ell}(H^\varphi,J^\varphi;
[\gamma_1;w_1],[\gamma_2;w_2],[\gamma_3,w_3])$} the set of all
pairs $(u;z_1^+,\dots,z_{\ell}^+)$ of maps $u: \Sigma \to M$
and marked points $z_i^+ \in \Sigma$ that satisfy the following conditions:
 \begin{enumerate}
 \item The map $\overline u = u\circ \varphi$ satisfies the equation:
\be\label{eq:HJCR11}
\frac{\partial\overline u}{\partial \tau} + J^\varphi\Big(\frac{\partial\overline u}
{\partial t} -X_{H^\varphi}(\overline u)\Big) = 0.
\ee
\item The energy
$$
E_{(H^\varphi,J^\varphi)}(u)= \frac{1}{2} \int \Big(\Big|\frac{\partial\overline u}{\partial \tau}\Big|^2_{J^\varphi} + \Big|
\frac{\partial\overline u}
{\partial t} -X_{H^{\varphi}}(\overline u)\Big|_{J^\varphi}^2 \Big)\, dt\, d\tau
$$
is finite.
\item It satisfies the following three asymptotic boundary conditions:
$$
\lim_{\tau\to +\infty}u(\varphi(\tau, t)) = \gamma(t).
$$
$$
\lim_{\tau\to - \infty}u(\varphi(\tau, t)) =
\begin{cases}
\gamma_1(2t)   &t \le 1/2, \\
\gamma_2(2t-1)   &t \ge 1/2.
\end{cases}
$$
\item The homotopy class of $(w_1 \sqcup w_2) \# u$ is $[w_3]$ in $\pi_2(\gamma_3)$.
Here $(w_1 \sqcup w_2) \# u$ is the obvious
concatenation of $w_1$, $w_2$ and $u$.
\item
$z_1^+,\dots,z_{\ell}^+$ are mutually distinct.
\end{enumerate}
\end{defn}

We denote by
\begin{equation}\label{ev=()}
{\rm ev} = ({\rm ev} _1,\dots,{\rm ev} _{\ell}) : \overset{\circ}{{\CM}}_{\ell}(H^\varphi,J^\varphi;
[\gamma_1,w_1],[\gamma_2,w_2],[\gamma,w]) \to M^{\ell}
\end{equation}
the evaluation map which associates the point
$(u(z_1^+),\dots,u(z_{\ell}^+))$ to $(u;z_1^+,\dots,z_{\ell}^+)$.

\begin{rem} One can write the equation \eqref{eq:HJCR11} in a more invariant fashion into the
coordinate independent form
$$
(du + P_{H^\varphi}(u))^{(0,1)}= 0
$$
where $P_{H^\varphi}$ is a $u^*(TM)$-valued one form on $\Sigma$
and the $(0,1)$-part is taken with respect to $j_\Sigma(y)$ on $T_y\Sigma$ and
$J^\varphi(u(y))$ on $T_{u(y)}M$ at each $y \in \Sigma$. In terms of
$\varphi$, the pull-back $\varphi^*(P_{H^\varphi})$ can be written as
$
\varphi^*(P_{H^\varphi}) = X_{H_i}\, dt, \quad i = 1, \, 2, \, 3
$
on the ends of $\Sigma$ near the punctures.
\end{rem}

Now we have the following proposition that provides basic properties of the
moduli space $\overset{\circ}{{\CM}}_{\ell}(H^\varphi,J^\varphi;
[\gamma_1,w_1],[\gamma_2,w_2],[\gamma_3,w_3])$.
\par
For any $\alpha_1,\alpha_2 \in \pi_2(M)$,
there exists a canonical homeomorphism
\begin{equation}\label{homeo1110}
\aligned
& \overset{\circ}{{\CM}}_{\ell}(H^\varphi,J^\varphi;
[\gamma_1,w_1],[\gamma_2,w_2],[\gamma,w]) \\
&\cong
 \overset{\circ}{{\CM}}_{\ell}(H^\varphi,J^\varphi;
[\gamma_1,\alpha_1 \# w_1],[\gamma_2,\alpha_2 \# w_2],[\gamma, \alpha \# w]).
\endaligned
\end{equation}
Here we put $\alpha = \alpha_1 + \alpha_2$.
\begin{prop}\label{pantsBULKkura}
\begin{enumerate}
\item
$\overset{\circ}{{\CM}}_{\ell}(H^\varphi,J^\varphi;
[\gamma_1,w_1],[\gamma_2,w_2],[\gamma_3,w_3])$ has a compactification
${{\CM}}_{\ell}(H^\varphi,J^\varphi;
[\gamma_1,w_1],[\gamma_2,w_2],[\gamma_3,w_3])$ that is Hausdorff.
\item
The space ${{\CM}}_{\ell}(H^\varphi,J^\varphi;
[\gamma_1,w_1],[\gamma_2,w_2],[\gamma_3,w_3])$ has an orientable Kuranishi structure with corners.
\item
The normalized boundary of ${{\CM}}_{\ell}(H^\varphi,J^\varphi;
[\gamma_1,w_1],[\gamma_2,w_2],[\gamma_3,w_3])$ is described
by union of the following three types of direct products:
\begin{equation}\label{bdrypantsBULKkura2}
{{\CM}}_{\#\mathbb L_1}(H_1,J_1;[\gamma_1,w_1],[\gamma'_1,w'_1])
\times {{\CM}}_{\#\mathbb L_2}(H^\varphi,J^\varphi;
[\gamma'_1,w'_1],[\gamma_2,w_2],[\gamma_3,w_3])
\end{equation}
where the union is taken over all $[\gamma'_1,w'_1] \in \text{\rm Crit}(H_1)$,
 and $(\mathbb L_1,\mathbb L_2) \in \text{\rm Shuff}(\ell)$.
\begin{equation}\label{bdrypantsBULKkura3}
{{\CM}}_{\#\mathbb L_1}(H_2,J_2
;[\gamma_2,w_2],[\gamma'_2,w'_2])
\times {{\CM}}_{\#\mathbb L_2}(H^\varphi,J^\varphi;
[\gamma_1,w_1],[\gamma'_2,w'_2],[\gamma_3,w_3])
\end{equation}
where the union is taken over all $[\gamma'_2,w'_2] \in \text{\rm Crit}(H_2)$,
 and $(\mathbb L_1,\mathbb L_2) \in \text{\rm Shuff}(\ell)$.
\begin{equation}\label{bdrypantsBULKkura4}
\aligned
&{{\CM}}_{\#\mathbb L_1}(H^\varphi,J^\varphi;
[\gamma_1,w_1],[\gamma_2;w_1],[\gamma'_3,w'_3])\\
&\times  {{\CM}}_{\#\mathbb L_2}(H_1 * H_2,J_1 * J_2;[\gamma_3',w_3'],[\gamma_3,w_3])
\endaligned
\end{equation}
where the union is taken over all $[\gamma_3',w_3'] \in \text{\rm Crit}(H_1 * H_2)$,
 and $(\mathbb L_1,\mathbb L_2) \in \text{\rm Shuff}(\ell)$.
\item
The (virtual) dimension is given by
\begin{equation}\label{dimensionboundaryb11}
\aligned
&\dim{{\CM}}_{\ell}(H^\varphi,J^\varphi;
[\gamma_1,w_1],[\gamma_2,w_2],[\gamma_3,w_3]) \\
&= \mu_{H_1 * H_2}([\gamma_3,w_3]) - \mu_{H_1}([\gamma_1,w_1])
- \mu_{H_2}([\gamma_2,w_2])  +2\ell - n
\endaligned
\end{equation}
where $\mu_H : \mbox{\rm Crit}(\CA_H) \to  \Z$
is the Conley-Zehnder index.
\item There exists a system of orientations on the
moduli spaces
$$
{{\CM}}_{\ell}(H^\varphi,J^\varphi;
[\gamma_1,w_1],[\gamma_2,w_2],[\gamma_3,w_3])
$$
such that the isomorphism
$(3)$ above is compatible with the orientations.
\item The map
${\rm ev}$ \eqref{ev=()} extends to a strongly continuous smooth map
$$
{{\CM}}_{\ell}(H^\varphi,J^\varphi;
[\gamma_1,w_1],[\gamma_2,w_2],[\gamma_3,w_3]) \to M^{\ell},
$$
which we denote also by ${\rm ev}$.
The system of the evaluation maps is compatible with the boundary
description $(3)$ above.
\item
The homeomorphism $(\ref{homeo1110})$ extends to the compactifications
and their Kuranishi structures are identified by the homeomorphism.
\end{enumerate}
\end{prop}
The Figures \ref{Figure5},\ref{Figure6},\ref{Figure7} below draw the elements
corresponding to (\ref{bdrypantsBULKkura2}),
(\ref{bdrypantsBULKkura3}), (\ref{bdrypantsBULKkura4}), respectively.
\begin{figure}[h]
\centering
\includegraphics[scale=0.3]
{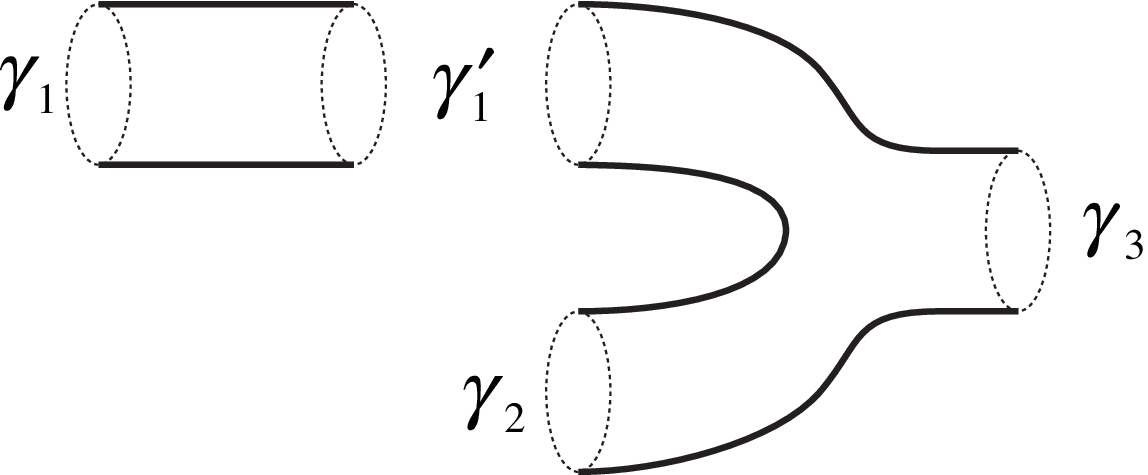}
\caption{An element of (\ref{bdrypantsBULKkura2})}
\label{Figure5}
\end{figure}
\begin{figure}[h]
\centering
\includegraphics[scale=0.3]
{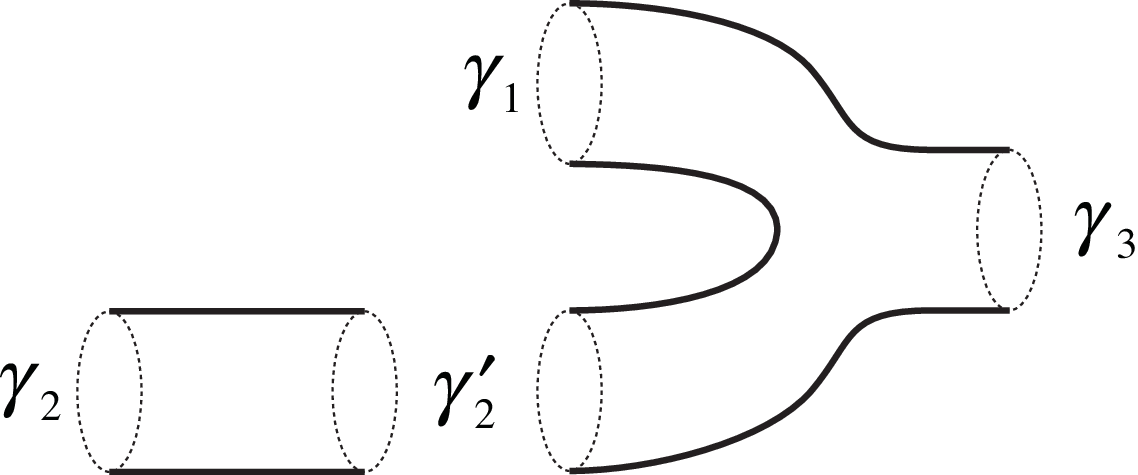}
\caption{An element of (\ref{bdrypantsBULKkura3})}
\label{Figure6}
\end{figure}
\begin{figure}[h]
\centering
\includegraphics[scale=0.3]
{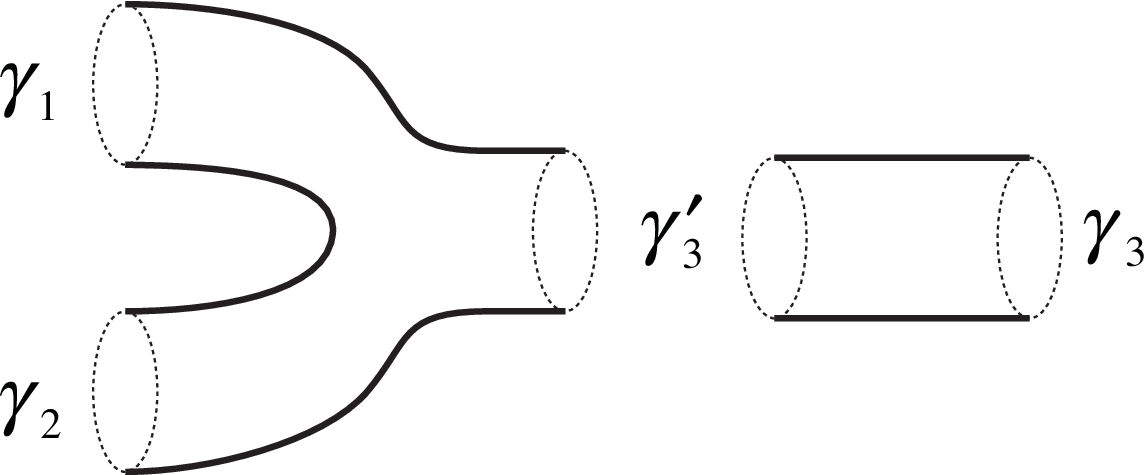}
\caption{An element of (\ref{bdrypantsBULKkura4})}
\label{Figure7}
\end{figure}
The proof of Proposition \ref{pantsBULKkura} is the same as that of Proposition \ref{connkura}
and so omitted.
(Note that the domain of an element of the moduli space
$\overset{\circ}{{\CM}}_{\ell}(H^\varphi,J^\varphi;
[\gamma_1,w_1],[\gamma_2,w_2],[\gamma_3,w_3])$ is
a sphere with three punctures, which is stable.  Therefore we do not need
to add marked points thereto, in order to transport the element of the
obstruction space to a nearby element. (See \cite[Definition 17.15]{fooo:techI}.)
The procedure of defining obstruction bundles
of the factors such as ${{\CM}}_{\#\mathbb L_1}(H_1,J_1;[\gamma_1,w_1],[\gamma'_1,w'_1])$
in (\ref{bdrypantsBULKkura2}) is determined in the course of
the proof of Proposition \ref{pantsBULKkura},
which is detailed in \cite[Section 31]{fooo:techI}. Therefore the rest of the proof is
the same as that of \cite[Part 4]{fooo:techI}.)
\par
Let
$\frak b \in H^{{\rm even}}(M;\Lambda_0 )$ which we decompose into
$\frak b = \frak b_0 + \frak b_2 + \frak b_+$ as in
(\ref{decompb}). We choose their representative closed forms respectively as before.
We then define the value
$\frak n^{\frak b}_{H^\varphi,J^\varphi;}([\gamma_1,w_1],[ \gamma_2,w_2],[\gamma_3,w_3])
\in \Lambda_0 $
by
\begin{equation}
\aligned
&\frak n^{\frak b}_{H^\varphi,J^\varphi;}([\gamma_1,w_1],[\gamma_2,w_2],[\gamma_3,w_3]) \\
&=
\sum_{\ell=0}^{\infty}
\frac{\exp(\int w_3^*\frak b_2 - \int w_2^*\frak b_2 - \int w_1^*\frak b_2)}{\ell !}\\
&
\quad\quad\quad
\int_{{{\CM}}_{\ell}(H^\varphi,J^\varphi;
[\gamma_1,w_1],[\gamma_2,w_2],[\gamma_3,w_3]) }
{\rm ev}_1^*\frak b_+ \wedge\dots \wedge {\rm ev}_\ell^*\frak b_+.
\endaligned
\end{equation}
We define a system of CF-perturbations on various
${{\CM}}_{\ell}(H^\varphi,J^\varphi;
[\gamma_1,w_1],[\gamma_2,w_2],[\gamma_3,w_3])$ that is
compatible with the description of their boundaries given in
Proposition \ref{pantsBULKkura} (3)
and use the system to define the integration in the right hand side.
\par
\begin{defn}\label{partialisderivative}
We put
\begin{equation}\label{multicldef111}\index{$\frak m_2^{\rm{cl}}$}
\aligned
&\frak m_2^{\rm{cl}}(\llb \gamma_1,w_1 \rrb \otimes \llb \gamma_2,w_2 \rrb) \\
&= \sum_{[\gamma_3,w_3]\in \mathrm{Crit}(\mathcal A_{H_1 * H_2})}
\frak n^{\frak b}_{H^\varphi,J^\varphi;}([\gamma_1,w_1],[\gamma_2,w_2],[\gamma_3,w_3])\llb \gamma_3,w_3 \rrb.
\endaligned
\end{equation}
\end{defn}

We can prove that the right hand side of (\ref{multicldef111}) converges
in $CF((H_1 * H_2,J_2);\Lambda^{\downarrow})$ in the same way as the proof of
Lemma \ref{adiccomv2}.
To define the right hand side of (\ref{multicldef111}) we suppose
that  $\llb \gamma_1,w_1 \rrb$ and $\llb \gamma_2,w_2 \rrb$
are represented by $[\gamma_1,w_1] \in \mathrm{Crit}(\mathcal A_{H_1})$ and $[\gamma_2,w_2]
\in \mathrm{Crit}(\mathcal A_{H_2})$, respectively, i.e., $\pi([\gamma_i,w_i]) = \llb \gamma_i,w_i \rrb$. The independence of the right hand side of (\ref{multicldef111})
of the the choice of the representatives follows from Proposition \ref{pantsBULKkura} (7).
We have thus defined (\ref{mproduct}).
\par
\begin{lem} The map $\frak m_2^{\rm{cl}}$ satisfies
$$
\partial_{(H_1 * H_2,J_1 * J_2)}^{\frak b} \circ \frak m_2^{\rm{cl}}
= \frak m_2^{\rm{cl}} \circ
\left(\partial_{(H_1,J_1)}^{\frak b} \widehat\otimes 1 + 1 \widehat\otimes  \partial_{(H_2,J_2)}^{\frak b}
\right).
$$
Here $\widehat{\otimes}$
\index{$\widehat\otimes$} is the graded tensor product of the linear maps, which is given by
$$
(F \widehat{\otimes} G )(x \otimes y) = (-1)^{\deg G\det' x}F(x) \otimes G(y).
$$
\end{lem}
\begin{proof}
This is a consequence of Proposition \ref{pantsBULKkura} (3),
Stokes' theorem (Theorem \ref{them48}) and
composition formula (Theorem \ref{compform}).  In fact, the three boundary types,
(\ref{bdrypantsBULKkura2}), (\ref{bdrypantsBULKkura3}), (\ref{bdrypantsBULKkura4}) correspond to
the three operations, $\frak m_2^{\rm{cl}} \circ \partial_{(H_1,J_1)}^{\frak b} \widehat\otimes 1$,
$\frak m_2^{\rm{cl}} \circ  \partial_{(H_2,J_2)}^{\frak b}$
and $\partial_{(H_1 * H_2,J_1 * J_2)}^{\frak b} \circ \frak m_2^{\rm{cl}}$,
respectively.
\end{proof}
Thus we have defined the pants product
\begin{equation}\label{mproducth}
\aligned
\frak m^{\rm{cl}}_2 : HF((H_1,J_1);\Lambda^{\downarrow}) &\otimes HF(M,H_2,J_2;\Lambda^{\downarrow}) \\
&\to HF(M,H_1 * H_2,J_1 * J_2;\Lambda^{\downarrow}).
\endaligned
\end{equation}
The next proposition shows that it respects the filtration.
\begin{prop}\label{filtpresm2} For all $\lambda_1, \, \lambda_2 \in \R$,
\beastar
&{}& \frak m_2^{\rm{cl}}
\left(
F^{\lambda_1} CF(M,H_1,J_1;\Lambda^{\downarrow}) \otimes F^{\lambda_2} CF(M,H_2,J_2;\Lambda^{\downarrow})
\right) \\
& \subseteq &
F^{\lambda_1+\lambda_2} CF(M,H_1 * H_2,J_1 * J_2;\Lambda^{\downarrow}).
\eeastar
\end{prop}
\begin{proof} We start with the following lemma
\begin{lem}\label{monolemma}
If
${{\CM}}_{\ell}(H^\varphi,J^\varphi;
[\gamma_1,w_1],[\gamma_2,w_2],[\gamma_3,w_3])$
is nonempty, then
$$
\mathcal A_{H_1}([\gamma_1,w_1])
+
\mathcal A_{H_2}([\gamma_2,w_2])
\ge
\mathcal A_{H_1 * H_2}([\gamma_3,w_3]).
$$
\end{lem}
\begin{proof}
Let $(u;z_1^+,\dots,z_k^+) \in \overset{\circ}{{\CM}}_{\ell}(H^\varphi,J^\varphi;
[\gamma_1,w_1],[\gamma_2,w_2],[\gamma_3,w_3])$ and $\tau_0 < 0$.
We identify
$$
h^{-1}(\tau_0) = S^1_1 \sqcup S^1_2.
$$
We denote the restriction of $u$ to $S^1_1$ by $\gamma_1^{\tau_0}$ and
the restriction of $u$ to $S^1_2$ by $\gamma_2^{\tau_0}$.
\par
We concatenate $w_1$ with $\cup_{\tau \le \tau_0} \gamma_1^{\tau_0}$
to obtain $w_1^{\tau_0}$ which bounds $\gamma_1^{\tau_0}$.
We define $w_2^{\tau_0}$ in the same way.
\par
By the same way as in Lemma \ref{filtered}, we derive
\begin{equation}\label{mono1}
\mathcal A_{H_1}([\gamma_1,w_1]) \ge \mathcal A_{H_1}([\gamma^{\tau_0}_1,w^{\tau_0}_1]),
\quad
\mathcal A_{H_2}([\gamma_2,w_2]) \ge \mathcal A_{H_2}([\gamma^{\tau_0}_2,w^{\tau_0}_2]).
\end{equation}
Next let $\tau_0 > 0$.
We denote the restriction of $u$ to $h^{-1}(\tau_0)$  by $\gamma^{\tau_0}$.
We concatenate $w_1 \sqcup w_2$ and the restriction of $u$ to
$\{z \in \Sigma \mid h(z) \le \tau_0\}$ to obtain $w^{\tau_0}$.
By the same way as in Lemma \ref{filtered}, we also derive
\begin{equation}\label{mono2}
\mathcal A_{H_1 * H_2}([\gamma^{\tau_0},w^{\tau_0}]) \ge \mathcal A_{H_1 * H_2}([\gamma_3,w_3]).
\end{equation}
It follows easily from definition that
\begin{equation}\label{mono3}
\lim_{\tau_0 \to 0}
\left(
\mathcal A_{H_1}([\gamma^{\tau_0}_1,w^{\tau_0}_1])
+ \mathcal A_{H_2}([\gamma^{\tau_0}_2,w^{\tau_0}_2])
\right)
=\lim_{\tau_0 \to 0}
\mathcal A_{H_1 * H_2}([\gamma^{\tau_0},w^{\tau_0}]).
\end{equation}
Lemma \ref{monolemma} follows easily from
(\ref{mono1}), (\ref{mono2}), (\ref{mono3}).
\end{proof}
Proposition \ref{filtpresm2} follows immediately from Lemma \ref{monolemma}.
\end{proof}

\subsection{Multiplicative property of Piunikhin isomorphism}
\label{subsec:multiplicative}
\index{Piunikhin isomorphism!multiplicative property|(}

In this subsection, we prove that the Piunikhin isomorphism interpolates
the quantum product $\cup_\frak b$ of $QH$ and the $\frak b$-deformed pants product
of $HF$.

Let $\chi : \R \to [0,1]$ be as in Definition \ref{defn:chi}.
For each $S  \in \R$, we define
$$
H_S^\varphi(z,x) = \chi(h(z)+S)H^\varphi(z,x)
$$
where $H^\varphi$ is as in (\ref{defsharp}).
\begin{figure}[h]
\centering
\includegraphics[scale=0.3]
{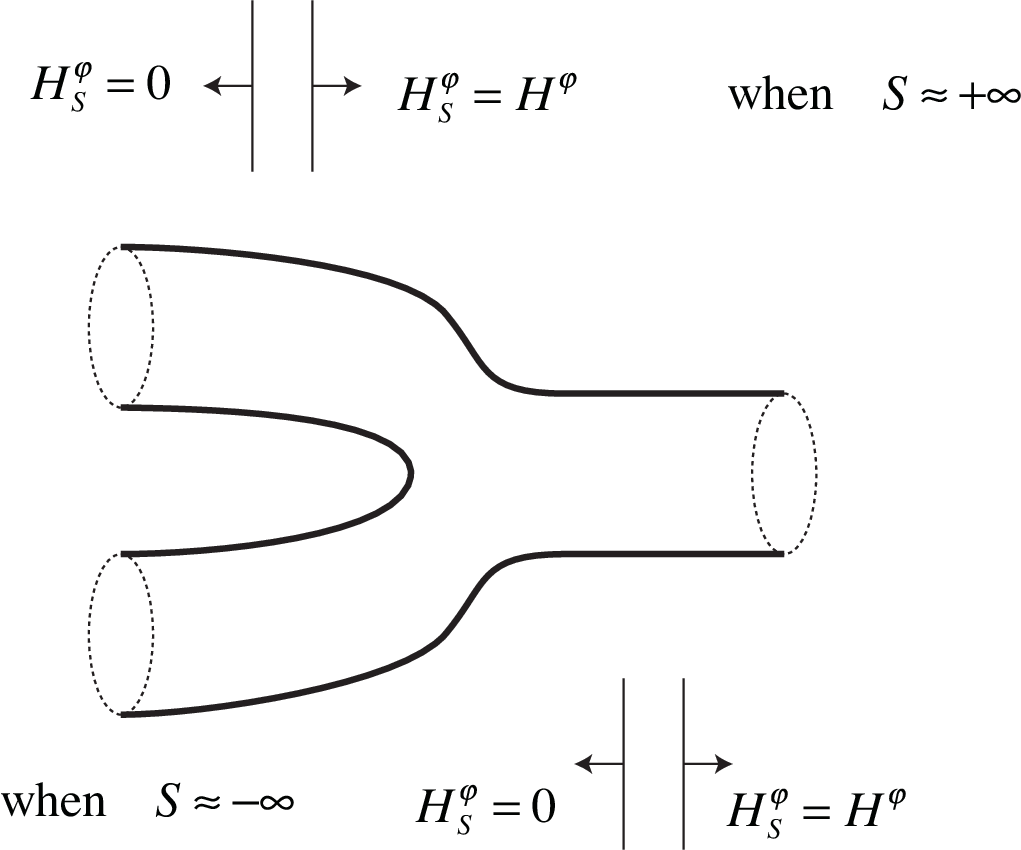}
\caption{Function $H_S^\varphi$}
\label{Figure8}
\end{figure}
Similarly we define a family $J_S^\varphi(z)$ so that
$$
J^{\varphi}_S(\varphi(\tau,t)) =
J^{\varphi(\tau,t)}(\varphi(\tau+S,t)).
$$
Due to the condition $J_t \equiv J_0$ near $t = 0$, this definition
smoothly extends to whole $\Sigma$.

With this preparation, we prove the following:

\begin{thm}\label{themprodcompati}
For $a_1,\, a_2 \in H^*(M;\Lambda) \setminus \{0\}$, we have
$$
\frak m_2^{\rm cl} (\mathcal P_{((H_1)_{\chi},(J_1)_{\chi}),\ast}(a_1^\flat),
\mathcal P_{((H_2)_{\chi},(J_2)_{\chi}),\ast}(a_2^\flat))
= \mathcal P_{((H_1 * H_2)_{\chi},(J_1 * J_2)_{\chi}),\ast} ((a_1 \cup^{\frak b} a_2)^\flat).
$$
\end{thm}
\begin{proof}
Let
$[\gamma,w] \in \mathrm{Crit}(\mathcal A_{H_1 * H_2})$.
\begin{defn}\label{pantsmoduliS}\index{$\mathcal M_{\ell}(H^\varphi_S,J^\varphi_S;**,
[\gamma,w])$}
We denote by $\overset{\circ}{{\CM}}_{\ell}(H^\varphi_S,J^\varphi_S;**,
[\gamma,w])$ the set of all pairs $(u;z_1^+,\dots,z_{\ell}^+)$ of maps
$u: \Sigma \to M$ and marked points $z_i^+ \in \Sigma$,
which satisfy the following conditions:
\begin{enumerate}
\item The map $\overline u: = u \circ \varphi$ satisfies
\be\label{eq:CRHJ12}
\frac{\del \overline u}{\del \tau} + J_S^\varphi\left(\frac{\del \overline u}{\del t}
- X_{H_S^\varphi}(\overline u)\right) = 0.
\ee
\item The energy
$$
E_{(H_S^\varphi,J_S^\varphi)} = \frac{1}{2}
\int \left(\Big|\frac{d\overline u}{d \tau}\Big|^2_{J_S^\varphi} + \Big|
\frac{\partial\overline u}
{\partial t} -X_{H_S^\varphi}(\overline u)\Big|_{J_S^\varphi}^2
\right) dt\, d\tau
$$
is finite.
\item It satisfies the following asymptotic boundary condition:
$$
\lim_{\tau\to +\infty}u(\varphi(\tau, t)) = \gamma(t).
$$
\item
Let $[u]$ be the homotopy class of $u$ in $\pi_2(\gamma)$. Then $[u] = [w]$.
\item
$z_1^+,\dots,z_{\ell}^+$ are mutually distinct.
\end{enumerate}
\end{defn}
Again here as in Remark \ref{rem:*} we put the $**$ in the notation of the moduli space  to highlight the
fact that the domain curve carries two negative punctures and one positive
puncture at which the asymptotic boundary condition $[\gamma,w]$ is assigned.
\par
We note that (\ref{eq:CRHJ12}) and the finiteness of energy imply that
there exist $p_1, p_2 \in M$ such that
\begin{equation}\label{p1p2def}
\lim_{\tau\to - \infty}u(\varphi(\tau, t)) =
\begin{cases}
p_1   &t < 1/2, \\
p_2   &t > 1/2.
\end{cases}
\end{equation}
Therefore the homotopy class of $u$ in $\pi_2(\gamma)$ is defined.
\par
We define the evaluation map
$$
{\rm ev}_{-\infty}
= ({\rm ev}_{-\infty,1},{\rm ev}_{-\infty,2})
: \overset{\circ}{{\CM}}_{\ell}(H_S^\varphi,J_S^\varphi;**,
[\gamma,w])
\to M^2
$$
by
$
{\rm ev}_{-\infty}(u) = (p_1,p_2)
$
where $p_1,p_2$ are as in (\ref{p1p2def}),
and
$$
{\rm ev} = ({\rm ev} _1,\dots,{\rm ev} _{\ell}) : \overset{\circ}{{\CM}}_{\ell}(H_S^\varphi,J_S^\varphi;**,
[\gamma,w]) \to M^{\ell}
$$
by
$
{\rm ev}(u;z_1^+,\dots,z_{\ell}^+)=(u(z_1^+),\dots,u(z_{\ell}^+)).
$

We put
\begin{equation}\label{eq:MMpara**}
\overset{\circ}{{\CM}}_{\ell}(para;H^\varphi,J^\varphi;**,
[\gamma,w]) =
\bigcup_{S \in \R}
\{S\} \times
\overset{\circ}{{\CM}}_{\ell}(H_S^\varphi,J_S^\varphi;**,
[\gamma,w]).
\end{equation}
\index{$\mathcal M_{\ell}(para;H^\varphi,J^\varphi;**,
[\gamma,w]) $}
The evaluation maps ${\rm ev}_{-\infty}$ and ${\rm ev}$ are defined on them in an obvious way.

\begin{prop}\label{pantsBULKkurapara}
\begin{enumerate}
\item
The moduli space $\overset{\circ}{{\CM}}_{\ell}(para;H^\varphi,J^\varphi;**,
[\gamma,w])$ has a compactification
${{\CM}}_{\ell}(para;H^\varphi,J^\varphi;**,
[\gamma,w])$ that is Hausdorff.
\item
The space ${{\CM}}_{\ell}(para;H^\varphi,J^\varphi;**,
[\gamma,w])$ has an orientable Kuranishi structure with boundary and corners.
\item
The normalized boundary of ${{\CM}}_{\ell}(para;H^\varphi,J^\varphi;**,
[\gamma,w])$ is described by
the union of following three types of direct or fiber products:
The first one is
\begin{equation}\label{bdrypantsBULKkurapara1}
{{\CM}}_{\#\mathbb L_1}(para;H^\varphi,J^\varphi;**,
[\gamma',w'])\times  {{\CM}}_{\#\mathbb L_2}(H_1 * H_2,J_1 *  J_2;[\gamma',w'],[\gamma,w])
\end{equation}
where the union is taken over all $[\gamma',w'] \in \text{\rm Crit}(H_1 * H_2)$
and $(\mathbb L_1,\mathbb L_2) \in \text{\rm Shuff}(\ell)$, which corresponds to
splitting-off a Floer trajectory at some finite parameters $S$.
(Figure $\ref{Figure9}$)
\par
The second one corresponds to $S = -\infty$ which is
\begin{equation}\label{bdrypantsBULKkurapara2}
{{\CM}}_{3+\#\mathbb L_1}^{\rm cl}(\alpha;J_0) \,\,
{}_{{\rm ev}_3}\times_{{\rm ev}_{-\infty}}  {{\CM}}_{\#\mathbb L_2}((H_1 * H_2)_{\chi},
(J_1 * J_2)_{\chi};*,
[\gamma';w']).
\end{equation}
Here ${{\CM}}_{3+\#\mathbb L_1}^{\rm cl}(\alpha;J_0)$ is
as defined in Section $\ref{sec:bigquantum}$.
The union is taken over all $(\mathbb L_1,\mathbb L_2) \in \text{\rm Shuff}(\ell)$
and $\alpha, w'$ such that the obvious concatenation $\alpha \# w'$ is homotopic
to $w$ the fiber product is taken over $M$.
(Figure $\ref{Figure10}$)
(See Definition $\ref{defnn325}$ for the definition of fiber product of Kuranishi structure.)
\par
The third type corresponds to $S = \infty$ which is
\begin{equation}\label{bdrypantsBULKkurapara3}
\aligned
\big( {{\CM}}_{\#\mathbb L_1}((H_1)_{\chi},(J_1)_{\chi};*,&[\gamma_1,w_1])
\times {{\CM}}_{\#\mathbb L_2}((H_2)_{\chi},(J_2)_{\chi};*,[\gamma_2,w_2])\big)\\
&\times {{\CM}}_{\#\mathbb L_3}(H^\varphi,J^\varphi;
[\gamma_1,w_1],[\gamma_2,w_2],[\gamma,w]).
\endaligned
\end{equation}
where the union is taken over all
 $(\mathbb L_1,\mathbb L_2,\mathbb L_3)$ the triple shuffle of
$\{1,\dots,\ell\}$, and
$[\gamma_1,w_1] \in\mbox{\rm Crit}(\CA_{H_1})$,
$[\gamma_2,w_2] \in\mbox{\rm Crit}(\CA_{H_2})$.
(Figure $\ref{Figure11}$)
\item
The (virtual) dimension is given by
\begin{equation}\label{dimensionboundaryb223}
\dim {{\CM}}_{\ell}(para;H^\varphi,J^\varphi;**,
[\gamma,w]) = \mu_{H_1 * H_2}([\gamma,w])  +2\ell + 1 + n.
\end{equation}
\item
We can define a system of orientations on the collection of moduli spaces
${{\CM}}_{\ell}(para;H^\varphi,J^\varphi;**,
[\gamma,w])$ that is compatible with the isomorphism $(3)$ above.
\item The map ${\rm ev}$ on extends to a strongly continuous smooth map
$$
{{\CM}}_{\ell}(para;H^\varphi,J^\varphi;**,
[\gamma,w]) \to M^{\ell},
$$
which we denote also by ${\rm ev}$.
The map is compatible with the boundary description $(3)$ thereof.
\item The map
${\rm ev}_{-\infty}$ defined above on $\overset{\circ}{{\CM}}_{\ell}(para;H^\varphi,J^\varphi;**,
[\gamma,w])$ extends also to a strongly continuous smooth map
$$
{{\CM}}_{\ell}(para;H^\varphi,J^\varphi;**,
[\gamma,w]) \to M^2,
$$
which we denote by ${\rm ev}_{-\infty}$.
The map is compatible with the boundary description $(3)$ thereof.
\end{enumerate}
\end{prop}
\begin{figure}[h]
\centering
\includegraphics[scale=0.3]
{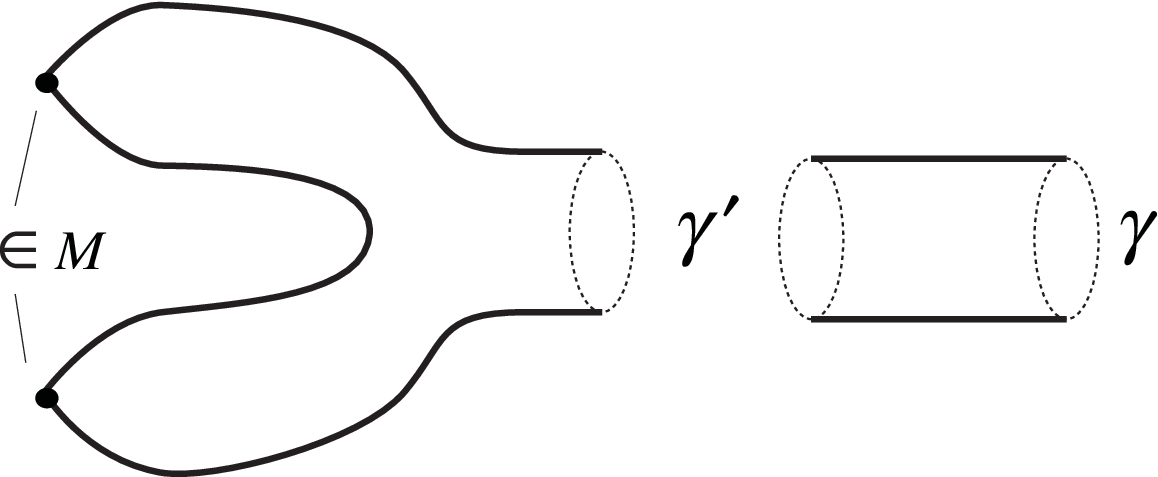}
\caption{An element of (\ref{bdrypantsBULKkurapara1})}
\label{Figure9}
\end{figure}
\begin{figure}[h]
\centering
\includegraphics[scale=0.3]
{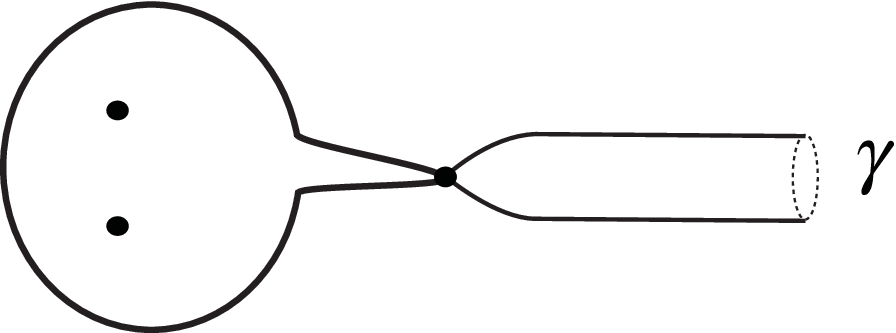}
\caption{An element of (\ref{bdrypantsBULKkurapara2})}
\label{Figure10}
\end{figure}
\begin{figure}[h]
\centering
\includegraphics[scale=0.3]
{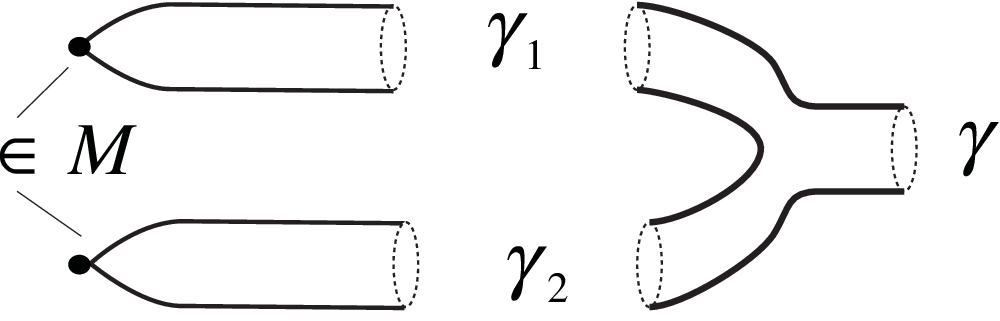}
\caption{An element of (\ref{bdrypantsBULKkurapara3})}
\label{Figure11}
\end{figure}
\begin{proof}
The proof is the same as other similar statements appearing in this and
several other previous papers, such as
\cite{fukaya-ono, fooo:book1, fooo:book2, fooo:techI}.
So it suffices to see how the boundary of our moduli space appears as in
(3).
\par
For each fixed $S$ the boundary of
${{\CM}}_{\ell}(H_S^\varphi,J_S^\varphi;**,
[\gamma,w])$ is described by
(\ref{bdrypantsBULKkurapara1}), with {\it para} being replaced by $S$.
We note that there is a `splitting end' where `bubble' occurs at $\tau \to \infty$. See Figure \ref{Figure9}.
\par
The case $S \to -\infty$ is described by (\ref{bdrypantsBULKkurapara2}).
We can prove the claim as follows.
We recall that $\lim_{S\to -\infty}(H_{S}^\varphi,J_{S}^\varphi) = (0,J_0)$
where $J_0$ is \emph{time-independent}.
We also remark that the moduli space
${{\CM}}_{3+\ell}^{\rm cl}(\alpha;J_0)$ is identified with the
moduli space of $(u;z^+_1,\dots,z^+_{\ell})$ such that the composition
$\overline u = u\circ \varphi$ satisfies the equation:
\be\label{eq:HJCR120}
\frac{\partial\overline u}{\partial \tau} + J_0\Big(\frac{\partial\overline u}
{\partial t} \Big) = 0
\ee
and
$\int u^*\omega< \infty$,
$[u] =\alpha$ in $H_2(M)$.
Therefore the `bubble' which slides to  $\tau \to -\infty$
is described by ${{\CM}}_{3+\ell}^{\rm cl}(\alpha;J_0)$.
(See Figure \ref{Figure10}.)
\par
The other potential `splitting end' where `bubble' occurs at $\tau \to -\infty$
has codimension two and do not appear here.
(This is because of the $S^1$ symmetry present on such a bubble.)
\par
Finally the case $S \to +\infty$ is described by (\ref{bdrypantsBULKkurapara3}).
(See Figure \ref{Figure11}.)
\end{proof}
To use ${{\CM}}_{\ell}(para; H^\varphi,J^\varphi;**,
[\gamma,w])$ to define an appropriate chain homotopy map we need to
find a perturbation on it that is compatible with the
description of its boundary given in Proposition \ref{pantsBULKkurapara} (3).
Since (\ref{bdrypantsBULKkurapara2}) involves fiber product we need
to find a perturbation so that $\mathrm{ev}_3$ is a submersion on the
perturbed moduli space. We need to use a continuous family of perturbations
(CF perturbation)
for this purpose.
The detail of the process of constructing such perturbations
is given in \cite[Section 12]{fooo:bulk}, \cite{fooo:tech2}, etc.
In \cite[Section 12]{fooo:bulk}  this notion is called
continuous family of multisections.  In \cite[Definition 7.47]{fooo:tech2}
where thorough detail of this construction is given, we call it a
\index{CF-perturbation}CF-perturbation.
In this paper we use the name of \cite{fooo:tech2}.
See Part \ref{part7} Definition \ref{defCFGCSFKUA}
 for the definition of CF-perturbation.
  \par
We regard ${{\CM}}_{3+\ell}^{\rm cl}(\alpha;J_0)$ as the
compactified moduli space of the pair
$(u;z^+_1,\dots,z^+_{\ell})$
satisfying (\ref{eq:HJCR120}) etc. Then we have a CF-perturbation
$\widetriangle{{\frak S}^{\epsilon}}$
on certain good coordinate system compatible with the
Kuranishi structure in Proposition \ref{pantsBULKkurapara}.
(See  Definition \ref{gcsystem} for the definition of
good coordinate system and  Definition \ref{defCFGCSFKUA}
 for the definition
of CF-perturbation.
Existence of a good coordinate system
is proved as Theorem \ref{thmexist}, \cite[Theorem 3.30]{fooo:tech2}
and existence of a CF-perturbation is proved as Theorem \ref{existperturbcont}, \cite[Theorem 7.49]{fooo:tech2}.)

Here we use the evaluation maps at the $1,2,4,\dots,(\ell + 3)$-th marked points
to receive an `input' and the evaluation map at the 3-rd marked point to give out
an `output'. This gives rise to a map
$
\mathrm{Corr}({{\CM}}_{3+\ell}^{\rm cl}(\alpha;J_0))
: \Omega(M)^{\otimes(2 +\ell)}
\to \Omega(M)
$
defined by
\index{$\text{\rm Corr}(\mathcal M_{3+\ell}^{\rm cl}(\alpha;J_0))$}
$$
\aligned
&\mathrm{Corr}({{\CM}}_{3+\ell}^{\rm cl}(\alpha;J_0))
(h_1,h_2,h_3,\dots,h_{\ell+2})\\
&= \mathrm{ev}_{3}!
\left( (\mathrm{ev}_{1}\times \mathrm{ev}_{2}\times
\mathrm{ev}_{4}\times  \dots
\times  \mathrm{ev}_{\ell+3})^*(h_1,h_2,h_3,\dots,h_{\ell+2});
\widetriangle{{\frak S}^{\epsilon}}\right),
\endaligned
$$
where $(\mathrm{ev}_{3})!$ is the integration along the fibers
 of ${{\CM}}_{3+\ell}^{\rm cl}(\alpha;J_0)$
under the evaluation map $\mathrm{ev}_{3} :
{{\CM}}_{3+\ell}^{\rm cl}(\alpha;J_0) \to M$,
which is defined via a choice of CF-perturbation $\widetriangle{{\frak S}^{\epsilon}}$.
See \cite[Definition 7.78]{fooo:tech2}  and Definition \ref{defintfiberGCS}.
The map depends on the choice of a CF-perturbation and $\epsilon$ on the chain level but
independent in the homology level. (See \cite[Proposition 8.15]{fooo:tech2}.)
We extend the above assignment
$\mathrm{Corr}({{\CM}}_{3+\ell}^{\rm cl}(\alpha;J_0))$ to a $\Lambda $-multi-linear map
$
\mathrm{Corr}({{\CM}}_{3+\ell}^{\rm cl}(\alpha;J_0)) : (\Omega(M) \widehat{\otimes} \Lambda)^{\otimes (2+\ell)}
\to \Omega(M) \widehat{\otimes} \Lambda
$
in the obvious way.
\par
Let $\frak b \in H^{{\rm even}}(M;\Lambda _0)$ and decompose
$\frak b = \frak b_0 + \frak b_2 + \frak b_+$ as in
(\ref{decompb}). We also regard them as differential forms
by the same kind of abuse of notations.
Let $a_1,a_2 \in \Omega(M)$. We put
\index{$\frak {gw}_{2;\alpha}^{\mathrm{cl}}$}
$$
\aligned
\frak {gw}_{2;\alpha}^{\mathrm{cl}}(a_1,a_2)  =
\sum_{\ell} \frac{\exp(\frak b_2\cap \alpha)}{\ell!}
\mathrm{Corr}({{\CM}}_{3+\ell}^{\mathrm cl}(\alpha;J_0))
(a_1,a_2,\frak b_{+},\dots,\frak b_{+}).
\endaligned
$$
We then define
\begin{equation}\label{chainlevelGW}
\frak {gw}_{2}(a_1,a_2) =
\sum_{\alpha} T^{\alpha \cap \omega} \frak {gw}_{2;\alpha}(a_1,a_2).
\end{equation}
We can easily prove that the right hand side of
(\ref{chainlevelGW}) converges in
$\Omega(M) \widehat\otimes \Lambda $.
Using the fact that Gromov-Witten invariant
is well-defined in the homology level
(this follows from the fact that
${{\CM}}_{3+\ell}^{\rm cl}(\alpha;J_0)$
has a Kuranishi structure {\it without boundary}),
we can easily show that $\frak {gw}_{2}$ induces a product map
$\cup^{\frak b}$ in the cohomology level.
\par
We now go back to the study of the moduli space
${{\CM}}_{\ell}(para;H^\varphi,J^\varphi;**, [\gamma,w])$.
We will need to construct a system of CF-perturbations thereon inductively over
the energy the explanation of which is in order.
We note that we have already defined CF-perturbations
of the moduli spaces which appear in the right hand
of Proposition \ref{pantsBULKkurapara} (3).
The fiber product in (\ref{bdrypantsBULKkurapara2}) induces a fiber product
of our CF-perturbations
in the sense of \cite[Definition 10.3]{fooo:tech2}, since $\mathrm{ev}_3$ is strongly
submersive
with respect to our CF-perturbation in the sense of Definition \ref{CFtransv}.

Other products appearing in (\ref{bdrypantsBULKkurapara1})
and (\ref{bdrypantsBULKkurapara3}) are direct products so
the perturbation of each of the factors immediately
induce one on the product.
Thus we have defined a CF-perturbation on the boundary.
It is compatible at the corners by the inductive construction of
CF-perturbations.
Therefore we can extend the given $\text{\rm CF}$-perturbation on the boundary
to the whole ${{\CM}}_{\ell}(para;H^\varphi,J^\varphi; **,[\gamma,w])$
by the general theory of Kuranishi structure.
We use this CF-perturbation to define the
integration on these moduli spaces given below.
\par
We now use our CF-perturbations to define an integration along
the fibers and put
$$
\aligned
&\frak n_{a_1,a_2,\frak b,para;H^\varphi,J^\varphi;
[\gamma,w ]} \\
&= \sum_{\ell}  \frac{\exp(\frak b_2\cap \alpha)}{\ell!}
\int_{{{\CM}}_{\ell}(para;H^\varphi,J^\varphi;
[\gamma,w])} \mathrm{ev}_{-\infty}^*(a_1,a_2)\wedge
\mathrm{ev}^*(\frak b_{+}\wedge \dots \wedge \frak b_{+}).
\endaligned$$
\index{$\frak H^{\frak b}_{H^\varphi,J^\varphi}(a_1,a_2)$}
\begin{defn}
$$
\frak H^{\frak b}_{H^\varphi,J^\varphi}(a_1,a_2)
=
\sum_{[\gamma,w] \in \mathrm{Crit}(\mathcal A_{H_1 * H_2})}\frak n_{a_1,a_2,\frak b,para;H^\varphi,J^\varphi;
[\gamma,w]}
\llb \gamma,w \rrb.
$$
\end{defn}
We can prove that the right hand side converges in $CF(M,H_1 * H_2,J_1 * J_2;\Lambda)$
in the same way as in the proof of convergence of the right hand side of
(\ref{boundaryop}).
\begin{lem}\label{compprodmainlemma}
We have
$$
\aligned
&\partial_{((H_1 * H_2)_{\chi},(J_1 * J_2)_{\chi})}^{\frak b}\circ \frak H^{\frak b}_{H^\varphi,J^\varphi}
+ \frak H^{\frak b}_{H^\varphi,J^\varphi} (\partial \widehat\otimes 1 + 1\widehat\otimes \partial )\\
&=
\mathcal P_{((H_1 * H_2)_{\chi};(J_1 * J_2)_{\chi})}^{\frak b} \circ \frak {gw}_2
- \frak m^{\rm cl}_2
\circ
\left(
\mathcal P_{((H_1)_{\chi};(J_1)_{\chi})}^{\frak b} \otimes
\mathcal P_{((H_2)_{\chi};(J_2)_{\chi})}^{\frak b}
\right).
\endaligned
$$
\end{lem}
\begin{proof}
Using the description of the moduli spaces given in Proposition \ref{pantsBULKkurapara},
we derive the lemma from Stokes' formula (Theorem \ref{them48},
\cite[Theorem 8.10]{fooo:tech2})
and composition formula (Theorem \ref{compform},
\cite[Theorem 10.20]{fooo:tech2}) similarly as before.
\end{proof}
Theorem \ref{themprodcompati} follows immediately from Lemma \ref{compprodmainlemma}.
\end{proof}
\index{Piunikhin isomorphism!multiplicative property|)}

\subsection{Wrap-up of the proof of triangle inequality}
\label{subsec:proofend}

Now we prove:

\begin{thm}\label{multiplicativemain}
Assume that $H_1, H_2, H_1 * H_2$ are all nondegenerate.
Then for any $a_1,a_2 \in QH^*_{\frak b}(M;\Lambda)$ we have:
$$
\rho^{\frak b}(H_1;a_1)
+
\rho^{\frak b}(H_2;a_2)
\ge \rho^{\frak b}(H_1 * H_2;a_1\cup^{\frak b}a_2).
$$
\end{thm}
\begin{proof}
Let $\epsilon >0$ and $\rho_i = \rho(H_i;a_i;\frak b)$.
Let $x_i \in F^{\rho_i+\epsilon}CF(M,H_i,J_i)$ such that
$\partial_{(H_i,J_i)}^{\frak b}(x_i) = 0$ and
$[x_i] = \mathcal P_{((H_i)_{\chi},(J_i)_{\chi}),\ast}^{\frak b}(a_i^\flat)
\in HF(M,H_i,J_i)$ ($i=1,2$).
\par
By Proposition \ref{filtpresm2} we have
$$
\frak m_2^{\rm cl}(x_1,x_2)
\in
F^{\rho_1+\rho_2+2\epsilon}CF(M,H_1 * H_2,J_1 * J_2).
$$
By Theorem \ref{themprodcompati} we have
$$
[\frak m_2^{\rm cl}(x_1,x_2)]
= \mathcal P_{((H_1 * H_2)_{\chi},(J_1 * J_2)_{\chi}),\ast}^{\frak b}((a_1 \cup^{\frak b} a_2)^\flat).
$$
Therefore by definition
$$
\rho(H_1 * H_2;a_1\cup^{\frak b}a_2;\frak b)
\le \rho_1+\rho_2+2\epsilon.
$$
Since $\epsilon >0$ is arbitrary,
Theorem \ref{multiplicativemain} follows.
\end{proof}

\section{Proofs of other axioms}
\label{sec:proofother}

We are now ready to complete the proof of Theorem \ref{axiomshbulk}.
\par
Note that the proof of Theorem \ref{homotopyinvbulk} (1),(3)
has been completed in Section \ref{sec:proofhomotopy} and hence
the invariant $\rho^{\frak b}(\widetilde{\phi};a)$
is well-defined for $\widetilde{\phi} \in \widetilde{\rm{Ham}}_{\rm nd}(M;\omega)$.
\par
For general $\widetilde{\psi}_H \in \widetilde{\rm{Ham}}(M;\omega)$,
not necessarily nondegenerate, we take nondegenerate $H_i$ which converges to $H$ in $C^0$-sense
and take the limit $\lim_{i\to \infty}\rho^{\frak b}(\widetilde{\psi}_{H_i};a)$. This
limit exists and is independent of $H_i$ by  Theorem \ref{contspect}.
We define this limit to be $\rho^\frak b(\widetilde{\phi};a)$ and
have thus defined $\rho^{\frak b}(\widetilde{\phi};a)$ in general.
We prove that it satisfies (1) - (7) of Theorem  \ref{axiomshbulk}.
\par
Statement (1) is Theorem \ref{bulkspectrality}.
\par
Statement (2) is immediate from definition.
\par
Now let us prove (3). In a way similar to the proof of Lemma \ref{connectinghomofilt},
we prove the following:
\begin{lem}
If $\mathcal M_{\ell}(H_\chi,J_\chi;*,[\gamma,w])$ is nonempty, then
$
\mathcal A_H([\gamma,w]) \le E^-(H).
$
\end{lem}
Therefore if $\lambda_H(a) < \lambda$ then
$$
\mathcal P_{(H_{\chi},J_{\chi})}^{\frak b}(a^\flat) \in
F^{\lambda +\Vert H_i\Vert_-}CF(M,H,J).
$$
It follows that
$$
\rho^{\frak b}(H;a) \le \lambda + E^-(H).
$$
\par
We apply this inequality to a sequence $H_i$ of Hamiltonians converging to $0$
such that $\widetilde{\psi}_{H_i}$ are nondegenerate. By taking the limit, we have
$$
\rho^{\frak b}(\underline 0;a) \le \lambda.
$$
Since this holds for any $\lambda >\frak v_q(a)$, $\rho^{\frak b}(\underline 0;a) \leq \frak v_q(a)$.
We refer to Proposition \ref{oppositeineq} in Section \ref{sec:appendix1}
for the proof of opposite inequality
\be\label{eq:rhogeH+}
\rho^{\frak b}(\underline 0;a) \ge \frak v_q(a).
\ee
\par
Statement (4) is immediate from construction.
\par
Statement (5) is Theorem \ref{multiplicativemain} in the nondegenerate case.
The general case then follows by an obvious limit argument.
\par
Statement (6) immediately follows from Theorem \ref{contspect}.
\par
Statement (7) is obvious from construction.
We have thus completed the proof of Theorem \ref{axiomshbulk}
except the opposite inequality \eqref{eq:rhogeH+} which
is deferred to Section  \ref{sec:appendix1}.
\qed

\part{Quasi-states and quasi-morphisms
via spectral invariants with bulk}
\label{part:quasihomo}

In this part, we show that Entov-Polterovich's theory can be
enhanced by involving spectral invariants with bulk, which we
have developed in Part 2. The generalization is rather straightforward requiring only a
small amount of new ideas. So a large portion of this part is actually a
review of the works by Entov-Polterovich and Usher
\cite{EP:morphism,EP:states,EP:rigid, ostrober2,usher:specnumber,usher:duality}.
(It seems, however, that the proof of Theorem \ref{dualitymain} below is not written
in detail to the level of generality that we provide here.)

\section{Partial symplectic quasi-states}
\label{sec:spec-displace}

We start this section by recalling the definition of Calabi homomorphism.
Let $H : [0,1] \times M \to \R$ be a time dependent
Hamiltonian and $\phi_H^t$ be the $t$ parameter family of
Hamiltonian diffeomorphisms induced by it.
We note that we do {\it not} assume that $H$ is normalized.
For an open proper subset $U \subset M$ we define
\begin{equation}
\Ham_U(M,\omega)
= \left\{\psi_H \in \Ham(M,\omega) \mid \, \text{$\supp H_t \subset U$ for any $t$} \right\}.
\end{equation}
We denote the universal covering space of
$\mathrm{Ham}_U(M,\omega)$ by $\widetilde{\mathrm{Ham}}_U(M,\omega)$.
Each time dependent Hamiltonian $H$ supported in $U$ determines an element
$\psi_H = \phi_H^1 \in \mathrm{Ham}_U(M,\omega)$,
together with its lifting $\widetilde \psi_H = [\phi_H]_U \in \widetilde{\mathrm{Ham}}_U(M,\omega)$.
Here $[\cdot]_U$ is the path homotopy class of $\phi_H$ in $\Ham_U(M,\omega)$.
We recall the following lemma due to \cite{calabi}, whose
proof we omit.
(See for example \cite[Theorem 4.2.7]{banyagabook}, \cite[p.328--p.329]{mcduffsalamonintro}.)
\begin{lem}\label{wdcalabi}
If $\supp H_t \subset U$ for all $t$, then the integral
$$
 \int_0^1 dt \int_M H_t \,\,\omega^n
$$
depends only on $\widetilde{\psi}_H \in \widetilde{\mathrm{Ham}}_U(M,\omega)$.
\end{lem}

\begin{defn}\label{calabihomo}
We define the homomorphism
$\operatorname{Cal}_U: \widetilde{\Ham}_U(M,\omega) \to \R$ by
$$
\operatorname{Cal}_U(\widetilde{\psi}_H)
= \int_0^1dt \int_M H_t\,\omega^n,
$$
which is called a {\it Calabi homomorphism}.\index{Calabi homomorphism}
We also put $\operatorname{Cal} = \operatorname{Cal}_M : \widetilde{\rm Ham}(M,\omega) \to \R$.
\index{$\text{\rm Cal}_U$}\index{$\text{\rm Cal}$}
\end{defn}
This is well-defined by Lemma \ref{wdcalabi}.

We next recall the notion of partial symplectic quasi-states
introduced by Entov-Polterovich \cite{EP:states}.
We say that a subset $U$ of $M$ is {\it displaceable}\index{displaceable subset}
(or Hamiltonian displaceable) if
there exists $\phi \in \text{\rm Ham}(M,\omega)$ such that
$\phi(U) \cap \overline U =\emptyset$.

\begin{defn}
[\cite{EP:states, EP:rigid}] \label{defn:zeta}\index{partial symplectic quasi-state} A \emph{partial
symplectic quasi-state} is defined to be a function $\zeta: C^0(M)
\to \R$ that satisfies the following properties:
\begin{enumerate}
\item {(Lipschitz continuity)} $|\zeta(F_1) - \zeta(F_2)| \leq \|F_1-F_2\|_{C^0}$.
\item  {(Semi-homogeneity)} $\zeta(\lambda F) = \lambda \zeta(F)$ for any
$F \in C^0(M)$ and $\lambda \in \R_{> 0}$.
\item {(Monotonicity)} $\zeta(F_1) \leq \zeta(F_2)$ for $F_1 \leq F_2$.
\item {(Additivity over constants and Normalization)}
$\zeta (F + \alpha) = \zeta(F) + \alpha$ for any continuous
function $F$ and any real number $\alpha$. In particular $\zeta(1) = 1$.
\item {(Partial additivity)} If two $F_1, \, F_2 \in C^\infty(M)$ satisfy
$\{F_1,F_2\} = 0$ and $\supp F_2$ is
displaceable, then $\zeta(F_1 + F_2) = \zeta(F_1)$.
\item  {(Symplectic invariance) }
$\zeta(F) = \zeta(F \circ \psi)$ for any
$\psi \in {\rm Symp}_0 (M,\omega)$.
\item {(Vanishing)} $\zeta(F) = 0$, provided $\supp F$ is displaceable.
\item{(Triangle inequality)} If $\{F_1, F_2 \}= 0$, $\zeta(F_1 + F_2) \geq \zeta(F_1) + \zeta(F_2)$.
\end{enumerate}
\end{defn}

The triangle inequality property is required in the definition in \cite{EP:rigid}, though
it is not in \cite{EP:states}.
The triangle inequality (8) is different from the one in \cite{EP:rigid} and
are adapted to our convention.  Namely, for a partial symplectic quasi-state $\zeta^{EP}$ in the
sense of Entov-Polterovich, $\zeta(H)=-\zeta^{EP}(-H)$ gives a partial symplectic quasi-state
in the sense of Definition  \ref{defn:zeta}.
We would like to point out that the above vanishing property (7) is actually
an immediate consequence of the axiom, partial additivity (5).
\par
The upshot of Entov-Polterovich's discovery is that the spectral invariant
function $H\mapsto \rho(H;1)$
naturally gives rise to an example of partial symplectic quasi-states,
which we denote by $\zeta_1$. In fact,
this \emph{spectral partial quasi-states}\index{partial symplectic quasi-state!spectral partial quasi-state}
is the only known example of
such partial symplectic quasi-states so far. We call any such partial
symplectic quasi-states constructed out of spectral invariants and
its bulk-deformed ones as a whole \emph{spectral partial quasi-states}.
The main result of the next section is to generalize Entov-Polterovich's construction
of spectral partial (symplectic) quasi-states by involving the spectral invariants with bulk.
\par
Recall that the Lie algebra of $\widetilde{\rm Ham}(M,\omega)$ or ${\rm Ham}(M,\omega)$
can be identified with $C^\infty(M)/\R \cong C^\infty(M)_0$, the set of
normalized autonomous Hamiltonian functions.
The functional $\zeta_1^\infty = \zeta_1|_{C^\infty(M)}$ is defined on
the central extension $C^\infty(M)$ of this Lie algebra.

In fact, $\zeta_1$ can be regarded as a `linearization' of another nonlinear
functional $\mu: \widetilde{\rm Ham}(M,\omega) \to \R$
which has the properties established in \cite{EP:states}
section 7. This becomes a genuine quasi-morphism
under a suitable algebraic condition such as semisimplicity of the quantum
cohomology ring of the underlying symplectic manifold $(M,\omega)$.
\par
We recall that the Hofer norm
$\Vert \widetilde\phi\Vert$ for $\widetilde\phi \in  \widetilde{\rm Ham}(M,\omega)$
is defined by
\begin{equation}\label{propdefmu}
\Vert \widetilde\phi\Vert
= \inf \left\{ \|H\| \mid [\phi_H] = \widetilde\phi
\right\}.
\end{equation}
Following \cite{EP:states},
we define another norm called the \emph{fragmentation norm}.
\index{fragmentation norm}\index{$\Vert \widetilde\phi\Vert_U$}
\begin{defn} Let $U$ be a given open subset of $M$. For any $\widetilde \phi \in \widetilde{\rm Ham}(M,\omega)$,
we say $\Vert \widetilde\phi\Vert_U \le m$ if and only if
there exists $\widetilde \psi_i \in \widetilde{\rm Ham}(M,\omega)$,
$\widetilde\phi_i \in \widetilde{\rm Ham}_U(M,\omega)$
for $i=1,\dots,m$
such that
$$
\widetilde\phi =
\prod_{i=1}^m (\widetilde \psi_i\widetilde\phi_i\widetilde \psi_i^{-1}).
$$
\end{defn}
The following fragmentation lemma of Banyaga \cite{banyagapaper} shows that
the norm $\Vert \widetilde\phi\Vert_U$ is always finite.
\begin{lem}[Banyaga]\label{fragmentation}
Let $U_i \subset M$ be open sets for $i=1,\dots,N$,
$U = \bigcup_{i=1}^N U_i$, and $\widetilde\phi \in \widetilde{\text{\rm Ham}}_U(M,\omega)$.
Then there exists $\widetilde\phi_j$ such that
$\widetilde\phi_j \in \text{\rm Ham}_{U_{i(j)}}(M,\omega)$ for some $i(j) \in \{1,\dots,N\}$
and
$$
\widetilde\phi = \widetilde\phi_1\dots \widetilde\phi_N.
$$
\end{lem}
\begin{proof}
We give a self contained proof below for the sake of completeness and for readers' convenience.
By an obvious induction argument it suffices to consider the case $N=2$, namely
$U = U_1 \cup U_2$. Let $\phi= \phi_H^1 \in \Ham_U(M,\omega)$.
We may assume without loss of generality that
$\widetilde\phi = \widetilde\psi_H$ and $\Vert H\Vert_{C^1} < \epsilon$,
where $\epsilon$ is a positive number depending only on $U_1,\,U_2$ and $U$ to be
determined later. (This is because any element of
$\widetilde{\text{\rm Ham}}_{U}(M,\omega)$
is a product of finitely many such $\widetilde\phi$'s.)
\par
We take a pair of  open subsets $U_1'' \subset U_1'$ so that
$\overline U_1'' \subset U_1' \subset \overline  U_1' \subset U_1$ and
$U_1'' \cup U_2 \supset \text{\rm supp}\,H$.
\par
Let $\eta : M \to [0,1]$ be a smooth cut-off function
such that $\supp \eta \subset U_1$ and that $\eta = 1$ on $U_1'$.
and  put $\phi_1 = \psi_{\eta\,H}$.
It is easy to see that if $\epsilon$ is sufficiently small then
$\phi_1 = \phi$ on $U''_1$, where
$\phi \in \text{\rm Ham}_U(M,\omega)$ is the projection
of $\widetilde\phi$.
Moreover we may assume that $\phi_1(x) = x$
for $x \notin U_1'' \cup U_2$.
\par
Therefore the support of $\phi_2 = \phi^{-1}_1\phi$ is on
$U_2$ and the support of $\phi_1$ is on $U_1$.
Using the fact that they are $C^1$-close to the identity, it follows
that $\widetilde \phi_1 \widetilde \phi_2 = \widetilde\phi$.
\end{proof}

\begin{defn}\label{defpscquasi-state}
A {\it partial quasi-morphism}\index{partial quasi-morphism}
on $\widetilde{\rm Ham}(M,\omega)$ is a function $\mu: \widetilde{\rm Ham}(M,\omega) \to \R$
that satisfies following conditions: Let $\widetilde\psi, \widetilde\phi \in \widetilde{\rm Ham}(M,\omega)$.
\begin{enumerate}
\item {(Lipschitz continuity)} $|\mu(\widetilde\psi) - \mu(\widetilde\phi)|
\leq C\Vert\widetilde\psi \widetilde\phi^{-1}\Vert$,
where $\Vert\,\,\Vert$ is the Hofer norm and $C$ is a constant independent of
$\widetilde\psi, \widetilde\phi$.
\item  {(Semi-homogeneity)} $\mu(\widetilde\phi^n) = n \mu(\widetilde\phi)$ for any
$n\in \Z_{\ge 0}$.
\item {(Controlled quasi-additivity)}
If $U \subset M$ is displaceable, then
there exists a constant $K$ depending only on $U$ such that
$$
\vert
\mu(\widetilde\psi \widetilde\phi)
- \mu(\widetilde\psi ) - \mu(\widetilde\phi)
\vert
< K
\min (\Vert \widetilde\psi\Vert_U, \Vert \widetilde\phi\Vert_U).
$$
\item  {(Symplectic invariance) }
$\mu(\widetilde\phi ) = \mu(\psi \circ \widetilde\phi \circ \psi^{-1})$ for all
$\widetilde\phi \in \widetilde{\rm Ham}(M,\omega)$ and
$\psi \in {\rm Symp}_0 (M,\omega)$.
\item {(Calabi Property)} If $U \subset M$ is displaceable, then
the restriction of $\mu$ to $\widetilde{\rm Ham}_U(M,\omega)$ coincides with
Calabi homomorphism $\operatorname{Cal}_U$.
\end{enumerate}
\end{defn}
\begin{rem}
\begin{enumerate}
\item The notion given in Definition 13.6 was introduced by Entov and Polterovich \cite{EP:states}
in the course of their construction of spectral partial quasi-states.  They did not name it.
In the earlier versions of the present paper, which was originally posted in the arXiv in 2011
(arXiv:1105.5123), we proposed to use the name `Entov-Polterovich pre-quasi-morphism'
for this function $\mu$. This name has been used by other researchers.
(See \cite{kawasaki}, \cite{wu-xu}, for example.)
Because of the referee's request, who said that the `same' notion was used in \cite{MVZ},
we change the name to the `partial quasi-morphism'.
\item
The name `partial quasi-morphism' is used in the paper \cite{MVZ}
by A. Monzner,  N. Vichery, F. Zapolsky, which was posted in
the arXiv in 2011 (arXiv:1111.0287). The same name was also used in Monzner's thesis \cite{Monzner}.
The  `partial quasi-morphism' used in \cite{MVZ, Monzner} is
a specialization of the notion given in Definition \ref{defpscquasi-state} to the cotangent bundle.
In the definition of `partial quasi-morphism' given
in \cite[Definition1.2]{Monzner} only the two conditions corresponding
to (3),(2) of Definition 13.6 are required but not others put above in Definition 13.6.
On the other hand, the `partial quasi-morphism' constructed
in \cite[Theorem 1.3]{MVZ} on the cotangent bundle, which is also duplicated as Theorem 3.5 in \cite{Monzner},
satisfies the conditions corresponding to those put in Definition 13.6.
More specifically, the conditions (1),(2),(3),(4),(5) of Definition 13.6
correspond to the ones (iii),(i),(v),(ii),(iv) in \cite[Theorem 1.3]{MVZ}, respectively.
On the cotangent bundle, \cite{MVZ,Monzner} shows that the partial quasi-morphism
satisfies the additional properties like (vi),(vii),(ix) in \cite[Theorem 1.3]{MVZ}.
It seems to be an interesting question to study how the properties (vi),(vii),(ix) in \cite[Theorem 1.3]{MVZ},
which are stated in terms of the presence of the zero section in cotangent bundle,
can be generalized to the general symplectic manifold beyond the cotangent bundle.
\end{enumerate}
\end{rem}

\section{Construction by spectral invariant with bulk}
\label{sec:constructspece}\index{spectral invariant!with bulk}

In this section we describe construction of an example of partial quasi-morphism
out of spectral invariants with bulk.
Let $\frak b \in H^*(M) \otimes \Lambda_0$
and  $0 \neq e \in QH^*_{\frak b}(M)$
satisfying
\begin{equation}\label{idenmp}
e \cup^{\frak b} e = e.
\end{equation}
An obvious example of such $e$ is $e = 1 \in  QH^*(M)$.
For given $\widetilde\psi_H \in \widetilde{\rm Ham}(M,\omega)$,
we consider the limit
\begin{equation}\label{homog}
\mu_e^{\frak b}(\widetilde\psi_H) =  \vol_\omega(M)\lim_{n\to +\infty}
\frac{\rho^{\frak b}((\widetilde\psi_H)^n;e)}{n}.
\end{equation}
We will later show that this limit always exists. In the mean time
we recall the relationship
$$
\rho^{\frak b}(\widetilde\psi_H;e) := \rho^{\frak b}(\underline H;e)
= \rho^{\frak b}(H;e) + \frac{1}{\vol_\omega(M)}\Cal(H)
$$
for any Hamiltonian $H$. In particular, the right hand side does not depend on
$H$ as long as $[\tilde \psi_H]$ remains the same element of
$\widetilde{\rm Ham}(M,\omega)$.

In particular, if $H$ is a {\it autonomous} Hamiltonian
and so $\psi_{nH} = (\psi_H)^n$, $\widetilde\psi_H^n = [\psi_{nH}]$, then
$\mu_e^{\frak b}(\widetilde\psi_H)$ can be also written as
\begin{equation}
\mu_e^{\frak b}(\widetilde\psi_H) = \vol_\omega(M)\lim_{n\to +\infty}
\frac{\rho^{\frak b}(nH;e)}{n} + \Cal(H).
\end{equation}
We define a (nonlinear) functional $\zeta_e^{\frak b}: C^0(M) \to \R$ by
\begin{equation}\label{homog2}
\zeta_e^{\frak b}(H) = - \lim_{n\to\infty} \frac{\rho^{\frak b}(nH;e)}{n}
\end{equation}
for $H \in C^\infty(M)$ and then extending to $C^0(M)$ by continuity.
Then for any $\widetilde \phi$ generated by \emph{autonomous} (smooth) Hamiltonian $H$,
whether it is normalized or not, we obtain the relationship
\be\label{eq:muzetae}
\frac{1}{\vol_\omega(M)} \mu_e^{\frak b}(\widetilde\psi_H) = - \zeta_e^{\frak b}(H) + \frac{1}{\vol_\omega(M)}\Cal(H).
\ee
If $H$ is normalized,
$\Cal(H)=0$ hence $\mu_e^{\frak b}(\widetilde\psi_H) =  - \vol_\omega(M) \zeta_e^{\frak b}(H)$.
\par
\begin{thm}\label{thm:state}
\begin{enumerate}
\item The limit $(\ref{homog})$ exists and so does the limit $(\ref{homog2})$.
\item $\mu_e^{\frak b}$ becomes a partial quasi-morphism
on $\widetilde{\rm Ham}(M,\omega)$.
\item $\zeta_e^{\frak b}$ becomes a partial symplectic quasi-state on $M$.
\end{enumerate}
\end{thm}
\begin{rem}
\begin{enumerate}
\item In case $\frak b = 0$, Theorem \ref{thm:state} is proved by
Entov-Polterovich \cite{EP:states}.
\item Actually in \cite{EP:states} several additional assumptions are imposed on
$(M,\omega)$. Those assumptions are now removed by Usher \cite{usher:specnumber,usher:duality}.
\item See also \cite{usher:talk} for some works related to the themes of the present paper.
\end{enumerate}
\end{rem}

We mostly follow the arguments presented in  \cite[pp.86-88]{EP:states} for the proof.
We give the proof of Theorem \ref{thm:state} in the following three subsections.

\subsection{Existence of the limit}

We begin with the following proposition.

\begin{prop}\label{calabiprop}
Let $U\subset M$ be an open set and $\phi :  M \to M$ a Hamiltonian diffeomorphism
such that $\phi(U) \cap \overline U = \emptyset$,
and $\widetilde{\phi} \in \widetilde{\text{\rm Ham}}(M)$ its lift.
Let $\widetilde \psi \in \widetilde{\text{\rm Ham}}_U(M,\omega)$, $\frak b \in
H^{{\rm even}}(M;\Lambda_0^{\downarrow})$ and $0 \neq a \in QH^*_{\frak b}(M)$.
Then
\begin{equation}\label{calabipropformula}
\rho^{\frak b}(\widetilde{\phi} \widetilde{\psi};a)
=
\rho^{\frak b}(\widetilde{\phi};a)
+ \frac{\mathrm{Cal}_U(\widetilde \psi)}{\vol_\omega(M)}.
\end{equation}
\end{prop}
\begin{proof}
The main idea of the proof of the proposition is due to
Ostrover \cite{ostrober}. It was used by Entov-Polterovich for the proof of
\cite[Lemma 7.2]{EP:states}, which we follow here.
\par
Let $F : [0,1]\times M \to \R$ be a normalized Hamiltonian
such that $[\phi_F] = \widetilde{\phi}$ and
let $H : [0,1]\times M \to \R$ be a Hamiltonian such that $\supp H_t$ is compact and
contained in $U$ for any $t$ and
that $\widetilde \psi =[\phi_H]$.
\par
By the assumption on $\phi$ and $\widetilde \psi$, we find that the fixed point set
$\mathrm{Fix}( \phi \circ \phi^t_H  )$ is independent of $t$.
We note
$\phi_{H^s * F}^1 =  {\phi} \circ {\psi}_{H^s}$,
where $H^s$ is the Hamiltonian generating the rescaled flow $t \mapsto \phi_H^{st}$ defined by
$$
H^s(t,x) = sH(st,x)
$$
and $*$ is the concatenation defined as in \eqref{eq:defconcat}.
Using the condition $\phi(U) \cap \overline U = \empty$, we find
$$\text{\rm Crit}({\mathcal A}_{H^0 * F})=\text{\rm
Crit}({\mathcal A}_{H^s * F}),
$$
which is the set consisting of $[\gamma, w]$ such that $\gamma(t)$
is the constant path at a fixed point of $\phi$
for $0 \leq t \leq 1/2$
and
$\gamma(2t-1)=\phi_H^{2t-1}(\gamma(1/2))$.
Note also that $\gamma(t)$ stays outside of ${\rm
supp}(H_t)$ for $0 \leq t \leq 1/2$.
This implies the following:

\begin{lem}\label{sigmainvAvalue}
For any $[\gamma,w] \in \text{\rm Crit}(\mathcal A_{H^{0} * F})$,
the number $\mathcal A_{H^s * F}([\gamma,w])$
is independent of $s$.
\end{lem}

We now consider the normalization of $H^s$
$$
\underline H^s(t,x)
= H^s(t,x) - \frac{s}{\vol_\omega(M)} \int_M H_{st}\omega^n.
$$
Then Lemma \ref{sigmainvAvalue} implies
$$
\aligned
\mathcal A_{\underline H^s * F}([\gamma,w])
&= \mathcal A_{H^s * F}([\gamma,w] ) +
\frac{s}{\vol_\omega(M)} \int_0^1 \int_M H_{st}\omega^n\,dt\\
&= \mathcal A_{F}([\gamma,w] ) + \frac{s}{\vol_\omega(M)} \int_0^1\int_M H_{st}\omega^n\,dt\\
&= \mathcal A_{F}([\gamma,w] ) + \frac{1}{\vol_\omega(M)} \int_0^s \int_M H_{r}\omega^n\,dr.
\endaligned
$$
Therefore we have
$$
\mathrm{Spec}(\widetilde \phi \circ [\phi_{H^s}];\mathfrak b)
= \mathrm{Spec}(\widetilde \phi;\mathfrak b)
+ \mathcal A_{F}([\gamma,w] ) + \frac{1}{\vol_\omega(M)} \int_0^s \int_M H_{r}\omega^n\,dr
$$
and so the function
$$
s \mapsto \rho^{\frak b}(\widetilde \phi \circ [\phi_{H^s}];a)
- \mathcal A_{F}([\gamma,w] ) + \frac{1}{\vol_\omega(M)} \int_0^s \int_M H_{r}\omega^n\,dr
$$
is continuous and takes values in the set
$\mathrm{Spec}(\widetilde \phi;\mathfrak b)$, which is a set of measure $0$
(see Corollary \ref{measure0b}.)
Therefore it must be constant. This finishes the proof of Proposition \ref{calabiprop}.
\end{proof}
Let $e$ and $\frak b$ be as in (\ref{idenmp}).
\begin{defn} \index{spectral displacement energy} Let $A$ be any displaceable closed subset of $M$. We define
the \emph{$\rho^{\frak b}_e$-spectral displacement energy} $\frak e(A;e;\frak b)$ by
\begin{equation}\label{eq:erhoA}
\frak e(A;e;\frak b) = \inf\{\rho^{\frak b}(\widetilde \phi;e)
+ \rho^{\frak b}(\widetilde \phi^{-1};e) \mid
\widetilde \phi
\in \widetilde{\mathrm{Ham}}(M,\omega),\,\,\widetilde \phi(A) \cap \overline A = \emptyset\}.
\end{equation}
\end{defn}
\begin{lem}\label{1310}
Let $U\subset M$  be an open set
which is Hamiltonian displaceable and
$\widetilde\psi \in \widetilde{\mathrm{Ham}}_U(M,\omega)$.
Then
\begin{equation}
\frak v_q(e) \le \rho^{\frak b}(\widetilde \psi;e) + \rho^{\frak b}(\widetilde \psi^{-1};e)
\le 2 \frak e(\overline U;e;\frak b). \label{14.6}
\end{equation}
\end{lem}
\begin{proof}
The following proof is the same as that of \cite[Lemma 7.2]{EP:states}.
First, Theorem \ref{axiomshbulk} (3), (5) and (\ref{idenmp}) imply
$$
\frak v_q(e)  = \rho( \underline{0};e;\frak b) \le
\rho^{\frak b}(\widetilde \psi^{-1};e) + \rho^{\frak b}(\widetilde \psi;e)
$$
which proves the first inequality of \eqref{14.6}.

Next, (\ref{calabipropformula}) implies
\begin{equation}\label{calabipropformula2}
\rho(\widetilde{\phi} \widetilde{\psi}^{-1};a;\frak b)
=
\rho^{\frak b}(\widetilde{\phi};a)
- \frac{\mathrm{Cal}_U(\widetilde \psi)}{\vol_\omega(M)}
\end{equation}
for any $\phi$ displacing $U$.

By the triangle inequality, we also have
\begin{equation}\label{formulatouseinlemma}
\aligned
\rho^{\frak b}(\widetilde \psi;e) &\le
 \rho^{\frak b}(\widetilde{\phi} \widetilde{\psi};e) +
\rho^{\frak b}(\widetilde{\phi}^{-1};e) \\
\rho^{\frak b}(\widetilde{\psi}^{-1};e) &\le
 \rho^{\frak b}(\widetilde{\phi} \widetilde{\psi}^{-1};e) +
\rho^{\frak b}(\widetilde{\phi}^{-1};e).
\endaligned
\end{equation}
Combining (\ref{calabipropformula}), (\ref{calabipropformula2})
and (\ref{formulatouseinlemma}) we derive
$$
\aligned
\frak v_q(e) &\le
\rho^{\frak b}(\widetilde \psi^{-1};e) + \rho^{\frak b}(\widetilde \psi;e)\\
&\le
 \rho^{\frak b}(\widetilde{\phi} \widetilde{\psi}^{-1};e) +
\rho^{\frak b}(\widetilde{\phi} \widetilde{\psi};e) +
2\rho^{\frak b}(\widetilde{\phi}^{-1};e)\\
&\le  2\rho^{\frak b}(\widetilde \phi;e)  +
2\rho^{\frak b}(\widetilde{\phi}^{-1};e).
\endaligned
$$
Since this holds for all $\widetilde \phi$ displacing $U$,
the second inequality of \eqref{14.6} follows.
\end{proof}

\begin{lem}\label{1311}
Suppose $U$ is displaceable and $\widetilde\psi
\in \widetilde{\Ham}_U(M,\omega)$.
Then for any $\widetilde\phi
\in \widetilde{\Ham}(M,\omega)$,
$$
\rho^{\frak b}(\widetilde \phi;e)
+ \rho^{\frak b}(\widetilde \psi;e)  - 2 \frak e(\overline U;e;\frak b)
\le
\rho^{\frak b}(\widetilde\phi \widetilde\psi ;e)\le
\rho^{\frak b}(\widetilde \phi;e)
+ \rho^{\frak b}(\widetilde \psi;e) .
$$
\end{lem}
\begin{proof}
The second inequality follows from Theorem \ref{axiomshbulk} (5) and (\ref{idenmp}).
The first inequality follows from
$$\aligned
\rho^{\frak b}(\widetilde\phi \widetilde\psi;e)
&\ge \rho^{\frak b}(\widetilde\phi;e) - \rho^{\frak b}(\widetilde \psi^{-1};e) \\
&\ge \rho^{\frak b}(\widetilde\phi;e) + \rho^{\frak b}(\widetilde \psi;e) - 2 \frak e(A;e;\frak b),
\endaligned$$
where the first inequality follows from Theorem \ref{axiomshbulk} (5) and the
second follows from Lemma \ref{1310}.
\end{proof}

\begin{cor}\label{cor14}
Let $\widetilde \psi_1, \dots, \widetilde \psi_m \in \widetilde{\text{\rm Ham}}(M,\omega)$
such that $\Vert \widetilde \psi_i \Vert_U = 1$  for $i=1,\dots,m$. Then
for any $\widetilde \phi \in \widetilde{\text{\rm Ham}}(M,\omega)$,
\begin{equation}\label{1314formula}
\vert
\rho^{\frak b}(\widetilde \psi_1\cdots \widetilde \psi_m\widetilde \phi;e)
- \sum_{i=1}^m \rho^{\frak b}(\widetilde \psi_i;e)
- \rho^{\frak b}(\widetilde \phi;e)
\vert
< 2m \frak e(\overline U;e;\frak b).
\end{equation}
\end{cor}
\begin{proof} By the hypothesis $\Vert \widetilde \psi_i \Vert_U = 1$, we can write
$\widetilde \psi_i = \widetilde \phi_i^{-1}\widetilde \psi'_i\widetilde \phi_i$
with $\widetilde \psi'_i \in \widetilde{\text{\rm Ham}}_U(M,\omega)$,
$\widetilde \phi_i \in \widetilde{\text{\rm Ham}}(M,\omega)$.
\par
The case $m=1$ follows from Lemma \ref{1311} which we apply
to $\psi_1(U)$ in place of $U$. (We note that
$\frak e(\overline U;e;\frak b) = \frak e(\psi_1(\overline U);e;\frak b)$.)
\par
Suppose the corollary is proved for $m-1$. Applying the
induction hypothesis to the case $m=2$, we have
$$
\vert
\rho^{\frak b}(\widetilde \psi_1\cdots \widetilde \psi_m\widetilde \phi;e)
- \rho^{\frak b}(\widetilde \psi_1;e)
- \rho^{\frak b}(\widetilde \psi_2\cdots \widetilde \psi_m\widetilde \phi;e)
\vert
< 2 \frak e(\overline U;e;\frak b)
$$
by Lemma \ref{1311}.
On the other hand, by the induction hypothesis we have
$$
\vert
\rho^{\frak b}(\widetilde \psi_2\cdots \widetilde \psi_m\widetilde \phi;e)
- \sum_{i=2}^m \rho^{\frak b}(\widetilde \psi_i;e)
- \rho^{\frak b}(\widetilde \phi;e)
\vert
< 2(m-1) \frak e(\overline U;e;\frak b).
$$
The inequality (\ref{1314formula}) follows.
\end{proof}

We now prove the convergence of the limit appearing in (\ref{homog}).

Let $\widetilde \psi \in \widetilde{\text{\rm Ham}}(M,\omega)$.
We fix a displaceable open subset $U \subset M$ and cover
$M = \cup_{i = 1}^m \phi_i(U)$ by $\phi_i \in Ham(M,\omega)$ for a sufficiently large $m$.
Then by applying Lemma \ref{fragmentation} to this decomposition of $M$,
we can factorize $\widetilde \psi = \widetilde \psi_1 \cdots \widetilde \psi_m$
so that $\supp \psi_i \subset \phi_i(U)$ for each $i =1, \cdots, m$.
In particular,
all $\widetilde \psi_i$ satisfy $\Vert \widetilde \psi_i\Vert_U = 1$.

We apply Corollary \ref{cor14} to these $\psi_i$'s by putting $\widetilde \phi = \overline 0$
and obtain
\begin{equation}\label{Eequality}
\vert
\rho^{\frak b}(\widetilde \psi^n;e)
- n\sum_{i=1}^m \rho^{\frak b}(\widetilde \psi_i;e)
\vert
\le 2mn \frak e(\overline U;e;\frak b).
\end{equation}
We put
$$
a_n = \rho^{\frak b}(\widetilde \psi^n;e)  + 2mn \frak e(\overline U;e;\frak b)
+ nm \vert\sup\{\rho^{\frak b}(\widetilde \psi_i;e) \mid i=1,\dots,m\}\vert.
$$
(\ref{Eequality}) implies that $a_n \ge 0$.
Theorem \ref{axiomshbulk} (5) implies $a_n + a_{n'} \ge a_{n+n'}$.
We recall the following:
\begin{lem}[M. Fekete \cite{fekete}]
If $a_n \ge 0$ and $a_n + a_{n'} \ge a_{n+n'}$,
then $\lim_{n\to \infty}a_n/n$ converges.
\end{lem}
\begin{proof}
The following proof is taken from
Problem 98 of \cite[ p 17]{polya}.
Since $a_{2^n}/{2^n}$ is non-increasing,
$\alpha = \liminf_{n\to \infty} a_n/n$ is a finite number.
Let $\epsilon >0$. We take any $n_0$ such that
$\vert a_{n_0}/n_0\vert \le \alpha + \epsilon$.
If $n' = n_0k + r$ with $r = 1,\dots,n_0-1$, then
$a_{n'} = a_{n_0k+r} \le ka_{n_0} + a_r$.
Therefore
$$
\frac{a_{n'}}{n'} \le \frac{a_{n_0}}{n_0} \frac{kn_0}{kn_0+r} + \frac{a_r}{n'}.
$$
Hence if $n'$ is sufficiently large, we have
$\alpha -\epsilon \le a_{n'}/{n'} < \alpha + 2\epsilon$
as required.
\end{proof}
We have thus proved that the limit
$$
\vol_\omega(M)\lim_{n\to +\infty} \frac{\rho^{\frak b}(\widetilde \phi^n;e)}{n}
$$
exists.

\subsection{Partial quasi-morphism property of $\mu_e^{\frak b}$}

In this subsection, we prove Theorem \ref{thm:state} (2).

The limit $\mu_e^{\frak b}(\widetilde \phi)$
satisfies Definition \ref{defpscquasi-state} (2) by construction.
Definition \ref{defpscquasi-state} (1) then follows from Theorem \ref{axiomshbulk} (6).
Definition \ref{defpscquasi-state} (4) follows from Theorem \ref{axiomshbulk} (4).
\par
We next prove the properties required in Definition \ref{defpscquasi-state} (3).
\begin{lem}\label{finallemmasquasi-state}
We have
\begin{equation}
\vert \mu_e^{\frak b}(\widetilde \psi\widetilde \phi)
- \mu_e^{\frak b}(\widetilde \psi)
- \mu_e^{\frak b}(\widetilde \phi)
\vert
\le 2\frak e(\overline U;e;\frak b)\vol_\omega(M)
\min (2\Vert \widetilde \psi\Vert_U - 1, 2\Vert \widetilde \phi\Vert_U - 1).
\end{equation}
\end{lem}
\begin{proof}
Without loss of any generality, we may assume that $\Vert \widetilde \psi\Vert_U
\le \Vert \widetilde \phi\Vert_U$.
The proof will be given by the induction over the fragmentation norm $\Vert \widetilde \psi\Vert_U$.
\par
We first consider the case $\Vert \widetilde \psi\Vert_U=1$.
Since $\Vert \widetilde \phi^j\widetilde \psi\widetilde \phi^{-j}\Vert_U = 1$, Corollary
\ref{cor14}, the identity
$$
(\widetilde \psi\widetilde \phi)^k = \left(\prod_{j=0}^{k-1} \widetilde \phi^j\widetilde \psi\widetilde \phi^{-j}\right)
\widetilde \phi^k
$$
and $\rho^{\frak b}(\widetilde \phi^j\widetilde \psi\widetilde \phi^{-j};e)
= \rho^{\frak b}(\widetilde \psi;e) $ (Theorem \ref{axiomshbulk} (4))
imply
$$
\vert \rho^{\frak b}((\widetilde \psi\widetilde \phi)^k;e)
- k\rho^{\frak b}(\widetilde \psi;e)  - \rho^{\frak b}(\widetilde \phi^k;e)
\vert
\le 2k \frak e(\overline U;e;\frak b).
$$
We use Corollary \ref{cor14} again to derive
$$
\vert\rho^{\frak b}(\widetilde \psi^k;e)
- k\rho^{\frak b}(\widetilde \psi;e)\vert
\le 2 k\frak e(\overline U;e;\frak b).
$$
Therefore
$$
\vert \rho^{\frak b}((\widetilde \psi\widetilde \phi)^k;e)
-\rho^{\frak b}(\widetilde \psi^k;e) - \rho^{\frak b}(\widetilde \phi^k;e)
\vert
\le 4k \frak e(\overline U;e;\frak b).
$$
The case $m=1$ of the lemma follows.
\par
Now suppose that the lemma is proved for all $\widetilde \psi \in \widetilde{\text{\rm Ham}}(M,\omega)$
that satisfy $\Vert \widetilde \psi\Vert_U = m-1$. Let
$\widetilde \psi$ be any element of $\widetilde{\text{\rm Ham}}(M,\omega)$ with $\Vert \widetilde \psi\Vert_U = m$.
We write $\widetilde \psi = \widetilde \psi_1\widetilde \psi_2$
with $\Vert \widetilde \psi_1\Vert_U = m-1$ and $\Vert \widetilde \psi_2\Vert_U =1$.
Then by the induction hypothesis
$$
\vert \mu_e^{\frak b}(\widetilde \psi\widetilde \phi)
- \mu_e^{\frak b}(\widetilde \psi_1)
- \mu_e^{\frak b}(\widetilde \psi_2\widetilde \phi)
\vert
\le 2\frak e(\overline U;e;\frak b)\vol_\omega(M)(2(m-1) - 1).
$$
The case $m=1$ gives
$$
\vert \mu_e^{\frak b}(\widetilde \psi_2\widetilde \phi)
- \mu_e^{\frak b}(\widetilde \psi_2)
- \mu_e^{\frak b}(\widetilde \phi)
\vert
\le  2\frak e(\overline U;e;\frak b)\vol_\omega(M)
$$
and
$$
\vert \mu_e^{\frak b}(\widetilde \psi)
- \mu_e^{\frak b}(\widetilde \psi_1)
- \mu_e^{\frak b}(\widetilde \psi_2)
\vert
\le  2\frak e(\overline U;e;\frak b)\vol_\omega(M).
$$
By combining these three inequalities, we have finished the proof of
Lemma \ref{finallemmasquasi-state}.
\end{proof}

Clearly Lemma \ref{finallemmasquasi-state} implies
\begin{equation}
\vert \mu_e^{\frak b}(\widetilde \psi\widetilde \phi)
- \mu_e^{\frak b}(\widetilde \psi)
- \mu_e^{\frak b}(\widetilde \phi)
\vert
\le 4\frak e(\overline U;e;\frak b)\vol_\omega(M)
\min (\Vert \widetilde \psi\Vert_U , \Vert \widetilde \phi\Vert_U ).
\end{equation}
Thus we have proved the property of Definition \ref{defpscquasi-state} (3).
\begin{rem}
We may take $K=4\frak e(\overline U;e;\frak b)\vol_\omega(M)$ for the constant in
Definition \ref{defpscquasi-state} (3).
\end{rem}
We next prove
Definition \ref{defpscquasi-state} (5).
Let $U \subset M$ be a displaceable open subset and
$\widetilde \psi \in \widetilde{\text{\rm Ham}}_U(M,\omega)$.
Let $\widetilde \phi  \in \widetilde{\text{\rm Ham}}(M,\omega)$
such that $\phi(U) \cap  \overline U = \emptyset$.
By Proposition \ref{calabiprop} applied to $\widetilde \psi^n$ we have
$$
\rho^{\frak b}(\widetilde \phi\widetilde \psi^n;e)
=
\rho^{\frak b}(\widetilde \phi;e) + \frac{n\text{\rm Cal}_U(\widetilde \psi)}{\vol_\omega(M)}.
$$
Using this equality and Lemma \ref{1311}, we obtain
\beastar
\left\vert
\rho^{\frak b}(\widetilde \psi^n;e) - \frac{n\text{\rm Cal}_U(\widetilde \psi)}{\vol_\omega(M)}
\right\vert
& = &\vert \rho^{\frak b}(\widetilde \psi^n;e) + \rho^{\frak b}(\widetilde \phi;e)
-\rho^{\frak b}(\widetilde \phi\widetilde \psi^n;e)\vert \\
& \le & 2 \frak e(\overline U;e;\frak b) < \infty.
\eeastar
Then dividing this inequality by $\frac{n\text{\rm Cal}_U(\widetilde \psi)}{\vol_\omega(M)} $
and letting $n \to \infty$,
we obtain $\mu^{\frak b}(\widetilde \psi) = \Cal_U(\widetilde \psi)$.
The proof of Theorem \ref{thm:state} (2) is complete.

\subsection{Partial symplectic quasi-state property of $\zeta^{\frak b}_e$}

In this subsection, we give the proof of  Theorem \ref{thm:state} (3), i.e.,
the functional $\zeta_e^{\frak b}: C^0(M) \to \R$
is a partial symplectic quasi-state.
For this purpose, we have only to consider autonomous
smooth Hamiltonian $F$'s in the rest of the proof.
Let $F$ be an autonomous Hamiltonian and take its normalization
\begin{equation}\label{defofFprime}
\underline F = F - \frac{1}{\vol_\omega(M)}\int_M F \, \omega^n.
\end{equation}
Then
\begin{equation}\label{nonhomoezure}
\rho^{\frak b}(nF;e) + \int_M nF \, \omega^n = \rho^{\frak b}(n\underline F;e) = \rho^{\frak b}(\widetilde \psi^n;e)
\end{equation}
for $\widetilde \psi = [\phi_F]$. Dividing this equation by $n$, we obtain
$$
\frac{\rho^{\frak b}(nF;e)}{n} + \frac{1}{\vol_\omega(M)}\int_M F \, \omega^n
= \frac{\rho^{\frak b}(\widetilde \psi^n;e)}{n}.
$$
Therefore convergence of (\ref{homog2}) follows from the
convergence of (\ref{homog}).
Thus $\zeta_e^{\frak b}(F)$ is defined for $F \in C^{\infty}(M)$.
\par
Definition \ref{defn:zeta} (1) is a consequence of
Theorem \ref{axiomshbulk} (6).
We can extend $\zeta_e^{\frak b}$ to $C^{0}(M)$
by the $F\in C^{\infty}(M)$ case of Definition \ref{defn:zeta} (1). The property of
Definition \ref{defn:zeta} (1) in the case $F\in C^{0}(M)$  then follows for this extended $\zeta_e^{\frak b}$.
\par
Since $\widetilde \psi_{H/m}^m = \widetilde \psi_{H}$ holds for autonomous Hamiltonian $H$,
we can prove the property of Definition \ref{defn:zeta} (2) in the case
$\lambda \in \Q_{\ge 0}$ by using Definition \ref{defpscquasi-state} (2)
and  (\ref{nonhomoezure}).
Then the case $\lambda \in \R_{\ge 0}$ follows from Definition \ref{defn:zeta} (1).
\par
Definition \ref{defn:zeta} (4) is immediate from (\ref{nonhomoezure}).
\par
The property of Definition \ref{defn:zeta} (6) is a consequence of Theorem \ref{axiomshbulk}  (4).
\par
Let us prove the property of Definition \ref{defn:zeta} (7).
Suppose $U$ is displaceable and the support of time independent Hamiltonian $F$ is in $U$.
We define $U'$ as in (\ref{defofFprime}).
We take $\widetilde \phi \in \widetilde{\text{\rm Ham}}(M,\omega)$
such that $\phi(U) \cap \overline U = \emptyset$.
By Proposition \ref{calabiprop}, we have:
\be\label{eq:nCalU}
\rho^{\frak b}(\widetilde \phi \widetilde \psi_F^n;e)
=
\rho^{\frak b}(\widetilde \phi;e)
+ \frac{n\text{\rm Cal}_U(\widetilde \psi_F)}{\vol_\omega(M)}.
\ee
(Here we use the fact that $\text{\rm Cal}_U$ is a homomorphism.)
\par
By (\ref{1314formula}) we also have
$$
\vert
\rho^{\frak b}(\widetilde \phi \widetilde \psi_F^n;e)
-
\rho^{\frak b}(\widetilde \phi;e)
-
\rho^{\frak b}(\widetilde \psi_F^n;e)
\vert
<
2\frak e(\overline U;e;\frak b).
$$
Substituting \eqref{eq:nCalU} into this inequality, and then dividing by
$n$ and taking the limit, we obtain
$$
\lim_{n\to\infty}
\frac{\rho^{\frak b}(\widetilde \psi_F^n;e)}{n}
=  \frac{\text{\rm Cal}_U(\widetilde \psi_F)}{\vol_\omega(M)}
= \frac{1}{\vol_\omega(M)}
\int_M F\omega^n.
$$
On the other hand, we have
$$
\lim_{n\to\infty}
\frac{\rho^{\frak b}(\widetilde \psi_F^n;e)}{n} =
\lim_{n\to\infty}
\frac{\rho^{\frak b}(n\underline F;e)}{n} =
- \zeta^{\frak b}_e(\underline F)
$$
and hence
$$
\zeta_e^{\frak b}(F) = \zeta^{\frak b}_e(\underline F) +
\frac{1}{\vol_\omega(M)} \int_M F\omega^n= 0.
$$
\par
We next prove the property of Definition \ref{defn:zeta} (3).
Let $F_1 \le F_2$. We put $H= F_1$ and $H' =F_2$ and apply the
argument of the proof of Theorem \ref{Jindepedence} and obtain
a chain map
$$
\mathcal P_{(F^\chi;J^\chi),\#}^{\frak b}
:
(CF(M;F_1;\Lambda^{\downarrow}),\partial_{(F_1,J_1)}^{\frak b})
\to
(CF(M;F_2;\Lambda^{\downarrow}),\partial_{(F_2,J_2)}^{\frak b}).
$$
Using $F_1 < F_2$ and Lemma \ref{connectinghomofilt}
we have
$$
\mathcal P_{(F^\chi;J^\chi),\#}^{\frak b}(F^{\lambda}CF(M;F_1;\Lambda^{\downarrow}))
\subset
F^{\lambda}CF(M;F_2;\Lambda^{\downarrow})).
$$
Let $x \in F^{\lambda}CF(M;F_1;\Lambda^{\downarrow})$
such that
$[x] = [\mathcal P_{((F_1)_\chi,(J_1)_\chi),\#}^{\frak b}(e^\flat)]$ and $\vert\lambda - \rho(F_1;e) \vert < \epsilon$.
Then by Proposition \ref{compaticompositte} we have
$[\mathcal P_{(F^\chi,J^\chi),\#}^{\frak b}(x)]
=  [\mathcal P_{((F_2)_\chi,(J_2)_\chi),\#}^{\frak b}(e^\flat)]$ and
$
\mathcal P_{(F^\chi,J^\chi),\#}^{\frak b}(x)
\in F^{\lambda}CF(M;F_2;\Lambda^{\downarrow})).
$
Therefore
$
\rho^{\frak b}(F_2;e) \le \rho^{\frak b}(F_1;e) + \epsilon.
$
It implies
$
\zeta_e^{\frak b}(F_1) \le \zeta_e^{\frak b}(F_2),
$
as required.
\par
Next we prove the property of Definition \ref{defn:zeta} (5).
By the assumption $\{F_1,F_2\} = 0$ we have
$$
\widetilde \psi_{F_1}\widetilde \phi_{F_2}
=  \widetilde \psi_{F_2}\widetilde \phi_{F_1}
= \widetilde \psi_{F_1+F_2}.
$$
Therefore by Definition \ref{defpscquasi-state} (3) we have
$$
\aligned
&\vert
\rho^{\frak b}((\widetilde \psi_{F_1}\widetilde \psi_{F_2})^n;e)
-
\rho^{\frak b}((\widetilde \psi_{F_1})^n;e)
-
\rho^{\frak b}((\widetilde \psi_{F_2})^n;e)
\vert\\
&=
\vert
\rho^{\frak b}((\widetilde \psi_{F_1})^n(\widetilde \psi_{F_2})^n;e)
-
\rho^{\frak b}((\widetilde \psi_{F_1})^n;e)
-
\rho^{\frak b}((\widetilde \psi_{F_2})^n;e)
\vert
\le K\Vert (\widetilde \psi_{F_2})^n\Vert_U = K.
\endaligned
$$
Here $U$ is a displaceable open set containing the support of $F_2$.
Therefore we have
$$
\mu_e^{\frak b}(\widetilde \psi_{F_1}\widetilde \psi_{F_2})
=\mu_e^{\frak b}(\widetilde \psi_{F_1})
+ \mu_e^{\frak b}(\widetilde \psi_{F_2})
=
\mu_e^{\frak b}(\widetilde \psi_{F_1})
+ \text{\rm Cal}_U(F_2).
$$
(We use Definition \ref{defpscquasi-state} (5) in the second equality.)
$$
\zeta_e^{\frak b}(F_1+F_2)
=
\zeta_e^{\frak b}(F_1)
$$
is now a consequence of (\ref{homog2}).
The triangle inequality for $\zeta_e^{\frak b}$ follows the triangle inequality for
the spectral invariant $\rho^{\frak b}$, since $\zeta_e^{\frak b}(F)= - \rho^{\frak b}(F;e)$.
The proof of Theorem \ref{thm:state} is now complete.
\qed

\section{Poincar\'e duality and spectral invariant}
\label{sec:duality}

\subsection{Statement of the result}
\label{subsec:dualitystatement}

Let $\pi : \Lambda^{\downarrow} \to \C$ be the
projection to $\C \subset \Lambda^\downarrow$. Denote by
$
\langle \cdot , \cdot\rangle : \Omega(M) \otimes \Omega(M) \to \C
$
the Poincar\'e duality pairing
$
\langle h_1,h_2\rangle = \int_M h_1\wedge h_2.
$
We extend the pairing to
$$
\langle \cdot , \cdot\rangle :( \Omega(M) \widehat\otimes\Lambda^{\downarrow}) \otimes (\Omega(M)
\widehat\otimes\Lambda^{\downarrow}) \to \Lambda^{\downarrow}
$$
so that it becomes $\Lambda^{\downarrow}$-bilinear.
We put
$
\Pi(a,b) = \pi(\langle a, b\rangle)
$
which induces a $\C$-bilinear pairing
$$
\Pi : H(M;\Lambda^{\downarrow}) \otimes H(M;\Lambda^{\downarrow}) \to \C.
$$
\par
The main result of this section is:
\begin{thm}\label{dualitymain}
Let $\frak b \in H(M;\Lambda^{\downarrow}_0)$, $a \in QH_{\frak b}^*(M)$,
and $\widetilde\phi \in \widetilde{\mathrm{Ham}}(M,\omega)$. Then we have
\begin{equation}
\rho^{\frak b}(\widetilde\phi; a)
= - \inf_b\{\rho^{\frak b}(\widetilde\phi^{-1};b)
\mid \Pi(a,b) \ne 0\}.
\end{equation}
\end{thm}
\begin{rem}
For the case $\frak b = 0$, this theorem is due to Entov-Polterovich
under the monotonicity assumption. (See \cite[Lemma 2.2]{EP:morphism}.)
The assumptions on $M$ which \cite{EP:morphism} imposed are removed and
Theorem \ref{dualitymain} itself is proved by Usher in \cite{usher:talk}.
\end{rem}

\subsection{Algebraic preliminary}
\label{subsec:dualityalg}

In this section we prove some algebraic lemmas used in the proof of
Theorem  \ref{dualitymain}.
A similar discussion was given by Usher in \cite{usher:duality}.
\par
We work in the situation of Subsections \ref{subsec:Usher}.
We put $G=\R$ in this subsection.
Namely $C(G) = C(G') = C$.
Note in this case we may take the basis $e_i$ such that
$\lambda_q(e_i) =0$.
Let $\partial : C \to C$ be a boundary operator.
We choose the standard basis $e'_i,e''_i, e'''_i$ as in Subsection \ref{subsec:Usher}.
Let $D$ be another finite dimensional $\Lambda^{\downarrow}$ vector space.
We assume that there exists a $\Lambda^{\downarrow}$ bilinear
pairing
$$
\langle \cdot,\cdot\rangle  :
C \times D \to \Lambda^{\downarrow}
$$
that is perfect. (Namely it induces an isomorphism
$C \to D^*$ to the dual space $D^*$ of $D$.)
Let $\{e_{i}^* \mid i=1,\dots, N\}$ be the dual basis of $\{e_{i} \mid i=1,\dots, N\}$.
We use it to define the filtration $F^{\lambda}D$ in the same way as
$F^{\lambda}C$.
(We assume $\lambda_q(e_{i}^*) = 0$.)
\par
It is easy to see that if $x \in F^{\lambda_1}C$, $y \in F^{\lambda_2}D$ then
\begin{equation}
\langle x,y\rangle \in F^{\lambda_1+\lambda_2}\Lambda^{\downarrow}.
\end{equation}
We define the adjoint
$
\partial^* : D \to D
$
by
$$
\langle x,\partial^* y\rangle = \langle \partial x,y\rangle.
$$
It is easy to see that $\partial^*\circ \partial^* =0$ and
$\partial^*(F^{\lambda}D) \subset F^{\lambda}D$.
\begin{defn}
We call $(D,\partial^*)$ the {\it filtered dual complex}\index{filtered dual complex} of $(C,\partial)$.
\end{defn}
We take a dual basis of $\{e'_i \mid i=1,\dots, b\} \cup
\{e''_i \mid i=1,\dots, h\}
\cup
\{e'''_i \mid i=1,\dots, b\} $.
Namely we take the basis so that
$$
\langle e'_i ,e'''_{*i}\rangle = 1,  \quad
\langle e''_i ,e''_{*i}\rangle = 1, \quad
\langle e'''_i ,e'_{*i}\rangle = 1
$$
and all the other pairings among the basis are zero.
It is easy to see that
$\{e'_{*i} \mid i=1,\dots, b\}$ is a basis of
$\text{\rm Im}\,\partial^*$
and $\{e'_{*i} \mid i=1,\dots, b\} \cup
\{e''_{*i} \mid i=1,\dots, h\} $ is a basis of
$\text{\rm Ker}\,\partial^*$.
\par
In the same way as in (\ref{infattainH00}) we have
\begin{equation}\label{infattainH2}
\inf\{ \lambda_q(x) \mid x \in \text{\rm Ker}\,\partial^*\!,\, b =[x]\}
=
\lambda_q\left(\sum_{i=1}^h b_i e''_{*i}\right)
\end{equation}
for $b \in H(D;\partial^*)$.
We define $\lambda_q(b)$ for $b \in H(D;\partial^*)$ by the left hand side.
\par
The pairing $\langle \cdot,\cdot\rangle$ induces a perfect $\Lambda^{\downarrow}$ pairing between
$H(C;\partial)$ and $H(D;\partial^*)$, which we also denote by $\langle \cdot,\cdot\rangle$.
By applying (\ref{infattainH00}) and  (\ref{infattainH2}) to $(C,\del)$ instead of $(D,\del^*)$, we
obtain
\begin{lem}\label{dualityalgmainlemma}
\begin{equation}
\lambda_q(a) = \sup \{ \frak v_q(\langle a,b\rangle)
\mid b \in H(F^0D;\partial^*)\}
\end{equation}
for $a \in H(C;\partial)$.
\end{lem}

\subsection{Duality between Floer homologies}
\label{subsec:dualityHF}

Let $H$ be a one-periodic time dependent Hamiltonian on $M$ such  that $\psi_H$ is non-degenerate.
We consider the chain complex
$(CF(M,H;\Lambda^{\downarrow}),\partial^{\frak b}_{(H,J)})$
which is defined in Section \ref{sec:deform-bdy}.
\par
Let $\{\gamma_i \mid i=1,\dots,N\} = \text{\rm Per}(H)$ and fix a choice of their lifts
$\llb \gamma_i,w_i \rrb \in \widehat{\rm Per}(H)$.
Then we put
$$
e_i = q^{-\mathcal A_H(\llb \gamma_i,w_i \rrb)} \llb \gamma_i,w_i \rrb
\in CF(M,H;\Lambda^{\downarrow}).
$$
We note that $\lambda_H(e_i) = 0$ and $e_i$ is independent of $w_i$.
$\{e_i \mid 1,\dots, N\}$ is a $\Lambda^{\downarrow}$ basis of
$CF(M,H;\Lambda^{\downarrow})$. It is easy to see that the filtration
of $CF(M,H;\Lambda^{\downarrow})$ defined as in Subsection
\ref{subsec:Usher} coincides with the filtration defined in Definition \ref{valuationv}.
\par
We define $\widetilde H$ by
\begin{equation}\label{tildeHHH}
\widetilde H(t,x) = - H(1-t,x).
\end{equation}
We have
$\phi_{\widetilde H}^t = \phi_{H}^{1-t} \circ (\phi_{H}^1)^{-1}$.
In particular,
$\psi_{\widetilde H} = (\psi_{H})^{-1}$.
Hence $\psi_{\widetilde H}$ is also non-degenerate.
\par
The main result of this subsection is as follows:
\begin{prop}\index{duality}
We can choose the perturbation etc. that are used in the definition of
$(CF(M,\widetilde H;\Lambda^{\downarrow}),\partial_{(\widetilde H,J)}^{\frak b})$ such that
there exists a perfect pairing
$$
\langle \cdot,\cdot \rangle :
CF(M, H;\Lambda^{\downarrow})
\times CF(M,\widetilde H;\Lambda^{\downarrow}) \to \Lambda^{\downarrow}
$$
by which the filtered complex
 $(CF(M,\widetilde H;\Lambda^{\downarrow}),\partial_{(\widetilde H,\widetilde J)}^{\frak b})$
is identified with the dual filtered complex of $(CF(M,H;\Lambda^{\downarrow}),\partial_{(H,J)}^{\frak b})$.
\end{prop}
\begin{proof}
Let $\gamma \in \text{\rm Per}(H)$.
It is then easy to see that
$$
\widetilde \gamma(t) = \gamma(1-t) \in \text{\rm Per}(\widetilde H).
$$
If $w : D^2 \to M$ satisfies $w\vert_{\partial D} = \gamma$, then
$\widetilde w(z) = w(\overline z)$ satisfies $\widetilde w\vert_{\partial D} = \widetilde \gamma$.
We have thus defined
\begin{equation}
\iota : \text{\rm Crit}(\mathcal A_H) \to \text{\rm Crit}(\mathcal A_{\widetilde H})
\end{equation}
by $[\gamma,w] \mapsto [\widetilde \gamma,\widetilde w]$.
It is easy to see
\begin{equation}\label{sumAdual0}
\mathcal A_H([\gamma,w])
+ \mathcal A_{\widetilde H}([\widetilde \gamma,\widetilde w]) = 0
\end{equation}
and
\be\label{eq:sumw*omega}
\int w^*\omega + \int \widetilde w^*\omega = 0.
\ee
Let $(u;z_1^+,\dots,z_{\ell}^+) \in \overset{\circ}{\CM}_{\ell}(H,J;[\gamma,w], [\gamma',w'])$.
We define
$
\iota : \R \times S^1 \to \R \times S^1
$
by $\iota(\tau,t) = (-\tau,1-t)$
and put
\begin{equation}\label{startransf}
\widetilde u = u \circ \iota.
\end{equation}
It is easy to find that
$$
(\widetilde u;\widetilde z_1^+,\dots,\widetilde z_\ell^+)\in
\overset{\circ}{\CM}_{\ell}(\widetilde H,\widetilde J; [\widetilde \gamma',\widetilde w'],[\widetilde \gamma,\widetilde w]).
$$
We have thus defined a homeomorphism
$$
\mathfrak I : \overset{\circ}{\CM}_{\ell}(H,J;[\gamma,w], [\gamma',w'])
\to \overset{\circ}{\CM}_{\ell}(\widetilde H,\widetilde J; [\widetilde \gamma',\widetilde w'],[\widetilde \gamma,\widetilde w])
$$
by
$$
\mathfrak I (u;z_1^+,\dots,z_{\ell}^+) = (\widetilde u;\widetilde z_1^+,\dots,\widetilde z_\ell^+).
$$
We can extend it to their compactifications and then
it becomes an isomorphism between spaces with Kuranishi structure:
$$
\mathfrak I :  {\CM}_{\ell}(H,J;[\gamma,w], [\gamma',w'])
\to  {\CM}_{\ell}(\widetilde H,\widetilde J; [\widetilde \gamma',\widetilde w'],[\widetilde \gamma,\widetilde w]).
$$
We take a CF-perturbation of
$ {\CM}_{\ell}(\widetilde H,\widetilde J; [\widetilde \gamma',\widetilde w'],[\widetilde \gamma,\widetilde w])$
so that it coincides with one for
${\CM}_{\ell}(H,J;[\gamma,w], [\gamma',w'])$
by the above isomorphism.
Then we have
$$
\frak n_{(H,J);\ell}([\gamma,w],[\gamma',w']) (h_1,\ldots,h_{\ell})
=
\frak n_{(\widetilde H,\widetilde J);\ell}([\widetilde \gamma',\widetilde w'],[\widetilde \gamma,\widetilde w ]) (h_1,\ldots,h_{\ell}),
$$
where the left hand side is defined in (\ref{eq:nww'C}).
Therefore
\begin{equation}\label{dualnvalue}
\frak n_{(H,J)}^{\frak b}(\llb \gamma,w \rrb,\llb \gamma',w' \rrb)
= \frak n_{(\widetilde H,\widetilde J)}^{\frak b}(\llb \widetilde \gamma',\widetilde w' \rrb,
\llb \widetilde \gamma,\widetilde w \rrb).
\end{equation}
Here  the left hand side is defined in (\ref{nf66666}).
\begin{defn}
Let $ \llb \gamma,w \rrb \in \widehat{\Per}(H)$,
$\llb \widetilde \gamma',\widetilde w' \rrb \in \widehat{\Per}(\widetilde H)$.
We define
\begin{equation}\label{pairHFdef}
\langle
\llb \gamma,w\rrb, \llb\widetilde \gamma',\widetilde w'\rrb
\rangle
=\begin{cases}
0 &\text{if $\gamma \ne  \gamma'$,}\\
q^{-(w\cap \omega + \widetilde w'\cap \omega)}
&\text{if $\gamma =  \gamma'$.}
\end{cases}
\end{equation}
We can extend (\ref{pairHFdef}) to a $\Lambda^{\downarrow}$
bilinear pairing
$$
\langle \cdot, \cdot\rangle : CF(M,H;\Lambda^{\downarrow})
\times CF(M,\widetilde H;\Lambda^{\downarrow})
\to \Lambda^{\downarrow},
$$
which becomes a perfect pairing.
\end{defn}
By (\ref{eq:sumw*omega}) we have
\begin{equation}\label{pairHFdef2}
\langle
\llb \gamma,w \rrb,\llb \widetilde \gamma,\widetilde w \rrb
\rangle = 1.
\end{equation}
\begin{lem}
\begin{equation}\label{dualityequality}
\langle
\partial_{(H,J)}^{\frak b}(\llb \gamma_1,w_1 \rrb),\llb \widetilde \gamma_2,\widetilde w_2 \rrb
\rangle
=
\langle
\llb \gamma_1,w_1 \rrb,\partial_{(\widetilde H,\widetilde J)}^{\frak b}(\llb \widetilde \gamma_2,\widetilde w_2 \rrb)
\rangle.
\end{equation}
\end{lem}
\begin{proof}
By definition the left hand side is
$$\aligned
&\sum_{w'_2 \in \pi_2(\gamma_2)} \frak n^{\frak b}_{(H,J)}([\gamma_1,w_1]),[\gamma_2, w'_2])
q^{-(w'_2\cap\omega + \widetilde w_2\cap \omega)} \\
&=\sum_{\alpha \in \pi_2(M)}
\frak n^{\frak b}_{(H,J)}[\gamma_1,w_1],[\gamma_2,\alpha\# w_2])
q^{-\alpha \cap\omega}.
\endaligned$$
On the other hand, the right hand side is
$$
\aligned
&\sum_{\widetilde w'_1 \in \pi_2(\gamma'_1)}
\frak n^{\frak b}_{(\widetilde H,\widetilde J)}([\widetilde \gamma_2,\widetilde w_2]),
[\widetilde \gamma_1,\widetilde w'_1])
q^{-(w_1\cap \omega + \widetilde w'_1\cap \omega)}\\
&=\sum_{\alpha \in \pi_2(M)}
\frak n^{\frak b}_{(\widetilde H,\widetilde J)}([\widetilde \gamma_2,\widetilde w_2]),[\widetilde \gamma_1,
\widetilde{(-\alpha)\# w_1}])
q^{-\alpha \cap \omega}.
\endaligned
$$
By (\ref{dualnvalue})  this is equal to
$$
\frak n^{\frak b}_{(H,J)}([ \gamma_1,{(-\alpha)\# w_1}],[ \gamma_2,w_2])
q^{-\alpha \cap\omega}.
$$
Since
$$
\frak n^{\frak b}_{(H,J)}([ \gamma_1,{(-\alpha)\# w_1}],[ \gamma_2,w_2])
=
\frak n^{\frak b}_{(H,J)}([ \gamma_1,{w_1}],[ \gamma_2,\alpha\# w_2]),
$$
by Proposition \ref{connBULKkura} (7),
the lemma follows.
\end{proof}
Then (\ref{pairHFdef2}) and (\ref{dualityequality}) imply the
proposition.
\end{proof}

\subsection{Duality and Piunikhin isomorphism}
\label{subsec:dualityandP}\index{Piunikhin isomorphism}

In this subsection we prove:

\begin{thm}\label{dualityandPthm}\index{duality}
For $a,a' \in H^*(M;\Lambda)$
 we
denote by $a^\flat, (a')^\flat$ the homology classes Poincar\'e dual to $a,a'$ respectively. (See Notations and Conventions $(22)$.)
Then we have
\begin{equation}
\langle
\CP_{(H_\chi,J_\chi),\ast}^{\frak b}(a^\flat),
\mathcal P_{(\widetilde H_\chi,\widetilde J_\chi),\ast}^{\frak b}((a')^\flat)
\rangle
=
\langle
a,
a'
\rangle.
\end{equation}
\end{thm}
\begin{proof}
We consider  two chain maps
$
: (\Omega(M) \widehat\otimes \Lambda)
\otimes
(\Omega(M) \widehat\otimes \Lambda)
\to \Lambda \cong
\Lambda^{\downarrow}
$
\begin{equation}\label{firstpairing}
h \otimes h' \mapsto \int_M h \wedge h'
\end{equation}
and
\begin{equation}\label{secondparing}
h \otimes h'  \mapsto \langle \CP_{(H_\chi,J_\chi),\#}^{\frak b}(h^\flat),
\mathcal P_{(\widetilde H_\chi,\widetilde J_\chi),\#}^{\frak b}((h')^\flat)\rangle.
\end{equation}

Here we regard $\Lambda^{\downarrow}$ as a chain complex with trivial boundary operator.
To prove Theorem \ref{dualityandPthm} it suffices to show that (\ref{firstpairing}) is
chain homotopic to (\ref{secondparing}). For this purpose, we will use the following
parameterized moduli space
$$
{{\CM}}_{\ell}(para;H_\chi,J_\chi;*,*;C) = \bigcup_{S\geq 0}
\{ S \} \times {{\CM}}_{\ell}(H_{\chi}^S,J_{\chi}^S;*,*;C)
$$
equipped with Kuranishi structure and CF-perturbation that is compatible along the boundary.
We refer readers to Definition \ref{moduliforH1} in Section \ref{sec:appendix1}
for the precise description of the moduli space ${\CM}_{\ell}(para; H_\chi,J_\chi;*,*;C)$
defined in (\ref{266below}).
Here $C \in \pi_2(M)$.
\par
We denote $\tilde \chi = \tilde \chi(\tau) = \chi(1-\tau)$.
Some boundary component of ${{\CM}}_{\ell}(para;H_\chi,J_\chi;*,*;C)$
in (\ref{bdryQBULKkurapara1}) will contain a direct factor of the type
${{\CM}}_{\#\mathbb L_2}(H_{\tilde \chi},J_{\tilde \chi};[\gamma,w],*)$
whose definition is given in Definition \ref{moduliforQ}.
We consider the map
\begin{equation}\label{identfypandqmoduli}
\frak J :
{{\CM}}_{\ell}(\widetilde H_\chi,\widetilde J_\chi;*,[\widetilde \gamma,\widetilde w])
\to {{\CM}}_{\ell}(H_{\tilde \chi},J_{\tilde \chi};[\gamma,w],*)
\end{equation}
defined by
$$
\frak J((u;z_1^+,\dots,z_{\ell}^+))
=(\tilde u;\tilde z_{1}^+,\dots,\tilde z_{\ell}^+),
$$
where the right hand side is defined as in (\ref{startransf}).
The map (\ref{identfypandqmoduli}) is extended to an
isomorphism of spaces with Kuranishi structures.
\par
Recall that when we considered $\mathcal P_{(\widetilde H_{{\chi}},\widetilde J_{{\chi}})}^{\frak b}$,
we made a choice of a CF-perturbation on
${{\CM}}_{\ell}(\widetilde H_\chi,\widetilde J_\chi;*,[\widetilde \gamma,\widetilde w])$.
This CF-perturbation induces a CF-perturbation on
$
{{\CM}}_{\ell}(H_{\tilde \chi},J_{\tilde \chi};[\gamma,w],*)
$
via the isomorphism \eqref{identfypandqmoduli}.
\par
We equip ${{\CM}}_{\ell}(para;H_\chi,J_\chi;*,*;C)$ with a system of CF-perturbations
that is compatible along the boundary with respect to
this choice of CF-perturbation on the direct factor
${{\CM}}_{\ell}(H_{\tilde \chi},J_{\tilde \chi};[\gamma,w],*)$ appearing
in (\ref{bdryQBULKkurapara1}).

\begin{rem}
On the other hand,
when we will define $\mathcal Q_{(H_{\tilde \chi},J_{\tilde \chi})}^{\frak b}$
in Section \ref{sec:appendix1}, we take another family of CF-perturbations
on ${{\CM}}_{\ell}(H_{\tilde \chi},J_{\tilde \chi};[\gamma,w],*)$.
This is {\it different} from the CF-perturbation defined above.
\end{rem}

Now let $h,h'$ be differential forms on $M$.
We define
\begin{equation}\label{overlineHdef}
\aligned
&\overline{\mathcal H}_{(H_\chi,J_\chi)}^{\frak b}(h,h')
\\& =
\sum_C \sum_{\ell=0}^{\infty}
\frac{\exp(\int_C \frak b_2)}{\ell !}
q^{-C \cap \omega} \\
&\quad\int_{\!{{\CM}}_{\ell}(para; H_\chi,J_\chi;*,*;C)}
ev_{+\infty}^* h
\wedge ev_{-\infty}^* h' \wedge
\text{\rm ev}^*(\underbrace{\frak b_{+},
\dots,\frak b_{+}}_{\ell}),
\endaligned
\end{equation}
where $\frak b_2$ is the summand (more precisely its representative
closed two-form) in the decomposition
$\frak b = \frak b_0 + \frak b_2 + \frak b_{+}$ as before and we use the above
chosen CF-perturbation on
${{\CM}}_{\ell}(para; H_\chi,J_\chi;*,*;C)$ to define an integration on it.
The formula (\ref{overlineHdef}) defines a map
$$
\overline{\mathcal H}_{(H_\chi,J_\chi)}^{\frak b} :
(\Omega(M) \widehat\otimes \Lambda)
\otimes
(\Omega(M) \widehat\otimes \Lambda)
\to
\Lambda^{\downarrow}.
$$
It follows from Lemma \ref{HBULKkurapara} (3)
that $\overline{\mathcal H}_{(H_\chi,J_\chi)}^{\frak b}$ is a
chain homotopy between
(\ref{firstpairing}) and (\ref{secondparing}).
The proof of Theorem \ref{dualityandPthm} is complete.
\end{proof}

\subsection{Proof of Theorem \ref{dualitymain}}
\label{subsec:dualitycomplproof}

Now we prove Theorem  \ref{dualitymain}. Let $\frak b$ be given.
Once Theorem \ref{dualityandPthm} is established,
the proof is the same as \cite{EP:states}.
It suffices to prove it in the case when $\widetilde\phi$ is nondegenerate.
We take $H$ such that $\widetilde\phi = \widetilde\psi_H$.
Let $a \in QH^*_{\frak b}(M)$ and $\epsilon >0$.
By Lemma \ref{dualityalgmainlemma},
we have $b' \in QH^*_{\frak b}(M)$ such that
\begin{equation}\label{estimatefrombelowPD}
\frak v_q\left(
\left\langle
\CP_{(H_\chi,J_\chi),\ast}^{\frak b}(a^\flat),
\mathcal P_{(\widetilde H_{\chi},\widetilde J_{\chi}),\ast}^{\frak b}((b')^\flat)
\right\rangle
\right)
\ge \rho^{\frak b}(H;a) - \epsilon
\end{equation}
and
$$
\lambda_{\widetilde H}(\mathcal P_{(\widetilde H_{\chi},\widetilde J_{\chi}),\ast}^{\frak b}((b')^\flat)
) \le 0.
$$
Let $\lambda$ be the left hand side of (\ref{estimatefrombelowPD}).
Then
$$
0 =
\frak v_q\left(\left\langle
\CP_{(H_\chi,J_\chi),\ast}^{\frak b}(a^\flat),
\mathcal P_{(\widetilde H_{\chi},\widetilde J_{\chi}),\ast}^{\frak b}(q^{-\lambda}(b')^\flat)
\right\rangle\right)
=
\frak v_q\left(\langle a,q^{-\lambda}b'
\rangle\right).
$$
(We use Theorem \ref{dualityandPthm} here.)
We put $b  = q^{-\lambda}b'$.
Then by definition
$$
\Pi(a,b) \ne 0.
$$
Thus, since $\lambda_{\widetilde H}(\mathcal P_{(\widetilde H_{\chi},\widetilde J_{\chi}),\ast}^{\frak b}(b^\flat))
= -\lambda + \lambda_{\widetilde H}(\mathcal P_{(\widetilde H_{\chi},\widetilde J_{\chi}),\ast}^{\frak b}((b')^\flat))
\le -\lambda$, we have
$$
\rho^{\frak b}(H;a) - \epsilon
\le \lambda
\le
-
\inf \{
\rho^{\frak b}(\widetilde\psi_H^{-1};b)
\mid \Pi(a,b) \ne 0
\}.
$$
Hence
\begin{equation}\label{estimatefrombelowPD2}
\rho^{\frak b}(\widetilde\psi_H;a)
\le
-
\inf \{
\rho^{\frak b}(\widetilde\psi_H^{-1};b)
\mid \Pi(a,b) \ne 0
\}.
\end{equation}
On the other hand, if $\Pi(a,b) \ne 0$ then
$$
\frak v_q
(\langle a, b\rangle) \ge 0.
$$
It implies
$$
\frak v_q\left(
\left\langle
\CP_{(H_\chi,J_\chi),\ast}^{\frak b}(a^\flat),
\mathcal P_{(\widetilde H_{\chi},\widetilde J_{\chi}),\ast}^{\frak b}(b^\flat)
\right\rangle
\right)
\ge 0.
$$
Hence
$$
\lambda_H(\CP_{(H_\chi,J_\chi),\ast}^{\frak b}(a^\flat))
+
\lambda_{\widetilde H}(\mathcal P_{(\widetilde H_{\chi},\widetilde J_{\chi}),\ast}^{\frak b}(b^\flat))
\ge 0.
$$
Therefore
\begin{equation}\label{estimatefromabove}
\rho^{\frak b}(\widetilde\psi_H;a)
\ge
-
\inf \{
\rho^{\frak b}(\widetilde\psi_H^{-1};b)
\mid \Pi(a,b) \ne 0
\}.
\end{equation}
(\ref{estimatefrombelowPD2}) and (\ref{estimatefromabove})
imply Theorem  \ref{dualitymain}.
\qed

\section{Construction of quasi-morphisms via spectral invariant with bulk}
\label{sec:const-morphism}\index{quasi-morphism! construction}

The next definition is due to Entov-Polterovich (see \cite[Section 1.1]{EP:morphism}).
\begin{defn}\label{def:Qhomo}
A function $\mu : \widetilde{\text{\rm Ham}}(M,\omega) \to \R$
is called a {\it homogeneous Calabi quasi-morphism}\index{quasi-morphism!Calabi quasi-morphism} if the following three conditions are
satisfied.
\begin{enumerate}
\item It is a quasi-morphism. Namely there exists a constant $C$ such that
for any $\widetilde\phi, \widetilde\psi \in \widetilde{\text{\rm Ham}}(M,\omega)$
we have
$$
\vert \mu(\widetilde\phi  \widetilde\psi) - \mu(\widetilde\phi)
- \mu(\widetilde\psi) \vert < C,
$$
where $C$ is independent of $\widetilde\phi, \widetilde\psi$.
\item
It is homogeneous. Namely
$
\mu(\widetilde\phi^n) = n\mu(\widetilde\phi)
$
for $n\in \Z$.
\item
If $\widetilde\phi \in \widetilde{\text{\rm Ham}}_U(M,\omega)$ and $U$ is
a displaceable open subset  of $M$, then
we have
$
\mu(\widetilde\phi ) = \text{\rm Cal}_U(\widetilde\phi).
$
\end{enumerate}
\end{defn}
\begin{rem}
We note that we have the canonical homomorphism
\linebreak
$\widetilde{\text{\rm Ham}}_U(M,\omega)
\to \widetilde{\text{\rm Ham}}(M,\omega) $. We use this homomorphism
to make sense out of the left hand side of the identity (3) above.
\end{rem}

The following is the analog to  \cite[Theorem 3.1]{EP:morphism} whose
proof is essentially the same once Theorem  \ref{dualitymain} is at our disposal.

\begin{thm}\label{thm:morphism}
Let $\frak b \in H^{{\rm even}}(M;\Lambda_0)$.
Suppose that there is a ring isomorphism
$
QH^*_{\frak b}(M) \cong \Lambda
\times Q
$
and let $e \in QH^*_{\frak b}(M)$ be the
idempotent corresponding to the unit of the first factor of the right hand side.
Then the function \index{$\mu_e^{\frak b}$}
$$
\mu_e^{\frak b}: \widetilde{{\rm Ham}}(M,\omega) \to \R
$$
is a homogeneous Calabi quasi-morphism.
\end{thm}
\begin{rem}
An observation by McDuff is that a sufficient condition for the existence of
Calabi quasi-morphism is an existence of a direct product factor of
a quantum cohomology that is a field.
\cite{EP:morphism} used quantum homology over $\Lambda^{\downarrow}(\Q)$,
that is the set of $\sum a_iq^{\lambda_i}$ with $a_i \in \Q$.
Here we use the (downward) universal Novikov ring $\Lambda^{\downarrow}$,
where $a_i \in \C$.
Since $\Lambda^{\downarrow}$ is an algebraically closed field (see \cite[Appendix A]{fooo:toric1}) and
the quantum cohomology ring
is finite dimensional thereover, the direct product factor of
a quantum cohomology is isomorphic to $\Lambda$,
if it is a field. So our assumption of Theorem \ref{thm:morphism}
is equivalent to McDuff's in case $\frak b =0$.
\end{rem}
\begin{proof}
Let $\frak b$ and $e$ be as in Theorem \ref{thm:morphism}.
We first prove the property (1) of Definition \ref{def:Qhomo}.
We begin with the following lemma.
\begin{lem}\label{dualitygyakmain}
$$
\rho^{\frak b}(\widetilde\phi;e) \le 3\frak v_q(e) - \rho^{\frak b}(\widetilde\phi^{-1};e).
$$
\end{lem}
\begin{proof}
Let $b \in H(M;\Lambda)$ such that
$\Pi(e,b) \ne 0$. Such a $b$ exists by the nondegeneracy of the Poincar\'e pairing.
We write $b = (b_1,b_2)$ with respect to the decomposition
$QH_{\frak b}^{\ast}(M) \cong \Lambda^\downarrow \times Q$.
Using the Frobenius property of quantum cohomology
(see, for example, \cite{Manin}) we obtain
$$
\langle e,b\rangle
= \langle e \cup^{\frak b} e,b\rangle
= \langle e, e \cup^{\frak b} b\rangle
= \langle e, b_1\rangle.
$$
Therefore $\Pi(e,b_1)  = \Pi(e,b)  \ne 0$.
\begin{sublem}\label{sublem166}
$\frak v_q(b_1) \ge 0.$
\end{sublem}
\begin{proof}
We have
$b_1 = x e$ for some $x \in \Lambda \setminus \{0\}$. We decompose $e = \sum_{d =0}^{2n} e_d$
with $e_d \in H^d(M;\C) \otimes \Lambda$.
We denote by ${\mathbf 1} \in H^0(M; \C)$ the unit of the cohomology ring.
Then
\beastar
\Pi(e,b_1) & = &\pi(\langle e, b_1\rangle) = \pi(\langle e, xe \rangle)
= \pi(\langle e\cup^{\frak b} e, x {\mathbf 1}\rangle) \\
& = & \pi(\langle e, x {\mathbf 1} \rangle) = \pi(\langle x e, {\mathbf 1} \rangle) = \pi (\langle b_1, {\mathbf 1} \rangle).
\eeastar
Therefore $\Pi(e,b_1) \ne 0$ implies $\frak v_q (\langle b_1, {\mathbf 1} \rangle) >0$.
Since $\frak v_q (b_1) \geq \frak v_q (\langle b_1, {\mathbf 1} \rangle )$,
we obtain  $\frak v_q(b_1) \ge 0$ as required.
\end{proof}
Let $xe = b_1$ and $x \in \Lambda \setminus \{0\}$ as above.
Then
$$
\frak v_q(x) + \frak v_q(e) = \frak v_q(b_1) \ge 0.
$$
Since
$
b_1^{-1} = x^{-1}e,
$
we get
$$
\frak v_q(b_1^{-1}) =  -\frak v_q(x) + \frak v_q(e) = (\frak v_q(e) - \frak v_q(b_1))
+ \frak v_q(e) \le 2 \frak v_q(e).
$$
Therefore
$$
\aligned
\rho^{\frak b}(\widetilde\phi^{-1};b)
&\ge
\rho^{\frak b}(\widetilde\phi^{-1};b_1)
-
\rho^{\frak b}(\underline 0;e) \\
&\ge
\rho^{\frak b}(\widetilde\phi^{-1};e)
-
\rho^{\frak b}(\underline 0;b^{-1}_1)
-
\rho^{\frak b}(\underline 0;e)
=
\rho^{\frak b}(\widetilde\phi^{-1};e)
- \frak v_q(b^{-1}_1)
- \frak v_q(e)
\\
&\ge \rho^{\frak b}(\widetilde\phi^{-1};e) - 3\frak v_q(e).
\endaligned
$$
Here we use the identity $b_1 \cup^{\frak b} b_1^{-1} = e$ and the triangle inequality for the second inequality.
Lemma \ref{dualitygyakmain} now follows from Theorem \ref{dualitymain}.
Here $\underline 0$ is the identity element.
\end{proof}
\begin{cor}\label{rhoqh}
$$
\rho^{\frak b}(\widetilde\phi;e)
+ \rho^{\frak b}(\widetilde \psi;e)
\ge
\rho^{\frak b}(\widetilde\psi\widetilde\phi;e)
\ge
\rho^{\frak b}(\widetilde\phi;e)
+ \rho^{\frak b}(\widetilde \psi;e)  - 3\frak v_q(e).
$$
\end{cor}
\begin{proof}
The first inequality is a consequence of Theorem \ref{axiomshbulk} (5).
We have
$$
\rho^{\frak b}(\widetilde\psi\widetilde\phi;e)
\ge
\rho^{\frak b}(\widetilde \psi;e)
- \rho^{\frak b}(\widetilde\phi^{-1};e)
\ge
\rho^{\frak b}(\widetilde \psi;e) + \rho^{\frak b}(\widetilde\phi;e) - 3\frak v_q(e).
$$
Here the first inequality follows from Theorem \ref{axiomshbulk} (5)
and the second inequality follows from Lemma \ref{dualitygyakmain}.
\end{proof}
We use Corollary \ref{rhoqh} inductively to show
\begin{equation}\label{rhomultieq}
\left\vert
\rho^{\frak b}(\widetilde\phi_1\cdots \widetilde\phi_k;e)
-
\sum_{i=1}^k \rho^{\frak b}(\widetilde\phi_i;e)
\right\vert
\le
3k\frak v_q(e).
\end{equation}
Therefore
\begin{equation}
\left\vert
\rho^{\frak b}((\widetilde\phi \widetilde\psi)^n;e)
-
n \rho^{\frak b}(\widetilde\phi;e)
-
n \rho^{\frak b}(\widetilde \psi;e)
\right\vert
\le
6n\frak v_q(e).
\nonumber
\end{equation}
\begin{equation}
\left\vert
\rho^{\frak b}(\widetilde\phi^n;e)
-
n \rho^{\frak b}(\widetilde\phi;e)
\right\vert
\le
3n\frak v_q(e).
\nonumber
\end{equation}
\begin{equation}
\left\vert
\rho^{\frak b}(\widetilde\psi^n;e)
-
n \rho^{\frak b}(\widetilde \psi;e)
\right\vert
\le
3n\frak v_q(e).
\nonumber
\end{equation}
Hence
$$
\left\vert
\rho^{\frak b}((\widetilde\phi \widetilde\psi)^n;e)
-
\rho^{\frak b}(\widetilde\phi^n;e)
-
\rho^{\frak b}(\widetilde\psi^n;e)
\right\vert
\le
12n\frak v_q(e).
$$
By dividing the last inequality by $n$ and taking its limit as $n \to \infty$, we obtain
$$
\vert
\mu_e^{\frak b}(\widetilde\phi \widetilde\psi)
-
\mu_e^{\frak b}(\widetilde\phi )
-
\mu_e^{\frak b}(\psi)
\vert
\le
12\frak v_q(e).
$$
Thus, $\mu_e^{\frak b}$ is a quasi-morphism.
\begin{rem}\label{rem168}
\begin{enumerate}
\item
The constant $C$ in Definition \ref{def:Qhomo} can be taken to be $12\frak v_q(e)$
for the quasi-morphism in Theorem \ref{thm:morphism}.
\item
Our proof of Lemma \ref{dualitygyakmain} is slightly simpler
than \cite[Lemma 3.2]{EP:morphism}, since we may assume that the field which is a
direct factor of quantum cohomology is $\Lambda$
and so we do not need a result from general
non-Archimedean geometry which is quoted in  \cite{EP:morphism}.
By the same reason we obtain an explicit bound.
\end{enumerate}
\end{rem}
Definition \ref{def:Qhomo} (2)
follows from Theorem \ref{defpscquasi-state} (5).
\par
The homogeneity of
$\mu_e^{\frak b}$ follows from
$$
\rho^{\frak b}(\underline 0;e)
\le
\rho^{\frak b}(\widetilde{\phi}^n;e) + \rho^{\frak b}(\widetilde{\phi}^{-n};e)
\le
3\frak v_q(e)
$$
and Definition \ref{defpscquasi-state} (2).
 The proof of Theorem \ref{thm:morphism} is complete.
\end{proof}
Theorem \ref{existquasihomo} is immediate from Theorem \ref{thm:morphism}.

\part{Spectral invariants and Lagrangian Floer theory}
\label{part:OCGWandLag}

The purpose of this part is to prove Theorem \ref{supportmain}.
The proof is based on open-closed Gromov-Witten theory\index{open-closed Gromov-Witten theory}
developed in \cite[Section 3.8]{fooo:book1}, which induces a
map from the quantum cohomology of the ambient
symplectic manifold to the Hochschild cohomology of $A_\infty$ algebra (or more
generally that of Fukaya category of $(M,\omega)$).
This map is defined in \cite{fooo:book1}
for arbitrary compact symplectic manifold and
its (weakly) unobstructed Lagrangian submanifold.
See also \cite[Section 4.7]{fooo:toricmir}  for several properties
of this map and various related works.
For our purpose, we need only a small portion thereof,
that is, the part constructed in \cite[Theorem 3.8.62]{fooo:book1}
to which we restrict ourselves in this paper,
except in Section \ref{subsec:Hochschild}.
\par
The main new part of the proof is the
construction of a map from Floer homology of periodic Hamiltonians
to the Floer cohomology of Lagrangian submanifold,
through which the map from quantum cohomology
to Floer cohomology of Lagrangian submanifold factors (Definition \ref{def:I} and Proposition \ref{184mainprop}).
We also study its properties, especially those related to the
filtration.
\par
In Parts 4 and 5, we fix a compatible almost complex structure $J$ that is
$t$-{\it independent}.

\section{Operator $\frak q$; review}
\label{sec:q-map} \index{operator $\frak q$}

In this section, we review a part of the results from \cite[Section 3.8]{fooo:book1}.
\par
Let $(M,\omega)$ be a compact symplectic manifold and $L$  its
relatively spin Lagrangian submanifold. We consider smooth differential forms on $M$.
Note that in \cite{fooo:book1, fooo:book2, fooo:bulk} we used smooth singular chains
instead of differential forms to represent cohomology
classes on $M$. In this paper we use differential forms because
we use them in the discussion of Floer homology in Part 2.
The construction of the operator $\frak q$ in this section
is a minor modification of the one given in  \cite[Section 3.8]{fooo:book1}
where smooth singular chains on $M$ are used.
\par
We will introduce a family of operators denoted by
\index{$\frak q_{\ell,k;\beta}$}
\begin{equation}\label{eq:mapq}
\frak q_{\ell,k;\beta}:
E_{\ell} (\Omega(M)[2]) \otimes B_k(\Omega(L)[1]) \to  \Omega(L)[1].
\end{equation}
Explanation of the various notations appearing in \eqref{eq:mapq}
is in order.
$\beta$ is an element of the image of $\pi_2(M,L) \to
H_2(M,L;\Z)$ and $C[i]$ is the degree shift of a $\Z$ graded
$\C$-vector space $C$ by $i$ defined by
$(C[i])^d = C^{d+i}$.
We recall from Notations and Conventions
(19)--(20) that
$E_{\ell}C$ is the quotient of
$B_{\ell}C=\underbrace{C\otimes \cdots \otimes C}_{\text{$\ell$ times}}$
by the symmetric group action.
The map \eqref{eq:mapq} is a $\C$-linear map of degree $1
-\mu_{L}(\beta)$ here $\mu_{L}$ is the Maslov index.
\index{Maslov index}\index{$\mu_{L}(\beta)$}
(See \cite[Definition 2.1.15]{fooo:book1}.)
\par
We next describe the main properties of $\frak q_{\ell,k;\beta}$.
Recall from Notations and Conventions
(19)--(20) again that both
$
BC = \bigoplus_{k=0}^{\infty} B_kC
$
and
$
EC = \bigoplus_{\ell=0}^{\infty} E_{\ell}C
$
have the structure of coassociative
coalgebra with coproduct $\Delta$ respectively.
We also consider a map $\Delta^{n-1}: BC \to (BC)^{\otimes n}$ or $EC \to
(EC)^{\otimes n}$
\index{$\Delta^{n-1}$}
defined by
$$
\Delta^{n-1} = (\Delta \otimes  \underbrace{id \otimes \cdots
\otimes id}_{n-2}) \circ (\Delta \otimes  \underbrace{id \otimes
\cdots \otimes id}_{n-3}) \circ \cdots \circ \Delta.
$$
Following Sweedler's notation \cite{sweedler},
\index{Sweedler's notation} we express
an element $\text{\bf x} \in BC$ as
\index{$\Delta^{n-1}$}
\begin{equation}\label{deltawritebyc}
\Delta^{n-1}(\text{\bf x}) = \sum_c \text{\bf x}^{n;1}_c \otimes
\cdots \otimes \text{\bf x}^{n;n}_c
\end{equation}
\index{$\text{\bf x}^{n;i}_c$}
where $c$ runs over some index set depending on $\text{\bf x}$.
Here
we note that by Notations and Conventions (20) we
always use the coproducts $\Delta_{\rm decon}$\index{coproduct!deconcatenation}
on $B(\Omega(L)[1])$ and $\Delta_{\rm shuff}$\index{coproduct!shuffle}
on $E(\Omega(M)[2])$, respectively.
Thus for ${\bf x} \in B(\Omega(L)[1])$
the equation \eqref{deltawritebyc} expresses
the decomposition of
$\Delta_{\rm decon}^{n-1}({\bf x})$, while
for ${\bf y} \in E(\Omega(M)[2])$
the equation \eqref{deltawritebyc} expresses
the decomposition of
$\Delta_{\rm shuff}^{n-1}({\bf y})$.
For an element
$
\text{\bf x} = x_1 \otimes \cdots \otimes x_k \in B_k(\Omega(L)[1])
$
we put the shifted degree $\deg'x_i = \deg x_i - 1$ and
$
\deg' \text{\bf x} = \sum \deg'x_i = \deg \text{\bf x} - k.
$
(Recall $\deg x_i$ is the cohomological degree of $x_i$ before
shifted.)
The next result is the de Rham version of  \cite[Theorem 3.8.32]{fooo:book1}.
\begin{thm}\label{qproperties} The operators $\frak
q_{\beta;\ell,k}$ satisfy the following properties:
\begin{enumerate}
\item
For each $\beta$ and $\text{\bf x} \in B_k(\Omega(L)[1])$,
$\text{\bf y} \in E_k(\Omega(M)[2])$,
\begin{equation}\label{qmaineq}
0 =
\sum_{\beta_1+\beta_2=\beta}\sum_{c_1,c_2}
(-1)^*
\frak q_{\beta_1}(\text{\bf y}^{2;1}_{c_1};
\text{\bf x}^{3;1}_{c_2} \otimes
\frak q_{\beta_2}(\text{\bf y}^{2;2}_{c_1};\text{\bf x}^{3;2}_{c_2})
\otimes \text{\bf x}^{3;3}_{c_2})
\end{equation}
where
$
* = \deg'\text{\bf x}^{3;1}_{c_2} +
\deg'\text{\bf x}^{3;1}_{c_2} \deg \text{\bf y}^{2;2}_{c_1}
+\deg \text{\bf y}^{2;1}_{c_1}.
$
In $(\ref{qmaineq})$ and hereafter, we write $\frak q_{\beta}(\text{\bf y};\text{\bf x})$ in place of
$\frak q_{\ell,k;\beta}(\text{\bf y};\text{\bf x})$ if
$\text{\bf y} \in E_{\ell}(\Omega(M)[2])$, $\text{\bf x} \in B_{k}(\Omega(L)[1])$.
We use notation $(\ref{deltawritebyc})$ in $(\ref{qmaineq})$.
\item If $1 \in E_0(\Omega(M)[2])$ and $\text{\bf x} \in B_k(\Omega(L)[1])$ then
\begin{equation}\label{qism}
\frak q_{0,k;\beta}(1;\text{\bf x}) = \frak m_{k;\beta}(\text{\bf x}).
\end{equation}
Here $\frak m_{k;\beta}$ is the filtered $A_{\infty}$ structure on
$\Omega(L)$.
\item Let $\text{\bf e}$ be the $0$ form (function) on $L$
which is $1$ everywhere. Let $\text{\bf x}_i \in B(\Omega(L)[1])$ and we put
$\text{\bf x} = \text{\bf x}_1 \otimes \text{\bf e} \otimes \text{\bf x}_2
\in B(\Omega(L)[1])$. Then
\begin{equation}\label{unital}
\frak q_{\beta}(\text{\bf y};\text{\bf x}) = 0
\end{equation}
except the following case :
\begin{equation}\label{unital2}
\frak q_{\beta_0}(1;\text{\bf e} \otimes x) =
(-1)^{\deg x}\frak q_{\beta_0}(1;x \otimes \text{\bf e}) = x,
\end{equation}
where $\beta_0 = 0 \in H_2(M,L;\Z)$ and $x \in \Omega(L)[1]
= B_1(\Omega(L)[1])$.
Note $1$ in $(\ref{unital2})$ is $1 \in E_0(\Omega(M)[2])$.
\end{enumerate}
\end{thm}
The singular homology version of
Theorem \ref{qproperties} is proved in  \cite[Sections 3.8 and 7.4]{fooo:book1, fooo:book2}. The version where we use de Rham cohomology
for $L$ and cycles (of smooth submanifolds)
on $M$ is given in  \cite[Section 6]{fooo:toric1}
for the case when $M$ is a toric manifold and $L$ is a
Lagrangian torus fiber of $M$.
\par
Since the details of this construction will be needed for the proof of Theorem \ref{thm:heavy}
later in Section \ref{sec:heavy}, we explain the construction of the relevant operators
and the main ideas used in the proof of Theorem \ref{qproperties} here,
although it is a straightforward modification of the
constructions given in \cite{fooo:book1, fooo:book2, fooo:toric1}.
\index{$\mathcal M_{k+1;\ell}(L;\beta)$}

\begin{defn}\label{diskmoduli1}
We denote by $\overset{\circ}{\mathcal M}_{k+1;\ell}(L;\beta)$
the set of all $\sim$ equivalence classes of triples
$(u;z_1^{+},\dots,z_{\ell}^+;z_0,\dots,z_k)$ satisfying the following:
\begin{enumerate}
\item
$u : (D^2,\partial D^2) \to (M,L)$ is a pseudo-holomorphic map
such that $u(\partial D^2) \subset L$.
\item $z_1^{+},\dots,z_{\ell}^+$ are points in the interior of $D^2$
which are mutually distinct.
\item
$z_0,\dots,z_k$ are points on the boundary
$\partial D^2$ of $D^2$.
They are mutually distinct.
$z_0,\dots,z_k$ respects the counterclockwise cyclic order on $\partial D^2$.
\item The homology class $u_*([D^2,\partial D^2])$
is $\beta \in H_2(M,L;\Z)$.
\end{enumerate}
\par\medskip
We say that
$(u;z_1^{+},\dots,z_{\ell}^+;z_0,\dots,z_k)
\sim
(u';z_1^{\prime +},\dots,z^{\prime +}_{\ell};z'_0,\dots,z'_k)$
if there exists a biholomorphic map $v : D^2 \to D^2$
such that
$$
u' \circ v = u,
\quad
v(z_i^+) = z_i^{\prime +},
\quad
v(z_i) = z_i^{\prime}.
$$
\end{defn}
We define an evaluation map
$$
(\text{\rm ev},\text{\rm ev}^{\partial})
= (\text{\rm ev}_1,\dots,\text{\rm ev}_{\ell};
\text{\rm ev}^{\partial}_0,\dots,\text{\rm ev}^{\partial}_k)
: \overset{\circ}{\mathcal M}_{k+1;\ell}(L;\beta)
\to M^{\ell} \times L^{k+1}
$$
by
$$
\text{\rm ev}_i([u;z_1^{+},\dots,z_{\ell}^+;z_0,\dots,z_k])
= u(z_i^+),
\quad
\text{\rm ev}^{\partial}_i([u;z_1^{+},\dots,z_{\ell}^+;z_0,\dots,z_k])
= u(z_i).
$$
\begin{prop}\label{disckura}
\begin{enumerate}
\item The moduli space
$\overset{\circ}{\mathcal M}_{k+1;\ell}(L;\beta)$ has a compactification
${\mathcal M}_{k+1;\ell}(L;\beta)$ that is Hausdorff.
\item
The space ${\mathcal M}_{k+1;\ell}(L;\beta)$ has an orientable Kuranishi structure with corners.
\item
The (normalized) boundary of ${\mathcal M}_{k+1;\ell}(L;\beta)$ in the
sense of Kuranishi structure is described by the following fiber product
over $L$.
\begin{equation}\label{eq177}
\partial{\mathcal M}_{k+1;\ell}(L;\beta)
= \bigcup {\mathcal M}_{k_1+1;\# \L_1}(L;\beta_1)
{}_{\text{\rm ev}^{\partial}_{0}}\times_{\text{\rm ev}^{\partial}_i} {\mathcal M}_{k_2+1;\# \L_2}(L;\beta_2),
\end{equation}
where the union is taken over all $(\L_1,\L_2) \in \text{\rm Shuff}(\ell)$,
$k_1,k_2\in \Z_{\ge 0}$ with $k_1 + k_2 = k$ and $\beta_1,\beta_2
\in H_2(M,L;\Z)$ with $\beta_1 +\beta_2 = \beta$.
\item
There exists a map $\mu_L : H_2(M,L;\Z) \to  2\Z$,
Maslov index, such that the (virtual) dimension satisfies
the following equality $(\ref{dimensiondisc})$.
\begin{equation}\label{dimensiondisc}
\dim {\mathcal M}_{k+1;\ell}(L;\beta) = n+\mu_L(\beta)
+k -2 + 2\ell.
\end{equation}
\item
We can define orientations of ${\mathcal M}_{k+1;\ell}(L;\beta)$ so that
$(3)$ above is compatible with this orientation
    in the sense of  \cite[Proposition {\rm 8.3.3}]{fooo:book2}.
\item The evaluation map is extended to the
compactification so that it is compatible with $(\ref{eq177})$.
\item $ev_0^{\partial}$ is weakly submersive.
\item
The Kuranishi structure is compatible with the forgetful map  of the
boundary marked points.
\item
The Kuranishi structure is invariant under the permutation of interior marked
points.
\item
The Kuranishi structure is invariant under the cyclic permutation of the
boundary  marked
points.
\end{enumerate}
\end{prop}

\begin{rem} We remark that we require weak submersivity of $ev_0^\partial$, not
for other evaluation maps. This is because we pull back differential forms by $\text{\rm ev}_i^\partial$
for $i \geq 1$ while we push forward them by $\text{\rm ev}_0^\partial$. As well-known,
it is straightforward to pull back differential forms but one needs
some submersivity of the map to push forward them. We would also remark that
one cannot require the evaluation maps at all marked points simultaneously
weakly submersive. (See  \cite[Remark 3.2]{fukaya:cyc} for more explanation for this point.)
\end{rem}

An element of the right hand side of (\ref{eq177}) is drawn in  Figure \ref{Figure12} below.
\begin{figure}[h]
\centering
\includegraphics[scale=0.3]
{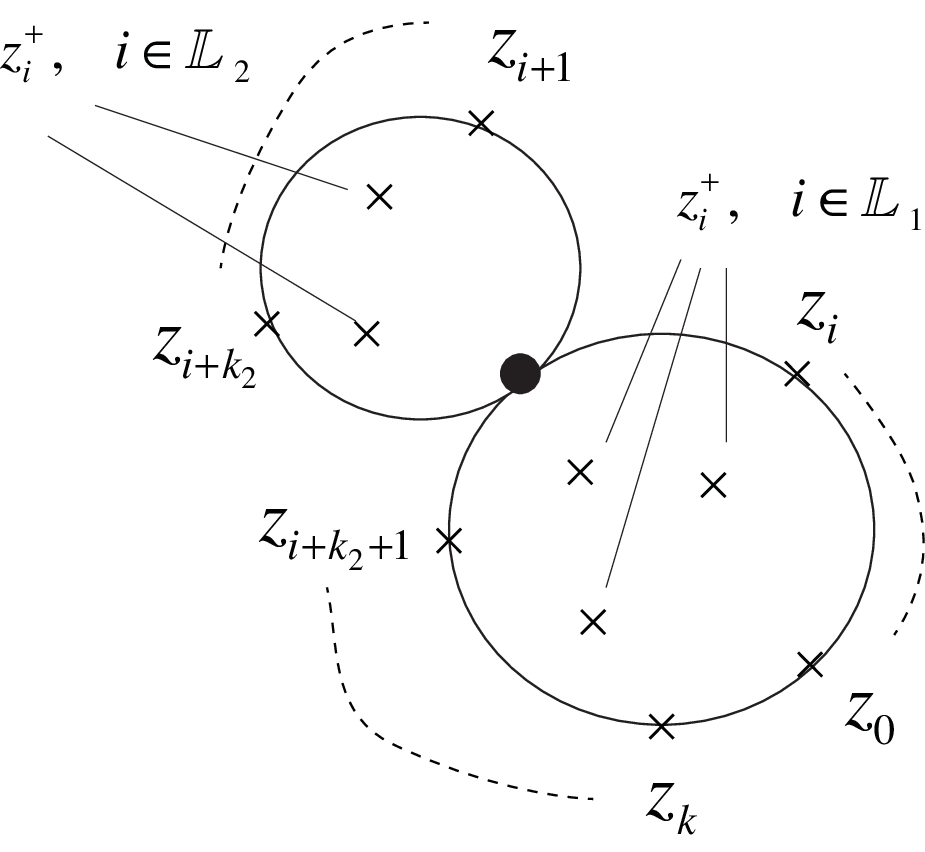}
\caption{An element of the right hand side of (\ref{eq177})}
\label{Figure12}
\end{figure}
Proposition \ref{disckura} (1) - (7) is proved
in \cite[Propositions 7.1.1,7.1.2]{fooo:book2}  for
the case of $\ell =0$. The same proof applies to case $\ell \ne 0$.
The Kuranishi structure satisfying the additional properties (8), (9), (10)
is constructed in \cite[ Corollary 3.1]{fukaya:cyc}.
We refer \cite[Definition 3.1]{fukaya:cyc}  for the precise meaning of the statement
(8).
\begin{lem}\label{existmkulti1}
There exists
a system of CF-perturbations
on the moduli spaces ${\mathcal M}_{k+1;\ell}(L;\beta)$
such that the following holds:
\begin{enumerate}
\item
It is transversal to zero in the sense of Definition $\ref{CFtransv}$.
\item
It is compatible with the
description of the boundary in Proposition $\ref{disckura}$ $(3)$ above.
\item
It is compatible
with the forgetful map of the boundary marked points.
\item
It is compatible
with the permutation of interior marked points.
\item
It is compatible with cyclic permutation of the boundary marked points.
\item
$\text{\rm ev}_0^{\partial}$ restricted to the
zero set of this system of CF-perturbations is a submersion.
In other words, $\text{\rm ev}_0^{\partial}$ is \index{strongly submersive}
strongly submersive
with respect to our  CF-perturbation in the sense of Definition $\ref{CFtransv}$, \cite[Definition 7.48]{fooo:tech2}.
\end{enumerate}
\end{lem}
\par
\begin{proof}
Existence of such a system of CF-perturbations
is established in \cite[Corollary 5.2]{fukaya:cyc}
by an induction over $\beta\cap \omega$ and $\ell$.
(See also \cite[Theorem 7.49]{fooo:tech2}.)
\end{proof}
\begin{rem}
Strictly speaking, we need to fix $E_0$ and $\ell_0$ and restrict ourselves
to those moduli spaces ${\mathcal M}_{k+1;\ell}(L;\beta)$
satisfying $\beta\cap\omega \le E_0$ and $\ell \le \ell_0$,
in order to take care of the problem of `running out'
pointed out in \cite[Subsection 7.2.3]{fooo:book2}.
We can handle this in the same way as in \cite{fooo:book2}.
With the de Rham version it is simpler to resolve this problem than
with the singular homology version used in \cite{fooo:book2}. The detail of this
de Rham version is provided in \cite[Section 14]{fukaya:cyc}
for the case $\ell = 0$. The case $\ell \ne 0$ can be handled
by  the same way with the homological algebra
developed in \cite[Section 7.4]{fooo:book2}.
\end{rem}
Let $g_1,\dots,g_{\ell} \in \Omega(M)$ and
$h_1,\dots,h_k \in \Omega(L)$ and $\beta$ with
$(\beta,\ell) \ne (0,0)$. We define
\begin{equation}\label{defqformula}
\aligned
&\frak q_{\ell,k;\beta}(g_1,\dots,g_{\ell},h_1,\dots,h_k)\\
&=
(\text{\rm ev}^{\partial}_{0})!
\left(
\text{\rm ev}_1^*g_1\wedge \dots \wedge \text{\rm ev}_{\ell}^*g_{\ell}
\wedge
\text{\rm ev}_1^{\partial *}h_1\wedge \dots \wedge \text{\rm ev}_{k}^{\partial *}h_{k}
\right).
\endaligned
\end{equation}
Here we use the evaluation map
$$
(\text{\rm ev},\text{\rm ev}^{\partial})
= (\text{\rm ev}_1,\dots,\text{\rm ev}_{\ell};
\text{\rm ev}^{\partial}_0,\dots,\text{\rm ev}^{\partial}_k)
: {\mathcal M}_{k+1;\ell}(L;\beta)
\to M^{\ell} \times L^{k+1}
$$
as the correspondence via the CF-perturbation of the moduli space ${\mathcal M}_{k+1;\ell}(L;\beta)$
given in Lemma \ref{existmkulti1}.
(See Definition \ref{def320222} and \cite[Definition 9.13]{fooo:tech2}
for the definition of $(\text{\rm ev}^{\partial}_{0}!$.)
For $\beta = \beta_0 = 0$, $\ell = 0$ we put
\begin{equation}\label{defqformula2}
\aligned
&\frak q_{0,k;\beta_0}(h_1,\dots,h_k)
=
\begin{cases}
0 & k\ne 1,2 \\
(-1)^{n+1+\deg h_1} dh_1 & k=1\\
(-1)^{\deg h_1(\deg h_1+1)} h_1 \wedge h_2
& k=2.
\end{cases}
\endaligned
\end{equation}
\par
Theorem \ref{qproperties} (1) is a consequence of
Proposition \ref{disckura} (2)(3),
Stokes' formula (Theorem \ref{them48} and \cite[Theorem 9.26]{fooo:tech2}) and the composition formula
(Theorem \ref{compform}, \cite[Theorem 10.20]{fooo:tech2}).
\par
We may regard (\ref{qism}) as the definition of its right hand side.
So Theorem \ref{qproperties} (2) is obvious.
\par
Theorem \ref{qproperties} (3) is a consequence of
Proposition \ref{disckura} (8) and the compatibility of
CF-perturbations to this forgetful map.
See \cite[Section 7]{fukaya:cyc}  for a detailed explanation of this point.
The proof of Theorem \ref{qproperties} is now complete.
\qed
\par

\begin{rem}\label{rem:deferenceofq}
\begin{enumerate}
\item For $g_1\otimes \dots \otimes g_{\ell} \in B_{\ell}(\Omega(M))$ and $h_1\otimes \dots \otimes h_k \in B_k(\Omega(L))$ we defined
$\frak q_{\ell,k;\beta}(g_1,\dots, g_{\ell},h_1, \dots, h_k)$ by \eqref{defqformula}.
Thanks to Proposition \ref{disckura} (9) this is invariant under the permutation of $g_1, \dots, g_{\ell}$. Thus the operator $\frak q_{\ell,k;\beta}$ descends to the operator
$$
\frak q_{\ell,k;\beta}:
E_{\ell} (\Omega(M)[2]) \otimes B_k(\Omega(L)[1]) \to  \Omega(L)[1].
$$
\item
The right hand side of \eqref{defqformula} is
the same as that   of \cite[Definition 6.10]{fooo:toricmir}
even with the same constant in front, but different from the one
given in   \cite[(3.8.68)]{fooo:book1} and   \cite[(6.10)]{fooo:bulk} in the
coefficient.
In \cite{fooo:book1,fooo:toric1,fooo:bulk}, as we noted in Notations and Conventions (20),
we denoted by $E_{\ell}C$ the
$\text{\rm Perm}(\ell)$-invariant subset of $BC$
and used the deconcatenation coproduct\index{coproduct!deconcatenation} on it.
Indeed, we have the equality
$
\frak q_{\ell,k;\beta}^{\rm book}=
\frac{1}{\ell !} \frak q_{\ell,k;\beta}
$
where we denote by $\frak q_{\ell,k;\beta}^{\rm book}$ the operator given in \cite[(6.10)]{fooo:bulk} (or
in \cite[(3.8.68)]{fooo:book1}).
However, we can see that this difference does not cause any trouble in the proof of Theorem 17.1 by
just noticing the identity
$
{\bf y}_c^{2;1}\otimes {\bf y}_c^{2;2}=
\frac{\ell_1 ! \ell_2 !}{\ell !}
{\bf y}_c^{2;1 \prime}\otimes {\bf y}_c^{2;2\prime}
$
on $E_{\ell_1}C \otimes E_{\ell_2}C$, where
the left (resp. right) hand side is the
$(\ell_1,\ell_2)$-component in the
decomposition of $\Delta_{\rm decon}{\bf y}$ for
the invariant set (resp. $\Delta_{\rm shuff}{\bf y}$ for the quotient space). Here
we identify the quotient set with the invariant subset by the map
$
[y_1\otimes \dots \otimes y_{\ell}] \mapsto
\frac{1}{\ell !} \sum
_{\sigma \in \frak S_{\ell}} (-1)^{\ast}
y_{\sigma(1)}\otimes \dots \otimes y_{\sigma(\ell)}
$
with $\ast=\sum_{i<j;\sigma(i)>\sigma(j)}\deg y_i \deg y_j$.
\end{enumerate}
\end{rem}

We next explain how we use the map $\frak q$ to deform the filtered
$A_{\infty}$ structure $\frak m$ on $L$. In this section we use the
universal Novikov ring $\Lambda_{0}$.

\begin{defn}\label{deformedqdef} Put $\text{\bf b} = (\frak b_0,\text{\bf b}_{2;1},\frak b_+,b_+)$
where
\begin{equation}
\aligned
&\frak b_0  \in  H^0(M;\Lambda_0), \quad
&\text{\bf b}_{2;1}  \in H^2(M,L;\C),\\
&\frak b_+  \in  H^2(M;\Lambda_+)
\oplus \bigoplus_{k\ge 2} H^{2k}(M;\Lambda_0), \quad
&b_+  \in  \Omega^1(L)\widehat\otimes\Lambda_+
\oplus\bigoplus_{k\ge 2} \Omega^{2k-1}(L) \widehat\otimes\Lambda_0.
\endaligned
\nonumber
\end{equation}
Here $ \Omega^1(L)\widehat\otimes\Lambda_+$ is a completion of an algebraic tensor product,
$\Omega^1(L) \otimes\Lambda_+$.

We represent  $\frak b_0$, $\frak b_+$ by
closed differential forms which are denoted by the same letters.
\medskip
For each $k \ne 0$, we define $\frak m_{k}^{\text{\bf b}}$ by
\index{$\frak m_{k}^{\text{\bf b}}$}
\begin{equation}\label{mkdefeq}
\aligned
&\frak m_{k}^{\text{\bf b}}(x_1,\ldots,x_k) \\
&= \sum_{\beta\in H_2(M,L:\Z)}
\sum_{\ell=0}^{\infty}\sum_{m_0=0}^{\infty}\cdots
\sum_{m_k=0}^{\infty}T^{\omega\cap \beta}
\frac{\exp(\text{\bf b}_{2;1} \cap \beta)}{\ell!}\\
&\hskip2cm
\frak
q_{\ell,k+\sum_{i=0}^k m_i;\beta}(\frak b_{+}^{\otimes\ell};
b_{+}^{\otimes m_0},x_1,b_{+}^{\otimes m_1},\ldots,
b_{+}^{\otimes m_{k-1}},x_k,b_{+}^{\otimes m_k}),
\endaligned
\end{equation}\index{bulk deformation!filtered $A_{\infty}$ algebra}
where $x_i \in \Omega(L)$.
We extend it $\Lambda_0$-linearly to $\Omega(L)
\widehat\otimes \Lambda_0$.
\par
For $k=0$, we define $\frak m_0^{\text{\bf b}}$ by
\begin{equation}\label{mkdefeq2}
\aligned
\frak m_{0}^{\text{\bf b}}(1)
= \frak b_0 + \sum_{\beta\in H_2(M,L:\Z)}
\sum_{\ell=0}^{\infty}\sum_{m=0}^{\infty}T^{\omega\cap \beta}
\frac{\exp(\text{\bf b}_{2;1} \cap \beta)}{\ell!}
\frak
q_{\ell,k+m;\beta}(\frak b_{+}^{\otimes\ell};
b_{+}^{\otimes m}).
\endaligned
\end{equation}
Here we embed $H^0(M;\Lambda_0) = \Lambda_0 \subset \Omega^0(L) \widehat\otimes \Lambda_0$
as $\Lambda_0$-valued constant functions on $M$.
\end{defn}

We can prove that the right hand sides in \eqref{mkdefeq}, \eqref{mkdefeq2} converge in
$T$-adic topology in the same way as in Lemma \ref{adiccomv1}.
\begin{rem}
The weight appearing in (\ref{mkdefeq})
is mostly the same as the one appearing in the definition of Gromov-Witten invariants.
(See Remark \ref{rem53}, and \cite[Section 4.1]{fooo:toricmir}.)
\end{rem}
\begin{lem}\label{bulkdef} The family
$\{\frak m^{\text{\bf b}}_{k}\}_{k=0}^{\infty}$ defines a filtered $A_{\infty}$
structure on $\Omega(L) \widehat{\otimes} \Lambda_{0}$.
\end{lem}
\begin{proof}
The proof is a straightforward calculation using
Theorem \ref{qproperties}.
See  \cite[Lemma 3.8.39]{fooo:book1} for the detail of the proof of
such a statement in the purely abstract context.
\end{proof}
We regard the constant function $1$ on $L$ as a differential
zero-form which we denote by $\text{\bf e}_L$.
\begin{defn}\label{bulkMCelement}
\index{$\widehat{\mathcal M}_{\text{\rm weak,def}}(L;\Lambda_{0})$}
Denote by $\widehat{\mathcal M}_{\text{\rm weak,def}}(L;\Lambda_{0})$ the
set of all the elements $\text{\bf b} = (\frak b_0,\text{\bf b}_{2;1},\frak b_+,b_+)$ as in
Definition \ref{deformedqdef} that satisfy the equation
$$
\frak m_0^{\text{\bf b}}(1) = c\text{\bf e}_L
$$
for $c = c(\text{\bf b}) \in \Lambda_+$.
We define $\frak{PO}(\text{\bf b}) \in \Lambda_+$ to be the coefficient $c(\text{\bf b})$, i.e.,
by the equation
$$
\frak m_0^{\text{\bf b}}(1) = \frak{PO}(\text{\bf b}) \text{\bf e}_L.
$$
\index{$\frak{PO}$}
We call the map $ \frak{PO}: \widehat{\mathcal M}_{\text{\rm
weak,def}}(L;\Lambda_{0}) \to \Lambda_+ $ {\it the potential
function}. \index{potential function} \index{$\frak{PO}$} We also define the projection
$$
\pi: \widehat{\mathcal M}_{\text{\rm weak,def}}(L;\Lambda_{0})
\to H^0(M;\Lambda_0) \oplus H^2(M,L;\C) \oplus H^2(M;\Lambda_+)
\oplus \bigoplus_{k\ge 2}  H^{2k}(M;\Lambda_0)
$$
by
$
\pi (\frak b_0,\text{\bf b}_{2;1},\frak b_+,b_+) =
(\frak b_0,\text{\bf b}_{2;1},\frak b_+).
$
\end{defn}
Let $\text{\bf b}^{(i)} =  (\frak b^{(i)}_0,\text{\bf b}^{(i)} _{2;1},\frak b^{(i)} _+,b^{(i)} _+)
\in \widehat{\mathcal M}_{\text{\rm weak,def}}(L;\Lambda_{0,nov}^+)$
($i=1,2$)
such that
$
\pi(\text{\bf b}^{(1)}) = \pi(\text{\bf b}^{(0)}).
$
We define an operator
$
\delta^{\text{\bf b}^{(1)},\text{\bf b}^{(0)}}: \Omega(L)\widehat\otimes\Lambda_0 \to \Omega(L)\widehat\otimes\Lambda_0
$
of degree $+1$ by
$$
\delta^{\text{\bf b}^{(1)},\text{\bf b}^{(0)}}(x)
= \sum_{k_1,k_0} \frak m^{\overline{\text{\bf b}}}_{k_1+k_0+1}((b_+^{(1)})^{\otimes k_1}
\otimes x \otimes (b_+^{(0)})^{\otimes k_0}),
$$
where
$\overline{\text{\bf b}} = (\frak b^{(0)}_0,\text{\bf b}^{(0)} _{2;1},\frak b^{(0)} _+,0)
= (\frak b^{(1)}_0,\text{\bf b}^{(1)} _{2;1},\frak b^{(1)} _+,0)$.
We remark that if $\text{\bf b}_1 = \text{\bf b}_0 = \text{\bf b}$
we have
\begin{equation}
\delta^{\text{\bf b}^{(1)},\text{\bf b}^{(0)}}
= \frak m_1^{\text{\bf b}}.
\end{equation}
\index{$\delta^{\text{\bf b}^{(1)},\text{\bf b}^{(0)}}$}
\begin{lem}\label{delta2}
$$
(\delta^{\text{\bf b}^{(1)},\text{\bf b}^{(0)}}\circ \delta^{\text{\bf b}^{(1)},\text{\bf b}^{(0)}})(x)
= (-\frak{PO}(\text{\bf b}^{(1)}) + \frak{PO}(\text{\bf b}^{(0)})) x.
$$
\end{lem}
\begin{proof}
This is an easy consequence of Theorem \ref{qproperties}.
See \cite[Proposition 3.7.17]{fooo:book1}  for its proof.
\end{proof}

This enables us to give the following definition

\begin{defn}\label{FLoercohbulk}(\cite[Definition 3.8.61]{fooo:book1}.)
\index{$HF((L,\text{\bf b}^{(1)}),(L,\text{\bf b}^{(0)});\Lambda_{0})$}
For a given pair $\text{\bf b}^{(1)}, \text{\bf b}^{(0)} \in \widehat{\mathcal M}_{\text{\rm weak,def}}(L;\Lambda_{0})$
satisfying
$
\pi(\text{\bf b}^{(1)}) = \pi(\text{\bf b}^{(0)}),\quad
\frak{PO}(\text{\bf b}^{(1)}) = \frak{PO}(\text{\bf b}^{(0)}),
$
we define the {\em Lagrangian Floer (co}homology \index{Lagrangian Floer cohomology}
\index{Floer homology! Lagrangian submanifold} by:
$$
HF((L,\text{\bf b}^{(1)}),(L,\text{\bf b}^{(0)});\Lambda_{0})
= \frac{\text{\rm Ker}(\delta^{\text{\bf b}^{(1)},\text{\bf b}^{(0)}})}
{\text{\rm Im}(\delta^{\text{\bf b}^{(1)},\text{\bf b}^{(0)}})}.
$$
\index{$HF((L,\text{\bf b}^{(1)}),(L,\text{\bf b}^{(0)});\Lambda_{0})$}
When $\text{\bf b}^{(1)} = \text{\bf b}^{(0)} = \text{\bf b}$,
we just write
$HF((L,\text{\bf b});\Lambda_{0})$ for simplicity.
\end{defn}

Put $CF_{\text{\rm dR}}(L;\Lambda)=\Omega(L) \widehat\otimes \Lambda$.
Then
$(CF_{\text{\rm dR}}(L;\Lambda), \delta^{\text{\bf b}^{(1)},\text{\bf b}^{(0)}})$
forms a cochain complex. The cochain complex $CF_{\text{\rm dR}}(L;\Lambda)$ carries a natural
filtration given by
\begin{equation}
F^{\lambda}CF_{\text{\rm dR}}(L;\Lambda)
= T^{\lambda}\Omega(L) \widehat\otimes \Lambda_0.
\end{equation}
\begin{lem}
We have
$$
 \delta^{\text{\bf b}^{(1)},\text{\bf b}^{(0)}}
 (F^{\lambda}CF_{\text{\rm dR}}(L;\Lambda))
 \subset F^{\lambda}CF_{\text{\rm dR}}(L;\Lambda).
$$
\end{lem}
\begin{proof}
Since the symplectic area of a pseudo-holomorphic map is nonnegative,
$\beta\cap \omega \geq 0$  if
${\mathcal M}_{k+1;\ell}(L;\beta)$ is nonempty.
Therefore
if $\frak q_{\ell,k;\beta}$ is nonzero then $\beta\cap \omega$ is nonnegative.
The lemma follows from this fact and the definition.
\end{proof}

This enables us to define the following {\em Lagrangian version of
spectral numbers} \index{Lagrangian version of
spectral numbers} associated to $L$.

\begin{defn}\label{defn:ellLx}
For $x \in HF((L,\text{\bf b}^{(1)}),(L,\text{\bf b}^{(0)});\Lambda)$
we put
\index{$\rho_L^{\text{\bf b}^{(1)},\text{\bf b}^{(0)}}(x)$}
\begin{equation}\label{rhofilt}
\aligned
\rho_L^{\text{\bf b}^{(1)},\text{\bf b}^{(0)}}(x)  =
-\sup \{ \lambda \mid& \exists \widehat x \in F^{\lambda}CF_{\text{\rm dR}}(L;\Lambda)),\,
 \delta^{\text{\bf b}^{(1)},\text{\bf b}^{(0)}}(\widehat x) =0, \, \\
& \quad [\widehat x] = x \in HF((L,\text{\bf b}^{(1)}),(L,\text{\bf b}^{(0)});\Lambda)\}.
\endaligned
\end{equation}
\end{defn}
\begin{rem}\label{filtminusrem}
We put minus sign in (\ref{rhofilt}) for the sake of consistency with
Parts 2 and 3. In fact, $\frak v_q = -\frak v_T$ via the
isomorphism $\Lambda^\downarrow \cong \Lambda$.
\end{rem}
We can show
\be\label{eq:ellLbxne}
\rho_L^{\text{\bf b}^{(1)},\text{\bf b}^{(0)}}(x) >- \infty
\ee
if $x \ne 0$. (See \cite{usher:specnumber} or Lemma \ref{lem1816} of this paper
for the detail.)

We next define a closed-open map \index{closed-open map} from the cohomology of the ambient space
to the Floer cohomology of $L$.
Let $\text{\bf b} =
(\frak b_0,\text{\bf b} _{2;1},\frak b_+,b_+)
\in \widehat{\mathcal M}_{\text{\rm weak,def}}(L;\Lambda_{0})$,
take $g \in \Omega(M)$ and define a map $i_{\text{\rm qm},\text{\bf b}}(g):
\Omega(M) \otimes \Lambda_0 \to CF_{\text{dR}}(L;\Lambda_0)$ by
\begin{equation}\label{iqmdefformula}
\aligned
i_{\text{\rm qm},\text{\bf b}}(g)
= & (-1)^{\deg g}\sum_{\beta\in H_2(M,L:\Z)}
\sum_{\ell_1=0}^{\infty}\sum_{\ell_2=0}^{\infty} \sum_{k=0}^{\infty}
T^{\omega\cap \beta}
\frac{\exp(\text{\bf b}_{2;1} \cap \beta)}{(\ell_1+\ell_2+1)!} \times \\
& \qquad \frak q_{\ell_1+\ell_2+1,k;\beta} (\frak b_{+}^{\otimes\ell_1}\otimes g \otimes \frak b_{+}^{\otimes\ell_2};
b_{+}^{\otimes k}).
\endaligned
\end{equation}
Here `\text{\rm qm}' in the subindex of the map $i_{\text{\rm qm},{\bf b}}$
stands for the `quantum effect' or the effect of pseudoholomorphic discs. This effect
also has some interaction with deformation parameter ${\bf b}$.

It follows in the same way as in Lemma \ref{adiccomv1} that the right hand side converges in
$T$-adic topology.

\begin{lem}
The map $i_{\text{\rm qm},\text{\bf b}}$ is a chain map.
Namely,
$$
\delta^{{\text{\bf b}},\text{\bf b}}\circ i_{\text{\rm qm},\text{\bf b}}
= \pm i_{\text{\rm qm},\text{\bf b}} \circ d.
$$
\end{lem}
\begin{proof}
This is a consequence of Theorem \ref{qproperties}.
See \cite[Theorem 3.8.62]{fooo:book1}.
We recall from   \cite[Remark 3.5.8]{fooo:book1} that
$\frak m_{1;\beta_0}$ in \eqref{qism} satisfies
$\frak m_{1;\beta_0}(h)=(-1)^{n+\deg h +1}dh$ for $h \in \Omega^{\deg h}(M)$.
\end{proof}
We thus obtain a homomorphism
\begin{equation}\label{ambcohtoHFL}\index{$i_{\text{\rm qm},\text{\bf b}}^{\ast}$}
i_{\text{\rm qm},\text{\bf b}}^{\ast}
:
H^*(M;\Lambda_0) \to HF^*((L,\text{\bf b});\Lambda_{0}).
\end{equation}
\begin{rem}
The homomorphism (\ref{ambcohtoHFL}) is indeed a ring homomorphism. 
It is proved in \cite[Section 2.6]{fooo:toricmir}  for the toric case. See \cite[Section 4.7]{fooo:toricmir}
 and \cite{AFOOO} for the
general case.
\end{rem}

Composing the map $i_{\text{\rm qm},\text{\bf b}}^{\ast}$with $\rho_L^{{\bf b}^{(1)}, {\bf b}^{(0)}}$
in  Definition \ref{defn:ellLx}, we introduce:

\begin{defn}
For each $0 \neq a \in H^*(M;\Lambda)$, we define
\begin{equation}\label{rhoforL}
\rho^{\mathbf b}_L(a)
= \rho_L^{\text{\bf b},\text{\bf b}}(i_{\text{\rm qm},\text{\bf b}}^{\ast}(a))
\end{equation}
for
$\text{\bf b} \in  \widehat{\mathcal M}_{\text{\rm weak,def}}(L;\Lambda_{0})$.
\end{defn}

Therefore by the finiteness \eqref{eq:ellLbxne},
$\rho_L^{\mathbf b}(a) > -\infty$  for any $a \neq 0$, provided there exists some
element $\text{\bf b} \in \widehat{\mathcal M}_{\text{\rm weak,def}}(L;\Lambda_{0})$
such that $i_{\text{\rm qm},\text{\bf b}}^{\ast}(a) \ne 0$.

\section{Criterion for heaviness of Lagrangian submanifolds}
\label{sec:heavy}

In this section, we incorporate the Lagrangian Floer theory into
the theory of spectral invariants and Calabi quasi-morphisms
of Hamiltonian flows and symplectic quasi-states.

\subsection{Statement of the results}
\label{subsec:stateheavyness}

We review the notions of heavy and
superheavy subsets of a symplectic manifold $(M,\omega)$ introduced by
Entov and Polterovich \cite[Definition 1.3]{EP:rigid}.
(See also \cite{albers, biran-cor} for some related results.)

\begin{defn}\index{heavy subset}\index{superheavy subset} Let $\zeta$ be a partial symplectic quasi-state on $(M,\omega)$.
A closed subset $Y \subset M$ is called \emph{$\zeta$-heavy} if
\be\label{eq:zetaheavy}
\zeta(H) \leq \sup \{H(p) \mid p \in Y\}
\ee
for any $H \in C^0(M)$.
\par
A closed subset $Y \subset M$ is called {\it $\zeta$-superheavy} if
\be\label{eq:zetasheavy}
\zeta(H) \geq \inf \{H(p)\mid  p \in Y\}
\ee
for any  $H \in C^0(M)$.
\end{defn} \index{heavy subset!$\zeta$-heavy}\index{superheavy subset!$\zeta$-superheavy}

\begin{rem}\label{rem:heavysuperheavy}
\begin{enumerate}
\item
Due to the different sign conventions from \cite{EP:rigid} as mentioned in Subsection \ref{subsec:conv}, Remark \ref{rem:EPversusFOOO}
and also because we use quantum \emph{cohomology} class $a$ in the definition
of the spectral invariants $\rho(H;a)$, the above definition looks opposite to
that of \cite{EP:rigid}. However after taking these
different convention and usage, this definition of heaviness or of superheaviness
of a given subset $S \subset (M,\omega)$ indeed is equivalent to
that of \cite{EP:rigid}.
\item
Following the proof of \cite[Proposition 4.1]{EP:rigid},
we can obtain
a characterization of a $\zeta$-heavy set or a $\zeta$-superheavy set as follows:
A closed subset $Y \subset M$ is $\zeta$-heavy if and only if for every
$H \in C^{\infty}(M)$ with $H\vert_{Y}=0$, $H\ge 0$
one has $\zeta(H)=0$.
A closed subset $Y \subset M$ is $\zeta$-superheavy if and only if for every
$H \in C^{\infty}(M)$ with $H\vert_{Y}=0$, $H\le 0$
one has $\zeta(H)=0$.
Due to the different sign convention again,
this statement is in a slightly different form  \cite[Proposition 4.1]{EP:rigid}.
Using this characterization and our triangle inequality
Definition \ref{defn:zeta} (8) and the monotonicity (3),
we can show that every $\zeta$-superheavy subset is $\zeta$-heavy. This is nothing but  \cite[Proposition 4.2]{EP:rigid}.
\item
Furthermore, we can show
\cite[Proposition 4.3]{EP:rigid} as it is.
Namely for any $\zeta$-superheavy
set $Y$, and any $\alpha \in \R$ and
$H\in C^{\infty}(M)$ with $H\vert_{Y}=\alpha$ we have $\zeta(H)=\alpha$.
\item
 \cite[Entov-Polterovich Theorem 1.4 (iii)]{EP:rigid} proved that
for any partial symplectic quasi-state $\zeta$,
every $\zeta$-superheavy set intersects every $\zeta$-heavy subset. See Theorem \ref{sheavyintersectheavy}.
\end{enumerate}
\end{rem}

The definitions of heaviness and super-heaviness \cite{EP:rigid} involve only
{\it autonomous} Hamiltonian. We first enhance the definition by involving
time-dependent Hamiltonian. For this purpose, the following definition
is useful.
\begin{defn}
Let $H: [0,1] \times M \to \R$ be a Hamiltonian and $Y \subset M$ be a closed
subset.
We put $H_t(x) = H(t,x)$. For such a pair $(H,Y)$ we associate two constants $E^\pm(H;Y)$
by
\index{$E^-(H;Y)$}\index{$E^+(H;Y)$}\index{$E(H;Y)$}
\begin{equation}\label{eq:EHYpm}
\aligned
E^-(H;Y) & =  \int_0^1 - \min (H_t|_Y) \, dt =
 \int_0^1 \max (- H_t|_Y) \, dt \\
E^+(H,Y) & =  \int_0^1 \max (H_t|_Y) \, \\
E(H;Y) & =  E^-(H;Y) + E^+(H;Y).
\endaligned
\end{equation}
\end{defn}
We remark that when $Y = M$, $E^\pm(H;M)$ often appear in relation to the
energy estimate and Hofer geometry (see \cite[Theorem 3.1]{oh:dmj}  for example),
and $E(H;M)$ is the Hofer norm, $\|H\|$.

We note
$$
E^\pm(\underline H;Y) = E^\pm(H;Y)
\mp \frac{1}{\vol_\omega(M)} \Cal(H)
$$
and so
\begin{equation}\label{noneedtonormalize}
E^-(H;Y) + E^+(H;Y) = E^-(\underline H;Y) + E^+(\underline H;Y)
\end{equation}
depend only on the Hamiltonian path $\phi_H$, but not on the normalization constant.
\begin{defn}
For $\widetilde \psi \in \widetilde{\Ham}(M,\omega)$, we define
\begin{equation}\label{eq:eHYpm}
\aligned
e^-(\widetilde \psi;Y) & = \inf_{H}\{E^-(\underline H;Y) \mid \widetilde \psi = [\phi_H]\}
\\
e^+(\widetilde \psi;Y) & = \inf_{H}\{E^+(\underline H;Y) \mid \widetilde \psi = [\phi_H]\}
\\
e(\widetilde \psi;Y) & = \inf_H\{E(H;Y) \mid \widetilde \psi = [\phi_H]\}.
\endaligned
\end{equation}
Note thanks to (\ref{noneedtonormalize}) we do not need to normalize
$H$ in the definition of $e(\widetilde \psi;Y)$.
\end{defn}
\index{$e^-(\widetilde \psi;Y)$}\index{$e^+(\widetilde \psi;Y)$}\index{$e(\widetilde \psi;Y)$}
We note $e(\widetilde \psi;Y) \geq e^+(\widetilde \psi;Y) + e^-(\widetilde \psi;Y)$.

\begin{defn}\label{tdheavy}
Let $\mu: \widetilde{\Ham}(M,\omega) \to \R$ be an
partial quasi-morphism.
A closed subset $Y \subset M$ is called \emph{$\mu$-heavy} if
we have
\be\label{eq:muheavy}
- \mu(\widetilde \psi) \le \vol_\omega(M)
e^+(\widetilde \psi;Y)
\ee
for any $\widetilde \psi$.
A closed subset $Y \subset M$ is called \emph{$\mu$-superheavy} if
we have
\be\label{eq:musheavy}
- \mu(\widetilde \psi) \ge - \vol_\omega(M) e^-(\widetilde \psi;Y)
\ee
for any $\widetilde \psi$.
\end{defn}\index{heavy subset!$\mu$-heavy}\index{superheavy subset!$\mu$-superheavy}
}
\begin{rem}\label{rem:muimplieszeta}
We note that our definition of $\mu$-heaviness
is given in terms of the universal covering space of $\text{Ham}(M,\omega)$ while the $\zeta$-heaviness
is in terms of the autonomous functions.
Now consider $\widetilde \psi=[\phi_H]$ for an autonomous $H$. Then by (\ref{eq:muzetae}), we derive
$$
- \mu(\widetilde \psi) = - \mu(\phi_H) = \vol_\omega(M) \zeta(\underline H)
= \left( \vol_\omega(M) \zeta(H) - \Cal(H) \right)
$$
for autonomous $H$. On the other hand, we also have
$$
-e^+(\widetilde \psi;Y) \ge -E^+(\underline H;Y) =  -E^+(H;Y)
+ \frac{1}{\vol_\omega(M)} \Cal(H)
$$
for arbitrary time-dependent Hamiltonian $H$. Therefore $\mu$-heaviness of $L$ implies
$\zeta$-heaviness of $L$.
Similarly, we can also see that $\mu$-superheaviness implies $\zeta$-superheaviness.
However, since not every element $\widetilde \psi$ can be realized by an autonomous Hamiltonian,
a priori the $\mu$-heaviness (resp. $\mu$-super-heaviness)
is a stronger notion than the $\zeta$-heaviness (resp. $\zeta$-super-heaviness).
In fact, the definition of $\mu$-heaviness can be
given by replacing the right hand side of (\ref{eq:muheavy}) by some invariant defined in terms of the
loop space of the Lagrangian submanifolds. (See Section 25.4 for the related remark.)
It is an interesting problem to further investigate their relationship.
\end{rem}

The following result is due to Entov-Polterovich \cite{EP:rigid} which will be used later in Section \ref{sec:exotic}.
We give a proof for reader's convenience.
\begin{thm}{\rm (\cite[Theorem 1.4]{EP:rigid})}\label{sheavyintersectheavy}
Let $\zeta$ be a partial symplectic quasi-state.
If $Y \subset M$ is $\zeta$-superheavy and
$Z \subset M$ is $\zeta$-heavy, then for any
$\psi \in \text{\rm Symp}_0(M,\omega)$ we have
$$
\psi(Y) \cap Z \ne \emptyset.
$$
\end{thm}
\begin{proof}
Since superheaviness is invariant under symplectic  diffeomorphisms contained in
$\text{\rm Symp}_0(M,\omega)$, we may assume that
$\psi$ is the identity map.
Suppose $Y \cap Z = \emptyset$.
We define $H : M\to \R$ such that $H = 1$ on $Y$ and $H = -1$ on $Z$.
Then since $Y$ is $\zeta$-superheavy,
$
\zeta(H) \ge \inf\{ H(y) \mid y \in Y\} = 1.
$
On the other hand, since  $Z$ is $\zeta$-heavy,
we have
$
\zeta(H) \le \sup \{ H(z) \mid z \in Z\} = -1.
$
This is a contradiction.
\end{proof}

Now the following is the main theorem of this paper whose proof is completed in
Subsection \ref{subsec:Heavynesscomplete}.

\begin{thm}\label{thm:heavy} Let $L\subset M$ be a relatively spin compact Lagrangian submanifold,
and $\text{\bf b} = (\frak b_0,\text{\bf b}_{2;1},\frak b_+,b_+)
\in \widehat{\mathcal M}_{\text{\rm weak,def}}(L;\Lambda_{0})$ as in
Definition $\ref{bulkMCelement}$.
We put
$$
\frak b = i^*(\text{\bf b}_{2;1}) + \frak b_+ \in H^{{\rm even}}(M;\Lambda_0),
$$
where $i^* : H^{2}(M,L;\Lambda_0) \to H^{2}(M;\Lambda_0)$ is the natural homomorphism.
Let $e \in H^*(M;\Lambda)$.
\begin{enumerate}
\item
If $e \cup^{\frak b} e = e$ and
\begin{equation}\label{localnonzero}
0 \ne i_{\text{\rm qm},\text{\bf b}}^{\ast}(e) \in HF^*((L,\text{\bf b});\Lambda),
\end{equation}
then $L$ is $\zeta_e^{\frak b}$-heavy and is $\mu_e^{\frak b}$-heavy.
\item
If there is a direct factor decomposition $QH_{\frak b}^*(M;\Lambda)
\cong \Lambda \times Q'$ as a ring and $e$ comes from the unit of the direct factor $\Lambda$
which satisfies \eqref{localnonzero},
then $L$ is $\zeta_e^{\frak b}$-superheavy and is $\mu_e^{\frak b}$-superheavy.
\end{enumerate}
\end{thm}

\subsection{Floer homologies of periodic Hamiltonians
and of Lagrangian submanifolds}
\label{subsec:constructhomo}

The main part of the proof of Theorem \ref{thm:heavy}
is the proof of the following proposition.

\begin{prop}\label{182mainprop}
Let $L$, $\text{\bf b}$, and  $\frak b$ be as in Theorem $\ref{thm:heavy}$
and $a \in QH^*_{\frak b}(M;\Lambda)$.
Then
\begin{equation}
\rho^{\frak b}(H;a) \ge -E^+(H;L) + \rho^{\mathbf b}_L(a)
\end{equation}
for any Hamiltonian $H$. Here $\rho_L^{\mathbf b}(a)$ is as in $(\ref{rhoforL})$.
Equivalently, we have
\begin{equation}
\rho^{\frak b}(\widetilde \psi;a) \ge -e^+(\widetilde \psi;L) + \rho^{\mathbf b}_L(a).
\end{equation}
\end{prop}

For the proof of Proposition \ref{182mainprop} we
will use a map $\frak I_{(H,J)}^{\text{\bf b},\frak b}: CF(M,H;\Lambda^{\downarrow})
\to
CF_{\text{dR}}(L;\Lambda^{\downarrow})
$
to be introduced in Definition \ref{def:I} the main properties of which we state
in Propositions \ref{energyHH} and \ref{184mainprop}.
The proof of Proposition \ref{182mainprop} will be completed in
Subsection \ref{subsec:Heavynesscomplete}.

To give the definition of the map $\frak I_{(H,J)}^{\text{\bf b},\frak b}$,
we start with introducing a moduli spaces relating the Hamiltonian periodic orbits of
$H$ and Lagrangian submanifold $L$ as Albers did in \cite{albers}. We refer readers
to Remark \ref{rem:albers} for detailed remark related to this moduli space and its usage.
See also \cite[Subsection 4.7]{fooo:toricmir}.
We recall that we fix a \emph{$t$-independent} $J$ throughout in Part 4.

\begin{defn}
Let $[\gamma,w] \in \text{\rm Crit}(\mathcal A_H)$ and
$\beta \in H_2(M,L;\Z)$.
We denote by $\overset{\circ}{\mathcal M}_{k+1;\ell}(H, J;[\gamma,w],L;\beta)$
\index{$\mathcal M_{k+1;\ell}(H, J;[\gamma,w],L;\beta)$}
the set of all triples
$(u;z_1^{+},\dots,z_{\ell}^+;z_0,\dots,z_k)$ satisfying the following:
\begin{enumerate}
\item
$u : (-\infty,0] \times S^1 \to M$ is a smooth map
such that $u(0,t) \in L$.
\item
The map $u$ satisfies the equation
\be\label{eq:HJCR18sec}
\dudtau + J \Big(\dudt -X_{H_t} (u)\Big) = 0.
\ee
Here $H_t(x) = H(t,x)$.
\item The energy
$$
E_{(H,J);L}= \frac{1}{2} \int \Big(\Big|\dudtau\Big|^2_{J} + \Big|
\dudt - X_{H_t}(u)\Big|_{J}^2 \Big)\, dt\, d\tau
$$
is finite.
\item The map $u$ satisfies the following asymptotic boundary condition
\begin{equation}
\lim_{\tau\to -\infty}u(\tau, t) = \gamma(t).
\end{equation}
\item $z_1^{+},\dots,z_{\ell}^+$ are points in $(-\infty,0)\times S^1$
which are mutually distinct.
\item
$z_0,\dots,z_k$ are mutually distinct points on the boundary
which are ordered counterclockwise on $S^1$
with respect to the boundary orientation coming from
$(-\infty, 0] \times S^1$.
We always set $z_0 = (0,0)$.
\item The homology class of the concatenation of $w$ and $u$ is $\beta$.
\end{enumerate}
\end{defn}
\begin{figure}[h]
\centering
\includegraphics[scale=0.3]
{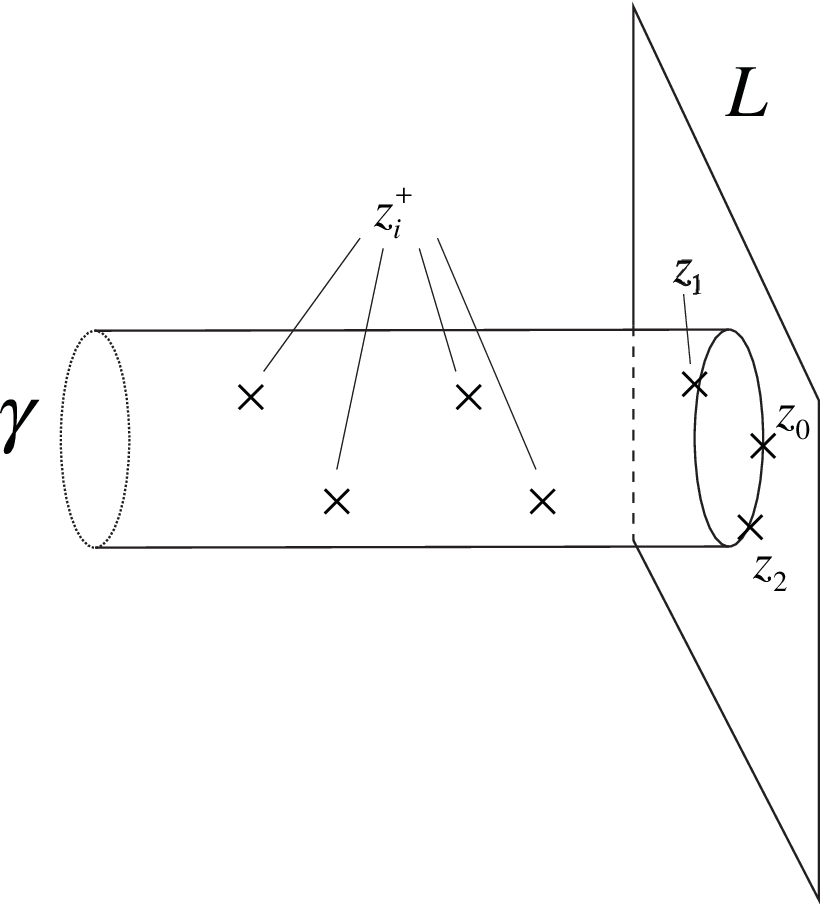}
\caption{An element of $\overset{\circ}{\mathcal M}_{k+1;\ell}(H, J;[\gamma,w],L;\beta)$}
\label{Figure13}
\end{figure}
For any $\alpha \in \pi_2(\gamma)$ we have a canonical homeomorphism
\begin{equation}\label{formula1813}
\overset{\circ}{\mathcal M}_{k+1;\ell}(H, J;[\gamma,\alpha\# w],L;\beta)
\cong \overset{\circ}{\mathcal M}_{k+1;\ell}(H, J;[\gamma,w],L;\beta+\alpha).
\end{equation}
In particular if $w$ is homologous to $w'$ we have
\begin{equation}\label{formula18132}
\overset{\circ}{\mathcal M}_{k+1;\ell}(H, J;[\gamma,w],L;\beta)
\cong \overset{\circ}{\mathcal M}_{k+1;\ell}(H, J;[\gamma,w'],L;\beta).
\end{equation}
We define an evaluation map
$$
(\text{\rm ev},\text{\rm ev}^{\partial})
= (\text{\rm ev}_1,\dots,\text{\rm ev}_{\ell};
\text{\rm ev}^{\partial}_0,\dots,\text{\rm ev}^{\partial}_k)
: \overset{\circ}{\mathcal M}_{k+1;\ell}(H,J;[\gamma,w],L;\beta)
\to M^{\ell} \times L^{k+1}
$$
where
$$
\text{\rm ev}_i([u;z_1^{+},\dots,z_{\ell}^+;z_0,\dots,z_k])
= u(z_i^+),
\quad
\text{\rm ev}^{\partial}_i([u;z_1^{+},\dots,z_{\ell}^+;z_0,\dots,z_k])
= u(z_i).
$$
\begin{lem}\label{FBULKkura18}
\begin{enumerate}
\item
The moduli space
$\overset{\circ}{\mathcal M}_{k+1;\ell}(H,J;[\gamma,w],L;\beta)$ has a compactification
${\mathcal M}_{k+1;\ell}(H,J;[\gamma,w],L;\beta)$ that is Hausdorff.
\item
The space ${\mathcal M}_{k+1;\ell}(H,J;[\gamma,w],L;\beta)$ has an orientable Kuranishi structure with corners.
\item
The normalized boundary of ${\mathcal M}_{k+1;\ell}(H,J;[\gamma,w],L;\beta)$ is described by
the union of the following two types of fiber or direct products.
\begin{equation}\label{bdryFFhamLkinmodu1}
\bigcup {{\CM}}_{\#\mathbb L_1}(H,J;[\gamma,w],[\gamma',w'])
\times {\CM}_{k+1;\#\mathbb L_2}(H,J;[\gamma',w'],L;\beta),
\end{equation}
where the union is taken over all $(\gamma',w') \in \text{\rm Crit}(\mathcal A_H)$,
 and $(\mathbb L_1,\mathbb L_2) \in \text{\rm Shuff}(\ell)$.
\begin{equation}\label{bdryFFhamLkinmodu2}
\bigcup {\mathcal M}_{k_1+1;\# \L_1}(L;\beta_1)
{}_{\text{\rm ev}^{\partial}_{0}}\times_{\text{\rm ev}^{\partial}_i} {\mathcal M}_{k_2+1;\# \L_2}
(H,J;[\gamma,w],L;\beta_2),
\end{equation}
where the union is taken over all
$(\mathbb L_1,\mathbb L_2) \in \text{\rm Shuff}(\ell)$, $k_1,k_2$ with
$k_1 + k_2 = k$, $i \le k_2$, and $\beta_1,\beta_2$ with $\beta_1 + \beta_2 = \beta$.
(See $(\ref{shuff})$ for the notation $\text{\rm Shuff}(\ell)$.)
\item
Let $\mu_H : \mbox{\rm Crit}(\CA_H) \to  \Z$
be the Conley-Zehnder index and
$\mu_L : H_2(M,L;\Z) \to  2\Z$ the Maslov index. Then the (virtual) dimension
is given by
\begin{equation}\label{dimensionFL}
\dim {\mathcal M}_{k+1;\ell}(H,J;[\gamma,w],L;\beta) =
\mu_L(\beta) - \mu_{H}([\gamma,w]) +2\ell + k -2 + n.
\end{equation}
\item
We can define a system of orientations on the moduli spaces
${\CM}_{k+1;\ell}(H,J;[\gamma,w],L;\beta)$ that is compatible with the
isomorphism $(3)$ above.
The compatibility for the boundary of type
\eqref{bdryFFhamLkinmodu2} is in the sense of
\cite[Proposition {\rm 8.3.3}]{fooo:book2}.
\item The evaluation map
$({\rm ev}, {\rm ev}^{\partial})$ given above extends to a strongly continuous smooth map
$$
{\CM}_{k+1;\ell}(H,J;[\gamma,w],L;\beta) \to M^{\ell} \times L^{k+1},
$$
which we denote also by the
same symbol. It is compatible with $(3)$.
\item $ev_0^{\partial}$ is weakly submersive.
\item
The Kuranishi structure is compatible with the forgetful map of the
boundary marked points.
\item
The Kuranishi structure is invariant under the permutation of the interior marked
points.
\item
The Kuranishi structure is invariant under the cyclic permutation of the
boundary  marked
points.
\item
The homeomorphisms $(\ref{formula1813})$, $(\ref{formula18132})$ extend to the compactifications
and their Kuranishi structures are identified by the homeomorphisms.
\end{enumerate}
\end{lem}
Elements of (\ref{bdryFFhamLkinmodu1}) and (\ref{bdryFFhamLkinmodu2}) are
drawn in Figure \ref{Figure14} and \ref{Figure15} below, respectively.
\begin{center}
\begin{figure}[h]
 \begin{tabular}{cc}
 \begin{minipage}[t]{0.45\hsize}
\centering
\includegraphics[scale=0.3]
{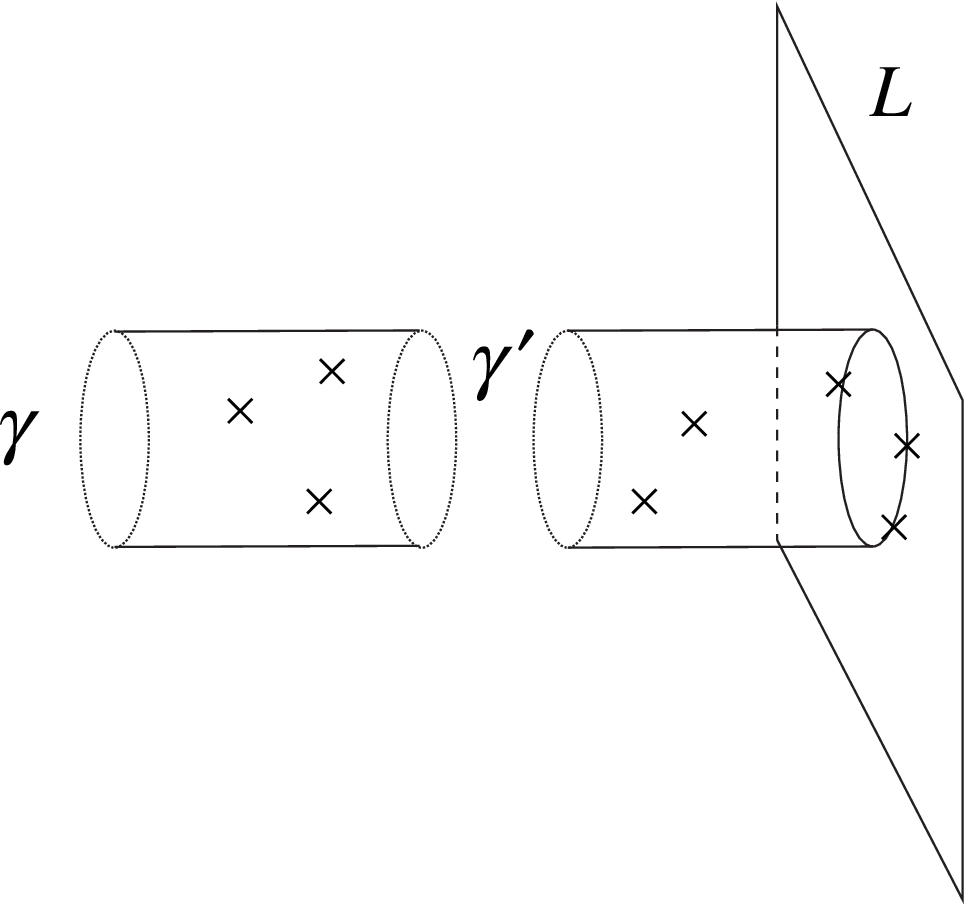}
\caption{An element of (\ref{bdryFFhamLkinmodu1})}
\label{Figure14}
\end{minipage} &
 \begin{minipage}[t]{0.45\hsize}
\centering
\includegraphics[scale=0.3]
{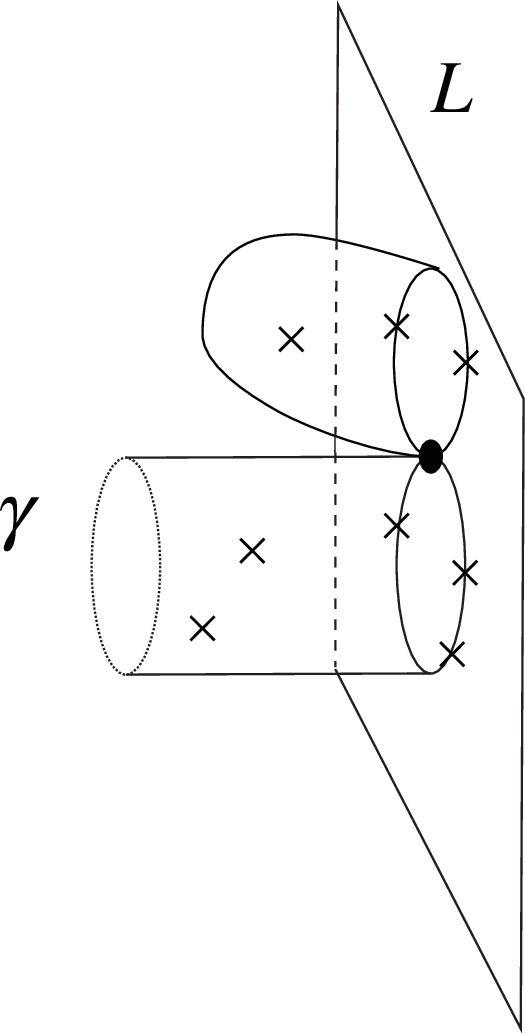}
\caption{An element of (\ref{bdryFFhamLkinmodu2})}
\label{Figure15}
\end{minipage}
\end{tabular}
\end{figure}
\end{center}
The proof of Lemma \ref{FBULKkura18}
is the same as those of Propositions \ref{connkura},
\ref{disckura}, which are detailed in \cite[Parts 4 and 5]{fooo:techI}, and so omitted.
\begin{rem}\label{rem:albers}
\begin{enumerate}
\item
The same moduli space was used by Albers \cite{albers} in the
monotone case. According to Entov-Polterovich \cite[p.779]{EP:rigid},
their motivation to define the heaviness comes from
\cite{albers}.
See also \cite{fooo:analysis} for the relevant analytic details needed for the construction of
the moduli space.
Note that \cite{fooo:analysis} treats the case when $H=0$. However, as is explained in \cite[Section 30]{fooo:techI}
(see especially \cite[Lemma 30.24]{fooo:techI}) the analytic detail given in \cite{fooo:analysis}
can be applied to Floer's equation (\ref{eq:HJCR18sec}) without change.
\item We refer to  \cite[Section 4.7]{fooo:toricmir} for a usage of the map $i_{\text{\rm qm},{\bf b}}$
in the study of mirror symmetry.
In the monotone case for ${\bf b} = 0$, the map $i_{\text{\rm qm},{\bf b}}$ coincides with
the map considered by Albers \cite{albers} and also by Biran-Cornea \cite{biran-cor}.
We refer to \cite[Remark 4.7.8]{fooo:toricmir} for further details on this relationship.
\end{enumerate}
\end{rem}
\begin{lem}\label{existmultiFF}
There exists a system of CF-perturbations on our
moduli space
${\mathcal M}_{k+1;\ell}(H,J;[\gamma,w],L;\beta)$
with the following properties.
\begin{enumerate}
\item
It is transversal to $0$.
\item
It is compatible with the description of the boundary in Proposition
$\ref{FBULKkura18}$ $(3)$.
\item
The evaluation map $\text{\rm ev}_0^{\partial}$ is strongly submersive
with respect to our CF-perturbation in the sense of
Definition $\ref{CFtransv}$ or \cite[Definition 7.48]{fooo:tech2}.
\item
It is compatible with forgetful map of the
boundary marked points.
\item
It is invariant under the permutation of the interior marked
points.
\item
It is invariant under the cyclic permutation of the
boundary  marked
points.
\end{enumerate}
\end{lem}
The compatibility in item (2) above is described as follows.
The description of the boundary in Proposition
$\ref{FBULKkura18}$ $(3)$ identifies the normalized boundary
of ${\mathcal M}_{k+1;\ell}(H,J;[\gamma,w],L;\beta)$ with the fiber product
of similar moduli spaces.
Using the fact that ${\rm ev}^{\partial}_0$ is strongly submersive etc. we can
define the fiber product CF-perburbation. (See Subsection \ref{subsec:compcorrespondence}.)
The compatibility here means that the restriction of the CF-perturbation
of ${\mathcal M}_{k+1;\ell}(H,J;[\gamma,w],L;\beta)$ to its boundary
coincides with the fiber product CF-perturbation.

The proof of Lemma \ref{existmultiFF} is similar to the proof of
Lemma \ref{existmkulti1}
(and also that of \cite[Corollary 5.2]{fukaya:cyc}  and \cite[Theorem 7.49]{fooo:tech2}) and so omitted.
\par
Let $CF(M,H;\C)$ be the $\C$ vector space over the basis $\text{\rm Crit}(\mathcal A_H)$.
We use this moduli space to define an operator \index{$\frak q^{H;[\gamma,w]}_{\ell,k;\beta}$}
for $[\gamma,w] \in \text{\rm Crit}(\mathcal A_H)$ and $\beta
\in H_2(M,L;\Z)$.
\begin{equation}\label{1814}
\frak q^{H;[\gamma,w]}_{\ell,k;\beta}
: E_{\ell} (\Omega(M)[2])
\otimes  B_k(\Omega(L)[1])
\to \Omega(L)[1]
\end{equation}
as follows.
Let $g_1,\dots,g_{\ell} \in \Omega(M)$ and $h_1,\dots,h_k \in \Omega(L)$. We define
\begin{equation}\label{defqformula2}
\aligned
&\frak q^{H;[\gamma,w]}_{\ell,k;\beta}(g_1,\dots,g_{\ell};h_1,\dots,h_k)\\
&=
(\text{\rm ev}^{\partial}_{0})!
\left(
\text{\rm ev}_1^*g_1\wedge \dots \wedge \text{\rm ev}_{\ell}^*g_{\ell}
\wedge
\text{\rm ev}_1^{\partial *}h_1\wedge \dots \wedge \text{\rm ev}_{k}^{\partial *}h_{k}
\right).
\endaligned
\end{equation}
Here we use the evaluation map
$$
(\text{\rm ev},\text{\rm ev}^{\partial})
= (\text{\rm ev}_1,\dots,\text{\rm ev}_{\ell};
\text{\rm ev}^{\partial}_0,\dots,\text{\rm ev}^{\partial}_k)
: {\mathcal M}_{k+1;\ell}(H,J;[\gamma,w],L;\beta)
\to M^{\ell} \times L^{k+1}
$$
and the correspondence given by this moduli space
via our CF-perturbation. (See Part \ref{part7} and \cite[Definition 7.78]{fooo:tech2}.)
The next proposition states the main property of this operator.
\begin{prop}\label{prop:qH-relation} The operators $\frak q^{H;[\gamma,w]}_{\ell,k;\beta}$ have the following properties:
\begin{enumerate}
\item
$\mathfrak q_{\ell,k;\beta}^{H;[\gamma,w]}$ satisfies
\begin{equation}\label{eq:qH-relation}
\aligned
0 & = \sum_{\beta_1 + \beta_2 = \beta}\sum_{c_1,c_2} (-1)^*
\frak q_{\beta_1}(\text{\bf y}^{2;1}_{c_1};
\text{\bf x}^{3;1}_{c_2} \otimes
\frak q_{\beta_2}^{H;[\gamma,w]}\left(\text{\bf y}^{2;2}_{c_1};\text{\bf x}^{3;2}_{c_2}) \otimes \text{\bf x}^{3;3}_{c_2}\right)\\
& \quad +\sum_{c_1,c_2}\sum_{[\gamma'w']
\in {\rm Crit}(\mathcal A_H)}(-1)^{**}\frak n_{(H,J);\vert \text{\bf y}^{2;2}_{c_2}\vert}
([\gamma,w],[\gamma',w'])(\text{\bf y}_{c_2}^{2;2})
\\
&\qquad\qquad\qquad\qquad\qquad\qquad\qquad\qquad
\frak q_{\beta}^{H;[\gamma',w']}\left(\text{\bf y}_{c_1}^{2;1};
\text{\bf x}\right)
\endaligned
\end{equation}
where
$
* = \deg'\text{\bf x}^{3;1}_{c_2} +
\deg'\text{\bf x}^{3;1}_{c_2} \deg \text{\bf y}^{2;2}_{c_1}
+\deg \text{\bf y}^{2;1}_{c_1},
** = \deg  \text{\bf y}^{2;1}_{c_1}.
$
\par
The number
$\frak n_{(H,J);\vert \text{\bf y}^{2;2}_{c_2}\vert}
([\gamma,w],[\gamma',w'])(\text{\bf y}^{2;2}_{c_2}) \in \C$ is defined in
$(\ref{eq:nww'C})$.
(Here $\vert \text{\bf y}^{2;2}_{c_2}\vert$  is defined by
$\text{\bf y}^{2;2}_{c_2} \in E_{\vert \text{\bf y}^{2;2}_{c_2}\vert}(\Omega(M)[2])$.)
\par
In $\eqref{eq:qH-relation}$ and hereafter, we simplify our notation by
writing $\frak q^{H;[\gamma,w]}_{\beta}(\text{\bf y};\text{\bf x})$,  $\frak q_{\beta}(\text{\bf y};\text{\bf x})$ in place of
$\frak q_{\ell,k;\beta}^{H;[\gamma,w]}(\text{\bf y};\text{\bf x})$,
$\frak q_{\ell,k;\beta}(\text{\bf y};\text{\bf x})$ if
$\text{\bf y} \in E_{\ell}(\Omega(M)[2])$, $\text{\bf x} \in B_{k}(\Omega(L)[1])$.
We use the  notation $(\ref{deltawritebyc})$ here.

\item Let $\text{\bf e}_L$ be the constant function $1$ which we regard degree $0$ differential
form on $L$. Let $\text{\bf x}_i \in B(H(L;R)[1])$ and we put
$\text{\bf x} = \text{\bf x}_1 \otimes \text{\bf e}_L \otimes \text{\bf x}_2
\in B(H(L;R)[1])$. Then
\begin{equation}\label{unitalL}
\frak q_{\beta}^{H;[\gamma,w]}(\text{\bf y};\text{\bf x}) = 0.
\end{equation}
\end{enumerate}
\end{prop}
\begin{proof}
Using Stokes' formula (Theorem \ref{them48}, \cite[Theorem 8.11]{fooo:tech2}) and  composition formula (Theorem \ref{compform},
\cite[Theorem 10.20]{fooo:tech2}),
Statement (1) follows from Lemma \ref{existmultiFF} (2) and
Proposition \ref{FBULKkura18} (3). Statement
(2) follows from \ref{existmultiFF} (4) and
Proposition \ref{FBULKkura18} (8).
\end{proof}
Let $\text{\bf b} = (\frak b_0,\text{\bf b}_{2;1},\frak b_+,b_+)$ as in
Definition \ref{deformedqdef}.
We put
$\frak b = i^*(\text{\bf b}_{2;1}) + \frak b_+$.
We extend $\frak q^{H;[\gamma,w]}_{\ell,k;\beta}$ by
$\Lambda$ linearity in the formula (\ref{HFtoHFLdef}) below.

\begin{defn}\label{def:I}
We define
$
\frak I_{(H,J)}^{\text{\bf b},\frak b}:
CF(M,H;\Lambda^{\downarrow})
\to
CF_{\text{dR}}(L;\Lambda^{\downarrow})
$
by
\begin{equation}\label{HFtoHFLdef}
\aligned
\frak I_{(H,J)}^{\text{\bf b},\frak b}(\llb \gamma,w \rrb)
=\sum_{\beta}\sum_{\ell=0}^{\infty}\sum_{k=0}^{\infty}
&q^{-(\beta\cap \omega - w\cap \omega)}
\frac{\exp(\text{\bf b}_{2;1}\cap \beta - i^*(\text{\bf b}_{2;1}) \cap w)}{\ell!}
\\
&
\frak q^{H,[\gamma,w]}_{\ell,k;\beta}(\frak b_+^{\otimes \ell};b_+^{\otimes k}).
\endaligned
\end{equation}
Here we suppose $\llb \gamma,w \rrb \in \widehat{\rm Per}(H)$
is represented by $[\gamma,w] \in \text{\rm Crit}(\mathcal A_H)$, i.e.,
$\pi([\gamma,w]) = \llb \gamma,w \rrb$, to define the right hand side.
We can show the independence of the representative $[\gamma,w]$ as follows.
Suppose we take another choice $[\gamma,w']$.
Then $w' = \alpha \# w'$ for $\alpha \in K_2(\gamma)$.
(Lemma \ref{determinedifference}.)
We find
$$
\frak q^{H,[\gamma,w']}_{\ell,k;\beta}(\frak b_+^{\otimes \ell};b_+^{\otimes k})
=
\frak q^{H,[\gamma,w]}_{\ell,k;\beta+\alpha}(\frak b_+^{\otimes \ell};b_+^{\otimes k})
$$
by Proposition \ref{FBULKkura18} (11).
Therefore the sum in the right hand side is independent of
the choice of $[\gamma,w]$.
\end{defn}
We can prove the convergence of the right hand side of (\ref{HFtoHFLdef})
in $q$-adic topology in the same way as in Lemma \ref{adiccomv1}.
\begin{lem}\label{lem1816}
We have
$$
\delta^{\bf b} \circ \frak I_{(H,J)}^{\text{\bf b},\frak b}
= \frak I_{(H,J)}^{\text{\bf b},\frak b} \circ \partial_{(H,J)}^{\frak b}, \quad \delta^{\bf b}(x)
= (-1)^{\deg x} \frak m_1^{\text{\bf b}}(x).
$$
\end{lem}
The proof is a straightforward calculation using Proposition \ref{prop:qH-relation}
and (\ref{defboundary}) so omitted.
This gives rise to a map
\begin{equation}\label{1820}
\frak I_{(H,J)}^{\text{\bf b},\frak b, \ast}
: HF_*(M,H;\Lambda^{\downarrow})
\to
HF^*((L,\text{\bf b});\Lambda^{\downarrow}).
\end{equation}
\begin{rem}
We can show that the map (\ref{1820}) is a ring homomorphism
with respect to the pants product in the left hand side and $\pm \frak m_2$
in the right hand side up to sign. We do not prove it here since we do not use it.
\end{rem}

\subsection{Filtration and the map $\frak I_{(H,J)}^{\text{\bf b},\frak b}$}
\label{subsec:Iandfilt}

We define a filtration on $CF_{\text{dR}}(L;\Lambda^{\downarrow})$
$$
F^{-\lambda}CF_{\text{dR}}(L;\Lambda^{\downarrow})
: = F^{\lambda}CF_{\text{dR}}(L;\Lambda)
= T^{\lambda}(\Omega(L) \widehat\otimes \Lambda_0)
$$
by identifying $CF_{\text{dR}}(L;\Lambda^{\downarrow})$ with $CF_{\text{dR}}(L;\Lambda)$
via the change $T = q^{-1}$ of formal parameters.
Similarly we put
$$
F^{\lambda}(\Omega(M) \widehat\otimes \Lambda^{\downarrow})
= q^{-\lambda}(\Omega(M) \widehat\otimes \Lambda_0^\downarrow).
$$
This is consistent with Definitions
\ref{Lambdahatonashi}, \ref{valuationv}.
See Notations and Conventions (18) in Section \ref{sec:introduction}
and also Remark \ref{filtminusrem}.
In this subsection we prove the following:
\begin{prop}\label{energyHH} For all $\lambda \in \R$,
$$
\frak I_{(H,J)}^{\text{\bf b},\frak b}
(F^{\lambda}(\Omega(M) \widehat\otimes \Lambda^{\downarrow}))
\subseteq
F^{\lambda + E^+(H;Y)} CF_{\text{dR}}(L;\Lambda^{\downarrow}).
$$
\end{prop}
\begin{proof}
The proposition immediately follows from Lemma \ref{energyestimateLandF} below.
\end{proof}

\begin{lem}\label{energyestimateLandF}
If ${\mathcal M}_{k+1;\ell}(H,J;[\gamma,w],L;\beta)$
is nonempty, we have
$$
\mathcal A_H([\gamma,w])
\ge - E^+(H;Y) -\beta\cap \omega.
$$
\end{lem}
\begin{proof}
Let $(u;z_1^+\dots,z_{\ell}^+,z_0,\dots,z_k) \in
\overset{\circ}{\mathcal M}_{k+1;\ell}(H,J;[\gamma,w],L;\beta)$.
We calculate
$$
\aligned \int u^*\omega
&=
E_{H,J}(u)
- \int_{(-\infty,0]\times S^1} \frac{\partial}{\partial \tau}
(H \circ u) d\tau dt \\
& \ge
\lim_{\tau\to-\infty}\int_{S^1} H(t, u(\tau,t))\, dt - \int_{S^1} H(0,u(0,t))\, dt \\
& = \int_{S^1} H(t, \gamma(t))\, dt - \int_0^1 H(0,u(0,t))\, dt \\
& \ge
\int_{S^1} H_t(\gamma(t))\, dt - E^+(H;Y).
\endaligned
$$
Recalling $\beta\cap \omega - \int w^*\omega = \int u^*\omega$ from
$\beta = [w\# u]$, we obtain
$$
- \int w^* \omega - \int_{S^1} H_t(\gamma(t))\, dt \ge -
E^+(H;Y) -  \beta\cap \omega.
$$
The lemma follows.
\end{proof}

\subsection{Identity $\frak I_{(H,J)}^{\text{\bf b},\frak b, \ast} \circ \CP_{(H_\chi,J),\ast}^{\frak b}
=
i_{\text{\rm qm},\text{\bf b}}^{\ast}$}
\label{subsec:PIirelation}\index{Piunikhin isomorphism}

In this subsection we prove:
\begin{prop}\label{184mainprop}
For any $a \in H^*(M)\otimes\Lambda$
the identity
$$
\frak I_{(H,J)}^{\text{\bf b},\frak b, \ast} \circ \CP_{(H_\chi,J),\ast}^{\frak b}(a^{\flat})
=
i_{\text{\rm qm},\text{\bf b}}^{\ast}(a)
$$
holds in cohomology. Here $a^{\flat} \in H_*(M;\Lambda^{\downarrow})$ is the homology class Poincar\`e dual to
$a\in H^*(M)\otimes \Lambda$ as in
Notations and Conventions $(22)$.
\end{prop}
\begin{proof}
For $S \ge 0$ we put
$$
H^S(\tau,t,x) = \chi(\tau+S+20)H(t,x)
$$
where $\chi$ is as in Definition \ref{defn:chi}.
We also put
$H^S_{\tau,t}(x) = H^S(\tau,t,x)$.

\begin{defn}\label{Schijodulis}
Denote by
$$
\overset{\circ}{\mathcal M}_{k+1;\ell}(H^S,J;*,L;\beta)
$$
the set of all triples
$(u;z_1^{+},\dots,z_{\ell}^+;z_0,\dots,z_k)$ satisfying the following:
\begin{enumerate}
\item
$u : (-\infty,0] \times S^1 \to M$ is a smooth map
such that $u(0,t) \subset L$.
\item
$u$ satisfies the equation
\be\label{eq:HJCR18sec2}
\dudtau + J\Big(\dudt -X_{H^S_{\tau,t}}(u)\Big) = 0.
\ee
\item The energy
$$
\frac{1}{2} \int \Big(\Big|\dudtau\Big|^2_{J} + \Big|
\dudt - X_{H^S_{\tau,t}}(u)\Big|_{J}^2 \Big)\, dt\, d\tau
$$
is finite.
\item $z_1^{+},\dots,z_{\ell}^+$ are points in $(-\infty,0)\times S^1$
which are mutually distinct.
\item
$z_0,\dots,z_k$ are points on the boundary
$\{0\}\times S^1$.
They are mutually distinct.
$z_0,\dots,z_k$ respects the counterclockwise cyclic order on $S^1$.
We always set $z_0 = (0,0)$.
\item The homology class of  $u$ is $\beta$.
\end{enumerate}
\end{defn}
\begin{figure}[h]
\centering
\includegraphics[scale=0.3]
{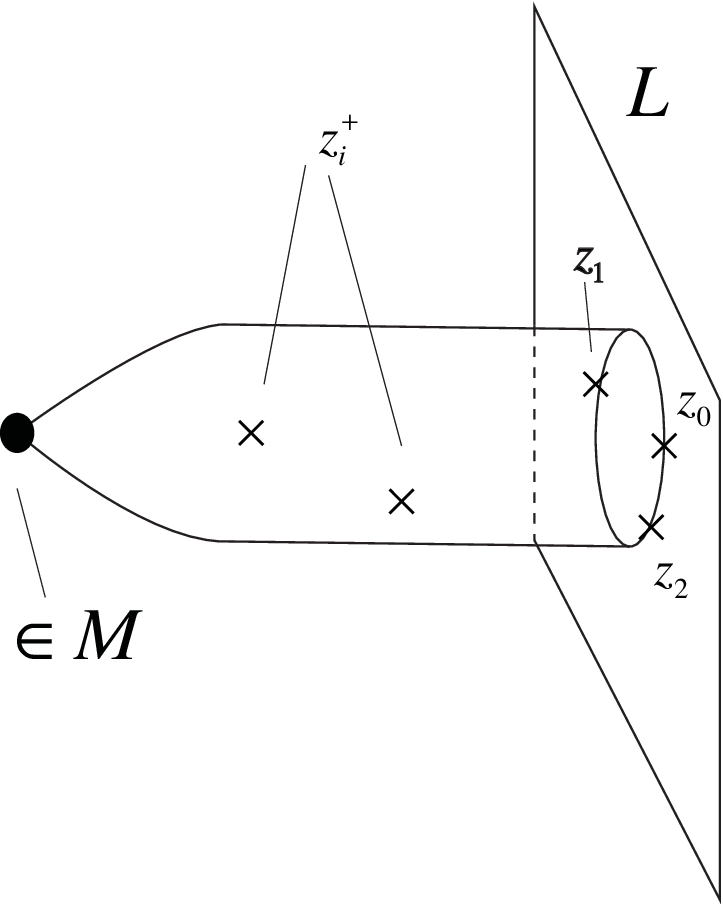}
\caption{An element of $
\overset{\circ}{\mathcal M}_{k+1;\ell}(H^S,J;*,L;\beta)
$}
\label{Figure16}
\end{figure}
We define an evaluation map
$$
(\text{\rm ev},\text{\rm ev}^{\partial})
= (\text{\rm ev}_1,\dots,\text{\rm ev}_{\ell};
\text{\rm ev}^{\partial}_0,\dots,\text{\rm ev}^{\partial}_k)
: \overset{\circ}{\mathcal M}_{k+1;\ell}(H^S,J;*,L;\beta)
\to M^{\ell} \times L^{k+1}
$$
by
$$
\text{\rm ev}_i(u;z_1^{+},\dots,z_{\ell}^+;z_0,\dots,z_k)
= u(z_i^+),
\quad
\text{\rm ev}^{\partial}_i([u;z_1^{+},\dots,z_{\ell}^+;z_0,\dots,z_k])
= u(z_i).
$$
We also define
$
\text{\rm ev}_{-\infty} : \overset{\circ}{\mathcal M}_{k+1;\ell}(H^S,J;*,L;\beta)
\to M
$
by
$$
\text{\rm ev}_{-\infty}(u;z_1^{+},\dots,z_{\ell}^+;z_0,\dots,z_k)
= \lim_{\tau\to -\infty}u(\tau,t).
$$
By (2), (3) and the removable singularity theorem, the limit of the right hand side exists and
is independent of $t$.
We put
\begin{equation}
\overset{\circ}{\mathcal M}_{k+1;\ell}(para;H,J;*,L;\beta)
=
\bigcup_{S\in [0,\infty)} \{S\} \times
\overset{\circ}{\mathcal M}_{k+1;\ell}(H^S,J;*,L;\beta),
\end{equation}
where $(\text{\rm ev},\text{\rm ev}^{\partial}) $ and $\text{\rm ev}_{-\infty}$ are defined on it
in an obvious way.
\index{$\mathcal M_{k+1;\ell}(H^S,J;*,L;\beta)$}
\begin{lem}\label{SFBULKkura18}
\begin{enumerate}
\item
The moduli space
$\overset{\circ}{\mathcal M}_{k+1;\ell}(para;H,J;*,L;\beta)$ has a compactification
\index{$\mathcal M_{k+1;\ell}(para;H,J;*,L;\beta)$}
${\mathcal M}_{k+1;\ell}(para;H,J;*,L;\beta)$ that is Hausdorff.
\item
The space ${\mathcal M}_{k+1;\ell}(para;H,J;*,L;\beta)$ has an orientable Kuranishi structure with corners.
\item
The normalized boundary of ${\mathcal M}_{k+1;\ell}(para;H,J;*,L;\beta)$ is described by
the union of the three types of direct or fiber products:
\begin{equation}\label{bdryFFhamLkinmodu22a}
\bigcup {\mathcal M}_{\# \L_1}(H_\chi,J;*,[\gamma,w])
\times {\mathcal M}_{k+1;\# \L_2}(H,J;[\gamma,w],L;\beta),
\end{equation}
where the union is taken over all $[\gamma,w] \in \text{\rm Crit}(\mathcal A_H)$,
$(\mathbb L_1,\mathbb L_2) \in \text{\rm Shuff}(\ell)$.
(Figure $\ref{Figure185}$)
(Here ${\mathcal M}_{\# \L_1}(H_\chi,J;*,[\gamma,w])$ is the moduli space
defined in Definition $\ref{buldPmoduli}$ and Proposition $\ref{piuBULKkura}$.
We write $J$ in place of $J_{\chi}$ since in Part $4$ we use a fixed $J$ which is
independent of $t$ and $\tau$.)

\begin{equation}\label{bdryFFhamLkinmodu22b}
\bigcup {\mathcal M}_{k_1+1;\# \L_1}(L;\beta_1)
{}_{\text{\rm ev}^{\partial}_{0}}\times_{\text{\rm ev}^{\partial}_i} {\mathcal M}_{k_2+1;\# \L_2}
(para;H,J;*,L;\beta_2),
\end{equation}
(Figure $\ref{Figure186}$)
where the union is taken over all
$(\mathbb L_1,\mathbb L_2) \in \text{\rm Shuff}(\ell)$, $k_1,k_2$ with
$k_1 + k_2 = k$, $i \le k_2$, and $\beta_1,\beta_2$ with $\beta_1 + \beta_2 = \beta$.
\begin{equation}\label{bdryFFhamLkinmodu22c}
{\mathcal M}_{k+1;\ell}(H^0,J;*,L;\beta),
\end{equation}
that is a compactification of the $S=0$ case of
$\overset{\circ}{\mathcal M}_{k+1;\ell}(H^S,J;*,L;\beta)$.
\item
Let
$\mu_L : H_2(M,L;\Z) \to  2\Z$ be the Maslov index. Then the (virtual) dimension satisfies
the following equality:
\begin{equation}\label{dimensionFLF}
\dim {\mathcal M}_{k+1;\ell}(para;H,J;*,L;\beta) =
\mu_L(\beta) +2\ell + k - 1 + n.
\end{equation}
\item
We can define a system of orientations on
the collection of moduli spaces ${\mathcal M}_{k+1;\ell}(para;H,J;*,L;\beta)$  that
is compatible with the isomorphism $(3)$ above.
For the boundary of type \eqref{bdryFFhamLkinmodu22b} the compatibility means
the same as in Lemma {\rm \ref{FBULKkura18}} $(5)$.
\item
$({\rm ev}, {\rm ev}^{\partial}, {\rm ev}_{-\infty})$ extends to a strongly continuous smooth map
$$
{\mathcal M}_{k+1;\ell}(para;H,J;*,L;\beta) \to M^{\ell+1} \times L^{k+1},
$$
which we denote  by the same symbol.
It is compatible with $(3)$.
\item $ev_0^{\partial}$ is weakly submersive.
\item
The Kuranishi structure is compatible with forgetful map of the
boundary marked points.
\item
The Kuranishi structure is invariant under the permutation of the interior marked
points.
\item
The Kuranishi structure is invariant under the cyclic permutation of the
boundary  marked
points.
\end{enumerate}
\end{lem}
\begin{center}
\begin{figure}[h]
 \begin{tabular}{cc}
 \begin{minipage}[t]{0.45\hsize}
        \centering
\includegraphics[scale=0.3]
{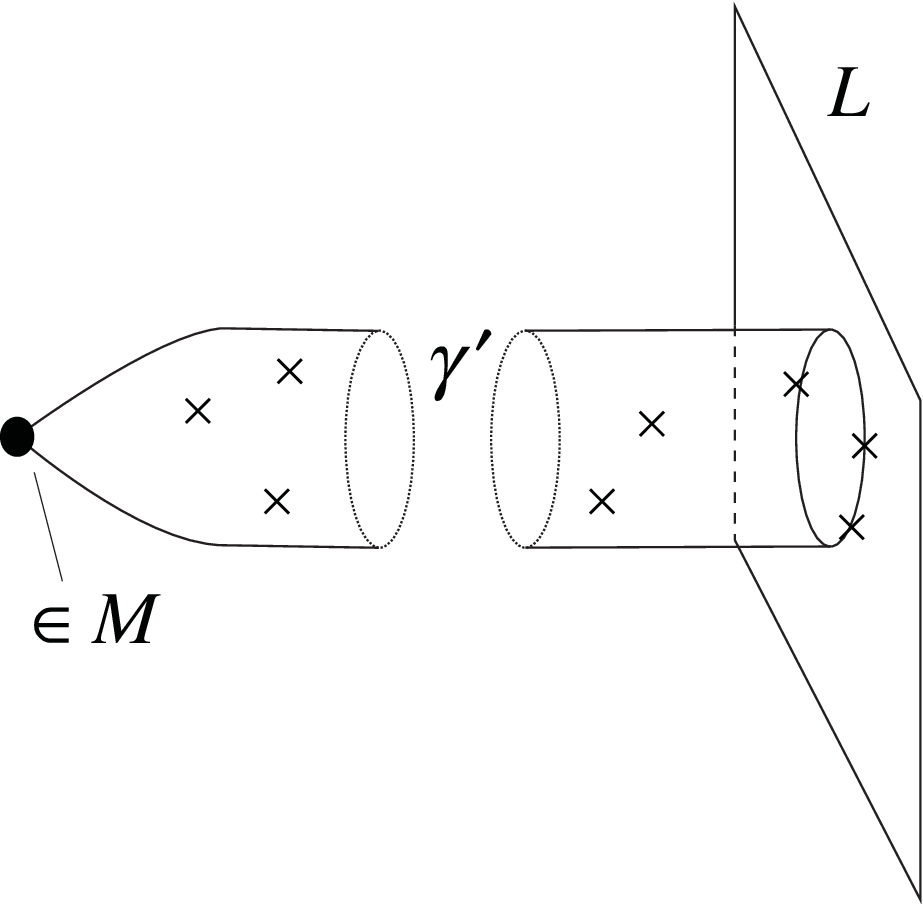}
\caption{An element of (\ref{bdryFFhamLkinmodu22a})}
\label{Figure185}
\end{minipage} &
\begin{minipage}[t]{0.45\hsize}
\centering
\includegraphics[scale=0.3]
{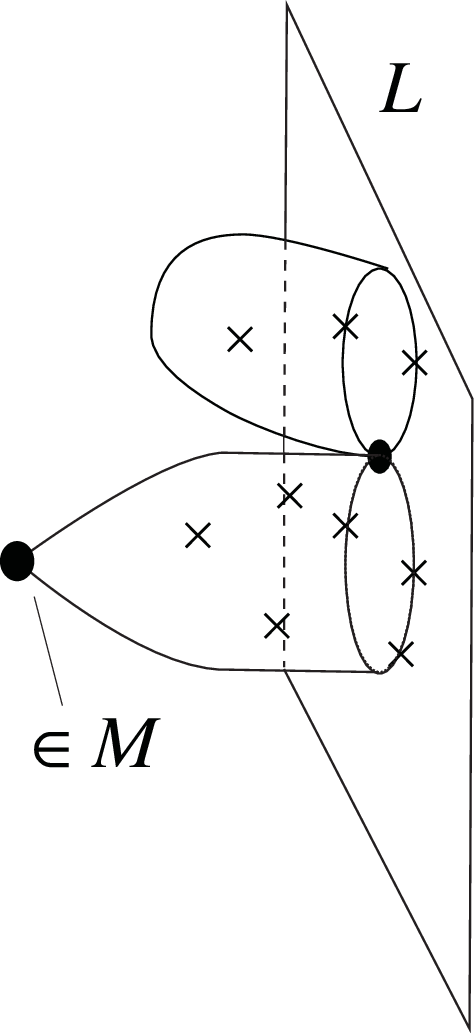}
\caption{An element of (\ref{bdryFFhamLkinmodu22b})}
\label{Figure186}
\end{minipage}
\end{tabular}
\end{figure}
\end{center}
The proof of Lemma \ref{SFBULKkura18} is the same as that of Propositions \ref{connkura},
which is detailed in \cite{fooo:techI} Parts 4 and 5.
It suffices to observe that (\ref{bdryFFhamLkinmodu22a}) appears at the limit $S \to \infty$.
\begin{lem}\label{existmultiFFS}
There exists a system of CF-perturbations on our
moduli spaces
${\mathcal M}_{k+1;\ell}(para;H,J;*,L;\beta)$
with
the following properties:
\begin{enumerate}
\item
It is transversal to $0$.
\item
It is compatible with the description of the boundary in Proposition
$\ref{SFBULKkura18}$ $(3)$.
\item
The map $\text{\rm ev}_0^{\partial}$ is strongly submersive with respect
to our CF-perturbation in the sense of \cite[Definition 7.48]{fooo:tech2}.
\item
It is compatible with forgetful map of the
boundary marked points.
\item
It is invariant under the permutation of the interior marked points.
\item
It is invariant under the cyclic permutation of the boundary marked points.
\end{enumerate}
\end{lem}
The proof is the same as that of Lemma \ref{existmultiFF}
(and  \cite[Corollary 5.2]{fukaya:cyc},
 \cite[Theorem 7.49]{fooo:tech2}) and is omitted.
We now define
\index{$\frak q^{H,S\ge 0}_{\ell,k;\beta}$}
\begin{equation}
\frak q^{H,S\ge 0}_{\ell,k;\beta}
: E_{\ell} (\Omega(M)[2])
\otimes \Omega(M)[1]
\otimes  B_k(\Omega(L)[1])
\to \Omega(L)[1]
\end{equation}
by sending $(g_1,\dots,g_{\ell};h;h_1,\dots,h_k)$ to
$$
(\text{\rm ev}^{\partial}_{0})!
\left(
\text{\rm ev}_1^*g_1\wedge \dots \wedge \text{\rm ev}_{\ell}^*g_{\ell}
\wedge \text{\rm ev}_{-\infty}^*h \wedge
\text{\rm ev}_1^{\partial *}h_1\wedge \dots \wedge \text{\rm ev}_{k}^{\partial *}h_{k}
\right).
$$
Here
\begin{equation}\label{defqformula2SS}
\aligned
(\text{\rm ev},\text{\rm ev}^{\partial},\text{\rm ev}_{-\infty})
&= (\text{\rm ev}_1,\dots,\text{\rm ev}_{\ell};\text{\rm ev}^{\partial};
\text{\rm ev}^{\partial}_0,\dots,\text{\rm ev}^{\partial}_k)  \\
&: {\mathcal M}_{k+1;\ell}(para;H,J;*,L;\beta)
\to M^{\ell+1} \times L^{k+1}
\endaligned
\end{equation}
is the natural evaluation map and $(\text{\rm ev}^{\partial}_{0})!$
is the integration along the fibers of $\text{\rm ev}^{\partial}_{0}$. (The precise definition of
this integration associated to the CF-perturbation
is given in \cite[Definition 7.78]{fooo:tech2}. See also Definition \ref{def320222}.)
We define
$$
\frak q^{H,S=0}_{\ell,k;\beta}
: E_{\ell} (\Omega(M)[2])
\otimes \Omega(M)[1]
\otimes  B_k(\Omega(L)[1])
\to \Omega(L)[1]
$$
by using ${\mathcal M}_{k+1;\ell}(H^0,J;*,L;\beta)$ in \eqref{bdryFFhamLkinmodu22c} in the same way.

\begin{defn} Define a map
$
\frak J_{(H^0,J)}^{\text{\bf b},\frak b}:
\Omega(M) \widehat\otimes \Lambda^{\downarrow}
\to
CF_{\text{dR}}(L;\Lambda^{\downarrow})
$
by
\begin{equation}\label{HFtoHFLdefS=0}
\aligned
\frak J_{(H^0,J)}^{\text{\bf b},\frak b}(h)
=\sum_{\beta}\sum_{\ell=0}^{\infty}\sum_{k=0}^{\infty}
q^{-\beta\cap \omega}
\frac{\exp(\text{\bf b}_{2;1}\cap \beta )}{\ell!}
\frak q^{H,S=0}_{\ell,k;\beta}(\frak b_+^{\otimes \ell},h,b_+^{\otimes k}).
\endaligned
\end{equation}
We also define
$
\frak H_{(H,J)}^{\text{\bf b},\frak b}:
\Omega(M) \widehat\otimes \Lambda^{\downarrow}
\to
CF_{\text{dR}}(L;\Lambda^{\downarrow})
$
by
\begin{equation}\label{HFtoHFLdefS>0}
\aligned
\frak H_{(H,J)}^{\text{\bf b},\frak b}(h)
=\sum_{\beta}\sum_{\ell=0}^{\infty}\sum_{k=0}^{\infty}
q^{-\beta\cap \omega}
\frac{\exp(\text{\bf b}_{2;1}\cap \beta)}{\ell!}
\frak q^{H,S\ge 0}_{\ell,k;\beta}(\frak b_+^{\otimes \ell},h,b_+^{\otimes k}).
\endaligned
\end{equation}
\end{defn}
\begin{lem}
We have
$$
\delta^{\text{\bf b}} \circ \frak H_{(H,J)}^{\text{\bf b},\frak b}
\pm \frak H_{(H,J)}^{\text{\bf b},\frak b} \circ d
=
\frak I_{(H,J)}^{\text{\bf b},\frak b} \circ \CP_{(H_\chi,J)}^{\frak b} \circ \flat
-\frak J_{(H^0,J)}^{\text{\bf b},\frak b}.
$$
Here $\flat : \Omega^*(M) \widehat\otimes \Lambda^{\downarrow}
\cong
\Omega_{\dim M-*}(M) \widehat\otimes \Lambda^{\downarrow}$ as in
\eqref{eq:deRhamchain}.
\end{lem}\label{chainhomotopy1}
\begin{proof}
This follows from Lemma \ref{SFBULKkura18} after considering the correspondence by using the moduli space in Lemma \ref{SFBULKkura18} (1)
and using Stokes' formula (Theorem \ref{them48}, \cite[Theorem 8.11]{fooo:tech2})
and  composition formula (Theorem \ref{compform}, \cite[Theorem 10.20]{fooo:tech2}).
Indeed, the first term of left hand side corresponds to
(\ref{bdryFFhamLkinmodu22b}).
The first and second terms of the right hand side
correspond to (\ref{bdryFFhamLkinmodu22a}) and
(\ref{bdryFFhamLkinmodu22c}), respectively.
\end{proof}
\par
We next construct a chain homotopy between
$\frak J_{(H^0,J)}^{\text{\bf b},\frak b}$
and $i_{\text{\rm qm},\text{\bf b}}$.
Let $\sigma \in [0,1]$.
We replace (\ref{eq:HJCR18sec2}) by
\be\label{eq:HJCR18sec23}
\dudtau + J\Big(\dudt - \sigma X_{H^0_t}(u)\Big) = 0
\ee
in Definition \ref{Schijodulis} to define
$\overset{\circ}{\mathcal M}_{k+1;\ell}(\sigma,H^0,J;*,L;\beta) $.
We put
$$
\overset{\circ}{\mathcal M}_{k+1;\ell}([0,1],H^0,J;*,L;\beta)
=
\bigcup_{\sigma \in [0,1]}
\{\sigma\} \times \overset{\circ}{\mathcal M}_{k+1;\ell}(\sigma,H^0,J;*,L;\beta) .
$$
We can prove a lemma similar to Lemmas \ref{SFBULKkura18}, \ref{existmultiFFS}
using the compactification ${\mathcal M}_{k+1;\ell}([0,1],H^0,J;*,L;\beta)$
in place of ${\mathcal M}_{k+1;\ell}(para,H,J;*,L;\beta)$
in (\ref{defqformula2SS}) and (\ref{HFtoHFLdefS>0}), and define
$$
\overline{\frak H}_{(H^0,J)}^{\text{\bf b},\frak b}:
\Omega(M) \widehat\otimes \Lambda^{\downarrow}
\to
CF_{\text{dR}}(L;\Lambda^{\downarrow}).
$$
Then in a similar way we can show
$$
\delta^{\text{\bf b}} \circ \overline{\frak H}_{(H^0,J)}^{\text{\bf b},\frak b}
\pm \overline{\frak H}_{(H^0,J)}^{\text{\bf b},\frak b} \circ d
= \frak J_{(H^0,J)}^{\text{\bf b},\frak b}
-i_{\text{\rm qm},\text{\bf b}}.
$$
Combining this with Lemma \ref{chainhomotopy1},
we finish the proof of Proposition \ref{184mainprop}.
\end{proof}

\subsection{Heaviness of $L$}
\label{subsec:Heavynesscomplete}
We are now ready to complete the proofs of Proposition
\ref{182mainprop} and Theorem \ref{thm:heavy}.
\begin{proof}[Proof of Proposition \ref{182mainprop}]
For $\epsilon >0$ we take $x \in F^{\lambda}CF(M,H;\Lambda^{\downarrow})
\cong F^{-\lambda}CF(M,H;\Lambda)$
such that $[x] = \CP_{(H_\chi,J),\ast}^{\frak b}(a^\flat)$
and $\lambda \le \rho^{\frak b}(H;a)+\epsilon$.
By Proposition \ref{energyHH} we have
\begin{equation}
\frak I_{(H,J)}^{\text{\bf b},\frak b}(x)
\in H^{\lambda+E^+(H;Y)+\epsilon}CF(L;\Lambda^{\downarrow}).
\end{equation}
On the other hand, Proposition \ref{184mainprop}
shows that
$$
[\frak I_{(H,J)}^{\text{\bf b},\frak b}(x)]
=
\frak I_{(H,J)}^{\text{\bf b},\frak b,\ast}\circ \CP_{(H_\chi,J),\ast}^{\frak b}(a^\flat)
= i_{\text{\rm qm},\text{\bf b}}^{\ast}(a).
$$
Therefore
$
\lambda + E^+(H;Y) \ge \rho^{\mathbf b}_L(a)
$
by definition.
It implies
$
\rho^{\frak b}(H;a)+\epsilon \ge \rho^{\mathbf b}_L(a) - E^+(H;Y).
$
The proof of Proposition
\ref{182mainprop} is complete.
\end{proof}
\par\medskip
\begin{proof}[Proof of Theorem \ref{thm:heavy}]
We first prove (1).
By Remark \ref{rem:muimplieszeta} it suffices to prove $\mu_e^{\frak b}$-heaviness.
Let $H : M \times S^1 \to \R$ be a normalized periodic Hamiltonian.
We put
\begin{equation}\label{Hndef}
H_{(n)}(t,x) =  nH(nt - [nt],x),
\end{equation}
where $[c]$ is the largest integer such that $c \ge [c]$.
It is easy to see that $\widetilde\psi_{H_{(n)}} = (\widetilde\psi_{H})^n$.
We apply Proposition
\ref{182mainprop} to $H_{(n)}$ and obtain
$$
\rho^{\frak b}( (\widetilde\psi_{H})^n;e)
\ge  - n \inf E^+(H;L) + \rho^{\mathbf b}_L(e).
$$
Therefore by definition we have
$$
- \mu_e^{\frak b}(\widetilde\psi_{H}) \le \vol_\omega(M)E^+(H;L).
$$
Thus Theorem \ref{thm:heavy} (1) is proved.
\par
We turn to the proof of (2). Again
it suffices to prove $\mu_e^{\frak b}$-superheaviness.
\par
We use our assumption to apply Lemma \ref{dualitygyakmain} and obtain
\begin{equation}\label{finalstep1}
\rho^{\frak b}((\widetilde\psi_{H})^n;e)
\le 3\frak v_q(e) - \rho^{\frak b}((\widetilde\psi_{H})^{-n};e).
\end{equation}
We put $\widetilde H(t,x) = -H(1-t,x)$ and then obtain
$\widetilde H_{(n)}$ as in (\ref{Hndef}).
We then apply Proposition \ref{182mainprop} to $\widetilde H_{(n)}$ and obtain
\begin{equation}\label{finalstep2}
\rho^{\frak b}((\widetilde\psi_{H})^{-n};e)
\ge - n E^-(H;L) + \rho^{\mathbf b}_L(e).
\end{equation}
By (\ref{finalstep1}) and (\ref{finalstep2}) we have
$$
\rho^{\frak b}((\widetilde\psi_{H})^n;e)
\le  n E^-(H;L)
+ 3\frak v_q(e) - \rho^{\mathbf b}_L(e).
$$
Therefore
$$
- \mu_e^{\frak b}(\widetilde\psi_{H})  \ge -\vol_\omega(M)E^-(H;L)
$$
as required. The proof of Theorem \ref{thm:heavy} is now complete.
\end{proof}

\section{Linear independence of quasi-morphisms.}
\label{sec:independence}

In this section we prove Corollary \ref{lieind}.
We use the same notations as those in this corollary.
Let $U_i \subset M$, $i=-N,\dots,N$ be open sets such that
$\overline U_i \cap \overline U_j = \emptyset$
for $i\ne j$ and $L_i \subset U_i$ for
$i=1,\dots,N$.
For $i=-N,\dots,N$, let $\rho_i$ be nonnegative smooth functions on $M$
such that
$
\supp \rho_i \subset U_i$ ($i=-N,\dots,N$),
$\rho_i \equiv 1 $ on $L_i$ ($i=1,\dots,N$), and
$ \int_M\rho_i \omega^n =c$ where $c>0$ is independent of $i$ ($i=1,\dots,N$).
We then put
$$
H_i
= \vol_\omega(M)^{-1}
\left(
\rho_i
-
\rho_{-i}
\right)
$$
and regard them as time independent normalized Hamiltonian functions.
We put $\widetilde\psi_i = \widetilde\psi_{H_i}$.
Since the support of $H_i$ is disjoint from that of $H_j$ it follows that
$\widetilde\psi_i$ for $i \ne j$ commutes with $\widetilde\psi_j$.
Namely they generate a subgroup isomorphic to $\Z^N$.
\par
For $(k_1,\dots,k_N) \in \Z^N$ we consider
$$
\widetilde\phi =
\prod_{i=1}^N \widetilde\psi_i^{k_i}.
$$
Note that $\widetilde\phi = \widetilde\psi_H$ where
$
H = \sum_{i=1}^N k_i H_i.
$
Since $L_i$ is $\mu_{e_i}^{\frak b_i}$-superheavy and
$\mu_{e_i}^{\frak b_i}$-heavy, we have
$$
\vol_\omega(M)
\inf \{- H(x) \mid x \in L_i \}
\le
\mu_{e_i}^{\frak b_i}(\widetilde\phi )
\le
\vol_\omega(M)
\sup\{-H(x) \mid x \in L_i \}.
$$
Therefore
$
\mu_{e_i}^{\frak b_i}(\widetilde\phi) = -k_i.
$
The proof of Corollary  \ref{lieind} is complete.
\qed

\part{Applications}

In this part, we provide applications of the results obtained in the previous parts.
Especially combining them with the calculations we carried out in
a series of papers \cite{fooo:toric1, fooo:bulk, fooo:toricmir} in the
case of toric manifolds, we prove Corollary \ref{existtoric},
and Theorem \ref{uncount} for the case of $k$ ($\ge 2$)  points blow up of $\C P^2$.
The latter example has been studied in \cite{fooo:bulk}.
We also examine a continuum of Lagrangian tori in $S^2 \times S^2$ discovered by the present
authors in \cite{fooo:S2S2} and prove Theorem \ref{uncount}.

\section{Lagrangian Floer theory of toric fibers: review}
\label{toriceLagreview}

\subsection{Toric manifolds: review}
\label{subsec:toricreview}

In this subsection we review a very small portion of the theory of toric variety.
See for example \cite{fulton}
for a detailed account of toric varieties.
\par
Let $(M,\omega,J)$ be a K\"ahler manifold, where $J$ is its complex structure
and $\omega$ its K\"ahler form.
Let $n$ be the complex dimension of $M$. We assume $n$ dimensional real torus $T^n
= (S^1)^n$ acts effectively on $M$ such that $J$ and $\omega$ are preserved by the
action. We call such $(M,\omega,J)$ a {\it K\"ahler toric manifold}\index{K\"ahler toric manifold}
if the $T^n$ action has a moment map in the sense we describe below.
Hereafter we simply say $(M,\omega,J)$ (or $M$) is a toric manifold.
\par
Let $(M,\omega,J)$ be as above. We say a map
$\pi = (\pi_1,\dots,\pi_n) : M \to \R^n$ is a {\it moment map} \index{moment map} if the following holds.
We consider the $i$-th factor $S^1_i$ of $T^n$. (Here $i=1,\dots,n$.)
Then $\pi_i : M\to \R$ is the moment map of the action of  $S^1_i$.
In other words, we have the following identity  of $\pi_i$
\begin{equation}\label{momentmapeq}
2\pi d\pi_i= \omega(\widetilde{\frak t}_i, \cdot)
\end{equation}
where $\widetilde{\frak t}_i$ is the Killing vector field
associated to the action of the circle
 $S^1_i$ on $X$.
\begin{rem}\label{rem2pi1}
We put $2\pi$ in Formula (\ref{momentmapeq})
in order to eliminate this factor from (\ref{elljandarea}).
See Remark \ref{rem2pi2}.
\end{rem}
\par
Let $\text{\bf u} \in \text{Int}P$.
Then the inverse image $\pi^{-1}(\text{\bf u})$ is a
Lagrangian submanifold which is an orbit of the $T^n$ action.
We put
\begin{equation}\label{L(u)def}
L(\text{\bf u}) = \pi^{-1}(\text{\bf u}).
\end{equation}
This is a Lagrangian torus.
\par
It is well-known that $P = \pi(M)$ is a convex polytope.
We can find a finitely many affine functions $\ell_j: \R^n \to \R$ ($j=1,\dots,m$)
such that \index{$\ell_j$}
\begin{equation}
P
= \{ \text{\bf u} \in \R^n \mid \ell_j(\text{\bf u}) \ge 0, \quad \forall j=1,\dots,m\}.
\end{equation}
We put $\partial_j P = \{\text{\bf u} \in P \mid \ell_j(\text{\bf u})= 0\} $
and $D_j = \pi^{-1}(\partial_j P)$.
($\dim_{\R} \partial_j P = n-1$.)
$D_1 \cup \dots \cup D_m$ is called the {\it toric divisor}\index{toric divisor}.
\par
Moreover we may choose $\ell_j$ so that the following holds.
\begin{conds}\label{condellj}
\begin{enumerate}
\item
We put
$$
d\ell_j = \vec v_j = (v_{j,1},\dots,v_{j,n}) \in \R^n.
$$
Then $v_{j,i} \in \Z$.
\item
Let $p$ be a vertex of $P$. Then the number of faces $\partial_j P$ which contain
$p$ is $n$. Let $\partial_{j_1} P,\dots, \partial_{j_n} P $ be those faces. Then
$\vec v_{j_1},\dots, \vec v_{j_n}$ (which is contained in $\Z^n$ by item (1)) is
a basis of $\Z^n$.
\end{enumerate}
\end{conds}
The existence of such $\ell_j$ and the property above is proved in \cite{Del}.
The affine functions $\ell_j$ have the following geometric interpretation.
Let $\text{\bf u} \in \text{\rm Int}P$.
There exists $m$ elements
$\beta_j \in H_2(M,L(\text{\bf u});\Z)$ such that
\begin{equation}\label{charbetaj}
\beta_j \cap D_{j'} =
\begin{cases}
1  & j = j', \\
0  & j \ne j'.
\end{cases}
\end{equation}
Then the following area formula
\begin{equation}\label{elljandarea}
\int_{\beta_j} \omega = \ell_j(\text{\bf u})
\end{equation}
is proved in \cite[Theorem 8.1]{cho-oh}.
(See \cite[Section 2]{fooo:toric1}  also.)
\begin{rem}\label{rem2pi2}
Note in \cite[Theorem 8.1]{cho-oh},
\cite[Section 2]{fooo:toric1}
there is a factor $2\pi$ in the right hand side of
(\ref{elljandarea}). We eliminate it
by slightly changing the notation of
moment map (See Remark \ref{rem2pi1}.)
Note in \cite{fooo:toric1} the contribution
of the pseudo-holomorphic disc of homology class
$\beta$ in $\frak m_k$ has
weight $T^{\beta \cap \omega/2\pi}$.
In this paper and in \cite{fooo:book1, fooo:book2}
the weight is $T^{\beta \cap \omega}$.
\end{rem}

\subsection{Review of Floer cohomology of toric fiber}
\label{subsec:toricHFreview}

Let $\widehat{\mathcal H}^{2k}$ be the $\C$ vector space whose basis is given by
the cohomology classes which are the Poincar\'e-dual to the fundamental homology classes of
complex codimension $k$ submanifolds of $M$ which arise as transversal intersections
of $k$ irreducible components $D_{j_1},\dots,D_{j_k}$ of the
toric divisor.
For $k=0$ we let $\widehat{\mathcal H}^{0} = \C$ and its basis is
regarded as a codimension $0$ submanifold $M$ itself.
For $k\ne 0$ the inclusion map induces an isomorphism
\begin{equation}\label{homologyisoML}
\widehat{\mathcal H}^{2k}  \cong
H_{2n-2k}(M \setminus L(\text{\bf u});\C)
\cong H^{2k}(M, L(\text{\bf u});\C).
\end{equation}
There exists a short exact sequence
\begin{equation}\label{shortexacthom}
0
\to
H_{2n-2k}(M;\Z)
\to
H_{2n-2k}(M,L(\text{\bf u});\Z)
\to
H_{2n-2k-1}(L(\text{\bf u});\Z)
\to
0.
\end{equation}
Note that $L(\text{\bf u})$ is a torus and so $H(L(\text{\bf u});\Z) $
is a free abelian group.
We fix a splitting of (\ref{shortexacthom}) and identify
\begin{equation}\label{homidentify}
H_{2}(M,L(\text{\bf u});\Z)
\cong
H_{2}(M;\Z)
\oplus
H_{1}(L(\text{\bf u});\Z).
\end{equation}
For $k \neq 0$, we also fix a $\C$-linear subspace
${\mathcal H}^{2k} \subset \widehat{\mathcal H}^{2k}$ such that \index{${\mathcal H}^{2k}$}
the homomorphism induced by the inclusion $H_{2n-2k}(M \setminus L(\text{\bf u});\Z) \to
H_{2n-2k}(M;\Z)$ restricts to an isomorphism from ${\mathcal H}^{2k}$ to $H_{2n-2k}(M;\C)$.
For $k=0$, we have
$$
{\mathcal H}^{0}\cong \C \cong H^{0}(M ;\C)
$$
where the first isomorphism is given by definition and the second is induced by
the canonical unit of $H^{0}(M ;\C)$.
We recall that the odd degree cohomology of toric manifolds are
all trivial.
\par
We take the direct sum and denote
${\mathcal H} = \bigoplus_{k=0}^{n} {\mathcal H}^{2k}.$
We take its basis $\{PD([D_a]) \mid a=0,\dots, B\}$ so that
$D_0 = [M]$ (whose Poincar\'e dual is the unit),
each of $D_1,\dots,D_{B_2}$ is an irreducible component of the toric divisor
($B_2 = \text{\rm rank} \, H_2(M;\Q)$)
and $D_{B_2+1}\dots, D_B$ are transversal intersection of
irreducible components of the toric divisors.
($B+1 =\text{\rm rank}\, H(M;\Q)$.)
We put $\text{\bf e}_a^M = PD([D_a])$.
\par
We put $\underline B = \{1,\ldots,B\}$ and denote the set of all maps $\text{\bf p}:
\{1,\ldots,\ell\} \to \underline B$ by
$Map(\ell,\underline B)$. We write $\vert\text{\bf
p}\vert = \ell$ if $\text{\bf p} \in Map(\ell,\underline B)$.
\par
For $k,\ell \in \Z_{\ge 0}$, $\text{\bf p} \in Map(\ell,\underline B)$ and $\beta \in H_2(M,L(\text{\bf u});\Z)$ we
define a fiber product
\begin{equation}\label{fiberproductwithbeta}
\mathcal M_{k+1;\ell}(L(\text{\bf u});\beta;\text{\bf p})
=
\mathcal M_{k+1;\ell}(L(\text{\bf u});\beta)
{}_{(ev_1,\ldots,ev_{\ell})}
\times_{M^{\ell}} \prod_{i=1}^{\ell} D_{\text{\bf p}(i)},
\end{equation}
\index{$\mathcal M_{k+1;\ell}(L(\text{\bf u});\beta)$}
where $\mathcal M_{k+1;\ell}(L(\text{\bf u});\beta)$
is the moduli space defined in Definition \ref{diskmoduli1} and
Proposition \ref{disckura}.
\par
Let $\text{\rm Perm}(\ell)$ be the symmetric group of order $\ell!$.
It acts on $\mathcal M_{k+1;\ell}(L(\text{\bf u});\beta)$ as the
permutation of the interior marked points.
We define $\sigma \cdot \text{\bf p} = \text{\bf
p} \circ \sigma^{-1}$.
They induce a map
$\sigma_* : \mathcal M_{k+1;\ell}(L(\text{\bf u});\beta;\text{\bf p})
\to \mathcal M_{k+1;\ell}(L(\text{\bf u});\beta;\sigma\cdot\text{\bf p})$.
\par
Since $L(\text{\bf u})$ is a $T^n$ orbit,
$\mathcal M_{k+1;\ell}(L(\text{\bf u});\beta;\text{\bf p})$ has a
$T^n$ action induced by the action on $M$.
To describe the boundary of
$\mathcal M_{k+1;\ell}(L(\text{\bf u});\beta;\text{\bf p})$
we need to prepare some notations.
We will define a map
\begin{equation}\label{split}
\text{Split}: \text{\rm Shuff}(\ell) \times Map(\ell,\underline B)
\longrightarrow \bigcup_{\ell_1+\ell_2 = \ell} Map(\ell_1,\underline
B) \times Map(\ell_2,\underline B),
\end{equation}
as follows: Let $\text{\bf p} \in Map(\ell,\underline B)$ and
$(\mathbb L_1,\mathbb L_2) \in \text{\rm Shuff}(\ell)$. We put
$\ell_j = \# (\mathbb L_j)$ and let $\frak i_j: \{1,\ldots,\ell_j\}
\to \mathbb L_j$ be the order preserving bijection. We consider
the map $\text{\bf p}_j: \{1,\ldots,\ell_j\} \to \underline B$
defined by $ \text{\bf p}_j(i) = \text{\bf p}(\frak i_j(i)) $, and
set
$$
\text{Split}((\mathbb L_1,\mathbb L_2),\text{\bf p}): = (\text{\bf
p}_1,\text{\bf p}_2).
$$
\begin{lem}\label{torickuranishi}
\begin{enumerate}
\item
$\mathcal M_{k+1;\ell}(L(\text{\bf u});\beta;\text{\bf p})$ has
a Kuranishi structure with corners.
\item
The Kuranishi structure is invariant under the $T^n$ action.
\item
Its normalized boundary is described by the union of fiber products:
\begin{equation}
\mathcal
M_{k_1+1;\#\L_1}(L(\text{\bf u});\beta_1;\text{\bf p}_1)
{}_{{\rm ev}^{\partial}_0}\times_{{\rm ev}^{\partial}_i} \mathcal M_{k_2+1;\#\L_2}(L(\text{\bf u});\beta_2;\text{\bf p}_2)
\end{equation}
where the union is taken over all $(\L_1,\L_2) \in \text{\rm Shuff}(\ell)$,
$k_1,k_2$ with $k_1 + k_2 = k$ and $\beta_1,\beta_2 \in H_2(M,L(\text{\bf u});\Z)$ with
$\beta = \beta_1+\beta_2$. We put $
\text{\rm Split}((\mathbb L_1,\mathbb L_2),\text{\bf p}) = (\text{\bf
p}_1,\text{\bf p}_2).
$
\item The virtual dimension is given by
\begin{equation}
\dim \mathcal M_{k+1;\ell}(L(\text{\bf u});\beta;\text{\bf p})
= n + \mu_{L(\text{\bf u})}(\beta) + k - 2 + 2\ell - \sum_{i=1}^{\ell} 2\deg D_{\text{\bf p}(i)}.
\end{equation}
\item The evaluation maps ${\rm ev}^{\partial}_i$ at the boundary marked points of $\mathcal M_{k+1;\ell}(L(\text{\bf u});\beta)$
define  strongly continuous smooth maps on $\mathcal M_{k+1;\ell}(L(\text{\bf u});\beta;\text{\bf p})$, which we also denote by
${\rm ev}^{\partial}_i$.
They are compatible with the description $(3)$ of the boundary.
\item
We can define an orientation of the Kuranishi structure so that it is compatible with $(3)$.
\item The evaluation map
${\rm ev}^{\partial}_0$ is weakly submersive.
\item
The Kuranishi structure is compatible with the permutation of the
interior marked points.
\item
The Kuranishi structure is compatible with the forgetful map of
the $i$-th boundary marked point for $i=1,\ldots,k$.
(We do not require the compatibility with the forgetful map of  the $0$-th marked point.)
\end{enumerate}
\end{lem}
Lemma \ref{torickuranishi} is proved in \cite[Section 6]{fooo:bulk}  and
\cite[Section 4.3]{fooo:toricmir}.
See Figure \ref{Figure12}.
\par

\begin{lem}\label{toricmmultsect}
There exists a system of multisections on $\mathcal M_{k+1;\ell}(L(\text{\bf u});\beta;\text{\bf p})$
with the following properties:
\begin{enumerate}
\item They are transversal to $0$.
\item They are invariant under the $T^n$ action.
\item They are compatible with the description of
the boundary in Lemma $\ref{torickuranishi}$ $(3)$.
\item The restriction of ${\rm ev}^{\partial}_0$ to the zero set of this
multisection is a submersion.
\item They are invariant under permutation of the
interior marked points.
\item The multisection is compatible with the forgetful
map of the $i$-th boundary marked point for $i=1,\ldots,k$.
\end{enumerate}
\end{lem}
This is also proved in \cite[Section 6]{fooo:bulk}  and
\cite[Section 4.4]{fooo:toricmir}.
We note that (4) is a consequence of (2).
\par
Let $h_1,\dots,h_k \in \Omega(L(\text{\bf u}))$.
We then define a differential form
on $L(\text{\bf u})$ by
\index{$\frak q_{\ell,k;\beta}^{T}(\text{\bf p};h_1,\ldots,h_k)$}
\begin{equation}
\mathfrak q_{\ell,k;\beta}^{T}(\text{\bf p};h_1,\ldots,h_k)
= ({\rm ev}^{\partial}_0)! ({\rm ev}^{\partial}_1,\ldots,{\rm ev}^{\partial}_k)^*(h_1 \wedge \cdots \wedge h_k).
\label{37.188}\end{equation}
\index{operator $\frak q$!operator $\frak q^T$}
Here the superscript $T$ stands for $T^n$ equivariance.
By Lemma \ref{toricmmultsect} (4) the integration along the fibers $({\rm ev}^{\partial}_0)!$ is
well-defined.
Then Lemma \ref{toricmmultsect} (5) implies that the operation (\ref{37.188}) is
invariant under the permutation of the factors of $\text{\bf p}$.
Therefore we can $\C$-linearly extend the definition \eqref{37.188} to define a map
\begin{equation}\label{eq:mapqT}
\frak q^T_{\ell,k;\beta}:
E_{\ell} (\mathcal H[2]) \otimes B_k(\Omega(L(\text{\bf u}))[1]) \to  \Omega(L(\text{\bf u}))[1].
\end{equation}
Identification of the de Rham cohomology group $H(L(\text{\bf u});\C)$ of $L(\text{\bf u})$
with the subspace of $T^n$-invariant differential forms on the torus $L(\text{\bf u})$ and
the property stated in Lemma \ref{toricmmultsect} (2) imply that  the operations $\frak q^T_{\ell,k;\beta}$ induce
\begin{equation}\label{eq:mapqTcan}
\frak q^T_{\ell,k;\beta}:
E_{\ell} (\mathcal H[2]) \otimes B_k(H(L(\text{\bf u});\C)[1]) \to  H(L(\text{\bf u});\C)[1].
\end{equation}
For $\beta=\beta_0 = 0$, we define $\frak q^T_{0,k;\beta_0}$ to be the map given in
(\ref{defqformula2}).
Then the operators  $\frak q^T_{\ell,k;\beta}$ satisfy all the conclusions of Theorem \ref{qproperties}.
We use Theorem \ref{qproperties} to define
$\frak m_{k}^{T,\text{\bf b}}$
for
$\text{\bf b} = (\frak b_0,\text{\bf b}_{2;1},\frak b_+,b_+)$
in the same way as we did in Definition \ref{deformedqdef}.
We have thus obtained a filtered $A_{\infty}$ algebra
$(CF_{\text{\rm dR}}(L(\text{\bf u});\Lambda_0),\{\frak m_k^{T,\text{\bf b}}
\}_{k=0}^{\infty})$ with
$$
CF_{\text{\rm dR}}(L(\text{\bf u});\Lambda_0)
= \Omega(L(\text{\bf u}))\widehat{\otimes} \Lambda_0.
$$
This is the filtered $A_{\infty}$ algebra used in \cite{fooo:bulk}.
(In \cite{fooo:bulk} $\frak q^{T}$ is denoted by $\frak q^{dR}$.)
In particular, if $\frak m_0^{T,\text{\bf b}}(1) \equiv 0 \mod \text{\bf e}_L\Lambda_+$,
$\frak m_1^{T,\text{\bf b}}\circ \frak m_1^{T,\text{\bf b}} = 0$.
We put $\delta^{T,\text{\bf b}} = \frak m_1^{T,\text{\bf b}}$ and
define its associated Floer cohomology by
\index{$HF_T((L,\text{\bf b});\Lambda_0)$}
\begin{equation}\label{THFdef}
HF_T((L,\text{\bf b});\Lambda_0)
=
\frac{\text{\rm Ker} \,\delta^{T,\text{\bf b}}}
{\text{\rm Im} \,\delta^{T,\text{\bf b}}}.
\end{equation}

We put subscript $T$ in the notation to indicate that we are using
a $T^n$-equivariant perturbation. In a series of papers
\cite{fooo:toric1,fooo:bulk,fooo:toricmir} we studied Floer cohomology (\ref{THFdef})
and described its nonvanishing property in terms of the critical point
theory of certain non-Archimedean analytic function, called the \emph{potential
function}. An explanation of this potential function is given in Subsection \ref{subsec:toricHFproperties}.

\begin{rem}\label{remark206666}
Here we note that we use a system of multisections and integrations on their zero sets in Lemma \ref{toricmmultsect}.
They are described in \cite[appendix C]{fooo:toric1}.
In other part of this paper, we use a system of CF-perturbations and define integration along the fiber
using CF-perturbation.
See Remarks \ref{rem28777} and \ref{multiseccont}.
\end{rem}

\subsection{Relationship with the Floer cohomology in Section \ref{sec:q-map}}
\label{subsec:relationHF}

To apply the Floer cohomology  (\ref{THFdef})
for the purpose of studying spectral invariants,
we need to show that (\ref{THFdef}) is isomorphic to the
Floer cohomology we used in Part 3.
We use the next proposition for this purpose.
We denote $\text{\bf b}^{(0)} = (\frak b_0,\text{\bf b}_{2;1},\frak b_+,0)$ as before.
\begin{prop}\label{Ainfinityequiv} Let $\text{\bf u} \in \text{\rm Int} P$.
The filtered $A_{\infty}$ algebra
$(CF_{\text{\rm dR}}(L(\text{\bf u});\Lambda_0),\{\frak m_k^{\text{\bf b}^{(0)}}
\}_{k=0}^{\infty})$
is homotopy equivalent to
$(CF_{\text{\rm dR}}(L(\text{\bf u});\Lambda_0),\{\frak m_k^{T,\text{\bf b}^{(0)}}
\}_{k=0}^{\infty})$
as a unital filtered $A_{\infty}$ algebra.
\end{prop}
Here the first filtered $A_{\infty}$ algebra is defined in Lemma \ref{bulkdef}
and the second one is defined at the end of Subsection \ref{subsec:toricHFreview}.
They are de Rham versions of
the filtered $A_{\infty}$ algebra associated to a Lagrangian submanifold,
which was established in \cite[Theorem A]{fooo:book1}.
Since the two constructions are slightly
different from each other in technical points, we give a proof of
Proposition \ref{Ainfinityequiv} in Section \ref{sec:appendix3}
for completeness' sake.
\par
Let $\text{\bf e}_L = 1$ be the differential $0$-form on $L$
which is the unit of our filtered $A_{\infty}$ algebra.
We put:
\index{$\widehat{\mathcal M}_{\text{\rm weak},\text{\rm def}}
(L(\text{\bf u});\Lambda_+;\text{\bf b}^{(0)})$}
\index{$\widehat{\mathcal M}^T_{\text{\rm weak},\text{\rm def}}
(L(\text{\bf u});\Lambda_+;\text{\bf b}^{(0)})$}
\begin{equation}\label{MC111}
\aligned
&\widehat{\mathcal M}_{\text{\rm weak},\text{\rm def}}
(L(\text{\bf u});\Lambda_+;\text{\bf b}^{(0)}) \\
&=
\left\{
b_+ \in H^{odd}(L(\text{\bf u});\Lambda_+)\,
\Big\vert \, \sum_{k=0}^{\infty} \frak m_k^{\text{\bf b}^{(0)}}(b_+^k)
\equiv 0 \mod \text{\bf e}_L \Lambda_+
\right\}
\endaligned
\end{equation}
and
\begin{equation}
\aligned
&\widehat{\mathcal M}^T_{\text{\rm weak},\text{\rm def}}
(L(\text{\bf u});\Lambda_+;\text{\bf b}^{(0)}) \\
&=
\left\{
b_+ \in H^{odd}(L(\text{\bf u});\Lambda_+)\,
\Big\vert\, \sum_{k=0}^{\infty} \frak m_k^{T,\text{\bf b}^{(0)}}(b_+^k)
\equiv 0 \mod \text{\bf e}_L \Lambda_+
\right\}.
\endaligned
\end{equation}
We write $\text{\bf b}(b_+)
= (\frak b_0,\text{\bf b}_{2;1},\frak b_+,b_+)$.
Then
$$
\text{\bf b}(b_+)
\in
\widehat{\mathcal M}_{\text{\rm weak},\text{\rm def}}
(L(\text{\bf u});\Lambda_0)
$$
for
$b_+
\in \widehat{\mathcal M}_{\text{\rm weak},\text{\rm def}}
(L(\text{\bf u});\Lambda_+;\text{\bf b}^{(0)})$.
\par
Similar fact holds for
$b_+
\in \widehat{\mathcal M}^T_{\text{\rm weak},\text{\rm def}}
(L(\text{\bf u});\Lambda_+;\text{\bf b}^{(0)})$.
\par
Proposition \ref{Ainfinityequiv} and the
homotopy theory of filtered $A_{\infty}$ algebras as given in
\cite[Chapter 4]{fooo:book1}  immediately imply the following:

\begin{cor}
There exists a map
$$
\mathfrak J_* :
\widehat{\mathcal M}^T_{\text{\rm weak},\text{\rm def}}
(L(\text{\bf u});\Lambda_+;\text{\bf b}^{(0)})
\to
\widehat{\mathcal M}_{\text{\rm weak},\text{\rm def}}
(L(\text{\bf u});\Lambda_+;\text{\bf b}^{(0)})
$$
such that for each $b_+
\in \widehat{\mathcal M}^T_{\text{\rm weak},\text{\rm def}}
(L(\text{\bf u});\Lambda_+;\text{\bf b}^{(0)})$
there exists a chain homotopy equivalence
$$
\frak J_*^{b_+} :
(CF(L(\text{\bf u}));\Lambda),\delta^{T,\text{\bf b}(b_+)})
\to
(CF(L(\text{\bf u}));\Lambda),\delta^{\text{\bf b}(\mathfrak J_*(b_+))})
$$
that preserves the filtration.
\end{cor}
(We note that $\mathfrak J_*$ induces an isomorphism after
taking a gauge equivalence. We do not use this fact in this paper.)
\par
We next use $\frak q^T_{\ell,k;\beta}$ in place of
$\frak q_{\ell,k;\beta}$ in (\ref{iqmdefformula}) to define a chain map
\index{$i_{\text{\rm qm},\text{\bf b}}^T$}
$$
i_{\text{\rm qm},\text{\bf b}}^T
: (\Omega(M) \widehat{\otimes} \Lambda,d)
\to (CF(L(\text{\bf u});\Lambda),\delta^{T,\text{\bf b}}).
$$\index{$i_{\text{\rm qm},\text{\bf b}}^T$}
\begin{lem}\label{ainfinityequivcomp}
$\frak J_*^{b_+}\circ i_{\text{\rm qm},\text{\bf b}}$ is chain homotopic
to $i_{\text{\rm qm},\text{\bf b}}^T$.
\end{lem}
The proof is parallel to Proposition \ref{Ainfinityequiv}  and is given in Section
\ref{sec:appendix3}.
\par
Now the following is an immediate consequence.
\begin{cor}
If we replace the condition \eqref{localnonzero} in Theorem $\ref{thm:heavy}$
by  $HF((L,\text{\bf b});\Lambda)$ and
$i_{\text{\rm qm}}^{\ast}$ by $HF_T((L,\text{\bf b});\Lambda)$ and
$i^{T, \ast}_{\text{\rm qm}}$ respectively, then the conclusion of Theorem $\ref{thm:heavy}$
still holds.
\end{cor}

\subsection{Properties of Floer cohomology $HF_T((L,\text{\bf b});\Lambda)$:
review}
\label{subsec:toricHFproperties}

We now go back to the study of Floer cohomology $HF_T((L,\text{\bf b});\Lambda)$
introduced in Subsection \ref{subsec:toricHFreview}. The main properties of
$HF_T((L,\text{\bf b});\Lambda)$ were established in \cite{fooo:toric1,fooo:bulk}.

\begin{prop}\label{H1iswbdchain}
If $b_+ \in H^1(L(\text{\bf u});\Lambda_+)$, then
$$
\frak m_0^{T,\text{\bf b}(b_+)}(1) \equiv 0 \mod \text{\bf e}_L\Lambda_+.
$$
In particular, we have an embedding $H^1(L(\text{\bf u});\Lambda_+) \hookrightarrow
\widehat{\mathcal M}^T_{\text{\rm weak},\text{\rm def}}(L(\text{\bf u});\Lambda_+;\text{\bf b}^{(0)})$.
\end{prop}

This is nothing but \cite[Proposition 4.3]{fooo:toric1}  and \cite[Proposition 3.1]{fooo:bulk}.
We omit its proof and refer readers to the above references for the details.
The proof is based on a dimension counting argument.
We remark that the proof of Proposition \ref{H1iswbdchain}
does not work if we replace $\frak m_0^{T,\text{\bf b}(b_+)}$
by $\frak m_0^{\text{\bf b}(b_+)}$.
This is because we used a {\it continuous family} of multisections (or CF-perturbations)
for the definition of $\frak m_k^{\text{\bf b}(b_+)}$ in Section \ref{sec:q-map}, which
obstructs the above mentioned dimension counting argument.
(See \cite[Remark 3.2.2]{fooo:toricmir}.)
Actually this is {\it the} reason why we use $T^n$-invariant cycles $D_a$ instead of differential forms
to represent cohomology classes of $M$ in \cite{fooo:toric1,fooo:bulk} and in this section.
\par
Now we consider $\text{\bf b}^{(0)} = (\frak b_0,\text{\bf b}_{2;1},\frak b_+,0)$
and $b_+ \in H^1(L(\text{\bf u});\Lambda_+)$.
By Proposition \ref{H1iswbdchain} we have
\begin{equation}
\text{\bf b}(b_+) \in
\widehat{\mathcal M}^T_{\text{\rm weak},\text{\rm def}}
(L(\text{\bf u});\Lambda_0).
\end{equation}
We use the splitting (\ref{homidentify}) to regard
$
\text{\bf b}(b_+)
= (\frak b_0,\text{\bf b}_{2;1},\frak b_+,b_+)
$ as an element of
$
H^{{\rm even}}(M;\Lambda_0) \oplus H^{1}(L(\text{\bf u});\Lambda_0).
$
So hereafter we define
$HF_T((L(\text{\bf u}),(\frak b,b));\Lambda_0)$
for $\frak b \in H^{{\rm even}}(M;\Lambda_0)$ and $b
\in  H^{1}(L(\text{\bf u});\Lambda_0)$.
This is the Floer cohomology we studied in \cite{fooo:bulk}.
\par
We define the potential function
$$
\frak{PO}^{\text{\bf u}}
:
H^{{\rm even}}(M;\Lambda_0) \times H^{1}(L(\text{\bf u});\Lambda_0)
\to
\Lambda_+
$$
by the equation
\begin{equation}\label{POtoric}
\frak m_0^{T,(\frak b,b)}(1)
=
\frak{PO}^{\text{\bf u}}(\frak b,b)\,\text{\bf e}_L.
\end{equation}
We now review the results of \cite{fooo:toric1,fooo:bulk}
on the potential function $\frak{PO}^{\text{\bf u}}$\index{potential function} and
its role in the study of Floer cohomology.
\par
We let $\text{\bf u} \in \text{\rm Int}P$ and fix a basis $\{\text{\bf e}_i\}_{i=1}^n$ of $H^1(L(\text{\bf u});\Z)$.
Identifying $L(\text{\bf u})$ with $T^n$ by the action,
we can find a basis  $\{\text{\bf e}_i\}_{i=1}^n$ for all $\text{\bf u} \in \text{\rm Int} P$
simultaneously in a canonical way.
Let $b \in H^1(L(\text{\bf u});\Lambda_0)$ express it as a linear combination
\begin{equation}\label{coordinatexu}
b = \sum x^{\text{\bf u}}_i \text{\bf e}_i
\end{equation}
where $x^{\text{\bf u}}_i \in \Lambda_0$. Thus $(x^{\text{\bf u}}_1,\dots,x^{\text{\bf u}}_n)$ are coordinates on
$H^1(L(\text{\bf u});\Lambda_0)$.
(To specify that the coordinate system is defined on the torus $L(\text{\bf u})$
associated to $\text{\bf u}$
we put ${\text{\bf u}}$ in the expression $x_i^{\text{\bf u}}$ above.)
Let $x_i^{\text{\bf u}} = x_{i,0}^{\text{\bf u}} + x_{i,+}^{\text{\bf u}}$
where $x_{i,0}^{\text{\bf u}} \in \C$ and
$x_{i,+}^{\text{\bf u}} \in \Lambda_+$. We put
\begin{equation}\label{yisexpx}
y_i^{\text{\bf u}} = \exp({x_{i,0}^{\text{\bf u}}}) \exp({x_{i,+}^{\text{\bf u}}})
\in \Lambda_{0} \setminus \Lambda_{+}.
\end{equation}
We note that $\exp({x_{i,0}^{\text{\bf u}}}) \in \C \setminus \{0\}$ makes sense in the
usual Archimedean sense, and
$$
\exp({x_{i,+}^{\text{\bf u}}}) := \sum_{k=0}^{\infty} (x_{i,+}^{\text{\bf u}})^k/k!
$$
converges in $T$-adic topology since $\frak v_T({x_{i,+}^{\text{\bf u}}})> 0$.
\par
Let $S^1_i$ be the $i$-th factor of $T^n$ which represents the basis element
$\text{\bf e}_i$. We choose our moment map $\pi : M \to \R^n$ so that
its $i$-th component is the moment map of the $S^1_i$ action.
In this way we fix the coordinates on the affine space $\R^n$ which contains
$P$.
Note that there is still a freedom to choose the origin $\text{\bf 0} \in \R^n$.
We do not specify this choice since it does not affect the story.
\par
Let
$
\text{\bf u} = (u_1,\dots,u_n) \in \text{\rm Int} P.
$
We put
\begin{equation}\label{yandyu}
y_i = T^{u_i}y_i^{\text{\bf u}}.
\end{equation}
We do not put $\text{\bf u}$ in the notation $y_i$ above.
This is justified by Theorem \ref{strongconvPO}.
\begin{rem}
For the notational convenience we assume $\text{\bf 0} \in P$.
Then we will have $y_i = y_i^{\text{\bf 0}}$.
\end{rem}
\par
With respect to the above coordinates, we may regard
$\frak{PO}^{\text{\bf u}}$ as a function of $(x_1^{\text{\bf u}}, \cdots, x_n^{\text{\bf u}})$
\index{$\frak{PO}^{\text{\bf u}}_{\frak b}$}
$$
\frak{PO}^{\text{\bf u}}(\frak b;b)
=
\frak{PO}^{\text{\bf u}}_{\frak b}(x^{\text{\bf u}}_1,\dots,x^{\text{\bf u}}_n)
$$
where $x^{\text{\bf u}}_k$ ($k=1,\dots,n$) are the variables defined in (\ref{coordinatexu}).
\par
As we will see in Theorem \ref{strongconvPO}, $\frak{PO}^{\text{\bf u}}_{\frak b}$ as a function of
$y_1,\dots,y_n$ will be independent of $\text{\bf u}$.
The resulting function is contained in an appropriate
completion of the Laurent polynomial ring
$\Lambda[y_1,\dots,y_n,y_1^{-1},\dots,y_n^{-1}]$.
Description of this completion is in order now.
By the change of variables (\ref{yandyu}), there exists an isomorphism
$$
\Lambda[y_1,\dots,y_n,y_1^{-1},\dots,y_n^{-1}]
\cong
\Lambda[y^{\text{\bf u}}_1,\dots,y^{\text{\bf u}}_n,(y^{\text{\bf u}}_1)^{-1},\dots,(y^{\text{\bf u}}_n)^{-1}]
$$
for any $\text{\bf u} \in \text{\rm Int} P$. In other words
any element of $\frak P \in \Lambda[y_1,\dots,y_n,y_1^{-1},\dots,y_n^{-1}]$ can be written
as a {\it finite} sum
\begin{equation}\label{laurenexpressu}
\frak P
=
\sum a_{k_1,\dots,k_n} (y_1^{\text{\bf u}})^{k_1}\dots
(y_n^{\text{\bf u}})^{k_n}.
\end{equation}
In other words $a_{k_1,\dots,k_n}\in \Lambda$ are zero except for a finite number of them.
We define a valuation
\index{$\frak v_{\text{\bf u}}$}
\begin{equation}
\frak v_{\text{\bf u}} (\frak P)
= \min \{ \frak v_T(a_{k_1,\dots,k_n}) \mid a_{k_1,\dots,k_n} \ne 0\}
\end{equation}
for each ${\bf u} \in \text{\rm Int}P$.
In particular $\frak v_{\text{\bf u}}(y_i^{\text{\bf u}}) = 0$.
This defines a family of non-Archimedean valuations on
the ring $\Lambda[y_1,\ldots,y_n,y_1^{-1},\ldots,y_n^{-1}]$.
This valuation is characterized by the equation
$$
{\frak v}_{\text{\bf u}}(y_i) = u_i
$$
for $\text{\bf u} \in \text{Int} P$. From this expression, the family
can be extended to the full closed moment polytope $P$ in the obvious way.

We now put
\index{$\frak v_{P}$}
$$
\frak v_{P}(\frak P)
= \inf \{ \frak v_{\text{\bf u}} (\frak P) \mid \text{\bf u} \in P \}.
$$
This defines a norm (but not a valuation) on $\Lambda[y_1,\dots,y_n,y_1^{-1},\dots,y_n^{-1}]$ and
the function $d_P$ defined by
\begin{equation}
d_P(\frak P,\frak Q)
= \exp^{-\frak v_{P}(\frak P - \frak Q)}
\end{equation}
gives rise to a metric thereon.
\par
For $\epsilon >0$, denote
$$
P_{\epsilon} = \{ \text{\bf u} \in P \mid \forall i\,\, \ell_i(\text{\bf u}) \ge \epsilon\}.
$$
We define another metric on
$\Lambda[y_1,\dots,y_n,y_1^{-1},\dots,y_n^{-1}]$ by
\begin{equation}
d_{\overset{\circ}P}(\frak P,\frak Q)
= \sum_{n=n_0}^{\infty}2^{-n}\exp^{-\frak v_{P_{1/n}}(\frak P - \frak Q)}.
\end{equation}
(Here we take $n_0$ sufficiently large so that $P_{1/n_0}$ is
nonempty.) This series obviously converges because $\frak
v_{P_{\epsilon'}} \leq \frak v_{P_{\epsilon}}$ if $\epsilon' <
\epsilon$.
\begin{defn}\label{def:completion}\index{completion of Laurent polynomial ring}
We denote
the completion of $\Lambda[y_1,\dots,y_n,y_1^{-1},\dots,y_n^{-1}]$
with respect to the metric $d_P$ by
$\Lambda\langle\!\langle y,y^{-1}\rangle\!\rangle^P$.
\index{$\Lambda\leftineqineq y,y^{-1}\rightineqineq^P$}
\par
We denote by
$
\Lambda\langle\!\langle
y,y^{-1}\rangle\!\rangle^{\overset{\circ}P}
$
\index{$\Lambda \leftineqineq y,y^{-1}\rightineqineq^{\overset{\circ}P}$}
the completion of $\Lambda[y_1,\dots,y_n,y_1^{-1},\dots,y_n^{-1}]$
with respect to the metric $d_{\overset{\circ}P}$.
\end{defn}
In other words, $\Lambda\langle\!\langle
y,y^{-1}\rangle\!\rangle^P$ (resp. $\Lambda\langle\!\langle
y,y^{-1}\rangle\!\rangle^{\overset{\circ}P}$) is the set of all $\frak P$'s such that
for any $\text{\bf u} \in P$ (resp. $\text{\bf u} \in \text{Int}\,P$)
we may write $\frak P$ as a
{\it possibly infinite} sum of the form (\ref{laurenexpressu})
such that
$\lim_{\vert k_1\vert + \cdots + \vert k_n\vert \to \infty}
\frak v_T(a_{k_1,\dots,k_n}) = +\infty$.
\begin{rem}\label{rem:notationtoric2}
In \cite{fooo:bulk}, we used a slightly different notation
$\Lambda^P\langle\!\langle
y,y^{-1}\rangle\!\rangle$ instead of $\Lambda\langle\!\langle
y,y^{-1}\rangle\!\rangle^P$.
\end{rem}

Now we have:
\begin{thm}\label{strongconvPO}
If $\frak b \in H^{{\rm even}}(M;\Lambda_0)$, then
\begin{equation}\label{POinstrictlyconv}
\frak{PO}_{\frak b}^{\text{\bf u}} \in
\Lambda\langle\!\langle
y,y^{-1}\rangle\!\rangle^{\overset{\circ}P}.
\end{equation}
\par
If
$\frak b \in H^{{\rm even}}(M;\Lambda_+)$, then
\begin{equation}\label{POinstrictlyconv2}
\frak{PO}_{\frak b}^{\text{\bf u}} \in
\Lambda\langle\!\langle
y,y^{-1}\rangle\!\rangle^P.
\end{equation}
\end{thm}
We explain the meaning of (\ref{POinstrictlyconv}).
Let $\frak P \in \Lambda\langle\!\langle
y,y^{-1}\rangle\!\rangle^{\overset{\circ}P}$
and $\text{\bf u} \in \text{Int}\,P$. As we mention above,
$\frak P$ is written as a series  of the form (\ref{laurenexpressu})
with
$$
\lim_{\vert k_1\vert + \cdots + \vert k_n\vert \to \infty}
\frak v_T(a_{k_1,\dots,k_n}) = +\infty.
$$
Let $b = \sum x^{\text{\bf u}}_i \text{\bf e}_i$.
Then by putting (\ref{yisexpx})
and plugging it in (\ref{laurenexpressu})
the series converges in $T$-adic topology and
we obtain an element of $\Lambda$.
Thus we obtain a function
$$
\frak P^{\text{\bf u}} : H^1(L(\text{\bf u});\Lambda) \to \Lambda.
$$
The statement (\ref{POinstrictlyconv}) means that there exists
$\frak P \in \Lambda\langle\!\langle
y,y^{-1}\rangle\!\rangle^{\overset{\circ}P}$ such that the above $\frak P^{\text{\bf
u}} $ coincides with $\frak{PO}^{\text{\bf u}}_{\frak b}$ for {\it
any} $\text{\bf u} \in \text{Int}\,P$. (We note that we require
$\frak P$ to be {\it independent} of $\text{\bf u}$.)
The meaning of (\ref{POinstrictlyconv2}) is similar.
\par
Actually we can show the following:
\begin{lem}
Let $\frak P \in  \Lambda\langle\!\langle
y,y^{-1}\rangle\!\rangle^{\overset{\circ}P}$. Then  $\frak P$
is written as a series
\begin{equation}\label{Pexpand2}
\frak P=
\sum a_{k_1,\dots,k_n} y_1^{k_1}\dots
y_n^{k_n}
\end{equation}
which converges in $d_{\overset{\circ}P}$ topology. For any $(\frak
y_1,\dots,\frak y_n) \in \Lambda^n$ with
$$
(\frak v_T(\frak y_1),\dots,\frak v_T(\frak y_n)) \in \text{\rm Int}\,P
$$
the series
\begin{equation}\label{Pexpand25}
\sum a_{k_1,\dots,k_n} \frak y_1^{k_1}\dots
\frak y_n^{k_n}
\end{equation}
converges in $T$-adic topology.
\par
Let  $\frak P \in  \Lambda\langle\!\langle
y,y^{-1}\rangle\!\rangle^P$. Then
$\frak P$ is written as a series
$(\ref{Pexpand2})$
which converges in $d_{P}$ topology.
For any $(\frak y_1,\dots,\frak y_n) \in \Lambda^n$ with
$$
(\frak v_T(\frak y_1),\dots,\frak v_T(\frak y_n)) \in P
$$
the series $(\ref{Pexpand25})$ converges in $T$-adic topology.
\end{lem}
The proof is elementary and so omitted.
\par
Theorem \ref{strongconvPO} is \cite[Theorem 3.14]{fooo:bulk}.
We do not discuss its proof in this paper but refer to \cite{fooo:bulk}
for the details.
\par
We next discuss the relationship between the potential function
and the nonvanishing of Floer cohomology.
We first note that one can define the logarithmic derivative
\begin{equation}\label{Pderivay}
y_i\frac{\partial \frak P}{\partial y_i}
\end{equation}
for an element $\frak P$ of
$\Lambda\langle\!\langle
y,y^{-1}\rangle\!\rangle^{\overset{\circ}P}$.
In fact, regarding the expression (\ref{Pexpand2}) of $\frak P$ as a power series,
we define
$$
y_i\frac{\partial \frak P}{\partial y_i}
=
\sum a_{k_1,\dots,k_n} k_iy_1^{k_1}\dots
y_n^{k_n}.
$$
It is easy to see that this series converges with respect to
$d_{\overset{\circ}P}$-topology and so defines an element of
$\Lambda\langle\!\langle
y,y^{-1}\rangle\!\rangle^{\overset{\circ}P}$.

\begin{defn}
Let $\frak P \in \Lambda\langle\!\langle
y,y^{-1}\rangle\!\rangle^{\overset{\circ}P}$
and $\frak y = (\frak y_1,\dots,\frak y_n) \in \Lambda^n$ with
$$
(\frak v_T(\frak y_1),\dots,\frak v_T(\frak y_n))
\in \text{\rm Int}\, P.
$$
We say that $\frak y$ is a {\it critical point} of $\frak P$ if it satisfies
$$
\left(y_i\frac{\partial \frak P}{\partial y_i}\right)(\frak y) = 0
$$
for all $i =1,\dots,n$.
\par
For each critical point $\frak y$, we define a point $\text{\bf u}(\frak y) \in \text{\rm Int}\, P$ by
\begin{equation}\label{budeffromy1}
\text{\bf u}(\frak y) = (\frak v_T(\frak y_1),\dots,\frak v_T(\frak y_n)),
\end{equation}
and an element $b = b(\frak y) \in H^1(L(\text{\bf u}(\frak y)),\Lambda_0)$ by
\begin{equation}\label{budeffromy2t}
b(\frak y) = \sum x^\frak y_i e_i,
\quad
x^\frak y_i
= \log (T^{-\frak v_T(\frak y_i)}y_i).
\end{equation}
\end{defn}

Here the meaning of $\log$ in  (\ref{budeffromy2t}) is as follows.
Note that $\frak v_T(T^{-\frak v_T(\frak y_i)}y_i) = 0$. Therefore we can write
$$
T^{-\frak v_T(\frak y_i)}y_i
= c_1 (1 + c_2)
$$
for some $c_1 \in \C \setminus \{0\}$, $c_2 \in \Lambda_+$.
Then we define
$$
\log (T^{-\frak v_T(\frak y_i)}y_i )
= \log c_1 + \sum_{n=1}^{\infty} (-1)^{n}\frac{c_2^{n+1}}{n+1}.
$$
(Here we choose a branch of $\log c_1$ so that its imaginary part lies in $[0,2\pi)$, for example.)

\begin{thm}\label{Floercrit} Let $\frak b \in H^{{\rm even}}(M;\Lambda_0)$.
If $\frak y$ is a critical point of $\frak{PO}_{\frak b}$,
\index{potential function!critical point}
$$
HF((L(\text{\bf u}(\frak y)),(\frak b,b(\frak y));\Lambda)
\cong
H(T^n;\Lambda).
$$
\par
Conversely if
$$
HF((L(\text{\bf u}),(\frak b,b));\Lambda) \ne 0,
$$
there exists a critical point $\frak y$ of $\frak{PO}_{\frak b}$
such that
$$
\text{\bf u} = \text{\bf u}(\frak y),
\qquad
b = b(\frak y).
$$
\end{thm}
Theorem \ref{Floercrit} is nothing but \cite[Theorem 5.5]{fooo:bulk}.
We refer readers to \cite{fooo:bulk} for its proof.
\par
We next describe the relation of $\frak{PO}_{\frak b}$ to the
quantum cohomology.
Consider the closed ideal of the Fr\'echet ring
$\Lambda\langle\!\langle
y,y^{-1}\rangle\!\rangle^{\overset{\circ}P}$ generated by
$\left\{y_i\frac{\partial \frak{PO}_{\frak b}}{\partial y_i} \mid  i=1,\dots,n\right\}$.
We denote the quotient ring by
$$
\mathrm {Jac}(\frak{PO}_{\frak b};\Lambda)
=
\frac{\Lambda\langle\!\langle
y,y^{-1}\rangle\!\rangle^{\overset{\circ}P}}
{\text{Clos}_{\overset{\circ}d}\left(
y_i\frac{\partial \frak{PO}_{\frak b}}{\partial y_i}:
i=1,\dots,n
\right)}
$$
\index{$\text{\rm Jac}(\frak{PO}_{\frak b};\Lambda)$}\index{Jacobian ring}
which we call the {\it Jacobian ring} of $\frak{PO}_{\frak b}$.
We define a map
$$
\frak{ks}_{\frak b} :
H(M;\Lambda) \to  \mathrm {Jac}(\frak{PO}_{\frak b};\Lambda)
$$
\index{$\frak{ks}_{\frak b}$}
called {\it Kodaira-Spencer map}
\index{Kodaira-Spencer map $\frak{ks}_{\frak b}$}
as follows.
Let $\{\text{\bf e}^{M}_i\}$ be a basis of $H(M;\Q)$.
We write an element of $H(M;\Lambda)$ as
$\sum w_i \text{\bf e}^{M}_i$, $w_i \in \Lambda$.
We may express
$$
\frak{PO}_\frak b
= \sum a_{k_1,\dots,k_n}(\frak b) y_1^{k_1}\dots y_n^{k_n},
$$
where $a_{k_1,\dots,k_n}(\frak b)$ is a function of
$w_i$ (where $\frak b = \sum w_i \text{\bf e}^{M}_i$).
(See (\ref{fomulaforPO4}) for the precise expression of $\frak{PO}_\frak b$.)
Then $a_{k_1,\dots,k_n}(\frak b)$ is a formal power series
of $w_i$ with coefficients in $\Lambda$ which  converges in
$T$-adic topology.

Therefore the partial derivatives
$
\frac{\partial a_{k_1,\dots,k_n}}{\partial w_i}
$
are well-defined. Then we have
$$
\frac{\partial \frak{PO}}{\partial w_i}(\frak b)
=
\sum \frac{\partial a_{k_1,\dots,k_n}}{\partial w_i}(\frak b) y_1^{k_1}\dots
y_n^{k_n}.
$$
For each $\frak b \in H(M;\Lambda_0)$, the right hand side converges and
defines an element of $\Lambda\langle\!\langle
y,y^{-1}\rangle\!\rangle^{\overset{\circ}P}$.
\par
Now we define the map $\frak{ks}_{\frak b}$ by setting its value to be
\begin{equation}\label{defKS}
\frak{ks}_{\frak b}({\bf e}^{M}_i)
=
\left[
\frac{\partial \frak{PO}}{\partial w_i}(\frak b)
\right] \in \mathrm {Jac}(\frak{PO}_{\frak b};\Lambda).
\end{equation}

\begin{thm}\label{JacisHQ} The map
$\frak{ks}_{\frak b}$ defines a ring isomorphism
$$
(QH(M;\Lambda),\cup^{\frak b}) \cong  \mathrm {Jac}(\frak{PO}_{\frak b};\Lambda).
$$
\end{thm}
This is \cite[ Theorem 1.1.1]{fooo:toricmir} for whose proof we refer readers thereto.
\begin{rem}
\cite[Theorem 1.1.1]{fooo:toricmir}  is stated as a result over
$\Lambda_0$-coefficients which is stronger than Theorem \ref{JacisHQ}.
We do not use this stronger isomorphism over $\Lambda_0$-coefficients in the present paper.
\end{rem}
\par
We also need a result on the structure of the Jacobian ring
$\mathrm {Jac}(\frak{PO}_{\frak b};\Lambda)$.
\begin{defn}
We say a critical point $\frak y$ of
$\frak{PO}_{\frak b}$ is {\it nondegenerate} if
$$
\det \left[ y_iy_j\frac{\partial^2\frak{PO}_{\frak b}}{\partial y_i\partial y_j}
\right]_{i,j=1}^{i,j=n} (\frak y) \neq 0.
$$
 We say $\frak{PO}_{\frak b}$ is a {\it Morse function} if
all of its critical points are nondegenerate.
\end{defn}
Let $\text{\rm Crit}(\frak{PO}_{\frak b})$ be the set of all critical
points of $\frak{PO}_{\frak b}$.
\index{$\text{\rm Crit}(\frak{PO}_{\frak b})$}

\begin{defn}\label{def:Jaceigensp}
For $\frak y =(\frak y_1,\dots,\frak y_n) \in \text{\rm Crit}(\frak{PO}_{\frak b})$,
we define the subset $\text{\rm Jac}(\frak{PO}_{\frak b};\frak y) \subset
\mathrm {Jac}(\frak{PO}_{\frak b};\Lambda)$ as follows:
If we regard $y_i \in \Lambda\langle\!\langle
y,y^{-1}\rangle\!\rangle^{\overset{\circ}P}$, the multiplication by
$y_i$ induces an action on
$\text{\rm Jac}(\frak{PO}_{\frak b}; \Lambda)$. We denote the corresponding
endomorphism by $\widehat y_i$. Then we put
\begin{equation}\label{2035}
\aligned
\text{\rm Jac}(\frak{PO}_{\frak b};\frak y)
=
\{
x \in \text{\rm Jac}(\frak{PO}_{\frak b}; \Lambda)
\mid \
&(\widehat y_i - \frak y_i)^N x = 0,
\\
&\text{for all $i$ and sufficiently large $N$}
\}.
\endaligned
\end{equation}
\end{defn}
\begin{prop}\label{Morsesplit}
\begin{enumerate}
\item There is a factorization of the Jacobian ring
$$
\text{\rm Jac}(\frak{PO}_{\frak b}; \Lambda) \cong \prod_{\frak y \in \text{\rm Crit}(\frak{PO}_{\frak b})}
\text{\rm Jac}(\frak{PO}_{\frak b};\frak y)
$$
as a direct product of rings. \index{$\text{\rm Jac}(\frak{PO}_{\frak b};\frak y)$}
\item
For each critical point $\frak y \in \text{\rm Crit}(\frak P\frak O_{\frak b})$,
$\text{\rm Jac}(\frak{PO}_{\frak b};\frak y)$ is a local ring.
\item $\frak y$ is a nondegenerate critical point of $\frak P\frak O_{\frak b}$ if and only if
$\text{\rm Jac}(\frak{PO}_{\frak b};\frak y) \cong \Lambda$.
\end{enumerate}
\end{prop}

Proposition \ref{Morsesplit} is \cite[Proposition 1.2.16]{fooo:toricmir}, to which
we refer readers for its proof.

It follows from Proposition \ref{Morsesplit} that
the set of indecomposable idempotents of $\mathrm {Jac}(\frak{PO}_{\frak b};\Lambda)$
one-one corresponds to $\text{\rm Crit}(\frak{PO}_{\frak b})$.
We denote by $1_{\frak y} \in \text{\rm Jac}(\frak{PO}_{\frak b};\frak y)$ the unit of
the ring $\text{\rm Jac}(\frak{PO}_{\frak b};\frak y)$ which corresponds to an
idempotent of $\mathrm {Jac}(\frak{PO}_{\frak b};\Lambda)$.
Denote by $e_{\frak y}$ the idempotent of $(QH(M;\Lambda),\cup^{\frak b})$
corresponding to $1_{\frak y}$ under the isomorphism $\frak{ks}_{\frak b}$ in Theorem \ref{JacisHQ}.

We are finally ready to describe the map
\begin{equation}
i_{\text{qm},(\frak b,b)}^{T, \ast} :
QH_{\frak b}^*(M;\Lambda) \to HF^*((L(\text{\bf u}),(\frak b,b));\Lambda)
\end{equation}\index{$i_{\text{qm},(\frak b,b)}^{T, \ast}$}
in our situation.
\begin{thm}\label{calcithm}
Let $\frak b \in H^{{\rm even}}(M;\Lambda_0)$,
$\frak y$ a critical point of $\frak{PO}_{\frak b}$,
$a \in H(M;\Lambda)$ and let $\frak P \in \Lambda\langle\!\langle
y,y^{-1}\rangle\!\rangle^{\overset{\circ}P}$ satisfy
\begin{equation}\label{condforP}
\frak{ks}_{\frak b}(a) = [\frak P] \mod
\text{\rm Clos}_
{\overset{\circ}d}\left(
y_i\frac{\partial \frak{PO}_{\frak b}}{\partial y_i}:
i=1,\dots,n
\right).
\end{equation}
Then we have
\begin{equation}\label{calcisharp}
i_{\text{\rm qm},(\frak b(\frak y),b(\frak y))}^{T, \ast}(a)
= \frak P(\frak y)\text{\bf e}_{L(\text{\bf u}(\frak y))}.
\end{equation}
\end{thm}
\begin{proof}
This is \cite[Lemma 3.3.1]{fooo:toricmir}. Since
its proof was omitted in \cite{fooo:toricmir}, we provide its proof here.
\par
We note that the right hand side of (\ref{calcisharp}) is
independent of the choices of $\frak P$ satisfying (\ref{condforP})
 because $y_i\frac{\partial \frak{PO}_{\frak b}}{\partial y_i}$
is zero at $\frak y$.
\par
Let $\frak b = \frak b_0 + \frak b_2 + \frak b_+$ be as in (\ref{decompb}).
We express $\frak b = \sum w_i(\frak b) \text{\bf e}^M_i$ and
$b = \sum y_i(b)\text{\bf e}_i = b_0 + b_+$ where
$b_0 \in H^1(L(\text{\bf u});\C)$ and
$b_+ \in H^1(L(\text{\bf u});\Lambda_+)$.
By definition, we have
\begin{equation}\label{fomulaforPO1}
\frak{PO}_{\frak b}(b)
=
\frak b_0 +
\sum_{\beta,k,\ell}
T^{\omega\cap \beta}
\frac {\exp(\frak b_2 \cap\beta + b_0 \cap \del \beta)}{\ell!} \frak q^T_{k,\ell,\beta}(\frak b_+^{\otimes \ell},b_+^{\otimes k}),
\end{equation}
where we identify $H^0(L(\text{\bf u});\Lambda)= \Lambda$.
\par
We now consider the splitting
$\frak b = \frak b_0 + \widehat{\frak b}_2 + \frak b_{\text{\rm high}}$
such that
$$
\widehat{\frak b}_2 \in H^2(M;\Lambda_0), \qquad
\frak b_{\text{\rm high}} \in \bigoplus_{k>1} H^{2k}(M;\Lambda_0).
$$
Using a relative version of the divisor equation (\cite[Lemmas 7.1 and 9.2]{fooo:bulk}), we can rewrite (\ref{fomulaforPO1})
as
\begin{equation}\label{fomulaforPO2}
\frak{PO}_{\frak b}(b)
= \frak b_0 + \sum_{\beta,\ell} T^{\omega\cap \beta}
\frac{\exp(\widehat{\frak b}_2 \cap\beta + b \cap \partial\beta)}{\ell !}
\frak q^T_{0,\ell,\beta}(\frak b_{\text{\rm high}}^{\otimes \ell},1).
\end{equation}

We re-enumerate $D_1,D_2,\dots$ so that
$\{D_1,\dots,D_{B_1}\}$ becomes a $\Q$-basis of $H^2(M;\Q)$ and let
$w_1,\dots,w_{B_1}$ are the corresponding coordinates of $H^2(M;\Q)$.
Then we write
$\frak b_{\text{\rm high}} = \sum_{i=B_1+1}^B w_i\text{\bf e}_i^M$ for $\frak b_{\text{\rm high}}$
where we use the coordinate $w' = (w_{B_1+1},\dots,w_B)$ for $\frak b_{\text{\rm high}}$.
Then we define
a series $P_{\beta}$ of $w'$ by the equation
\begin{equation}\label{formulaforPO3}
P_{\beta}(w')
=
\sum_{\ell=0}^{\infty}
\frac{1}{\ell !}\frak q^T_{0,\ell,\beta}(\frak b_{\text{\rm high}}^{\otimes \ell},1).
\end{equation}
\begin{lem} The series $P_{\beta}$ is indeed a polynomial of $w_{B_1+1},\dots,w_B$, i.e.,
$$
P_{\beta}(w') \in \Lambda[w_{B_1+1},\dots,w_B].
$$
\end{lem}
\begin{proof}
Since each of the component of $\frak b_{\text{\rm high}}$ has degree $4$ or higher,
we can show  by a dimension counting argument that
$\frak q^T_{0,\ell,\beta}(\frak b_{\text{\rm high}}^{\otimes \ell},1)$
is nonzero only for a finite number of $\ell$.
Since each $\frak q^T_{0,\ell,\beta}(\frak b_{\text{\rm high}}^{\otimes \ell},1)$
is a polynomial of $w'$, $P_{\beta}(w')$ is also a polynomial as asserted.
\end{proof}

We put
$$
\frak w_i = e^{w_i} = \sum_{k=0}^{\infty} \frac{w_i^k}{k!}
$$
for $i = 1, \ldots, B_1$.
It follows from (\ref{fomulaforPO2}) and (\ref{formulaforPO3}) that we can write
\begin{equation}\label{fomulaforPO4}
\frak{PO}_{\frak b}
=
w_0 +
\sum_{\beta}
T^{\omega\cap \beta}
\frak w_1^{\beta\cap D_1}\cdots \frak w_{B_1}^{\beta\cap D_{B_1}}
y_1^{\partial \beta \cap e_1}\cdots y_n^{\partial \beta \cap e_n}
P_{\beta}(w').
\end{equation}
Here we regard $\beta$ as an element of $H_2(M;L(\text{\bf 0});\Z)$ to define the cap product $\omega \cap \beta$ in
(\ref{fomulaforPO4}).
\par
We will compare the value
$i_{\text{\rm qm},(\frak b(\frak y),b(\frak y))}^{T, \ast}(\text{\bf e}_i^M)$
with the $w_i$-derivative of (\ref{fomulaforPO4}).
By definition, we have
\begin{equation}\label{formulaforiei1}
\aligned
&i_{\text{\rm qm},(\frak b(\frak y),b(\frak y))}^{T, \ast}(\text{\bf e}_i^M)\\
&=
\sum_{\beta,k,\ell_1,\ell_2}
T^{\omega\cap \beta}
\frac{\exp(\frak b_2 \cap\beta + b_0 \cap \del \beta)}{(\ell_1 +\ell_2 +1) !} \frak q^T_{k,\ell_1+\ell_2+1,\beta}
(\frak b_+^{\otimes \ell_1}\otimes \text{\bf e}_i^M\otimes \frak b_+^{\otimes \ell_2},b_+^{\otimes k}).
\endaligned
\end{equation}

For a further analysis of the power series,
we consider the following three cases separately:
\par\smallskip
\noindent
(Case 1; $i=0$):
It is easy to see that
$$
\frak q_{k,\ell,\beta}^T(\frak b_+^{\otimes\ell_1}\otimes \text{\bf e}^M_0 \otimes
\frak b_+^{\otimes\ell_2};b_+^{\otimes k})
= 0
$$
unless $\beta =0$ and $k=\ell=0$.
Therefore we have
$
i_{\text{\rm qm},(\frak b(\frak y),b(\frak y))}^{T, \ast}(\text{\bf e}^M_0)
= \frak q_{0,0,0}^T(\text{\bf e}^M_0) = \text{\bf e}_L.
$
Since
$
\frac{\partial\, \frak{PO}_{\frak b(w)}}{\partial w_0}  = 1
$
by (\ref{fomulaforPO4}), we have (\ref{calcisharp}) for $a=\text{\bf e}_0^M$.
\par\smallskip
\noindent
(Case 2; $i>B_1$):
By \cite[Lemmas 7.1 and 9.2]{fooo:bulk}  we can rewrite (\ref{formulaforiei1})
to
\begin{equation}\label{formulaforiei2}
\aligned
& i_{\text{\rm qm},(\frak b(\frak y),b(\frak y))}^{T, \ast}(\text{\bf e}_i^M)\\
= &
\sum_{\ell_1,\ell_2,\beta}
\frac{\frak w_1^{\beta\cap D_1}\cdots \frak w_{B_1}^{\beta\cap D_{B_1}}
\frak y_1^{\partial \beta \cap e_1}\cdots \frak y_n^{\partial \beta \cap e_n}}
{(\ell_1 +\ell_2 +1) !} \frak q^T_{0,\ell_1+\ell_2+1,\beta}
(\frak b_{\text{\rm high}}^{\otimes \ell_1}\otimes \text{\bf e}_i^M\otimes \frak b_{\text{\rm high}}^{\otimes \ell_2},1).
\endaligned
\end{equation}
By differentiating \eqref{formulaforPO3} in $w_i$ and recalling
${\frak b}_{\text{\rm high}} = \sum_{i=B_1+1}^B w_i {\bf e}_i^M$, we compute
$$
\frac{\partial P_{\beta}}{\partial w_i}
=
\sum_{\ell_1,\ell_2}
\frac{1}{(\ell_1 +\ell_2 +1) !}
 \frak q^T_{0,\ell_1+\ell_2+1,\beta}
(\frak b_{\text{\rm high}}^{\otimes \ell_1}\otimes \text{\bf e}_i^M \otimes \frak b_{\text{\rm high}}^{\otimes \ell_2},1).
$$
Therefore comparing this with \eqref{formulaforiei2}, we conclude
\begin{equation}\label{iequalPconcl}
i_{\text{\rm qm},(\frak b(\frak y),b(\frak y))}^{T, \ast}(\text{\bf e}_i^M)
=\frac{\partial  \frak{PO}_{\frak b}}{\partial w_i}(\frak y),
\end{equation}
as required.
\par\smallskip
\noindent
(Case 3; $i=1,\dots,B_1$): The equality
(\ref{formulaforiei2}) also holds in this case.
Then, by \cite[Lemma 7.1 and 9.2]{fooo:bulk}   (``(relative) divisor equation''), we obtain
\begin{equation}\label{formulaforiei3}
\aligned
&i_{\text{\rm qm},(\frak b(\frak y),b(\frak y))}^{T, \ast}(\text{\bf e}_i^M)\\
&=
\sum_{\ell,\beta}
(\beta\cap D_i)
\frac{\frak w_1^{\beta\cap D_1}\cdots \frak w_{B_1}^{\beta\cap D_{B_1}}
\frak y_1^{\partial \beta \cap e_1}\cdots \frak y_n^{\partial \beta \cap e_n}}
{(\ell_1 +\ell_2 +1) !} \frak q^T_{0,\ell_1+\ell_2+1,\beta}
(\frak b_{\text{\rm high}}^{\otimes\ell},1).
\endaligned
\end{equation}
Substitution of
$$
\frac{\partial \frak w_i^{\beta\cap D_i}} {\partial w_i}
= (\beta\cap D_i) \frak w_i^{\beta\cap D_i}
$$
into the right hand side of \eqref{formulaforiei3} turns it into
the derivative $\frac{\partial  \frak{PO}_{\frak b}}{\partial w_i}(\frak y)$ and
hence gives rise to (\ref{iequalPconcl} also in this case.
The proof of Theorem \ref{calcithm} is now complete.
\end{proof}

\section{Spectral invariants and quasi-morphisms for toric manifolds}
\label{sec:etoricspectre}

\subsection{$\mu_e^{\frak b}$-heaviness of the Lagrangian fibers in
toric manifolds.}
\label{toricheavyfiber}

Let $(M,\omega)$ be a compact toric manifold, $P$ its moment polytope. Let
$\frak b \in H^{{\rm even}}(M;\Lambda_0)$. We consider the factorization
$$
QH_{\frak b}^*(M;\Lambda) \cong \prod_{\frak y \in \text{\rm Crit}(\frak{PO}_{\frak b})}
QH_{\frak b}(M;\frak y)
$$
corresponding to the one given
in Proposition $\ref{Morsesplit}$ via Theorem $\ref{JacisHQ}$ so that
$QH_{\frak b}(M;\frak {\frak y})$ is the factor corresponding to
$\text{\rm Jac}(\frak{PO}_{\frak b};\frak y)$.

\begin{thm}\label{toricheavymain}
Let $\frak y = (\frak y_1,\dots,\frak y_n) \in \text{\rm Crit}(\frak{PO}_{\frak b})$
and $e_{\frak y} \in QH_{\frak b}(M;\frak y)$ be the
corresponding idempotent.
We put
$$
\text{\bf u}(\frak y) = (\frak v_T(\frak y_1),\dots,\frak v_T(\frak y_n)) \in
\text{\rm Int}\,P.
$$
Then the following holds:
\begin{enumerate}
\item $L(\text{\bf u}(\frak y))$ is $\mu_{e_{\frak y}}^{\frak b}$-heavy.
\item If $\frak y$ is a nondegenerate critical point, then
$L(\text{\bf u}(\frak y))$ is $\mu_{e_{\frak y}}^{\frak b}$-superheavy.
\end{enumerate}
\end{thm}
\begin{proof}
Theorem \ref{toricheavymain} follows from
Theorems \ref{thm:heavy}, \ref{calcithm}, Proposition \ref{Morsesplit}
combined with the following lemma below.
\end{proof}

\begin{lem}\label{calksi}
Let $\frak{ks}_{\frak b}(e_{\frak y}) = 1_{\frak y} = [\frak P]$
for an element $\frak P \in \Lambda\langle\!\langle
y,y^{-1}\rangle\!\rangle^{\overset{\circ}P}$. Then $\frak P$ satisfies
$$
\frak P(\frak y) =1.
$$
\end{lem}
\begin{proof}
The ring homomorphism
$
{\frak P}\mapsto {\frak P}(\frak y) :
\Lambda\langle\!\langle
y,y^{-1}\rangle\!\rangle^{\overset{\circ}P}
\to \Lambda
$
induces a ring homomorphism
$
\text{\rm eval}_{\frak y} :
\text{\rm Jac}(\frak{PO}_{\frak b}; \Lambda)
\to \Lambda.
$
The ring homomorphism
$\text{\rm eval}_{\frak y}$ is unital and so is surjective.
\par
Let $\frak y' \in \text{\rm Crit}(\frak{PO}_{\frak b})$,
$\frak y' \ne \frak y$ and let
${\frak P} \in \Lambda\langle\!\langle y,y^{-1}\rangle\!\rangle^{\overset{\circ}P}$
be an element such that
$[{\frak P}] \in \text{\rm Jac}(\frak{PO}_{\frak b}; \frak y')$.
By definition
$
(\widehat y_i - \frak y'_i)^N [{\frak P}] = 0
$
in $\text{\rm Jac}(\frak{PO}_{\frak b}; \Lambda)$ for some large integer $N \in \N$.
Therefore applying $\text{\rm eval}_{\frak y}$ we obtain
$
(\frak y_i - \frak y'_i)^N \text{\rm eval}_{\frak y}([{\frak P}]) = 0.
$
Since $\frak y_i - \frak y'_i \ne 0$
for some $i$ by the hypothesis, we conclude
$\text{\rm eval}_{\frak y}([{\frak P}]) = 0$.
\par
Therefore by Proposition \ref{Morsesplit},
the homomorphism $\text{\rm eval}_{\frak y}$  is nonzero
on the factor $\text{\rm Jac}(\frak{PO}_{\frak b}; \frak y)$.
Since $1_{\frak y} = [\frak P]$ is the unit of this factor, we conclude
$\frak P(\frak y) =1$, as required.
\end{proof}

\subsection{Calculation of the leading order term of the potential function in the
toric case: review.}
\label{toricPOcalcu}\index{potential function!leading order term}

We put
\begin{equation}\label{defzj}
z_i = T^{\ell_i(\text{\bf 0})}y_1^{\partial \beta_i \cap \text{\bf e}_1}\cdots y_n^{\partial \beta_i \cap  \text{\bf e}_n}
\in \Lambda\langle\!\langle y,y^{-1} \rangle\!\rangle^{\overset{\circ}P}
\end{equation}
for $i =1, \ldots, m$.
We assume
\begin{equation}
\frak b - \sum_{i=1}^{B_1} \overline{\frak b}_i \text{\bf e}_i^M
\in
H^2(M;\Lambda_+) \oplus \bigoplus_{k\ne 1} H^{2k} (M;\Lambda_0),
\end{equation}
where $\overline{\frak b}_i\in \C$.

\begin{thm}\label{toricPOcalcthm}
We have
\begin{equation}
\frak{PO}_{\frak b}
=
\frak b_0 +
\sum_{i=1}^m e^{\overline{\frak b}_i}z_i
+
\sum_j T^{\lambda_j} P_j(z_1,\dots,z_m)
\end{equation}
where $P_j \in \Lambda[z_1,\dots,z_m]$, $\lambda_j \in \R_{>0}$,
$\lim_{j\to\infty}\lambda_j  = \infty$.
\par
In case $(M,\omega)$ is Fano and
$\frak b\in H^2(M;\Lambda_0)$, we have
\begin{equation}\label{FanoPO}
\frak{PO}_{\frak b}
=
\sum_{i=1}^m e^{w_i}z_i
\end{equation}
where $\frak b = \sum w_i\text{\bf e}^M_i$.
\end{thm}
\begin{proof}
Theorem \ref{toricPOcalcthm} is \cite[Theorem 3.5]{fooo:bulk}.
(See also \cite[Theorem 8.2]{fooo:survey}.)
We sketch  its proof below.
The proof uses the result of Cho-Oh \cite{cho-oh} which
was summarized in \cite[Theorem 11.1]{fooo:toric1}  as follows:
\begin{enumerate}
\item
If $\overset{\circ}{\mathcal M}_{1;0}(\beta) \ne \emptyset$,
$\mu_{L(\text{\bf u})}(\beta) = 2$ then
$\beta = \beta_j$ for $j=1,\dots,m$, where $\beta_j$ is as in (\ref{charbetaj}).
In this case $\overset{\circ}{\mathcal M}_{1;0}(\beta_j)
= {\mathcal M}_{1;0}(\beta_j) = T^n$ and the evaluation map
$\text{\rm ev}_0^{\partial} : {\mathcal M}_{1;0}(\beta_j) \to L(\text{\bf u})$
has degree $1$.
\item If
${\mathcal M}_{1;0}(\beta) \ne \emptyset$, $\beta \ne \beta_j$ ($j=1,\dots,m$)
then
$$
\beta = \sum_{j=1}^m k_j \beta_j + \alpha,
$$
where $\sum k_j > 0$, $k_j \ge 0$ and $\alpha \in \pi_2(M)$ with
$\alpha \cap \omega > 0$.
\end{enumerate}
Using this description of the moduli space ${\mathcal M}_{1,0}(\beta)$,
we calculate the summands of the right hand side of
(\ref{fomulaforPO4}) now.
\par
For $\beta = \beta_j$ we have
$$
\aligned
&T^{\omega\cap \beta_j}
\frak w_1^{\beta_j\cap D_1}\cdots \frak w_{B_1}^{\beta_j\cap D_{B_1}}
y_1^{\partial \beta_j \cap \text{\bf e}_1}\cdots y_n^{\partial \beta_j \cap  \text{\bf e}_n}
P_{\beta_j}(w')\\
& =
e^{w_j}z_j
= (e^{\overline{\frak b}_j} + (\text{higher order})) z_j.
\endaligned
$$
For $\beta \ne \beta_j$ ($j=1,\dots,m$) we have
$$\aligned
T^{\omega\cap \beta}
\frak w_1^{\beta\cap D_1}\cdots \frak w_{B_1}^{\beta\cap D_{B_1}}
y_1^{\partial \beta \cap \text{\bf e}_1}\cdots y_n^{\partial \beta \cap  \text{\bf e}_n}
P_{\beta}(w')
=
T^{\alpha \cap \omega} \prod_{j=1}^m (e^{k_jw_j} z_j^{k_j}).
\endaligned$$
Therefore combining the above discussion, we have proved Theorem \ref{toricPOcalcthm}.
\end{proof}

\subsection{Existence of Calabi quasi-morphism on
toric manifolds.}
\label{toricexistqh}

In this subsection we complete the proof of Corollary  \ref{existtoric}.
We begin with the following lemma

\begin{lem}\label{gennondeg}
The set of vectors $(c_1,\dots,c_m) \in (\C\setminus \{0\})^m$
with the following properties is dense in $(\C \setminus \{0\})^m$:
\par
The function $f$ defined by
\begin{equation}\label{defoff}
f(y_1,\dots,y_n) = \sum_{i=1}^m c_i y_1^{\partial \beta \cap \text{\bf e}_1}\cdots y_n^{\partial \beta \cap  \text{\bf e}_n}
\end{equation}
restricts to a Morse function on $(\C\setminus \{0\})^n$.
\end{lem}
This lemma is proved in \cite{kushnire} (see
\cite[Corollary 5.12]{iritani2}, \cite[Proposition 8.8]{fooo:toric1}  for the discussion in this context).
\begin{cor}\label{nondegcor1}
Write $\frak b = \sum_{i=1}^m \log c_i [D_i] \in H(M;\C)$ and consider the sum
$$
\frak{PO}_{\overline{\frak b},0}
= \sum_{i=1}^m c_i  z_i
\in \Lambda[y_1,\dots,y_n,y_1^{-1},\dots,y_n^{-1}].
$$
Then the set of $(c_1,\dots,c_m) \in (\C\setminus \{0\})^m$ for which
$\frak{PO}_{\overline{\frak b},0}$ becomes a Morse function is dense in
$(\C\setminus \{0\})^m$.
\end{cor}
\begin{proof}
Suppose that $\frak{PO}_{\overline{\frak b},0}$
is not a Morse function. Consider a degenerate critical point
$\frak y = (\frak y_1,\dots,\frak y_n)$ each of whose coordinates
is a `formal Laurent power series' of $T$.
(We put `formal Laurent power series' in the quote
since the exponents of $T$ are real numbers which are not necessarily integers.)
By \cite[Lemma 8.5]{fooo:toric1}, those series are convergent when we put
$T=\epsilon$ for sufficiently small $\epsilon >0$.
Then for $c'_i = c_i {\epsilon}^{\ell_i(\text{\bf 0})}$
the function (\ref{defoff}) will not be a Morse function.
Corollary \ref{nondegcor1} follows from this observation and Lemma \ref{gennondeg}.
\end{proof}
\begin{cor}\label{nondegcor2}
For any compact toric manifold $M$ there exists
an element $\frak b
\in H^{{\rm even}}(M;\Lambda)$ such that
$\frak{PO}_{\frak b}$ is a Morse function.
\end{cor}
\begin{proof}
By \cite[Theorem 10.4]{fooo:toric1}  we can prove that if
$\frak{PO}_{\overline{\frak b},0}$ is a Morse function
then $\frak{PO}_{\frak b}$ is also a Morse function.
(Actually the case $\overline{\frak b}=0$ is stated there.
However the general case can be proved in the same way.)
Therefore Corollary \ref{nondegcor2} follows from
Corollary \ref{nondegcor1}.
\end{proof}

Corollary  \ref{existtoric} follows immediately from
Corollary \ref{nondegcor2},
Proposition \ref{Morsesplit},
Theorem \ref{JacisHQ}
and Theorem \ref{thm:morphism}.
\qed

\subsection{Defect estimate of a quasi-morphism $\mu_e^{\frak b}$.}
\label{subsec:estimateC}

Using the calculations we have performed, we can obtain some explicit estimates
of the norm of the defect $\operatorname{Def}\mu_e$\index{defect}
of spectral quasi-morphism $\mu_e$. We define
$$
\vert\operatorname{Def}\vert(\mu_e) = \sup_{\widetilde\psi, \widetilde\phi}|\mu_e(\widetilde \psi\widetilde \phi)
- \mu_e(\widetilde \psi) - \mu_e(\widetilde \phi)|.
$$
We illustrate this estimate by an example.
\par
We consider $(M,\omega) = \C P^n$ with moment polytope
$$
\{(u_1,\dots,u_n) \mid u_i \ge 0, \sum u_i \le 1\}.
$$
Set $\frak b = \text{\bf 0}$. It is well known that the small quantum cohomology
$QH(\C P^n;\Lambda)$ is isomorphic to
$\Lambda[x]/(x^{n+1} - T)$, where $x \in H^2(\C P^n;\C)$ is the
standard generator.
This is isomorphic to the direct product of $n+1$ copies of $\Lambda$.
Therefore we have $n+1$ quasi-morphisms $\mu_{e_k}$.
($k=0,\dots,n$.)
It is actually defined on the Hamiltonian diffeomorphism group $\text{\rm Ham}(\C P^n,\omega)$
itself. (\cite[Section 4.3]{EP:morphism}.)
(It is unknown whether they are different from one another.)
\begin{prop}\label{estimateC} Let $e_k$ and $\mu_{e_k}$ for $k = 0, \ldots, n$ be as above.
Then
$$
\vert\operatorname{Def}\vert(\mu_{e_k}) \le \frac{12n}{n+1}.
$$
\end{prop}
\begin{proof}
We have
$$
\frak{PO}_{\text{\bf 0}} = y_1 + \dots + y_n + T(y_1\dots y_n)^{-1}.
$$
See for example \cite[Example 5.2]{fooo:toric1}.
\par
Let $\chi_k = \exp(2\pi k\sqrt{-1}/(n+1))$.
The critical points of $\frak{PO}_{\text{\bf 0}}$ are
$$
\frak y_k = T^{1/(n+1)}(\chi_k,\dots,\chi_k),
\quad k=0,1,\dots,n.
$$
We put
$$
P_k = \frac{\prod_{i\ne k} (y_1 - T^{1/(n+1)}\chi_i)}
{\prod_{i\ne k} (T^{1/(n+1)}\chi_k - T^{1/(n+1)}\chi_i)}.
$$
Since
$$
P_k(\frak y_{\ell})
=
\begin{cases}
0 &k\ne \ell, \\
1 &k = \ell,
\end{cases}
$$
it follows that $[P_k] = \frak{ks}_{\text{\bf 0}}(e_{\frak y_k})$ in the Jacobian ring.
Therefore, using  $\frak{ks}_{\text{\bf 0}}(x) = [y_1]$ also,
we have
$$
\frak v_q(e_{\frak y_k}) = -\frak v_T(e_{\frak y_k}) = \frac{n}{n+1}.
$$
Proposition \ref{estimateC} now follows from Remark \ref{rem168} (1).
\end{proof}
Note we chose our symplectic form $\omega$ so that
$\int_{\C P^1}\omega = 1$. (See  (\ref{elljandarea})
and Remark \ref{rem2pi2}.)

\section{Lagrangian tori in $k$-points blow up of
$\C P^2$ ($k\ge 2$).}
\label{sec:CP2blow}

In this section, we prove Theorem \ref{uncount} (3) in the case of
$k$-points blow up of $\C P^2$ $(k\ge 2)$.
We use the example of \cite[Section 5]{fooo:bulk}, which we review now.
\par
We first consider $2$-points blow up $M$ of $\C P^2$.
We put a toric
K\"ahler form on it $\omega_{\alpha,\beta}$ such that
the moment polytope is given by
\begin{equation}\label{momentpoly}
P_{\alpha,\beta} = \{ (u_1,u_2) \mid 0\le u_1\le 1, ~0 \le u_2 \le 1-\alpha,~
\beta \le u_1+u_2 \le 1\}.
\end{equation}
Here
\begin{equation}\label{Kcone}
(\alpha,\beta) \in  \{ (\alpha,\beta) \mid 0 \le \alpha,\beta,\,\,
\alpha + \beta \le 1\}.
\end{equation}
We are interested in the case $\beta =(1-\alpha)/2$ and
write
$M_{\alpha} = (M,\omega_{\alpha,(1-\alpha)/2})$ where $\alpha > 1/3$.
We denote
$$
D_1 = \pi^{-1}(\partial P_{\alpha,(1-\alpha)/2} \cap \{(u_1,u_2) \mid u_2 = 0\})
$$
and put
\begin{equation}\label{defbulk}
\frak b_{\kappa} = T^{\kappa} PD([D_1]) \in H^2(M_{\alpha};\Lambda_+), \quad \kappa >0.
\end{equation}
Then by (\ref{FanoPO}), we have obtained
\begin{equation}\label{POwithoutbulk}
\aligned
\mathfrak{PO}_{\frak b_{\kappa}} =
y_1 + e^{T^{\kappa}} y_2 + T^{1-\alpha}y_2^{-1}
+ Ty_1^{-1}y_2^{-1}
+ T^{-(1-\alpha)/2}y_1y_2.
\endaligned
\end{equation}
Now consider a family of Lagrangian torus fibers
\begin{equation}\label{LuinCP2}
L(u) = L(u,(1-\alpha)/2),
\end{equation}
for
$
(1-\alpha)/2 < u < (1+\alpha)/4.
$
Note that $\alpha >1/3$ implies $
(1-\alpha)/2 < 1/3 < (1+\alpha)/4.
$
Then for any such $u$ we can show the following.
\begin{thm}\label{CP2heavy}
If $1/3 \le u < (1+\alpha)/4$, we take
$\kappa(u) = (1+\alpha)/2 - 2u>0$.
If $(1-\alpha)/2 < u <1/3$, we take
$\kappa(u) = u- (1-\alpha)/2 >0$.
Then
$L(u) \subset M_{\alpha}$ is $\mu_{e}^{\frak b_{\kappa(u)}}$-superheavy
with respect to an appropriate idempotent $e$ of
$QH_{\frak b_{\kappa(u)}}^*(M_{\alpha};\Lambda)$.
\end{thm}
\begin{proof}
Let $\text{\bf u} = (u,(1-\alpha)/2)$. We put
$$
y^{\text{\bf u}}_1 = T^{-u_1}y_1 = T^{-u}y_1,\,~  y^{\text{\bf u}}_2 = T^{-u_2}y_2
= T^{-(1-\alpha)/2}y_2
$$
in (\ref{POwithoutbulk}) to obtain
\begin{equation}\label{POwithoutbulk1}
\aligned
\mathfrak{PO}_{\frak b_{\kappa(u)}}
&=&
T^uy^{\text{\bf u}}_1 + e^{T^{\kappa(u)}}T^{(1-\alpha)/2} y^{\text{\bf u}}_2
+ T^{(1-\alpha)/2}(y^{\text{\bf u}}_2)^{-1} \\
&{}& \quad+ T^{(1+\alpha)/2-u}(y^{\text{\bf u}}_1)^{-1}(y^{\text{\bf u}}_2)^{-1}
+ T^{u}y^{\text{\bf u}}_1y^{\text{\bf u}}_2.
\endaligned
\end{equation}
See \cite[(5.10)]{fooo:bulk}.
We first consider the case that $1/3 < u < (1+\alpha)/4$ and $\kappa(u) = (1+\alpha)/2 - 2u$.
Then the calculation in Case 1 of \cite[Section 5]{fooo:bulk}  shows that
the potential function
$\mathfrak{PO}_{\frak b_{\kappa(u)}}$ has
nondegenerate critical points $\frak y(u)=(\frak y_1(u), \frak y_2(u))$ such that
$$
(T^{-u}\frak y_1(u),T^{-(1-\alpha)/2}\frak y_2(u)) \equiv (\pm\sqrt{-2},-1) \mod \Lambda_+.
$$
Each of them corresponds to an idempotent $e_{\frak y(u)}$ of $QH_{\frak b_{\kappa(u)}}(M_{\alpha};\Lambda)$.
Theorem \ref{toricheavymain} implies that $L(u)$ is $\mu_{e_{\frak y(u)}}^{\frak b_{\kappa(u)}}$-superheavy.
When $u=1/3$ and $\kappa(u) = (1+\alpha)/2 - 2u$,
Case 4 of \cite[Section 5]{fooo:bulk}  shows that
there are nondegenerate critical points. (Note that we are using $\frak b_{\kappa}$ as
\eqref{defbulk} so $w=1$ in \cite[(5.14)]{fooo:bulk}.)
If $(1-\alpha)/2 < u <1/3$ and
$\kappa(u) = u- (1-\alpha)/2$, Case 3 of \cite[Section 5]{fooo:bulk}  shows that
there is a nondegenerate critical point as well.
Thus Theorem \ref{CP2heavy} follows from Theorem \ref{toricheavymain}.
\end{proof}

\begin{proof}[Proof of Theorem \ref{uncount} (3)]
Since $\frak y(u)$ is a nondegenerate critical point, Theorem \ref{JacisHQ} and Proposition
\ref{Morsesplit} imply that
$e = e_{\frak y(u)}$ is the unit of the direct factor of $QH_{\frak b_{\kappa(u)}}(M_{\alpha};\Lambda)$
that is isomorphic to $\Lambda$.
Therefore by
Theorem \ref{thm:morphism}
$\mu_{e_{\frak y(u)}}^{\frak b_{\kappa(u)}}$ is a Calabi quasi-morphism.
By Corollary \ref{lieind}, the set
$
\left\{\mu_{e_{\frak y(u)}}^{\frak b_{\kappa(u)}}\right\}_{u\in ((1-\alpha)/2,(1+\alpha)/4)}
$
is linearly independent. Thus we have  constructed a continuum of linearly independent
Calabi quasi-morphisms parametrized by $u \in ((1-\alpha)/2,(1+\alpha)/4)$.
The proof of Theorem \ref{uncount} for the case of two-points blow-up
of $\C P^2$ is complete.
\par
To prove the existence of a continuum of linearly independent Calabi
quasi-morphisms\index{quasi-morphism!Calabi
quasi-morphism} for the case of three-points blow-up of $\C P^2$, we consider the
K\"ahler toric surface $(M,\omega)$ whose moment polytope is given by
$$
P_{\alpha,(1-\alpha)/2} \setminus \{(u_1,u_2) \mid 1 - \epsilon < u_2\}
$$
for a sufficiently small $\epsilon$. Then
$(M,\omega)$ is a three-points blow-up of $\C P^2$.
Its potential function is given by
$$
(\text{\rm \ref{POwithoutbulk1}}) + T^{1-\epsilon}y_1^{-1}.
$$
It is easy to see that the extra term $T^{1-\epsilon}y_1^{-1}$ is of higher order,
when $({\mathfrak v}_T(y_1), {\mathfrak v}_T(y_2)) \linebreak = (u,(1-{\alpha})/2)$, $u\in (1/3,(1+{\alpha})/4)$.
So by the same argument as in the case of two-points blow-up,
we can prove Theorem \ref{uncount}.
For the general $k$-points blow-up with $k \geq 3$, we can repeat the same argument.
(See \cite[page 111]{fooo:toric1}  for a relevant study of $k$-points blow-up of $\C P^2$.)
\end{proof}

\section{Lagrangian tori in $S^2 \times S^2$}
\label{sec:exotic}

In this section we prove Theorem \ref{uncount} for the case of $S^2\times S^2$,
which is equipped with the symplectic structure $\omega \oplus \omega$.
We also prove Theorem \ref{dips2s2}.
We first recall the description of the family of Lagrangian tori
constructed in \cite{fooo:S2S2}.

\subsection{Review of the construction from \cite{fooo:S2S2}}
\label{subsec:review}

We consider the toric Hirzebruch surface $F_2(\alpha)$ ($\alpha >0$)\index{Hirzebruch surface}  whose
moment polytope is
\begin{equation}\label{POord}
P(\alpha) =
\{
(u_1,u_2) \in \R^2 \mid u_i \ge 0,
u_2 \le 1-\alpha, u_1+2u_2 \le 2
\}.
\end{equation}
Note that $F_2(\alpha)$ is not Fano but nef, i.e. every holomorphic sphere has non-negative
Chern number.
In fact, the divisor $D_1 \cong \C P^1$ associated to the facet $\partial_1P(\alpha) = \{ u \in \partial P(\alpha) \mid u_2 = 1-\alpha\}$
has
$c_1(D_1) = 0$.

\begin{thm}{\rm(\cite[Theorem {3.1}]{fooo:S2S2})}\label{F2PO}
We put $\frak b =\text{\bf 0}$.
The potential function $\frak{PO}_{\text{\bf 0}}$ of $F_2(\alpha)$ has the form
\begin{equation}\label{POord2}
\frak{PO}_{\text{\bf 0}} =
y_1+y_2 + T^2y_1^{-1}y_2^{-2} + T^{1-\alpha}(1+T^{2\alpha}) y_2^{-1}.
\end{equation}
\end{thm}

We consider the limit $\alpha \to 0$ of the Hirzebruch surface $F_2(\alpha)$.
At $\alpha =0$, the limit polytope is the triangle
\begin{equation}\label{P0}
P(0) =
\{ (u_1,u_2) \in \R^2 \mid u_i \ge 0,
u_2 \le 1, u_1+2u_2 \le 2
\}
\end{equation}
and the limit $F_2(0)$ is an orbifold with a singularity of the form
$\C^2/\{\pm 1\}$. We cut out a neighborhood of
the singularity of $F_2(0)$ and paste the Milnor fiber back into
the neighborhood to obtain the desired manifold. We denote it by $\widehat F_2(0)$.

Consider the preimage $Y(\varepsilon)$ of $P(\varepsilon) \subset P(0)$, $0<\varepsilon <1$,
under the moment map
$\pi : F_2(0) \setminus \{ O\} \to P(0) \setminus \{(0,1)\}$, where $O$ is the singularity of $F_2(0)$.
We can put a natural glued symplectic form on $\widehat F_2(0) =
Y(\varepsilon) \cup D_r(T^* S^2)$ in a way that the given toric
symplectic form on $Y(\varepsilon)$ is unchanged on
$
Y(\varepsilon) \setminus N(\varepsilon) \subset Y(\varepsilon) \setminus \partial Y(\varepsilon),
$
where $N(\varepsilon)$ is a collar neighborhood of $\partial Y(\varepsilon)$.
Since $H^2(S^3/\{\pm 1\};\Q) = 0$, the glued symplectic form does not depend on the choices
of $\varepsilon> 0$ or the gluing data up to the symplectic diffeomorphism.
This symplectic manifold is symplectomorphic to
$
(S^2, \omega_{\mathrm{std}}) \times (S^2, \omega_{\mathrm{std}})$ (\cite[Proposition 5.1]{fooo:S2S2}.)
In other words, we have symplectomorphisms
\be\label{symplecticisoS2S2}
\phi_\e: (\widehat F_2(0),\omega_\e) \to (S^2\times S^2, \omega_{\mathrm{std}}
\oplus \omega_{\mathrm{std}}).
\ee
We denote
\begin{equation}
T(\rho) = \phi_\e(L(1/2-\rho,1/2+\rho)), \quad 0 \le \rho < \frac{1}{2}
\end{equation}
where $L(1/2-\rho,1/2+\rho) = \pi^{-1}(1/2-\rho,1/2+\rho)$ regarded as a Lagrangian
submanifold of $(\widehat F_2(0), \omega_\e)$.
We refer to   \cite[Sections 3 and 4]{fooo:S2S2} for the detailed explanation
of this construction.

\subsection{Superheaviness of $T(\rho)$.}
\label{subsec:POb-T(u)}
Recall from  \cite[Section 5]{fooo:S2S2} that
we have a family $\bigcup_{a \in \C} X_a$ where
$X_a$ is biholomorphic to $\C P^1 \times \C P^1$
for $a\ne 0$ and $X_0$ is biholomorphic to $F_2$.
(See  \cite[Lemma 5.1]{fooo:S2S2}.)
The smooth trivialization of the simultaneous resolution $\bigcup_{a \in \C} X_a$
of $F_2(0)$ constructed in  \cite[Section 6]{fooo:S2S2}
identifies the homology class $[D_1]$ in $X_0$ and $[S^2_{{\mathrm{van}}}]$ in $X_a$. Beside this, the relative homology class $\beta_1$ in $X_0$ which satisfies $\beta_1 \cap [D_1] =1$ and
does not intersect with other toric divisors
can be also regarded as a homology class in $X_a$.
The homology classes $\beta_1$ and $\beta_1+[S^2_{\mathrm{van}}]$ satisfy the relations
\begin{equation}\label{intnumber}
\beta_1 \cap [S^2_{\mathrm{van}}] = 1, \quad (\beta_1+[S^2_{\mathrm{van}}])\cap [S^2_{\mathrm{van}}] = -1.
\end{equation}
We consider the cohomology class
$$
\frak b(\rho) = T^{\rho}PD[S^2_{\mathrm{van}}] \in H^2(\widehat F_2(0),\Lambda_+).
$$

Using the 4-dimensionality and special properties of $\widehat F_2(0)$,
we proved the following in \cite[Lemma 4.2]{fooo:S2S2}.

\begin{lem}
$$
H^1(T(u);\Lambda_0)
\subset
\{ b \in H^{odd}(T(u);\Lambda_0)
\mid (\frak b(\rho),b) \in
\widehat{\CM}_{\rm def,weak}(T(u);\Lambda_0)\}.
$$
\end{lem}
In \cite{fooo:S2S2}, we showed that the potential function for $T(0)$, i.e., ${\mathbf u}_0=(1/2,1/2)$ is
$$\mathfrak{PO}=T^{1/2}(y^{\bf u_0}_1 + y^{\bf u_0}_2 + (y^{\bf u_0}_1)^{-1}(y^{\bf u_0}_2)^{-2}
+ 2 (y^{\bf u_0}_2)^{-1}).$$
We find that there are two critical points $\pm(1/2,2)$, see \cite[Digression 4.1]{fooo:S2S2}.
Hence there exist two $b \in H^1(T(0);\Lambda_0)$ modulo $H^1(T(0);2\pi \sqrt{-1}\Z)$
such that $HF((T(0),b);\Lambda) \neq 0$.

When we consider the bulk deformation by $\frak b(\rho)$,
 (\ref{intnumber}) and
Theorem \ref{F2PO} imply that the potential function of $T({\bf u})$ with bulk,
$\frak{PO}_{\frak b(\rho)}$, becomes
\begin{equation}\label{PObulk}
\aligned
\frak{PO}_{\frak b(\rho)}
= T^{u_1}y^{\text{\bf u}}_1 + T^{u_2}y^{\text{\bf u}}_2
&+T^{2-u_1-2u_2}(y^{\text{\bf u}}_1)^{-1}(y^{\text{\bf u}}_2)^{-2}
\\ &+
(e^{T^{\rho}}
+ e^{-T^{\rho}})T^{1 -u_2}(y^{\text{\bf u}}_2)^{-1}.
\endaligned
\end{equation}
(See \cite[Theorem 3.5]{fooo:bulk}  and \cite[Formula (47)]{fooo:S2S2}.)
Now we put
\begin{equation}\label{rhoichi}
2\rho = u_2 - u_1 = u_2 - (1-u_2) = 2u_2 - 1
\nonumber
\end{equation}
and consider (\ref{PObulk}) at $\text{\bf u}= (u_1, u_2)$ for some $\rho$.
Namely, $\text{\bf u} = (1/2-\rho,1/2+\rho)$.
Then the potential function with bulk ${\mathfrak b}(\rho)$ of  $T(0)$ is written as
$$
T^{1/2} (y^{\bf u_0}_1 +y^{\bf u_0}_2 + (y^{\bf u_0}_1)^{-1} (y^{\bf u_0}_2)^{-2} + (e^{T^{\rho}} + e^{-T^{\rho}}) (y^{\bf u_0}_2)^{-1}).
$$
See   \cite[Formula (47)]{fooo:S2S2} with $u_1 = u_2 = 1/2$.
There are two critical points, which are
$ (\frak y^0_1(\rho),\frak y^0_2(\rho))=(\epsilon (e^{T^{\rho/2}} + e^{-T^{\rho/2}})^{-1}, \epsilon (e^{T^{\rho/2}} + e^{-T^{\rho/2}}))$
with $\epsilon = \pm 1$.
Hence $b^0(\rho)=b(\frak y^0(\rho)) = (\log \frak y^0_1(\rho),\log \frak y^0_2(\rho))
\in H^1(T(0);\Lambda_0)$,
$$
HF((T(0),(\frak b(\rho),b^0(\rho)));\Lambda) \ne 0.
$$

For $T(\rho)$, the potential function with bulk ${\mathfrak b}_{\rho}$ is written as
$$
T^{1/2 - \rho} (y^{\bf u}_1 + T^{2\rho} y^{\bf u}_2 + (y^{\bf u}_1)^{-1} (y^{\bf u}_2)^{-2} + (e^{T^{\rho}}
+e^{-T^{\rho}}) (y^{\bf u}_2)^{-1}).
$$
See \cite[Formula (47)]{fooo:S2S2} with $u_1= 1/2 -\rho, u_2 = 1/2 + \rho$.
There are two critical points, which are
$(\frak y_1(\rho), \frak y_2(\rho))=(\epsilon T^{\rho} (e^{T^{\rho/2}} - e^{-T^{\rho/2}})^{-1}, - \epsilon T^{-\rho} (e^{T^{\rho/2}} - e^{-T^{\rho/2}}))$.
It follows that for
$b(\rho)=b(\frak y(\rho)) = (\log \frak y_1(\rho), \log \frak y_2(\rho))
\in H^1(T(\rho);\Lambda_0)$,
$$
HF((T(\rho),(\frak b(\rho),b(\rho)));\Lambda) \ne 0.
$$

In summary, we have:

\begin{lem} \label{lem:critexistence}
(1)  There exist two $b \in H^1(T(0);\Lambda_0)/H^1(T(0);2\pi \sqrt{-1}\Z)$
such that
$$HF((T(0),b);\Lambda) \neq 0.$$

\noindent
(2)  There exist two $b^0(\rho) \in H^1(T(0);\Lambda_0)/H^1(T(0);2\pi \sqrt{-1}\Z)$ such that
$$
HF((T(0),(\frak b(\rho),b^0(\rho));\Lambda) \ne 0.
$$

\noindent
(3)  There exist two $b(\rho) \in H^1(T(\rho);\Lambda_0)/H^1(T(\rho);2\pi\sqrt{-1}\Z)$ such that
$$
HF((T(\rho),(\frak b(\rho),b(\rho));\Lambda) \ne 0.
$$
\end{lem}

Using this lemma, we obtain the following results.

\begin{thm}\label{thm:T(u)}
(1) \ \ There exists an idempotent $e$ of a field factor of $QH(S^2 \times S^2;\Lambda)$ such that
$T(0)$ is $\mu_e$-superheavy.

\noindent
(2) \ \ For any $0 \le \rho < \frac{1}{2}$, there exist idempotents $e_{\rho}$ and $e^0_{\rho}$, each of which is
an idempotent of a field factor of $QH_{\frak b(\rho)}(S^2 \times S^2;\Lambda)$ such that
$T(\rho)$ is $\mu_{e_{\rho}}^{\frak b(\rho)}$-superheavy and $T(0)$ is $\mu_{e^0_{\rho}}^{\frak b(\rho)}$-superheavy.
\end{thm}

\begin{proof}
We first show that $QH(S^2 \times S^2;\Lambda)$ and $QH_{\frak b(\rho)}(S^2 \times S^2;\Lambda)$ are semi-simple.
For this purpose, we consider the toric structure as the monotone product of $S^2$.
Let ${\mathfrak b}=a [S^2_{\rm van}]$, $a \in \Lambda_+$.  ($\mathfrak b=\mathfrak b(\rho)$ if $a=T^{\rho}$,
while ${\mathfrak b}=0$ if $a=0$.)
We pick points $pt_1$ (resp. $pt_2$) on the first (resp. the second)
factor of $S^2 \times S^2$ in the hemisphere in the classes $\beta_1$, $\beta_2$ that contributes to
the coefficients of $y_1$, $y_2$ respectively in the potential function.
We represent the homology class $[S^2_{\rm van}]$ by the anti-diagonal whose
homology class is given by $[S^2 \times pt_2] - [pt_1 \times S^2]$ in $H_2(S^2 \times S^2)$.
The potential function  of $S^1_{\rm eq} \times S^1_{\rm eq}$ with bulk $\mathfrak b$ is written as
$$
\frak{PO}_{\mathfrak b}=T^{1/2}(e^a y_1 + y_1^{-1} + e^{-a} y_2 + y_2^{-1}).
$$
It has four nondegenerate critical points $(\epsilon_1 e^{-a/2}, \epsilon_2 e^{a/2})$
with $\epsilon_1, \epsilon_2 = \pm 1$.
The critical values are $2(\epsilon_1  e^{a/2}+ \epsilon_2 e^{a/2}) T^{1/2}$.
By   \cite[Theorem 6.1]{fooo:toric1} (Fano toric case) and \cite[ Theorem 1.1]{fooo:toricmir}, we find that the quantum cohomology with bulk deformation by $\mathfrak b$ is
factorized into four copies of $\Lambda$:
$$
QH_{\mathfrak b}(S^2 \times S^2;\Lambda) \cong \bigoplus_{i=1}^4 \Lambda {\mathbf e}^{\mathfrak b}_i.
$$
Here ${\mathbf e}^{\mathfrak b}_1, \dots, {\mathbf e}^{\mathfrak b}_4$ are the idempotents corresponding to
the critical points of $\mathfrak{PO}_{\mathfrak b}$ with $(\epsilon_1,\epsilon_2)=(1,1),(1,-1),(-1,1), (-1,-1)$,
respectively.  (When $\mathfrak b =0$, we simply write them as ${\mathbf e}_i$.)
In particular, $QH_{\mathfrak b}(S^2 \times S^2;\Lambda)$ is semi-simple.

By Lemma \ref{lem:critexistence} (1) and (3), there exists $b$ (resp. $b(\rho)$)
such that $HF((T(0),b);\Lambda) \neq 0$ (resp.
$HF(T(\rho), (\frak b(\rho),b(\rho));\Lambda) \neq 0$).
Hence \cite[Theorem 3.8.62]{fooo:book1}, with  \cite[(3.8.36.2)]{fooo:book1} taken into account,
implies that the maps
$$
i_{{\rm qm},T(0),b}^*:QH(S^2 \times S^2;\Lambda) \to HF((T(0),b);\Lambda),
$$
$$
i_{{\rm qm},T(\rho), (\frak b(\rho), b(\rho))}^*:QH_{\frak b(\rho)}(S^2 \times S^2;\Lambda) \to HF((T(\rho) (\frak b(\rho),b(\rho));\Lambda)
$$
are nontrivial.
In particular, there is at least one idempotent $e_0 \in QH(S^2 \times S^2;\Lambda)$
(resp. $e_{\rho} \in QH_{\frak b(\rho)}(S^2 \times S^2;\Lambda)$) such that
$i_{{\rm qm},T(0),b_i}^*(e_0) \neq 0$ (resp.
$i_{{\rm qm},T(\rho), (\frak b(\rho), b(\rho))}^*(e_{\rho}) \neq 0$).  Hence $T(0)$ is $\mu_{e}$-superheavy
and $T(\rho)$ is $\mu_{e_\rho}^{\frak b(\rho)}$-superheavy.
\end{proof}

The sphere $S^2_{\rm van}$ is a Lagrangian submanifold, which is disjoint from $T(\rho)$.
We have the following

\begin{lem}\label{antidiagonal:unobstruct}
The Lagrangian sphere $S^2_{\rm van}$, which is the anti-diagonal in $S^2 \times S^2$,
is unobstructed and
$$HF(S^2_{\rm van};\Lambda) \cong H(S^2_{\rm van};\Lambda) \neq 0.$$
\end{lem}
\begin{proof}
Note that the anti-diagonal in $S^2 \times S^2$ can be seen as a fixed point set of an anti-symplectic involution.
Then Theorem 1.3 with $k=0, \ell=0$ in
\cite{fooo:invol} implies that ${\mathfrak m}_0(1)=0$, since
the Maslov index of any holomorphic disc in $(S^2 \times S^2, S^2_{\rm van})$ is divisible by $4$.
See also \cite[Corollary 1.6]{fooo:invol}.
The second assertion follows from
\cite[Theorem D (D.3)]{fooo:book1}.
\end{proof}
Since $H^1(S^2_{\text{\rm van}};\Lambda) = 0$, there is at most one bounding cochain up to gauge equivalence. Lemma
\ref{antidiagonal:unobstruct} implies existence of bounding cochain that
 satisfies $\frak m^b(1)=0$.(Recall that in general a  weak bounding cochain satisfies $\frak m^b(1) = \frak{ PO}(b) {\bf e}_L$.)
and so the value of the potential function at $b$ is zero.
By the same argument as in the case of $T(0)$, we find an idempotent $e'$ of a field factor of
$QH(S^2 \times S^2;\Lambda)$
and such that $i_{\rm qm, S^2_{\rm van}}^{\ast}(e') \neq 0$.

Since each of  $e_0$, $e'$ is an idempotent of a field factor of $QH(S^2 \times S^2;\Lambda)$
and $e_{\rho}$ is an idempotent of a field factor of
$QH_{\frak b(\rho)}(S^2 \times S^2;\Lambda)$, there exist corresponding  Calabi quasi-morphisms
$\mu_{e_0}^{\bf 0}$, $\mu_{e'}$, $\mu_{e_{\rho}}^{\frak b(\rho)}$ from $\widetilde{\text{\rm Ham}}(S^2 \times S^2)$ to $\R$.
Since $T(\rho)$, $\rho \in [0,1/2)$ and $S^2_{\rm van}$ are mutually disjoint, Corollary \ref{lieind} implies
Theorem \ref{uncount} (1). This completes the proof of
Theorem \ref{uncount} (1). \qed
\par
\smallskip
Furthermore,
since homogenous quasi-morphisms are homomorphisms on abelian subgroups and $\pi_1(\text{\rm Ham}(S^2 \times S^2))
\cong \Z/2\Z \times \Z/2\Z$ (see \cite{gromov}), they descend to quasi-morphisms on $\text{\rm Ham}(S^2 \times S^2)$.
We denote them by
$\overline{\mu}_{e_0}^{\bf 0}$, $\overline{\mu}_{e'}$ and
$\overline{\mu}_{e_{\rho}}^{\frak b(\rho)}$.
Thus we also obtain the following.

\begin{cor}\label{S2S2qm}\index{quasi-morphism!Calabi quasi-morphism}
The Calabi quasi-morphisms $\overline{\mu}_{e_0}^{\bf 0}$,
$\overline{\mu}_{e'}$ and $\overline{\mu}_{e_{\rho}}^{\frak b(\rho)}$ on $\text{\rm Ham}(S^2 \times S^2)$
are linearly independent from one another.
\end{cor}

Generally,
let $(M, \omega)$ be a closed symplectic manifold.
Suppose that $\widetilde{\text{\rm Ham}}(M,\omega)$ has infinitely many linearly independent homogeneous
quasimoprhisms $\mu_i$. Now, under this hypothesis,
we state a sufficient condition for the existence of infinitely many linearly independent homogeneous
quasi-morphisms on $\text{\rm Ham}(M,\omega)$.

\begin{prop}
Suppose that $\pi_1(\text{\rm Ham}(M,\omega))$ is finitely generated.
If there are infinitely many linearly independent homogeneous quasi-morphisms $\mu_i$ on
$\widetilde{\text{\rm Ham}}(M,\omega)$,
then there are infinitely many linearly independent homogeneous quasi-morphisms on
$\text{\rm Ham}(M,\omega)$.
The same statement holds for Calabi quasi-morphisms.
\end{prop}

\begin{proof} Let $\mu_i$ for $i \in \N$ be the given infinite family of
linearly independent homogeneous quasi-morphisms.
Pick generators $g_1, \dots, g_A \in \pi_1(\text{\rm Ham}(M,\omega))$.
Denote by $K$ the maximal integer such that,
for some $i_1, \dots, i_K$, the vectors
$$(\mu_{i_j}(g_1), \dots, \mu_{i_j}(g_A)) \in \R^A, \quad  j=1, \dots, K$$
are linearly independent.
We re-arrange the ordering of $g_1, \dots, g_A$ so that $i_1=1, \dots, i_K=K$.
For $k > K$, we can find $a_\ell(k) \in \R$ such that the quasi-morphism
$$
\mu'_k : = \mu_k  - \sum_{\ell=1}^{K} a_\ell(k) \mu_\ell
$$
vanish at $g_i$ for all $i=1, \dots, A$.
Since the restriction of a homogeneous quasi-morphism on an abelian subgroup is a homomorphism,
$\mu'_k$ vanishes on $\pi_1({\text{\rm \text{\rm Ham}}}(M,\omega))$, which we regard as a subgroup of
$\widetilde{\text{\rm Ham}}(M,\omega)$. (Recall that $\pi_1({\text{\rm Ham}}(M,\omega))$
is abelian since $\text{\rm Ham}(M,\omega))$ is a topological group.)
Therefore all $\mu'_k$, for $k>K$, descend to homogeneous quasi-morphisms on $\text{\rm Ham}(M,\omega)$.
Linear independence of $\mu'_k$, $k>K$, follows from the standing hypothesis of linear indepedence of the set
$\{\mu_i \mid i \in \N \}$.
\par
For the statement concerning Calabi quasi-morphisms, we take one more
$\mu_{K+1}$.
Then, for $k>K+1$, choose $a_i(k)$, $i=1, \dots, K+1$ such that
$
\mu'_k (g_i)=
\mu_k (g_i) - \sum_{i=1}^{K+1} a_i(k)
\mu_i(g_i)
$
are zero for $i=1, \dots, A$ and $\sum_{i=1}^{K+1} a_i(k) \neq 1$.
Then after a suitable rescaling, $\mu'_k$ becomes a Calabi quasi-morphism.
$\{\mu'_k \mid k > K+1\}$ is the set of linearly independent Calabi quasi-morphisms.
\end{proof}

\begin{rem}
In case $M$ is either a $k$ ($\ge 3$) points blow up of $\C P^2$ or
cubic surface, we can descend our family of Calabi
quasi-morphisms on $\widetilde{\text{\rm Ham}}(M,\omega)$ to
one on ${\text{\rm Ham}}(M,\omega)$ in the same way as above if we can show that
$\pi_1({\text{\rm Ham}}(M,\omega))$ is a finitely generated group.
\end{rem}
\par
\smallskip
Next we give the proof of Theorem \ref{dips2s2} using Theorem \ref{thm:T(u)} together with
Theorem \ref{sheavyintersectheavy}.

\begin{proof}[Proof of Theorem \ref{dips2s2}]
Note that $T(u)$ in Theorem \ref{dips2s2} is
$T(\rho)$.
Since $S^1_{\rm eq} \times S^1_{\rm eq}$ is the unique Lagrangian torus
fiber with respect to the monotone toric structure on $S^2 \times S^2$, $S^1_{\rm eq} \times S^1_{\rm eq}$
is superheavy with respect to the quasi-morphism $\mu_{e}^{\mathfrak b}$ associated with any idempotent
$e$ of the field factor of $QH_{\mathfrak b_{\rho}}(S^2 \times S^2;\Lambda)$.
(Here we consider
$\mathfrak b=\mathfrak b_{\rho}=T^{\rho} [S^2_{\rm van}]$.)

From Theorem \ref{thm:T(u)} we know that $T(\rho)$ is superheavy with respect to the quasi-morphism
$\mu_{e_{\frak y(\rho)}}^{\frak b(\rho)}$
associated with a suitable idempotent $e_{\frak y(\rho)}$ of $QH_{\mathfrak b_{\rho}}(S^2 \times S^2;\Lambda)$.
Since the superheaviness is invariant under symplectomorphisms,
$\varphi(T(\rho))$ is also $\mu_{e_{\frak y(\rho)}}^{\frak b(\rho)}$-superheavy for any
$\varphi \in {\rm Ham}(S^2(1) \times S^2(1))$.
Since superheavy sets with respect to the same quasi-morphism must intersect by
Theorem \ref{sheavyintersectheavy}, we have
\begin{equation}\label{eq:Theorem 1.13}
\varphi(T(\rho)) \cap (S^1_{\rm eq} \times S^1_{\rm eq}) \neq \emptyset.
\end{equation}
By \cite[Theorem 0.3.C]{gromov}, ${\rm Symp} (S^2(1)\times S^2(1))$ deformation retracts
to the trivial $\Z/2\Z$ extension of $SO(3) \times SO(3)$ where $\Z/2\Z$ is generated by switching the factors.
Thus the identity component of ${\rm Symp} (S^2(1)\times S^2(1))$ is
the same as ${\rm Ham}(S^2(1) \times S^2(1))$ and
switching the factors fixes
$S^1_{\rm eq} \times S^1_{\rm eq}$.
Hence \eqref{eq:Theorem 1.13} holds
for any $\varphi \in {\rm Symp}(S^2(1) \times S^2(1))$.
\end{proof}

\begin{rem}
We can also give a different proof of Theorem \ref{thm:T(u)}
which follows the way similar to the toric case as follows.
(A similar argument is used in Section \ref{sec:cubic}.)
It follows from \cite[Theorem 9.1]{fooo:toricmir}  that for $\alpha >0$ the map
$$
\frak{ks}_{\frak b(\rho)} :
QH_{\frak b(\rho)}(F_2({\alpha});\Lambda) \to \text{\rm Jac}(\frak{PO}_{\frak b(\rho)};\Lambda)
$$
is a ring homomorphism.
(Here $\frak{PO}_{\frak b(\rho)}$ is the potential function (\ref{PObulk}) of $F_2({\alpha})$
computed with respect to its canonical toric structure.)
Since we can take limit $\alpha \to 0$ of $\frak{PO}_{\frak b(\rho)}$
and the potential function of $\widehat F_2(0)$ with respect to its toric structure
is nothing but this limit by the result from \cite{fooo:S2S2}, we find that
$$
\frak{ks}_{\frak b(\rho)} :
QH_{\frak b(\rho)}(\widehat F_{2}(0);\Lambda) \to \text{\rm Jac}(\frak{PO}_{\frak b(\rho)};\Lambda)
$$
is also a ring homomorphism. (See \cite{AFOOO} for the proof of this multiplicative property.)
\begin{lem}\label{lem234} The map
$\frak{ks}_{\frak b(\rho)}$ is surjective.
\end{lem}
\begin{proof}
We can check that $\frak{PO}_{\frak b(\rho)}$ has exactly $4$ critical points
if $\rho \in (0,1)$ and has exactly $2$ critical points in case $\rho =0$.
We can also check that those critical points are
nondegenerate.
Therefore $\text{\rm Jac}(\frak{PO}_{\frak b(\rho)};\Lambda) \cong \Lambda^4$
if $\rho \in (0,1)$ and $\text{\rm Jac}(\frak{PO}_{\frak b(\rho)};\Lambda) \cong \Lambda^2$
if $\rho =0$.
We put
$$
z_1 = T^{1 -u_2}(y^{\text{\bf u}}_2)^{-1},
\quad
z_2 = T^{u_1}y^{\text{\bf u}}_1,
\quad
z_3 = T^{u_2}y^{\text{\bf u}}_2,
\quad
z_4 = T^{2-u_1-2u_2}(y^{\text{\bf u}}_1)^{-1}(y^{\text{\bf u}}_2)^{-2}.
$$
By the same way as in \cite[Lemma 1.2.4]{fooo:toricmir}, we can show that
$z_1,\dots,z_4$ generate a $\Lambda$-subalgebra that is dense in
$\Lambda \langle\!\langle y,y^{-1}\rangle\!\rangle^{\overset{\circ}P}$. (See Definition \ref{def:completion} for the notation.)
Therefore, since $\text{\rm Jac}(\frak{PO}_{\frak b(\rho)};\Lambda)$ is  finite dimensional,
they generate $\text{\rm Jac}(\frak{PO}_{\frak b(\rho)};\Lambda)$ as a $\Lambda$-algebra.
\par
Let $D_i$ ($i=1,\dots,4$) be  the divisors of $X$ associated to the facets
of the moment polytope
$u_2 = 1-\alpha$, $u_1=0$, $u_2=0$, $u_1+2u_2 =2$ respectively.
It is easy to see that
$
\frak{ks}_{\frak b(\rho)}(PD[D_i]) = z_i
$
for $i=2,3,4$ and
$
\frak{ks}_{\frak b(\rho)}(PD[D_1]) = (e^{T^{\rho}}-e^{-T^{\rho}})z_1.
$
The lemma follows.
\end{proof}
\par
Let $\rho \ne 0$. Then $\text{\rm Jac}(\frak{PO}_{\frak b(\rho)};\Lambda)
\cong \Lambda^4$.
Since the Betti number of $\widehat F_0(0)$ is $4$,
Lemma \ref{lem234} implies that
$\frak{ks}_{\frak b(\rho)}$ is an isomorphism.
\par
Let  $\rho = 0$.
Using the fact that $QH_{\frak b(0)}(X;\Lambda)$ is
semisimple,
and $\text{\rm Jac}(\frak{PO}_{\frak b(0)};\Lambda) \linebreak
\cong \Lambda^2$, we derive from Lemma \ref{lem234} that
there exists an idempotent $e_{\frak y(0)}$
that is a unit of the direct factor
$\cong \Lambda$ of $QH_{\frak b(0)}(X;\Lambda)$,
such that  $\frak{ks}_{\frak b(0)}(e_{\frak y(0)}) \ne 0$.
(In the case $\rho \ne 0$, existence of such an idempotent
$e_{\frak y(\rho)}$ is immediate from the fact
that $\frak{ks}_{\frak b(\rho)}$ is an isomorphism.)
\par
Thus in the way similar to the proofs of Theorem \ref{calcithm} and Lemma \ref{calksi}
we find that
$
i_{\text{\rm qm},(\frak b(\rho),b(\frak y^u))}^{\ast}(e_{\frak y(\rho)}) \ne 0.
$
In fact, we can use a de Rham representative of the Poincar\'e dual to
$[S^2_{\mathrm{van}}]$ that is supported in a small neighborhood of
$S^2_{\mathrm{van}}$ and in particular can be taken to be disjoint from $T(\rho)$.
Therefore the above calculation of $\frak{PO}_{\frak b(\rho)}$
makes sense in the homology level.
\end{rem}

\subsection{Critical values and eigenvalues of $c_1(M)$.}
\label{subsec:eigenvalue}
In the course of the proof of Theorem \ref{uncount} (1) given in Subsection \ref{subsec:POb-T(u)},
we proved existence of an idempotent associated to a field factor of quantum cohomology
which is not in the kernel of the maps $i_{{\rm qm},T(0),b_i}^*$,
$i_{{\rm qm},T(\rho), (\frak b(\rho), b(\rho))}^*$. At the end of
this subsection, we specify those idempotents.

We give a digression on the critical values of the potential function and the eigenvalues of the quantum
multiplication by the first Chern class $c_1(M)$ on a general closed symplectic
manifold $(M,\omega)$.
We start with an easy observation.

\begin{lem} \label{lem:pd maslov}
For an oriented Lagrangian submanifold $L \subset M$,
there is a cycle $D$ of codimension $2$ in $M \setminus L$ such that the Maslov index
is equal to twice of the intersection number with $D$, i.e.,
$\mu(\beta)=2\beta \cdot D$ for any $\beta \in H_2(M,L;\Z)$.
\end{lem}

\begin{proof}
Since $L$ is an oriented Lagrangian submanifold, the top exterior power $\bigwedge^n_{\C} TM$ is a trivial
complex line bundle on $L$, where $2n=\dim M$.
Moreover, the volume form of $L$ gives a non-vanishing
section ${\it s}_L$ of $\bigwedge^n_{\C} TM \vert_L$.
We extend ${\it s}_L$ to a section $\it s$ of $\bigwedge^n_{\C}TM$, which is transversal to the zero section.
Then the zero locus $D$ of $\mathit s$ represents the Poincar\'e dual of the first Chern class $c_1(M)$
and the Maslov index $\mu_L:H_2(M,L;\Z) \to \Z$ is given by the twice of the intersection number with $D$.
\end{proof}

For our present purpose, we restrict ourselves to the case of the triple $(M,\omega, J)$
and an oriented Lagrangian submanifold $L \subset X$ satisfy the property that $\mu(\beta) \geq 2$ whenever
the moduli space ${\mathcal M}(L;J;\beta)$ of ($J$-holomorphic) bordered stable maps
 in class $\beta$ is nonempty.
\footnote{See Section 2.1.2 \cite{fooo:book1}
 for the definition of this moduli space.} of bordered stable maps in the class $\beta \neq 0$.
See \cite[Appendix 1]{fooo:S2S2}  for some related results under this condition.
The following theorem was proved (in the Fano toric case) by Auroux \cite[Theorem 6.1]{auroux}.

\begin{thm} \label{thm:c_1&PO}\index{potential function!critical value}
Let ${\mathfrak b}$ be a cycle of codimension $2$ in $M$ with coefficients in $\Lambda_+$ and
$b \in {\mathcal M}_{{\rm weak,def}}(L;\frak b)$.
Then, for any cycle $A$ in $M$,
\begin{equation}\label{c_1 and PO}
i^*_{{\rm qm},({\mathfrak b},b)}(c_1(M)  \cup_{\mathfrak b} PD(A))
=\mathfrak{PO}_{\mathfrak b}(b) i^*_{{\rm qm},({\mathfrak b},b)}(PD(A))
\end{equation}
in $HF((L,{\mathfrak b},b);\Lambda)$.
\end{thm}
\begin{proof}
Let $D$ be the cycle in $M \setminus L$ obtained in
Lemma \ref{lem:pd maslov}.
Since $c_1(M)$ is the Poincar\'e dual of $D$ as a cycle in $M$, we use $D$ to prove the formula
\eqref{c_1 and PO}.

The strategy of the proof is the same as in the proof of \cite[Theorem 2.6.1]{fooo:toricmir}.
Let ${\mathcal M}_{k+1;\ell+2}$ be the moduli space of genus zero bordered stable curves
with $k+1$ boundary marked points $z_0, \dots, z_k$ and $\ell +2$ interior marked points $z_1^+, \dots,
z_{\ell+2}^+$ with connected boundary.
Let ${\mathcal M}_{k+1;\ell+2}(L;\beta)$ be the moduli space of bordered stable maps
in $M$ attached to $L$ in class $\beta \in H_2(M,L;\Z)$ whose domain is a
genus zero bordered semi-stable curves with $k+1$ boundary marked points and $\ell + 2$ interior
marked points with connected boundary.
Here the boundary marked points are
ordered counter-clockwise.
We denote by ${\rm ev}^+_j$ the evaluation map
at the $j$-th interior marked point and by ${\rm ev}_i$ the evaluation map at the $i$-th boundary
marked point $z_i$.  We set $\mathbf{\rm ev}^+=({\rm ev}^+_1, \dots, {\rm ev}^+_{\ell+2})$
and ${\rm ev}=({\rm ev}_1, \dots, {\rm ev}_k)$.

For cycles $Q_1, \dots, Q_{\ell +2}$ in $M$ and chains $P_1, \dots, P_k$ in $L$,
we define
\begin{eqnarray}
 & {\mathcal M}_{k+1, \ell+2}(\beta ; Q_1,  \ldots, Q_{\ell +2};
P_1, \dots, P_k)  \nonumber \\
:= &{\mathcal M}_{k+1, \ell +2}(\beta)_{{{\rm ev}}^+ \times {\rm ev}} \times_{(M^{\ell +2} \times L^k)} (Q_1 \times \cdots \times Q_{\ell +2} \times P_1 \times \cdots \times P_k).
\nonumber
\end{eqnarray}

By stabilizing the domain of the stable map and forgetting the boundary marked
points $z_1, \dots, z_k$ and
the interior marked points
$z^+_3, \dots, z^+_{\ell +2}$, we obtain the forgetful map
$$
\mathfrak{forget}: {\mathcal M}_{k+1, \ell +2}(\beta ;A \otimes D \otimes
{\mathfrak b}^{\otimes \ell};b^{\otimes k}) \to {\mathcal M}_{1;2}.$$

The moduli space ${\mathcal M}_{1;2}$ of bordered stable curves of genus $0$, connected boundary with two interior marked points
and one boundary marked point is a complex manifold with boundary of complex dimension $1$.
It can be easily shown to be homeomorphic to the unit disc (see e.g., \cite[Section 2.6]{fooo:toricmir}).
We pick two distinguished points $[\Sigma_0]$, $[\Sigma_1]$ in ${\mathcal M}_{1;2}$
given as follows.

The bordered stable curve $\Sigma_0$ is the union of the unit disc with $z_0=1$
on its boundary and the Riemann sphere with two marked points $z^+_1, z^+_2$,
away from the nodal point of $\Sigma_0$.
On the other hand, the bordered stable curve $\Sigma_1$ consists of the union of two copies $D_0$, $D_1$
of the unit disc with a boundary node so that we can put
$z_0=1$, $z^+_1=0$ in $D_0$, $z^+_2=0$ in $D_1$ and the boundary node
corresponds to $-1 \in \partial D_0$, $1 \in \partial D_1$. (See Figures 19, 20.)
\begin{center}
\begin{figure}[h]
\begin{tabular}{cc}
 \begin{minipage}[t]{0.45\hsize}
\centering
\includegraphics[scale=0.3]
{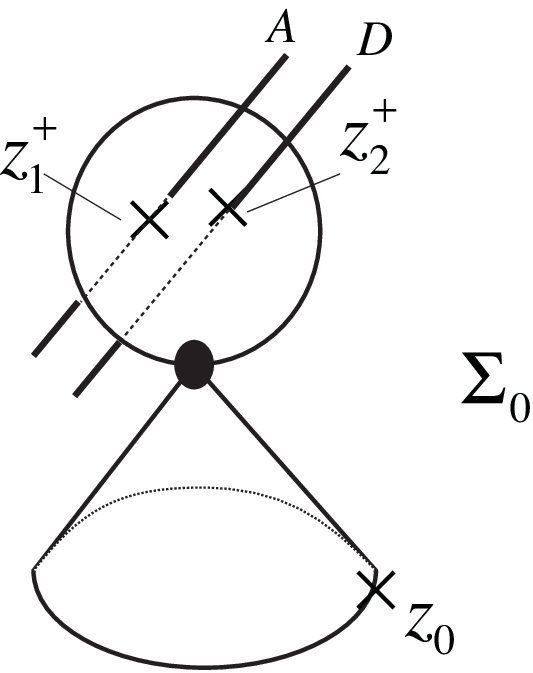}
\caption{$\Sigma_0$}
\label{Figure17}
\end{minipage} &
\begin{minipage}[t]{0.45\hsize}
\centering
\includegraphics[scale=0.3]
{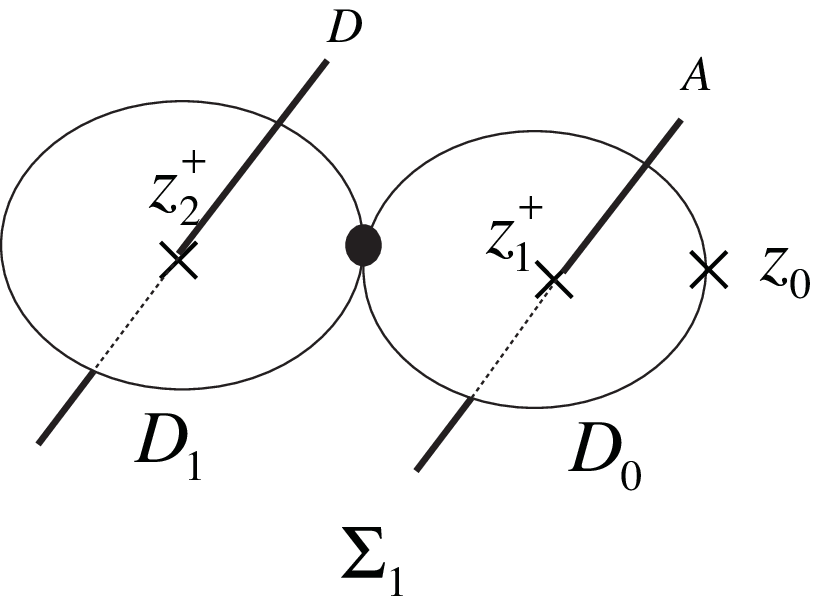}
\caption{$\Sigma_1$}
\label{Figure18}
\end{minipage}
 \end{tabular}
\end{figure}
\end{center}
In our case, since the Maslov index $\mu (\beta)$ is at least $2$ if
${\mathcal M}(L;J;\beta) \neq \emptyset$,
and $\mathfrak b$ is represented by a codimension $2$ cycle, it is enough to
study holomorphic discs of Maslov index $2$ for the computation of
$\mathfrak{m}_0^{\mathfrak b, b}(1)$.
We also recall that the Maslov index of a holomorphic discs attached to $L \subset M$ is equal to
$2 (\beta \cap D)$ where $\beta \cap D$ is the intersection number with $D$.
In particular, if $\mu(\beta) =2$, we have $\beta \cap D = 1$. Combining the above discussions, we
compute
$$
\aligned
\mathfrak{q}([D \otimes e^{\mathfrak b}];e^b)
& :=
\mathfrak{q}\left(\sum_{\ell_1,\ell_2}
\frac{1}{(\ell_1 +\ell_2 +1)!} \mathfrak b^{\otimes \ell_1}\otimes D \otimes {\mathfrak b}^{\otimes \ell_2};e^b \right) \\
& =
\mathfrak{q}(e^{\mathfrak b};e^b) =\mathfrak{PO}_{\mathfrak b}(b)\cdot {\bf e}
\endaligned
$$ for
$b \in {\mathcal M}_{\rm def, weak}(L,{\mathfrak b})$.
Here ${\bf e}$ is the unit.

Once we obtain this identity, we can derive that
the sum of contributions from ${\rm ev}_0:\mathfrak{forget}^{-1}([\Sigma_1]) \to L$
gives rise to the right hand side of \eqref{c_1 and PO}, which is
$\mathfrak{PO}_{\mathfrak b}(b) i^*_{{\rm qm},(\frak b,b)}(PD(A))$.
We use the evaluation map at the marked point corresponding to the boundary
node of $D_1$ to obtain  a differential form on $L$
using the moduli space of $D_1$ in Figure 23.2 as a correspondence. Then it gives
$
\frak m^{\mathfrak b}_0(1) = \mathfrak{PO}_{\mathfrak b}(b) \cdot 1,
$ which is proportional to the unit (the fundamental class).
We remark that we fix the conformal structure of $D_0$.
So there is no freedom of moving the node on $D_0$.
Therefore when we use the evaluation map at $z_0$ to obtain
a differential form of $L$ from the moduli space of
Figure 23.2, it gives rise to
$ \mathfrak{PO}_{\mathfrak b}(b)i^*_{{\rm qm},(\frak b,b)}(PD(A))$
as claimed.

On the other hand, it is straightforward to derive from construction that
the contribution of ${\rm ev}_0:\mathfrak{forget}^{-1}([\Sigma_0]) \to L$
is equal to $i^*_{{\rm qm},(\frak b,b)}(c_1(M) \cup_{\mathfrak b} PD(A))$
which is the left hand side of \eqref{c_1 and PO}.
Finally the proof of the equality of the two contributions from
the two correspondences
$$
{\rm ev}_0:\mathfrak{forget}^{-1}([\Sigma_0]) \to L, \quad {\rm ev}_0:\mathfrak{forget}^{-1}([\Sigma_1]) \to L
$$
follows the way similar to that of \cite{fooo:toricmir} Theorem 2.6.1, especially using
Lemma 2.6.3 therein. This then completes the proof of Theorem \ref{thm:c_1&PO}.
\end{proof}
\begin{cor} \label{ev and crit}
If $A$ is an eigenvector of $c_1(M) \cup^{\frak b}$ on $QH_{\frak b}(M;\Lambda)$ with eigenvalue $\lambda$
and  $i^*_{{\rm qm},({\mathfrak b},b)}(PD(A)) \neq 0$, then
$\lambda=\mathfrak{PO}_{\mathfrak b}(b)$.
\end{cor}
We return to the discussion on $T(\rho)$.
For $T(0)=T(\rho=0)$, we can find that the potential function (without bulk deformations) of $T(0)$ has two critical points
with critical values $\pm 4 T^{1/2}$, by the result of the calculation in \cite[Digression 4.1]{fooo:S2S2}, where
$T(0)$ is denoted by $T({\mathbf u}_0)$.  We have two bounding cochains $b_1,b_2$
with critical values $4T^{1/2}, -4T^{1/2}$ up to gauge equivalence.
\par
\cite[Theorem 3.8.62]{fooo:book1} with   \cite[(3.8.36.2)]{fooo:book1} and Lemma \ref{antidiagonal:unobstruct}
and Lemma \ref{lem:critexistence} (1) taking into account
implies that
$$
i_{{\rm qm},S^2_{\rm van}}^*:H(S^2 \times S^2;\Lambda) \to HF(S^2_{\rm van};\Lambda),
$$
resp.
$$
i_{{\rm qm},T(0),b_i}^*:H(S^2 \times S^2;\Lambda) \to HF((T(0),b_i);\Lambda)
$$ sends
$\sum_{j=1}^4 {\mathbf e}_j$ to the unit $PD[S^2_{\rm van}] \neq 0$ of $HF(S^2_{\rm van};\Lambda)$,
resp. the unit $PD[T(0)] \neq 0$ of $HF((T(0),b_i);\Lambda)$, $i=1,2$.

Recall that $QH(S^2 \times S^2;\Lambda)$ is semi-simple and decomposes into  $\bigoplus_{i=1}^4 \Lambda {\bf e}_i$.
We may assume that ${\bf e}_1,{\bf e}_4$ are eigenvectors of the quantum multiplication by $c_1(S^2 \times S^2)$
with eigenvalues $\pm 4 T^{1/2}$ and ${\bf e}_2,{\bf e}_3$ are those with eigenvalue $0$.
Comparing the critical values of the potential function and eigenvalues of the quantum multiplication
by $c_1(S^2 \times S^2)$, Theorem \ref{thm:c_1&PO} implies that
\beastar
i_{{\rm qm},S^2_{\rm van}}^*({\mathbf e}_2 + {\mathbf e}_3)& = & PD[S^2_{\rm van}],\\
i_{{\rm qm},T(0),b_1}^*({\mathbf e}_1)& = & PD[T(0)],\\
i_{{\rm qm},T(0),b_2}^*({\mathbf e}_4)& = & PD[T(0)].
\eeastar
We may assume that $i_{{\rm qm}, S^2_{\rm van}}^*({\mathbf e}_2) \neq 0$.
By Theorem \ref{thm:heavy} (2), we find that $S^2_{\rm van}$ is $\mu_{{\mathbf e}_2}$-superheavy and
while $T(0)$ is $\mu_{{\mathbf e}_1}$-superheavy and $\mu_{{\mathbf e}_4}$-superheavy.
On the other hand,
since $S^2_{\rm van}$ and $T(0)$ are disjoint, two quasi-morphisms corresponding to
$\mu_{{\mathbf e}_1}$ and $\mu_{{\mathbf e}_2}$ are distinct by
Theorem \ref{sheavyintersectheavy} (and Remark \ref{rem:muimplieszeta}).
This statement is mentioned without proof in
\cite[Remark 7.1]{fooo:S2S2}.
\par
As we showed in Subsection \ref{subsec:POb-T(u)},
the potential function of $T(0)$ with bulk deformation by ${\mathfrak b}_{\rho}$ has
two critical points
$ (\epsilon (e^{T^{\rho/2}} + e^{-T^{\rho/2}})^{-1}, \epsilon (e^{T^{\rho/2}} + e^{-T^{\rho/2}}))$
with $\epsilon = \pm 1$.
The associated critical values are $\pm 2 (e^{T^{\rho}/2} + e^{-T^{\rho}/2})T^{1/2}$.

For $T(\rho)$, the potential function with bulk ${\mathfrak b}_{\rho}$ has critical points
$(\epsilon T^{\rho} (e^{T^{\rho/2}} - e^{-T^{\rho/2}})^{-1}, - \epsilon T^{-\rho} (e^{T^{\rho/2}} - e^{-T^{\rho/2}})$.
The critical values are $\pm 2(e^{T^{\rho}/2} - e^{-T^{\rho}/2})T^{1/2}$.

We now consider the maps
$$
i_{{\rm qm},T(0), (\frak b(\rho), b_i)}^*:QH_{\frak b(\rho)}(S^2 \times S^2;\Lambda) \to HF((T(0) (\frak b(\rho),b_i);\Lambda),
$$
$$
i_{{\rm qm},T(\rho), (\frak b(\rho), b(\rho)_i)}^*:QH_{\frak b(\rho)}(S^2 \times S^2;\Lambda) \to HF((T(\rho) (\frak b(\rho),b(\rho)_i);\Lambda).
$$
These maps have non-zero values at $b_i \in H^1(T(\rho),\Lambda_0)$ and $b(\rho)_i \in H^1(T(\rho),\Lambda_0), \, i = 1, 2$
respectively. (See the proof of Theorem \ref{uncount} (1) given in Subsection \ref{subsec:POb-T(u)}.)
The eigenvalues of the $\frak b(\rho)$-deformed quantum multiplication by $c_1(S^2 \times S^2)$ are
computed as follows.
It follows from  \cite[Remark 5.3 and  Theorem 1.9]{fooo:toric1}(Fano toric case),
\cite[Theorem 1.1.4]{fooo:toricmir} (general toric case) that ${\mathbf e}^{\mathfrak b(\rho)}_1, \dots,
{\mathbf e}^{\mathfrak b(\rho)}_4$ are
eigenvectors of the quantum multiplication with eigenvalues
given by
$$
\aligned
2( e^{T^{\rho}/2}+ e^{-T^{\rho}/2})T^{1/2}, \quad & 2(e^{T^{\rho}/2} - e^{-T^{\rho}/2})T^{1/2}, \\
2(-e^{T^{\rho}/2}+ e^{-T^{\rho}/2})T^{1/2}, \quad & -2(e^{T^{\rho}/2}+ e^{-T^{\rho}/2})T^{1/2},
\endaligned
$$
respectively.

Hence, by Corollary \ref{ev and crit}, $b_i$, $b(\rho)_i$ can be arranged so that
$$
i_{{\rm qm},T(0), (\frak b(\rho), b_1)}^*({\bf e}^{\mathfrak b(\rho)}_1)=i_{{\rm qm},T(0), (\frak b(\rho), b_2)}^*({\bf e}^{\mathfrak b(\rho)}_4)=PD[T(0)]
$$
and
$$
i_{{\rm qm},T(\rho), (\frak b(\rho), b(\rho)_1)}^*({\bf e}^{\mathfrak b(\rho)}_2)=
i_{{\rm qm},T(\rho), (\frak b(\rho), b(\rho)_2)}^*({\bf e}^{\mathfrak b(\rho)}_3)=PD[T(\rho)].
$$
\begin{rem}
The Lagrangian sphere $S^2_{\rm van}$ is unobstructed without  bulk deformation as we saw in
Lemma \ref{antidiagonal:unobstruct}.
Since the self-intersection number of $S^2_{\rm van}$ is $-2$,
${\mathfrak m}_0^{{\mathfrak b}_{\rho}}(1) = -2 T^{\rho}PD [pt]$,
it gets obstructed after the bulk deformation by ${\mathfrak b}_{\rho}$ for $\rho \neq 0$.
\end{rem}

\section{Lagrangian tori in the cubic surface}
\label{sec:cubic}
This section owes much to the paper \cite{nnu2} of Nishinou-Nohara-Ueda,
especially its Subsection 4.1 of the version 1
(arXiv:0812.0066v1). That section contained an error
which seems to be a reason why the subsection was removed
from the second version (arXiv:0812.0066v2).
However, using a result by Chan-Lau \cite{chanlau},
(actually, in \cite[Section 5]{nnu2}  of the second version
they independently obtained the relevant result for the cubic surface by a different argument), we can correct
this error. This provides an interesting example which we discuss in this section.
We would like to emphasize that the idea of exploiting a toric degeneration\index{toric degeneration}
in the calculation of the potential function of a non-toric manifold
used in this section and in \cite{fooo:S2S2} is due to Nishinou-Nohara-Ueda \cite{nnu1}
who successfully applied the idea to various examples.
\par
Following \cite[Subsection 4.1]{nnu2}  of its version one, we consider a family
of cubic surfaces given by
\begin{equation}
M_{t}
= \{
[x:y:z:w] \in \C P^3 \mid xyz - w^3 = t(x^3+y^3+z^3+w^3)
\}
\end{equation}
parametrized by $t \in \C$. For $t\ne 0$ this gives a smooth
surface. For $t=0$, $M_0$ becomes a toric variety with the $(\C^*)^2$-action
$$
(\alpha,\beta)[x:y:z:w]
= [\alpha x:\beta y: \alpha^{-1}\beta^{-1}z:w].
$$
The Fubini-Study form on $\C P^3$ induces a symplectic structure on $M_t$.
This symplectic structure on $M_0$ is invariant under the action
of real torus $T^2 \subset (\C^*)^2$. The moment polytope of this action is
given by
\begin{equation}
P =
\{
(u_1,u_2) \in \R^2 \mid \ell_i(u_1,u_2) \ge 0, \,\,
i=1,2,3
\}
\end{equation}
where
\begin{equation}
\aligned
\ell_1(u_1,u_2)
&=  -u_1 + 2u_2 + 1,\\
\ell_2(u_1,u_2)
&=  2u_1  - u_2 + 1,\\
\ell_3(u_1,u_2)
&=  -u_1  - u_2 + 1.
\endaligned
\end{equation}
\begin{figure}[h]
\centering
\includegraphics[scale=0.3]
{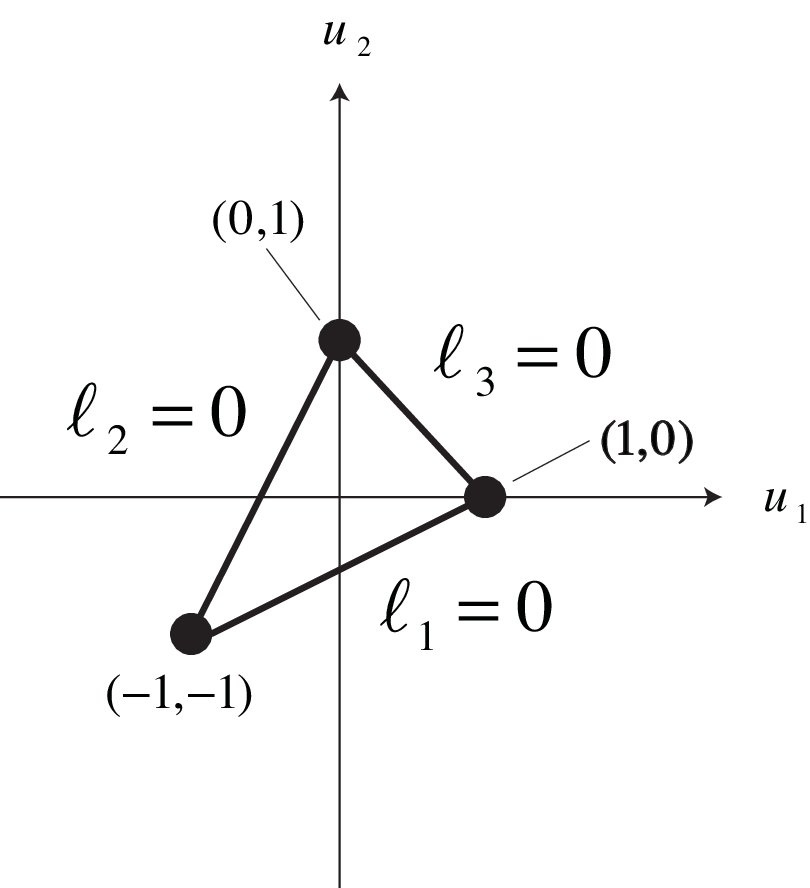}
\caption{Polytope $P$}
\label{Figure19}
\end{figure}
The moment polytope $P$ is an isosceles triangle, whose center of gravity is origin.
The three vertices of $P$ correspond to the three singular points of $M_0$.
The variety $M_0$ is a toric orbifold with three singular points of $A_2$-type.
\par
We can deform those three singular points by gluing the Milnor fiber of
the $A_2$ singularity by the same way as in Section \ref{sec:exotic}
to obtain a symplectic manifold $M$. It is easy to see that $M$ is symplectomorphic
to $M_t$ for $t\ne 0$.
(Note $M_t$ is symplectomorphic to $M_{t'}$ if $t,t'\ne 0$.)
\par
We consider
$$
\frak Z =
\big(\R_{\ge 0}(1,0)) \sqcup (\R_{\ge 0}(0,1)) \sqcup (\R_{\ge 0}(-1,-1)\big)
\cap \text{\rm Int} P.
$$
For $\text{\bf u} \in \frak Z$ we consider $\pi^{-1}(\text{\bf u})
\subset M_0$.
In the same way as in Section \ref{sec:exotic} we may regard it as a
Lagrangian torus in $M$. We denote it by $T(\text{\bf u})$.
\begin{figure}[h]
\centering
\includegraphics[scale=0.3]
{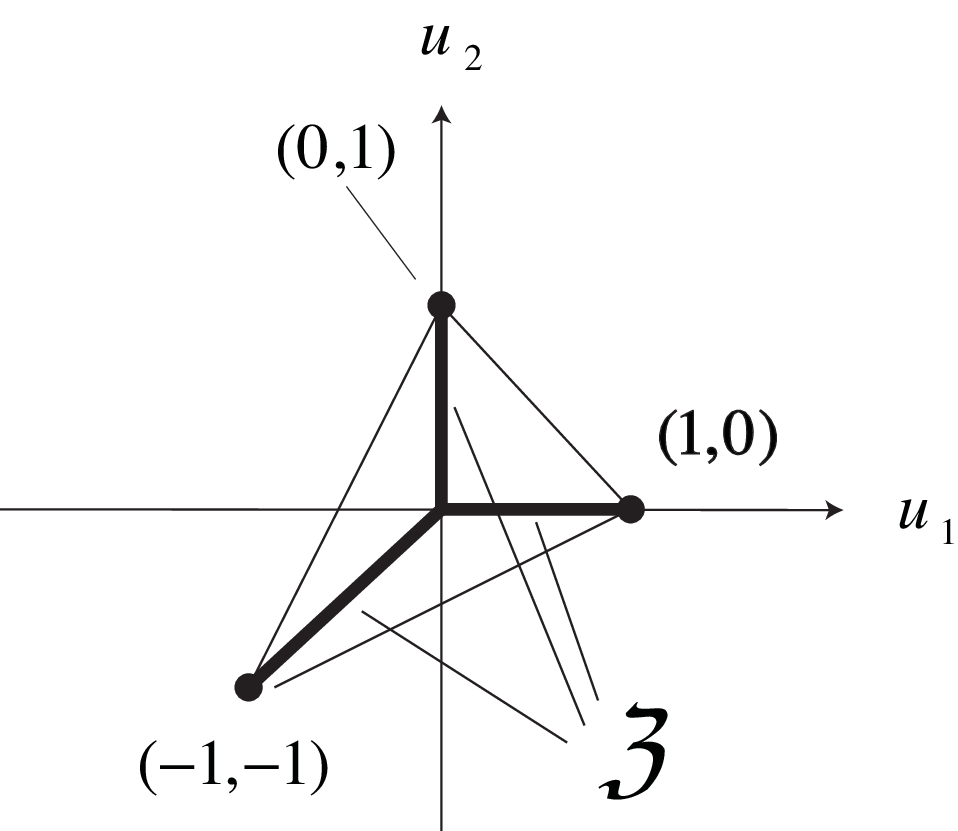}
\caption{Set $\frak Z$}
\label{Figure20}
\end{figure}
\begin{thm}\label{cubicmain}
For each  $\text{\bf u} \in \frak Z$, there exist
$\frak b(\text{\bf u}) \in H^2(M;\Lambda_+)$
and $b(\text{\bf u}) \in H^1(T(\text{\bf u});\Lambda_0)$
such that
$$
HF((T(\text{\bf u}),(\frak b(\text{\bf u}),b(\text{\bf u})));\Lambda) \ne 0.
$$
Moreover there exists $e_{\text{\bf u}}$ that is a unit of a direct
product factor $e_{\text{\bf u}}\Lambda = \Lambda$ of
$QH_{\frak b(\text{\bf u})}(M;\Lambda)$ such that
$$
i_{\text{\rm qm},(\frak b(\text{\bf u}),b(\text{\bf u}))}^{\ast}(e_{\text{\bf u}}) \ne 0
\in HF((T(\text{\bf u}),(\frak b(\text{\bf u}),b(\text{\bf u})));\Lambda).
$$
\end{thm}
We can use this theorem in the same way as in Section
\ref{sec:exotic} to show the following.
\begin{cor}
\begin{enumerate}
\item
Each of $T(\text{\bf u})$ is non-displaceable.
\item
$T(\text{\bf u})$ is not Hamiltonian isotopic to
$T(\text{\bf u}')$ if $\text{\bf u} \ne \text{\bf u}'$.
\item
There exist uncountably many homogeneous Calabi quasi-morphisms\index{quasi-morphism!Calabi quasi-morphism}
$$
\mu_{e_{\text{\bf u}}}^{\frak b(\text{\bf u})}
: \widetilde{\text{\rm Ham}}(M;\omega) \to \R
$$
which are linearly independent.
\end{enumerate}
\end{cor}
\begin{proof}[Proof of Theorem \ref{cubicmain}]
We consider a toric {\it resolution} of our orbifold $M_0$,
which we denote by $M(\epsilon)$.
We may take it so that its  moment polytope is
\begin{equation}
P_{\epsilon}
=
\{
(u_1,u_2) \in P \mid
\ell_i^{\epsilon}(u_1,u_2) \ge 0, \quad i=4,\dots,9
\},
\end{equation}
where
\begin{equation}
\aligned
\ell_4^{\epsilon}(u_1,u_2) &=
u_1 + 1 - \epsilon = \frac{1}{3}(2\ell_1+\ell_2) - \epsilon, \\
\ell_5^{\epsilon}(u_1,u_2) &=
u_2 + 1 - \epsilon = \frac{1}{3}(\ell_1+2\ell_2) - \epsilon, \\
\ell_6^{\epsilon}(u_1,u_2) &=
u_1-u_2 + 1 - \epsilon = \frac{1}{3}(2\ell_2+\ell_3) - \epsilon ,\\
\ell_7^{\epsilon}(u_1,u_2) &=
-u_2 + 1 - \epsilon = \frac{1}{3}(\ell_2+2\ell_3) - \epsilon, \\
\ell_8^{\epsilon}(u_1,u_2) &=
-u_1 + 1 - \epsilon = \frac{1}{3}(2\ell_3+\ell_1) - \epsilon, \\
\ell_9^{\epsilon}(u_1,u_2) &=
-u_1+u_2 + 1 - \epsilon = \frac{1}{3}(\ell_3+2\ell_1) - \epsilon.
\endaligned
\end{equation}
We put
$$
D_i = \pi^{-1}(\partial_i P_{\epsilon}),
\quad
\partial_i P_{\epsilon}
= \{(u_1,u_2) \in P_{\epsilon} \mid \ell_i^{\epsilon}(u_1,u_2) = 0\},
$$
for $i=4,\dots,9$.
($D_i$, $i=1,2,3$ are defined in the same way.)

\begin{figure}[h]
\centering
\includegraphics[scale=0.3]
{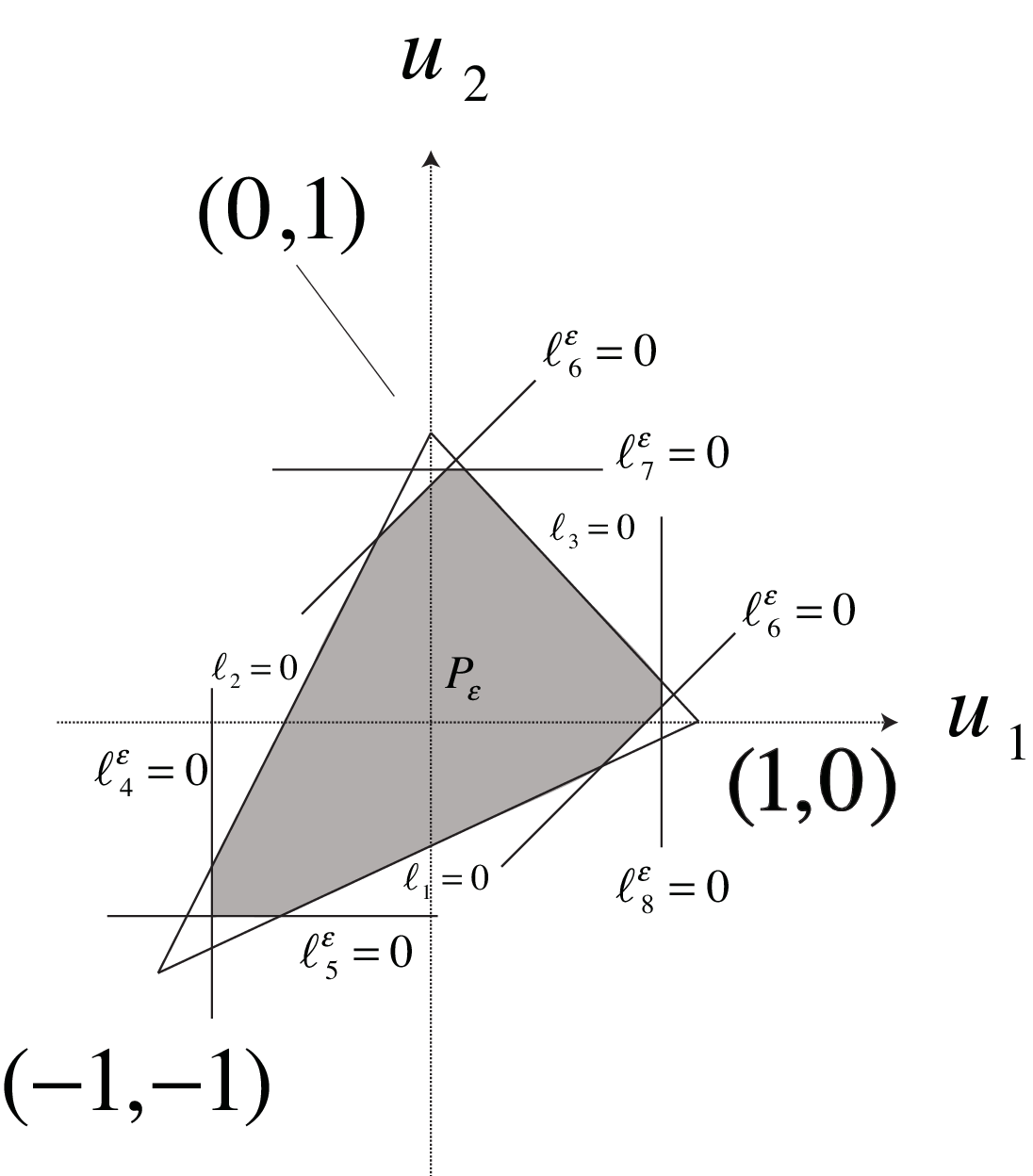}
\caption{Polytope $P_{\epsilon}$}
\label{Figure21}
\end{figure}
We note that $M(\epsilon)$ is nef but is not Fano.
In fact, $c_1(M(\epsilon)) \cap D_i = 0$ for $i=4,\dots,9$.
The potential function of $M(\epsilon)$ is calculated by Chan and Lau.
In fact, $M(\epsilon)$ is $X_{11}$ in the table given in  \cite[p.19]{chanlau}.
\par
Using the fact that $M$ is monotone, we can apply the argument of
\cite[Section 6]{fooo:S2S2}  to show that we can take the limit $\epsilon \to 0$
to calculate the potential function of $T(\text{\bf u})$ in $M$.
The result is the following.
\par
Let $e_1,e_2$ be a basis of $H^1(T(\text{\bf u});\Z)$ and put
$b = x_1e_1+x_2e_2 \in H^1(T(\text{\bf u});\Lambda_0)$.
We put $\overline y_i = e^{x_i}$ and
$y_i = T^{u_i}\overline y_i$, where ${\bf u}=(u_1,u_2)$.
\begin{thm}\label{cubicpocalcu}
The potential function of $T(\text{\bf u}) \subset M$ is given by
\begin{equation}\label{POcubic}
\aligned
\frak{PO}
= T\big(
&y_1^{-1}y_2^{-1} (y_1+y_2)^3
+
y_1^{-1}y_2^2(y_1y_2^{-1} + y_2^{-1})^3 \\
&+
y_1^{2}y_2^{-1}(y_1^{-1} + y_1^{-1}y_2 )^3
- y_1^{-1}y_2^2 - y_1^2y_2^{-1}  - y_1^{-1}y_2^{-1}
\big).
\endaligned
\end{equation}
\end{thm}
We postpone the proof of Theorem \ref{cubicpocalcu} till later
when we prove Theorem \ref{cubicbubblecal} using (\ref{CLcalcu}).

\begin{cor}\label{displacecubic}
For each $\text{\bf u}\in \frak Z$, there exists
$b \in H^1(T(\text{\bf u}),\Lambda_0)$ such that
$$
HF((T(\text{\bf u}),b),(T(\text{\bf u}),b);\Lambda_0)
\cong
H(T^2;\Lambda_0).
$$
\end{cor}
\begin{proof}
We define $Y_1,Y_2$ by the formula
\begin{eqnarray}
Y_1^3 &=& y_1^2y_2^{-1}, \label{defz1}\\
Y_1^2Y_2 &=& y_1. \label{defz2}
\end{eqnarray}
Note for each $y_1,y_2$ there are 3 choices of $Y_1$ satisfying
(\ref{defz1}).
Then (\ref{defz2}) uniquely determines $Y_2$.
Thus $(Y_1,Y_2) \mapsto (y_1,y_2)$ is a three to one correspondence.
\par
Now we can rewrite (\ref{POcubic}) as follows:
$$
\aligned
\frak{PO}
&=
T
\left\{
(Y_1+Y_2)^3 + (Y_1+Y_1^{-1}Y_2^{-1})^3 + (Y_2+Y_1^{-1}Y_2^{-1})^3
- Y_1^3 -  Y_2^3 - Y_1^{-3}Y_2^{-3})
\right\} \\
&=
T((Y_1+Y_2+Y_1^{-1}Y_2^{-1})^3 - 6).
\endaligned
$$
Therefore
$$
\aligned
\frac{1}{T}\frac{\partial \frak{PO}}{\partial Y_1}
&=
3(1 - Y_1^{-2}Y_2^{-1})(Y_1+Y_2+Y_1^{-1}Y_2^{-1})^2 \\
\frac{1}{T}\frac{\partial \frak{PO}}{\partial Y_2}
&=
3(1 - Y_1^{-1}Y_2^{-2})(Y_1+Y_2+Y_1^{-1}Y_2^{-1})^2.
\endaligned
$$
Therefore the critical point is either
\begin{equation}\label{critz1}
Y_1=Y_2,
\quad
Y_1^3 = 1
\end{equation}
or
\begin{equation}\label{critz2}
Y_1+Y_2+Y_1^{-1}Y_2^{-1}=0.
\end{equation}
The equation (\ref{critz1}) gives a single solution $y_1=y_2=1$.
\par
By \eqref{defz1}, \eqref{defz2}, the equation \eqref{critz2} is equivalent to
$y_1 + y_2 +1 = 0$ in old coordinates $(y_1,y_2)$.
Therefore the point $(y_1,y_2)$ given by the coordinates
$$
\aligned
y_1 = & Y_1^2Y_2 = -c^3T^{3v}(1-c^{-3}T^{-3v}  + \cdots),\\
y_2 = & Y_1Y_2^2 =  c^3T^{3v}(1-2c^{-3}T^{-3v} + \cdots),
\endaligned
$$
is a critical point, when ${\frak v}_T(c) = 0$. Therefore the expression shows that
it is a critical point with its valuation vector given by
$\text{\bf u} = (3v,3v) \in \R_{>0}(-1,-1)$,
for any $v<0$.
The corollary now follows from the obvious $\Z_3$-symmetry.
\end{proof}

\begin{rem}
Corollary \ref{displacecubic} implies that the Jacobian ring
$$
\text{\rm Jac}(\frak{PO};\Lambda) =
\frac{\Lambda\langle\!\langle
y,y^{-1}\rangle\!\rangle^{\overset{\circ}P}}{\text{\rm Clos}_{d_{{\overset{\circ}P}}}
\left(y_i \frac{\partial \frak{PO} }{\partial y_i} ; i =1,2\right)}
$$
is infinite dimensional over $\Lambda$. Recall that in the toric case
the Jacobian ring is always finite dimensional
since the Kodaira-Spencer map
$
\frak{ks}_{\text{\bf 0}} : QH(X;\Lambda) \to  \text{\rm Jac}(\frak{PO};\Lambda)
$
is an isomorphism. (See Theorem \ref{JacisHQ}.)
\end{rem}
Corollary \ref{displacecubic} implies the existence of a continuum of
mutually disjoint non-displaceable Lagrangian tori in a cubic surface.
To show the existence of infinitely many Calabi quasi-morphisms  and
prove Theorem \ref{uncount} (2), we need to
study bulk deformations.
Let
\begin{equation}
\vec w = (w_1,\dots,w_9) \in \Lambda_0^9.
\end{equation}
We put
\begin{equation}
\frak b({\vec w}) = \sum_{i=1}^9 w_i PD(D_i).
\end{equation}
\begin{thm}\label{cubicbubblecal}
We have
$$
\aligned
&\frac{1}{T}(\frak{PO}_{\frak b({\vec w})} - \frak{PO}) \\
=
&(e^{w_1}-1) y_1^{-1}y_2^2 + (e^{w_2}-1) y_1^{2}y_2^{-1}
+ (e^{w_3} - 1)  y_1^{-1}y_2^{-1}  \\
&+ (e^{w_4} + e^{w_5-w_4} + e^{-w_5} -3) y_1
+ (e^{w_5} + e^{w_4-w_5} + e^{-w_4} -3) y_2 \\
&+ (e^{w_6} + e^{w_7-w_6} + e^{-w_7} -3) y_1y_2^{-1}
+ (e^{w_7} + e^{w_6-w_7} + e^{-w_6} -3) y_2^{-1}
\\
&+ (e^{w_8} + e^{w_9-w_8} + e^{-w_9} -3) y_1^{-1}
+ (e^{w_9} + e^{w_8-w_9} + e^{-w_8} -3) y_1^{-1} y_2.
\endaligned$$
\end{thm}
\begin{proof}
We consider the term
$3y_1^{-1}y_2^{-1} y_1^2y_2 = 3y_1$ in (\ref{POcubic}).
This term comes from the moduli space
$\mathcal M(\beta)$ where
$
\beta = \beta_4 + \alpha
$
with
\begin{equation}\label{beta4int}
\beta_4 \cap D_j =
\begin{cases}
1 &j =4 \\
0 & j\ne 4,
\end{cases}
\end{equation}
and $\alpha \in H_2(M;\Z)$
with
$$
\alpha = k_1 [D_4] + k_2 [D_5].
$$
We define
$$
d(k_1,k_2) = \text{\rm deg} (\text{\rm ev}_0 : \mathcal M_1(\beta_4 + k_1 [D_4] + k_2 [D_5]) \to L(\text{\bf u})).
$$
By the result of Chan-Lau, \cite[Theorem 1.1]{chanlau},
(and the fact that the potential functions
are continuous with respect to the limit $\epsilon \to 0$), we derive
\begin{equation}\label{CLcalcu}
d(k_1,k_2) =
\begin{cases}
1 &(k_1,k_2) = (0,0),(1,0),(1,1) \\
0 & \text{otherwise}.
\end{cases}
\end{equation}
This result is also obtained independently in Section 5 of
the second version of \cite{nnu2} based on a different argument.
Therefore, by the proof of \cite[Proposition 9.4]{fooo:bulk}, the coefficient of
$y_1$ in $\frak{PO}_{\frak b_{a,b}(u)}$ is given by
$$\aligned
\sum_{k_1,k_2} d(k_1,k_2)
&\exp(w_4 [D_4] \cap [\beta_4+ k_1 [D_4] + k_2 [D_5]]) \\
&\exp(w_5 [D_5]  \cap [\beta_4+ k_1 [D_4] + k_2 [D_5]]) \\
&\hskip-3cm=
e^{w_4} + e^{w_5-w_4} + e^{-w_5}.
\endaligned$$
(Here we use (\ref{beta4int}) and
$[D_4]\cdot [D_4] = [D_5]\cdot [D_5] = -2$, $[D_4]\cdot [D_5] = 1$.)
\par
In the same way the coefficient of $y_2$ in $\frak{PO}_{\frak b(\vec w)}$ is given by
$e^{w_5} + e^{w_4-w_5} + e^{-w_4}$.
This proves the second line of the right hand side formula.
The third and fourth lines of the formula can be proved in the same way.
The proof of the first line is easier and so omitted.
\end{proof}
We put
$$
\vec w_0 = (0,0,0,w_0,w_0,w_0,w_0,w_0,w_0,w_0),
\qquad
e^{2w_0} +  e^{w_0} +1 = 0.
$$
Theorem \ref{cubicbubblecal} implies
\begin{equation}\label{POw0}
\frak{PO}_{\frak b(\vec w_0)}
=
T(y_1^{-1}y_2^2 + y_1^2y_2^{-1}  + y_1^{-1}y_2^{-1}).
\end{equation}

\begin{rem}
According to \cite[Proposition 3.10]{iritani3} , the function
(\ref{POw0}) is the Landau-Ginzburg superpotential of the
mirror of the toric orbifold $M_0$.
\end{rem}
\begin{lem}
$\frak{PO}_{\frak b(\vec w_0)}$ has $9$ critical points.
All of them have valuation $0$ and are nondegenerate.
\end{lem}
\begin{proof}
We can easily check that the critical points are given
by $y_1^3 = y_2^3 =1$.
\end{proof}
\begin{lem}
For a generic $\vec w$ the set of critical points of $\frak{PO}_{\frak b(\vec w)}$
consists of $9$ elements all of which are nondegenerate.
\end{lem}
\begin{proof}
The Newton polytope of the Laurent polynomial  $\frak{PO}_{\frak b(\vec w)}$
of $y$ has volume $9/2$. Therefore by
the result of Kushnirenko \cite{kushnire}
the number of critical points are at most $9$.
Since it is exactly $9$ for $\vec w = \vec w_0$, it is so
for a generic $\vec w$.
Since the number is maximal it must be nondegenerate.
\end{proof}
\begin{lem}\label{semisimplecubic}
Suppose that the set of critical points of $\frak{PO}_{\frak b(\vec w)}$
consists of $9$ elements all of which are nondegenerate.
We also assume that the valuation vectors of the critical points are in the
interior of the moment polytope $P$.
We also assume that none of the following happens.
\begin{enumerate}
\item
$e^{3w_4} = 1$, $e^{2w_4} = e^{w_5}$.
\item
$e^{3w_6} = 1$, $e^{2w_6} = e^{w_7}$.
\item
$e^{3w_8} = 1$, $e^{2w_8} = e^{w_9}$.
\end{enumerate}
\par
Then the homomorphism
$$
\frak{ks}_{\frak b(\vec w)} : QH_{\frak b(\vec w)}(X;\Lambda) \to
\text{\rm Jac}(\frak{PO}_{\frak b(\vec w)};\Lambda)
$$
is an isomorphism.
Moreover $QH_{\frak b(\vec w)}(X;\Lambda)$ is semi-simple.
\end{lem}
\begin{proof}
We can prove that $\frak{ks}_{\frak b(\vec w)}$ is a ring homomorphism in a similar way as in
\cite[Theorem 2.6.1]{fooo:toricmir}.  (See \cite{AFOOO} for the detail.)
We put
$$
z_i = y_1^{\frac{\partial \ell_i}{\partial y_1}}y_2^{\frac{\partial \ell_i}{\partial y_2}}.
$$
In a way similar to \cite[Lemma 1.2.4]{fooo:toricmir},  we can prove that the set
$\{ z_i \mid i=1,\dots,9\}$ generates a dense $\Lambda$-subalgebra
of $\Lambda\langle\!\langle
y,y^{-1}\rangle\!\rangle^{\overset{\circ}P}$.
Since $\text{\rm Jac}(\frak{PO}_{\frak b(\vec w)};\Lambda)$
is finite dimensional, it is generated by the
images of $z_i, i=1,\dots,9$ as a $\Lambda$-algebra.
\par
By differentiating the formula given in
Theorem \ref{cubicbubblecal}, we find that the cohomology class
$PD[D_i]$ is mapped to $e^{w_i}z_i$ by $\frak{ks}_{\frak b(\vec w)}$
for all $i=1,2,3$.
We calculate
$$
\left(
\begin{matrix}
\frak{ks}_{\frak b(\vec w)}(PD([D_4])) \\
\frak{ks}_{\frak b(\vec w)}(PD([D_5]))
\end{matrix}
\right)
=
\left(
\begin{matrix}
e^{w_4}- e^{w_5-w_4} & e^{w_4-w_5} - e^{-w_4} \\
e^{-w_4 +w_5} - e^{-w_5}& e^{w_5}- e^{w_4-w_5}
\end{matrix}
\right)
\left(
\begin{matrix}
z_4
\\
z_5
\end{matrix}
\right)
$$
By assumption the matrix in the right hand side is nonzero.
Therefore the image of $\frak{ks}_{\frak b(\vec w)}$
contains either $z_4$ or $z_5$. Since
$z_4z_5 = z_1z_2$, it contains both of $z_4$ and $z_5$.
In  a similar way we find that the image of
$\frak{ks}_{\frak b(\vec w)}$ contains $z_6, \dots, z_9$.
Therefore $\frak{ks}_{\frak b(\vec w)}$ is surjective.
\par
The rank of
$\text{\rm Jac}(\frak{PO}_{\frak b(\vec w)};\Lambda)$ is $9$
that is equal to the Betti number of $X$.
Therefore $\frak{ks}_{\frak b(\vec w)}$ is an isomorphism and
$QH_{\frak b(\vec w)}(X;\Lambda) \cong \Lambda^9$. This in particular proves that
the ring $QH_{\frak b(\vec w)}(X;\Lambda)$ is semi-simple.
\end{proof}
We next put
$$
\vec w_{u;c}
= (0,0,0,w(u;c),w(u;c),0,0,0,0),
\qquad
e^{w(u;c)} + 1 + e^{-w(u;c)} = 3+cT^u
$$
with $c \in \C \setminus \{0\}$, $u\ge 0$.
(We observe that $\frak v_T(w(u;c)) = u/2$.)

\begin{lem} Let $\vec w_{u;c}$ be the vector given as above,
and assume $c$ is chosen generically for $u=0$. Then
\begin{enumerate}
\item
$\frak{PO}_{\frak b(\vec w_{u;c})}$ has exactly $3$ nonzero critical points.
\item
The valuation of the critical points are $(0,0)$.
\item
All the three critical points are nondegenerate.
\end{enumerate}
\end{lem}
\begin{proof}
We change the variable from $(y_1,y_2)$ to $(Y_1, Y_2)$
determined by the relation (\ref{defz1}),  (\ref{defz2}).
Then by substituting the vector $w = \vec w_{u;c}$ above
into the formula given in Theorem \ref{cubicbubblecal} we obtain
$$
\frak{PO}_{\frak b(\vec w_{u;c})}
= T((Y_1+Y_2+Y_1^{-1}Y_2^{-1})^3 - 6)
+ c T^{1+u}(Y_1+Y_2)Y_1Y_2.
$$
Therefore the critical point equation $\nabla\frak{PO}_{\frak b(\vec w_{u;c})} = 0$ is equivalent to
\begin{eqnarray}
&3(1 - Y_1^{-2}Y_2^{-1})(Y_1+Y_2+Y_1^{-1}Y_2^{-1})^2 + cT^uY_2(2Y_1+Y_2) = 0,
\label{critpowuc1}
\\
&3(1 - Y_1^{-1}Y_2^{-2})(Y_1+Y_2+Y_1^{-1}Y_2^{-1})^2 + cT^uY_1(2Y_2+Y_1) = 0.
\label{critpowuc2}
\end{eqnarray}
In particular we obtain the equation
$$
\frac{2Y_1+Y_2}{Y_1^2Y_2 - 1}
=
\frac{2Y_2+Y_1}{Y_2^2Y_1 - 1}.
$$
Since $Y_1 \ne 0$, $Y_2 \ne 0$ this equation implies of the
following two alternatives:
$$
\text{\rm either }\,  Y_1 = Y_2 \quad \text{ \rm or }\,  Y_1Y_2(Y_1+Y_2) = -1.
$$
\par
In case $Y_1Y_2(Y_1+Y_2) = -1$, (\ref{critpowuc1}) and (\ref{critpowuc2}) imply
$2Y_1+Y_2 = 0$ and $2Y_2+Y_1 = 0$ respectively. But this equation also has trivial solution
which is ruled out since $Y_1 \neq 0 \neq Y_2$.
Therefore $Y_1=Y_2$ must hold. If we put $x=Y^3_1=Y^3_2$,
then the system of equations (\ref{critpowuc1}), (\ref{critpowuc2}) is equivalent to
\begin{equation}\label{3irdorderX}
(x-1)(2x+1)^2 + cT^u x^3 = 0.
\end{equation}
This equation has three simple roots. (We use genericity of $c$ in case $u =0$.)
This proves $(1)$.
(Note $y_1 = y_2 = x$.)
\par
Now we prove (2).
If $u=0$ $x \in \C \setminus \{0\}$. Therefore the valuations of $y_1,y_2$ are
$0$.
If $u>0$, then $x \equiv 1$ or $-1/2$ modulo $\Lambda_+$.
Therefore $\frak v_T(x) =0$ also.
This proves (2).
\par
We next prove (3).
We calculate
\begin{equation}
\left(
y_iy_j \frac{\partial^2 \frak{PO}_{\frak b(\rho)}}{\partial y_i\partial y_j}
\right)_{i,j=1}^2
=
\left(
\begin{matrix}
A  &   B \\
B            &A
\end{matrix}
\right)
\end{equation}
where
$$
A=\frac{T}{x^2}(4x^3+6x^2+6x+2), \quad
B=\frac{T}{x^2}(-4x^3-6x^2 + 1).
$$
Here we use (\ref{3irdorderX}) during the calculation.
Therefore we get
$$
\text{det}
\left(
y_iy_j \frac{\partial^2 \frak{PO}_{\frak b(\rho)}}{\partial y_i\partial y_j}
\right)_{i,j=1}^2
= 3T^2 \frac{(2x+1)^4}{x^4}
$$
and so (3) follows.
\end{proof}
We now consider an affine line $C \cong \C$ contained in
$\C^9$ so that it contains $\vec w_0$ and
$\vec w_{0;c}$ where $c$ is generic.
We denote by $\frak{b}(a)$ the bulk ambient cycle associated to $a$ and
consider the set
 $$
 \frak X_0
 =
 \{ (a;y_1,y_2) \mid a \in C,  \nabla(\frak{PO}_{\frak b(a)}) = 0
 \,\,\text{at $(y_1,y_2) \in \C^2$} \}.
$$
We take the Zariski closure of $ \frak X_0$ in $C \times \C P^1\times \C P^1$
and denotes it by $ \frak X$.
We have a natural projection
$
\pi :  \frak X \to C.
$
At a generic point $a\in C$ the fiber of $\pi$ consists of $9$ points.
At $a = \vec w_{0;c}$ the fiber of $\pi$ intersects with $ \frak X_0$ at $3$ points
and those three points are all simple.
Therefore
$
\pi^{-1}(\vec w_{0;c}) \setminus  \frak X_0 \ne \emptyset.
$
Then there exist Laurent power series
$$
y_1(w) = \sum_{k\ge k_{0,1}} y_{1,k} w^k,
\quad
y_2(w) = \sum_{k\ge k_{0,2}} y_{2,k} w^k,
$$
and
$$
a(w) = \sum_{k\ge 0} a_k w^k
$$
such that the following holds:
\begin{enumerate}
\item
$y_1(w)$, $y_2(w)$ converge for $\vert w\vert \in (0,\epsilon)$.
\item
$a(w)$ converges for $\vert w \vert \in [0,\epsilon)$.
\item
$(a(w);y_1(w),y_2(w)) \in \frak X_0$ for $\vert w \vert \in (0,\epsilon)$.
\item
$(y_1(0),y_2(0)) \in (\C P^1)^2 \setminus (\C\setminus\{0\})^2$.
\item
$a(0) = \vec w_{0;c}$.
\end{enumerate}
Now we consider
$(y_1(T^{\rho}),y_2(T^{\rho})) \in \Lambda^2$ and
$a(T^{\rho}) \in \Lambda_0^9$.
Then (3) implies that $(y_1(T^{\rho}),y_2(T^{\rho}))$ is a critical
point of $\frak{PO}_{\frak b(a(T^{\rho}))}$.
\begin{lem}
If $(\frak v_T(y_1(T^{\rho})),\frak v_T(y_2(T^{\rho}))) \in P$,
then
$$
(\frak v_T(y_1(T^{\rho})),\frak v_T(y_2(T^{\rho}))) \in \frak Z.
$$
\end{lem}
\begin{proof}
If $(u_1,u_2) = (\frak v_T(y_1),\frak v_T(y_2)) \in P \setminus \frak Z$, then
there exist $i=1,2,3$ such that
$
\ell_i(u_1,u_2) < \ell_j(u_1,u_2)
$
for each $j \in \{1,\dots,9\}$, $j\ne i$.
It easily follows that $(y_1,y_2)$ is not a critical point of $\frak{PO}_{\frak b(a(T^{\rho}))}$.
\end{proof}
It is also easy to see that
$\rho \mapsto (\frak v_T(y_1(T^{\rho})),\frak v_T(y_2(T^{\rho})))$
is continuous and
$$
\lim_{\rho \to 0}  (\frak v_T(y_1(T^{\rho})),\frak v_T(y_2(T^{\rho})))
= (0,0).
$$
Moreover we find that $\frak v_T(y_1(T^{\rho}))$ and
$\frak v_T(y_2(T^{\rho}))$ are either increasing or decreasing, and
$(\frak v_T(y_1(T^{\rho})),\frak v_T(y_2(T^{\rho})))$ diverges as $\rho \to \infty$.
(This is a consequence of (4) above.)
Therefore there exists $\rho_1>0$ such that
$$
\rho \mapsto (\frak v_T(y_1(T^{\rho})),\frak v_T(y_2(T^{\rho})))
$$
defines a homeomorphism between $(0,\rho_1)$
and one of the sets
$$
\frak Z_1 = \{ (-u,-u) \mid u \in (0,1)\}, \,
\frak Z_2 =  \{ (u,0) \mid u \in (0,1)\}, \,
\frak Z_3 =  \{ (0,u) \mid u \in (0,1)\}.
$$
Note that there exist $6$ choices of such $(y_1(q),y_2(q))$ for given
$a(w)$, after replacing $C$ by an appropriate branched cover
branching at $w = \vec w_{0;c}$. This is because the order of the set $\pi^{-1}(a(w)) \subset \frak X$ is $9$
for generic $w$ and the set $\pi^{-1}(a(0)) \subset \frak X$
consists of $3$ points all of which are simple.
Each of such $6$ choices determines $\rho_1$ above.
We take its minimum and denote it by $\rho_0$.
Thus we have proved the following:
\begin{lem}\label{existbulkcubic}
\begin{enumerate}
\item
For each $\rho \in (0,\rho_0)$,
there exist exactly $9$ critical points of $\frak{PO}_{\frak b(a(T^{\rho}))}$.
They are simple and
their valuation vectors are always contained in the interior of $P$.
\item
We may take a choice of $(y_1(q),y_2(q))$ as above so that
$$
\rho \mapsto (\frak v_T(y_1(T^{\rho})),\frak v_T(y_2(T^{\rho})))
$$
defines a homeomorphism between $(0,\rho_0)$
and one of
$\frak Z_1$,
$\frak Z_2$,
$\frak Z_3$.
\end{enumerate}
\end{lem}
Lemmas \ref{semisimplecubic}, \ref{existbulkcubic} and
Theorem \ref{thm:heavy} imply that the following holds for one of $i=1,2,3$.
We also note that we can prove Theorem \ref{calcithm}
(\ref{calcisharp}) in our situation, where we replace $i_{\text{\rm qm},(\frak b(\frak y),b(\frak y))}^{T, \ast}$
by $i_{\text{\rm qm},(\frak b(a(T^{\rho})),b(a(T^{\rho})))}^{\ast}$.
\begin{lem}\label{supreheavycubic}
For each $\text{\bf u} \in  \frak Z_i$, there exist
$\frak b(\text{\bf u})$ and
$e(\text{\bf u}) \in QH_{\frak b(\text{\bf u})}(X;\Lambda)$ such that:
\begin{enumerate}
\item
$e(\text{\bf u})\Lambda \subset QH_{\frak b(\text{\bf u})}(X;\Lambda)$
is a direct factor.
\item
$T(\text{\bf u})$ is $\zeta_{e(\text{\bf u})}^{\frak b(\text{\bf u})}$-superheavy.
\end{enumerate}
\end{lem}
Once we have Lemma \ref{supreheavycubic} for some $i$,
then by symmetry we obtain the same conclusion for the other two of $i=1,2,3$.
The proof of
Theorem \ref{cubicmain} is now completed.
\end{proof}

\section{Detecting spectral invariant via Hochschild cohomology}
\label{subsec:Hochschild}
In this section we prove the following theorem.
For a critical point $\frak y$ of the potential
function $\frak{PO}_{\frak b}$ we recall the subset
$\text{\rm Jac}(\frak{PO}_{\frak b};\frak y)
\subset \text{\rm Jac}(\frak{PO}_{\frak b};\Lambda)$
from Definition \ref{def:Jaceigensp}.
Corresponding to this subspace,
we put
$$
QH_{\frak b}(M;\Lambda;\frak y)
:=
\{
s \in QH_{\frak b}(M;\Lambda)
\mid \frak{ks}_{\frak b}(s) \in \text{\rm Jac}(\frak{PO}_{\frak b};\frak y)
\}.
$$
\begin{thm}\label{appliHOch}
Let $(M,\omega)$ be a compact toric manifold and
$\frak b \in H^{{\rm even}}(M;\Lambda_0)$. Suppose that
$\frak y$ is a critical point of the potential function $\frak{PO}_{\frak b}$.
Let $\text{\bf u} = \text{\bf u}(\frak y)$ and let $b = b(\frak y)$ be
those defined as in Theorem $\ref{Floercrit}$.
Denote by $e_{\frak y}$ the corresponding idempotent of $QH_{\frak b}(M,\omega;\frak y)$.
Then
$L(\text{\bf u})$ is $\mu_{e_{\frak y}}^{\frak b}$-superheavy.
\end{thm}
\index{superheavy subset!$\mu$-superheavy}
This theorem improves  Theorem $\ref{toricheavymain}$ in that superheaviness holds
without assuming nondegeneracy of $\frak y$.
\begin{prob}
Let $\frak y$ be a degenerate critical point of $\frak{PO}_{\frak b}$.
When does $\mu_{e_{\frak y}}^{\frak b}$ become a quasi-morphism?
\end{prob}

\subsection{Hochschild cohomology of filtered $A_{\infty}$ algebra: review}
\label{subsec:Hochschild}
\index{Hochschild cohomology!of filtered $A_{\infty}$ algebra}

We will use Hochschild cohomology in the proof of Theorem \ref{appliHOch}.
Let $(C,\{\frak m_k\}_{k=0}^{\infty})$ be a unital and gapped filtered $A_{\infty}$ algebra.
(See \cite[Section 3]{fooo:book1}  for the definition of filtered $A_{\infty}$
algebra etc.)
In the situation of Theorem \ref{appliHOch}, we have
$\frak m_0(1) = {\frak {PO}}(b)\cdot {\bf e}$
that is not zero in general.
(Here $\bf e$ is the unit of $(C,\{\frak m_k\}_{k=0}^{\infty})$.)
We {\it redefine} the $A_\infty$ structure by setting $\frak m_0(1) =0$
but not changing other operators.
By the unitality of the given $A_\infty$ algebra,
we find that the $A_{\infty}$ relations still hold.
After this modification
we will assume
$$\frak m_0 = 0
$$
in this section.

We consider the Hochschild cochain modules
\begin{equation}
\aligned
CH^k(C,C) &= Hom_{\Lambda}(B_kC[1],C[1]), \\
CH(C,C) &= \prod_{k=0}^{\infty} CH^k(C,C),\\
\mathcal N^kCH(C,C)
&=
CH(C,C) / \prod_{k'> k} CH^{k'}(C,C).
\endaligned
\end{equation}
\par
We define Hochschild differential $\delta_H :
CH(C,C) \to CH(C,C)$ by
\begin{equation}
\aligned
\delta_H(\varphi)(x_1,\dots,x_k)
= &\sum_{i,\ell} (-1)^{*_1}\varphi(x_1,\dots,\frak m_{\ell}(x_i,\dots),\dots) \\
&+ \sum_{i,\ell} (-1)^{*_2}\frak m_{\ell}(x_1,\dots,\varphi(x_i,\dots),\dots),
\endaligned\end{equation}
where
$*_1 = \deg\varphi(\deg x_1 + \dots + \deg x_{i-1} + i - 1)$,
$*_2 = \deg x_1 + \dots + \deg x_{i-1} + i$.
It is easy to check $\delta_H \circ \delta_H = 0$. So
$(CH(C,C),\delta_H)$ is a (co)chain complex,
which we call {\it Hochschild cochain complex}.
\par
By our assumption $\frak m_0 = 0$, we have
$$
\delta_H\left(\prod_{k'> k} CH^{k'}(C,C)\right)
\subset \prod_{k'> k} CH^{k'}(C,C).
$$
Therefore the map $\delta_H$ descends to a map
$\delta_H : \mathcal N^kCH(C,C) \to \mathcal N^kCH(C,C)$ in the quotient.
We call this filtration the {\it number filtration}.
Note we have
$$
(CH(C,C),\delta_H) = \projlim_{k\to \infty} (\mathcal N^kCH(C,C),\delta_H).
$$
\par
We assume that $(C,\{\frak m_k\}_{k=0}^{\infty})$ has a
strict unit $\text{\bf e}$.
We define its submodule
$$
CH^{\text{\rm red},k}(C,C)
= \{\varphi \in CH^k(C,C) \mid \varphi(\cdots,\text{\bf e},\cdots) = 0\}
$$
and define
$CH^{\text{\rm red}}(C,C)$,
$\mathcal N^k CH^{\text{\rm red}}(C,C)$ in a similar way.
We can easily show
$
\delta_H\left(CH^{\text{\rm red}}(C,C)\right)
\subset CH^{\text{\rm red}}(C,C)
$
which enables us to define {\it reduced Hochschild cochain complex}
$(CH^{\text{\rm red}}(C,C),\delta_H)$.
The latter complex carries a number filtration.
The cohomology of the reduced Hochschild cochain complex
is written as
$HH^{\text{\rm red}}(C,C)$ and is called {\it reduced Hochschild cohomology}.
\par
We note that
$$
CH^k(C,C) \cong Hom_{\C}(B_k\overline C[1],\overline C[1])\otimes \Lambda,
$$
where $\overline C \otimes_{\C}\Lambda = C$.
We can then use the filtration
$F^{\lambda}C = \overline C \otimes_{\C}T^{\lambda}\Lambda_0$
to define another filtration $F^{\lambda}CH^k(C,C)$ on $CH^k(C,C)$.
We call this filtration the {\it energy filtration}.
Using the condition
$
\frak m_k \left(
F^{\lambda_1}C \otimes \cdots \otimes F^{\lambda_k}C\right)
\subset
F^{\lambda_1 + \dots + \lambda_k}C
$
(\cite[(3.2.12.6)]{fooo:book1}), $\delta_H$ preserves the
energy filtration.

\subsection{From quantum cohomology to Hochschild cohomology}
\label{subsec:Hochschild2}

Let $L$ be a relative spin Lagrangian submanifold of a compact
symplectic manifold $(M,\omega)$ and $\frak b \in H^{{\rm even}}(M;\Lambda_0)$.
Let $b = b_0 + b_+\in \Omega^{odd}\otimes \Lambda_0$ be an element
satisfying the Maurer-Cartan equation
$$
\sum_{\beta}\sum_{k=0}^{\infty} T^{\omega \cap \beta} \exp(b_0 \cap \partial \beta)\frak m_{k,\beta}(b_+,\dots,b_+) = 0.
$$
For each such pair $(\frak b,b)$, we obtain a unital gapped filtered $A_{\infty}$ algebra
$(\Omega(L)\widehat{\otimes} \Lambda,\{\frak m_k^{\frak b,b}\})$.
We define
\begin{equation}
\frak q_*^{\bf b}
= H(M;\Lambda)
\to
CH^{\text{\rm red}}(\Omega(L)\widehat{\otimes} \Lambda,\Omega(L)\widehat{\otimes} \Lambda)
\end{equation}
as follows.
We put
$\text{\bf b} = (\frak b_0,\text{\bf b}_{2;1},\frak b_+,b_+)$
and define for $k\ne 0$:
\begin{equation}\label{mkdefeq2}
\aligned
&\frak q_{*}^{\text{\bf b}}(a)(x_1,\ldots,x_k) \\
&= \sum_{\beta\in H_2(M,L:\Z)}
\sum_{\ell_1=0}^{\infty}\sum_{\ell_2=0}^{\infty}\sum_{m_0=0}^{\infty}\cdots
\sum_{m_k=0}^{\infty}T^{\omega\cap \beta}
\frac{\exp(\text{\bf b}_{2;1} \cap \beta)}{(\ell_1+\ell_2+1)!}\\
&\frak
q_{\ell_1+\ell_2+1,k+\sum_{i=0}^k m_i;\beta}(\frak b_{+}^{\otimes\ell_1}
\otimes a \otimes \frak b_{+}^{\otimes\ell_2};
b_{+}^{\otimes m_0},x_1,b_{+}^{\otimes m_1},\ldots,
b_{+}^{\otimes m_{k-1}},x_k,b_{+}^{\otimes m_k}),
\endaligned
\end{equation}
where $x_i \in \Omega(L)$.
Here the notations are the same as in (\ref{mkdefeq}).
Recall that we assume $\frak m_0 =0$ in this section.
\par
Using Theorem \ref{qproperties} (1), we can show that
$\delta_H(\frak q_{*}^{\text{\bf b}}(a))= 0$.
Therefore we have a map \index{operator $\frak q$}
\begin{equation}
\frak q_{*}^{\bf b}
:H(M;\Lambda)
\to
HH_{\text{\bf b}}^{\text{\rm red}}(\Omega(L)\widehat{\otimes} \Lambda,\Omega(L)\widehat{\otimes} \Lambda)
\end{equation}
to the reduced Hochschild cohomology.
Here
$HH_{\text{\bf b}}^{\text{\rm red}}(\Omega(L)\widehat{\otimes} \Lambda,\Omega(L)\widehat{\otimes} \Lambda)$ is the
reduced Hochschild cohomology with respect to the filtered $A_{\infty}$
structure $\frak m^{\text{\bf b}}$.

\begin{rem}
\begin{enumerate}
\item
By composing $\frak q_{*}^{\text{\bf b}}$ with the projection
$$
HH^{\text{\rm red}}_{\text{\bf b}}(\Omega(L)\widehat{\otimes} \Lambda,\Omega(L)\widehat{\otimes} \Lambda)
\to
\mathcal N_0HH_{\text{\bf b}}^{\text{\rm red}}(\Omega(L)\widehat{\otimes} \Lambda,\Omega(L)\widehat{\otimes} \Lambda)
= HF^*((L,\text{\bf b});\Lambda),
$$
we obtain a map
$
H(M;\Lambda)
\to HF^*((L,\text{\bf b});\Lambda).
$
This coincides with the map
$i^{\ast}_{\text{\rm qm},HF((L,\text{\bf b});\Lambda)
}$ given in (\ref{ambcohtoHFL}).
\item
$HH^{\text{\rm red}}_{\text{\bf b}}(\Omega(L)\widehat{\otimes} \Lambda,\Omega(L)\widehat{\otimes} \Lambda)
$ has a filtered $A_{\infty}$ structure. (See \cite[4.7.4 and
Remark 4.7.1 (1)]{fooo:toricmir}.) Moreover
$\frak q_{*}^{\text{\bf b}}$ is a ring homomorphism. (See \cite{AFOOO} for the proof in general).
\end{enumerate}
\end{rem}

\subsection{Proof of Theorem \ref{appliHOch}}
\label{subsec:Hoshchildmainproof}
To prove Theorem \ref{appliHOch}, we need to explore some
estimates of the spectral invariant which are analogs of the ones developed in Parts 1-4.
In the previous parts we use $\Lambda^{\downarrow}$ coefficients
and the valuation $\frak v_q$, while we use $\Lambda$ coefficients
and the valuation $\frak v_T$ in Subsections \ref{subsec:Hochschild}, \ref{subsec:Hochschild2}.
To translate the valuation $\frak v_T$ for
any element $x$ defined over
$\Lambda$ into $\frak v_q$,
we just define
$$\frak v_q(x)=-\frak v_T(x),
$$ because $T=q^{-1}$. See Notations and Conventions (18).
We use this notation throughout this subsection.

The following is an analog of Proposition \ref{182mainprop}.
\begin{prop}\label{HHmainprop}
Let $L,\text{\bf b}$ be as above and $a \in H(M;\Lambda)$.
Then
$$
\rho^{\frak b}(H;a)
\ge - E^+(H;L) + \frak v_q(\frak q_{*}^{\text{\bf b}}(a)).
$$
\end{prop}
Here for $x \in HH^{\text{\rm red}}(\Omega(L)\widehat{\otimes} \Lambda,\Omega(L)\widehat{\otimes} \Lambda)$
we define
$$
\frak v_q(x) = -\frak v_T(x)
=-\sup\{\lambda
\mid
\exists \, \widetilde x \in F^{\lambda}CH^{\text{\rm red}}(\Omega(L)\widehat{\otimes} \Lambda,\Omega(L)\widehat{\otimes} \Lambda), x =[\widetilde x]
\}.
$$
\begin{proof}
Using the operator $\frak q_{\beta;\ell;k}^{H;[\gamma,w]}$ in
(\ref{1814}), we define
\begin{equation}
\frak q_*^{H,\text{\bf b}}
: CF(M;H;\Lambda)
\to
CH^{\text{\rm red}}_{\text{\bf b}}(\Omega(L)\widehat{\otimes} \Lambda,\Omega(L)\widehat{\otimes} \Lambda)
\end{equation}
by
\begin{equation}\label{mkdefeq3}
\aligned
&\frak q_{*}^{H,\text{\bf b}}(\llb \gamma,w \rrb)(x_1,\ldots,x_k) \\
&= \sum_{\beta\in H_2(M,L:\Z)}
\sum_{\ell=0}^{\infty}\sum_{m_0=0}^{\infty}\cdots
\sum_{m_k=0}^{\infty}T^{\omega\cap \beta}
\frac{\exp(\text{\bf b}_{2;1} \cap \beta)}{\ell!}\\
&\hskip1cm \frak
q^{H;[\gamma,w]}_{\ell,k+\sum_{i=0}^k m_i;\beta}(\frak b_{+}^{\otimes\ell};
b_{+}^{\otimes m_0},x_1,b_{+}^{\otimes m_1},\ldots,
b_{+}^{\otimes m_{k-1}},x_k,b_{+}^{\otimes m_k}).
\endaligned
\end{equation}
Here we take a lift $[\gamma,w] \in {\rm Crit}(\mathcal A_H)$ of
$\llb \gamma,w \rrb \in \widehat{\text{\rm Per}}(H)$ to define the right hand side
and can show that it is independent of the lift $[\gamma,w]$
by Proposition \ref{FBULKkura18} (11).
\par
Using Proposition \ref{prop:qH-relation}, we can  find that
it induces a map
$$
\frak q_{*}^{H,{\bf b}}
:
HF(M,H;\Lambda) \to HH_{\text{\bf b}}^{\text{\rm red}}(\Omega(L)\widehat{\otimes} \Lambda,\Omega(L)\widehat{\otimes} \Lambda).
$$
\begin{lem}\label{HHtrianglecomm}
$\frak q_*^{H,\text{\bf b}}\circ \CP_{(H_\chi,J)}^{\frak b}$
is chain homotopic $\frak q_*^{\text{\bf b}}$.
\end{lem}
The proof is the same as that of Proposition
\ref{184mainprop} and is omitted.
\par
We can use Lemma \ref{HHtrianglecomm}
to prove Proposition \ref{HHmainprop}
in the same way as we used
Proposition \ref{184mainprop}
to prove Proposition \ref{182mainprop}.
Thus Proposition \ref{HHmainprop} follows.
\end{proof}
\par
\smallskip
Now we restrict ourselves to the case of toric manifold $(M,\omega)$.
For the toric case we can use the $T^n$-equivariant operator $\frak q^T$ in place of $\frak q$
in Proposition \ref{HHmainprop} and the Hochschild complex
$
CH(H(L(\text{\bf u});\Lambda),H(L(\text{\bf u});\Lambda))
$
defined on de Rham cohomology group (instead on the
\emph{space of differential forms}).
In fact the $A_\infty$ structure $\frak m^T_*$ itself was defined thereon.
\par
Let $\frak b \in H^{{\rm even}}(M;\Lambda)$ and let
$\frak y$ be a critical point of
$\frak{PO}_{\frak b}$,
$\text{\bf u} = \text{\bf u}(\frak y)$ and
$b = b(\frak y)$ as in Theorem \ref{Floercrit}.
We put ${\text{\bf b}} = {\text{\bf b}}(\frak y) = (\frak b,b(\frak y))$.
\begin{lem}\label{infofq}
The restriction of the map
$$\frak q_*^{T,\text{\bf b}}
: QH_{\frak b}(M;\Lambda) \to
HH_{\text{\bf b}}^{\text{\rm red}}(H(L(\text{\bf u});\Lambda),
H(L(\text{\bf u});\Lambda))
$$ to
$QH_{\frak b}(M;\Lambda;\frak y)\subset QH_{\frak b}(M;\Lambda)$,
is injective.
\end{lem}
\begin{proof}
We recall from \cite[Lemma 4.7.5]{fooo:toricmir}  that we have the map
$$
HH(H(L(\text{\bf u});\Lambda),
H(L(\text{\bf u});\Lambda))
\to
\text{\rm Jac}(\frak{PO};\frak y).
$$
(More precisely, we constructed a
map from $HH(H(L(\text{\bf u});\Lambda),
H(L(\text{\bf u});\Lambda))$
to a formal power series version of
the Jacobian ring in \cite[Lemma 4.7.5]{fooo:toricmir}.
As we mentioned at the end of \cite[page 329]{fooo:toricmir}, this formal power series version
is isomorphic to the Jacobian ring
$\text{\rm Jac}(\frak{PO};\frak y)$, which is its adic convergent version.
The detail of the proof of this isomorphism
will be written in \cite{AFOOO2}.
It seems likely that this is known to the
expert of rigid analytic geometry.)
Its bulk deformation gives rise to a map
\begin{equation}\label{HHtoJac}
HH_{\text{\bf b}}(H(L(\text{\bf u});\Lambda),
H(L(\text{\bf u});\Lambda))
\to
\text{\rm Jac}(\frak{PO}_{\text{\bf b}};\frak y).
\end{equation}
Then the composition of the restriction of $\frak q_*^{T,\text{\bf b}}$
to $QH_{\frak b}(M;\Lambda;\frak y)$
with the canonical injective map
$
HH_{\text{\bf b}}^{\text{\rm red}}(H(L(\text{\bf u});\Lambda),
H(L(\text{\bf u});\Lambda))
\to
HH_{\text{\bf b}}(H(L(\text{\bf u});\Lambda),
H(L(\text{\bf u});\Lambda))
$
and
(\ref{HHtoJac}) induces an isomorphism
$QH_{\frak b}(M;\Lambda;\frak y) \cong
\text{\rm Jac}(\frak{PO}_{\frak b};\frak y)$.
Hence the lemma.
\end{proof}
Since
the image of the map
${\frak q}_*^{T,\text{\bf b}}$
is a finite dimensional vector space over $\Lambda$, we can apply the
argument of Subsection \ref{subsec:Usher}
to find a standard basis \index{standard basis}
${\frak q}_*^{T,\text{\bf b}}(e_1),\dots,{\frak q}_*^{T,\text{\bf b}}(e_k)$ of the image of
${\frak q}_*^{T,\text{\bf b}}$.
Then we have
\begin{equation}\label{259formu}
\aligned
\frak v_q\left(\frak q_*^{T,\text{\bf b}}\left(\sum_{i=1}^k x_i e_i\right)\right)
& = \max \{\frak v_q(x_i {\frak q}_*^{T,\text{\bf b}}(e_i)) \mid i=1,\dots,k \} \\
& \ge \max \{\frak v_q(x_i ) \mid i=1,\dots,k \} - C_1
\endaligned
\end{equation}
where $C_1$ is a constant independent of $x_i$.
\par
Now we are ready to complete the proof of
Theorem  \ref{appliHOch}.
Let
$\widetilde \psi_H \in \widetilde{\text{\rm Ham}}(M,\omega)$.
By Theorem \ref{dualitymain} we have
\begin{equation}\label{25100}
\rho^{\frak b}(\widetilde \psi_H^n;e_{\frak y})
=
- \inf \{
\rho^{\frak b}(\widetilde \psi_H^{-n};b)
\mid \Pi(e_{\frak y},b)\ne 0
\}.
\end{equation}
Let us estimate the right hand side of (\ref{25100}).
Suppose $\Pi(e_{\frak y},b) \ne 0$.
We put
$$
e_{\frak y}\cup^{\frak b}b = \sum_{i=1}^k x_i e_i,
\qquad x_i \in \Lambda.
$$
Since $\Pi(e_{\frak y},b) \ne 0$, Sublemma \ref{sublem166} implies
$
\frak v_q(e_{\frak y}\cup^{\frak b} b)\ge 0.
$
Therefore
\begin{equation}\label{2550}
\max\{\frak v_q(x_i) \mid i = 1,\dots, k\}
\ge C_2,
\end{equation}
where
$C_2 = -\max\{\frak v_q(e_i) \mid i = 1,\dots, k\}$.
\par
By the triangle inequality,
$$
\rho^{\frak b}(\widetilde \psi_H^{-n};b)
\ge
\rho^{\frak b}(\widetilde \psi_H^{-n};\sum_{i=1}^k x_ie_i)
- \rho^{\frak b}(\underline 0;e_{\frak y}).
$$
Here $\underline 0$ is the identity map $= \widetilde \psi_0$.
It follows from Proposition \ref{HHmainprop} that the right hand side is
not smaller than
$$
n \int_0^1 \min(- \widetilde H_t|_{L({\bf u})})\, dt +\frak v_q\left(\frak q_*^{T,\frak b}\left(\sum_{i=1}^k x_i e_i\right)\right)
-  \rho^{\frak b}(\underline 0;e_{\frak y}).
$$
Here $\tilde H$ is as in (\ref{tildeHHH}).
By (\ref{259formu}), this is not smaller than
$$
n \int_0^1 \min(H_t|_{L({\bf u})})\, dt +\max\{ \frak v_q(x_i) \mid i=1,\dots, k\}
- C_1 - \rho^{\frak b}(\underline 0;e_{\frak y}).
$$
Combining the above, we obtain
$$
\aligned
\rho^{\frak b}(\widetilde \psi_H^{-n};b)
\ge & \int_0^1 \min(H_t|_{L({\bf u})})\, dt + C_2 - C_1 - \rho^{\frak b}(\underline 0;e)\\
= & - n \int_0^1 \max(- H_t|_{L({\bf u})}) \, dt + C_2 - C_1 - \rho^{\frak b}(\underline 0;e) \\
= & - n E^-(H;L({\bf u})) + C_2 - C_1 - \rho^{\frak b}(\underline 0;e).
\endaligned
$$
Therefore using  \eqref{25100}, we obtain
$$
-\frac{\rho^{\frak b}(\widetilde \psi_H^n;e_{\frak y})}{n}
\ge - E^-(H;L({\bf u}))
- \frac{C_3}{n},
$$
where $C_3$ is independent of $n$.
By letting $n \to \infty$, we
have finished the proof of Theorem \ref{appliHOch}.
\qed

\subsection{A remark}
\label{subsec:Hochschildremark}

In Theorems \ref{toricheavymain}
and  \ref{appliHOch}, we use Lagrangian Floer
theory to estimate the spectral invariant in terms
of the values of the Hamiltonian on the Lagrangian submanifolds.
One can use a variant of this technique to obtain an estimate
of spectral invariant using various other invariant appearing in symplectic topology.
\index{$E^+(H;\mathcal L(Y))$}
\par
We introduce the invariant $E^+(H;\mathcal L(Y))$ defined by
\begin{equation}\label{defSS}
E^+(H;\mathcal L(Y)) : = \sup \left\{\int_0^1  H(t,\gamma(t))\, dt
~\vert~ \gamma \in \mathcal L(Y)\right\}.
\end{equation}
This is an invariant stronger than $E^+(H;Y)$ in that $E^+(H;\CL(Y)) \leq E^+(H;Y)$ and
more directly related to the loop space $\CL(Y)$ of $Y$.
Using this invariant, we can improve the statement of
Proposition \ref{182mainprop} to the following
\be\label{eq:loopE-HY}
\rho^{\frak b}(H;a) \geq -E^+(H;\CL(Y)) + \rho^{\mathbf b}_L(a).
\ee
This formula suggests that we may use symplectic homology $SH(V)$
(\cite{floerhofer}) of a subset $V\subset M$ and
the Viterbo functoriality (Viterbo \cite{viterbo2}, Abouzaid-Seidel \cite{abousaiseidel})
in place of Lagrangian Floer theory in certain cases, for example, in the case where $V$ is a
Darboux-Weinstein neighborhood of a Lagrangian submanifold $L$.
For the case where the Floer homology $HF(L)$ is isomorphic to $H(L)$
(such as the case $L$ is exact), the symplectic homology $SH(V)$ is related to the homology of
the loop space of $L$ (Salamon-Weber \cite{salamonweber}, Abbondandolo-Schwartz \cite{abbsch}, which is
in turn closely related to the Hochschild cohomology of $H(L)$.
(See also \cite{fukaya;loop}).)
Thus in that case the method using symplectic
homology becomes equivalent to those
using Hochschild cohomology that we have established
in this section.
\par
Eliashberg-Polterovich \cite{eliashpol} use
symplectic homology to estimate the spectral invariant
in the case of Lagrangian tori in $S^2\times S^2$.
Through the above mentioned equivalence, their argument is related to ours given in Section
\ref{sec:exotic}.

\part{Appendix}
\section{$\CP_{(H_\chi,J_\chi),\ast}^{\frak b}$ is an isomorphism}
\label{sec:appendix1} \index{Piunikhin isomorphism}

In Section \ref{sec:deform-bdy} we introduced the Piunikhin map
$\mathcal P_{(H_\chi,J_\chi),\ast}^{\frak b}$ with bulk deformation.
In this section we complete the proof of Theorem \ref{Pbulkiso}:

\begin{thm}\label{appendmain} The Piunikhin map with bulk deformation
$$
\mathcal P_{(H_\chi,J_\chi),\ast}^{\frak b}
: H_*(M;\Lambda^{\downarrow})
\to
HF^{\frak b}_*(M,H,J;\Lambda^{\downarrow})
$$
is a $\Lambda^{\downarrow}$-module isomorphism.
\end{thm}
\begin{proof}
We first construct another map
\index{$\mathcal Q_{(H_{\tilde\chi},J_{\tilde\chi}),\ast}^{\frak b}$}
\begin{equation}\label{Qdomainandtarget}
\mathcal Q_{(H_{\tilde\chi},J_{\tilde\chi}),\ast}^{\frak b}
: HF^{\frak b}_*(M,H;\Lambda^{\downarrow})
\to H_*(M;\Lambda^{\downarrow})
\end{equation}
in the direction opposite to $\mathcal P_{(H_\chi,J_\chi),\ast}^{\frak b}$.
(Here and hereafter we sometimes omit $J$ in the notation of $HF^{\frak b}_*(M,H,J;\Lambda^{\downarrow})$.)
This will be
carried out by constructing the associated chain map
\begin{equation}
CF(M,H;\Lambda^{\downarrow})
\to \Omega(M) \widehat{\otimes} \Lambda^{\downarrow}.
\end{equation}
Let $\chi \in \CK$ be as in Definition \ref{defn:chi} and $[\gamma,w]
\in \text{\rm Crit}(\mathcal A_H)$. For the construction of this chain map, we need to consider
the dual version of $\chi$. To distinguish the two different types of
elongation functions, we recall that we denote
$$
\tilde \chi(\tau)= -\chi(1-\tau)
$$
for $\chi \in \CK$.
We also use $(H_{\chi},J_{\chi})$ defined in (\ref{eq:paraHJ}).
(In this section $J=\{J_t\}$ is a $t \in S^1$ parametrized family of compatible almost complex
structures.)

We consider the elongated family
$(H_{\tilde \chi},J_{\tilde \chi})$ defined by:
$$
H_{\tilde \chi}(\tau,t,x) = \tilde \chi(\tau) H_t(x), \quad J_{\tilde \chi}(\tau,t) = J_{-\tilde \chi(\tau),t},
$$
where $J_{s,t}$ is as in (\ref{HsJs}).

\begin{defn}\label{moduliforQ}
\index{$\mathcal M_{\ell}(H_{\tilde \chi},J_{\tilde \chi};[\gamma,w],*)$}
We denote by $\overset{\circ}{{\CM}}_{\ell}(H_{\tilde \chi},J_{\tilde \chi};[\gamma,w],*)$ the set of all
pairs
$(u;z_1^+,\dots,z_{\ell}^+)$
of maps
$u: \R \times S^1 \to M$ and
$z_i^+\in \R \times S^1$ which satisfy the following conditions:
 \begin{enumerate}
 \item The map $u$ satisfies the equation:
\be\label{eq:HJCRQ}
\dudtau + J_{\tilde \chi}\Big(\dudt - X_{H_{\tilde \chi}}(u)\Big) = 0.
\ee
\item The energy
$$
E_{(H_{\tilde \chi},J_{\tilde \chi})}(u) = \frac{1}{2} \int \Big(\Big|\dudtau\Big|^2_{J_{\tilde \chi}} + \Big|
\dudt - X_{H_{\tilde \chi}}(u)\Big|_{J_{\tilde \chi}}^2 \Big)\, dt\, d\tau
$$
is finite.
\item The map $u$ satisfies the following asymptotic boundary condition.
$$
\lim_{\tau\to -\infty}u(\tau, t) = \gamma(t).
$$
\item The homology class of the concatenation of $u$ and $w$
is homotopic to $0$.
\item $z_i^+$ are all different from one another.
\end{enumerate}
\end{defn}
$(u;z^+_1,\dots,z^+_{\ell}) \mapsto (u(z^+_1),\dots,u(z^+_{\ell}))$
defines an evaluation map
$$
{\rm ev} = ({\rm ev}_1\dots,{\rm ev}_{\ell}) = \overset{\circ}{\CM}_{\ell}(H_{\tilde \chi},J_{\tilde \chi};[\gamma,w],*) \to M^{\ell}.
$$

\begin{lem}\label{QBULKkura}
\begin{enumerate}
\item The moduli space
$\overset{\circ}{\CM}_{\ell}(H_{\tilde \chi},J_{\tilde \chi};[\gamma,w],*)$ has a compactification
${{\CM}}_{\ell}(H_{\tilde \chi},J_{\tilde \chi};[\gamma,w],*)
$
that is Hausdorff.
\item
The space ${{\CM}}_{\ell}(H_{\tilde \chi},J_{\tilde \chi};[\gamma,w],*)$ has an orientable Kuranishi structure with corners.
\item
The normalized boundary of ${{\CM}}_{\ell}(H_{\tilde \chi},J_{\tilde \chi};[\gamma,w],*)$ is described by
\begin{equation}\label{bdryQmodu}
\aligned
&\partial {\CM}_{\ell}(H_{\tilde \chi},J_{\tilde \chi};[\gamma,w],*) \\
&= \bigcup
{\CM}(H,J;[\gamma,w],[\gamma',w'])
\times
{\CM}(H_{\tilde \chi},J_{\tilde \chi};[\gamma',w'],*),
\endaligned\end{equation}
where the union is taken over all $[\gamma',w'] \in \text{\rm Crit}(\mathcal A_H)$
and $(\mathbb L_1,\mathbb L_2) \in \text{\rm Shuff}(\ell)$.
\item
Let $\mu_H : \mbox{\rm Crit}(\CA_H) \to  \Z$,
be the Conley-Zehnder index. Then the (virtual) dimension satisfies
the following equality:
\begin{equation}\label{dimensionboundarybQ2}
\dim {\CM}_{\ell}(H_{\tilde \chi},J_{\tilde \chi};[\gamma,w],*) = n-\mu_H([\gamma,w]) +2\ell.
\end{equation}
\item
We can define a system of orientations on ${\CM}_{\ell}(H_{\tilde \chi},J_{\tilde \chi};[\gamma,w],*)$ so that
the isomorphism $(3)$ above is compatible with this orientation.
\item
${\rm ev}$ extends to a strongly continuous smooth map
${\CM}_\ell(H_{\tilde \chi},J_{\tilde \chi};[\gamma,w],*) \to M^{\ell}
$, which we denote also by ${\rm ev}$.
It is compatible with $(3)$.
\item
The map $\text{\rm ev}_{+\infty}$ which sends $(u;z^+_1,\dots,z^+_{\ell})$ to
$
\lim_{\tau\to +\infty}u(\tau, t)
$
extends to a weakly submersive map ${\CM}_\ell(H_{\tilde \chi},J_{\tilde \chi};[\gamma,w],*) \to M$,
which we also denote by ${\rm ev}_{+\infty}$.
It is compatible with $(3)$.
\end{enumerate}
\end{lem}
The proof of Lemma \ref{QBULKkura} is the same as that of Proposition \ref{connkura},
which is detailed in \cite[Parts 4 and 5]{fooo:techI}, and
so omitted.
\par
We take a CF-perturbation $\widehat{\frak S}$
on the moduli space
$
{\CM}_\ell(H_{\tilde \chi},J_{\tilde \chi};[\gamma,w],*)
$
which is compatible with (3) and such that $\text{\rm ev}_{+\infty}$ is strongly submersive
with respect to $\widehat{\frak S}$.
\par
Let $h_1,\dots,h_{\ell} \in \Omega(M)$.
We define
$
\frak n_{(H,J),*}([\gamma,w ])(h_1,\dots,h_{\ell})
\in \Omega(M)
$
by
\begin{equation}\label{Qdefev}
\frak n_{(H,J),*}([\gamma,w])(h_1,\dots,h_{\ell})
=
(\text{\rm ev}_{+\infty})!
\left(
{\rm ev}_1^*h_1 \wedge\dots \wedge {\rm ev}_\ell^*h_\ell
; \widehat{\frak S^{\epsilon}}
\right).
\end{equation}
Here $(\text{\rm ev}_{+\infty})!$ is the
integration along the fibers of the map $\text{\rm ev}_{+\infty}$
which is defined by using a CF-perturbation $\widehat{\frak S}$
of the space  ${\CM}_\ell(H_{\tilde \chi},J_{\tilde \chi};[\gamma,w],*)$.
(See Definition \ref{def320222}, \cite[Definition 7.78]{fooo:tech2}.)
\par
Let $\frak b \in H^{{\rm even}}(M;\Lambda^\downarrow_0)$. We split
$\frak b = \frak b_0 + \frak b_2 + \frak b_{+}$ as in
(\ref{decompb}).
We take closed forms which represent $\frak b_0$, $\frak b_2$, $\frak b_{+}$
and write them by the same symbols.
\begin{defn}\label{bdrycoefbQQ0}
Let $\llb \gamma,w \rrb \in \widehat{\text{\rm Per}}(H)$. We define
\begin{equation}\label{bdrycoefbQQ}
\aligned
&\mathcal Q^\frak b_{(H_{\tilde \chi},J_{\tilde \chi})}(\llb \gamma,w \rrb)
=
\sum_{\alpha \in \pi_2(M)}\sum_{\ell=0}^{\infty}\frac{\exp({\int (\alpha\# w)^* \frak b_2})}{\ell!}
\\
&\qquad\qquad\qquad\qquad\qquad
q^{
-\int w^* \omega +\int (\alpha\# w)^* \omega}
\frak n_{(H,J),*}([\gamma,\alpha\# w])(\underbrace{\frak b_{+},
\dots,\frak b_{+}}_{\ell}).
\endaligned
\end{equation}
Note $\{[\alpha\# w] \mid \alpha \in \pi_2(M)\} = \pi_2(\gamma)$.
We can use it to show that the right hand side is independent of the choice
of $[\gamma,w] \in \text{\rm Crit}(\mathcal A_H)$.
\end{defn}
We can prove that the sum in (\ref{bdrycoefbQQ}) converges in
$q$-adic topology in the same way as in Lemma \ref{adiccomv1}.
We have thus defined (\ref{Qdomainandtarget}). Then
\begin{equation}
\partial \circ \mathcal Q^\frak b_{(H_{\tilde \chi},J_{\tilde \chi})}
= \mathcal Q^\frak b_{(H_{\tilde \chi},J_{\tilde \chi})} \circ \partial_{(H,J)}^{\frak b}
\end{equation}
is a consequence of Lemma \ref{QBULKkura} (3),
Stokes' theorem (Theorem \ref{them48}, \cite[Theorem 8.11]{fooo:tech2})
and composition formula (Theorem \ref{compform},  \cite[Theorem 10.20]{fooo:tech2}).
(Here $\partial$ is defined by (\ref{eq:deRhamchain}).)
\begin{prop}\label{QP-1}
$\mathcal Q^\frak b_{(H_{\tilde \chi},J_{\tilde \chi})}\circ \mathcal P^\frak b_{(H_\chi,J_\chi)}$
is chain homotopic to the identity.
\end{prop}
\begin{proof}
For $S \in [1,\infty)$ define $H^S_{\chi}$ as follows:
\begin{equation}
H^S_{\chi}(\tau,t,x)
=\begin{cases}
 \chi(\tau+S+1)H_t(x)  & S\ge 1, \tau \le 0 \\
 \widetilde\chi(\tau-S-1)H_t(x)  & S\ge 1, \tau \ge 0.
\end{cases}
\end{equation}
We extend it to $S\in [0,1]$ by
\begin{equation}
H^S_{\chi}(\tau,t,x) = SH^1_{\chi}(\tau,t,x).
\end{equation}
The function $H^S_{\chi}$ may not be smooth on $S$ at $S=1$, $\tau \in [-10,10]$.
We modify it on a neighborhood of $S=1$, $\tau \in [-10,10]$
so that it becomes smooth and denote it by the same symbol.
We define $(S,\tau,t) \in [0,\infty) \times \R \times [0,1]$ parametrized family
of compatible almost complex structures $J^S_{\chi}$ as follows.
For $S \in [1,\infty)$ we put
\begin{equation}
J^S_{\chi}(\tau,t) =
\begin{cases}
J_{ \chi(\tau+S+1),t} & S\ge 1, \tau \le 0, \\
J_{\tilde \chi(\tau-S-1),t} & S\ge 1, \tau \ge 0.
\end{cases}
\end{equation}
We extend it to $S \in [0,1]$ so that the following is satisfied.
\begin{equation}
J^S_{\chi}(\tau,t) =
\begin{cases}
J_0 & \tau \le -10, \\
J_0 & \tau \ge +10, \\
J_0 & S=0, \\
J_0 & \text{$t$ is in a neighborhood of $[1]$.}
\end{cases}
\end{equation}
\begin{defn}\label{moduliforH1}
Let $C \in \pi_2(M)$.
For each $0 \leq S < \infty$, we denote by $\overset{\circ}{{\CM}}_{\ell}(H^S_{\chi},J^S_{\chi};*,*;C)$ the set of all
pairs $(u;z_1^+,\dots,z_{\ell}^+)$ of maps
 $u: \R \times S^1 \to M$,
$z_i^+\in \R \times S^1$ which satisfy the following conditions:
 \begin{enumerate}
 \item The map $u$ satisfies the equation:
\be\label{eq:HJCRHH}
\dudtau + J^S_{\chi}\Big(\dudt - X_{H^S_{\chi}}(u)\Big) = 0.
\ee
\item The energy
$$
\frac{1}{2} \int \Big(\Big|\dudtau\Big|^2_{J^S_{\chi}} + \Big|
\dudt - X_{H^S_{\chi}}(u)\Big|_{J^S_{\chi}}^2 \Big)\, dt\, d\tau
$$
is finite.
\item The homotopy class of $u$ is $C$.
\item $z_i^+$ are all distinct each other.
\end{enumerate}
\end{defn}
We note that (\ref{eq:HJCRHH}) and the finiteness of energy imply that
there exist $p_1, p_2 \in M$ such that
\begin{equation}\label{p1p2Hdef}
\lim_{\tau\to - \infty}u(\tau, t) = p_1,
\qquad
\lim_{\tau\to +\infty}u(\tau, t) = p_2.
\end{equation}
Therefore the homology class of $u$ is well-defined.
We define the evaluation map
$$
({\rm ev}_{-\infty},{\rm ev}_{+\infty})
: \overset{\circ}{{\CM}}_{\ell}(H^S_{\chi},J^S_{\chi};*,*;C)
\to M^2
$$
by
$
({\rm ev}_{-\infty},{\rm ev}_{+\infty})(u) = (p_1,p_2),
$
where $p_1,p_2$ are as in (\ref{p1p2Hdef}).
\par
We put
\index{$\mathcal M_{\ell}(para;H_\chi,J_\chi;*,*;C)$}
\begin{equation}\label{266below}
\overset{\circ}{{\CM}}_{\ell}(para;H_\chi,J_\chi;*,*;C)
=
\bigcup_{S\ge 0}
\{S\} \times
\overset{\circ}{{\CM}}_{\ell}(H^S_{\chi},J^S_{\chi};*,*;C),
\end{equation}
where ${\rm ev}$, ${\rm ev}_{-\infty}$ and ${\rm ev}_{+\infty}$ are defined on it.
\par
To describe the boundary of the compactification of
$\overset{\circ}{{\CM}}_{\ell}(para;H_\chi,J_\chi;*,*;C)$ we define
another moduli space.
\begin{defn}
We denote by $\widehat{\overset{\circ}{{\CM}}}_{\ell}(H=0,J_0;*,*;C)$
the set of all $(u;z_1^+,\dots,z_{\ell}^+)$ that
satisfy $(1)$,\dots,$(4)$ of Definition \ref{moduliforH1}
with $S=0$.
\end{defn}
Note that $H$ actually does not appear in
$(1)$,\dots,$(4)$ of Definition \ref{moduliforH1}
in case $S=0$.
There exists an $\R \times S^1$ action on
$\widehat{\overset{\circ}{{\CM}}}_{\ell}(H=0,J_0;*,*;C)$
that is induced by the $\R \times S^1$
action on $\R \times S^1$, the source of the map $u$.
In fact, the equation (\ref{eq:HJCRHH}) is preserved by
$\R \times S^1$ action in case $S=0$.
See \cite[Definition 28.4]{fooo:techI}  for the definition of
$S^1$ equivariant Kuranishi structure and \cite[Section 30]{fooo:techI}
for its construction.
\par
We define the evaluation maps
\begin{equation}\label{eq:evell}
\text{\rm ev} =
(\text{\rm ev}_1,\dots,\text{\rm ev}_{\ell})
:
\widehat{\overset{\circ}{{\CM}}}_{\ell}(H=0,J_0;*,*;C)\to M^{\ell}
\end{equation}
and
\begin{equation}\label{eq:ev2}
(\text{\rm ev}_{+\infty},\text{\rm ev}_{-\infty})
: \widehat{\overset{\circ}{{\CM}}}_{\ell}(H=0,J_0;*,*;C)\to M^2
\end{equation}
in an obvious way. We put
$$
\aligned
\overset{\circ}{\CM}_{\ell}(H=0,J_0;*,*;C)
&=
\widehat{\overset{\circ}{\CM}}_{\ell}(H=0,J_0;*,*;C)/\R,
\\
\widehat{\overset{\circ}{\overline{\CM}}}_{\ell}(H=0,J_0;*,*;C)
&=
\widehat{\overset{\circ}{\CM}}_{\ell}(H=0,J_0;*,*;C)/S^1,
\\
\overset{\circ}{\overline{\CM}}_{\ell}(H=0,J_0;*,*;C)
&=
\widehat{\overset{\circ}{\CM}}_{\ell}(H=0,J_0;*,*;C)/(\R\times S^1).
\endaligned
$$
Then
$\widehat{\overset{\circ}{{\CM}}}_{\ell}(H=0,J_0;*,*;C)$,
$\overset{\circ}{\CM}_{\ell}(H=0,J_0;*,*;C)$,
$\widehat{\overset{\circ}{\overline{\CM}}}_{\ell}(H=0,J_0;*,*;C)$
and
$\overset{\circ}{\overline{\CM}}_{\ell}(H=0,J_0;*,*;C)$
can be compactified.
We denote the corresponding compactifications
by $\widehat{{{\CM}}}_{\ell}(H=0,J_0;*,*;C)$,
${\CM}_{\ell}(H=0,J_0;*,*;C)$,
$\widehat{\overline{\CM}}_{\ell}(H=0,J_0;*,*;C)$ and
${\overline{\CM}}_{\ell}(H=0,J_0;*,*;C)$, respectively.
The compactifications are obtained as follows.
Fix an identification of $\R \times S^1$ with $\C P ^1 \setminus \{N, S\}$, where
$N, S$ are the limits as $\tau \to \pm \infty$, respectively.
For each $(u;z^+_1,\dots,z^+_{\ell}) \in \widehat{\overset{\circ}{{\CM}}}_{\ell}(H=0,J_0;*,*;C)$,
we regard $u$ as a map from $\C P^1$ and consider its graph in $\C P^1 \times M$.
Then we identify the space $\widehat{\overset{\circ}{{\CM}}}_{\ell}(H=0,J_0;*,*;C)$ with
the space $\overset{\circ}{\mathcal N}_{\ell}(H=0,J_0;*,*;C)$ of their graphs.
Take its stable map compactification ${\mathcal N}_{\ell}(H=0,J_0;*,*;C)$, which is identified with
$\widehat{{{\CM}}}_{\ell}(H=0,J_0;*,*;C)$.  (The component, which has degree $1$ to
$\C P^1$-factor  is the component with a parametrized solution of \eqref{eq:HJCRQ}.)
The group $\R \times S^1$ acts on the first factor of $\C P^1 \times M$
and induces an action on ${\mathcal N}_{\ell}(H=0,J_0;*,*;C)$.
By taking the quotient of ${\mathcal N}_{\ell}(H=0,J_0;*,*;C)$ by $\R$, $S^1$, $\R \times S^1$,
we obtain the compactification ${\CM}_{\ell}(H=0,J_0;*,*;C)$,
$\widehat{\overline{\CM}}_{\ell}(H=0,J_0;*,*;C)$ and
${\overline{\CM}}_{\ell}(H=0,J_0;*,*;C)$, respectively.
Each of them carries a Kuranishi structure and
evaluation maps that extend to its compactification.
We note that
${\overline{\CM}}_{\ell}(H=0,J_0;*,*;C)$
is identified with ${{\CM}}^{\text{cl}}_{\ell+2}(C)$
which is introduced in Section \ref{sec:bigquantum}.
\begin{lem}\label{HBULKkurapara}
\begin{enumerate}
\item
The moduli space $\overset{\circ}{{\CM}}_{\ell}(para;H_\chi,J_\chi;*,*;C)$ has a compactification
${{\CM}}_{\ell}(para;H_\chi,J_\chi;*,*;C)$ that is Hausdorff.
\item
The space ${{\CM}}_{\ell}(para;H_\chi,J_\chi;*,*;C)$ has an orientable Kuranishi structure with corners.
\item
The normalized boundary of ${{\CM}}_{\ell}(para;H_\chi,J_\chi;*,*;C)$ is described by
the union of following four types of direct or fiber products:
\begin{enumerate}
\item[(i)]
\begin{equation}\label{bdryQBULKkurapara1}
{{\CM}}_{\#\mathbb L_1}(H_\chi,J_\chi;*,[\gamma,w])
\times
 {{\CM}}_{\#\mathbb L_2}(H_{\widetilde{\chi}},J_{\widetilde{\chi}};[\gamma,w'],*)
\end{equation}
where the union is taken over all $[\gamma,w] \in \text{\rm Crit}(H)$,
and $(\mathbb L_1,\mathbb L_2) \in \text{\rm Shuff}(\ell)$. Here the bounding disc
$w'$ is defined by $[C\# w] = [w']$.
(Figure $\ref{Figure261}$)
\item[(ii)]
\begin{equation}\label{bdryQBULKkurapara2}
{{\CM}}_{\#\mathbb L_1}(H=0,J_0;*,*;C_1){}_{\text{\rm ev}_{+\infty}}\times_{\text{\rm ev}_{-\infty}}
  {{\CM}}_{\#\mathbb L_2}(para;H_\chi,J_\chi;*,*;C_2)
\end{equation}
where the union is taken over all $C_1,C_2\in \pi_2(M)$
and $(\mathbb L_1,\mathbb L_2) \in \text{\rm Shuff}(\ell)$ such that
$C_1 + C_2 = C$.
The fiber product is taken over $M$.
(Figure $\ref{Figure262}$)
(See Definition $\ref{defnn325}$ for the definition of fiber product.)
\item[(iii)]
\begin{equation}\label{bdryQBULKkurapara3}
{{\CM}}_{\#\mathbb L_1}(para;H_\chi,J_\chi;*,*;C_1){}_{\text{\rm ev}_{+\infty}}\times_{\text{\rm ev}_{-\infty}}
{{\CM}}_{\#\mathbb L_2}(H=0,J_0;*,*;C_2)
\end{equation}
where the union is taken over all $C_1,C_2 \in \pi_2(M)$
and $(\mathbb L_1,\mathbb L_2) \in \text{\rm Shuff}(\ell)$ such that
$C_1 + C_2 = C$.
The fiber product is taken over $M$.
(Figure $\ref{Figure263}$)
\item[(iv)] And
\begin{equation}\label{bdryQBULKkurapara4}
\widehat{{\CM}}_{\#\mathbb L}(H=0,J_0;*,*;C).
\end{equation}
\end{enumerate}
\item
The (virtual) dimension is given by
\begin{equation}\label{dimensionHHboundaryb223}
\dim {{\CM}}_{\ell}(para;H^\chi,J^\chi;*,*;C) = 2 c_1(M)\cap C + 2n +2\ell -1.
\end{equation}
\item
We can define a system of orientations on ${{\CM}}_{\ell}(para;H_\chi,J_\chi;*,*;C)$
that is compatible with the isomorphism $(3)$ above.
\item
The map ${\rm ev}$ \eqref{eq:evell} extends to a weakly submersive map
$
{{\CM}}_{\ell}(para;H_\chi,J_\chi;*,*;C) \to M^{\ell}
$, which we denote also by ${\rm ev}$.
It is compatible with  $(3)$.
\item
The map $({\rm ev}_{-\infty},{\rm ev}_{+\infty})$ \eqref{eq:ev2} also extends to a weakly submersive map
$$
{{\CM}}_{\ell}(para;H_\chi,J_\chi;*,*;C) \to M^2,
$$ which we denote by $({\rm ev}_{-\infty},{\rm ev}_{+\infty})$.
It is compatible with $(3)$.
\end{enumerate}
\end{lem}
\begin{figure}[h]
\centering
\includegraphics[scale=0.3]
{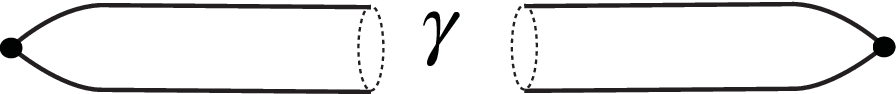}
\caption{An element of (\ref{bdryQBULKkurapara1})}
\label{Figure261}
\end{figure}
\begin{figure}[h]
\centering
\includegraphics[scale=0.3]
{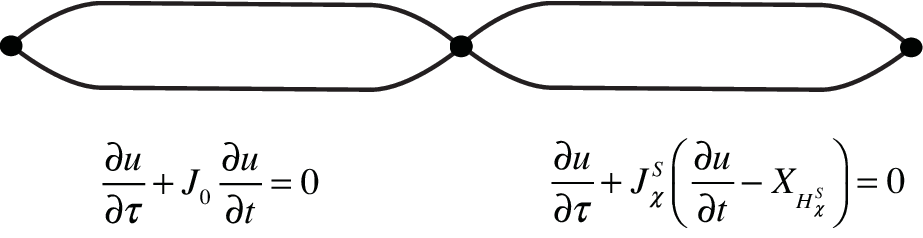}
\caption{An element of (\ref{bdryQBULKkurapara2})}
\label{Figure262}
\end{figure}
\begin{figure}[h]
\centering
\includegraphics[scale=0.3]
{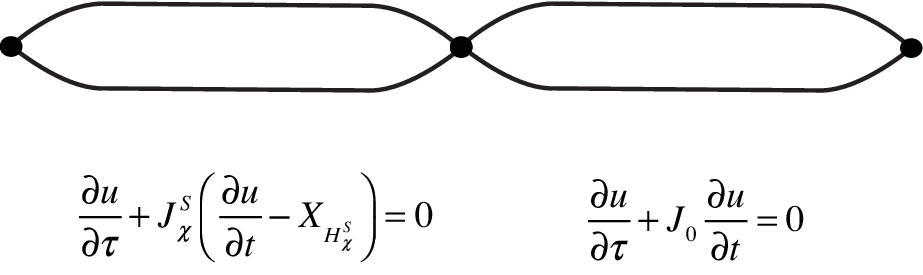}
\caption{An element of (\ref{bdryQBULKkurapara3})}
\label{Figure263}
\end{figure}
\begin{proof}
The proof is similar to the proofs of
Propositions \ref{piuBULKkura}, \ref{connkura} (See \cite[Parts 4 and 5]{fooo:techI}).
So we only mention the way how the four different types of boundary components
appear.
In fact, (\ref{bdryQBULKkurapara1}) appears when $S \to \infty$,
(\ref{bdryQBULKkurapara4}) appears when $S = 0$.
(\ref{bdryQBULKkurapara2}), (\ref{bdryQBULKkurapara3})
 appear when $S$ is bounded and is away from $0$.
(\ref{bdryQBULKkurapara2}) is the case there is some bubble which
slides to $\tau \to -\infty$ and
(\ref{bdryQBULKkurapara3}) is the case there is some bubble which
slides to $\tau \to +\infty$.
\end{proof}
We now use a system of CF-perturbations
$\widehat{\mathfrak S}=\{\widehat{\mathfrak S^{\epsilon}}\}$
on ${{\CM}}_{\ell}(para; H_\chi,J_\chi;*,*;C)$ such that it is compatible
with the description of its boundary given in
Lemma \ref{HBULKkurapara} (3) and that
$\text{\rm ev}_{+\infty}$ is
strongly submersive with respect to $\widehat{\mathfrak S}$,
in the sense of Definition \ref{CFtransv}.
We need some particular choice thereof at some of the factors
of the boundary component.
\par
We observe that there exist maps
\begin{equation}\label{HJCforget1}
\widehat{{{\CM}}}_{\ell}(H=0,J_0;*,*;C)
\to
\overline{{{\CM}}}_{\ell}(H=0,J_0;*,*;C)
\end{equation}
and
\begin{equation}\label{HJCforget2}
{{{\CM}}}_{\ell}(H=0,J_0;*,*;C)
\to
\overline{{{\CM}}}_{\ell}(H=0,J_0;*,*;C).
\end{equation}
Various evaluation maps factor through them.
We take our CF-perturbation obtained by the pull-back with respect to
the maps (\ref{HJCforget2}), (\ref{HJCforget1}) on
the first factor of (\ref{bdryQBULKkurapara2}),
the second factor of (\ref{bdryQBULKkurapara3}) and on (\ref{bdryQBULKkurapara4}).
\par
Using the system of CF-perturbations as above, we define the map
$$
\frak H_{H,J;C}^{\frak b} :
\Omega(M) \widehat\otimes \Lambda^{\downarrow}
\to \Omega(M) \widehat\otimes \Lambda^{\downarrow}
$$
by putting
$$
\frak H_{H,J}^{\frak b}(h)
=
\sum_{\ell=0}^{\infty}\sum_C\frac{\exp(C \cap\frak b_2)}{\ell!} q^{-
C\cap \omega} \text{\rm ev}_{+\infty}!
\left(
\text{\rm ev}_{-\infty}^*h \wedge
\text{\rm ev}^*(\underbrace{\frak b_{+},
\dots,\frak b_{+}}_{\ell}) ; \widehat{\mathfrak S^{\epsilon}}
\right).
$$
Here each summand of the right hand side is given by the correspondence by the
moduli space ${\CM}_{\ell}(para;H_\chi,J_\chi;*,*;C)$.
(See Definition \ref{defn354}, \cite[Definition 7.85]{fooo:tech2}.)
\begin{lem}\label{MBchainhomotopyprop}
$$
\partial \circ \frak H_{H,J}^{\frak b} + \frak H_{H,J}^{\frak b}\circ \partial
= \mathcal Q^\frak b_{(H_{\tilde \chi},J_{\tilde \chi})}\circ \mathcal P^\frak b_{(H_\chi,J_\chi)} - id.
$$
\end{lem}
\begin{proof}
The proof is based on Lemma \ref{HBULKkurapara} (3) and Stokes' formula
(Theorem \ref{them48}, \cite[Lemma 12.13]{fooo:bulk}, \cite[Corollary 8.13]{fooo:tech2}) and the composition
formula (Theorem \ref{compform}, \cite[Lemma 12.15]{fooo:bulk}, \cite[Theorem 10.20]{fooo:tech2}).
We note that (\ref{bdryQBULKkurapara1})
corresponds to the composition
$\mathcal Q^\frak b_{(H_{\tilde \chi},J_{\tilde \chi})}\circ \mathcal P^\frak b_{(H_\chi,J_\chi)}$.
Using the compatibility of the CF-perturbation
and evaluation map to   (\ref{HJCforget2})
it is easy to see that the contribution of (\ref{bdryQBULKkurapara2})
and (\ref{bdryQBULKkurapara3}) vanishes.
\par
By the same reason, the contribution of  (\ref{bdryQBULKkurapara4})
vanishes by the  the compatibility with (\ref{HJCforget1})), except the case $\ell =0$ and $C=0$.
In the latter cases the moduli space is identified with $M$ and $\text{\rm ev}_{\pm \infty}$
 with the identity map. Therefore the contribution induces the
identity map $: \Omega(M) \to \Omega(M)$.
This finishes the proof of Lemma \ref{MBchainhomotopyprop}.
\end{proof}
Therefore the proof of Proposition \ref{QP-1} is now complete.
\end{proof}
Similarly as in Proposition \ref{QP-1}
we can prove that
$\mathcal P^\frak b_{(H_\chi,J_\chi)}\circ \mathcal Q^\frak b_{(H_{\tilde \chi},J_{\tilde \chi})}$ is chain
homotopic to the identity.
Hence the proof of Theorem \ref{appendmain} is now complete.
\end{proof}

We now complete the proof of Theorem \ref{axiomshbulk} (3).
It remains to prove the following:
\begin{prop}\label{oppositeineq} For each $a \in QH^{\frak b}(M)$,
$$
\rho^{\frak b}(\underline 0;a) \ge \frak v_q(a) ( = - \frak v_T(a)).
$$
\end{prop}
\begin{proof}
\begin{lem}
If ${{\CM}}_{\ell}(H_{\tilde \chi},J_{\tilde \chi};[\gamma,w],*)$ is nonempty, then
$$
\mathcal A_H([\gamma,w]) \ge -E^+(H).
$$
\end{lem}
The proof is similar to the proof of Lemma
\ref{connectinghomofilt} and so omitted.
\begin{cor}
$$
\mathcal Q_{(H_{\tilde \chi},J_{\tilde \chi})}^{\frak b}
\left(
F^{\lambda}CF(M,H,J;\Lambda^{\downarrow})
\right)
\subseteq
q^{\lambda + E^+(H)}\Omega(M) \widehat{\otimes}
\Lambda^{\downarrow}.
$$
\end{cor}
\begin{proof}
Let
$
x \in F^{\lambda}CF(M,H,J;\Lambda^{\downarrow}).
$
We choose $[ \gamma, w_{\gamma}] \in \text{\rm Crit}(\mathcal A_H)$ for each of
$\gamma \in \text{\rm Per}(H)$ and put
$$
x = \sum_{\gamma \in \text{\rm Per}(H)}
x_{\gamma}\llb \gamma,w_\gamma \rrb,
$$
with
\begin{equation}\label{23formula1}
\frak v_q(x_\gamma) + \mathcal A_H(\llb \gamma,w_\gamma \rrb)
\le \lambda.
\end{equation}
By (\ref{bdrycoefbQQ}) we have
\begin{equation}\label{23formula2}
\lambda_q(\mathcal Q_{(H_{\tilde \chi},J_{\tilde \chi})}^{\frak b}(x))
\le
\max_{[\gamma, w']}
(-w_{\gamma}\cap \omega
+ w'\cap \omega
+ \frak v_q(x_\gamma)),
\end{equation}
where the maximum in the right hand side is taken over all
$[\gamma,w'] \in \text{\rm Crit}(\mathcal A_H)$
such that ${{\CM}}_{\ell}(H,J;[\gamma, w'],*)$ is nonempty.
\par
We note that
\begin{equation}\label{23formula3}
-w_{\gamma}\cap \omega
+ w'\cap \omega
=
-\mathcal A_H(\llb \gamma, w' \rrb)
+ \mathcal A_H(\llb \gamma, w_{\gamma} \rrb).
\end{equation}
By (\ref{23formula1}), (\ref{23formula2}), (\ref{23formula3}) we obtain
$$
\lambda_q(\mathcal Q_{(H_{\tilde \chi},J_{\tilde \chi})}^{\frak b}(x)) \le \lambda + E^+(H)
$$
as required.
\end{proof}
We take a sequence of normalized Hamiltonians
$H_i$ such that $\lim_{i\to\infty}\Vert H_i\Vert =0$ and $\widetilde\psi_{H_i}$
is nondegenerate.
Let $x \in CF(M,H_i,J;\Lambda^{\downarrow})$
such that $\partial_{(H_i,J)}^{\frak b}x = 0$,
$[x] = \mathcal P_{((H_i)_\chi,J_\chi)}^{\frak b}(a^\flat)$, and
$$
\vert
\lambda_{H_i}(x) - \rho^{\frak b}(H_i;a)
\vert
< \epsilon.
$$
Then $[\mathcal Q_{((H_i)_{\tilde\chi},J_{\tilde\chi})}^{\frak b}(x)] = a^\flat$ and
$$
\lambda_q (\mathcal Q_{((H_i)_{\tilde\chi},J_{\tilde\chi})}^{\frak b}(x))
\le
\rho^{\frak b}(H_i;a) + \epsilon
+ E^+(H_i).
$$
Since $\epsilon$ is arbitrary small and $\lim_{i\to\infty} E^+(H_i) = 0$,
we obtain the proposition.
\end{proof}

\begin{rem}
Actually there is a problem of `running-out of the given Kuranishi neighborhood'
in the above proof when the energy level increases as
mentioned in \cite[Section 7.2.3]{fooo:book2}.
In order to handle the problem, we first work
over the $\Lambda^{\downarrow}_0$ coefficients and perform the construction
up to a pre-given finite energy level. Then we take an inductive limit as we
let energy level go to infinity.
The technical difficulty performing this construction is much
simpler than that of \cite[Section 7.2]{fooo:book2}, since here
we need to take an inductive limit of chain complex (or DGA). This
is much simpler than the case of $A_{\infty}$ algebra in general discussed in
\cite[Section 7.2]{fooo:book2}. So we omit the detail.
\end{rem}

\section{Independence on the de Rham representative of $\frak b$.}
\label{sec:appendix2}

In this section we prove Theorem \ref{homotopyinvbulk} (2).
Let $H$ be a one periodic Hamiltonian on $M$ such that
$\psi_H$ is nondegenerate.
Let $\frak b(0),\frak b(1) \in \Omega(M) \widehat\otimes \Lambda^{\downarrow}$
such that $d\frak b(0) = d\frak b(1) = 0$. We assume that there exists
$\frak c \in \Omega(M) \widehat\otimes \Lambda^{\downarrow}$
such that
\begin{equation}
\frak b(1) - \frak b(0) = d\frak c.
\end{equation}
Then we prove that $\rho^{\mathfrak b(0)}(\widetilde{\phi}_H;a)=\rho^{\mathfrak b(1)}(\widetilde{\phi}_H;a)$.
Firstly we consider the case that $\mathfrak b_2(0)=\mathfrak b_2(1)$.
Here $\mathfrak b_2(0), \mathfrak b_2(1) \in H^2(M;\C)$ as in  (\ref{decompb}).
After establishing Theorem \ref{homotopyinvbulk} (2) under the condition that
$\mathfrak b_2(0)=\mathfrak b_2(1)$, we show that the invariant $\rho^{\mathfrak b}(\widetilde{\phi}_H;a)$
does not depend on the choice of representative of the cohomology class $[\mathfrak b_2]$.

We consider the ring of strictly convergent power series
\index{$\Lambda^{\downarrow}\leftineqineq
s\rightineqineq$}
\begin{equation}
\Lambda^{\downarrow}\langle\!\langle
s\rangle\!\rangle
=
\left.\left\{
\sum_{k=0}^{\infty} x_k s^k ~\right\vert~
x_k \in \Lambda^{\downarrow},
\lim_{k\to\infty}\frak v_q(x_k) = -\infty
\right\}.
\end{equation}
Here $s$ is a formal parameter.
We denote by $\text{\rm Poly}(\R;CF(M;H;\Lambda^{\downarrow}))$ the set of formal expressions of the form
$$
x(s)+ ds \wedge y(s)
$$
where
$$
x(s),y(s) \in CF(M;H;\Lambda^{\downarrow})\otimes_{\Lambda^{\downarrow}}\Lambda^{\downarrow}\langle\!\langle s\rangle\!\rangle.
$$
\par
For $s_0 \in \R$ we define
$$
\text{\rm Eval}_{s=s_0}
: \text{\rm Poly}(\R;CF(M;H;\Lambda^{\downarrow}))
\to
CF(M;H;\Lambda^{\downarrow})
$$
by
\begin{equation}\label{Eval}
\text{\rm Eval}_{s=s_0}(x(s)+ ds \wedge y(s)) = x(s_0).
\end{equation}
We note that, for $x(s) = \sum_{k=0}^{\infty} x_k s^k \in
CF(M;H;\Lambda^{\downarrow})\otimes_{\Lambda^{\downarrow}}\Lambda^{\downarrow}\langle\!\langle s\rangle\!\rangle$ with $x_k \in CF(M;H;\Lambda^{\downarrow})$,
the series
$x(s_0) = \sum_{k=0}^{\infty} x_k s_0^k $ converges in $q$-adic topology for $s_0 \in \R$.
\par\smallskip
We put
\begin{equation}\label{dcisb}
\frak b(s) = s\frak b(1) + (1-s)\frak b(0).
\end{equation}
For each $s_0 \in \R$ we define
$$
\partial_{(H,J)}^{\frak b(s_0)} :
CF(M;H;\Lambda^{\downarrow})
\to
CF(M;H;\Lambda^{\downarrow})
$$
by (\ref{defboundary}).

\begin{lem}
There exists a $\Lambda^{\downarrow}\langle\!\langle s\rangle\!\rangle$-module homomorphism
$$
\partial^{\frak b(\cdot)}_{(H,J)} :
CF(M;H;\Lambda^{\downarrow})\otimes_{\Lambda^{\downarrow}}\Lambda^{\downarrow}\langle\!\langle s\rangle\!\rangle
\to
CF(M;H;\Lambda^{\downarrow})\otimes_{\Lambda^{\downarrow}}\Lambda^{\downarrow}\langle\!\langle s\rangle\!\rangle
$$
such that
\begin{equation}\label{formulaeval}
\text{\rm Eval}_{s=s_0}
\circ \partial^{\frak b(\cdot)}_{(H,J)}
= \partial^{\frak b(s_0)}_{(H,J)} \circ \text{\rm Eval}_{s=s_0},
\quad
\partial^{\frak b(\cdot)}_{(H,J)}\circ \partial^{\frak b(\cdot)}_{(H,J)} = 0.
\end{equation}
\end{lem}
\begin{proof}
We split $\frak b(s) = \frak b_0(s) +\frak b_2(s) +\frak b_+(s)$
as in (\ref{decompb}).
Then we have $\frak b_2(s) = s\frak b_2(1) + (1-s) \frak b_2(0)$ etc.
We use it to see that
$$
\frak n_{(H,J);\ell}([\gamma,w ],[\gamma',w'])
(\underbrace{\frak b_+(s),\dots,\frak b_+(s)}_{\ell})
$$
is a polynomial of order $\le \ell$ in $s$ with coefficients in $\C$. (See \eqref{eq:nww'C}, (\ref{bdrycoefb}).)
\par
Therefore the series (\ref{bdrycoefb}) of the current context converges for each $s$ and so
the totality thereof define an element
$$
\frak n_{(H,J)}^{\frak b(s)}([\gamma,w],[\gamma',w' ]) \in \Lambda^{\downarrow}\langle\!\langle s\rangle\!\rangle.
$$
Hence we can define the map $\partial^{\frak b(\cdot)}_{(H,J)}$ by
replacing $\frak b$ by $\frak b(s)$ in (\ref{defboundary}).
The first formula in (\ref{formulaeval}) is easy to show.
The second formula follows from the first one.
\end{proof}
We next put
$$\aligned
&\frak n^{\frak c,1}_{(H,J)}([\gamma,w],[\gamma',w' ])\\
&=
\sum_{\ell_1=0}^{\infty}\sum_{\ell_2=0}^{\infty}
\frac{\exp(w'\cap \frak b_2(s) - w\cap \frak b_2(s))}{(\ell_1+\ell_2+1)!}\\
&\quad\quad\quad
\frak n_{(H,J);\ell_1+\ell_2+1}([\gamma,w],[\gamma',w'])
(\underbrace{\frak b_+(s),\dots,\frak b_+(s)}_{\ell_1},\frak c,
\underbrace{\frak b_+(s),\dots,\frak b_+(s)}_{\ell_2})\\
&\in \Lambda^{\downarrow}\langle\!\langle s\rangle\!\rangle
\endaligned
$$
and define a map $\partial_{(H,J)}^{\frak c}: CF(M;H;\Lambda^\downarrow) \to CF(M;H;\Lambda^\downarrow)$ by
\begin{equation}
\partial_{(H,J)}^{\frak c}(\llb \gamma,w \rrb)
= \sum_{[\gamma',w' ] \in \text{\rm Crit}(\mathcal A_H)}
\frak n^{\frak c,1}_{(H,J)}([\gamma,w],[\gamma',w'])\llb \gamma',w' \rrb.
\end{equation}
\begin{lem}\label{partsformula}
\begin{equation}
\frac{\partial}{\partial s} \circ \partial^{\frak b(\cdot)}_{(H,J)}
-
\partial^{\frak b(\cdot)}_{(H,J)} \circ \frac{\partial}{\partial s}
= \partial_{(H,J)}^{\frak c}
\circ \partial^{\frak b(\cdot)}_{(H,J)}
-
\partial^{\frak b(\cdot)}_{(H,J)}\circ
 \partial_{(H,J)}^{\frak c}.
\end{equation}
\end{lem}
\begin{proof}
Using Proposition \ref{connBULKkura} (3), Stokes' formula
(Theorem \ref{them48})
and the composition formula (Theorem \ref{compform}), we obtain
\begin{equation}\label{franformulawasure}
\aligned
&\sum_{i=1}^{\ell}
(-1)^*
\frak n_{(H,J);\ell}([\gamma,w],[\gamma',w'])
(h_1,\dots, dh_i,\dots, h_{\ell}) \\
&=
\sum_{(\L_1,\L_2) \in \text{\rm Shuff}(\ell)}\sum_{[\gamma'',w'']}
(-1)^{**}\frak n_{(H,J);\# \L_1}([\gamma,w],[\gamma'',w''])
(h_{i_1},\dots, h_{i_{\# \L_1}})
\\
&\qquad\qquad\qquad\qquad\qquad
\frak n_{(H,J);\# \L_2}([\gamma'',w''],[\gamma',w'])
(h_{j_1},\dots, h_{j_{\# \L_2}}),
\endaligned
\end{equation}
where
$\L_1 =\{ i_1,\dots, i_{\# \L_1}\},
\L_2 =\{ j_1,\dots, j_{\# \L_2}\}$,
$$
* =\deg h_1 + \dots + \deg h_{i-1}, \quad
** = \sum_{i\in \L_1, j \in \L_2; j < i} \deg h_i\deg h_j.
$$
Using (\ref{dcisb}) and (\ref{franformulawasure})
we can prove Lemma \ref{partsformula} easily.
\end{proof}
We define the map
$$
\partial_{(H,J)}^{(\frak b(\cdot),\frak c)} :
\text{\rm Poly}(\R;CF(M;H;\Lambda^{\downarrow}))
\to
\text{\rm Poly}(\R;CF(M;H;\Lambda^{\downarrow}))
$$
by
\begin{equation}
\aligned
&
\partial_{(H,J)}^{(\frak b(\cdot),\frak c)}(x(s) +  ds \wedge y(s))\\
&=
\partial_{(H,J)}^{\frak b(\cdot)}(x(s))
- ds \wedge \frac{\partial}{\partial s} (x(s))
+ ds \wedge \partial_{(H,J)}^{\frak c}(x(s))
- ds \wedge \partial_{(H,J)}^{\frak b(\cdot)}(y(s)).
\endaligned
\end{equation}
Then the second formula of (\ref{formulaeval}) and
Lemma \ref{partsformula} imply
$$
\partial_{(H,J)}^{(\frak b(\cdot),\frak c)}\circ
\partial_{(H,J)}^{(\frak b(\cdot),\frak c)} = 0.
$$
Thus
$
(\text{\rm Poly}(\R;CF(M;H;\Lambda^{\downarrow})) ,\partial_{(H,J)}^{(\frak b(\cdot),\frak c)})
$
is a chain complex.
The first formula of (\ref{formulaeval}) implies that
\begin{equation}\label{2510}
\text{\rm Eval}_{s=s_0}
: (\text{\rm Poly}(\R;CF(M;H;\Lambda^{\downarrow})),\partial_{(H,J)}^{(\frak b(\cdot),\frak c)})
\to
(CF(M;H;\Lambda^{\downarrow}),\partial_{(H,J)}^{\frak b(s_0)})
\end{equation}
is a chain map.
\par
We define a filtration
$F^{\lambda}\text{\rm Poly}(\R;CF(M;H;\Lambda^{\downarrow}))$ on $\text{\rm Poly}(\R;CF(M;H;\Lambda^{\downarrow}))$
by
$$
\aligned
&F^{\lambda}\text{\rm Poly}(\R;CF(M;H;\Lambda^{\downarrow})) \\
&=
\{
x(s) + ds \wedge y(s)
\mid
x(s) = \sum x_ks^k, y(s) = \sum y_ks^k, \lambda_H(x_k), \lambda_H(y_k) \le \lambda
\}.
\endaligned
$$
\begin{lem}\label{parafilt}
 $\partial_{(H,J)}^{(\frak b(\cdot),\frak c)}$ and
$\text{\rm Eval}_{s=s_0}$ preserves the filtration $F^{\lambda}$.
\end{lem}
The proof is easy and is omitted.
\begin{lem}\label{evische} The map
$(\ref{2510})$ is a chain homotopy equivalence.
\end{lem}
\begin{proof}
If $x(s) + ds \wedge y(s) \in F^{\lambda}\text{\rm Poly}(\R;CF(M;H;\Lambda^{\downarrow}))$, then we have
$$
\partial_{(H,J)}^{(\frak b(\cdot),\frak c)}(x(s) + ds \wedge y(s))
- ds \wedge \frac{\partial x}{\partial s}(s)
\in
F^{\lambda-\epsilon}\text{\rm Poly}(\R;CF(M;H;\Lambda^{\downarrow}))
$$
for some positive $\epsilon$.
We use this fact to prove Lemma \ref{evische} by the same way as in
the proof given in \cite[Proposition 4.3.18]{fooo:book1}.
\end{proof}
We next define the map
\begin{equation}\label{eq:PbcHJ}
\mathcal P_{(H_\chi,J_\chi)}^{(\frak b(\cdot),\frak c)}
: \Omega(M) \otimes \Lambda^{\downarrow}
\to
\text{\rm Poly}(\R;CF(M;H;\Lambda^{\downarrow})).
\end{equation}
First, for each fixed $s_0$, we define a map
\begin{equation}
\mathcal P_{(H_\chi,J_\chi)}^{\frak b(s_0)}
:
\Omega(M) \otimes \Lambda^{\downarrow}
\to
CF(M;H;\Lambda^{\downarrow})
\end{equation}
by the formula (\ref{Piunikinb}). We then obtain a map
\begin{equation}
\mathcal P_{(H_\chi,J_\chi)}^{\frak b(\cdot)}
:
\Omega(M) \otimes \Lambda^{\downarrow}
\to
CF(M;H;\Lambda^{\downarrow}) \otimes_{\Lambda^{\downarrow}}
\Lambda^{\downarrow}\langle\!\langle s\rangle\!\rangle
\end{equation}
that satisfies
\begin{equation}
\text{\rm Eval}_{s=s_0} \circ \mathcal P_{(H_\chi,J_\chi)}^{\frak b(\cdot)}
=
\mathcal P_{(H_\chi,J_\chi)}^{\frak b(s_0)}.
\end{equation}
Similarly to (\ref{bdrycoefb2}), we define the term
$$\aligned
&\frak n^{\frak c,1}_{(H_\chi,J_\chi)}(h;[\gamma,w])\\
&=
\sum_{\ell_1=0}^{\infty}\sum_{\ell_2=0}^{\infty}
\frac{\exp( \int w^{*}\frak b_2(s))}{(\ell_1+\ell_2+1)!}\\
&\quad\quad\quad
\frak n_{(H,J);\ell_1+\ell_2+1}(h;[\gamma,w])
(\underbrace{\frak b_+(s),\dots,\frak b_+(s)}_{\ell_1},\frak c,
\underbrace{\frak b_+(s),\dots,\frak b_+(s)}_{\ell_2})\\
&\in \Lambda^{\downarrow}\langle\!\langle s\rangle\!\rangle.
\endaligned
$$
Using this, we then define
a map $\mathcal P_{(H_\chi,J_\chi)}^{\frak c}:\Omega(M) \otimes \Lambda^\downarrow
\to CF(M;H;\Lambda^\downarrow)$ by
$$
\mathcal P_{(H_\chi,J_\chi)}^{\frak c}(h)
=
\sum_{[\gamma,w] \in {\rm Crit}(\mathcal A_H)}
\frak n^{\frak c,1}_{(H_\chi,J_\chi)}(h;[\gamma,w]) \llb \gamma,w \rrb.
$$
Finally we define the map \eqref{eq:PbcHJ} by putting
\begin{equation}
\mathcal P_{(H_\chi,J_\chi)}^{(\frak b(\cdot),\frak c)}(h)
=
\mathcal P_{(H_\chi,J_\chi)}^{\frak b(\cdot)}(h)
+ ds \wedge \mathcal P_{(H_\chi,J_\chi)}^{\frak c}(h).
\end{equation}
\begin{lem}\label{Pandpara}
We have
$$
\partial_{(H,J)}^{(\frak b(\cdot),\frak c)}
\circ \mathcal P_{(H_\chi,J_\chi)}^{(\frak b(\cdot),\frak c)}
=
\mathcal P_{(H_\chi,J_\chi)}^{(\frak b(\cdot),\frak c)} \circ \partial
$$
and
$$
\text{\rm Eval}_{s=s_0} \circ \partial_{(H,J)}^{(\frak b(\cdot),\frak c)}
=
\mathcal P_{(H_\chi,J_\chi)}^{\frak b(s_0)}.
$$
\end{lem}
The proof is straightforward calculation and is omitted.
\par
Combining Lemmas \ref{parafilt}, \ref{evische}, \ref{Pandpara},
we easily prove
$\rho^{\frak b(0)}(\widetilde\psi_H,a) = \rho^{\frak b(1)}(\widetilde\psi_H,a)$.
The proof of Theorem \ref{homotopyinvbulk} (2) is complete
under the condition that $\mathfrak b_2(0)=\mathfrak b_2(1)$.
\qed

Next, for $\mathfrak b(0), \mathfrak b(1)$ such that $\mathfrak b(1)-\mathfrak b(0)= d \mathfrak c$
for some $\mathfrak c$, we consider $\mathfrak b' = \mathfrak b(0) + d (\mathfrak c - \mathfrak c_1)$.
Here $\mathfrak c_1$ is the  $\Omega^1(M;\C)$-component of $\mathfrak c$ in the decomposition
$\Omega^1(M;\C) \oplus \Omega^1(M;\Lambda_-^{\downarrow}) \oplus \Omega^{\geq 3} (M;\Lambda^{\downarrow})$.
We showed that $\rho^{\mathfrak b(0)}(\widetilde{\phi}_H,a)=\rho^{\mathfrak b'}(\widetilde{\phi}_H,a)$.
The remaining task is to show that
$\rho^{\mathfrak b'}(\widetilde{\phi}_H,a)=\rho^{\mathfrak b(1)}(\widetilde{\phi}_H,a)$.
Namely we prove  Theorem \ref{homotopyinvbulk} (2) in the case that
$\mathfrak b(1) - \mathfrak b(0) = d \mathfrak c_1$ with $\mathfrak c_1 \in \Omega^1(M;\C)$.

We define $I:CF(M;H;\Lambda^{\downarrow}) \to CF(M;H;\Lambda^{\downarrow})$ by
$$I(\llb \gamma,w \rrb) = \exp (\int_{S^1} \gamma^* \mathfrak c_1) \llb \gamma,w \rrb.$$
Then we find that $I$ gives an isomorphism of Floer chain complexes
$$
I:(CF(M;H;\Lambda^{\downarrow}), \partial^{\mathfrak b'}_{(J,H)}) \to
(CF(M;H;\Lambda^{\downarrow}), \partial^{\mathfrak b(1)}_{(J,H)})
$$
and
$$I \circ {\mathcal P}^{\mathfrak b'}_{(H_{\chi},J_{\chi})} = {\mathcal P}^{\mathfrak b(1)}_{(H_{\chi},J_{\chi})}.$$
Hence the proof of Theorem \ref{homotopyinvbulk} (2).
\qed

\begin{rem}
A cocycle $\frak b \in \Omega^2(M;\C)$ induces a representation
$$\frak{rep}_{\frak b}: a \in \pi_1({\mathcal L}(M);\ell_0) \mapsto \exp
\left(\int_{C_a} {\frak b}\right) \in \C^{*},$$
where $C_a:S^1 \times S^1 \to M$ is the mapping corresponding to the loop
$a$ in ${\mathcal L}(M)$.
Then we can consider Floer complex of the Hamiltonian system with
coefficients in the local system
corresponding to $\frak{rep}_{\frak b}$.
If ${\frak b}(0)$ and ${\frak b}(1)$ are cohomologous, the corresponding
local systems are isomorphic,
hence Floer cohomology groups with coefficients in these local systems are isomorphic.
Here we gave the isomorphism $I$ directly without dealing with the
isomorphism of the local systems.
\end{rem}

\section{Proof of Proposition 20.7.}
\label{sec:appendix3}

The purpose of this section is to prove
Proposition \ref{Ainfinityequiv} and Lemma \ref{ainfinityequivcomp}.
In this section we fix a $t$-independent $J$.

\subsection{Pseudo-isotopy of filtered $A_{\infty}$ algebra}
\label{subsec:Pseudo-isotopy} \index{pseudo-isotopy! of filtered
$A_\infty$ algebra}

In this subsection, we review the notion of pseudo-isotopy of filtered $A_{\infty}$ algebra,
which was introduced in \cite{fukaya:cyc} Definition 8.5.
We consider $L \times [0,1]$ and use $s$ for the coordinate of the interval $[0,1]$.
We put $\overline C = \Omega(L)$ and
$$
C^{\infty}([0,1]\times \overline C)
=
\Omega([0,1]\times L).
$$
An element of $C^{\infty}([0,1]\times \overline C) $ is written uniquely as
$$
x(s) + ds \wedge y(s)
$$
where $x(s), y(s)$ are smooth differential forms on $[0,1]\times L$ that
do not contain $ds$, i.e., that satisfies
$\frac{\del}{\del s}\rfloor x(s) = 0 = \frac{\del}{\del s}\rfloor y(s)= 0$.
It follows that $x(s_0), y(s_0) \in \overline C$ for each fixed $s_0$.
\par
Suppose that, for each $s \in [0,1]$, $k,\ell$, $\beta \in \pi_2(M;L)$ we have operators
\begin{equation}\label{param}
\frak m^s_{k,\beta} : B_k(\overline C[1]) \to \overline C[1]
\end{equation}
of degree $-\mu(\beta)+1$ and
\begin{equation}\label{parac}
\frak c^s_{k,\beta} : B_k(\overline C[1]) \to \overline C[1]
\end{equation}
of degree $-\mu(\beta)$.
\begin{defn}\label{def:smoothop}
We say $\frak m^s_{k,\beta}$ is {\it smooth} if for
each $x_1,\ldots,x_k \in \overline C$ we may regard
$
\frak m^s_{k,\beta}(x_1,\ldots,x_k)
$
as an element of $C^{\infty}([0,1],\overline C)$ with vanishing $ds$ component.
The smoothness of $\frak c^s_{k,\beta}$ is defined in the same way.
\end{defn}
Suppose that there exists a subset $\widehat G$ of
$H_2(M,L;\Z)$ such that
$\{ \omega \cap \beta \mid \beta \in \widehat G\}$ is a discrete subset of
$\R_{\ge 0}$.
Let $G$ be the monoid generated by this set.
Here we assume that the operators $\frak m^s_{k,\beta}, \frak c^s_{k,\beta}$
are given for $\beta \in \widehat G$.

\begin{defn}\label{pisotopydef}
We say $(C,\{\frak m^s_{k,\beta}\},\{\frak c^s_{k,\beta}\})$
is a {\it pseudo-isotopy} \index{pseudo-isotopy} of $G$-gapped filtered $A_{\infty}$ algebras\footnote{
See \cite[Definition 3.2.26]{fooo:book1}
for the definition of $G$-gapped-ness of a filtered $A_{\infty}$ algebra.} if the following holds:
\smallskip
\begin{enumerate}
\item  $\frak m^s_{k,\beta}$ and $\frak c^s_{k,\beta}$ are smooth.
\item For each (but fixed) $s$, $(C,\{\frak m^s_{k,\beta}\})$ defines a filtered $A_{\infty}$
algebra.
\item  For each given $x_i \in \overline C[1]$, $i = 1, \ldots, k$, the operators satisfy
\begin{equation}\label{isotopymaineq}
\aligned
&\frac{d}{ds} \frak m_{k,\beta}^s(x_1,\ldots,x_k) \\
&+ \sum_{k_1+k_2=k}\sum_{\beta_1+\beta_2=\beta}\sum_{i=1}^{k-k_2+1}
(-1)^{*}\frak c^s_{k_1,\beta_1}(x_1,\ldots, \frak m_{k_2,\beta_2}^s(x_i,\ldots),\ldots,x_k) \\
&- \sum_{k_1+k_2=k}\sum_{\beta_1+\beta_2=\beta}\sum_{i=1}^{k-k_2+1}
\frak m^s_{k_1,\beta_1}(x_1,\ldots, \frak c_{k_2,\beta_2}^s(x_i,\ldots),\ldots,x_k)\\
&=0.
\endaligned
\end{equation}
Here $* = \deg' x_1 + \ldots + \deg'x_{i-1}$.
\item
$\frak m_{k,\beta_0}^s$ is independent of $s$, and $\frak c_{k,\beta_0}^s = 0$.
Here $\beta_0 = 0 \in H_2(M;L;\Z)$.
\end{enumerate}
\end{defn}
We consider $x_i(s) + ds \wedge y_i(s) = \text{\bf x}_i \in C^{\infty}([0,1],\overline C)$.
We define the operators $\widehat{\frak m}_{k,\beta}: C^\infty([0,1], \overline{C})^{\otimes k} \to
C^\infty([0,1],\overline{C})$ by
\begin{equation}
\widehat{\frak m}_{k,\beta}(\text{\bf x}_1,\ldots,\text{\bf x}_k)
= x(s) + ds \wedge y(s),
\end{equation}
where
\begin{subequations}\label{combineainf}
\begin{equation}
x(s) = {\frak m}^s_{k,\beta}(x_1(s),\ldots,x_k(s))
\end{equation}
\begin{equation}\label{combineainfmain}
\aligned
y(s)
=
& \frak c^s_{k,\beta}
(x_1(s),\ldots,x_k(s)) \\
&-\sum_{i=1}^k (-1)^{*_i} \frak m^t_{k,\beta}
(x_1(s),\ldots,x_{i-1}(s),y_i(s),x_{i+1}(s),\ldots,x_k(s))
\endaligned\end{equation}
if $(k,\beta) \ne (1,\beta_0)$ and
\begin{equation}
y(s) = \frac{d}{ds} x_1(s) + \frak m_{1,0}^s (y_1(s))
\end{equation}
if $(k,\beta) = (1,\beta_0)$. Here $*_i$ in (\ref{combineainfmain}) is
$*_i = \deg' x_1 +\ldots+\deg'x_{i-1}$.
\end{subequations}

\begin{lem}\label{dtcomiAinf}
The equation $(\ref{isotopymaineq})$ is equivalent to the filtered $A_{\infty}$
relation of $\{\widehat{\frak m}_{k,\beta}\}$ defined by $(\ref{combineainf})$.
\end{lem}

The proof is a straightforward calculation.

\begin{defn}\label{smoothins}\index{pseudo-isotopy!unital}
A pseudo-isotopy
$(C,\{\frak m^s_{k,\beta}\},\{\frak c^s_{k,\beta}\})$
is said to be {\it unital}
if there exists $\text{\bf e} \in \overline C^0$ such that
$\text{\bf e}$ is a unit of $(C,\{\frak m^s_{k,\beta}\})$ for each $s$ and if
$
\frak c^s_{k,\beta}(\ldots,\text{\bf e},\ldots) = 0
$
for each $k,\beta$ and $s$.
\end{defn}
In our situation the unit $\text{\bf e}$ is always $\text{\bf e}^L$, the
constant function $1$ on $L$.

\begin{thm}\label{pisotohomotopyequiv}
If  $(C,\{\frak m^s_{k,\beta}\},\{\frak c^s_{k,\beta}\})$ is a unital pseudo-isotopy, then
there exists a unital filtered $A_{\infty}$ homomorphism from $(C,\{\frak m^{0}_{k,\beta}\})$
to $(C,\{\frak m^{1}_{k,\beta}\})$ that  has a homotopy
inverse.
\end{thm}
\begin{proof}
The cyclic version of this theorem is \cite[Theorem 8.2]{fukaya:cyc}.
Since we do not require cyclic symmetry here,
the proof of Theorem \ref{pisotohomotopyequiv} is easier.
In fact, it follows from
\cite[Theorem 4.2.45]{fooo:book1}  as follows.
We have a filtered $A_{\infty}$ homomorphism
$$
\text{\rm Eval}_{s=s_0} : (C^{\infty}([0,1]\times \overline C),\{\widehat{\frak m}^s_{k,\beta}\})
\to
(\overline C,\{\frak m^{s_0}_{k,\beta}\})
$$
defined by
$
(\text{\rm Eval}_{s=s_0})_1(a(s) + ds \wedge b(s))
= a(s_0)
$
and
$
(\text{\rm Eval}_{s=s_0})_k
= 0
$
for $k\ne 1$.  Then using the $A_\infty$ Whitehead theorem (\cite[Lemma 4.2.45]{fooo:book1})
we can show that it is a homotopy equivalence.
Theorem \ref{pisotohomotopyequiv} follows.
\end{proof}

\subsection{Difference between $\frak m^T$ and $\frak m$.}
\label{subsec:DifferencemTm}

We will construct a pseudo-isotopy between two
filtered $A_{\infty}$ structures $\{\frak m^{T,\frak b}_{k,\beta}\}$
and $\{\frak m^{\frak b}_{k,\beta}\}$ on $\overline C = \Omega(L)$.
Here the first one is defined in Section \ref{toriceLagreview}
and the second one is defined in Section \ref{sec:q-map}.
We note that the difference of these two constructions are
roughly as follows:
\begin{enumerate}
\item
We represent $\frak b$ by a $T^n$-invariant cycle $D_a$ that is a submanifold to
define $\{\frak m^{T,\frak b}_{k,\beta}\}$.
In other words, we use a
 current that may not be smooth when we define
$\{\frak m^{T,\frak b}_{k,\beta}\}$. On the other hand, we represent $\frak b$ by a smooth differential form
when we define $\{\frak m^{\frak b}_{k,\beta}\}$.
\item
In the definition of $\{\frak m^{T,\frak b}_{k,\beta}\}$ we first take the fiber product
(\ref{fiberproductwithbeta}) and then use a multisection to achieve transversality.
On the other hand, in the definition of $\{\frak m^{\frak b}_{k,\beta}\}$,
we first perturb (by CF-perturbations) the moduli space $\mathcal M_{k+1:\ell}(\beta)$ then
pull back the differential form representing $\frak b$ to the
zero set of the multisection.
In other words the perturbation to define  $\{\frak m^{\frak b}_{k,\beta}\}$
is independent of the ambient cohomology class $\frak b$.
\end{enumerate}
\begin{rem}\label{rem28777}
We note that there are various reasons why we need to
take cycles and multisections (rather than taking CF-perturbation), when we construct
$\{\frak m^{T,\frak b}_{k,\beta}\}$ in the toric case. The most important reason lies in Proposition
\ref{H1iswbdchain}. This is related to point (1) above. The reason
why we first need to take the fiber product
(\ref{fiberproductwithbeta}) is explained in \cite[Remark 11.4]{fooo:toric1}.
\par
On the other hand to develop the theory of spectral invariant with bulk
deformation in the general setting,
it seems simplest to  always use de Rham representative and CF-perturbations.
\end{rem}

In the next subsection, we will construct a pseudo-isotopy of filtered $A_{\infty}$
structures interpolating $\{\frak m^{T,\frak b}_{k,\beta}\}$  and
$\{\frak m^{\frak b}_{k,\beta}\}$. Below we handle the above (1)
and (2) separately. We construct the pseudo-isotopy resolving (1) in
Subsection \ref{subsec:Smoothing} and construct the pseudo-isotopy
resolving (2) in Subsection \ref{subsec:ainfinityequivcomp}.

\subsection{Smoothing $T^n$-invariant chains.}
\label{subsec:Smoothing} \index{smoothing}

Let $D_a = D_{i_1}\cap \dots \cap D_{i_k}$ be a transversal intersection of
$k$ irreducible components of the toric divisor,
$ND_a$ its normal bundle, and $\exp : ND_a \to M$ the exponential map
with respect to a $T^n$-invariant Riemannian metric.
Let $\mathcal U_a \subset \Gamma(ND_a)$ be a finite dimensional submanifold
of the space of smooth sections
of $ND_a$ such that if $\frak u \in \mathcal U_a$ and $\rho \in [0,1]$ then
$\rho\frak u \in \mathcal U_a$.
We assume that it has the following properties.

\begin{proper}\label{Uproper}
\begin{enumerate}
\item
The exponential map $\text{\rm Exp} : D_a \times \mathcal U_a \to M$ defined by
\begin{equation}\label{Exmap}
\text{\rm Exp}(\frak u,x) = \exp(\frak u(x))
\end{equation}
is a submersion.
\item
$\Vert \frak u(x) \Vert < \epsilon$, where $\epsilon$ is a
sufficiently small positive number determined later.
\end{enumerate}
\end{proper}
We put $d_a = \dim \mathcal U_a$.
\par
Let $\text{\bf p} : \{1,\dots,\ell\} \to \underline B$ be as in the beginning of Subsection
\ref{subsec:toricHFreview}.
We put
\begin{equation}
\mathcal U(\text{\bf p}) = \prod_{i=1}^{\ell} \mathcal U_{\text{\bf p}(i)},
\qquad
\text{\bf p}(\mathcal U) = \prod_{i=1}^{\ell}
(D_{\text{\bf p}(i)} \times \mathcal U_{\text{\bf p}(i)}).
\end{equation}
The map (\ref{Exmap}) induces
\begin{equation}
 \text{\rm Exp} : \text{\bf p}(\mathcal U) \to M^{\ell}.
\end{equation}
For $k,\ell \in \Z_{\ge 0}$ and $\beta \in H_2(M,L(\text{\bf u});\Z)$ we
define a fiber product
\begin{equation}\label{fiberproductwithbetawithU}
\mathcal M_{k+1;\ell}(L(\text{\bf u});\beta;\text{\bf p}(\mathcal U) )
=
\mathcal M_{k+1;\ell}(L(\text{\bf u});\beta)
{}_{(ev_1,\ldots,ev_{\ell})}
\times_{ \text{\rm Exp}}  \text{\bf p}(\mathcal U),
\end{equation}
where $\mathcal M_{k+1;\ell}(L(\text{\bf u});\beta)$
is a moduli space defined in Definition \ref{diskmoduli1} and
Proposition \ref{disckura}.
(Compare (\ref{fiberproductwithbeta}).)
We can define the evaluation map at the boundary marked points
\index{$\mathcal M_{k+1;\ell}(L(\text{\bf u});\beta;\text{\bf p}(\mathcal U) )$}
$$
\text{\rm ev}^{\partial}
=
(\text{\rm ev}^{\partial}_1,\dots,\text{\rm ev}^{\partial}_k)
:\mathcal M_{k+1;\ell}(L(\text{\bf u});\beta;\text{\bf p}(\mathcal U) )
\to L(\text{\bf u})^{k+1}
$$
in an obvious way.
We also have a projection
$$
\pi_{\mathcal U} :
\mathcal M_{k+1;\ell}(L(\text{\bf u});\beta;\text{\bf p}(\mathcal U) )
\to
\mathcal U(\text{\bf p})
$$
to the $\mathcal U_a$-factors.
By definition we have
\begin{equation}
\pi_{\mathcal U}^{-1}(\text{\bf 0}) =
\mathcal M_{k+1;\ell}(L(\text{\bf u});\beta;\text{\bf p}).
\end{equation}
\begin{lem}\label{torickuranishiU}
\begin{enumerate}
\item
$\mathcal M_{k+1;\ell}(L(\text{\bf u});\beta;\text{\bf p}(\mathcal U))$ has
a Kuranishi structure with corners.
\item
It induces the Kuranishi structure on $\pi_{\mathcal U}^{-1}(\text{\bf 0})$
as described in Lemma $\ref{torickuranishi}$.
\item
Its normalized boundary is described by the union of the following fiber products:
\begin{equation}
\mathcal
M_{k_1+1;\#\L_1}(L(\text{\bf u});\beta_1;\text{\bf p}_1(\mathcal U))
{}_{{\rm ev}^{\partial}_0}\times_{{\rm ev}^{\partial}_i} \mathcal M_{k_2+1;\#\L_2}(L(\text{\bf u});\beta_2;\text{\bf p}_2(\mathcal U))
\end{equation}
where the union is taken over all $(\L_1,\L_2) \in \text{\rm Shuff}(\ell)$,
$k_1,k_2$ with $k_1 + k_2 = k$ and $\beta_1,\beta_2 \in H_2(M,L(\text{\bf u});\Z)$ with
$\beta = \beta_1+\beta_2$. We put $
\text{\rm Split}((\mathbb L_1,\mathbb L_2),\text{\bf p}) = (\text{\bf
p}_1,\text{\bf p}_2).
$
\item
The dimension is
\begin{equation}
\aligned
&\dim \mathcal M_{k+1;\ell}(L(\text{\bf u});\beta;\text{\bf p}(\mathcal U)) \\
&= n + \mu_{L(\text{\bf u})}(\beta) + k - 2 + 2\ell - \sum_{i=1}^{\ell} 2\deg D_{\text{\bf p}(i)}
+ \sum_{i=1}^{\ell} d_{\text{\bf p}(i)}.
\endaligned
\end{equation}
\item The evaluation maps ${\rm ev}^{\partial}_i$ at the boundary marked points of $\mathcal M_{k+1;\ell}(L(\text{\bf u});\beta)$
define maps on $\mathcal M_{k+1;\ell}(L(\text{\bf u});\beta;\text{\bf p}(\mathcal U))$, which we denote by
${\rm ev}^{\partial}_i$ also.
They are compatible with $(3)$.
\item
We can define an orientation of the Kuranishi structure so that it is compatible with $(3)$.
\item
${\rm ev}^{\partial}_0 \times \pi_{\mathcal U}$ is weakly submersive.
\item
The Kuranishi structure is invariant under the permutation of interior marked points.
\item
The Kuranishi structure is compatible with the forgetful map  of
the $1$st, $2$nd, \dots, $k$-th boundary marked points.
(We do not require that it is compatible with the forgetful map of the $0$-th marked point.)
\end{enumerate}
\end{lem}
The proof is the same as that of Lemma \ref{torickuranishi}.
We note that (7) is a consequence of (2) if we take $\epsilon$
in Properties \ref{Uproper} (2) to be small enough.

\begin{lem}\label{toricmmultsectU}
There exists a system of multisections on
$\mathcal M_{k+1;\ell}(L(\text{\bf u});\beta;\text{\bf p}(\mathcal U))$
with the following properties.
\begin{enumerate}
\item They are transversal to $\text{\bf 0}$.
\item They coincide with the multisection in Lemma $\ref{toricmmultsect}$
on the induced Kuranishi charts of $\pi_{\mathcal U}^{-1}(\text{\bf 0})$.
\item They are compatible with the description of
the boundary in Lemma $\ref{torickuranishi}$ $(3)$.
\item The restriction of ${\rm ev}^{\partial}_0 \times \pi_{\mathcal U}$ to the zero set of this
multisection is a submersion.
\item They are invariant under the permutation of interior marked points.
\item They are compatible with the forgetful map of
the $1$st, $2$nd, \dots, $k$-th boundary marked  points.
\end{enumerate}
\end{lem}
The proof is mostly the same as the proof of Lemma \ref{toricmmultsect}.
We only observe that (4) is a consequence of (2) if
$\epsilon$ is sufficiently small.
\par
For each $a = 1,\dots,B$, we choose a compactly supported smooth differential form
$\chi_a$ of top
degree on $\mathcal U_a$ such that $\int_{\mathcal U_a} \chi_a = 1$.
For  $\text{\bf p} : \{1,\dots,\ell\} \to \underline B$, we put
$$
\chi_{\text{\bf p}} = \prod_{i=1}^{\ell} \chi_{\text{\bf p}(i)} \in
\Omega(\mathcal U(\text{\bf p})).
$$
Let $h_1,\dots,h_k \in \Omega(L(\text{\bf u}))$.
We then define a differential form
on $L(\text{\bf u})$ by
\begin{equation}\label{q45Uformula}
\mathfrak q_{\ell,k;\beta}^{S}(\text{\bf p};h_1,\ldots,h_k)
= ({\rm ev}^{\partial}_0)! ({\rm ev}^{\partial}_1,\ldots,{\rm ev}^{\partial}_k
,\pi_{\mathcal U})^*(h_1 \wedge \cdots \wedge h_k \wedge \chi_{\text{\bf p}}),
\end{equation}
where we use the evaluation map
$$
({\rm ev}^{\partial}_0,\dots,{\rm ev}^{\partial}_k,\pi_{\mathcal U})
: \mathcal M_{k+1;\ell}(L(\text{\bf u});\beta;\text{\bf p}(\mathcal U))
\to L(\text{\bf u})^{k+1} \times \mathcal U(\text{\bf p})
$$
and $({\rm ev}_0)!$ is the integration along the fibers.
(See \cite[Definition C.1]{fooo:toric1}.)
Here the superscript $S$ stands for smoothing.
By Lemma \ref{toricmmultsectU} (4), the integration along the fibers is
well-defined. By Lemma \ref{toricmmultsectU} (5)
the operators $\mathfrak q_{\ell,k;\beta}^{S}$ is invariant
under the permutation of the components of $\text{\bf p}$.
Therefore by the $\C$-linearity we define
\begin{equation}\label{eq:mapqS}
\frak q^S_{\ell,k;\beta}:
E_{\ell} (\mathcal H[2]) \otimes B_k(\Omega(L(\text{\bf u}))[1]) \to  \Omega(L(\text{\bf u}))[1].
\end{equation}
Using these operators in the same way as in Definition \ref{deformedqdef}
to define
$\{\frak m_{k}^{S,\text{\bf b}}\}$
for
$\text{\bf b} = (\frak b_0,\text{\bf b}_{2;1},\frak b_+,b_+)$.
Thus we have obtained a filtered $A_{\infty}$ algebra
$(CF_{\text{\rm dR}}(L(\text{\bf u});\Lambda_0),\{\frak m_k^{S,\text{\bf b}}
\}_{k=0}^{\infty})$. Here we recall
$$
CF_{\text{\rm dR}}(L(\text{\bf u});\Lambda_0)
= \Omega(L(\text{\bf u}))\widehat{\otimes} \Lambda_0.
$$

\begin{lem}\label{piso1} The filtered $A_\infty$ algebra
$(CF_{\text{\rm dR}}(L(\text{\bf u});\Lambda_0),\{\frak m_k^{S,\text{\bf b}}
\}_{k=0}^{\infty})$
is pseudo-isotopic to
$(CF_{\text{\rm dR}}(L(\text{\bf u});\Lambda_0),\{\frak m_k^{T,\text{\bf b}}
\}_{k=0}^{\infty})$
as a unital filtered $A_{\infty}$ algebra.
\end{lem}

\begin{proof}
Let $\delta^a_{\text{\bf 0}}$ be the distributional $d_a$-form on $\mathcal U_a$
supported at the origin $\text{\bf 0}$ and satisfy $\int \delta^a_{\text{\bf 0}} = 1$.
(Namely it is the delta function times the volume form.)
We also take a distributional $(d_a -1)$-form $\kappa_a$ on $\mathcal U_a$
with the following properties.
\begin{proper}
\begin{enumerate}
\item $d\kappa_a = \chi_a - \delta^a_{\text{\bf 0}}$.
\item $\kappa_a$ is smooth outside the origin.
\end{enumerate}
\end{proper}
We put
\begin{equation}
\chi^s_a = s \chi_a + (1-s)  \delta^a_{\text{\bf 0}}
\end{equation}
and
$$
\chi_{\text{\bf p}}^s = \prod_{i=1}^{\ell} \chi^s_{\text{\bf p}(i)}
$$
which defines a distributional $\sum d_{\text{\bf p}(i)}$ form on $\mathcal U(\text{\bf p})$.
We then define
\begin{equation}
\mathfrak q_{\ell,k;\beta}^{S,s}(\text{\bf p};h_1,\ldots,h_k)
= ({\rm ev}^{\partial}_0)! ({\rm ev}^{\partial}_1,\ldots,{\rm ev}^{\partial}_k)^*(h_1 \wedge \cdots \wedge h_k \wedge \chi_{\text{\bf p}}^s).
\label{q45Usformula}\end{equation}
Note that because $\chi_{\text{\bf p}}^s$ is only a distributional form,
existence of the pull-back thereof is not automatic. However
using Lemma \ref{toricmmultsectU} (4) we can show that the pull-back exists and
the right hand side of (\ref{q45Usformula}) gives rise to a smooth differential form.
\par
By extending the definition \eqref{q45Usformula} multi-linearly, we obtain the homomorphism
\begin{equation}\label{eq:mapqSs}
\frak q^{S,s}_{\ell,k;\beta}:
E_{\ell} (\mathcal H[2]) \otimes B_k(\Omega(L(\text{\bf u}))[1]) \to  \Omega(L(\text{\bf u}))[1].
\end{equation}
We use them to define $\{\frak m_k^{S,s,\text{\bf b}}\}$ in the same way as in Definition \ref{deformedqdef}.
Then it is smooth (with respect to $s$ coordinate) in the sense of
Definition \ref{def:smoothop}.
\par
\begin{sublem}
$(CF_{\text{\rm dR}}(L(\text{\bf u});\Lambda_0),\{\frak m_k^{S,s,\text{\bf b}}
\}_{k=0}^{\infty})$
is a unital filtered $A_{\infty}$ algebra.
Moreover we have:
$$
\frak m_k^{S,0,\text{\bf b}} = \frak m_k^{T,\text{\bf b}},
\qquad
\frak m_k^{S,1,\text{\bf b}} = \frak m_k^{S,\text{\bf b}}.
$$
\end{sublem}
The proof is easy and  omitted.
\par
We next denote
$$
\kappa_{i,\text{\bf p}}^s
=
\chi^s_{\text{\bf p}(1)} \wedge
\cdots \wedge \chi^s_{\text{\bf p}(i-1)}
\wedge
\kappa_{\text{\bf p}(i)}
\wedge
\chi^s_{\text{\bf p}(i+1)} \wedge
\cdots \wedge \chi^s_{\text{\bf p}(\ell)}
$$
and define
\begin{equation}
\aligned
&\mathfrak{qc}_{\beta;\ell,k}^{S,s}(\text{\bf p};h_1,\ldots,h_k) \\
&=
\sum_{i=1}^{\ell} (-1)^{*(i)} ({\rm ev}^{\partial}_0)! 
({\rm ev}^{\partial}_1,\ldots,{\rm ev}^{\partial}_k)^*(h_1 \wedge \cdots \wedge h_k \wedge \kappa_{i,\text{\bf p}}^s).
\endaligned
\label{qc45Usformula}\end{equation}
\par
See Remark \ref{remarksign} for the sign.
In the same way as the operator $\frak q^{S,s}_{\ell,k;\beta}$ defines  $ \frak m_k^{S,s,\text{\bf b}}$,
the operator $\mathfrak{qc}_{\beta;\ell,k}^{S,s}$ induces an operator, which
we write $ \frak c_k^{S,s,\text{\bf b}}$.
Using Stokes' formula (\cite[Lemma C.9]{fooo:toric1}) and Composition formula (\cite[Lemma C.10]{fooo:toric1}),
together with Lemmas \ref{torickuranishiU}, \ref{toricmmultsectU}, we easily derive that
$$
(CF_{\text{\rm dR}}(L(\text{\bf u});\Lambda_0),\{\frak m_k^{S,s,\text{\bf b}}
\}_{k=0}^{\infty},\{\frak c_k^{S,s,\text{\bf b}}
\}_{k=0}^{\infty})
$$
is the required pseudo-isotopy.
The proof of Lemma \ref{piso1} is complete.
\end{proof}

\subsection{Wrap-up of the proof of Proposition \ref{Ainfinityequiv}.}
\label{subsec:CompletionProposition}

In this subsection, we construct a pseudo-isotopy between
$(CF_{\text{\rm dR}}(L(\text{\bf u});\Lambda_0),\{\frak m_k^{S,\text{\bf b}}
\}_{k=0}^{\infty})$
(which is defined in Subsection \ref{subsec:Smoothing})
and
$(CF_{\text{\rm dR}}(L(\text{\bf u});\Lambda_0),\{\frak m_k^{\text{\bf b}}
\}_{k=0}^{\infty})$
(which is defined in Definition \ref{deformedqdef}.)
\par
Together with Theorem  \ref{pisotohomotopyequiv} and
Lemma \ref{piso1},
this will complete the proof of Proposition \ref{Ainfinityequiv}.
\par
In Section \ref{sec:appendix2}, we already proved that
the homotopy equivalence class of
$$
(CF_{\text{\rm dR}}(L(\text{\bf u});\Lambda_0),\{\frak m_k^{\text{\bf b}}
\}_{k=0}^{\infty})
$$
is independent of the choice of de Rham representative $\text{\bf b}$.
We make this choice more specifically below.
\par
Let $D_a$ be as in the beginning of Subsection \ref{subsec:Smoothing}.
We put
\begin{equation}\label{repbetahow}
\frak b_a = \text{\rm Exp}!  (\pi_{\mathcal U}^*\chi_a),
\end{equation}
where we use the map
$(\text{\rm Exp},\pi_{\mathcal U}) : D_a \times \mathcal U_a
\to M \times \mathcal U_a$ as a correspondence.
Clearly $\frak b_a $ is a de Rham representative of the
Poincar\'e dual to $[D_a]$.
The de Rham cohomology classes $\{[\frak b_a]\}_{a=1}^{B}$ form a basis of $\bigoplus_{k\ne 0}H^k(M;\C)$.
We use them to specify the de Rham representatives of the elements of
$\bigoplus_{k\ne 0}H^k(M;\Lambda)$.
(We represent the $0$-th cohomology class by the constant function.)
\par
We next review two Kuranishi structures
on $\mathcal M_{k+1;\ell}(L(\text{\bf u});\beta;\text{\bf p}(\mathcal U))$
and their perturbations
that will enter in the construction of
a pseudo-isotopy used in the proof of Proposition \ref{Ainfinityequiv}.
\par\medskip
\noindent{\bf (Kuranishi structure and perturbations 1)}
Consider the natural projection
\begin{equation}\label{projectionwithoutp}
\pi :
\mathcal M_{k+1;\ell}(L(\text{\bf u});\beta;\text{\bf p}(\mathcal U))
\to
\mathcal M_{k+1;\ell}(L(\text{\bf u});\beta).
\end{equation}
We have chosen and fixed a Kuranishi structure on
$\mathcal M_{k+1;\ell}(L(\text{\bf u});\beta)$ in Proposition
\ref{disckura}.
We pull it back by the map (\ref{projectionwithoutp}).
It defines a Kuranishi structure $\mathcal K_1$.
\par
In Lemma \ref{existmkulti1} we took and fixed a CF-perturbation
on $\mathcal M_{k+1;\ell}(L(\text{\bf u});\beta)$.
We pull it back by the map (\ref{projectionwithoutp})
and obtain a CF-perturbation
of the Kuranishi structure $\mathcal K_1$.
We denote it by $\widehat{\frak S_1} = \{\widehat{\frak S_1^{\epsilon}}\}$.
We use them to define
$\mathfrak q_{\ell,k;\beta}^{\widehat{\frak S_1^{\epsilon}}}$ by
\begin{equation}
\mathfrak q_{\ell,k;\beta}^{\widehat{\frak S_1^{\epsilon}}}(\text{\bf p};h_1,\ldots,h_k)
= ({\rm ev}^{\partial}_0)! (({\rm ev}^{\partial}_1,\ldots,{\rm ev}^{\partial}_k
,\pi_{\mathcal U})^*(h_1 \wedge \cdots \wedge h_k \wedge \chi_{\text{\bf p}});\widehat{\frak S_1^{\epsilon}}),
\label{q45frasformula}\end{equation}
where
we use the evaluation map
$$
({\rm ev}^{\partial}_0,\dots,{\rm ev}^{\partial}_k,\pi_{\mathcal U})
: \mathcal M_{k+1;\ell}(L(\text{\bf u});\beta;\text{\bf p}(\mathcal U))
\to L(\text{\bf u})^{k+1} \times \mathcal U(\text{\bf p}),
$$
and the CF-perturbation $\widehat{\frak S_1}$ to define integration along
the fiber in the right hand side. See \cite[Definitions 7.78 and 9.13]{fooo:tech2}.
\begin{lem} Let $\mathfrak q_{\ell,k;\beta}$ be the map defined in $(\ref{defqformula})$.
Then we have
$\mathfrak q_{\ell,k;\beta}^{\widehat{\frak S_1^{\epsilon}}}
= \mathfrak q_{\ell,k;\beta}$.
\end{lem}
This lemma is obvious from the definition and (\ref{repbetahow}).
\par\medskip
\noindent{\bf (Kuranishi structure and perturbations 2)}
In Lemma \ref{torickuranishiU},
we equipped $\mathcal M_{k+1;\ell}(L(\text{\bf u});\beta;\text{\bf p}(\mathcal U))$
with a Kuranishi structure.
We name the Kuranishi structure $\mathcal K_2$.
In Lemma \ref{toricmmultsectU},
we took a multisection of $\mathcal K_2$.
We name the multisection $\frak s_2$.
They determine the operators
$\mathfrak q_{\ell,k;\beta}^{S}$ by (\ref{q45Uformula}).
\par\smallskip
Thus we have described two systems of Kuranishi structures and perturbations.
We next define a system of
Kuranishi structures and CF-perturbations on
$[0,1]\times \mathcal M_{k+1;\ell}(L(\text{\bf u});\beta;\text{\bf p}(\mathcal U))$
which interpolate the two systems.
\par
We define
$$
\widehat{\text{\rm ev}}_{i}^{\partial}
: [0,1]\times \mathcal M_{k+1;\ell}(L(\text{\bf u});\beta;\text{\bf p}(\mathcal U))
\to [0,1] \times L(\text{\bf u})
$$
by
$\widehat{\text{\rm ev}}_{i}^{\partial} = (\pi_s, {\text{\rm ev}}_{i}^{\partial})$
where $\pi_s$ is the projection to the $[0,1]$ factor.
(We use $s$ as the coordinate of this factor.)

\begin{lem}\label{torickuranishiU01}
\begin{enumerate}
\item
$[0,1]\times \mathcal M_{k+1;\ell}(L(\text{\bf u});\beta;\text{\bf p}(\mathcal U))$ has
a Kuranishi structure with corners.
\item
Its restriction to
$\{ 0\} \times \mathcal M_{k+1;\ell}(L(\text{\bf u});\beta;\text{\bf p}(\mathcal U))$ coincides with $\mathcal K_1$
and its restriction to
$\{ 1 \} \times \mathcal M_{k+1;\ell}(L(\text{\bf u});\beta;\text{\bf p}(\mathcal U))
$ coincides with $\mathcal K_2$.
\item
Its normalized boundary is described by the union of
$$\partial ([0,1]) \times \mathcal M_{k_1+1;\#\L_1}(L(\text{\bf u});\beta_1;\text{\bf p}_1(\mathcal U)).
$$
and the union of fiber products
\begin{equation}\label{28211}
\aligned
&\Big( [0,1]\times \mathcal M_{k_1+1;\#\L_1}(L(\text{\bf u});\beta_1;\text{\bf p}_1(\mathcal U))
\Big) \\
&{}_{\widehat{\rm ev}^{\partial}_0}\times_{\widehat{\rm ev}^{\partial}_i}
\Big([0,1]\times \mathcal M_{k_2+1;\#\L_2}(L(\text{\bf u});\beta_2;\text{\bf p}_2(\mathcal U))\Big)
\endaligned
\end{equation}
where the union is taken over all $(\L_1,\L_2) \in \text{\rm Shuff}(\ell)$,
$k_1,k_2$ with $k_1 + k_2 = k$ and $\beta_1,\beta_2 \in H_2(M,L(\text{\bf u});\Z)$ with
$\beta = \beta_1+\beta_2$. We put $
\text{\rm Split}((\mathbb L_1,\mathbb L_2),\text{\bf p}) = (\text{\bf
p}_1,\text{\bf p}_2).
$
\item
The (virtual) dimension is given by
\begin{equation}
\aligned
&\dim \mathcal M_{k+1;\ell}(L(\text{\bf u});\beta;\text{\bf p}(\mathcal U))
\\
&= n + \mu_{L(\text{\bf u})}(\beta) + k - 1 + 2\ell - \sum_{i=1}^{\ell} 2\deg D_{\text{\bf p}(i)}
+ \sum_{i=1}^{\ell} d_{\text{\bf p}(i)}.
\endaligned
\end{equation}
\item The evaluation map $\widehat{\rm ev}^{\partial}_i$ at any boundary marked point of
$\mathcal M_{k+1;\ell}(L(\text{\bf u});\beta)$
defines a map on $\mathcal M_{k+1;\ell}(L(\text{\bf u});\beta;\text{\bf p}(\mathcal U))$, which we also denote by
$\widehat{\rm ev}^{\partial}_i$.
It is compatible with $(3)$.
\item
We can define  a coherent system of orientations on the Kuranishi structures
 so that it is compatible with $(3)$.
\item
$\widehat{\rm ev}^{\partial}_0 \times \pi_{\mathcal U}$ is weakly submersive.
\item
The Kuranishi structure is invariant under the permutation of interior marked points.
\item
The Kuranishi structure is compatible with the forgetful map  of
the $1$st, $2$nd, \dots, $k$-th boundary marked points.
(We do not require that it is compatible with the forgetful map of the $0$-th marked point.)
\end{enumerate}
\end{lem}
The proof is the same as that of Lemma \ref{torickuranishi} and is omitted.
\begin{lem}\label{toricmmultsectU01}
There exists a system of CF-perturbations
$\widehat{\frak S} = \{\widehat{{\frak S}^{\epsilon}}\}$
of the Kuranishi structure on
$[0,1]\times \mathcal M_{k+1;\ell}(L(\text{\bf u});\beta;\text{\bf p}(\mathcal U))$
given in Lemma $\ref{torickuranishiU01}$
with the following properties:
\begin{enumerate}
\item They are transversal to $0$.
\item Its restriction to
$\{ 0\} \times \mathcal M_{k+1;\ell}(L(\text{\bf u});\beta;\text{\bf p}(\mathcal U))
$ coincides with $\widehat{\frak S_1}$
and its restriction to
$\{ 1\} \times \mathcal M_{k+1;\ell}(L(\text{\bf u});\beta;\text{\bf p}(\mathcal U))$ coincides with $\frak s_2$.
\item They are compatible with the description of
the boundary in Lemma $\ref{torickuranishiU01}$ $(3)$.
\item The map
$\widehat{\rm ev}^{\partial}_0 \times \pi_{\mathcal U}$ is
strongly submersive with respece to $\widehat{\frak S}$.
\item They are invariant under the permutation of interior marked points.
\item They are compatible with the forgetful map of
the $1$st, $2$nd, \dots, $k$-th boundary marked points.
\end{enumerate}
\end{lem}
As we show in Remark \ref{multiseccont} we can identify a multisection as a special
case of the CF-perturbation.
Taken this fact and the existence theorem of CF-perturbation
(\cite[Theorem 7.49]{fooo:tech2}) into account,
the proof is the same as that of Lemma \ref{toricmmultsect}
and so omitted.
\par
We now define operators \index{$\frak q^{para}_{\ell,k;\beta}$}
$$
\frak q^{para}_{\ell,k;\beta} :
E_{\ell} (\mathcal H[2]) \otimes B_k(\Omega(L(\text{\bf u}))[1]) \to  \Omega([0,1]\times L(\text{\bf u}))[1]
$$
by
\begin{equation}\label{q45Uformula01}
\mathfrak q_{\ell,k;\beta}^{para}(\text{\bf p};h_1,\ldots,h_k)
= (\widehat{\rm ev}^{\partial}_0)! ((\widehat{\rm ev}^{\partial}_1,\ldots,
\widehat{\rm ev}^{\partial}_k
,\pi_{\mathcal U})^*(h_1 \wedge \cdots \wedge h_k \wedge \chi_{\text{\bf p}}); \widehat{\frak S^{\epsilon}}),
\end{equation}
where we use the evaluation map
$$
(\widehat{\rm ev}^{\partial}_0,\dots,\widehat{\rm ev}^{\partial}_k,\pi_{\mathcal U})
: \mathcal M_{k+1;\ell}(L(\text{\bf u});\beta;\text{\bf p}(\mathcal U))
\to ([0,1]\times L(\text{\bf u}))^{k+1} \times \mathcal U(\text{\bf p})
$$
and $(\widehat{\rm ev}_0)!$ is the integration along the fibers
defined via the CF-perturbation $\widehat{\frak S^{\epsilon}}$.
(Definition \ref{def320222}, \cite[Definition 7.78]{fooo:tech2}.)
\par
We decompose $\mathfrak q_{\ell,k;\beta}^{para}$ into the sum of the form which does
not contain $ds$ and one which contains $ds$ and write:
$$
\mathfrak q_{\ell,k;\beta}^{para}
= \mathfrak q_{\ell,k;\beta}^{para,1}
+ ds \wedge \mathfrak q_{\ell,k;\beta}^{para,2}:
$$
More specifically, the two summands are given by
$$
\mathfrak q_{\ell,k;\beta}^{para,2} = \frac{\del}{\del s} \rfloor \mathfrak q_{\ell,k;\beta}^{para},
\quad \mathfrak q_{\ell,k;\beta}^{para,1} = \mathfrak q_{\ell,k;\beta}^{para} - ds \wedge \mathfrak q_{\ell,k;\beta}^{para,2}.
$$
Now we put
\begin{equation}\label{mkdefeqP}
\aligned
&\frak m_{k}^{\text{\bf b}}(x_1,\ldots,x_k) \\
&= \sum_{\beta\in H_2(M,L:\Z)}
\sum_{\ell=0}^{\infty}\sum_{m_0=0}^{\infty}\cdots
\sum_{m_k=0}^{\infty}T^{\omega\cap \beta}
\frac{\exp(\text{\bf b}_{2;1} \cap \beta)}{\ell!}\\
&\hskip2cm
\frak
q^{para,1}_{\ell,k+\sum_{i=0}^k m_i;\beta}(\frak b_{+}^{\otimes\ell};
b_{+}^{\otimes m_0},x_1,b_{+}^{\otimes m_1},\ldots,
b_{+}^{\otimes m_{k-1}},x_k,b_{+}^{\otimes m_k}),
\endaligned
\end{equation}
\begin{equation}\label{mkdefecP}
\aligned
&\frak c_{k}^{\text{\bf b}}(x_1,\ldots,x_k) \\
&= \sum_{\beta\in H_2(M,L:\Z)}
\sum_{\ell=0}^{\infty}\sum_{m_0=0}^{\infty}\cdots
\sum_{m_k=0}^{\infty}T^{\omega\cap \beta}
\frac{\exp(\text{\bf b}_{2;1} \cap \beta)}{\ell!}\\
&\hskip2cm
\frak
q^{para,2}_{\ell,k+\sum_{i=0}^k m_i;\beta}(\frak b_{+}^{\otimes\ell};
b_{+}^{\otimes m_0},x_1,b_{+}^{\otimes m_1},\ldots,
b_{+}^{\otimes m_{k-1}},x_k,b_{+}^{\otimes m_k}).
\endaligned
\end{equation}
They define maps from
$B_k(\Omega(L(\text{\bf u})) \widehat{\otimes} \Lambda)$ to
$(\Omega([0,1]\times L(\text{\bf u})))\widehat{\otimes} \Lambda
$.
By Lemmas \ref{torickuranishiU01},\ref{toricmmultsectU01},
$\frak m_{k}^{\text{\bf b}}$ and $\frak c_{k}^{\text{\bf b}}$ define a unital pseudo-isotopy
between
$(CF_{\text{\rm dR}}(L(\text{\bf u});\Lambda_0),\{\frak m_k^{S,\text{\bf b}}
\}_{k=0}^{\infty})$
and
$(CF_{\text{\rm dR}}(L(\text{\bf u});\Lambda_0),\{\frak m_k^{\text{\bf b}}
\}_{k=0}^{\infty})$.
The proof of Proposition \ref{Ainfinityequiv} is now complete.
\qed
\medskip

\begin{rem}\label{remarksign}
 We can handle the sign needed in the argument of this section by
the same way as in \cite{fooo:toric1}.
(See the end of  \cite[Appendix C]{fooo:toric1} for the relevant explanation.)
\end{rem}

\begin{rem}\label{multiseccont}
In this remark we explain how we can regard a multisection as a CF-perturbation.
Let $X$ be a space with Kuranishi structure and
${\widetriangle{\mathcal U}} = \{\mathcal U_{\frak p} \mid \frak p \in \frak P\}$ be
a good coordinate system of $X$.
Here $\mathcal U_{\frak p} = (U_{\frak p},\mathcal E_{\frak p},\psi_{\frak p},s_{\frak p})$
is a Kuranishi chart. (Namely $U_{\frak p}$ is an orbifold, $\mathcal E_{\frak p}$ a vector
bundle on it, $s_{\frak p}$ its section, and $\psi_{\frak p}$ is a homeomorphism
from $s_{\frak p}^{-1}(0)$ to an open set of $X$.)
We suppress the coordinate change from the notation of Kuranishi chart for simplicity.
Let $f : (X,\widetriangle{\mathcal U}) \to N$ be a weakly submersive map to a manifold $N$.

We recall that a CF-perturbation is a continuous family parametrized by $\epsilon>0$,
while a multisection does not involve such a parameter.
We just call a multisection trasversal to $\{0\}$ a transveral multisection.
Based on the consturction of a transveral multisection described in \cite[Section 13]{fooo:tech2}, we
can construct an $\epsilon$-parametrized family of multisections. We may regard it as
a multisection defined on $(X \times [0,1),{\widetriangle{\mathcal U}}\times [0,1))$
that restricts to the Kuranishi map at $X \times \{0\}$.
Moreover we may assume that it is transversal to $\{0\}$ on $X \times (0,1)$.
Then for each generic $\epsilon > 0$ the restriction $\frak s_{\epsilon}$ of
$\frak s$ to $X \times \{\epsilon\}$ defines a transversal multisection.
In the situation where we can take a multisection so that the restriction of $f$
to the zero set of the multisection on $(X,\widetriangle{\mathcal U})$
is a submersion, we may require in addition that the restriction of $f$
to the zero set of $\frak s_{\epsilon}$ becomes
a submersion for generic $\epsilon$. This is the case such as that of Lemma \ref{toricmmultsectU01},
Lemma \ref{toricmmultsect}, or the case of $\dim N =0$ (for example (\ref{boundaryop})).

In the latter situation, we will associate to the above constructed multisection a CF-perturbation
$\widetriangle{\mathfrak S}$ of $\widetriangle{\mathcal U}$
with the following properties:
\begin{enumerate}
\item $\widetriangle{\mathfrak S}$ is transversal to $0$.
\item $f$ is strongly submersive with respect to  $\widetriangle{\mathcal U}$.
\item
Let $h$ be a differential form on $(X,\widetriangle{\mathcal U})$
in the sense of Definition \ref{def323232}.
Then for a generic choice of sufficiently small $\epsilon$ we have
\begin{equation}\label{2827}
f!(h;\widetriangle{\mathfrak S^{\epsilon}})
=
f!(h;\mathfrak s_{\epsilon}).
\end{equation}
Here the left hand side is the integration along the fibers defined by
using the CF-perturbation $\widetriangle{\mathfrak S}$ (See Definition \ref{def320222},
\cite[Definition 7.78]{fooo:tech2})
and the right hand side the integration along the fibers
on the zero set of $\frak s_{\epsilon}$, which is defined in
\cite[Definition C.1]{fooo:toric1}.
\end{enumerate}

Now construction of the CF-perturbation $\widetriangle{\mathfrak S}$
is in order.

First, we note that the given multisection $\frak s$ by definition
gives a multisection $\frak s_{\frak p}$ of  $\mathcal U_{\frak p}
\times [0,1)$ for each $\frak p$.
We take an open coverging $U_{\frak p} = \bigcup_{a}U_{\frak p,a}$ so that each $a$
is associated to orbifold charts $(V_{\frak p,a},\Gamma_{\frak p,a},\phi_{\frak p,a})$
inducing a homeomorphism $V_{\frak p,a}/\Gamma_{\frak p,a} \cong
U_{\frak p,a}$.
We also take a trivialization $(V_{\frak p,a} \times E_{\frak p,a})/
\Gamma_{\frak p,a}  \cong
\mathcal E_{\frak p}\vert_{U_{\frak p,a}}
$.
In this coordinate, the multisection $\frak s_{\frak p}$ induces a map
$$
\frak s_{\frak p,a} = (\frak s_{\frak p,a,i})_{i=1}^{\ell_{\frak p,a}}
: V_{\frak p,a} \times [0,1) \to E_{\frak p,a}^{\ell_{\frak p,a}}
$$
together with a group homomorphism $\sigma_{\frak p,a} : \Gamma_{\frak p,a}
\to {\rm Perm}({\ell_{\frak p,a}})$ to the permutation group ${\rm Perm}({\ell_{\frak p,a}})$ of
order $\ell_{\frak p,a} !$ such that for each $i$ it satisfies
$$
\frak s_{\frak p,a,i}(\gamma x,\epsilon) = \frak s_{\frak p,a,\sigma_{\frak p,a}(\gamma)(i)}(x,\epsilon)
$$
for $x \in  V_{\frak p,a}$, $\epsilon \in [0,1)$ and $\gamma \in \Gamma_{\frak p,a}$.

We use this map to define a CF-perturbation $\mathcal S_{\frak p,a}
= (W'_{\frak p,a},\omega_{\frak p,a},\frak s_{\frak p,a})$ in terms of the given
chart so that it satisfies the conditions given in \cite[Definition 7.8]{fooo:tech2}
as follows.

We take an $\ell_{\frak p,a}$ dimensional vector space $W_{\frak p,a}=\R^{\ell_{\frak p,a}}$
and $\Gamma_{\frak p,a}$ acts on $W_{\frak p,a}$ by  $\sigma_{\frak p,a}$.
Let $w_i = (0,\dots,0,1,0,\dots,0) \in W_{\frak p,a}$ with the $i$-th component $1$.
We take a small neighborhood $W_i$ of $w_i$ in $W_{\frak p,a}$
and let $W'_{\frak p,a}$ be the union of $W_i$, $i=1,\dots,\ell_{\frak p,a}$,
which is $\Gamma_{\frak p,a}$ invariant.
We define a map
$
\frak s_{\frak p,a} : V_{\frak p,a} \times [0,1)\times  W'_{\frak p,a}
\to E_{\frak p,a}
$
by setting
$$
\frak s_{\frak p,a}(x,\epsilon,w) = \frak s_{\frak p,a,i}(x,\epsilon)
$$
if $w \in W_i$.
This map is $\Gamma_{\frak p,a}$-equivariant.
We take a compactly supported $\Gamma_{\frak p,a}$-invariant smooth $\ell_{\mathfrak p,a}$-form $\omega_{\frak p,a}$ on
$ W'_{\frak p,a} $ satisfying
$
\int_{W_i} \omega_{\frak p,a} = \frac{1}{\ell_{\frak p,a}}
$
and put
$$
\mathcal S_{\frak p,a} = (W'_{\frak p,a},\omega_{\frak p,a},\frak s_{\frak p,a}).
$$
Then $\mathcal S_{\frak p,a}$ is a CF-perturbation and,
by construction, the equality (\ref{2827}) holds  if $h$ is supported on $U_{\frak p,a}$.

Again by construction we may take them so that the restriction of $\mathcal S_{\frak p,a}$
to $U_{\frak p,a} \cap U_{\frak p,a'}$ is equivalent to the restriction of $\mathcal S_{\frak p',a}$
to $U_{\frak p,a} \cap U_{\frak p,a'}$
in the sense of \cite[Definition 7.16]{fooo:tech2}. (See also Definition \ref{defn32321}.)
 We thus obtain a CF-perturbation
$\mathfrak S_{\frak p}$ of $\mathcal U_{\frak p}$ for each $\frak p$.
By construction, $\{\mathfrak S_{\frak p}\}$ is compatible with coordinate change
in the sense of \cite[Definition 7.47]{fooo:tech2}
(that is, Definition \ref{defn355555} (1),(2) hold)
and  so define a CF-perturbation $\widetriangle{\mathfrak S}$ of $ (X,\widetriangle{\mathcal U})$.
The CF-perturbation $\widehat{\frak S}$ of our
Kuranishi structure is induced from $\widetriangle{\mathfrak S}$ by
\cite[Lemma 9.11]{fooo:tech2}.
Then (\ref{2827}) follows from the fact that $\mathcal S_{\frak p,a}$ satisfies
the equality in case $h$ is in one chart.
\end{rem}

\subsection{Proof of Lemma \ref{ainfinityequivcomp}}
\label{subsec:ainfinityequivcomp}

In this subsection we prove Lemma \ref{ainfinityequivcomp}.
Let
$$
\frak m^{1,\text{\bf b}}_{k} :
B_k((\Omega([0,1]\times L(\text{\bf u}))
\widehat{\otimes} \Lambda) [1]) \to  (\Omega([0,1]\times L(\text{\bf u}))
\widehat{\otimes} \Lambda) [1]
$$
be the filtered $A_{\infty}$ structure
induced from the pseudo-isotopy defined in the proof of Lemma \ref{piso1}.
Let
$$
\frak m^{2,\text{\bf b}}_{k} :
B_k((\Omega([0,1]\times L(\text{\bf u}))
\widehat{\otimes} \Lambda) [1]) \to  (\Omega([0,1]\times L(\text{\bf u}))
\widehat{\otimes} \Lambda) [1]
$$
be the filtered $A_{\infty}$ structure
induced from the pseudo-isotopy constructed in Subsection \ref{subsec:CompletionProposition}.
\par
They induce chain complexes
$$
((\Omega([0,1]\times L(\text{\bf u}))
\widehat{\otimes} \Lambda),\frak m^{1,\text{\bf b}}_{1}),
\quad
((\Omega([0,1]\times L(\text{\bf u}))
\widehat{\otimes} \Lambda),\frak m^{1,\text{\bf b}}_{2}).
$$
We have chain homotopy equivalences
$$
\aligned
&
\text{\rm Eval}_{s=0}:
 ((\Omega([0,1]\times L(\text{\bf u}))
\widehat{\otimes} \Lambda),\frak m^{1,\text{\bf b}}_{1})
\to
(\Omega(L(\text{\bf u})
\widehat{\otimes} \Lambda),\frak m^{T,\text{\bf b}}_{1}), \\
&
\text{\rm Eval}_{s=1}:
 ((\Omega([0,1]\times L(\text{\bf u}))
\widehat{\otimes} \Lambda),\frak m^{1,\text{\bf b}}_{1})
\to
(\Omega(L(\text{\bf u})
\widehat{\otimes} \Lambda),\frak m^{S,\text{\bf b}}_{1}),
\endaligned
$$
and
$$
\aligned
&
\text{\rm Eval}_{s=0}:
 ((\Omega([0,1]\times L(\text{\bf u}))
\widehat{\otimes} \Lambda),\frak m^{2,\text{\bf b}}_{1})
\to
(\Omega(L(\text{\bf u})
\widehat{\otimes} \Lambda),\frak m^{S,\text{\bf b}}_{1}), \\
&
\text{\rm Eval}_{s=1}:
 ((\Omega([0,1]\times L(\text{\bf u}))
\widehat{\otimes} \Lambda),\frak m^{2,\text{\bf b}}_{1})
\to
(\Omega(L(\text{\bf u})
\widehat{\otimes} \Lambda),\frak m^{\text{\bf b}}_{1}),
\endaligned
$$
that are defined by (\ref{Eval}).
\par
Therefore to prove Lemma \ref{ainfinityequivcomp}
it suffices to construct chain maps:
$$\aligned
&i_{\text{\rm qm},\text{\bf b}}^S
: \Omega(M) \widehat{\otimes} \Lambda
\to \Omega(L(\text{\bf u})
\widehat{\otimes} \Lambda;\frak m_1^{T,\text{\bf b}}), \\
&i_{\text{\rm qm},\text{\bf b}}^1
: \Omega(M) \widehat{\otimes} \Lambda
\to  ((\Omega([0,1]\times L(\text{\bf u}))
\widehat{\otimes} \Lambda);\frak m^{1,\text{\bf b}}_{1}), \\
&i_{\text{\rm qm},\text{\bf b}}^2
: \Omega(M) \widehat{\otimes} \Lambda
\to  ((\Omega([0,1]\times L(\text{\bf u}))
\widehat{\otimes} \Lambda);\frak m^{2,\text{\bf b}}_{1}),
\endaligned
$$
that satisfy
$$
\aligned
\text{\rm Eval}_{s=0}\circ i_{\text{\rm qm},\text{\bf b}}^1
&= i_{\text{\rm qm},\text{\bf b}}^T, \qquad
\text{\rm Eval}_{s=1}\circ i_{\text{\rm qm},\text{\bf b}}^1
= i_{\text{\rm qm},\text{\bf b}}^S, \\
\text{\rm Eval}_{s=0}\circ i_{\text{\rm qm},\text{\bf b}}^2
&= i_{\text{\rm qm},\text{\bf b}}^S, \qquad
\text{\rm Eval}_{s=1}\circ i_{\text{\rm qm},\text{\bf b}}^2
= i_{\text{\rm qm},\text{\bf b}}.
\endaligned
$$
We can construct such
$i_{\text{\rm qm},\text{\bf b}}^S$,
$i_{\text{\rm qm},\text{\bf b}}^1$,
$i_{\text{\rm qm},\text{\bf b}}^2$
by modifying the definition of
$i_{\text{\rm qm},\text{\bf b}}$ (\ref{iqmdefformula}) in an obvious way.
The proof of Lemma \ref{ainfinityequivcomp} is complete.
\qed
\medskip

\section{Seidel homomorphism with bulk}
\label{sec:appendix4} \index{Seidel homomorphism with bulk}

In this section we generalize the Seidel homomorphism \cite{seidel:auto}
to a version with bulk deformation. We then generalize, in the next section, the results by Entov-Polterovich
\cite[Section 4]{EP:morphism}  and by McDuff-Tolman \cite{mc-tol} on  the relationship
between
the Seidel homomorphism and the Calabi quasi-morphism.
These generalizations are rather straightforward and do not
require novel ideas.

\subsection{Seidel homomorphism with bulk}
\label{subsec:seibulk}

In this subsection, we present a version of Seidel's construction
\cite{seidel:auto} that incorporates bulk deformations.

Let $H$ be a one-periodic Hamiltonian such that $\phi_H:[0,1] \to \Ham(M,\omega)$
defines a loop, i.e. satisfies $\psi_H = \phi_H^1 =id$. Such a loop is
called a \emph{Hamiltonian loop}. For such $H$, there is a
diffeomorphism $M \to \text{\rm Per}(H)$: For each point $p \in M$, we
assign a periodic orbit defined by
\begin{equation}\label{perppsi}
z_p^H(t) = \phi_H^t(p).
\end{equation}
Then the map
$
p \mapsto z_p^H
$
is a one-to-one correspondence $M \to \text{\rm Per}(H)$.
\par
Let $v : \R \times S^1 \to M$ be any continuous map.
We define $u :  \R \times S^1 \to M$ by
\begin{equation}
u(\tau,t) = \phi_H^t(v(\tau,t)).
\end{equation}

\begin{lem}\label{extendcylinder}
Let $p_-,\, p_+ \in M$. Then
$$
\lim_{\tau \to -\infty} u(\tau,t) = z_{p_-}^H(t),
\quad
\lim_{\tau \to +\infty} u(\tau,t) = z_{p_+}^H(t),
$$
if and only if
$$
\lim_{\tau \to -\infty}v(\tau,t) = p_-,
\quad
\lim_{\tau \to +\infty}v(\tau,t) = p_+.
$$
\end{lem}
The proof is a straightforward calculation.
For a map $u$ satisfying the above conditions, we define a homology class
$[u] \in H_2(M;\Z)$ by setting $[u] := [v]$. (Note by the asymptotic conditions given
$v$ extends to a map from $S^2$ so
$[v] \in H_2(M;\Z)$ is defined.)
\par
We define a symplectic fibration
$$
\pi : E_{\phi_H} \to \C P^1
$$
with its fiber isomorphic to $(M,\omega)$ as follows.
Let $D_{\pm}$ be two copies of the unit disc in $\C$.
Set  $U_1=D_- \times M$, $U_2=(\R \times S^1) \times M$ and $U_3=D_+ \times M$.
We glue them by the gluing maps
$$
I_-: (-\infty, 0) \times S^1 \times M \to  D_- \setminus \{0\}   \times M, \quad
I_-((\tau,t),x)=(e^{2\pi(\tau + \sqrt{-1}t)}, x)
$$
(where we regard $S^1 = \R/\Z$,) and
$$
I_+: (1,\infty) \times S^1 \times M \to D_+ \setminus \{\infty\} \times M,\quad
I_+((\tau,t),x) = (e^{-2\pi(\tau -1+ \sqrt{-1}t)},(\phi_{H}^{t})^{-1}(x)).
$$
We thus obtain
$$
E_{\phi_H} =  U_1 \cup U_2 \cup U_3.
$$
The projections to the second factor induce a map
$$
\pi : E_{\phi_H} \to D_- \cup (\R \times S^1) \cup D_+ \cong \C P^1.
$$
This defines a locally trivial fiber bundle and
the fiber of $\pi$ is diffeomorphic to $M$.
\par
In fact,  $E_{\phi_H} \to \C P^1$ becomes a Hamiltonian fiber bundle.
See \cite{GLS} for the precise definition of
Hamiltonian fiber bundle and its associated coupling form $\Omega$ that we
use below. We also refer to \cite{schwarz,entov,oh:alan} for their
applications to the Floer theory and spectral invariants.
\par
\begin{lem}\label{lem:coupling} The fibration $E_{\phi_H} \to \C P^1$ is a
Hamiltonian fiber bundle, i.e., it carries a coupling form
$\Omega$ on $E_{\phi_H}$ such that
\begin{enumerate}
\item $\Omega$ is closed and $\Omega|_{E_{\phi_H,\gamma}} = \omega$,
\item $\pi_!\Omega^{n+1} = 0$ where $\pi_!$ is the integration over fiber and
$2n= \dim M$.
\end{enumerate}
\end{lem}
\begin{proof} On each of $U_i$, $i=1,2,3$, we pull back $\omega$ by
the projection to $M$ and denote it by $\omega_i$.
We put $\omega_2' = \omega_2 + d(\chi H dt)$.
Then we find that $\omega_1$ on $U_1$, $\omega'_2$ on $U_2$ and $\omega_3$ on $U_3$
are glued to a closed $2$-form $\Omega$ on $E_{\phi_H}$.
The normalization condition on $H$ then gives rise to
the condition $\pi_!\Omega^{n+1} = 0$.
\end{proof}

Let $u : \R \times S^1 \to M$ be a continuous map. We denote the associated
section $\widehat u: \R \times S^1 \to E_{\phi_H}$  by the formula
\begin{equation}
\widehat u(\tau,t) = ((\tau,t),u(\tau,t))
\end{equation}
on $U_2$.
\begin{lem}\label{triextendtau}
Let $u: \R \times S^1 \to M$ be a continuous map. The following is equivalent:
\begin{enumerate}
\item There exists some $p_-, \, p_+ \in M$ such that
$$
\lim_{\tau \to -\infty}u(\tau,t) = z_{p_-}^H(t),
\quad
\lim_{\tau \to +\infty}u(\tau,t) = z_{p_+}^H(t).
$$
\item The map $\widehat u$ extends to a section
$s_u:  \C P^1  \to E_{\phi_H}$.
\end{enumerate}
\end{lem}
The proof is obvious by definition of $E_{\phi_H}$.
Let $u_1,u_2$ satisfy the condition (1) above.
We say that $u_1$ is homologous to $u_2$ if
$$
[\widehat u_1] = [\widehat u_2] \in H_2(E_{\phi_H};\Z).
$$
Let $\Pi_2(M;H)$
\index{$\Pi_2(M;H)$} be the set of the homology classes of such $u$.
We note that
$$
[\widehat u_1] - [\widehat u_2] \in \text{\rm Ker} (H_2(E_{\phi_H};\Z) \to H_2(\C P^1;\Z)).
$$
Therefore $\Pi_2(M;H)$
\index{$\Pi_2(M;H)$} is a principal homogeneous space
of the group $\text{\rm Ker} (H_2(E_{\phi_H};\Z) \to H_2(\C P^1;\Z))$.

We also have a natural marking $M \cong E_{\{0\}}$ of the fibration $E_{\phi_H} \to \C P^1$
via the inclusion map
$$
M \times \{0\} \subset M \times \C \subset E_{\phi_H}
$$
which we will fix once and for all. Then the natural inclusion induces a map
$H_2(M;\Z) \to \text{\rm Ker} (H_2(E_{\phi_H};\Z) \to H_2(\C P^1;\Z))$.
Therefore there exists an action
\begin{equation}\label{hgaction}
H_2(M;\Z)\times \Pi_2(M;H) \to \Pi_2(M;H)
\end{equation}
of the group $H_2(M;\Z)$ to $\Pi_2(M;H)$.
\begin{rem}
Theorem \ref{LaMcth} which we will prove later
implies that
$$
H_2(M;\Q) \cong \text{\rm Ker} (H_2(E_{\phi_H};\Q) \to H_2(\C P^1;\Q)).
$$
We however do not use this fact.
\end{rem}

Let $J_0$ be a compatible almost complex structure on $M$.
For $t\in S^1$, we define
\begin{equation}\label{295eq}
J^{H}_t = (\phi_H^{t})_*J_0.
\end{equation}
Since $\phi_H^t$ is a symplectic diffeomorphism, $J^H_{t}$ is compatible
with $\omega$.
We denote by $J^H = \{ J^H_{t}  \}_{t\in S^1}$ the above
$S^1$-parametrized family of compatible almost complex structures.
\par
We take $\chi \in \CK$ and consider
$H_{\chi}$ as in (\ref{eq:paraHJ}).
We also take an
$(\R \times S^1)$-parametrized family of
almost complex structures $J^H_\chi$ such that
\begin{equation}\label{296form}
J^H_\chi(\tau,t) =
\begin{cases}
J_0    & \tau \le 0, \\
J^H_t  & \tau \ge 1, \\
J_0      & \text{$t$ is in a neighborhood of  $[0] \in S^1$}.
\end{cases}
\end{equation}

\begin{defn}\label{stripmk0int2}
For $\alpha \in \Pi_2(M;H)$ we denote by
\index{$\mathcal M_{\ell}(H_\chi,J^H_\chi;*,z^H_{*}; \alpha)$}
$\overset{\circ}{{\CM}}_{\ell}(H_\chi,J^H_\chi;*,z^H_{*}; \alpha)$ the set of all
pairs $(u;z^+_1,\dots,z^+_{\ell})$ of maps
$u: \R \times S^1 \to M$
and
$z^+_1,\dots,z^+_{\ell} \in \R \times S^1$,
which satisfy the following conditions:
\begin{enumerate}
\item The map $u$ satisfies the equation:
\be\label{eq:HJCR3}
\dudtau + J^H_\chi\Big(\dudt - \chi(\tau)X_{H_t}(u)\Big) = 0.
\ee
\item The energy
$$
\frac{1}{2} \int \Big(\Big|\dudtau\Big|^2_{J^H_\chi} + \Big|
\dudt - \chi(\tau)X_{H_t}(u)\Big|_{J^H_\chi}^2 \Big)\, dt\, d\tau
$$
is finite.
\item
The map $u$ satisfies the condition that there exists $p_+ \in M$ such that
$$
\lim_{\tau \to +\infty}u(\tau,t) = z_{p_+}^H(t).
$$
\item The homology class of $u$ in $\Pi_2(M;H)$ is $\alpha$.
\item $z^+_i$ are all distinct.
\end{enumerate}
\end{defn}
\par
By our construction, the map
$$
\overline u: \R \times S^1 \to M, \quad \overline u(\tau,t)=(\phi_H^t)^{-1} u(\tau,t)
$$
is $J_0$-holomorphic on $[1,\infty) \times S^1$ on $M$.
Therefore we can apply removable singularity theorem
to $\overline{u}$ which gives rise to a section $\widehat u$ mentioned in
Lemma $\ref{triextendtau}$.
\par
We denote by
$$
{\rm ev}_{\pm \infty} : \overset{\circ}{{\CM}}_{\ell}(H_\chi,J^H_\chi;*,z^H_{*};\alpha) \to M
$$
the map which associates to $u$ the limit $
\lim_{\tau \to \pm \infty} \overline u(\tau,0)$.
We define the evaluation maps at $z_i$ :
$$
{\rm ev} = ({\rm ev}_1,\dots,{\rm ev}_{\ell})
: \overset{\circ}{{\CM}}_{\ell}(H_\chi,J^H_\chi;*,z^H_{*};\alpha) \to (E_{\phi_H})^{\ell}
$$
by
$$
{\rm ev}_i(u;z_1,\dots,z_{\ell}) = (z_i,u(z_i)) \in U_2 \subset E_{\phi_H}.
$$
\begin{defn}
For $\alpha \in H_2(M;\Z)$
we define $\widehat{\overset{\circ}{{\CM}}}_{\ell}(H,J^H;z^H_*, z^H_*; \alpha)$ to be
the set of all pairs $(u;z^+_1,\dots,z^+_{\ell})$ of maps
$u: \R \times S^1 \to M$
and  marked points
$z^+_1,\dots,z^+_{\ell} \in \R \times S^1$, which satisfy the following conditions:
\begin{enumerate}
\item
The map $u$ satisfies the equation:
\be\label{eq:HJCR4}
\dudtau + J^{H_t}\Big(\dudt- X_{H_t}(u)\Big) = 0.
\ee
\item
The energy
$$
\frac{1}{2} \int \Big(\Big|\dudtau\Big|^2_{J^{H_t}} + \Big|
\dudt-X_{H_t}(u)\Big|_{J^{H_t}}^2 \Big)\, dt\, d\tau
$$
is finite.
\item
There exist points $p_{\pm} \in M$ such that
$$
\lim_{\tau \to -\infty}u(\tau,t) = z^H_{p_-}, \quad \lim_{\tau \to +\infty}u(\tau,t) = z^H_{p_+} .
$$
\item
The homology class of $u$ is $\alpha$.
\item
$z^+_i$ are all distinct.
\end{enumerate}
There is an $\R$-action
on $\widehat{\overset{\circ}{{\CM}}}_{\ell}(H,J^H;z^H_{*},z^H_{*}; \alpha)$
that is induced by the translations of the $\R$ direction
(namely $\tau \mapsto \tau + c$).
The action is free if $\alpha \ne 0$ or $\ell \ne 0$.
We denote its quotient space by
$\overset{\circ}{{\CM}}(H,J^H;z^H_{*},z^H_{*}; \alpha)$.
If $\alpha =0 = \ell$, we {\it define}
$\overset{\circ}{{\CM}}_{\ell}(H,J^H;z^H_{*},z^H_{*}; \alpha)$ to be the
empty set.
\par
We define evaluation maps
$
{\rm ev}_{\pm \infty} : \widehat{\overset{\circ}{{\CM}}}_{\ell}(H,J^H;z^H_{*},z^H_{*};\alpha) \to M
$
by
\be\label{eq:evHJH}
{\rm ev}_{\pm \infty}(u) = \lim_{\tau\to\pm\infty}(\phi_H^t)^{-1}(u(\tau,t)).
\ee
Here we would like to point out that for any $u \in \widehat{\overset{\circ}{{\CM}}}_{\ell}(H,J^H;z^H_{*},z^H_{*}; \alpha)$
the right hand side of (\ref{eq:evHJH}) converges to $p_\pm \in M$
that is independent of $t$. Therefore the evaluation
map is well-defined.
The maps ${\rm ev}_{\pm \infty}$ factor through $\overset{\circ}{{\CM}}_{\ell}(H,J^H;z^H_{*},z^H_{*}; \alpha)$.
\par
We define
${\rm ev} = ({\rm ev}_1,\dots,{\rm ev}_{\ell}): \widehat{\overset{\circ}{{\CM}}}_{\ell}(H,J^H;z^H_{*},z^H_{*};\alpha) \to M^{\ell}$
as follows.
\begin{equation}\label{eq:evphiH}
{\rm ev}_i(u;z^+_1,\dots,z^+_{\ell})
= \phi_H^{-t}(u(z_i^+))
\end{equation}
where $z_i^+ = (\tau,t)$. It factors through
 $\overset{\circ}{{\CM}}_{\ell}(H,J^H;z^H_{*},z^H_{*}; \alpha)$ also.
\end{defn}
We consider the case $H=0$ in
$\widehat{\overset{\circ}{{\CM}}}_{\ell}(H,J^H;z^H_{*},z^H_{*};\alpha)$ and
write it
$\widehat{\overset{\circ}{{\CM}}}_{\ell}(H=0,J_0;*,*;\alpha)$.
(Note that $J^H_t = J_0$ if $H=0$.)
\begin{lem}\label{HandH0iso}
$\widehat{\overset{\circ}{{\CM}}}_{\ell}(H,J^H;z^H_{*},z^H_{*};\alpha)$ is
isomorphic to
$\widehat{\overset{\circ}{{\CM}}}_{\ell}(H=0,J_0;*,*;\alpha)$.
The isomorphism is compatible with evaluation maps and
$\R$ actions.
\end{lem}
\begin{proof}
Let
$(u;z^+_1,\dots,z^+_{\ell}) \in \overset{\circ}{{\CM}}_{\ell}(H,J^H;z^H_{*},z^H_{*};\alpha)$
we put
$$
v(\tau,t) = (\phi_H^t)^{-1}(u(\tau,t)).
$$
Then
$(v;z^{+\prime}_1,\dots,z^{+\prime}_{\ell}) \in \overset{\circ}{{\CM}}_{\ell}(H=0,J_0;*,*;\alpha)$. The assignment
$(u;z^+_1,\dots,z^+_{\ell}) \mapsto (v;z^{+\prime}_1,\dots,z^{+\prime}_{\ell})$ gives
the required isomorphism.
\end{proof}
We can prove that $\overset{\circ}{{\CM}}_{\ell}(H,J^H;z^H_{*},z^H_{*};\alpha)$ and
$\overset{\circ}{{\CM}}_{\ell}(H=0,J_0;*,*;\alpha)$ have
compactifications
${{\CM}}_{\ell}(H,J^H;z^H_{*},z^H_{*};\alpha)$ and
${{\CM}}_{\ell}(H=0,J_0;*,*;\alpha)$, respectively.
They have Kuranishi structures which are isomorphic.
We can define an $S^1$ action on ${{\CM}}_{\ell}(H=0,J_0;*,*;\alpha)$ by
using the $S^1$ action on $\R \times S^1$.
We then use the isomorphism to define an $S^1$ action on the domain
${{\CM}}_{\ell}(H,J^H;z^H_{*},z^H_{*};\alpha)$.
The evaluation maps are compatible with this action.
The isotropy group of this $S^1$ action is always finite.
(We note that we have $\alpha \ne 0$ or $\ell \ne 0$ by definition.)

\begin{lem}\label{PiupureBM}
\begin{enumerate}
\item
The moduli space $\overset{\circ}{{\CM}}_{\ell}(H_\chi,J^H_\chi;*,z^H_{*};\alpha)$ has a compactification
${{\CM}}_{\ell}(H_\chi,J^H_\chi;*,z^H_{*};\alpha)$ that is Hausdorff.
\item
The space ${{\CM}}_{\ell}(H_\chi,J^H_\chi;*,z^H_{*};\alpha)$ has an orientable Kuranishi structure with corners.
\item
The normalized boundary of ${{\CM}}_{\ell}(H_\chi,J^H_\chi;*,z^H_{*};\alpha)$ is described as the union of the
following two types of fiber products.
\begin{equation}\label{bdryhomomapPiuBM}
 \bigcup
{{\CM}}_{\#\mathbb L_1}(H_\chi,J^H_\chi;*,z^H_{*};\alpha_1) {}_{\text{ev}_{+\infty}}\times_{\text{ev}_{-\infty}}
{{\CM}}_{\#\mathbb L_2}(H,J^H;z^H_*,z^H_*;\alpha_2)
\end{equation}
where the union is taken over all $\alpha_1,\alpha_2 \in \Pi_2(M;H)$ with
$\alpha_1 + \alpha_2 = \alpha$  and $(\mathbb L_1,\mathbb L_2) \in \text{\rm Shuff}(\ell)$.
Here $\alpha_1 + \alpha_2$ is as in $(\ref{hgaction})$. The fiber product is taken over $M$.
\begin{equation}\label{bdryhomomapPiuBM2}
 \bigcup
{{\CM}}_{\#\mathbb L_1}(H=0,J_0;*,*;\alpha_1) {}_{\text{ev}_{+\infty}}\times_{\text{ev}_{-\infty}}
{{\CM}}_{\#\mathbb L_2}(H_\chi,J^H_\chi;*,z^H_{*};\alpha_2)
\end{equation}
where the union is taken over all $\alpha_1,\alpha_2  \in \Pi_2(M;H)$ with
$\alpha_1 + \alpha_2 = \alpha$  and $(\mathbb L_1,\mathbb L_2) \in \text{\rm Shuff}(\ell)$. The fiber product is taken over $M$.
\item
We may choose $\alpha_0 \in \Pi_2(M;H)$ such that the (virtual) dimension satisfies
the following equality $(\ref{dimension22})$.
\begin{equation}\label{dimension22}
\dim  {{\CM}}_{\ell}(H_\chi,J^H_\chi;*,z^H_{*};\alpha_0 + \alpha) = 2 c_1(M)\cap \alpha + 2n + 2\ell.
\end{equation}
\item
We can define a system of orientations on ${{\CM}}_{\ell}(H_\chi,J^H_\chi;*,z^H_{*};\alpha)$ so that
the isomorphisms $(3)$ above are compatible with this orientations.
\item
The valuation maps \eqref{eq:evphiH} extend to ${{\CM}}_{\ell}(H_\chi,J^H_\chi;*,z^H_{*};\alpha)$ in a
way compatible with
$(3)$ above.
\item
The map ${\rm ev}_{+\infty}$ becomes a weakly submersive map in the sense of
\cite[Definition {\rm  A1.13}]{fooo:book2} (Definition $\ref{mapkura}$).  Here ${\rm ev}_{+\infty}$ is defined in the same way as in
\eqref{eq:evHJH}.
\item
The Kuranishi structure is invariant under the permutation of interior marked points.
\end{enumerate}
\end{lem}
Here the compatibility with the evaluation maps claimed in (6)  above is
described as follows.
Let us consider the boundary in (\ref{bdryhomomapPiuBM}).
Let $i \in\L_2$  be the $j$-th element of $\L_2$.
We have
$$
\text{ev}_i : {{\CM}}_{\ell}(H_\chi,J^H_\chi;*,z^H_{*};\alpha)
\to E_{\phi_H}
$$
and
$$
\text{ev}_j : {{\CM}}_{\# \L_2}(H,J^H;z^H_*,z^H_*;\alpha)
\to M.
$$
Denote by $t$ the second coordinate of the marked point in $\R \times S^1$.
Then $(\phi_H^t)^{-1} \circ \text{ev}_j$ is equal to
second factor of the $\text{ev}_i$ with respect to $U_3 \cong D_+ \times M$.

The proof of
Lemma \ref{PiupureBM} is the same as the proof of Proposition \ref{connkura} and is omitted.
Note the end of ${{\CM}}_{\ell}(H_\chi,J^H_\chi;*,z^H_*;\alpha)$ at which
an element of $\overline{{\CM}}(0,J_0;*,*;\alpha)$ bubbles off as $\tau \to -\infty$
may be regarded as codimension $2$ because of $S^1$ symmetry, where
the latter moduli space corresponds to the case of $H=0$ for the moduli space
$\overline{{\CM}}(H_\chi,J^H_\chi;*,z^H_*;\alpha)$.
\par
To define the operators involving bulk deformations, we need the
following result due to Lalonde-McDuff-Polterovich \cite{mcdufflalondepol}.
\begin{thm}[Lalonde-McDuff-Polterovich]
\label{LaMcth}
There exists a section
$$
H^*(M;\C) \to H^*(E_{\phi_H};\C)
$$
of the $\C$-linear map
$
H^*(E_{\phi_H};\C) \to H^*(M;\C)
$
which is induced by the inclusion $M \times \{0\} \to E_{\phi_H}$.
\end{thm}
\begin{rem}
\begin{enumerate}
\item
Theorem \ref{LaMcth} is \cite[Theorem 3B]{mcdufflalondepol}.
We give a proof of
Theorem  \ref{LaMcth} in Subsection \ref{subsec:MLtheorem} for completeness' sake.
The proof we give in Subsection \ref{subsec:MLtheorem} is basically the same as
the one given in \cite{mcdufflalondepol}.
\item
The proof by \cite{mcdufflalondepol} as well as our proof in
Subsection \ref{subsec:MLtheorem} uses the construction which is closely related to the
definition of Seidel homomorphism. We use
Theorem \ref{LaMcth} to define the Seidel homomorphism with bulk.
However the argument is not circular by the following reason.
We do {\it not} use Theorem  \ref{LaMcth}
to define Seidel homomorphism in the case when
the bulk deformation $\frak b$ is zero.
The proof of Theorem  \ref{LaMcth} uses
the construction of Seidel homomorphism {\it without bulk} only,
that is the case $\frak b = 0$.
\end{enumerate}
\end{rem}
\par\medskip
Consider a system of CF-perturbations of
${{\CM}}_{\ell}(H,J_H;z^H_{*},z^H_{*};\alpha)$ and of
${{\CM}}_{\ell}(H=0,J_0;*,*;\alpha)$
which is transversal to $0$, $S^1$-equivariant and
compatible with the isomorphism in
Lemma \ref{HandH0iso}.
Moreover we may assume that the system is compatible
with the identification
\begin{equation}\label{bdryhomomapBM}
\aligned
&\partial {{\CM}}_{\ell}(H,J^H;z^H_{*},z^H_{*};\alpha) \\
&= \bigcup{{\CM}}_{\#\mathbb L_1}(H,J^H;z^H_{*},z^H_{*};\alpha_1) {}_{\text{ev}_{+\infty}}\times_{\text{ev}_{-\infty}}
{{\CM}}_{\#\mathbb L_2}(H,J^H;z^H_{*},z^H_{*};\alpha_2)
\endaligned
\end{equation}
of the boundary.
Furthermore we may assume that the
evaluation map $\text{\rm ev}_{+\infty}$ is strongly submersive
with respect to this CF-perturbations.
\par
Then, there exists a system of CF-perturbations of
the moduli space ${{\CM}}_{\ell}(H_\chi,J^H_\chi;*,z^H_{*};\alpha)$ such that
they are transversal to $0$,
compatible with the description of the
boundary in Lemma \ref{PiupureBM} (3)
and that $\text{\rm ev}_{+\infty}$ is strongly submersive with respect thereto.
\par
Let $\frak b = \frak b_0 + \frak b_2 + \frak b_+$ be as in
(\ref{decompb}).
Using Theorem \ref{LaMcth}, we regard them as de Rham cohomology classes
of $E_{\phi_H}$ and denote them as $\widehat{\frak b}_2$, $\widehat{\frak b}_+$.
\par
Now we define a map \index{$\mathcal S^{\frak b}_{(H_\chi,J_\chi)}$}
$$
\mathcal S^{\frak b}_{(H_\chi,J_\chi)} :  \Omega(M) \widehat{\otimes} \Lambda
\to \Omega(M) \widehat{\otimes} \Lambda
$$
as follows.
Let $h \in \Omega(M)$.
We put
$$
\mathcal S^{\frak b}_{(H_\chi,J_\chi);\alpha}(h)
=
\sum_{\ell=0}^{\infty}\frac{\exp({\int_\alpha \widehat{\frak b}_2})}{\ell!}
\text{\rm ev}_{+\infty !}\big(\text{\rm ev}^*(\underbrace{\widehat{\frak b}_+,
\dots,\widehat{\frak b}_+}_{\ell}) \wedge\text{\rm ev}_{-\infty}^*h\big)
$$
where we use the correspondence
$$
(\text{\rm ev};\text{\rm ev}_{-\infty},\text{\rm ev}_{+\infty})
: {{\CM}}_{\ell}(H_\chi,J^H_\chi;*,z^H_{*};\alpha)
\to E_{\phi_H}^{\ell} \times M^2,
$$
and the above constructed CF-perturbation to define integration along the fibers.
We define ${\int_\alpha \widehat{\frak b}_2}$ as follows.
Let $u \in \overset{\circ}{{\CM}}_{0}(H_\chi,J^H_\chi;*,z^H_{*};\alpha)$.
It induces a map $\widehat u : \C P^1 \to E_{\phi_H}$. We put
$$
\int_\alpha \widehat{\frak b}_2 = \int_{\C P^1} \widehat u^*  \widehat{\frak b}_2.
$$
It is easy to see that the integral depends only on $\alpha$ and is independent of the representative $u$.
\par
Let $u \in \overset{\circ}{{\CM}}_{0}(H_\chi,J^H_\chi;*,z^H_{*};\alpha)$
and $p = \text{\rm ev}_{+\infty}(u)$.
Then $[z_p^H,u] \in \text{\rm Crit}(\mathcal A_H)$.
We put
$$
\mathcal A_H(\alpha) := \mathcal A_H([z_p^H,u]).
$$
\par
We then define
$$
\mathcal S^{\frak b}_{(H_\chi,J_\chi)}
=  \sum_{\alpha} T^{-\mathcal A_H(\alpha)} \mathcal S^{\frak b}_{(H_\chi,J_\chi);\alpha}.
$$
\begin{lem}
$\mathcal A_H([z_p^H,u])$ depends only
on the homology class $\alpha$ but independent of its representative $u$.
\end{lem}
\begin{proof}
Recall the map $I_+: U_2 \to U_3$ which was defined as
$
I_+ (\tau,t,x) = (e^{-2\pi (\tau -1 +\sqrt{-1}t)},(\phi_H^{t})^{-1}(x)).
$
It is easy to see that
$
I_+ ^*\Omega=\omega
$
where $\omega$ is the pull back of the symplectic form of $M$
to $U_2, U_3$ and $\Omega$ is as in Lemma \ref{lem:coupling}.
\par
We have
$$
\int \widehat u^*\Omega
=
\int u^*\omega + \int H_t(z_p^H(t)) dt
= - \mathcal A_H([(z_p^H,w)]).
$$
The lemma follows from Stokes' theorem.
\end{proof}
\begin{lem} The map $\mathcal S^{\frak b}_{(H_\chi,J_\chi)}$ satisfies
$$
\mathcal S^{\frak b}_{(H_\chi,J^H_\chi)} \circ  d
= d  \circ \mathcal S^{\frak b}_{(H_\chi,J^H_\chi)}
$$
and so descends to a map
\begin{equation}
\mathcal S^{\frak b}_{(H_\chi,J^H_\chi),\ast}  : H^*(M;\Lambda)
\to H^*(M;\Lambda).
\end{equation}
\end{lem}
The lemma follows from Lemma \ref{PiupureBM} (3),
Stokes' formula (Theorem \ref{them48}, \cite[Corollary 8.13]{fooo:tech2}) and Composition formula
(Theorem \ref{compform}, \cite[Theorem 10.20]{fooo:tech2}).

\begin{thm}\label{seimain}
\begin{enumerate}
\item
$\mathcal S^{\frak b}_{(H_\chi,J^H_\chi),\ast}$ is independent of the family of compatible almost complex structures
$J^H_\chi$
and other choices involved such as multisection.
\item
$\mathcal S^{\frak b}_{(H_\chi,J^H_\chi),\ast}$ depends only on the homotopy class of the loop $t\mapsto \phi_H^t$
in the group of Hamiltonian diffeomorphisms.
\item
We have
$$
\mathcal S^{\frak b}_{(H_\chi,J^H_\chi),\ast} (x\cup^{\frak b}y) =
x\cup^{\frak b}\mathcal S^{\frak b}_{(H_{\chi},J^H_{\chi}),\ast}(y).
$$
\item
Let $H_1,H_2$ be periodic Hamiltonians satisfying
$\psi_{H_1} = \psi_{H_2} = identity$.
Then we have
$$
\mathcal S^{\frak b}_{((H_1\#H_2)_{\chi},J^{H_1\#H_2}_{\chi}),\ast} (x\cup^{\frak b}y)
=
\mathcal S^{\frak b}_{((H_1)_{\chi_1},J^{H_1}_{\chi}),\ast}(x)
\cup^{\frak b}\mathcal S^{\frak b}_{((H_2)_{\chi},J^{H_2}_{\chi}),\ast} (y).
$$
\end{enumerate}
\end{thm}
We define the map
\index{$\mathcal S^{\frak b}$}
$
\mathcal S^{\frak b}: \pi_1(\Ham(M,\omega)) \to H(M;\Lambda)
$
by
$$
\mathcal S^{\frak b}([\phi_H]) = \mathcal S^{\frak b}_{(H_\chi,J^H_\chi),\ast} (1).
$$
Here $H$ is a time-dependent Hamiltonian such that $\psi_H = 1$.
$[\phi_H]$ is the homotopy class of the  loop in $\text{\rm Ham}(M;\omega)$
determined by $t\mapsto \phi_H^t$.
$1$ is the unit of $H(M;\Lambda)$.
(Note that $1$ is also the unit with respect to the quantum cup product on
$QH_{\frak b}(M;\Lambda)$ with the bulk.)

The proof of Theorem \ref{seimain} will be given in
Subsection \ref{subsec:seibulk2} for completeness' sake.
\begin{cor}\label{seicor}
$\mathcal S^{\frak b}$ is a homomorphism to the group
$QH_{\frak b}(M;\Lambda)^{\times}$of
invertible elements of $QH_{\frak b}(M;\Lambda)$.
\end{cor}
\begin{defn}
We call the representation
$$
\mathcal S^{\frak b} : \pi_1(\text{\rm Ham}(M;\omega)) \to QH_{\frak b}(M;\Lambda)^{\times}.
$$
 {\it Seidel homomorphism with bulk}.
\end{defn}

\begin{rem}
As mentioned before the homomorphism $\mathcal S^{\frak b}$
is obtained by Seidel \cite{seidel:auto} in the case $\frak b =0$
under certain hypothesis on the symplectic manifold $(M,\omega)$.
Once the virtual fundamental chain technique had been established in the
year 1996, it is obvious that we can generalize  \cite{seidel:auto} to
arbitrary $(M,\omega)$.
The generalization to include bulk deformations is also  straightforward
and do not require  novel ideas.
\end{rem}
\begin{proof}
We prove Corollary \ref{seicor} assuming Theorem \ref{seimain}.
Let $[\phi_{H_i}] \in \pi_1(\text{\rm Ham}(M;\omega))$.
We have $[\phi_{H_1\#H_2}] = [\phi_{H_2}] [\phi_{H_1}] $.
Then using Theorem \ref{seimain} (3),(4) we have:
$$
\aligned
\mathcal S^{\frak b}([\phi_{H_1\#H_2}] )
&=
\mathcal S^{\frak b}_{((H_1\# H_2)_{\chi},J^{H_1\#H_2}_{\chi}),\ast}(1) \\
&=
\mathcal S^{\frak b}_{((H_1)_{\chi},J^{H_1}_{\chi}),\ast} (1)
\cup^{\frak b} \mathcal S^{\frak b}_{((H_2)_{\chi},J^{H_2}_{\chi}),\ast}(1)
=
\mathcal S^{\frak b}([\phi_{H_1}] )
\cup^{\frak b}
\mathcal S^{\frak b}([\phi_{H_2}] ).
\endaligned
$$
Thus $\mathcal S^{\frak b}$ is a homomorphism.
It implies in particular that the elements of the image are
invertible.
\end{proof}

\subsection{Proof of Theorem  \ref{seimain}}
\label{subsec:seibulk2}

The proof of Theorem \ref{seimain} (1),(2) is similar to the proof of
Theorem \ref{homotopyinvbulk} and so omitted.
\par
The proof of  Theorem \ref{seimain} (3),(4) is similar to the proof of
Theorem \ref{themprodcompati} and proceeds as follows.
\par
Let $\Sigma$ be as in Subsection \ref{subsec:pants}.
We also use the notations $h : \Sigma \to \R$, $\frak S \subset \Sigma$
and
$
\varphi: \R \times ((0,1/2) \sqcup (1/2,1)) \to \Sigma \setminus \frak S
$
etc. in Subsection \ref{subsec:pants}.
We define a $\Sigma$-parametrized family of almost complex structures $J^{H_1,H_2}$
by $J^{H_1,H_2}(\varphi(\tau,t)) = J^{H_1\# H_2}_t$.
We assume that $(H_1)_t = (H_2)_t = 0$ if $t$ is in a neighborhood of
$[0] \in S^1 = \R/\Z$.
Let
$H^{\varphi} : \Sigma \times M \to \R$ be a function as in
(\ref{defsharp}).

\begin{defn}\label{pantsmoduliBM}
\index{$\mathcal M_{\ell}(H^\varphi,J^{H_1,H_2};
z^{H_1}_*, z^{H_2}_*,z^{H_1 \# H_2}_*;\alpha)$}
We denote by $\overset{\circ}{{\CM}}_{\ell}(H^\varphi,J^{H_1,H_2};
z^{H_1}_*, z^{H_2}_*,z^{H_1 \# H_2}_*;\alpha)$ the set of all
pairs $(u;z_1^+,\dots,z_{\ell}^+)$ of maps
 $u: \Sigma \to M$
and $z_i^+ \in \Sigma$,
which satisfy the following conditions:
\begin{enumerate}
\item The map $\overline u = u\circ \varphi$ satisfies the equation:
\be\label{eq:HJCR11BM}
\frac{\partial\overline u}{\partial \tau} + J^{H_1,H_2}\Big(\frac{\partial\overline u}
{\partial t} -X_{H^\varphi}(\overline u)\Big) = 0.
\ee
\item The energy
$$
\frac{1}{2} \int \Big(\Big|\frac{\partial\overline u}{\partial \tau}\Big|^2_{J^{H_1,H_2}} + \Big|
\frac{\partial\overline u}
{\partial t} -X_{H^\varphi}(\overline u)\Big|_{J^{H_1,H_2}}^2 \Big)\, dt\, d\tau
$$
is finite.
\item There exist $p_{-,1}, p_{-,2}, p_{+} \in M$ such that $u$ satisfies the following three asymptotic boundary conditions.
$$
\lim_{\tau\to +\infty}u(\varphi(\tau, t)) = z^{H_1\# H_2}_{p_{+}}(t).
$$
$$
\lim_{\tau\to - \infty}u(\varphi(\tau, t)) =
\begin{cases}
z^{H_1}_{p_{-,1}}(2t)   &t \le 1/2, \\
z^{H_2}_{p_{-,2}}(2t-1)   &t \ge 1/2.
\end{cases}
$$
\item The homology class of $u$ is $\alpha$,
in the sense we explain below.
\item
$z_1^+,\dots,z_{\ell}^+$ are mutually distinct.
\end{enumerate}
Here the homology class of $u$ which we mention in (4) above is defined as follows.
We put
\begin{equation}\label{u and v}
v(\tau,t)
=
\begin{cases}
(\phi_{2H_1}^t)^{-1}(u(\tau, t))      \quad  \tau \leq 0, 0 \leq t \leq 1/2 \\
(\phi_{2H_2}^{t-1/2})^{-1}(u(\tau,t))  \quad \tau \leq 0, 1/2 \leq t \leq 1 \\
(\phi_{H_1\#H_2}^t)^{-1}(u(\tau,t)) \quad \tau \geq 0.
\end{cases}
\end{equation}
It defines a map $\Sigma \to M$ which extends to a continuous map
$v : S^2 \to M$. (Note that $\Sigma$ is
$S^2 \setminus \{ {\text {$3$ points}}\}$.)
The homology class of $u$ is by definition nothing but the class
$v_*([S^2]) \in H_2(M;\Z)$.
\par\smallskip
We denote by
$$
({\rm ev}_{-\infty,1},{\rm ev}_{-\infty,2},{\rm ev}_{+\infty})
: \overset{\circ}{{\CM}}_{\ell}(H^\varphi,J^{H_1,H_2};
z^{H_1}_*, z^{H_2}_*,z^{H_1 \# H_2}_*;\alpha) \to M^3
$$
the map which associates $(p_{-,1},p_{-,2},p_+)$ to $(u;z_1^+,\dots,z_{\ell}^+)$.
We also define an evaluation map
$$
{\rm ev} = ({\rm ev} _1,\dots,{\rm ev} _{\ell}) : \overset{\circ}{{\CM}}_{\ell}(H^\varphi,J^{H_1,H_2};
z^{H_1}_*, z^{H_2}_*,z^{H_1 \# H_2}_*;\alpha) \to M^{\ell}
$$
that associates to $(u;z_1^+,\dots,z_{\ell}^+)$ the point
$(u(z_1^+),\dots,u(z_{\ell}^+))$.
\end{defn}
\begin{lem}\label{pantsBULKkuraBM}
\begin{enumerate}
\item The moduli space
$\overset{\circ}{{\CM}}_{\ell}(H^\varphi,J^{H_1,H_2};
z^{H_1}_*, z^{H_2}_*,z^{H_1 \# H_2}_*;\alpha)$ has a compactification
${{\CM}}_{\ell}(H^\varphi,J^{H_1,H_2};
z^{H_1}_*, z^{H_2}_*,z^{H_1 \# H_2}_*;\alpha)$ that is Hausdorff.
\item
The space ${{\CM}}_{\ell}(H^\varphi,J^{H_1,H_2};
z^{H_1}_*, z^{H_2}_*,z^{H_1 \# H_2}_*;\alpha)$ has an orientable Kuranishi structure with corners.
\item
The normalized boundary of ${{\CM}}_{\ell}(H^\varphi,J^{H_1,H_2};
z^{H_1}_*, z^{H_2}_*,z^{H_1 \# H_2}_*;\alpha)$ is described
by the union of the following three types of fiber products.
\begin{equation}\label{bdrypantsBULKkura2BM}
{{\CM}}_{\#\mathbb L_1}(H_1,J^{H_1};z^{H_1}_,z^{H_1}_*;\alpha_1)\,\,
{}_{\text{\rm ev}_{+\infty}}\times_{\text{\rm ev}_{-\infty,1}}
{{\CM}}_{\#\mathbb L_2}(H^\varphi,J^{H_1,H_2};
z^{H_1}_*, z^{H_2}_*,z^{H_1 \# H_2}_*;\alpha_2)
\end{equation}
where the union is taken over all $\alpha_1,\alpha_2$ with $\alpha_1 + \alpha_2
= \alpha$,
 and $(\mathbb L_1,\mathbb L_2) \in \text{\rm Shuff}(\ell)$.
\begin{equation}\label{bdrypantsBULKkura3BM}
{{\CM}}_{\#\mathbb L_1}(H_2,J^{H_2};z^{H_2}_,z^{H_2}_*;\alpha_1)
\,\,
{}_{\text{\rm ev}_{+\infty}}\times_{\text{\rm ev}_{-\infty,2}}
{{\CM}}_{\#\mathbb L_2}(H^\varphi,J^{H_1,H_2};
z^{H_1}_*, z^{H_2}_*,z^{H_1 \# H_2}_*;\alpha_2)
\end{equation}
where the union is taken over all $\alpha_1,\alpha_2$ with $\alpha_1 + \alpha_2
= \alpha$,
 and $(\mathbb L_1,\mathbb L_2) \in \text{\rm Shuff}(\ell)$.
\begin{equation}\label{bdrypantsBULKkura4BM}
\aligned
&{{\CM}}_{\#\mathbb L_2}(H^\varphi,J^{H_1,H_2};
z^{H_1}_*, z^{H_2}_*,z^{H_1 \# H_2}_*;\alpha_1))\\
&
{}_{\text{\rm ev}_{+\infty}}\times_{\text{\rm ev}_{-\infty}}  {{\CM}}_{\#\mathbb L_1}(H_1\#
H_2,J^{H_1\#
H_2};z^{H_1 \# H_2}_*,z^{H_1 \# H_2}_*;\alpha_2)
\endaligned\end{equation}
where the union is taken over all $\alpha_1,\alpha_2$ with $\alpha_1 + \alpha_2
= \alpha$,
 and $(\mathbb L_1,\mathbb L_2) \in \text{\rm Shuff}(\ell)$.
 \item
The (virtual) dimension is given by
\begin{equation}\label{dimensionboundaryb11BM}
\dim{{\CM}}_{\ell}(H^\varphi,J^{H_1,H_2};
z^{H_1}_*, z^{H_2}_*,z^{H_1 \# H_2}_*;\alpha) = 2\ell + 2c_1(M)[\alpha ] + 2n.
\end{equation}
\item
We can define a system of orientations on ${{\CM}}_{\ell}(H^\varphi,J^{H_1,H_2};
z^{H_1}_*, z^{H_2}_*,z^{H_1 \# H_2}_*;\alpha)$ so that the isomorphism
$(3)$ above is compatible with this orientation.
\item
${\rm ev}_{-\infty,1}$, ${\rm ev}_{-\infty,2}$, ${\rm ev}_{+\infty}$, ${\rm ev}$ extend to
strongly continuous smooth maps on
${{\CM}}_{\ell}(H^\varphi,J^{H_1,H_2};
z^{H_1}_*, z^{H_2}_*,z^{H_1 \# H_2}_*;\alpha)$, which we denote also by the same symbol.
They are compatible with $(3)$.
\item
${\rm ev}_{+\infty}$ is weakly submersive.
\item
The Kuranishi structure is invariant under the permulation of interior marked points.
\end{enumerate}
\end{lem}
We take a system of CF-perturbations on
${{\CM}}_{\ell}(H^\varphi,J^{H_1,H_2};
z^{H_1}_*, z^{H_2}_*,z^{H_1 \# H_2}_*;\alpha)$
that are transversal to $0$, compatible with (3) above
and such that ${\rm ev}_{+\infty}$ is strongly submersive with respect
to our CF-perturbations.\par

We define
$
\frak m_{2;\alpha}^{\text{\rm cl}, H^\varphi}
: \Omega(M) \otimes \Omega(M)
\to  \Omega(M)
$
by
\index{$\frak m_{2;\alpha}^{\text{\rm cl};\frak b;H^\varphi}$}
\begin{equation}\label{2821}
\aligned
& \frak m_{2;\alpha}^{\text{\rm cl};\frak b;H^\varphi}(h_1,h_2)\\
&=
\sum_{\ell=0}^{\infty}
\frac{1}{\ell !}
(\text{\rm ev}_{+\infty})!
\left(
\text{\rm ev}_{-\infty,1}^* h_1 \wedge
\text{\rm ev}_{-\infty,2}^* h_2 \wedge
\text{\rm ev}^*(\underbrace{\frak b_+,\dots,\frak b_+}_{\ell})
\right),
\endaligned
\end{equation}
where we use ${{\CM}}_{\ell}(H^\varphi,J^{H_1,H_2};
z^{H_1}_*, z^{H_2}_*,z^{H_1 \# H_2}_*;\alpha)$
and the evaluation maps thereon as a correspondence to define the right hand side.
\par
We put
$$
\frak m_2^{\text{\rm cl},\frak b;H^\varphi}
=
\sum_{\alpha} e^{\alpha \cap \widehat{\frak b}_2}
T^{\alpha \cap \omega}
\frak m_{2;\alpha}^{\text{\rm cl};\frak b;H^{\varphi}}.
$$
It defines a chain map,
since the contribution of the boundaries described in (3) above
to the correspondence are
all zero by the $S^1$ equivariance.
Therefore $\frak m_2^{\text{\rm cl},\frak b;H^\varphi}$  defines a map
$$
\frak m_2^{\text{\rm cl},\frak b;H^\varphi} :
H^*(M;\Lambda)
\otimes_{\Lambda}
H^*(M;\Lambda)
\to
H^*(M;\Lambda).
$$

\begin{lem}\label{multiPiuBM}
$$
\frak m_{2}^{\text{\rm cl};\frak b;H^\varphi}
\circ \left(
{\mathcal S}_{((H_1)_{\chi},J^{H_1}_{\chi})}^{\frak b} \otimes
{\mathcal S}_{((H_2)_{\chi},J^{H_2}_{\chi})}^{\frak b}
\right)
$$
is chain homotopic to
$$
{\mathcal S}_{((H_1\#H_2)_{\chi},J^{H_1\#H_2}_{\chi})}^{\frak b} \circ \cup^{\frak b}.
$$
\end{lem}
\begin{proof}
For $S \in \R$ we define
$H_{S,\chi}^{\varphi} : \Sigma \times M \to \R$ by
\begin{equation}
H_{S,\chi}^{\varphi}(\varphi(\tau,t),x)
=
\chi(\tau+S) (H_1\#H_2)_t(x).
\end{equation}
Note that $J^H_{\chi}$ is the $(\R\times S^1)$-parameterized family of almost
complex structures as in (\ref{296form}).
For $S \in \R$ we define a $\Sigma$ parameterized family of
almost complex structures $J^{H_1,H_2}_{S,\chi}$ by
\begin{equation}
J^{H_1,H_2}_{S,\chi}(\varphi(\tau,t))
= J^{(H_1\#H_2)}_{\chi}(\tau+S,t).
\end{equation}
Let $\alpha \in \Pi_2(M;H_1\#H_2)$.

\begin{defn}\label{pantsmoduliBMpara}
For $S \in \R$,
we denote by $\overset{\circ}{{\CM}}_{\ell}(H^{\varphi}_{S,\chi},J^{H_1,H_2}_{S,\chi};
**,z^{H_1 \# H_2}_*;\alpha)$ the set of all
pairs $(u;z_1^+,\dots,z_{\ell}^+)$ of maps
 $u: \Sigma \to M$
and $z_i^+ \in \Sigma$,
which satisfy the following conditions:
 \begin{enumerate}
 \item The map $\overline u = u\circ \varphi$ satisfies the equation:
\be\label{eq:HJCR11BMpara}
\frac{\partial\overline u}{\partial \tau} + J^{H_1,H_2}_{S,\chi}\Big(\frac{\partial\overline u}
{\partial t} -X_{H^{\varphi}_{S,\chi}}(\overline u)\Big) = 0.
\ee
\item The energy
$$
\frac{1}{2} \int \Big(\Big|\frac{\partial\overline u}{\partial \tau}\Big|^2_{J^{H_1,H_2}_{S,\chi}} + \Big|
\frac{\partial\overline u}
{\partial t} -X_{H^{\varphi}_{S,\chi}}(\overline u)\Big|_{J^{H_1,H_2}_{S,\chi}}^2 \Big)\, dt\, d\tau
$$
is finite.
\item There exist $p_{-,1}, p_{-,2}, p_{+}
\in M$ such that $u$ satisfies the following three asymptotic boundary conditions.
$$
\lim_{\tau\to +\infty}u(\varphi(\tau, t)) = z^{H_1\# H_2}_{p_{+}}(t).
$$
$$
\lim_{\tau\to - \infty}u(\varphi(\tau, t)) =
\begin{cases}
p_{-,1}  &t \le 1/2, \\
p_{-,2}   &t \ge 1/2.
\end{cases}
$$
\item The homology class of $u$ is ${\alpha}$,
in the sense we explain below.
\item
$z_1^+,\dots,z_{\ell}^+$ are mutually distinct.
\end{enumerate}
\begin{figure}[h]
\centering
\includegraphics[scale=0.3]
{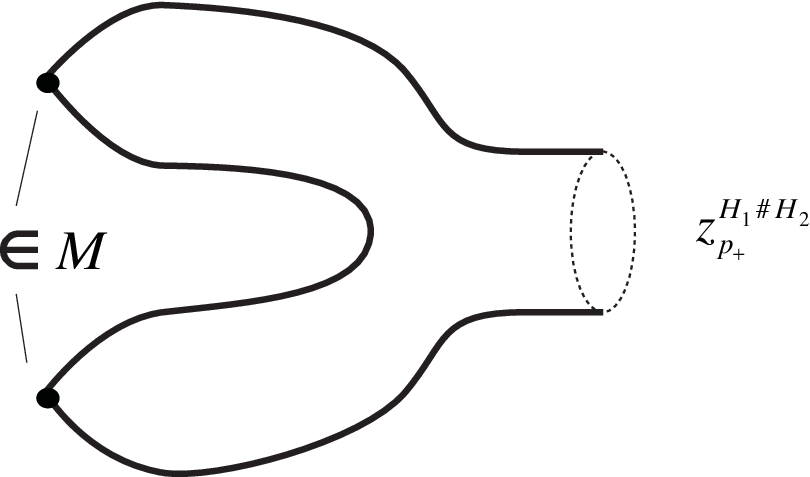}
\caption{An element of  $\overset{\circ}{{\CM}}_{\ell}(H^{\varphi}_{S,\chi},J^{H_1,H_2}_{S,\chi};
**,z^{H_1 \# H_2}_*;\alpha)$}
\label{Figure291}
\end{figure}
Here the homology class of $u$ which we mention in (4) above is defined as follows.
We consider $\Sigma \times M$ and
glue $M$ at the two ends corresponding to $\tau \to -\infty$ by
$(\phi_{H_1}^{2t})^{-1}$ $0 \leq t \leq 1/2$ and $(\phi_{H_2}^{2t-1})^{-1}$, $1/2 \leq t \leq 1$, respectively.
At the end corresponding to $\tau \to +\infty$ we glue $M$
but with twisting using the map $\phi_{H_1\# H_2}$ in the same
way as the definition of $E_{H^\varphi}$. We then obtain
$E_{H^\varphi}$. Actually this space
together with projection to $S^2 = \Sigma \cup \{ \text{3 points} \}$
can be identified with $E_{\phi_{H_1 \# H_2}}$.
We define
\begin{equation}\label{utouprimeprod}
\widehat u(\tau,t)
= ((\tau,t),u(\tau,t)) \in E_{H^\varphi}.
\end{equation}
It extends to a continuous map
$\widehat u : S^2 \to E_{H^\varphi}$.
The homology class of $\widehat u$ is well defined as an element of
$\Pi_2(M;H_1\# H_2)$.
\par
By Theorem \ref{LaMcth} we obtain $\widehat{\frak b}_2
\in H^2(E_{H^\varphi};\C)$ from
${\frak b}_2 \in H^2(M;\C)$.
\par\smallskip
We denote by
\begin{equation}\label{eq:ev3}
({\rm ev}_{-\infty,1},{\rm ev}_{-\infty,2},{\rm ev}_{+\infty})
:
\overset{\circ}{{\CM}}_{\ell}(H^{\varphi}_{S,\chi},J^{H_1,H_2}_{S,\chi};
**,z^{H_1 \# H_2}_*;\alpha)
\to M^3
\end{equation}
the map which associate $(p_{-,1},p_{-,2},p_+)$ to $(u;z_1^+,\dots,z_{\ell}^+)$.
We also define an evaluation map
\begin{equation}\label{eq:evEell}
{\rm ev} = ({\rm ev} _1,\dots,{\rm ev} _{\ell}): \overset{\circ}{{\CM}}_{\ell}(H^{\varphi}_{S,\chi},J^{H_1,H_2}_{S,\chi};
**,z^{H_1 \# H_2}_*;\alpha) \to (E_{H^\varphi})^{\ell}
\end{equation}
that associates to $(u;z_1^+,\dots,z_{\ell}^+)$ the point
$(\widehat u(z_i^+),\dots,\widehat u(z_{\ell}^+))$.
We put
\index{$\mathcal M_{\ell}(para;H^{\varphi}_{\chi},J^{H_1,H_2}_{\chi};
**,z^{H_1 \# H_2}_*;\alpha)
$}
$$
\overset{\circ}{{\CM}}_{\ell}(para;H^{\varphi}_{\chi},J^{H_1,H_2}_{\chi};
**,z^{H_1 \# H_2}_*;\alpha)
=
\bigcup_{S\in \R}
\{S\} \times \overset{\circ}{{\CM}}_{\ell}(H^{\varphi}_{S,\chi},J^{H_1,H_2}_{S,\chi};
**,z^{H_1 \# H_2}_*;\alpha).
$$
The above evaluation maps are defined on it in an obvious way.
\end{defn}
We can define a compactification
${{\CM}}_{\ell}(para;H^{\varphi}_{\chi},J^{H_1,H_2}_{\chi};
**,z^{H_1 \# H_2}_*;\alpha)$
of the moduli space
$
\overset{\circ}{{\CM}}_{\ell}(para;H^{\varphi}_{\chi},J^{H_1,H_2}_{\chi};
**,z^{H_1 \# H_2}_*;\alpha)$
and a system of Kuranishi structures on it, that are oriented with corners.
Its boundary is a union of the following five types of fiber products:
\begin{equation}\label{bdrypantsBULKkura2BMpara}
\aligned
&{{\CM}}_{\#\mathbb L_1}(H=0,J_0;*,*;\alpha_1) \\
&{}_{\text{\rm ev}_{+\infty}}\times_{\text{\rm ev}_{-\infty,1}}
{{\CM}}_{\#\mathbb L_2}(para;H^\varphi_{\chi},J^{H_1,H_2}_{\chi};
**,z^{H_1 \# H_2}_*;\alpha_2),
\endaligned
\end{equation}
where the union is taken over all $\alpha_1,\alpha_2$ with $\alpha_1 + \alpha_2
= \alpha$,
 and $(\mathbb L_1,\mathbb L_2) \in \text{\rm Shuff}(\ell)$.
\begin{figure}[h]
\centering
\includegraphics[scale=0.3]
{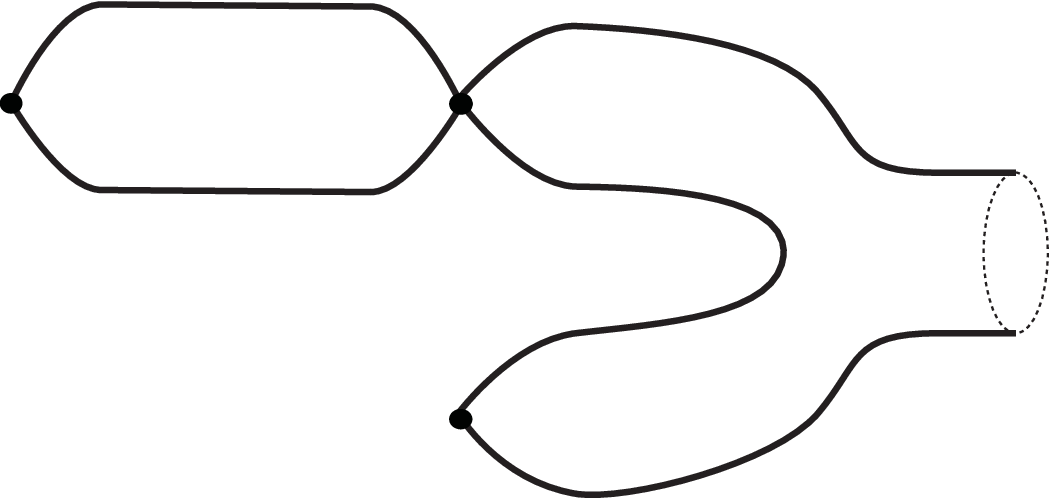}
\caption{An element of (\ref{bdrypantsBULKkura2BMpara})}
\label{Figure1}
\end{figure}
\begin{equation}\label{bdrypantsBULKkura3BMpara}
\aligned
&{{\CM}}_{\#\mathbb L_1}(H=0,J_0;*,*;\alpha_1)\\
&
{}_{\text{\rm ev}_{+\infty}}\times_{\text{\rm ev}_{-\infty,2}}
{{\CM}}_{\#\mathbb L_2}(para;H^\varphi_{\chi},J^{H_1,H_2}_{\chi};
**,z^{H_1 \# H_2}_*;\alpha_2),
\endaligned
\end{equation}
where the union is taken over all $\alpha_1,\alpha_2$ with $\alpha_1 + \alpha_2
= \alpha$,
 and $(\mathbb L_1,\mathbb L_2) \in \text{\rm Shuff}(\ell)$.
 \begin{center}
\begin{figure}[h]
 \begin{tabular}{cc}
 \begin{minipage}[t]{0.45\hsize}
\centering
\includegraphics[scale=0.3]
{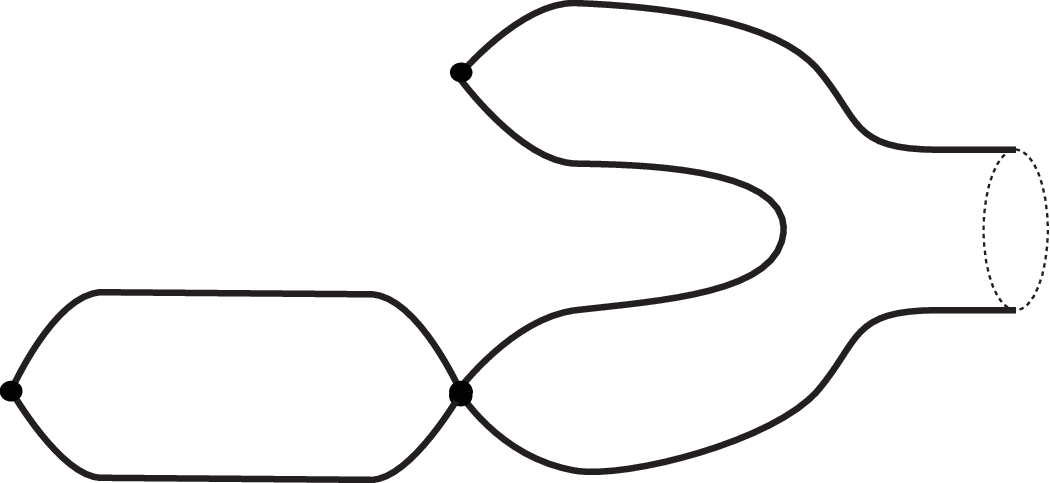}
\caption{An element of (\ref{bdrypantsBULKkura3BMpara})}
\label{Figure293}
\end{minipage} &
\begin{minipage}[t]{0.45\hsize}
\centering
\includegraphics[scale=0.3]
{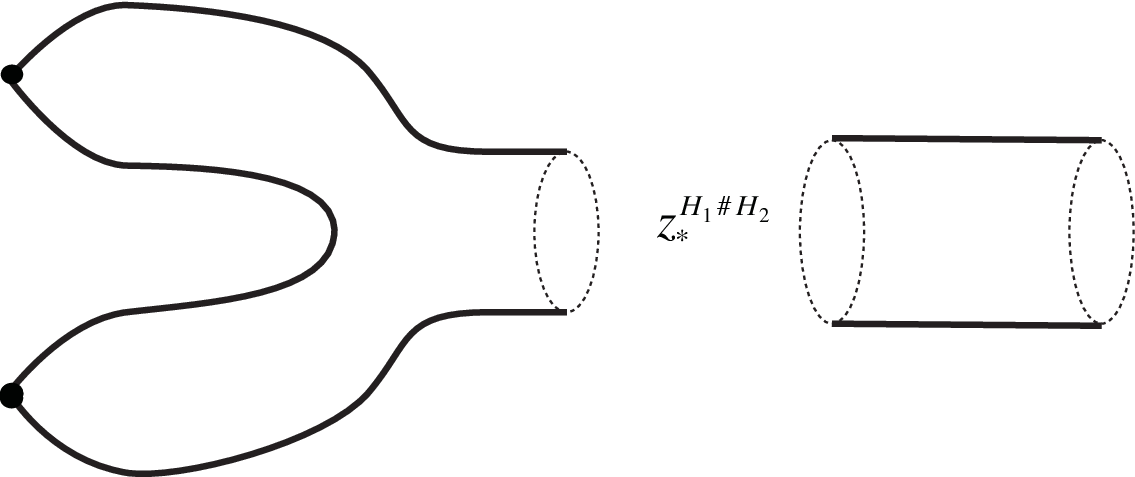}
\caption{An element of (\ref{bdrypantsBULKkura4BMpara})}
\label{Figure294}
\end{minipage}
\end{tabular}
\end{figure}
\end{center}

\begin{equation}\label{bdrypantsBULKkura4BMpara}
\aligned
&{{\CM}}_{\#\mathbb L_1}(para;H^\varphi_{\chi},J^{H_1,H_2}_{\chi};
**,z^{H_1 \# H_2}_*;\alpha_1)\\
&
{}_{\text{\rm ev}_{+\infty}}\times_{\text{\rm ev}_{-\infty}}
{{\CM}}_{\#\mathbb L_2}(H_1\# H_2,J^{H_1\#H_2};z^{H_1 \# H_2}_*,z^{H_1 \# H_2}_*;\alpha_2),
\endaligned\end{equation}
where the union is taken over all $\alpha_1,\alpha_2$ with $\alpha_1 + \alpha_2
= \alpha$,
 and $(\mathbb L_1,\mathbb L_2) \in \text{\rm Shuff}(\ell)$.
\begin{equation}\label{bdrypantsBULKkura5BMpara}
\aligned
{{\CM}}_{\#\mathbb L_1+3}^{\text{\rm cl}}(\alpha_1)
{}_{\text{\rm ev}_{3}}\times_{\text{\rm ev}_{-\infty}}
  {{\CM}}_{\#\mathbb L_2}((H_1\# H_2)_{\chi},J^{H_1\#H_2}_{\chi};*,z^{H_1 \# H_2}_*;\alpha_2),
\endaligned\end{equation}
where the union is taken over all $\alpha_1,\alpha_2$ with $\alpha_1 + \alpha_2
= \alpha$,
 and $(\mathbb L_1,\mathbb L_2) \in \text{\rm Shuff}(\ell)$.

\begin{center}
\begin{figure}[h]
 \begin{tabular}{cc}
 \begin{minipage}[t]{0.45\hsize}
\centering
\includegraphics[scale=0.3]
{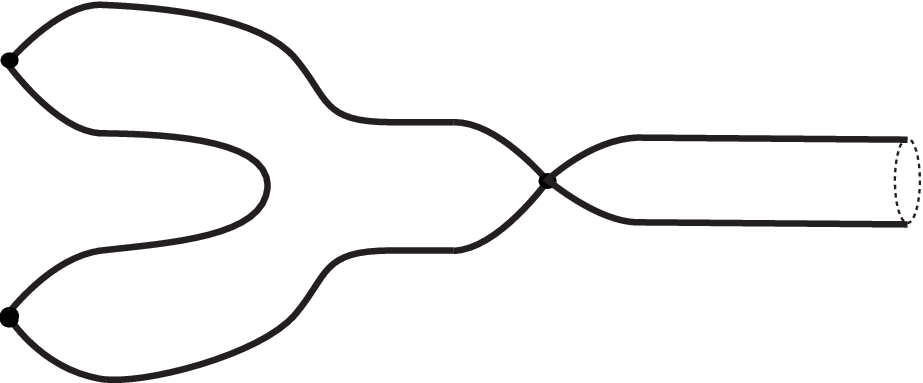}
\caption{An element of (\ref{bdrypantsBULKkura5BMpara})}
\label{Figure295}
 \end{minipage} &
 \begin{minipage}[t]{0.45\hsize}
 \centering
\includegraphics[scale=0.3]
{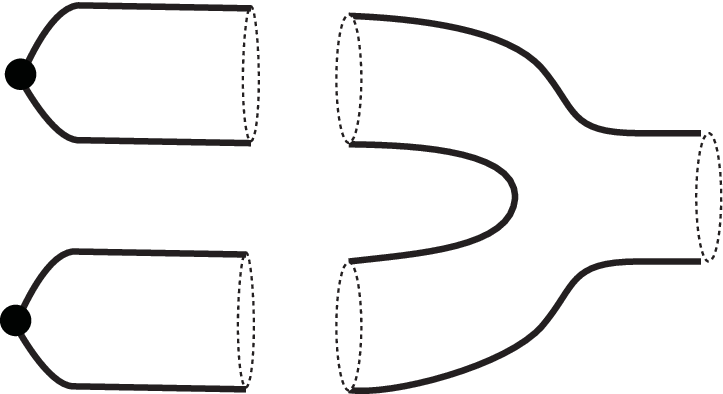}
\caption{An element of (\ref{bdrypantsBULKkura6BMpara})}
\label{Figure296}
\end{minipage}
\end{tabular}
 \end{figure}
\end{center}
\begin{equation}\label{bdrypantsBULKkura6BMpara}
\aligned
&{{\CM}}_{\#\mathbb L_1}((H_1)_{\chi};J^{H_1}_{\chi}
*,z^{H_1}_*;\alpha_1)) \times {{\CM}}_{\#\mathbb L_2}((H_2)_{\chi};J^{H_2}_{\chi};
*,z^{H_2}_*;\alpha_2))\\
&
{}_{(\text{\rm ev}_{+\infty},\text{\rm ev}_{+\infty})}
\times_{(\text{\rm ev}_{-\infty,1},\text{\rm ev}_{-\infty,2})}  {{\CM}}_{\#\mathbb L_3}(H^\varphi,J^{H_1,H_2};
z^{H_1}_*, z^{H_2}_*,z^{H_1 \# H_2}_*;\alpha_3)),
\endaligned\end{equation}
where the union is taken over all $\alpha_1, \alpha_2, \alpha_3$ with $\alpha_1 + \alpha_2
+ \alpha_3= \alpha$ and `triple shuffle' $(\mathbb L_1,\mathbb L_2,\mathbb L_3)$ of $\{1,\dots,\ell\}$.

Note that (\ref{bdrypantsBULKkura2BMpara}), (\ref{bdrypantsBULKkura3BMpara}),
The ends of types (\ref{bdrypantsBULKkura4BMpara}) are the ends which appear while $S$ is bounded.
(\ref{bdrypantsBULKkura5BMpara}) and (\ref{bdrypantsBULKkura6BMpara}) correspond to the case
$S \to -\infty$ and $S \to +\infty$, respectively.
Similarly as before the evaluation maps \eqref{eq:ev3}, \eqref{eq:evEell} extend to this
compactified moduli space, which we denote by the same notations.
\par
We next take a system of CF-perturbations on the moduli spaces
$$
{{\CM}}_{\ell}(para;H^\varphi_{\chi},J^{H_1\# H_2}_{\chi};
**,z^{H_1 \# H_2}_*;\alpha)
$$
so that it is transversal to $0$ and
the evaluation map $\text{\rm ev}_{+\infty}$ is strongly submersive with respect thereto.
Moreover we assume that it is compatible with the above description of the
boundary. We remark that the first factor of (\ref{bdrypantsBULKkura2BMpara}), (\ref{bdrypantsBULKkura3BMpara})
and the second factor of (\ref{bdrypantsBULKkura4BMpara}) carry $S^1$ actions such that the isotropy group is
finite.

To find an $S^1$ action on the second factor
${{\CM}}_{\#\mathbb L_3}(H^\varphi,J^{H_1,H_2};
z^{H_1}_*, z^{H_2}_*,z^{H_1 \# H_2}_*;\alpha_3)$
of (\ref{bdrypantsBULKkura4BMpara})
we map it to
${{\CM}}_{\#\mathbb L_3+2}(H^\varphi,J_0;\alpha_3)$,
by sending an element
$$
(\R\times S^1\cup\{-\infty,+\infty\},\{-\infty,+\infty\}\cup \vec z^+,u)
\in {{\CM}}_{\#\mathbb L_3+2}(H^\varphi,J_0;\alpha_3)
$$
to the element
$$
(\R\times S^1\cup\{-\infty,+\infty\},\vec z^+,\overline u)
\in {{\CM}}_{\#\mathbb L_3}(H^\varphi,J^{H_1,H_2};
z^{H_1}_*, z^{H_2}_*,z^{H_1 \# H_2}_*;\alpha_3)
$$
where
$
\overline u(\tau,t) = \phi^t_{H_1\#H_2}(u(\tau,t)).
$
The fibers of this map are the orbits of our
$S^1$ action.
Note the evaluation map ${{\CM}}_{\#\mathbb L_3}(H^\varphi,J^{H_1,H_2};
z^{H_1}_*, z^{H_2}_*,z^{H_1 \# H_2}_*;\alpha_3)
\to E_{H_1 \# H_2}^{\#\mathbb L_3}$ is
invariant under the $S^1$ action by definition.
\par
Therefore we may take our CF-perturbations so that they are $S^1$ equivariant on those factors.
Then the contributions of  (\ref{bdrypantsBULKkura2BMpara}), (\ref{bdrypantsBULKkura3BMpara}),
(\ref{bdrypantsBULKkura4BMpara}) become zero when we consider the correspondence of
the moduli space ${{\CM}}_{\ell}(para;H^\varphi_{\chi},J^{H_1\# H_2}_{\chi};
**,z^{H_1 \# H_2}_*;\alpha)$.
\par
Using the moduli space ${{\CM}}_{\ell}(para;H^\varphi_{\chi},J^{H_1 \# H_2}_{\chi};
**,z^{H_1 \# H_2}_*;\alpha)$ and the evaluation maps in the way similar to
the way used for (\ref{2821}), we obtain an operator
$$
\frak H : (\Omega(M) \widehat{\otimes} \Lambda)
\otimes (\Omega(M) \widehat{\otimes} \Lambda )
\to \Omega(M) \widehat{\otimes} \Lambda.
$$
Then $\frak H$ satisfies
\begin{equation}
\aligned
&d \circ\frak H+ \frak H\circ d  \\
&=
\frak m_{2;\alpha}^{\text{\rm cl};\frak b;H^\varphi}
\circ \left(
{\mathcal S}_{((H_1)_{\chi},J^{H_1}_{\chi})}^{\frak b} \otimes {\mathcal S}_{((H_2)_{\chi},J^{H_2}_{\chi})}^{\frak b}
\right) - {\mathcal S}_{((H_1\#H_2)_{\chi},J^{H_1\# H_2}_{\chi})}^{\frak b} \circ \cup^{\frak b}.
\endaligned
\end{equation}
In fact, the contributions of (\ref{bdrypantsBULKkura5BMpara}) and (\ref{bdrypantsBULKkura6BMpara}) correspond
to the first and second term of the right hand side of this identity respectively.
The proof of Lemma \ref{multiPiuBM} is complete.
\end{proof}
Theorem \ref{seimain} (4)
follows from Lemma \ref{multiPiuBM}.
\par
To prove (3) we apply (4) to the case
$
H_1  = 0, \, H_2 = H.
$
Then using the fact that $\mathcal S_{((H_1)_{\chi},J^{H_1}_{\chi}),\ast}^{\frak b} = id$
we have
\begin{equation}\label{almost3ban}
\mathcal S_{((0\# H)_{\chi},J^{0\# H}_{\chi}),\ast} (x \cup^{\frak b} y)
= x \cup^{\frak b} \mathcal S_{(H_\chi,J^H_\chi),\ast}(y)
\end{equation}
in homology.
Since $0\# H \sim H$,
we can prove $\mathcal S_{((0\# H)_\chi,J^{H\# 0}_{\chi}),\ast}
= \mathcal S_{(H_\chi,J^H_\chi),\ast} $ by using a homotopy between $0\# H$ and $H$. Then (\ref{almost3ban}) now implies
Theorem \ref{seimain} (3).
\par
Therefore the proof of Theorem \ref{seimain}  is now complete.
\qed

\subsection{Proof of Theorem  \ref{LaMcth}}
\label{subsec:MLtheorem}

As we mentioned before, we did not use Theorem \ref{LaMcth} in the definition of
$\mathcal S_{(H_{\chi},J^H_{\chi})}^0$ or in the proof of Theorem \ref{seimain}
for the  case $\frak b =0$.
We will use only that case in this subsection.
\par
We consider the moduli space
$\mathcal M_1(H_\chi,J^H_\chi;*,z^H_{*};\alpha)$ for the definition of
$$
\mathcal R_{(H_\chi,J^H_\chi)} : \Omega(M) \widehat{\otimes} \Lambda
\to
\Omega(E_{\phi_H}) \widehat{\otimes} \Lambda
$$
as follows.
Let $h \in \Omega(M)$. For each $\alpha$, we put
$$
\mathcal R_{(H_\chi,J^H_\chi),\alpha}(h)
=
(\text{\rm ev}_{1})! (\text{\rm ev}_{-\infty}^*(h))
$$
where we use the evaluation maps
$\text{\rm ev}_1 : \mathcal M_1(H_\chi,J^H_\chi;*,z^H_{*};\alpha) \to E_{\phi_H}$
and
$\text{\rm ev}_{-\infty} : \mathcal M_1(H_\chi,J^H_\chi;*,z^H_{*};\alpha) \to M$.
We then take the sum
$$
\mathcal R_{(H_\chi,J^H_\chi)} = \sum_{\alpha} T^{-\mathcal A_H(\alpha)}\mathcal R_{(H_\chi,J^H_\chi),\alpha}.
$$
We consider the inclusion map $i_{\infty} : M \to E_{\phi_H}$ to the fiber of $\infty \in \C P^1$.
By definition it is easy to see that
$
i_{\infty}^*\circ \mathcal R_{(H_\chi,J^H_\chi)}
$
is chain homotopic to $ \mathcal S_{(H_\chi,J^H_\chi)}$,
where $(\phi_H)^*(\gamma(t))=(\phi_H^t)^{-1}(\gamma (t))$.
Note that, if $\gamma(t)$ is a one-periodic orbit of $\phi_H^t$,
$(\phi_H)^*(\gamma(t))$ is a constant.   Therefore
\begin{equation}
i_{\infty}^{\ast} \circ \mathcal R_{(H_\chi,J^H_\chi),\ast} =  \mathcal S_{(H_\chi,J^H_\chi),\ast}
\end{equation}
in homology.
It follows that for $h \in H^*(M;\Lambda^{\downarrow})$  we have
\begin{equation}
\aligned
i_{\infty}^{\ast} (\mathcal R_{(H_\chi,J^H_\chi),\ast}(h \cup^{Q} \mathcal S(\widetilde\psi_H)^{-1}))
&=
\mathcal S_{(H_\chi,J^H_\chi),\ast}(h \cup^{Q} \mathcal S(\widetilde\psi_H)^{-1}\\
&=
h \cup^{Q} S(\widetilde\psi_H)\cup^{Q} \mathcal S(\widetilde\psi_H)^{-1} = h.
\endaligned\end{equation}
Thus
$$
h \mapsto \widehat h = \mathcal R_{(H_\chi,J^H_\chi),\ast}(h \cup^{Q} \mathcal S(\widetilde\psi_H)^{-1})
$$
is a required section.
The proof of Theorem \ref{LaMcth} is complete.
\qed

\section{Spectral invariants and Seidel homomorphism}
\label{subsec:seibulkspectral}

In this section we study the relationship between the Seidel homomorphism and
spectral invariants.

\subsection{Valuations and spectral invariants}
\label{subsec:valu}

The next theorem is a straightforward generalization of the
result  \cite[Theorem 4.3]{oh:jkms} and  \cite[Proposition 4.1]{EP:morphism}.

Let $H$ be a time-dependent normalized Hamiltonian such that $\psi_H = id$
and, let $\widetilde\psi_H$ be an associated element of $\pi_1(\text{\rm Ham}(M;\omega))$.

\begin{thm}\label{seivsspec}
For each $a \in QH_{\frak b}(M;\Lambda)$, we have
$$
\rho^{\frak b}(H;a)
= \lambda_q(a \cup^{\frak b} \mathcal S^{\frak b}(\widetilde\psi_H)).
$$
where $\lambda_q$ is as given in \eqref{defvq}.
\end{thm}
\begin{proof}
The proof is similar to the proof of Theorem \ref{contspect}.
Let $H_k$ be a sequence of normalized time dependent Hamiltonians
such that $\psi_{H_k}$ are nondegenerate and
$\lim_{k\to\infty} H_k = H$ in $C^0$ topology.
We put
\begin{equation}
F_{k,\chi}(\tau,t,x) = H_k(t,x) + \chi(\tau)(H(t,x)-H_k(t,x)): \R\times S^1\times M \to \R.
\end{equation}
\par
We fix $J_0$ and define $J^{H_k}$, $J^H$, $J^{H_k}_{\chi}$, $J^H_{\chi}$
as in (\ref{295eq}) and (\ref{296form}).
Let $J_{k,\chi}$  be an $\R\times S^1$ parametrized family of
almost complex structures such that
$$
J_{k,\chi}(\tau,t)
=
\begin{cases}
J^H_t   &\tau \ge 2, \\
J^{H_k}_t &\tau \le -2.
\end{cases}
$$
\par
Let $[\gamma,w] \in \text{\rm Crit}(\mathcal A_{H_k})$ and
$\alpha \in \Pi_2(M;H)$.

 \begin{defn}\label{stripmk0int24BM}
We denote by $\overset{\circ}{\CM}_{\ell}(F_{k,\chi},J_{k,\chi};[\gamma,w],*;\alpha)$ the set of all
pairs $(u;z^+_1,\dots,z^+_{\ell})$ of maps
 $u: \R \times S^1 \to M$ and $z_i^+ \in \R\times S^1$ which satisfy the following conditions:
 \begin{enumerate}
 \item The map $u$ satisfies the equation:
\be\label{eq:HJCR24}
\dudtau + J_{k,\chi}\Big(\dudt -X_{F_{k,\chi}}(u)\Big) = 0.
\ee
\item The energy
$$
\frac{1}{2} \int \Big(\Big|\dudtau\Big|^2_{J_{k,\chi}} + \Big|
\dudt - X_{{F_{k,\chi}}}(u)\Big|_{J_{k,\chi}}^2 \Big)\, dt\, d\tau
$$
is finite.
\item There exists $p$ such that  the following asymptotic boundary condition is satisfied.
$$
\lim_{\tau\to -\infty}u(\tau, t) = \gamma(t), \quad \lim_{\tau\to
+\infty}u(\tau, t) = z_p^H(t).
$$
\item The homology class of $w\# u$ is $\alpha$,
where $\#$ is the obvious concatenation.
\item
$z_i^+$ are distinct to each other.
\end{enumerate}
\end{defn}
The space $\overset{\circ}{\CM}_{\ell}(F_{k,\chi},J_{k,\chi};[\gamma,w],*;\alpha)$ has a
compactification
${\CM}_{\ell}(F_{k,\chi},J_{k,\chi};[\gamma,w],*;\alpha)$,
on which there exists a system of oriented Kuranishi structures with corners
which is compatible at the boundaries.
There exists a system of CF-perturbations of these Kuranishi structures so that
the map $(u;z^+_1,\dots,z^+_{\ell})
\mapsto \lim_{\tau\to +\infty}u(\tau, 0)$ defines a
strongly submersive map
${\CM}_{\ell}(F_{k,\chi},J_{k,\chi};[\gamma,w],*;\alpha) \to M$ with respect to this
CF-perturbation.
\par
We use this in the way similar to the way how we did several times to define a map
\index{$\mathcal P_{(F_{k,\chi},J^{k,\chi})}^{\frak b}$}
$$
\mathcal P_{(F_{k,\chi},J_{k,\chi})}^{\frak b}
: \Omega(M)\widehat{\otimes} \Lambda^{\downarrow} \to CF(M,H_k,J;\Lambda^{\downarrow}).
$$
Here we identify $\Omega(M)\otimes \Lambda^{\downarrow}$ and the Floer chain module $CF(M,H,J;\Lambda^{\downarrow})$
of Bott-Morse type using the Hamiltonian loop $\{\phi_H^t\}$.

\begin{lem}\label{compaticompositteBM}
$\mathcal P_{(F_{k,\chi},J_{k,\chi})}^{\frak b}\circ \mathcal P_{((H_k)_{\chi},J^{H_k}_{\chi})}^{\frak b}$ is chain homotopic to
$\flat \circ \mathcal S_{(H_{\chi},J^H_{\chi})}^{\frak b} \circ \flat $.
\end{lem}
The proof is similar to the proof of Proposition \ref{compaticompositte} and is omitted.
\begin{lem}\label{connectinghomofiltBM}
$$
\mathcal P_{(F_{k,\chi},J_{k,\chi})}^{\frak b}
\left(
F^{\lambda}CF(M,H_k,J;\Lambda^{\downarrow})
\right)
\subset
\Omega(M) \widehat{\otimes} q^{\lambda + E^-(H - H_k)}\Lambda_0^{\downarrow}.
$$
\end{lem}
The proof is similar to the proof of Lemma \ref{connectinghomofilt}
and (\ref{connectinghomofiltformu}) and so omitted.
Lemmas \ref{connectinghomofiltBM} implies
$$
\lambda_q(\mathcal S_{(H_{\chi},J_{\chi})}^{\frak b}(a))
\leq  \rho^{\frak b}(H_k;a) +  E^-(H - H_k).
$$
Taking the limit $k\to \infty$ we have
\begin{equation}
\rho^{\frak b}(H;a)  \geq \lambda_q(\mathcal S_{(H_{\chi},J_{\chi})}^{\frak b}(a)).
\end{equation}
We can prove the opposite inequality by using
$$
\mathcal Q_{(F_k^{\tilde \chi},J_{F_k}^{\tilde \chi})}^{\frak b}
 : CF(M,H_k,J;\Lambda^{\downarrow}) \to  \Omega(M) \widehat{\otimes} \Lambda^{\downarrow}
$$
that can be defined in a similar way as Definition \ref{bdrycoefbQQ0}.
(See the proof of Proposition \ref{oppositeineq} also.)

By Theorem \ref{seimain} (3) with $x=a, y=1$, we find that
$$
 \mathcal S^{\frak b}_{(H_{\chi},J^H_{\chi}),\ast} (a)= a \cup^{\frak b}  \mathcal S^{\frak b}(\widetilde{\psi}_H).
$$
The proof of Theorem \ref{seivsspec} is now complete.
\end{proof}
Let $H$ be a time dependent periodic Hamiltonian such that $\psi_H = id$.
We do not assume that $H$ is normalized.
Let $e\in QH_{\frak b}(M;\omega)$ with $e \cup^{\frak b} e = e$.
We assume that $e\Lambda^{\downarrow}$ is a direct product factor
of $QH_{\frak b}(M;\omega)$.
\begin{cor}\label{xicalccor}
We put $e \cup^{\frak b}\mathcal S^{\frak b}(\widetilde\psi_H) = xe $
with $x \in \Lambda$.
Then
$$
\zeta^{\frak b}(H;e) = \frak v_T(x) + \frac{1}{\vol_\omega(M)}\int_{[0,1]}\int_M H_t \,dt\omega^n.
$$
\end{cor}
\begin{proof}
We put
$$
\underline H_t = H_t - \frac{1}{\vol_\omega(M)}\int_M H_t \omega^n.
$$
It is a normalized Hamiltonian and $\psi_{\underline H} = \psi_H$. By Theorem \ref{seivsspec}
we have
$$
\rho^{\frak b}(\underline H;e) = - {\frak v}_T(e \cup^{\frak b} \mathcal S^{\frak b}(\widetilde\psi_H)).
$$
On the other hand,
$$
\rho^{\frak b}(H;e) = \rho^{\frak b}(\underline H;e) - \frac{1}{\vol_\omega(M)}\int_{[0,1]}\int_M H_t \,dt\omega^n.
$$
Therefore we have
$$
\aligned
\zeta^{\frak b}(H;e) &=  - \left( \lim_{k\to \infty}
\frac{\rho^{\frak b}(kH;e)}{k} \right)\\
&=
- \lim_{k\to \infty}
\frac{\rho^{\frak b}(k\underline H;e)}{k} + \frac{1}{\vol_\omega(M)}\int_{[0,1]}\int_M H_t \,dt\omega^n \\
&= \lim_{k\to \infty}
\frac{{\frak v}_T(e x^k)}{k}
+  \frac{1}{\vol_\omega(M)}\int_{[0,1]}\int_M H_t \,dt\omega^n \\
&= {\frak v}_T(x)
+\frac{1}{\vol_\omega(M)}\int_{[0,1]}\int_M H_t \,dt\omega^n,
\endaligned
$$
as required.
\end{proof}

\subsection{The toric case}
\label{subsec:seibulkspectraltoric}
In this section we generalize a result by McDuff-Tolman \cite{mc-tol}
to a version with bulk and apply the result for some calculation.
Our discussion here is a straightforward generalization of \cite{mc-tol}.
(See also \cite{EP:rigid} for the general discussion for the Hamiltonian loops.)

Let $H$ be a time {\it independent} normalized Hamiltonian.
We assume that $\psi_H = id$. We put
$$
H_{\text{\rm min}} = \inf \{ H(y) \mid y \in M\}
$$
and
$$
D_{\text{\rm min}} = \{ x \in M \mid H(x) = H_{\text{\rm min}} \}.
$$
Since $D_{\text{\rm min}}$ is a connected component of the
fixed point set of the $S^1$ action generated by $X_H$,
it follows that  $D_{\text{\rm min}}$ is a smooth submanifold.
We assume that $D_{\text{\rm min}}$ is of (real) codimension $2$.
We also assume the following:
\begin{asmp}\label{linkassum}
Let $p\in D_{\text{\rm min}}$ and $q \in M \setminus D_{\text{\rm min}}$ be
sufficiently close to $p$.
We consider the orbit $z_{q}^H(t) = \phi_H^t(q)$ and
a disk $w : (D^2,\partial D^2) \to (M,z_q^H)$ which
bounds $z_{q}^H$.  If $w$ is sufficiently small,
then
$$
[D_{\text{\rm min}}] \cdot w_*([D^2]) = +1.
$$
\end{asmp}
\par
Let $\frak b = \frak b_0 + \frak b_2 + \frak b_+$ as before.

\begin{thm}\label{HamS1action}
We have
$$
\mathcal S^{\frak b}([\phi_H])
\equiv
T^{ H_{\text{\rm min}}} e^{\overline{\frak b}_2\cap D_{\text{\rm min}}}  PD([D_{\text{\rm min}}])
\mod T^{ H_{\text{\rm min}}}\Lambda^{\downarrow}_-.
$$
\end{thm}
\begin{rem}
In the case $\frak b =0$ this is  \cite[Theorem 1.9]{mc-tol}.
Our generalization to the case $\frak b \ne 0$  is actually
straightforward.
\end{rem}
\begin{proof} We start with the following lemma.
\begin{lem}\label{energyeqBMmin}
If $\overset{\circ}{{\CM}}_{0}(H_{\chi},J^H_{\chi};*,z^H_{*};\alpha)$ is nonempty, then
\begin{equation}
\mathcal A_H(\alpha)
\le -H_{\text{\rm min}}.
\end{equation}
\end{lem}
\begin{proof}
Let $u \in \overset{\circ}{{\CM}}_{0}(H_{\chi},J^H_{\chi};*,z^H_{*};\alpha)$.
Then we have
$$
\aligned
\int u^*\omega
&=
\int \omega
\left(
\frac{\partial u}{\partial \tau}, \frac{\partial u}{\partial t}
\right)dtd\tau
=
\int
\omega
\left(
\frac{\partial u}{\partial \tau}, J^H_\chi\frac{\partial u}{\partial \tau}
+ \chi(\tau) X_{H}
\right)dtd\tau \\
&\ge - \int
\chi(\tau) \frac{\partial (H\circ u)}{\partial \tau}
dtd\tau \\
&\ge
-\int_{S^1} H(z_p^H(t))dt
+ \int \chi'(\tau) (H\circ u )dtd\tau
\\
&\ge
-\int_{S^1} H(z_p^H(t))dt
+ H_{\text{\rm min}}.
\endaligned
$$
Lemma \ref{energyeqBMmin} follows.
\end{proof}
We remark that the equality holds only when
$$
\int \omega
\left( \frac{\partial u}{\partial \tau}, J^H_\chi \frac{\partial u}{\partial \tau}
\right)\,dt\, d\tau = \int
\left|\frac{\partial u}{\partial \tau}\right|^2_{J^H_\chi}\,dt\,d\tau = 0
$$
and so
$
\frac{\partial u}{\partial \tau} = 0.
$
Therefore $u$ must be constant.
Moreover since $u(\tau,t) \to z_p^H(t)$ as $\tau\to \infty$,
the image of $u$ must lie in the zero locus of $X_H$. Thus
\begin{lem}\label{lem302}
If the equality holds in Lemma $\ref{energyeqBMmin}$,
$
{{\CM}}_{0}(H_{\chi},J^H_{\chi};*,z^H_{*};\alpha)
$
consists of constant maps to $D_{\text{\rm min}}$.
\end{lem}
Let $\alpha_0$ be the homology class such that Lemma \ref{energyeqBMmin} holds.
\begin{lem}\label{lem303}
The moduli space ${{\CM}}_{0}(H_{\chi},J^H_{\chi};*,z^H_{*};\alpha_0)$
is transversal and
$$
\text{\rm ev}_{1 \#}({{\CM}}_{0}(H_{\chi},J^H_{\chi};*,z^H_{*};\alpha_0) )
= [D_{\text{\rm min}}].
$$
\end{lem}
\begin{proof}
We consider $D_{\text{\rm min}} \times S^2 \subset E_{\phi_H}$.
Its tubular neighborhood is identified with a
neighborhood of zero section in the line bundle
$D_{\text{\rm min}} \times \mathcal O(-1) \to D_{\text{\rm min}} \times S^2$.
Here we identify $S^2 \cong \C P^1$ and $\mathcal O(-1)$ is a line bundle with
Chern number $-1$.
(We use Assumption \ref{linkassum} here.)
The moduli space ${{\CM}}_{0}(H_{\chi},J^H_{\chi};*,z^H_{*};\alpha_0)$
then is identified to the moduli space of the sections to the bundle.
$
D_{\text{\rm min}} \times \mathcal O(-1)\to S^2.
$
The lemma follows easily.
\end{proof}
Theorem \ref{HamS1action} now follows from
Lemmas \ref{energyeqBMmin}, \ref{lem302}, \ref{lem303}.
\end{proof}
We now specialize Theorem \ref{HamS1action} to the case of toric manifold.
Let $(M,\omega)$ be a compact K\"ahler toric manifold.
Then $T^n$ acts on $(M,\omega)$ preserving the K\"ahler form.
Let $\pi : M \to P \subset \R^n$ be the moment map.
Let $D_j = \pi^{-1}(\partial_j P)$,  $j=1,\dots,m$ be the irreducible
components of the toric divisor.
As in Section \ref{toriceLagreview} we have affine functions $\ell_j : \R^n \to \R$ such that
$
\partial_j P = \{ \text{\bf u} \in P \mid \ell_j(\text{\bf u}) = 0\}.
$
We put
$
d\ell_j = (k_{j,1},\dots,k_{j,n})
$
where $k_{j,1},\dots,k_{j,n}$ are integers which are coprime.
Let $S^1_j$ be a subgroup of $T^n$ such that
$$
S^1_j
=
\{
[k_{j,1}t,\dots,k_{j,n}t]  \mid t \in \R
\} \subset T^n
$$
where we identify $T^n = \R^n/\Z^n$.
We note that if we put $H = \ell_j$ then $\psi_H = identity$.
The next result is a corollary to Theorem \ref{HamS1action}.
$S^1_j$ determines an element of $\pi_1(\text{\rm Ham}(M,\omega))$,
which we denote by $[S^1_j]$.

\begin{thm}\label{toricseidel}
$$
\mathcal S^{\frak b}([S^1_j])
\equiv
T^{\vol(P)^{-1}\int_P \ell_j } e^{\overline{\frak b}_2 \cap D_j}PD([D_j])
\mod
T^{\vol(P)^{-1}\int_P \ell_j } \Lambda_+.
$$
\end{thm}
\begin{proof}
We note that $ \ell_j -  \text{\rm Vol}(P)^{-1}\int_P \ell_j$ is the normalized
Hamiltonian which generates $[S^1_j]$.
(This is because the push out measure $\pi_!(\omega^n)$ on $P$ is the
Lebesgue measure.
Its minimum is attained at $D_j$.
Therefore Theorem \ref{toricseidel} follows from
Theorem \ref{HamS1action}.
\end{proof}
Let $\text{\bf u}_{\text{\rm cnt}} \in P$ be the
center of gravity
and $e \in QH_{\frak b}(M;\Lambda)$
the idempotent, which corresponds to $\text{\bf u} \in P$
by Theorems \ref{Floercrit}, \ref{JacisHQ} and
Proposition \ref{Morsesplit}. The same identity was established
by Entov-Polterovich \cite[Theorem 1.15]{EP:rigid}  for general Hamiltonian loops, not just for
the circle action, which dealt the monotone case with ${\frak b} = 0$.
\begin{thm}\label{calcuxie}
$$
\mu_{e}^{\frak b}([S^1_j])
=\text{\rm Vol}(P)(\ell_j(\text{\bf u}) - \ell_j(\text{\bf u}_{\text{\rm cnt}})).
$$
In particular, $\mu_{e}^{\frak b} = 0$ on the image of $\pi_1(T^n) \to
\pi_1(\text{\rm Ham}(M,\omega))$ if and only if $\text{\bf u}
= \text{\bf u}_{\text{\rm cnt}}$.
\end{thm}
\begin{proof}
Let $\frak y$ be the critical point of $\frak{PO}_{\frak b}$ that
corresponds to $L(\text{\bf u})$.
(Namely $\text{\bf u}(\frak y) = \text{\bf u}$.)
By (\ref{defKS}), Theorem \ref{calcithm}, (\ref{calcisharp})
and Theorem \ref{toricPOcalcthm} we have
$$
i_{\text{\rm qm},(\frak b,b(\frak y))}^{\ast}(PD([D_j]))
=
\left[
\frac{\partial \frak{PO}_{\frak b}}{\partial w_j}
\right]
\equiv
e^{\overline{\frak b}_j} z_j
\mod \Lambda_+.
$$
Here $z_j$ is as in (\ref{defzj}).
By Lemma \ref{calksi}
$$
i_{\text{\rm qm},(\frak b,b(\frak y))}^{\ast}(e_{\frak y}) = 1.
$$
We put
$
b(\frak y) = \sum x_i e_i
$ and $
\partial \beta_j = \sum k_{ji}e_i,
$
where $e_i$ is a basis of $H(L(\text{\bf u}),\Z)$.
Then we get
$$
z_j(\frak y) = T^{\ell_j(\text{\bf u})}\prod_{i=1}^n \exp{k_{ji}x_i}.
$$
Therefore we have
$$
i_{\text{\rm qm},(\frak b,b(\frak y))}^{\ast}(\mathcal S^{\frak b}([S_j^1]))
\equiv
T^{\ell_j(\text{\bf u})-\text{\rm Vol}(P)^{-1}\int_P \ell_j}c
\mod T^{\ell_j(\text{\bf u})-\text{\rm Vol}(P)^{-1}\int_P \ell_j}\Lambda_+.
$$
where $c \in \C \setminus \{0\}$.
We note that $\text{\rm Vol}(P)^{-1}\int_P \ell_j
= \ell_j(\text{\bf u}_{\text{\rm cnt}})$.
\par
Let us assume that $\frak y$ is nondegenerate.
Then using also the multiplicativity of
$i_{\text{\rm qm},(\frak b,b(\frak y))}^{\ast}$
(\cite[Theorem 9.1]{fooo:toricmir})
we have
$$
e_{\frak y} \cup^{\frak b} \mathcal S^{\frak b}([S_j^1])
\equiv T^{\ell_j(\text{\bf u})-\ell_j(\text{\bf u}_{\text{\rm cnt}})}c e_{\frak y}
\mod T^{\ell_j(\text{\bf u})-\ell_j(\text{\bf u}_{\text{\rm cnt}})}\Lambda_+.
$$
Therefore by Corollary \ref{xicalccor} we obtain
$$\mu_{e}^{\frak b}([S^1_j])
=\text{\rm Vol}(P)(\ell_j(\text{\bf u}) - \ell_j(\text{\bf u}_{\text{\rm cnt}})).
$$
\par
In the general case we
recall that
$\text{\rm Jac}(\frak{PO}_{\frak b};\frak y)$ is a local ring
and the kernel of the homomorphism
$\text{\rm Jac}(\frak{PO}_{\frak b};\frak y)\to \Lambda$ defined by
$[\frak P] \mapsto \frak P(\frak y)$ is nilpotent.
Therefore
$$
e_{\frak y} \cup^{\frak b} \mathcal S^{\frak b}([S_j^1])
=
a e_{\frak y} + b
$$
with $a\in \Lambda$,
$$
a \equiv T^{\ell_j(\text{\bf u})-\ell_j(\text{\bf u}_{\text{\rm cnt}})}c
\mod T^{\ell_j(\text{\bf u})-\ell_j(\text{\bf u}_{\text{\rm cnt}})}\Lambda^{\downarrow}_+
$$
and $b$ is nilpotent.
We use it to show
$$
\lim_{k\to\infty}
\frac{{\frak v}_T(e_{\frak y} \cup^{\frak b} \mathcal S^{\frak b}([S_j^1])^k)}{k}
= \ell_j(\text{\bf u})-\ell_j(\text{\bf u}_{\text{\rm cnt}}).
$$
The proof of Theorem \ref{calcuxie} is now complete.
\end{proof}
\begin{rem}
Theorem \ref{calcuxie} also follows from
Theroems \ref{toricheavymain} and \ref{appliHOch}.
\end{rem}

\theoremstyle{theorem}
\newtheorem{adden}{Addendum}[section]

\part{Kuranishi structure and its CF-perturbation: summary}
\label{part7}

This part is a brief summary of the notion of Kuranishi structure and its perturbation.
There is nothing new in this part.
We include this part only for the convenience of the reader.
This is a summary which is {\it not} intended to include the
full detail of the construction.
We refer the reader to \cite[Section A1]{fooo:book2},
\cite[Section 12]{fooo:bulk}, \cite{fooo:tech2}
for the detail.
(The reader will find that the detail is, in principle, mostly parallel
to the standard theory of manifolds.)

\section{Kuranishi structure and good coordinate system}
\label{subsec:Kuranishi}

We first review a few basic terminologies of orbifolds.
See \cite[Section 15]{fooo:tech2}.  We remark that
we only consider effective orbifolds.

\subsection{Orbifold}
\label{subsec:orbifold}

\begin{defn}\label{def101}
Let $X$ be a topological space.
Its (effective) \index{oribfold chart} {\it orbifold chart} is  a triple $(V,\Gamma,\phi)$ such that
$V$ is a manifold, $\Gamma$  a finite group acting smoothly
and effectively on $V$, and $\phi : V \to X$
is a $\Gamma$ equivariant continuous map\footnote{The $\Gamma$ action
on $X$ is trivial.} which induces a
homeomorphism $\overline\phi : V/\Gamma \to X$ onto an open subset
$U$ of $X$.
\par
Let  $(V,\Gamma,\phi)$ be an orbifold chart and  $p \in V$.
\index{$(V,\Gamma,\phi)$}
We put $\Gamma_p = \{ \gamma \in \Gamma \mid \gamma p = p\}$.
Let $V_p$ be a $\Gamma_p$ invariant open neighborhood of $p$
in $V$. We assume the map
$\overline\phi : V_p/\Gamma_p \to X$ is injective.
(In other words, we assume that
$\gamma V_{p} \cap V_p \ne \emptyset$ implies $\gamma \in \Gamma_p$.)
We say such a triple $(V_{p},\Gamma_p,\phi\vert_{V_p})$
a {\it subchart} \index{orbifold ! subchart} of $(V,\Gamma,\phi)$.
\par
Let $(V_i,\Gamma_i,\phi_i)$  $(i=1,2)$ be orbifold charts of $X$.
We say that they are {\it compatible}
\index{orbifold ! compatible (orbifold chart)} if the following holds
for each $p_1 \in V_1$ and $p_2 \in V_2$ with
$\phi_1(p_1) = \phi_2(p_2)$.
\begin{enumerate}
\item
There exists a group isomorphism $h : (\Gamma_1)_{p_1}
\to (\Gamma_2)_{p_2}$.
\item
There exists an $h$ equivariant diffeomorphism $\tilde\varphi : V_{1,p_1}
\to V_{2,p_2}$. Here $V_{i,p_i}$ is a $(\Gamma_i)_{p_i}$
equivariant subset of $V_i$ such that
 $(V_{i,p_i},(\Gamma_i)_{p_i},\phi\vert_{V_{i,p_i}})$
 is a subchart.
 \item
 $\phi_2 \circ \tilde\varphi = \phi_1$ on $V_{1,p_1}$.
\end{enumerate}
\par
An {\it orbifold structure}
 on $X$
is represented by a set of orbifold charts $\{(V_i,\Gamma_i,\phi_i) \mid i \in I\}$
such that each two of the charts are compatible in the sense of (3)
above and
$
\bigcup_{i\in I} \phi_i(V_i) = X
$
is a locally finite open cover of $X$.
The set of all orbifold charts compatible with a given representative of orbifold
structure is said to be an orbifold structure\index{orbifold structure}.
\end{defn}
\begin{defn}\label{def262220}
Suppose that $X$, $Y$ have  orbifold structures represented by $\{(V^X_i,\Gamma^X_i,\phi^X_i) \mid i \in I\}$
and $\{(V^Y_j,\Gamma^Y_j,\phi^Y_j) \mid j \in J\}$,
respectively.
A continuous map $f : X \to Y$ is said to be an
{\it embedding}
\index{orbifold ! embedding of orbifolds}\index{embedding ! of orbifolds}
if the following holds.
\begin{enumerate}
\item
$f$ is an embedding of  topological spaces.
\item
Let $p \in V^X_i$, $q \in V^Y_j$ with
$f(\phi_i(p)) = \phi_j(q)$. Then we have the following.
\begin{enumerate}
\item
There exists an isomorphism of groups
$h_{p;ji} : (\Gamma_i^X)_p \to (\Gamma_j^Y)_q$.
\item
There exist  $V^X_{i,p}$ and $V^Y_{j,q}$ such that
$(V^X_{i,p},(\Gamma^X_i)_p,\phi_i\vert_{V^X_{i,p}})$
is a subchart for $i=1,2$.
There exists an $h_{p;ji}$ equivariant embedding of manifolds
$\tilde f_{p;ji}: V^X_{i,p} \to V^Y_{j,q}$.
\item
The diagram below commutes.
\begin{equation}\label{diag2633}
\xymatrix{
& V^X_{i,p} \ar[r]^{\tilde f_{p;ji}} \ar[d]_{\phi_{i}}
& V^Y_{j,q} \ar[d]_{\phi_{j}} &  \\
&X \ar[r]_{f} & Y \\
}
\end{equation}
\end{enumerate}
\end{enumerate}
Two orbifold embeddings are said to be {\it equal} \index{equal! to maps between orbifolds} if they coincide set-theoretically.
\par
A continuous map $f : X \to M$ from an orbifold to a manifold
$M$ is said to be a {\em smooth map} \index{smooth map! from orbifold to a manifold} if
for each orbifold chart $(V,\Gamma,\phi)$ of $X$ the composition
$f \circ \phi : V \to M$ is a smooth map between manifolds.
\end{defn}
\begin{defn}\label{obbundle}
Let $(X,\mathcal E,\pi )$ be a pair of  orbifolds $X, \mathcal E$ and $\pi : \mathcal E \to X$  a continuous map
between their underlying topological spaces.
An {\em orbibundle chart} \index{orbibundle chart} of $(X,\mathcal E)$ is $(V,E,\Gamma,\phi,\widehat\phi)$
with
$\frak V = (V,\Gamma,\phi)$ an orbifold chart,
$E$ is a finite dimensional vector space equipped with
\index{$(V,\Gamma,E,\phi,\widehat\phi)$}
a linear $\Gamma$
action and
$(V \times E,\Gamma,\widehat\phi)$ is an orbifold chart of the
orbifold $\mathcal E$,
such that
the diagram below commutes set-theoretically,
\begin{equation}\label{diag26399}
\xymatrix{
&V \times E  \ar[r]^{\widehat\phi} \ar[d]
&\mathcal E\ar[d]_{\pi}  \\
&V \ar[r]_{\phi}  &X \\
}
\end{equation}
where the left vertical arrow is the projection to the
first factor.
\par
By modifying Definition \ref{def101} in an obvious way we can define the notion of
subchart and then using it to define the structure of {\em orbibundle} \index{orbibundle} on $(X,\mathcal E,\pi )$.
(See \cite[Definition 15.17]{fooo:tech2}.)
\par
We also say {\it vector bundle} \index{vector bundle} on an orbifold instead of orbibundle.
\par
The notion of {\em embedding of orbibundle}  \index{embedding ! orbibundle} can be defined in the same way as
Definition \ref{def262220}. (See \cite[Definition 15.18]{fooo:tech2}.)
\end{defn}

\subsection{Kuranishi structure}
\label{subsec:kuranishi}

We next define the notion of Kuranishi structure.
Hereafter in this part, $X$ is always a separable metrizable space.
\begin{defn}\label{defnKchart}
A {\it Kuranishi chart}
\index{Kuranishi chart! Kuranishi chart} of $X$ is $\mathcal U =(U,\mathcal E,\psi,s)$ with
\index{$\mathcal U = (U,\mathcal E,\psi,s)$}
the following properties.
\begin{enumerate}
\item
$U$ is an orbifold.
\item
$\mathcal E$ is a vector bundle on $U$.
\item
$s$ is a smooth section of $\mathcal E$.
\item
$\psi : s^{-1}(0) \to X$ is a homeomorphism onto an open set.
\end{enumerate}
Here a {\em section} of an orbibundle $\pi : \mathcal E \to X$ is an embedding
$s : X \to \mathcal E$ (of orbifold) such that $\pi \circ s =$ the identity map.
\par
If $U'$ is an open subset of $U$, then by restricting $\mathcal E$, $\psi$ and $s$
to $U'$, we obtain a
Kuranishi chart which we write $\mathcal U\vert_{U'}$
and call it an {\it open subchart}.\index{Kuranishi chart! open subchart}
\par
The {\it dimension} of \index{dimension ! of Kuranishi chart} $\mathcal U =(U,\mathcal E,\psi,s)$
is by definition
$$
\dim \mathcal U = \dim U  - {\rm rank}\, \mathcal E.
$$
Here ${\rm rank}\, \mathcal E$ is the dimension  of the fiber $\mathcal E \to U$.
\par
We say that $\mathcal U =(U,\mathcal E,\psi,s)$ is {\it orientable}
\index{orientable !  Kuranishi chart} if $U$ and $E$ are orientable.
An {\it orientation} \index{orientation !  Kuranishi chart}
of $\mathcal U =(U,\mathcal E,s,\psi)$ is a pair of orientations of
$U$ and $\mathcal E$. (See \cite[Definition A1.17]{fooo:book2} for the precise meaning.)
An open subchart of an oriented Kuranishi chart is oriented.
\par
We call $U$ the \index{Kuranishi neighborhood} {\em Kuranishi neighborhood},
$\mathcal E$  the \index{obstruction bundle} {\em obstruction bundle},
$\psi$ the \index{parametrization map} {\em parametrization}, and
$s$ the {\em Kuranishi map}.\index{Kuranishi map}
\end{defn}
\begin{defn}\label{defKchart}
Let $\mathcal U = (U,\mathcal E,\psi,s)$, $\mathcal U'
= (U',\mathcal E',\psi',s')$ be Kuranishi charts of $X$.
An {\it embedding} of Kuranishi charts $: \mathcal U\to \mathcal U'$
\index{Kuranishi chart! embedding of Kuranishi charts}
\index{embedding ! of Kuranishi charts}
is  a pair $\Phi = (\varphi,\widehat\varphi)$ \index{$\Phi = (\varphi,\widehat\varphi)$} with the following properties.
\begin{enumerate}
\item
$\varphi : U \to U'$ is an embedding of orbifolds.
(See Definition \ref{obbundle})
\item
$\widehat\varphi : \mathcal E \to \mathcal E'$ is an embedding of vector bundles
over $\varphi$.
(See Definition \ref{defnKchart}.)
\item
$\widehat\varphi \circ s = s' \circ \varphi$.
\item
$\psi' \circ \varphi = \psi$ holds on $s^{-1}(0)$.
\item
For each $x \in U$ with $s(x) = 0$, the (covariant) derivative
$
D_{\varphi(x)}s'
$
induces an isomorphism
\begin{equation}\label{form3.1111}
\frac{T_{\varphi(x)}U'}{(D_x\varphi)(T_xU)}
\cong
\frac{\mathcal E'_{\varphi(x)}}{\widehat\varphi(\mathcal E_x)}.
\end{equation}
\end{enumerate}
\par
When $\mathcal U$ and $\mathcal U'$ are oriented
we say $\Phi$ is {\em compatible with orientation}
\index{coordinate change! compatible with orientation}
if the isomorphism (\ref{form3.1111}) induces an
orientation preserving isomorphism.
\end{defn}
\begin{defn}\label{kuranishineighborhooddef}
For $A \subseteq X$, a {\it Kuranishi neighborhood} of $A$ is a Kuranishi chart
such that ${\rm Im}(\psi)$ contains $A$.
In case $A=\{p\}$ we call it a Kuranishi neighborhood of $p$
or a Kuranishi chart at $p$.
\index{Kuranishi chart! Kuranishi neighborhood of $A$}
\end{defn}
\begin{defn}\label{coordinatechangedef}
Let $\mathcal U_1 = (U_1,\mathcal E_1,\psi_1,s_1)$, $\mathcal U_2
= (U_2,\mathcal E_2,\psi_2,s_2)$ be Kuranishi charts of $X$.
A {\it coordinate change in the weak sense} (resp. {\it in the strong sense}) from $\mathcal U_1$
\index{Kuranishi chart! coordinate change in the weak sense}
\index{Kuranishi chart! coordinate change in the strong sense}
\index{coordinate change! Kuranishi chart}
to
$\mathcal U_2$ is a triple $\Phi_{21} = (U_{21},\varphi_{21},\widehat\varphi_{21})$
with the following properties (1) and (2) (resp. (1), (2) and (3)):
\begin{enumerate}
\item
$U_{21}$ is an open subset of $U_1$.
\item
$(\varphi_{21},\widehat\varphi_{21})$ is an
embedding of Kuranishi charts $: \mathcal U_1\vert_{U_{21}}
\to \mathcal U_2$.
\item
$\psi_1(s_1^{-1}(0) \cap U_{21}) = {\rm Im}(\psi_1) \cap   {\rm Im}(\psi_2)$.
\end{enumerate}
In case $\mathcal U_1$ and $\mathcal U_2$ are oriented $\Phi_{21}$
is said to be {\it orientation preserving} \index{orientation preserving ! of
coordinate change} if it is so as an embedding.
\end{defn}
\begin{defn}\label{kstructuredefn}
A {\it Kuranishi structure} $\widehat{\mathcal U}$ of $X$
\index{Kuranishi structure! Kuranishi structure} \index{$\widehat{\mathcal U}
= (\{\mathcal U_p\},\{\Phi_{pq}\})$} assigns a Kuranishi neighborhood $\mathcal U_p
= (U_p,\mathcal E_p,\psi_p,s_p)$
of $p$ to each $p \in X$ and a coordinate change in weak sense
$\Phi_{pq} = (U_{pq},\varphi_{pq},\widehat\varphi_{pq}) : \mathcal U_q \to \mathcal U_p$
to each $p$, $q \in {\rm Im}(\psi_p)$ such that
$q \in \psi_q(s_q^{-1}(0) \cap U_{pq})$ and the following holds
for each $p$, $q \in {\rm Im}(\psi_p)$, $r \in \psi_q(s_q^{-1}(0) \cap U_{pq}) \cap Z$.
\par
We put
$U_{pqr} = \varphi_{qr}^{-1}(U_{pq}) \cap U_{pr}$. Then we have
\begin{equation}\label{form3333}
\Phi_{pr}\vert_{U_{pqr}} = \Phi_{pq}\circ
\Phi_{qr}\vert_{U_{pqr}}.
\end{equation}
\par
We also require that the dimension of $\mathcal U_p$ is independent of
$p$ which we call  the {\it dimension} of $\widehat{\mathcal U}$.
\index{dimension ! of Kuranishi structure}
\par
We replace $U_p$ by a smaller open subset and restrict $s_p$ etc. in an obvious way.
We thus obtain an {\em open substructure} \index{open substructure}.
\end{defn}
We remark that the equality (\ref{form3333}) is the set theoretical equality of the
maps involved.
\par
We next define the notion of a good coordinate system.
\begin{defn}\label{gcsystem}
A {\it good coordinate system} of  $X$
\index{good coordinate system}  is
$${\widetriangle{\mathcal U}}
= (({\frak P},\le), \{\mathcal U_{\frak p}\mid \frak p \in \frak P\},
\{\Phi_{\frak p\frak q}
\mid \frak q \le \frak p\})$$
\index{${\widetriangle{\mathcal U}}
= (({\frak P},\le), \{\mathcal U_{\frak p}\mid \frak p \in \frak P\},
\{\Phi_{\frak p\frak q}
\mid \frak q \le \frak p\})$}
such that:
\begin{enumerate}
\item
$({\frak P},\le)$ is a partially ordered set.
We assume $\# \frak P$ is finite.
\item
$\mathcal U_{\frak p}$ is a Kuranishi chart of $X$.
\item
$
\bigcup_{\frak p \in \frak P} U_{\frak p} \supseteq X.
$
\item
Let $\frak q \le \frak p$. Then
$\Phi_{\frak p \frak q} = (U_{\frak p \frak q},\varphi_{\frak p \frak q},\widehat\varphi_{\frak p \frak q})$ is a coordinate change
in strong sense:
$\mathcal U_{\frak q} \to \mathcal U_{\frak p}$
in the sense of Definition \ref{coordinatechangedef}.
\item
If $\frak r \le \frak q \le \frak p$, then by putting
$U_{\frak p\frak q\frak r} = \varphi_{\frak q\frak r}^{-1}(U_{\frak p\frak q}) \cap U_{\frak p\frak r}$ we have
\begin{equation}
\Phi_{\frak p \frak r}\vert_{U_{\frak p \frak q \frak r}} = \Phi_{\frak p \frak q}\circ
\Phi_{\frak q \frak r}\vert_{U_{\frak p \frak q \frak r}}.
\end{equation}
\item
If ${\rm Im}(\psi_{\frak p}) \cap {\rm Im}(\psi_{\frak q})
\ne \emptyset$, then either
$\frak p \le \frak q$ or $\frak q \le \frak p$ holds.
\item
We define a relation $\sim $ on the disjoint union
$\coprod_{\frak p \in \frak P}U_{\frak p}$ as follows.
Let $x \in U_{\frak p}, y \in U_{\frak q}$. We define $x\sim y$ if and
only if one of the following holds:
\begin{enumerate}
\item $\frak p = \frak q$ and $x=y$.
\item $\frak p \le \frak q$ and $y = \varphi_{\frak q\frak p}(x)$.
\item $\frak q\le \frak p$ and $x = \varphi_{\frak p\frak q}(y)$.
\end{enumerate}
Then $\sim$ is an equivalence relation.
\item
The quotient of $\coprod_{\frak p \in \frak P}U_{\frak p}/\sim$
by this equivalence relation is Hausdorff with respect to the
quotient topology.
\end{enumerate}
\end{defn}
\begin{defn}\label{defn32020202}
Let $\widehat{\mathcal U}$ be a Kuranishi structure and
${\widetriangle{\mathcal U}}$  a good coordinate system
of $X$.
A {\it strict KG-embedding}
\index{KG-embedding! strict} of $\widehat{\mathcal U}$ to ${\widetriangle{\mathcal U}}$
assigns, for each pair $(p,\frak p)$ with $p\in X$, $\frak p\in \frak P$
satisfying
$p \in {\rm Im}(\psi_{\frak p})$,
an embedding of Kuranishi charts
$
\Phi_{\frak p p} = (\varphi_{\frak p p},\widehat\varphi_{\frak p p}) :
\mathcal U_p \to \mathcal U_{\frak p}$ with the following
properties.
\par
If $\frak p,\frak q\in \frak P$,  $\frak q \le \frak p$,
$p \in {\rm Im}(\psi_{\frak p})$,
$q \in {\rm Im}(\psi_{p}) \cap \psi_{\frak q}(U_{\frak p\frak q})$,
then
the following diagram commutes.
\begin{equation}\label{diag33}
\xymatrix{
&\mathcal U_{q}\vert_{U_{pq}
\cap \varphi_{\frak q q}^{-1}(U_{\frak p \frak q})} \ar[r]^{\Phi_{\frak q q}} \ar[d]_{\Phi_{pq}}
&{\mathcal U}_{\frak q}\vert_{{U}_{\frak p\frak q}} \ar[d]_{\Phi_{\frak p \frak q}} \\
&\mathcal U_{p} \ar[r]^{\Phi_{\frak p p}} &{\mathcal U}_{\frak p} \\
}
\end{equation}
\par
A {\it KG-embedding} of $\widehat{\mathcal U}$ to ${\widetriangle{\mathcal U}}$
is by definition a strict KG-embedding \index{KG-embedding} of an open substructure of $\widehat{\mathcal U}$
to ${\widetriangle{\mathcal U}}$.
\end{defn}
The notions of KK-embeding \index{KK-embeding}(embedding of a Kuranishi structure to another Kuranishi structure)
and GG-embedding \index{GG-embedding}(embedding of a good coordinate system to another good coordinate system)
are defined in a similar way.  See \cite[Definitions 3.20 and 3.24]{fooo:tech2}.
\par
The basic existence theorem of good coordinate system is the following.
\begin{thm}\label{thmexist}
{\rm (\cite[Lemma 6.3]{fukaya-ono}, \cite[Lemma A1.11]{fooo:book1},
\cite[Theorem 3.10]{fooo:tech2})}
If  $\widehat{\mathcal U}$ is a Kuranishi structure of $X$
then there exists a good coordinate system ${\widetriangle{\mathcal U}}$
of  $X$  with a KG-embedding $\widehat{\mathcal U}$
to ${\widetriangle{\mathcal U}}$.
\end{thm}
See \cite[Section 11]{fooo:tech2} for the proof.
See \cite{fooohausdorf} for the detail of the argument to
establish Definition \ref{gcsystem} (7)(8).
We can also prove its relative version. (See  \cite[Proposition 5.11 etc.]{fooo:tech2}.)

\section{Strongly smooth map and fiber product}
\label{subsec:Kuranishifiber}

\subsection{Strongly smooth map}
\label{subsec:smoothmap}

\begin{defn}\label{mapkura}
Let $\widehat{\mathcal U}$ be a Kuranishi structure  of $X$.
\begin{enumerate}
\item
A {\it strongly continuous map $\widehat f$}
\index{map from Kuranishi structure ! strongly continuous map}
from $(X,\widehat{\mathcal U})$ to
a topological space $Y$ assigns a continuous map $f_{p}$ from
$U_{p}$ to $Y$ for each $p\in X$ such that
$f_p \circ \varphi_{pq} = f_q$ holds on $U_{pq}$.
\item
In the situation of (1), the map $f:X \to Y$ defined by
$f(p) = f_p(p)$ is a continuous map from $X$ to $Y$.
We call $f : X \to Y$ the {\it underlying continuous map} of $\widehat f$.
\item
When $Y$ is a smooth manifold, we say $\widehat f$ is {\it strongly smooth}
\index{map from Kuranishi structure ! strongly smooth}
if each  $f_p$ is smooth.
\item
A strongly smooth map is said to be {\it weakly submersive}
\index{map from Kuranishi structure ! weakly submersive} if each
$f_p$ is a submersion.
\end{enumerate}
We sometimes say $f$ is a strongly continuous map
(resp. a strongly smooth map, a weakly submersive map) in place of
$\widehat f$ is a strongly continuous map
(resp. a strongly smooth map, weakly submersive),
by an abuse of notation.
\par
We can define similar notions for the good coordinate system
by modifying the definition in an obvious way.
(See  \cite[Definition 3.38]{fooo:tech2}.)
\end{defn}
In  \cite[Proposition 6.49 (2)]{fooo:tech2} the next result is also proved.
\begin{adden}\label{add1}
In the situation of Theorem $\ref{thmexist}$,  suppose there exists a
strongly continuous map $\widehat f :  (X,\widehat{\mathcal U}) \to Y$.
Then we may  choose a good coordinate system ${\widetriangle{\mathcal U}}$ so that
$\widehat f$ is induced by a strongly smooth map
$\widetriangle f :  (X,\widetriangle{\mathcal U}) \to Y$.
If $Y$ is a manifold and $\widehat f$ is strongly smooth (resp. weakly submersive), we may choose
$\widetriangle f$ to be strongly smooth (resp. weakly submersive) also.
\end{adden}

\subsection{Fiber product}
\label{subsec:fproduct}

We next define a fiber product of Kuranishi structure.
Note if $f : X \to Y$ and $f' : M \to Y$ are continuous maps between topological spaces, their fiber product is a topological
space defined by
$$
X \,{}_f \times_{f'}  M  =
\{(x,y) \in X \times M \mid f(x) = f'(y)\}.
$$
\begin{defn}\label{directproduct}
For given $p_i \in X_i$, $i=1,2$, let
$\mathcal U_{p_i} = (U_{p_i},\mathcal E_{p_i},\psi_{p_i},s_{p_i})$ be their Kuranishi neighborhoods.
Then we define a Kuranishi neighborhood of $p = (p_1,p_2) \in X_1 \times X_2$ by
$\mathcal U_{p} = \mathcal U_{p_1} \times \mathcal U_{p_2} = (U_p,\mathcal E_p,\psi_p,s_p)$ where
$$
 (U_p,\mathcal E_p,\psi_p,s_p) =
(U_{p_1}\times U_{p_2},\mathcal E_{p_1}\times \mathcal E_{p_2},\psi_{p_1}\times \psi_{p_2},s_{p_1}\times
s_{p_2}) .
$$
This system satisfies the condition of
Kuranishi neighborhood (Definition \ref{kuranishineighborhooddef}.)
\par
Suppose $q_i \in X_i$ and $q = (q_1,q_2) \in X_1 \times X_2$.
If
$q \in \psi_p(s^{-1}_p(0))$, then it is easy to see that
$q_i \in \psi_{p_i}(s^{-1}_{p_i}(0))$ for $i=1,2$.
Therefore there exist coordinate changes
$\Phi_{p_iq_i} = ({\varphi_{p_iq_i}},\widehat{\varphi}_{p_iq_i},h_{p_iq_i})$ from
$\mathcal U_{q_i}$ to
$\mathcal U_{p_i}$.
We define
$$
\aligned
\Phi_{pq} &=
\Phi_{p_1q_1} \times  \Phi_{p_2q_2}=
(U_{pq},{\varphi}_{pq},\widehat{\varphi}_{pq})
 \\
&=
(U_{p_1q_1} \times U_{p_2q_2},{\varphi}_{p_1q_1}\times {\varphi}_{p_2q_2},\widehat{\varphi}_{p_1q_1}\times \widehat{\varphi}_{p_2q_2}).
\endaligned$$
This satisfies the condition of coordinate change
of
Kuranishi charts (Definition \ref{coordinatechangedef}).
\par
Then it is also easy to show that $(\{\mathcal U_{p_1} \times \mathcal U_{p_2}\},\{\Phi_{p_1q_1} \times  \Phi_{p_2q_2}\})$
defines a Kuranishi structure of $X_1 \times X_2$ in the sense of Definition
\ref{kstructuredefn}.
(We note that effectivity of an orbifold is preserved by the
direct product.)
We call this Kuranishi structure the {\it direct product Kuranishi structure}.
\index{Kuranishi structure! direct product}
\end{defn}
\begin{defn}\label{transverse1}
Let $(X,\widehat{\mathcal U})$ be a space with Kuranishi structure
and $\widehat f = \{f_p\}: (X,\widehat{\mathcal U}) \to N$  a strongly smooth map,
where $N$ is a smooth manifold of finite dimension.
Let $f' : M \to N$ be a smooth map between smooth manifolds.
Assume $M$ is compact.
We say $\widehat f$ is {\it weakly transversal} to $f'$
\index{weakly transversal to a smooth map}
on $X$ if the following holds.
Let  $(p,q) \in X \times_N M$ and
$\mathcal U_p = (U_p,E_p,s_p,\psi_p)$ be a Kuranishi neighborhood
of $p$.
We then require that
for each $(x,y) \in U_p\times M$ with
$f_p(x) = f'(y)$ we have
\begin{equation}\label{transformula}
(d_xf_p)(T_xU_p) + (d_yf')(T_yM) = T_{f(x)}N.
\end{equation}
\end{defn}
\begin{lem}\label{lem24}
In the situation of Definition $\ref{transverse1}$ we can define a Kuranishi structure on
$X \times_N M.
$
\end{lem}
The proof is easy. In fact the Kuranishi neighborhood is obtained by
the fiber product
$U_p \times_N M$.
See \cite[page 34]{fooo:tech2} for the detail. We call the Kuranishi structure obtained in
Lemma \ref{lem24} the \index{fiber product Kuranishi structure}
{\em fiber product Kuranishi structure} and write:
\begin{equation}\label{formB22}
(X,\widehat{\mathcal U}) \,{}_{f}\!\times_N M.
\end{equation}
\begin{defn}\label{defnn325}
Suppose $X_i$ $(i=1,2)$ have Kuranihi structures
$\widehat{\mathcal U_i}$ and
the maps
$\widehat f_i : (X_i,\widehat{\mathcal U_i}) \to N$
are strongly smooth.
We say that $\widehat f_1$ is \index{weakly transversal} {\em weakly transversal}  to
$\widehat f_2$ if
$$
\widehat f_1 \times \widehat f_2 :
(X_1,\widehat{\mathcal U_1}) \times (X_2,\widehat{\mathcal U_2}) \to N \times N
$$
is weakly transversal to the diagonal $N = \Delta_N \subset N\times N$ in the sense of
Definition \ref{transverse1}.
\par
Then we define the {\em fiber product} \index{Kuranishi structure! fiber product}
 \index{fiber product}
$(X_1,\widehat{\mathcal U_1}) {}_{f_1}\times_{f_2} (X_2,\widehat{\mathcal U_2})$
by
$$
(X_1,\widehat{\mathcal U_1}) {}_{f_1}\times_{f_2} (X_2,\widehat{\mathcal U_2})
=
((X_1,\widehat{\mathcal U_1})  \times  (X_2,\widehat{\mathcal U_2})) \times_{N\times N} \Delta_N
$$
where the right hand side is as in (\ref{formB22}).
\end{defn}

\section{CF perturbation and integration along the fiber}
\label{subsec:integration}

In this section we study integration of differential forms on the space with
Kuranishi structure or with good coordinate system.

\subsection{Differential form on the space with Kuranishi structure}
\label{subsec:diffform}
\begin{defn}\label{def323232}
A {\it differential $k$ form $\widehat h$ of a Kuranishi structure}
\index{differential form ! of Kuranishi structure}
$\widehat{\mathcal U}$ of $X$ assigns a differential $k$-form
$h_p$ on $U_p$ for each $p \in X$ such that
$\varphi_{pq}^*h_p = h_q$.
\par
Differential form on a space with good coordinate system can be defined
in the same way. (See  \cite[Definition 7.68]{fooo:tech2}.)
\end{defn}
\begin{defn}\label{def323233}
Let $\widehat f : (X,\widehat{\mathcal U}) \to M$ be a strongly
smooth map to a manifold $M$ and $h$  a differential form on $M$.
For $p \in X$, $\widehat f$ determines a smooth map $f_p : U_p \to M$
where $U_p$ is a Kuranishi neighborhood associated to $p \in X$ by
$\widehat{\mathcal U}$.
Then the family $\{f_p^* h \mid p \in X\}$ becomes a differential form
of $(X,\widehat{\mathcal U})$. We call it the {\em pull back}
\index{pull back of a differential form by strongly smooth map}
and write $\widehat f^* h$.
\index{$\widehat f^* h$}
\par
The case of good coordinate system is similar.
\end{defn}
\begin{defn}\label{def323234}
Let $\widehat h = \{h_p \mid p \in X\}$ be a differential form on $(X,\widehat{\mathcal U})$.
Then $\{dh_p \mid p \in X\}$  is also a differential form on $(X,\widehat{\mathcal U})$,
which we write $d\widehat h$.
The case of good coordinate system is similar.
\end{defn}

Let $\widehat h$ be a differential form on $(X,\widehat{\mathcal U})$.
We want to integrate $\widehat h$ on $X$. $X$ is locally written as
$s_p^{-1}(0)$,
which can be singular. So to make sense out of this integration we need to
perturb $s_p$ appropriately so that the zero set of the perturbed section
becomes smooth.
The notion of CF-perturbation (CF stands for `continuous family') is defined for this purpose.
Below we describe basic definitions of  CF-perturbation and integration along the fiber
of a differential form on a space equipped with good coordinate system or a Kuranishi structure together with
CF perturbation.
To obtain a global definition from the definition on each local chart, we need to
check certain consistency or well-defined-ness of the definitions.
We omit its explanation here since we describe its detail in
\cite[Section 7]{fooo:tech2}.

\subsection{CF-perturbation}
\label{subsec:CFperturbation}
\begin{defn}\label{defn32321}
Let $\mathcal U = (U,\mathcal E,\psi,s)$ be a Kuranishi chart of $X$.
A {\it CF-perturbation}
\index{CF-perturbation! on an orbifold chart} of $\mathcal U$ on an orbibundle chart
$\frak V = (V,\Gamma,E,\phi,\widehat\phi)$ of $(X,\mathcal E)$ consists of
$\mathcal S = (W,\omega,\{{\frak s}^{\epsilon}\})$,
\index{$\mathcal S = (W,\omega,\{{\frak s}^{\epsilon}\})$}
$0 < \epsilon  \le 1$,  with the following properties:
\begin{enumerate}
\item
$W$ is an open subset of a finite dimensional vector space
$\widehat W$ on which $\Gamma$ acts linearly.
$W$ is $\Gamma$ invariant.
\item
$
{\frak s}^{\epsilon} : V \times W \to E
$
is a $\Gamma$-equivariant smooth map for each $0 < \epsilon  \le 1$, which
depends smoothly on $\epsilon$.
We sometimes regard $\frak s^{\epsilon}$ as a family of sections
$V \times W \to E \times V \times W$ of the trivial bundle
$E \times V \times W \to V \times W$.
\item
For $(y,\xi) \in V\times W$ we have
\begin{equation}\label{C0convconti}
\lim_{\epsilon\to 0} {\frak s}^{\epsilon}(y,\xi)
= s(y)
\end{equation}
in compact $C^1$-topology
on $V\times W$.
\item
$\omega$ is a smooth differential form on $W$ of
degree $\dim W$ that is
$\Gamma$ invariant, of compact support and
$$
\int_{W} \omega = 1.
$$
We assume $\omega = \vert\omega\vert{\rm vol}$ where ${\rm vol}$ is a volume form of the oriented manifold $W$
and $\vert\omega\vert$ is a non-negative function.
\end{enumerate}
For each $0 < \epsilon \le 1$, we denote the
restriction of $\mathcal S$ at $\epsilon$, by $\mathcal S^{\epsilon} = (W,\omega,{\frak s}^{\epsilon})$.
\par
Let $\mathcal S_i = (W_i,\omega_i,\{{\frak s}_i^{\epsilon}\})$
($i=1,2$) be two CF-perturbations of $\mathcal U$ on $\frak V$.
We say $\mathcal S_2^{\epsilon}$ is a projection of
$\mathcal S_1^{\epsilon}$
if there exists a map $\Pi :
\widehat W_1 \to \widehat W_2$
with the following properties.
\begin{enumerate}
\item
$\Pi$ is a $\Gamma$ equivariant linear projection
which sends $W_1$ to $W_2$ and satisfies
$
\Pi_{!}(\omega_1) = \omega_2.
$
\item
For each $(y,\xi) \in V_x\times  W_1$
we have
$
{\frak s}_1^{\epsilon}(y,\Pi(\xi)) = {\frak s}_{2}^{\epsilon}(y,\xi).
$
\end{enumerate}
\par
Let $x \in X$. We write  $\mathcal S_1^{\epsilon}
\sim'_x \mathcal S_2^{\epsilon}$ if a restriction of $\mathcal S_1^{\epsilon}$ to a neighborhood
of $x$ is a projection of a restriction of the $\mathcal S_2^{\epsilon}$ to that neighborhood.
\par
We denote by $\sim_x$ the equivalence relation generated by $\sim'_x$.
\par
We say $\mathcal S_1^{\epsilon}$ is equivalent to $\mathcal S_2^{\epsilon}$ and write $\mathcal S_1^{\epsilon}
\sim \mathcal S_2^{\epsilon}$
if $\mathcal S_1^{\epsilon} \sim_x \frak \mathcal S_2^{\epsilon}$ at all $x$.
\end{defn}
\begin{defn}\label{defn33333}
Let $\mathcal U = (U,\mathcal E,\psi,s)$ be a Kuranishi chart of $X$.
A {\em CF-perturbation} of $\mathcal U$  \index{CF-perturbation! on a Kuranishi chart} is represented by
$(\{\mathcal S^{\epsilon}_{\frak r}
\mid \frak r \in \frak R\},\{\frak V_{\frak r}
\mid \frak r \in \frak R\})$ (where  $\frak R$ is a finite set,)
such that
\begin{enumerate}
\item $\frak V_{\frak r} = (V_{\frak r},\Gamma_{\frak r},E_{\frak r},\phi_{\frak r},\widehat\phi_{\frak r})$
is an orbibundle chart of $(X,\mathcal E)$ and the image $U_{\frak r}$ of $\phi_{\frak r}$ for various $\frak r$ is an open covering of $X$.
\item
$\mathcal S_{\frak r}  = (W_{\frak r} ,\omega_{\frak r}, \{{\frak s}_{\frak r} ^{\epsilon}\})$ is a
CF-perturbation of $\mathcal U$
on $\frak V_{\frak r}$.
\item
For each $x \in U_{\frak r_1} \cap U_{\frak r_2}$, we require $\mathcal S_{\frak r_1} \sim_x \mathcal S_{\frak r_2}$.
\end{enumerate}
\par
Two such $(\{\mathcal S^{\epsilon}_{\frak r}\},\{\frak V_{\frak r}\})$ and
$(\{\mathcal S^{\prime,\epsilon}_{\frak r'}\},\{\frak V_{\frak r'}\})$
are said to be {\it equivalent} \index{equivalent! equivalence of CF perturbation on a Kuranishi chart}
if for $x \in U_{\frak r} \cap U'_{\frak r'}$ the following holds.
$$
\mathcal S^{\epsilon}_{\frak r} \sim_x \mathcal S^{\prime,\epsilon}_{\frak r'}.
$$
\par
A CF-perturbation of $\mathcal U$ is by definition an equivalence class of this equivalence
relation.
We denote by $\frak S =[\{\mathcal S^{\epsilon}_{\frak r}
\mid \frak r \in \frak R\},\{\frak V_{\frak r}
\mid \frak r \in \frak R\}]$ the equivalence class of the representative
$(\{\mathcal S^{\epsilon}_{\frak r}\},\{\frak V_{\frak r}\})$.
\par
Let $\frak S$ be a CF-perturbation of $\mathcal U$.
For each $x \in U$ and a sufficiently small orbibundle chart
$\frak V_{x} = (V_{x},\Gamma_{x},E_{x},\phi_{x},\widehat\phi_{x})$,
$\frak S$ determines a CF-perturbation of $\mathcal U$
at $\frak V_{x}$ up to the equivalence.
We call an equivalence class a representative of the germ of CF-perturbation $\frak S$
at $x$.\footnote{We can define a sheaf of CF-perturbations
and germ makes sense using sheaf theory. See \cite[Subsection 7.2]{fooo:tech2}.}
\end{defn}
\begin{defn}\label{submersivepertconlocloc}
In the situation of Definitoin \ref{defn32321},
we consider
a CF-perturbation  $\mathcal S = (W,\omega,\{{\frak s}^{\epsilon}\})$
of $\mathcal U$ on $\frak V$.
\begin{enumerate}
\item
We say $\mathcal S$ is {\it transversal to $0$}
\index{CF-perturbation! transversal to $0$ !
on one orbifold chart}
if there exists $\epsilon_0 > 0$ such that the map ${\frak s}^{\epsilon}$ is transversal to $0$ on
a neighborhood of the support of $\omega$ for all $0 < \epsilon < \epsilon_0$.
In particular
$$
({\frak s}^{\epsilon})^{-1}(0)
=
\{(y,\xi) \in V \times W\mid {\frak s}^{\epsilon}(y,\xi) = 0\}
$$
is a smooth submanifold of $V \times W$ on
a neighborhood of the support of $\omega$.
\item
Let $f : V \to M$ be a $\Gamma$ equivariant smooth map.
We say $f$ is {\it strongly submersive} \index{strongly transversal} with respect to $(\frak V,\mathcal S)$ if $\mathcal S$ is transversal to $0$ and
there exists $\epsilon_0 > 0$ such that the map
\begin{equation}\label{form7575}
f \circ \pi_1\vert_{({\frak s}^{\epsilon})^{-1}(0)} :
({\frak s}^{\epsilon})^{-1}(0) \to M
\end{equation}
is a submersion on
a neighborhood of the support of $\omega$,
for all $0 < \epsilon < \epsilon_0$. Here $\pi_1 : V \times W \to V$ is the
projection.
\index{strongly submersive (w.r.t. CF-perturbation) !
on one chart}
\end{enumerate}
\par
In the situation of Definition \ref{defn33333},
we define transversality to $0$
(resp. strongly submersivity) of
CF perturbation  $\frak S =[\{\mathcal S^{\epsilon}_{\frak r}\},\{\frak V_{\frak r}\}]$ on a Kuranishi chart
as that of each of $\mathcal S^{\epsilon}_{\frak r}$.
We can easily prove that the definition does not depend
on the representative of the CF-perturbation $\frak S$.
\end{defn}
\begin{defn}\label{defn355555}
Let $\mathcal U = (U,\mathcal E,\psi,s)$, $\mathcal U'
= (U',\mathcal E',\psi',s')$ be Kuranishi charts of $X$
and $\Phi = (\varphi,\widehat\varphi)$
an embedding of Kuranishi charts $:\mathcal U
\to \mathcal U'$
in the sense of Definition \ref{defKchart}.
\par
Let $\frak S =[\{\mathcal S^{\epsilon}_{\frak r}
\mid \frak r \in \frak R\},\{\frak V_{\frak r}
\mid \frak r \in \frak R\}]$,
$\frak S' =[\{\mathcal S^{\prime,\epsilon}_{\frak r'}
\mid \frak r' \in \frak R'\},\{\frak V'_{\frak r'}
\mid \frak r' \in \frak R'\}]$
be CF-perturbations of $\mathcal U$ and
$\mathcal U'$, respectively.
\par
We say that $\frak S$, $\frak S'$ are {\it compatible
with the embedding}
\index{compatible
with the embedding! CF-perturbation} $\Phi$ if for each
$x \in U$ we can find a representatives
$(\frak V_x,\mathcal S^{\epsilon}_x)$,
$(\frak V'_x,\mathcal S^{\prime,\epsilon}_{y})$
of $\frak S$, $\frak S'$
at $x$, $y = \varphi(x)$, respectively, such that
the following holds.
We put $\frak V_x = (V_x,\Gamma,E_x,\phi_x,\widehat\phi_x)$,
$\frak V'_y = (V'_y,\Gamma'_y,E'_y,\phi'_y,\widehat\phi'_y)$
and
$\mathcal S_x  = (W_{x} ,\omega_{x}, \{{\frak s}_{x}^{\epsilon}\})$,
$\mathcal S'_y  = (W'_{y} ,\omega'_{y}, \{{\frak s}_{y} ^{\prime,\epsilon}\})$.
\begin{enumerate}
\item
$W_x = W'_y$. $\omega_{x} = \omega'_{y}$.
\item
The next diagram commutes.
\begin{equation}\label{diag93}
\xymatrix{
&E_x \times U_{x} \times W_x \ar[r]^{\widehat\varphi \times {\rm id}}
&E'_y \times U'_{y} \times W'_y  \\
&U_{x}\times W_x \ar[r]^{\varphi \times {\rm id}}
\ar[u]_{{\frak s}_{x}^{\epsilon} } & U'_{y} \times W'_y
\ar[u]_{{\frak s}_{y}^{\prime,\epsilon} }
\\
}
\end{equation}
\end{enumerate}
\end{defn}
\begin{defn}\label{defCFGCSFKUA}
\begin{enumerate}
\item
Let
${\widetriangle{\mathcal U}}
= (({\frak P},\le), \{\mathcal U_{\frak p}\mid \frak p \in \frak P\},
\{\Phi_{\frak p\frak q}
\mid \frak q \le \frak p\})$
be a good coordinate system of $X$.
A  {\em CF-perturbation} \index{CF-perturbation! good coordinate system} $\widetriangle{\frak S}$ of $(X,\widetriangle{\mathcal U})$
assigns a CF-perturbation $\frak S_{\frak p}$
of $\mathcal U_{\frak p}$ to each $\frak p \in
\frak P$, so that whenever $\frak q \le \frak p$,
the restrictions of the corresponding CF-perturbations $\frak S_{\frak p}$, $\frak S_{\frak q}$
to the domain of $\Phi_{\frak p\frak q}$
are compatible with the embedding $\Phi_{\frak p\frak q}$
in the sense of Definition \ref{defn355555}.
\item
Let
$\widehat{\mathcal U}
= (\{\mathcal U_{p}\mid p \in X\},
\{\Phi_{pq}
\})$ be a Kuranishi structure of $X$.
A {\em CF-perturbation $\widehat{\frak S}$}
\index{CF-perturbation! Kuranishi structure} of $(X,\widehat{\mathcal U})$
assigns a CF-perturbation $\frak S_{p}$
of $\mathcal U_{p}$ to each $p \in
X$, such that for any $q \in \psi_p(U_p)$,
 $\mathcal U_{p}$, $\mathcal U_{q}$
 are compatible with the embedding $\Phi_{pq}$
in the sense of Definition \ref{defn355555}.
\end{enumerate}
\end{defn}
\begin{defn}\label{CFtransv}
\begin{enumerate}
\item
A transversality to $0$ of the
CF-perturbation on a good coordinate system
or on a Kuranishi structure is
defined in an obvious way based on Definition
\ref{submersivepertconlocloc}.
\item
Let $\widehat f : (X,\widehat{\mathcal U}) \to M$
(resp. $\widetriangle f : (X,\widetriangle{\mathcal U}) \to M$)
be a strongly smooth map from a space with
Kuranishi structure (resp. good coordinate system).
We say $\widehat f$
(resp. $\widetriangle f$) is {\em strongly submersive with respect to}
$\widehat{\frak S}$
(resp. $\widetriangle{\frak S}$)
if it is so on each Kuranishi chart in the sense of Definition
\ref{submersivepertconlocloc}.
\index{map from Kuranishi structure ! strongly submersive}
\end{enumerate}
\end{defn}

A key existence result of CF-perturbation
is the following Theorem \ref{existperturbcont}.
To state Theorem \ref{existperturbcont} we need the next
definition.
\begin{defn}
We say a good coordinate system
${\widetriangle{\mathcal U'}}
= (({\frak P'},\le), \{\mathcal U'_{\frak p}\mid \frak p' \in \frak P'\},
\{\Phi'_{\frak p'\frak q'}
\mid \frak q' \le \frak p'\})$
is a {\em shrinking} \index{shrinking} of
${\widetriangle{\mathcal U}}
= (({\frak P},\le), \{\mathcal U_{\frak p}\mid \frak p \in \frak P\},
\{\Phi_{\frak p\frak q}
\mid \frak q \le \frak p\})$
if the following holds.
\begin{enumerate}
\item
${\frak P} = {\frak P}'$.
\item
$\mathcal U'_{\frak p}$ is a restriction of
$\mathcal U_{\frak p}$ to a relatively compact
open subset $U'_{\frak p}$ of $U_{\frak p}$.
\item
$\Phi'_{\frak p\frak q}$ is the restriction of
$\Phi_{\frak p\frak q}$ to
$
\varphi_{\frak p\frak q}^{-1}(U'_{\frak p})
\cap U'_{\frak q}.
$
\end{enumerate}
\end{defn}
\begin{thm}\label{existperturbcont}
{\rm (\cite[Theorem 7.49]{fooo:tech2})}
Let ${\widetriangle{\mathcal U}}$ be a good coordinate system
of $X$
and ${\widetriangle{\mathcal U'}}$ its shrinking.
\begin{enumerate}
\item
There exists a CF-perturbation $\widetriangle{\frak  S'}$
of $(X,{\widetriangle{\mathcal U'}})$ transversal to $0$.
\item
If $\widetriangle f : (X,{\widetriangle{\mathcal U}}) \to M$ is a
weakly submersive and strongly smooth map, then
we may take a CF-perturbation $\widetriangle{\frak  S}$
so that
$\widetriangle f$ is strongly submersive with respect to $\widetriangle{\frak  S}$.
\end{enumerate}
\end{thm}
Note in \cite{fooo:tech2} we use the notion of a support system
and state Theorem \ref{existperturbcont} in a slightly different way.
Here to shorten the description we use the notion of the shrinking
in place of the support system.
\par
We can prove a relative version of Theorem \ref{existperturbcont}.
See \cite[Proposition 7.52 and Lemma 7.53]{fooo:tech2}.
\par
For the Kuranishi structure (instead of the good coordinate system)
we have the following:

\begin{thm}\label{existperturbcontkura}
For any Kuranishi
structure ${\widehat{\mathcal U}}$ of $X$,
we can find a Kuranishi structure
${\widehat{\mathcal U'}}$ of $X$
and a KK-embedding
${\widehat{\mathcal U}} \to {\widehat{\mathcal U'}}$
such that the following holds.
\begin{enumerate}
\item
There exists a CF-perturbation $\widehat{\frak  S'}$
of $(X,{\widehat{\mathcal U'}})$ transversal to $0$.
\item
If $\widehat f : (X,{\widehat{\mathcal U}}) \to M$ is a
weakly submersive strongly smooth map, then
we may take the CF-perturbation $\widehat{\frak  S'}$
so that
$\widehat f$ is strongly submersive with respect to $\widehat{\frak  S'}$.
\end{enumerate}
\end{thm}
This theorem follows from Theorems \ref{existperturbcont}, \ref{thmexist},
Addendum \ref{add1} and \cite[Lemma 9.9]{fooo:tech2}.

\subsection{Integration along the fiber}
\label{subsec:Integration}

\begin{defn}\label{defn311111}
In the situation of Definition \ref{defn32321}
we assume that $f : U \to M$ is a smooth map to a smooth manifold $M$.
Let $\mathcal S = (W,\omega,\{{\frak s}^{\epsilon}\})$
be a CF-perturbation
of $\mathcal U$ on $\frak V = (V,\Gamma,E,\phi,\widehat{\phi})$.
We assume $f$ is strongly submersive
with respect to $\mathcal S$.
Let $h$ be a smooth differential form on $U$ that has compact support in $\phi(V)$.
Then we define a smooth differential form $f!(h;\mathcal S^{\epsilon})$
on $M$
for each $\epsilon > 0$ by the equation (\ref{form72}) below.
We call the differential form $f!(h;\mathcal S^{\epsilon})$
 the {\it integration along the fiber} of $h$ with respect to $f$, $\frak S$.
\index{integration along the fiber!  one chart}
\par
For any smooth differential form $\rho$  on $M$,
we require $f!(h;\mathcal S^{\epsilon})$ to satisfy
\begin{equation}\label{form72}
\# \Gamma
\int_{M} f!(h;\mathcal S^{\epsilon}) \wedge \rho
=
\int_{({\frak s}^{\epsilon})^{-1}(0)}
\pi_1^*(\tilde h) \wedge \pi_1^*(f^*\rho) \wedge \pi_2^*\omega.
\end{equation}
Here $\pi_1$ (resp. $\pi_2$) is the projection of $V \times W$ to the
first (resp. second) factor,
$\# \Gamma$ is the order of the finite set $\Gamma$, and
$\tilde h$ is the pull back of $h$ to $V$.
\par
We remark that
\begin{equation}\label{form35}
\deg f!(h;\mathcal S^{\epsilon})
= \deg h + \dim M - \dim \mathcal U.
\end{equation}
\end{defn}
The unique existence of such a differential form $f!(h;\mathcal S^{\epsilon})$
is an immediate consequence of the existence of integration along the fiber
of a smooth form by a proper submersion.
\begin{rem}
Note that the right hand side of (\ref{form72}) depends on $\epsilon$.
Therefore the differential form  $f!(h;\mathcal S^{\epsilon})$ depends on $\epsilon$.
\end{rem}

\begin{defn}
In the situation of Definition \ref{defn33333}
let
$\frak S =[\{\mathcal S^{\epsilon}_{\frak r}
\mid \frak r \in \frak R\},\{\frak V_{\frak r}
\mid \frak r \in \frak R\}]$
be a CF-perturbation of $\mathcal U$
and $f : U \to X$ a smooth map.
We assume $f$ is strongly submersive with respect to
$\frak S$.
We consider the open covering
\begin{equation}\label{cover35}
U = \bigcup_{\frak r \in \frak R} U_{\frak r}
\end{equation}
where $U_{\frak r} = \phi_{\frak r}(V_{\frak r})$.
Let $\chi_{\frak r}$  be a partition of unity
subordinate to the covering (\ref{cover35}).
Let $h$ be a differential form on $U$ of compact support.
We then define the {\em integration along the fiber}
\index{integration along the fiber!  on Kuranishi chart} of
$h$ with respect to $f$, $\frak S$ by
the following formula.
\begin{equation}\label{intdef}
f!(h;\frak S^{\epsilon})
=
\sum_{\frak r \in \frak R}
f!(\chi_{\frak r} h;\mathcal S_{
\frak r}^{\epsilon}).
\end{equation}
Here the right hand is defined by
Definition \ref{defn311111}.
The degree is given by (\ref{form35}).
\end{defn}
We can prove that the right hand side of
(\ref{intdef}) is independent of the choice of
partition of unity $\chi_{\frak r}$ as well as the
representative $(\{\mathcal S^{\epsilon}_{\frak r}
\mid \frak r \in \frak R\},\{\frak V_{\frak r}
\mid \frak r \in \frak R\})$ of $\frak S$.
(\cite[Proposition 7.80]{fooo:tech2}.)
\par
We next discuss the case of good coordinate
system $\widetriangle{\mathcal U}
= (({\frak P},\le), \{\mathcal U_{\frak p}\mid \frak p \in \frak P\},
\{\Phi_{\frak p\frak q}
\mid \frak q \le \frak p\})$ of $X$.
We use the equivalence relation
$\sim$ on the disjoint union $\coprod_{\frak p\in\frak P}U_{\frak p}$,
which is defined in Definition \ref{gcsystem} (7)
and denote by $\vert \widetriangle{\mathcal U}\vert$  the
set of this equivalence class.
This space is Hausdorff with respect to the quotient topology.
Suppose $\widetriangle{\mathcal U}$ is a shrinking of
$\widetriangle{\mathcal U}^+$.
Then $\vert \widetriangle{\mathcal U}\vert$
is a subset of $\vert \widetriangle{\mathcal U^+}\vert$.
We put the topology of
$\vert \widetriangle{\mathcal U}\vert$
that is induced from the quotient topology of
$\vert \widetriangle{\mathcal U^+}\vert$.
Then $\vert \widetriangle{\mathcal U}\vert$
is metrizable with respect to this topology.
(\cite[Proposition 2.11]{fooohausdorf}.)
Hereafter we use this topology on
$\vert \widetriangle{\mathcal U}\vert$.
We take its metric $d$ and fix it.
We remark that
by construction of $\vert \widetriangle{\mathcal U}\vert$ we have
natural embeddings $X \hookrightarrow \vert \widetriangle{\mathcal U}\vert$,
$U_{\frak p} \hookrightarrow \vert \widetriangle{\mathcal U}\vert$ so
we may regard them as subsets
\begin{equation}
X \subset \vert \widetriangle{\mathcal U}\vert,
\quad
U_{\frak p} \subset \vert \widetriangle{\mathcal U}\vert.
\end{equation}
We first discuss the notion of partition of unity with respect to the
covering $\{U_{\frak p} \mid \frak p \in \frak P\}$
of $\vert \widetriangle{\mathcal U}\vert$.
\begin{defn}
A continuous function $g : \vert \widetriangle{\mathcal U}\vert
\to \R$ is said to be {\it smooth}\index{smooth}
if the restriction of $g$ to any $U_{\frak p}$ is a
smooth function.
\end{defn}
\begin{defn}\label{partuni}
Let $K_{\frak p} \subset U_{\frak p}$ be a compact subset such that
\begin{equation}
\bigcup_{\frak p \in \frak P} \psi_{\frak p}(s_{\frak p}^{-1}(0) \cap K_{\frak p}) \supseteq X,
\end{equation}
and $\delta$ be a sufficiently small positive number.
\par
A $(\{K_{\frak p}\},\delta)$-{\it partition of unity} of $(X; \widetriangle{\mathcal U})$
\index{partition of unity ! good coordinate system}
assigns $\chi_{\frak p}$ to each $\frak p \in \frak P$, so that:
\begin{enumerate}
\item
$\chi_{\frak p} : \vert \widetriangle{\mathcal U}\vert \to [0,\infty)$ is a smooth function.
\item
The support of $\chi_{\frak p}$ is in the $\delta$ neighborhood of $K_{\frak p}$
in $\vert \widetriangle{\mathcal U}\vert$.\footnote{We remark that this $\delta$ neighborhood
is not contained in $U_{\frak p}$ in general.}
\item
$$
\sum_{\frak p \in \frak P} \chi_{\frak p} \equiv 1
$$
holds on the $\delta$-neighborhood $\frak U(X)$ of $X$ in $\vert \widetriangle{\mathcal U}\vert$.
\end{enumerate}
\end{defn}
Existence of a $(\{K_{\frak p}\},\delta)$-partition of unity of $(X; \widetriangle{\mathcal U})$ for
sufficiently small $\delta$ is proved in \cite[Proposition 7.67]{fooo:tech2}.

\begin{defn}\label{defintfiberGCS}
Let $(X,\widetriangle{\mathcal U})$ be a space equipped with an oriented good coordinate system
and $\widetriangle f : (X,\widetriangle{\mathcal U}) \to M$  a strongly smooth map.
Let $\widetriangle{\frak S} = \{ {\frak S}_{\frak p} \mid \frak p \in \frak P\}$ be a  CF-perturbation
of  $(X,\widetriangle{\mathcal U})$ such that $\widetriangle f$ is strongly submersive with respect to $\widetriangle{\frak S}$.
Let $\widetriangle h$ be a differential form on $(X,\widetriangle{\mathcal U})$.
We take $K_{\frak p}$ as in Definition \ref{partuni} and a sufficiently small $\delta$.
Let $\{\chi_{\frak p} \mid \frak p \in \frak P \}$ be a $(\{K_{\frak p}\},\delta)$-partition of unity of $(X; \widetriangle{\mathcal U})$.
\par
Then the integration along the fiber \index{integration along the fiber! good coordinate system} of $\widetriangle h$
with respect to  $\widetriangle{f}$, $\widetriangle{\frak S}$ is defined by the formula
\begin{equation}\label{formula714}
{\widetriangle f}!(\widetriangle h;\widetriangle{{\frak S}^{\epsilon}})
=
\sum_{\frak p\in \frak P}
f_{\frak p}!(\chi_{\frak p}h_{\frak p};{\frak S}^{\epsilon}_{\frak p}
\vert_{\frak U(X) \cap B_{\delta}(K_{\frak p}) \cap U_{\frak p}}).
\end{equation}
Here $B_{\delta}(K_{\frak p})$ is the $\delta$ neighborhood of $K_{\frak p}$ in
$\vert \widetriangle{\mathcal U}\vert$.
The degree is given by (\ref{form35}).
\end{defn}
\begin{prop}
The right hand side of $(\ref{formula714})$ depends on $\widetriangle h$, $\widetriangle{\frak S}$, $\epsilon$
but is independent of other choices such as partition of unity $\chi_{\frak p}$, $\delta$ and $K_{\frak p}$
provided $\delta$ and $\epsilon$ are sufficiently small.
\end{prop}
This is proved in \cite[Proposition 7.80]{fooo:tech2}.
\par
We next study the case of Kuranishi structure.

\begin{adden}\label{add2}
In the situation of Theorem $\ref{thmexist}$
suppose that we have a CF-perturbation $\widehat{\frak S}$
on $\widehat{\mathcal U}$.
Then we may choose $\widetriangle{\mathcal U}$   such that $\widehat{\frak S}$ is induced
from $\widetriangle{\frak S}$.\footnote{Namely ${\frak S}_p$
and ${\frak S}_{\frak p}$ are compatible with the embedding
$\Phi_{\frak p,p} : \mathcal U_p \to \mathcal U_{\frak p}$ in the sense of Definition \ref{defn355555}.
Here $\Phi_{\frak p,p}$ is a member of KG-embedding $\widehat{\mathcal U} \to \widetriangle{\mathcal U}$.}
\par
If $\widehat{\frak S}$ is transversal to $0$ then
we may choose $\widetriangle{\frak S}$ so that it is transversal to $0$.
\par
Suppose in addition that we are in the situation of Addendum $\ref{add1}$,
that is $\widehat f :  (X,\widehat{\mathcal U}) \to M$,
$\widetriangle f :  (X,\widetriangle{\mathcal U}) \to M$ are weakly submersive.
We assume that $\widehat f$ is strongly submersive with respect to $\widetriangle{\frak S}$.
Then we may choose $\widetriangle{\frak S}$ so that $\widetriangle f$
is strongly submersive with respect to  $\widetriangle{\frak S}$.
\end{adden}
See \cite[Proposition 9.10 (2)(3) and Proposition 6.49 (2)]{fooo:tech2} for the proof.
In case a differential form $\widehat h$ on $(X,\widehat{\mathcal U})$ is given
we may choose
a differential form $\widetriangle h$   on $(X,\widetriangle{\mathcal U})$
so that $\widehat h$ is induced by $\widetriangle h$.
\footnote{Namely $h_p$  is a pull back of
and $h_{\frak p}$ by the embedding
$\Phi_{\frak p,p} : \mathcal U_p \to \mathcal U_{\frak p}$.
Here $\Phi_{\frak p,p}$ is a member of KG-embedding $\widehat{\mathcal U} \to \widetriangle{\mathcal U}$.}
\begin{defn}\label{def320222}
Let $(X,\widehat{\mathcal U})$ be an oriented space with Kuranishi structure and
$\widehat f : (X,\widehat{\mathcal U}) \to M$  a weakly  submersive map.
Let $\widehat{\frak S}$ be a CF-perturbation of $(X,\widehat{\mathcal U})$
such that $\widehat f$ is strongly submersive with respect to $\widetriangle{\frak S}$.
Let $\widehat h$ be a differential form on $(X,\widehat{\mathcal U})$.
\par
Then we define the {\em integration along the fiber} of $\widehat h$ by the formula:
\index{integration along the fiber! Kuranishi structure}
\begin{equation}\label{formula715}
{\widehat f}!(\widehat h;\widehat{{\frak S}^{\epsilon}})
=
{\widetriangle f}!(\widetriangle h;\widetriangle{{\frak S}^{\epsilon}})
\end{equation}
where $\widetriangle f, \widetriangle h, \widetriangle{\frak S}$ are as in Addendum \ref{add2}
and right hand side is defined by (\ref{formula714}).
The degree is given by (\ref{form35}).
\end{defn}
\begin{prop}
The right hand side of $(\ref{formula715})$ depends on $\widehat h$, $\widehat{\frak S}$, $\epsilon$
but is independent of other choices such as $\widetriangle{\mathcal U}, \widetriangle f, \widetriangle h, \widetriangle{\frak S}$ if $\epsilon$ is sufficiently small.
\end{prop}
This is \cite[Theorem 9.14]{fooo:tech2}.

\section{Stokes' theorem}
\label{subsec:stokes}

\subsection{Normalized boundary}
\label{subsec:normalizedbdl}

To describe Stokes' theorem we first discuss a few points on the
boundary of manifold, orbifold and of Kuranishi structure.
Let $M$ be a manifold with corner.
We say a point $p \in M$ lies in the codimension $k$ corner of $M$
if $p$ has a neighborhood diffeomorphic to
$[0,1)^k \times \R^{n-k}$ where $p$ corresponds to the origin.
We denote by $\overset{\circ}{S^k}(M)$ the set of
all the points  which lie in the codimension $k$ corner of $M$.
$\overset{\circ}{S^k}(M)$ has a canonical structure of a manifold of
dimension $n-k$ (without boundary).
Its closure in general does not have a structure of a manifold with
corner. However we can easily show the next lemma.

\begin{lem}
There exists an $(n-k)$ dimensional manifold
with corner $\widehat{S}^k(M)$, such that the inclusion
$\overset{\circ}{S^k}(M) \to M$ extends to a map
$\pi : {\widehat{S}^k}(M) \to M$ such that
$$
\# (\pi^{-1}(p)) = \frac{\ell!}{k! (\ell-k)!}
$$
if $p \in \overset{\circ}{S^{\ell}}(M)$ for $\ell \ge $k.
\end{lem}
\begin{defn}\label{defnnormalizedboundary}
We denote  $\partial M = {\widehat{S}^1}(M)$, which we call the
{\em normalized boundary}
\index{normalized boundary! manifold} of $M$.
\index{normalized corner! manifold}
\index{$\widehat{S}^k(X)$}
We call it the boundary in case the meaning is clear.
\par
We call ${\widehat{S}^k}(M)$ the normalized codimension $k$ corner of $M$.
\end{defn}
\begin{exm}
Let $M = [0,\infty)^n$. Then $\overset{\circ}{S^{1}}(M)$ is a disjoint union of
$n$ copies each of which is identified with $(0,\infty)^{n-1}$.
$\partial M$ is then a {\it disjoint} union of the closures of those $n$ copies.
Namely it is a disjoint union of $n$ copies of $[0,\infty)^{n-1}$.
\end{exm}
\begin{defn}
\begin{enumerate}
\item An {\em orbifold chart with corner} \index{orbifold with corner} of $X$
is a triple $(V,\Gamma,\phi)$ satisfying the same condition as Definition \ref{def101}
except that this time $V$ may have a corner.
\item
We define the notion of the structure of {\em orbifold with corner} on $X$ using item (1)
in the same way as Definition \ref{def101}.
\item
By taking the normalized boundary (resp. corner)
of $V$ for each $V$ we can define the {\em normalized boundary} (resp. corner) of $X$,
which we denote by $\partial X$
(resp. ${\widehat{S}^k}(X)$). Then it itself carries the structure of an orbifold with corner.
\index{normalized boundary! orbifold}
\index{normalized corner! orbifold}
$\overset{\circ}{S^{\ell}}(X)$ is defined in the same way.
\item
If $X$ is an oriented orbifold with boundary then its
normalized boundary $\partial X$ is also oriented.
\item
The notion of {\em embedding of orbifold with corner}
is defined in the same way as Definition \ref{defn32020202}.
Note we require $f(\overset{\circ}{S^{\ell}}(X)) \subseteq
\overset{\circ}{S^{\ell}}(Y)$ for an embedding $f: X \to Y$ of orbifold with corner.
\end{enumerate}
\end{defn}

\begin{defn}\label{defn345}
By considering the case when $U$ is an orbifold with corners
in the definition of Kuranishi chart $\mathcal U = (U,\mathcal E,\psi,s)$,  we
define a {\em Kuranishi chart with corners}.
\index{Kuranishi chart! Kuranishi chart with corners}
\par
Using it in an obvious way we define {\em Kuranishi structure and
good coordinate system with corners}.
\index{Kuranishi structure! Kuranishi structure with corners}
\index{good coordinate system! good coordinate system with corners}
\par
\index{normalized boundary! Kuranishi structure}
\index{normalized corner! Kuranishi structure}
The {\em normalized boundary} $\partial(X,\widehat{\mathcal U})$
(resp. {\em normalized corner} $\widehat{S}^k(X,\widehat{\mathcal U})$) of a space with  Kuranishi structure with corners
can be defined by considering normalized boundary (resp. corner) on each
Kuranishi chart.
We can define normalized boundary and normalized corner
in the same way for the case of good coordinate system with corners in an obvious way.
\index{normalized boundary! good coordinate system}
\index{normalized corner! good coordinate system}
\end{defn}

\begin{defn}
Let $M_i$ be a manifold with corners, $f_i : M_i \to N$ a smooth map, for $i=1,2$.
We say $f_1$ is {\em transversal}
\index{transversality in the case of manifold with corners} to $f_2$ if
the restriction of $f_1$ to $\overset{\circ}{S^{k}}(M_1)$ is
transversal to the restriction of $f_2$ to $\overset{\circ}{S^{\ell}}(M_2)$
for any $k,\ell$.
Transversality to zero of the section on a manifold with corners is defined in the same way.
\par
The case of orbifold with corners is treated in the same way.
\par
Various definitions of the  transversality on various objects on Kuranishi structure
and/or good coordinate system can be generalized to the case
of Kuranishi structure
and/or good coordinate system with corners.
\end{defn}
Let $\widehat f : (X,\widehat{\mathcal U}) \to M$ be a map from a space with Kuranishi structure
with corners to a manifold (without boundary).
Then we can define its strong submersivity as above.
(Namely we require the strong submersivity of the restrictions to all the normalized
corners as well as  interior.)
It is then easy to see that if $\widehat f$ is weakly submersive then
it induces a weakly submersive map
 $\widehat{S}^k\widehat f : \widehat{S}^k(X,\widehat{\mathcal U}) \to M$.
 The same holds for good coordinate system.
\par
If $\widehat{\frak S}$ is a CF-perturbation of $(X,\widehat{\mathcal U})$ then
it induces CF-perturbation $\widehat{S}^k\widehat{\frak S}$ of $\widehat{S}^k(X,\widehat{\mathcal U})$.
The same holds for the case of good coordinate system.
This construction preserves the transversality to $0$ and the strong submersivity of maps.

\subsection{Statement of Stokes' theorem}
\label{subsec:stakes}
\index{Stokes' theorem}
\begin{thm}\label{them47}
Let $(X,\widetriangle{\mathcal U})$ be a space equipped with a good coordinate system with corners
and $\widetriangle h$ a differential form on it.
Let $\widetriangle{\frak S}$ be a CF-perturbation of $(X,\widetriangle{\mathcal U})$
and $\widetriangle f : (X,\widetriangle{\mathcal U}) \to M$  a strongly smooth map.
We assume $\widetriangle f$ is strongly submersive with respect to
$\widetriangle{\frak S}$.
Then we have the following equality:
\begin{equation}
d\left(\widetriangle f!(\widetriangle h;\widetriangle{{\frak S}^{\epsilon}})\right)
=
\widetriangle f!(d\widetriangle h;\widetriangle{{\frak S}^{\epsilon}})
+
\widetriangle f_{\partial}!(\widetriangle {h_{\partial}};\widetriangle{{\frak S}_{\partial}^{\epsilon}}).
\end{equation}
for each sufficiently small $\epsilon>0$.
Here $\widetriangle f_{\partial}$, $\widetriangle {h_{\partial}}$, $\widetriangle{\frak S}_{\partial}$ are
restrictions of $\widetriangle f$,  $\widetriangle {h}$, $\widetriangle{\frak S}$ to
$(\partial X,\partial \widetriangle{\mathcal U})$ respectively.
\end{thm}
This is  \cite[Theorem 8.11]{fooo:tech2}.
(See also \cite{fooo:bulk} Lemma 12.13.)
We remark that $d\widetriangle h$ is defined in Definition \ref{def323234}.
For the Kuranishi structure we have the same statement.
\begin{thm}\label{them48}
Let $(X,\widehat{\mathcal U})$ be a space equipped with a Kuranishi structure with corners
and $\widehat h$ a differential form on it.
Let $\widehat{\frak S}$ be a CF-perturbation of $(X,\widehat{\mathcal U})$ and
$\widehat f : (X,\widehat{\mathcal U}) \to M$  a strongly smooth map.
We assume $\widehat f$ is strongly submersive with respect to
$\widehat{\frak S}$.
Then we have the following equality:
\begin{equation}\label{fporm42}
d\left(\widehat f!(\widehat h;\widehat{{\frak S}^{\epsilon}})\right)
=
\widehat f!(d\widehat h;\widehat{{\frak S}^{\epsilon}})
+
\widehat f_{\partial}!(\widehat {h_{\partial}};\widehat{{\frak S}_{\partial}^{\epsilon}}).
\end{equation}
for each sufficiently small $\epsilon>0$.
Here $\widehat f_{\partial}$, $\widehat {h_{\partial}}$, $\widehat{\frak S}_{\partial}$ are
restrictions of $\widehat f$,  $\widehat {h}$, $\widehat{\frak S}$ to
$(\partial X,\partial \widehat{\mathcal U})$ respectively.
\end{thm}
This is  \cite[Proposition 9.26]{fooo:tech2}.
We remark that $d\widehat h$ is defined in Definition \ref{def323234}.
\par
The proof of Theorem \ref{them47} is similar to  the proof of
Stokes' theorem in manifold theory. Namely we use the partition of
unity and show that it suffices to prove the equality on each chart.
We then use Stokes' theorem in manifold theory to prove
the equality on one chart.
We omit the proof but refer the readers to \cite[Section 8]{fooo:tech2}.
Theorem \ref{them48} follows from Theorem \ref{them47}.
\par
We consider the case when $M$ is a point.
In this case the strong submersivity of $f : (X,\widehat{\mathcal U}) \to M$
with respect to $\widehat{\frak S}$
is nothing but the transversality to 0 of  $\widehat{\frak S}$.
\par
Let $\widehat h$ be a differential form of $(X,\widehat{\mathcal U})$
of degree $\dim (X,\widehat{\mathcal U})$.
Then we define
\begin{equation}\label{form34333}
\int_{((X,\widehat{\mathcal U}),\widehat{\frak S}^{\epsilon})} \widehat h
: =
{\widehat f}!(\widehat h;\widehat{{\frak S}^{\epsilon}}).
\end{equation}
Note in this case the right hand side is a differential 0 form on a
point, which is nothing but a real number. (It depends on $\epsilon$.)
We sometimes omit $\widehat{{\frak S}^{\epsilon}}$ from the notation
in case the choice of CF-perturbation is clear.
\par
In case $\widehat h = \dim (X,\widehat{\mathcal U}) + 1$,
(\ref{fporm42}) becomes the next formula:
\begin{equation}\label{stokessss}
\int_{(\partial(X,\widehat{\mathcal U}),\partial\widehat{\frak S}^{\epsilon})} \widehat{h_{\partial}}
=
\int_{((X,\widehat{\mathcal U}),\widehat{\frak S}^{\epsilon})} d\widehat h.
\end{equation}
If the obstruction bundles of $\widehat{\mathcal U}$ are all zero this Kuranishi
structure gives an orbifold structure of $X$. In this case
(\ref{stokessss}) is the standard Stokes' formula for orbifold.
Thus Theorems \ref{them48} and \ref{them47} are generalizations
of the standard Stokes' formula.

\section{Composition formula}
\label{sec:composition}

Another main result of the theory of integration along the fiber
is the composition formula which we describe in this section.
This is a version of Fubini's theorem on integration.

\subsection{Smooth correspondence and its perturbation}
\label{subsec:correspondence}

\begin{defn}
Let $M_s$ and $M_t$ be oriented compact smooth manifolds without boundary.
A {\em smooth correspondence} \index{smooth correspondence} from $M_s$ to $M_t$ is an object
$((X,\widehat{\mathcal U}),\widehat f_s,\widehat f_t)$ such that:
\begin{enumerate}
\item
$(X,\widehat{\mathcal U})$ is a space with an oriented Kuranishi structure.
(It may have boundary and corner.)
\item
$\widehat f_s : (X,\widehat{\mathcal U}) \to M_s$ is a strongly
smooth map.
\item
$\widehat f_t : (X,\widehat{\mathcal U}) \to M_t$ is a strongly
smooth and weakly submersive map.
\end{enumerate}
\end{defn}
Here $s$ and $t$ in the suffix stands for source and target, respectively.
\begin{defn}
A {\em perturbed smooth correspondence}
\index{perturbed smooth correspondence} from $M_s$ to $M_t$ is
an object $\Xi = ((X,\widehat{\mathcal U}),\widehat{\frak S},\widehat f_s,\widehat f_t)$
such that
\begin{enumerate}
\item
$((X,\widehat{\mathcal U}),\widehat f_s,\widehat f_t)$ is a
smooth correspondence from $M_s$ to $M_t$.
\item
$\widehat{\frak S}$ is a CF-perturbation of $(X,\widehat{\mathcal U})$
such that $\widehat f_t$ is strongly submersive with respect to
$\widehat{\frak S}$.
\end{enumerate}
\end{defn}
\begin{defn}
If $((X,\widehat{\mathcal U}),\widehat f_s,\widehat f_t)$ is a smooth correspondence
then $\partial(X,\widehat{\mathcal U})$ together with the restriction of
$\widehat f_s,\widehat f_t$ is a smooth correspondence, which we write
$\partial((X,\widehat{\mathcal U}),\widehat f_s,\widehat f_t)$.
\par
In the same way we can define the boundary of a perturbed
smooth correspondence $\Xi$
which we write $\partial\Xi$.
\end{defn}
\begin{defn}\label{defn354}
Let $\Xi = ((X,\widehat{\mathcal U}),\widehat{\frak S},\widehat f_s,\widehat f_t)$
be a perturbed smooth correspondence from $M_s$ to $M_t$.
We define an $\R$ linear map
$$
{\rm Corr}^{\epsilon}_{\Xi} : \Omega(M_s) \to \Omega(M_t)
$$
by \index{$\text{\rm Corr}^{\epsilon}_{\Xi}$}
\begin{equation}
{\rm Corr}^{\epsilon}_{\Xi}(h)
=
{\widehat f_t}!(\widehat f_s^* h;\widehat{{\frak S}^{\epsilon}}).
\end{equation}
Note the pull back $\widehat f_s^*$ is defined by
Definition \ref{def323233} and ${\widehat f_t}!$
is defined by Definition \ref{def320222}.
Here $\Omega(M_s)$, $\Omega(M_t)$ are the vector spaces
of all smooth differential forms on $M_s$, $M_t$, respectively.
\end{defn}
We remark that ${\rm Corr}^{\epsilon}_{\Xi}$ is a linear map
and
\begin{equation}
\deg {\rm Corr}^{\epsilon}_{\Xi}(h)
=
\deg h + \dim M_t - \dim (X,\widehat{\mathcal U}).
\end{equation}
Note $d\widehat f_t^* h = \widehat f_t^* (dh)$.
Therefore Stokes' theorem implies that:
\begin{equation}
d\circ {\rm Corr}^{\epsilon}_{\Xi}
= {\rm Corr}^{\epsilon}_{\Xi} \circ d + {\rm Corr}^{\epsilon}_{\partial\Xi}.
\end{equation}

\subsection{Composition of smooth correspondences}
\label{subsec:compcorrespondence}

We next define a composition of smooth correspondences.
\begin{defn}
Let $((X_{21},\widehat{\mathcal U_{21}}),\widehat f_{1;21},\widehat f_{2;21})$
and
$((X_{32},\widehat{\mathcal U_{32}}),\widehat f_{2;32},\widehat f_{3;32})$
be smooth correspondences from $M_1$ to $M_2$ and from $M_2$ to $M_3$
respectively.
Since $\widehat f_{2;21}$ is weakly submersive the fiber product
$$
(X_{21},\widehat{\mathcal U_{21}}) \,_{\widehat f_{2;21}}\times_{\widehat f_{2;32}}
(X_{32},\widehat{\mathcal U_{32}})
$$
is well-defined. We write it $(X_{31},\widehat{\mathcal U_{31}})$.
Then it is easy to see that $\widehat f_{1;21}$ (resp. $\widehat f_{3;32}$)
induces
$\widehat f_{1;31} : (X_{31},\widehat{\mathcal U_{31}}) \to M_1$
(resp. $\widehat f_{3;31} : (X_{31},\widehat{\mathcal U_{31}}) \to M_3$).
We can use weakly submersivity of $\widehat f_{2;21}$, $\widehat f_{3;32}$
to show that $\widehat f_{3;31}$ is weakly submersive.\footnote{See \cite[Lemma 10.18 (2)]{fooo:tech2} for its proof.}
We thus obtain a smooth correspondence
$$
((X_{31},\widehat{\mathcal U_{31}}),\widehat f_{1;31},\widehat f_{3;31}).
$$
We call it the {\em composition} \index{composition! smooth correspondences} of
$((X_{21},\widehat{\mathcal U_{21}}),\widehat f_{1;21},\widehat f_{2;21})$
and
$((X_{32},\widehat{\mathcal U_{32}}),\widehat f_{2;32},\widehat f_{3;32})$.
\end{defn}
\begin{defn}
Suppose we have two perturbed smooth correspondences
$\Xi_{21} = ((X_{21},\widehat{\mathcal U_{21}}),\widehat{\frak S_{21}},\widehat f_{1;21},\widehat f_{2;21})$
and
$\Xi_{32} = ((X_{32},\widehat{\mathcal U_{32}}),\widehat{\frak S_{32}},\widehat f_{2;32},\widehat f_{3;32})$
from $M_1$ to $M_2$ and from $M_2$ to $M_3$.
We consider the composition
$
((X_{31},\widehat{\mathcal U_{31}}),\widehat f_{1;31},\widehat f_{3;31})
$
 of
$((X_{21},\widehat{\mathcal U_{21}}),\widehat f_{1;21},\widehat f_{2;21})$
and
$((X_{32},\widehat{\mathcal U_{32}}),\widehat f_{2;32},\widehat f_{3;32})$.
Using strong submersivity of $\widehat f_{1;21}$
with respect to $\widehat{\frak S_{21}}$, we can easily construct the {\em fiber product}
\index{fiber product of CF-perturbations} of two
CF-perturbations
$$
\widehat{\frak S_{21}} \,_{\widehat f_{2;21}}\times_{\widehat f_{2;32}} \widehat{\frak S_{32}}
$$
which we denote by $\widehat{\frak S_{32}}$.\footnote{See \cite[Definition 10.13 (2)]{fooo:tech2} for its proof.}
Then using strong submersivity of $\widehat f_{1;21}$,  $\widehat f_{3;32}$
with respect to $\widehat{\frak S_{21}}$, $\widehat{\frak S_{32}}$,
we can show that $\widehat f_{3;32}$ is strongly submersive with respect to
$\widehat{\frak S_{32}}$.\footnote{See  \cite[Lemma 10.14 (2)]{fooo:tech2} for its proof.}
\par
We thus obtain a perturbed smooth correspondence
$$
((X_{31},\widehat{\mathcal U_{31}}),\widehat{\frak S_{31}},\widehat f_{1;31},\widehat f_{3;31}),
$$
for each sufficiently small $\epsilon > 0$.
We call it the {\em composition} \index{composition! perturbed
smooth correspondences} of $\Xi_{21}$ and $\Xi_{32}$ and write
$$
\Xi_{32} \circ \Xi_{21}.
$$
\end{defn}
\begin{equation}\label{diagram1010}
\xymatrix{
&& \frak X_{31} \ar[ld]\ar[rd] \\
& \frak X_{21} \ar[ld]\ar[rd] && \frak X_{32}\ar[ld]\ar[rd] \\
M_1 && M_2 && M_3}
\nonumber\end{equation}

\subsection{Statement of Composition formula}
\label{subsec:composition}

Now the composition formula is:
\index{composition formula}
\begin{thm}\label{compform}
Let  $\Xi_{21}$ (resp. $\Xi_{32}$) be a perturbed smooth correspondence from $M_1$ to $M_2$
(resp. from $M_2$ to $M_3$). Then we have
\begin{equation}\label{formcomposos}
{\rm Corr}^{\epsilon}_{\Xi_{32}} \circ {\rm Corr}^{\epsilon}_{\Xi_{21}}
=
{\rm Corr}^{\epsilon}_{\Xi_{32}\circ \Xi_{21}}
\end{equation}
for each sufficiently small $\epsilon >0$.
\end{thm}
Theorem \ref{compform} is
\cite[Theorem 10.20]{fooo:tech2}.
({See also \cite{fooo:bulk} Lemma 12.15.)
\par
We consider the case $M_2$ is a point.
Let $h_1$ and $h_3$ be differential forms on $M_1$ and $M_3$
of degree $\dim (X_{21},\widehat{\mathcal U_{21}})$
and $\dim (X_{32},\widehat{\mathcal U_{32}})$, respectively.
Then we have
$$
\int_{M_3}({\rm Corr}^{\epsilon}_{\Xi_{32}} \circ {\rm Corr}^{\epsilon}_{\Xi_{21}})(h_1) \wedge h_3
=
\int_{((X_{21},\widehat{\mathcal U_{21}}),\widehat{\frak S_{21}})}
\widehat f_{1;21}^* h_1
\times
\int_{((X_{32},\widehat{\mathcal U_{32}}),\widehat{\frak S_{32}})}
\widehat f_{3;32}^* h_3.
$$
On the other hand
$$
\int_{M_3}{\rm Corr}^{\epsilon}_{\Xi_{32}\circ \Xi_{21}}(h_1) \wedge h_3
=
\int_{((X_{21},\widehat{\mathcal U_{21}})
\times (X_{32},\widehat{\mathcal U_{32}}),
\widehat{\frak S_{21}} \times \widehat{\frak S_{32}})}
\widehat f_{1;21}^* h_1
 \times \widehat f_{3;32}^* h_3.
$$
Thus in this case (\ref{formcomposos}) is a generalization of Fubini's theorem.
\par
To prove Theorem \ref{compform} we use the partition of unity to reduce
it to one chart. Then we can use Fubini's theorem to prove it.
See \cite[Section 10]{fooo:tech2}.
\par

\bibliographystyle{amsalpha}
\bibliographystyle{amsalpha}

\include{index}
\printindex

\end{document}